\documentclass{amsart}

\newtheorem{theorem}{Theorem}[section]
\newtheorem{prop}[theorem]{Proposition}
\newtheorem{lemma}[theorem]{Lemma}
\newtheorem{cor}[theorem]{Corollary}

\theoremstyle{definition}
\newtheorem{definition}[theorem]{Definition}
\newtheorem{example}[theorem]{Example}
\newtheorem{problem}[theorem]{Problem}
\newtheorem{conjecture}[theorem]{Conjecture}

\theoremstyle{remark}
\newtheorem{remark}{Remark} 

\newcommand{\Q}{\mathbb{Q}}
\newcommand{\R}{\mathbb{R}}
\newcommand{\Z}{\mathbb{Z}}
\newcommand{\Id}{\mathrm{id}}
\newcommand{\Rank}{\mathrm{rank}}
\newcommand{\Map}{\mathrm{map}}
\newcommand{\Conv}{\mathrm{conv}}
\newcommand{\Affi}{\mathrm{affi}}
\newcommand{\Cone}{\mathrm{cone}}
\newcommand{\Convcone}{\mathrm{convcone}}
\newcommand{\Vect}{\mathrm{vect}}
\newcommand{\QVect}{\mathbb{Q}\text{-}\mathrm{vect}}
\newcommand{\Clos}{\mathrm{clos}}
\newcommand{\Stab}{\mathrm{stab}}
\newcommand{\Supp}{\mathrm{supp}}
\newcommand{\Ht}{\mathrm{height}}
\newcommand{\Ord}{\mathrm{ord}}
\newcommand{\In}{\mathrm{in}}
\newcommand{\Ps}{\mathrm{ps}}
\newcommand{\Spec}{\mathrm{Spec}}
\newcommand{\Comp}{\mathrm{comp}}
\newcommand{\Prm}{\mathrm{prm}}
\newcommand{\Div}{\mathrm{div}}
\newcommand{\Den}{\mathrm{den}}
\newcommand{\Op}{^\mathrm{op}}
\newcommand{\Mx}{^\mathrm{max}}
\newcommand{\Fc}{^\mathrm{fc}}
\newcommand{\Res}{\mathrm{res}}
\newcommand{\Codim}{\mathrm{codim}}
\newcommand{\p}{\mathfrak{p}}
\newcommand{\q}{\mathfrak{q}}
\newcommand{\A}{\alpha}
\newcommand{\B}{\beta}

\begin{document}

\title[New Ideas for
Resolution of Singularities]
{New Ideas for\\
Resolution of Singularities\\
in Arbitrary Characteristics}

\author{Tohsuke Urabe}
\address{Department of Mathematical Sciences, Ibaraki University, Mito, 
Ibaraki, 310-8512, Japan}
\email{urabe@mx.ibaraki.ac.jp}
\urladdr{http://mathmuse.sci.ibaraki.ac.jp/urabe/}


\date{June 22, 2010.}

\dedicatory{Dedicated to Professor Heisuke Hironaka and Professor Shreeram S. Abhyankar}

\begin{abstract}
The concept of the maximal contact is the key in Hironaka's resolution theory. It treats local theory, and it is not effective in positive characteristics. This is the essential reason why Hironaka's theory treats only the case of characteristic zero.

In this article we propose the substitute for the maximal contact, which is effective in any characteristics of the ground field. We replace the maximal contact by a theorem in the theory of torus embeddings. 

Using essential ideas here, we would like to establish the theory of resolution of singularities in arbitrary characteristics in a global sense in the forthcoming articles.   
\end{abstract}

\maketitle

\section{Introduction}
\label{intro}
The concept of the maximal contact is the key in Hironaka's resolution theory. (Hironaka~\cite{H75}, \cite{H74}, \cite{H64}(II, Chapter III, sections 7-10), Giraud~\cite{G75}, \cite{G74}, Hauser~\cite{Ha10}.) It treats local theory, and it is not effective in positive characteristics. This is the essential reason why Hironaka's theory treats only the case of characteristic zero. The maximal contact is closely related to the multiplicity of a hypersurface singularity and the Hilbert-Samuel function of a general singularity. 

In this article we propose the substitute for the multiplicity and the maximal contact, which is effective in any characteristics of the ground field. We replace the multiplicity by the degree of the Weierstrass polynomial, or, the height of the Newton polyhedron, and we replace the maximal contact by a theorem in the theory of torus embeddings (Kempf et al.~\cite{KKMS}, Fulton~\cite{F93}).

The idea of the degree of the Weierstrass polynomial can be found in Hironaka~\cite{H67} in low dimensional cases. However, he did not manipulate higher dimensional cases, because he did not apply the theory of torus embeddings. See also Cossart et al.~\cite{CJS}.

Some ideas in this article are inspired by the appendix of Abhyankar~\cite{A98} and Bogomolov~\cite{B96}.

Let $k$ be \emph{any} algebraically closed field in \emph{any} characteristics; let $n$ be any positive integer, let $z, x(1),x(2),\ldots,x(n)$ be $(n+1)$ of variables over $k$; let $A$ denote the ring of formal power series of variables $z, x(1),x(2),\ldots,x(n)$ with coefficients in $k$. By $ A'$ we denote the $k$-subalgebra of $A$ consisting of power series of $n$ of variables $x(1), x(2),\ldots,x(n)$. The unique maximal ideal of $A$ (respectively, $A'$) is denoted by $M(A)$ (respectively, $M(A')$). The set of invertible elements of $A$ (respectively, $A'$) is denoted by $A^\times$ (respectively, $A^{\prime\times})$. We have $A=M(A)\cup A^\times, M(A)\cap A^\times=\emptyset, A'=M(A')\cup A^{\prime\times}, M(A')\cap A^{\prime\times}=\emptyset$, 
$M(A)\in\Spec(A)$, $M(A')\in\Spec(A')$, $M(A)$ is the unique closed point of the affine scheme $\Spec(A)$, and $M(A')$ is the unique closed point of the affine scheme $\Spec(A')$.
Let $\Z_0$ and $\Z_+$ denote the set of non-negative integers and the set of positive integers respectively, and let $P=\{z, x(1),x(2),\ldots,x(n)\}$. Note that $P$ is a parameter system of $A$ and $P-\{z\}=\{x(1),x(2),\ldots,x(n)\}$ is a parameter system of $A'$.
For any $\phi\in A$ with $\phi\neq 0$ by $\Gamma_+(P,\phi)$ we denote the Newton polyhedron of $\phi$ over $P$. 

Any element $\phi\in A$ satisfying  $\phi=z^h+\sum_{i=0}^{h-1} \phi'(i)z^i$ for some $h\in\Z_0$ and some mapping $\phi':\{0,1,\ldots,h-1\}\rightarrow M(A')$ is called a $z$-\emph{Weierstrass polynomial} over $P$, and the non-negative integer $h$ is called the \emph{degree} of $\phi$.

We consider any $\phi\in A$ with $\phi\neq 0$. The Newton polyhedron $\Gamma_+(P,\phi)$ is of $z$\emph{-Weierstrass type}, if and only if, there exist uniquely $u\in A^\times, a\in\Z_0, h\in\Z_0$, a mapping $b:\{1,2,\ldots,n\}\rightarrow\Z_0$ and a mapping $\phi':\{0,1,\ldots,h-1\}\rightarrow M(A')$ satisfying $\phi=uz^a\prod_{i=1}^nx(i)^{b(i)}(z^h+\sum_{i=0}^{h-1}\phi'(i)z^i)$ and $\phi'(0)\neq 0$ if $h\geq 1$. If $\Gamma_+(P,\phi)$ is of $z$-Weierstrass type, and if moreover, any compact face $F$ of $\Gamma_+(P,\phi)$ satisfies $\dim F\leq 1$, then we say that $\Gamma_+(P,\phi)$ is $z$\emph{-simple}. See Section~\ref{concept} for the definition of $z$-simpleness. 
In case where $\Gamma_+(P,\phi)$ is of $z$-Weierstrass type and $h\geq 1$, any face  of $\Gamma_+(P,\phi)$ satisfying a certain condition is called a $z$\emph{-removable} face. See Section~\ref{concept} also for the definition of $z$-removable faces.

Let $R$ be any complete regular local ring, and let $\phi\in R$. We consider a parameter system $Q$ of $R$. We say that $\phi$ has \emph{normal crossings} over $Q$, if $\phi=u\prod_{x\in Q}x^{\Lambda(x)}$ for some invertible element $u\in R$ and some mapping $\Lambda:Q\rightarrow\Z_0$. We say that $\phi$ has \emph{normal crossings}, if $\phi$ has normal crossings over $Q$ for some parameter system $Q$ of $R$.

Furthermore, we denote
\begin{equation*}\begin{split}
&PW(1)= \{\phi\in A|\phi= u\prod_{\chi\in\mathcal{X}}(z+\chi)^{a(\chi)}\prod_{i=1}^n x(i)^{b(i)} \\
&\quad\text{for some }u\in A^\times, \text{ some }r\in\Z_0, \text{ some finite subset }\mathcal{X}\text{ of } M(A'),\\
&\quad \text{some mapping }a: \mathcal{X}\rightarrow \Z_+, \text{ and some mapping }b:\{1,2,\ldots,n\}\rightarrow\Z_0.\}\\
\end{split}\end{equation*}
\vfill

For any $h\in\Z_+$ with $h\geq 2$, we denote
\begin{equation*}\begin{split}
W(h)&=\{\phi\in A|\phi= z^h+\sum_{i=0}^{h-1} \phi'(i)z^i \\
&\qquad\quad \text{for some mapping }\phi':\{0,1,\ldots,h-1\}\rightarrow M(A') \text{ satisfying }\\
&\qquad\quad 
\chi^h+\sum_{i=0}^{h-1} \phi'(i)\chi^i\neq 0 \text{ for any }
\chi\in M(A').\},\\
PW(h)&=\{\phi\in A|\phi=\psi\psi'\text{ for some } \psi\in W(h) \text{ and some }\psi'\in PW(1).\},\\
SW(h)&=\{\phi\in A|\phi=\psi\psi'\text{ for some } \psi\in W(h) \text{ and some }\psi'\in PW(1)\text{ such that}\\
&\qquad\quad \Gamma_+(P,\psi) \text{ has no }z \text{-removable faces, and }\Gamma_+(P,\phi) \text{ is }z \text{-simple.}\}.\\
\end{split}\end{equation*}

Note that $A$ is a unique factorization domain
(Matsumura~\cite{M}.),
$1\in PW(1)\neq\emptyset$, $A=\{0\}\cup\cup_{h\in\Z_+}PW(h)$ if $n=1$, and $\Gamma_+(P,\phi)$ is $z$-simple for any $\phi\in A$ with $\phi\neq 0$ if $n=1$.
We consider any integer $h$ with $h\geq 2$. $\emptyset\neq W(h)\subset PW(h) \not\ni 0$, $\emptyset\neq SW(h)\subset PW(h)$. For any $\psi\in W(h)$, the Newton polyhedron $\Gamma_+(P,\psi)$ is of $z$-Weierstrass type, and the integer $h$ is equal to the $z$-height of $\Gamma_+(P,\psi)$. For any $\phi\in PW(h)$, the Newton polyhedron $\Gamma_+(P,\phi)$ is of $z$-Weierstrass type, and there exist uniquely $\psi\in W(h)$ and $\psi'\in PW(1)$ with $\phi=\psi\psi'$.

Our main theorem, Theorem~\ref{main} in this article claims the following: For any integer $h$ with $h\geq 2$ and any $\phi\in SW(h)$, there exists a composition $\sigma:\Sigma\rightarrow\Spec(A)$ of finite blowing-ups with centers in closed irreducible smooth subschemes of codimension two such that at any closed point $\A\in\Sigma$ with $\sigma(\A)=M(A)$ the following holds: Let $\mathcal{O}_{\Sigma,\A}^c$ denote the completion of the local ring $\mathcal{O}_{\Sigma,\A}$ of $\Sigma$ at $\A$. Note that the morphism $\sigma$ induces a homomorphism of $k$-algebras $\sigma^*:A\rightarrow\mathcal{O}_{\Sigma,\A}^c$, and $\mathcal{O}_{\Sigma,\A}^c$ and $A$ are isomorphic as $k$-algebras. We claim that there exists an isomorphism $\rho:\mathcal{O}_{\Sigma,\A}^c\rightarrow A$ of $k$-algebras such that $\rho\sigma^*(\phi)\in PW(g)$ for some $g\in\Z_+$ with $g<h$.

Since $g<h$, we can claim that any hypersurface singularity can be improved by a composition of finite blowing-ups.

\begin{remark}
We do \emph{not} claim that the centers of blowing-ups are contained in the singular locus of the subscheme to be resolved. It may be possible to improve our main theorems and to add stataments claiming that any centers of blowing-ups are contained in the singular locus of the subscheme to be resolved.
\end{remark}

We would like to show that for any $\phi\in A$ with $\phi\neq 0$, there exists a composition $\sigma:\Sigma\rightarrow\Spec(A)$ of finite blowing-ups with centers in closed irreducible smooth subschemes  such that at any closed point $\A\in\Sigma$ with $\sigma(\A)=M(A)$ the element $\sigma^*(\phi)\in \mathcal{O}_{\Sigma,\A}^c$ has normal crossings.

Note here that $\dim A'=\dim A-1<\dim A=n+1$, and any $\phi'\in A'$ with $\phi'\neq 0$ has normal crossings over $P-\{z\}$ if $n=1$. Therefore, we decide that we use induction on $n$, and we can assume the following claim $(*)$:

\renewcommand{\descriptionlabel}[1]%
	{\hspace{\labelsep}\textrm{#1}}
\begin{description}
\item[$(*)$]
For any  $\phi'\in A'$ with $\phi'\neq 0$, there exists a composition $\sigma':\Sigma'\rightarrow\Spec(A')$ of finite blowing-ups with centers in closed irreducible smooth subschemes  such that at any closed point $\A'\in\Sigma'$ with $\sigma' (\A')=M(A')$ the element $\sigma^{\prime*}(\phi')\in \mathcal{O}_{\Sigma',\A'}^c$ has normal crossings.
\end{description}

Let $\sigma':\Sigma'\rightarrow\Spec(A')$ be any composition of finite blowing-ups with centers in closed irreducible smooth subschemes. We consider a morphism $\Spec(A)\rightarrow\Spec(A')$ induced by the inclusion ring homomorphism $A'\rightarrow A$, the product scheme $\Sigma=\Sigma'\times_{\Spec(A')}\Spec(A)$, the projection $\sigma:\Sigma\rightarrow\Spec(A)$, and the projection $\pi:\Sigma\rightarrow\Sigma'$. We know the following (See Lemma~\ref{pull back blowing-ups}.):
\begin{enumerate}
\item The morphism $\sigma$ is a composition of finite blowing-ups with centers in closed irreducible smooth subschemes.
\item The pull-back $\sigma^*\Spec(A/zA)$ of the smooth prime divisor $\Spec(A/zA)$ of $\Spec(A)$ by $\sigma$ is a smooth prime divisor of $\Sigma$, and $\sigma^*\Spec(A/zA)\supset\sigma^{-1}(M(A))$.
\item The projection $\pi:\Sigma\rightarrow\Sigma'$ induces an isomorphism $\sigma^*\Spec(A/zA)\rightarrow\Sigma'$.
\end{enumerate}

Furthermore, we show the following three claims, assuming the above $(*)$:

First, for any intger $h$ with $h\geq 2$ and any $\phi\in PW(h)$, there exists a composition $\sigma':\Sigma'\rightarrow\Spec(A')$ of finite blowing-ups with centers in closed irreducible smooth subschemes  such that considering the product scheme $\Sigma=\Sigma'\times_{\Spec(A')}\Spec(A)$, the projection $\sigma:\Sigma\rightarrow\Spec(A)$ and the projection $\pi:\Sigma\rightarrow\Sigma'$, at any closed point $\A\in\Sigma$ with $\sigma(\A)=M(A)$ there exists an isomorphism $\rho:\mathcal{O}_{\Sigma,\A}^c\rightarrow A$ of $k$-algebras such that $\rho\sigma^*(z)=z$, $\rho\pi^*(\mathcal{O}_{\Sigma',\pi(\A)}^c)=A'$, and either $\rho\sigma^*(\phi)\in SW(h)$ or $\rho\sigma^*(\phi)\in PW(g)$ for some  positive integer $g\in \Z_+$ with $g<h$. (Theorem~\ref{erase faces} and Theorem~\ref{make simple}.)

Second, for any $\phi\in A$ with $\phi\neq 0$, there exists a composition $\sigma':\Sigma'\rightarrow\Spec(A')$ of finite blowing-ups with centers in closed irreducible smooth subschemes  such that considering the product scheme $\Sigma=\Sigma'\times_{\Spec(A')}\Spec(A)$, the projection $\sigma:\Sigma\rightarrow\Spec(A)$ and the projection $\pi:\Sigma\rightarrow\Sigma'$, at any closed point $\A\in\Sigma$ with $\sigma(\A)=M(A)$ there exists an isomorphism $\rho:\mathcal{O}_{\Sigma,\A}^c\rightarrow A$ of $k$-algebras such that $\rho\sigma^*(z)=z$, $\rho\pi^*(\mathcal{O}_{\Sigma',\pi(\A)}^c)=A'$ and $\rho\sigma^*(\phi)\in PW(h)$ for some $h\in\Z_+$. (Theorem~\ref{make Weierstrass type}.)

Third, for any $\phi\in PW(1)$, there exists a composition $\sigma:\Sigma\rightarrow\Spec(A)$ of finite blowing-ups with centers in closed irreducible smooth subschemes  such that at any closed point $\A\in\Sigma$ with $\sigma(\A)=M(A)$ the element $\sigma^*(\phi)\in\mathcal{O}_{\Sigma,\A}^c$ has normal crossings. (Theorem~\ref{make normal crossings}.)

Now, we would like to establish the theory of resolution of singularities in arbitrary characteristics in a global sense. The problem is to glue up local blowing-ups obtained by repeated application of the above four claims, and to construct a global blowing-up. In case $n=1$ it is easy to glue up them. We would like to solve this problem in case $n\geq 2$ and would like to write forthcoming articles, cooperating with Professor Heisuke Hironaka and young mathematicians.

We give proofs only to difficult parts of our claims. Most of our claims follow from definitions. 

\begin{center}
\textsc{Table of contents}
\end{center}

\begin{description}
\item[˜~\ref{intro}] Introduction
\item[˜~\ref{concept}] Notations and basic concepts
\item[˜~\ref{scheme}] Basic scheme theory
\item[˜~\ref{mainmain}] Main results
\item[˜~\ref{ring}] Ring theory
\item[˜~\ref{btcs}] Basic theory of convex sets
\item[˜~\ref{cones}] Convex cones and convex polyhedral cones
\item[˜~\ref{simplex}] Simplicial cones
\item[˜~\ref{decomposition}] Convex polyhedral cone decompositions
\item[˜~\ref{cpp}] Convex pseudo polyhedrons
\item[˜~\ref{bcd}] Barycentric subdivisions
\item[˜~\ref{ibcd}] Iterated barycentric subdivisions
\item[˜~\ref{simple}] Simpleness and semisimpleness
\item[˜~\ref{basic}] Basic subdivisions
\item[˜~\ref{upper}] Upper boundaries and lower boundaries
\item[˜~\ref{compatible}] Height, characteristic functions and compatible mappings
\item[˜~\ref{height inequalities}] The height inequalities
\item[˜~\ref{upward}] Upward subdivisions and the hard height inequalities
\item[˜~\ref{toric theory}] Schemes associated with simplicial cone decompositions
\item[˜~\ref{main proof}] Proof of the main theorem
\item[˜~\ref{submain proofs}] Proof of the submain theorems
\end{description}

The most important is ``the hard height inequality" in section eighteen. 
It depends heavily on ``the height inequality" in section seventeen.
In Sections 6-16 we develop exact theory of convex sets.

We consider any $h\in\Z_+$ with $h\geq 2$ and any $\phi\in SW(h)$ and we consider the face cone decomposition associated with the Newton polyhedron $\Gamma_+(P,\phi)$, which is a $z$-simple convex polyhedral cone decomposition in the dual vector space.  Starting from a simplicial cone in the dual vector space, repeating barycentric subdivisions, we construct the most natural simplicial subdivision of the face cone decomposition which we call the upward subdivision.  By the theory of torus embeddings, we have a repeatedly blown-up space corresponding to the simplicial subdivision.  Then, there exists an inequality which holds simultaneously at any point on the repeatedly blown-up space in the inverse image of the origin of the space containing the hypersurface defined by $\phi$.  This inequality guarantees that any hypersurface singularity can be improved by blowing-ups.

\section{Notations and basic concepts}
\label{concept}
We arrange notations and basic concepts related to Newton polyhedrons and commutative rings.

We denote
$$\Z_0=\{t\in\Z|t\geq 0\}, \qquad \Z_+=\{t\in\Z|t>0\},$$
$$\Q_0=\{t\in\Q|t\geq 0\}, \qquad \Q_+=\{t\in\Q|t>0\},$$
$$\R_0=\{t\in\R|t\geq 0\}, \qquad \R_+=\{t\in\R|t>0\}.$$

For any set $Z$ and any scheme $\Sigma$, by $\Id_Z$ and $\Id_\Sigma$ we denote the identity mapping $Z\rightarrow Z$ and the identity morphism $\Sigma\rightarrow\Sigma$ of schemes respectively.
For any set $Z$, the set of all subsets of $Z$ is denoted by $2^Z$.

Let $Z$ be any set, and let $X$ and $Y$ be any subsets of $Z$. The union $X\cup Y$ and the intersection $X\cap Y$ are defined. They are subsets of $Z$. The set of elements of $X$ not belonging to $Y$ is denoted by $X- Y$, and is called the \emph{difference} of $X$ and $Y$. $X- Y=\{z\in Z|z\in X, z\not\in Y\}$. 

Below assume that $Z$ is an abelian group. 

We consider any $r\in\Z_0$ and any mapping $x:\{1,2,\ldots\,r\}\rightarrow Z$.
For any $s\in\{0,1,\ldots, r\}$ we define inductively
\begin{equation*}
\sum_{i=1}^{s}x(i)=
\begin{cases}
0&\text{if $s=0$},\\
\sum_{i=1}^{s-1}x(i)+x(s)&\text{if $s\neq 0$}.
\end{cases}
\end{equation*}
Note that $\sum_{i=1}^{r}x(i)\in Z$, $\sum_{i=1}^{r}x(i)=0$ if $r=0$, and $\sum_{i=1}^{r}x(i)= \sum_{i=1}^{r}x(\nu(i))$ for any bijective mapping $\nu: \{1,2,\ldots\,r\}\rightarrow\{1,2,\ldots\,r\}$.

We consider any finite set $J$ and any mapping $x:J\rightarrow Z$.
Let $r\in\Z_0$ be the number of elements in $J$.
Choosing a bijective mapping $\nu: \{1,2,\ldots\,r\}\rightarrow J$, we define
$$\sum_{j\in J}x(j)=\sum_{i=1}^rx(\nu(i)).$$
Note that $\sum_{j\in J}x(j)\in Z$ and $\sum_{j\in J}x(j)$ does not depend on the choice of the bijective mapping $\nu: \{1,2,\ldots\,r\}\rightarrow J$ we used for the definition.
If $J=\emptyset$, then $\sum_{j\in J}x(j)=0$.

We call $\sum_{j\in J}x(j)$ the \emph{sum} of $x(j), j\in J$.

We consider any finite set $J$ and any mapping $X$ from $J$ to the set $2^Z$ of all subsets of $Z$.
We define
\begin{equation*}
\begin{split}
\sum_{j\in J}X(j)=\{z\in Z|&z=\sum_{j\in J}x(j)\text{ for some mapping }x:J\rightarrow Z\text{ satisfying}\\
&\quad x(j)\in X(j)\text{ for any }j\in J\}\in 2^Z.
\end{split}
\end{equation*}
Note that $\sum_{j\in J}X(j)$ is a subset of $Z$, $\sum_{j\in J}X(j)=\{0\}$ if $J=\emptyset$, and $\sum_{j\in J}X(j)=\emptyset$, if and only if, $J\neq\emptyset$ and $X(j)=\emptyset$ for some $j\in J$.

We call $\sum_{j\in J}X(j)$ the \emph{sum} of $X(j), j\in J$.

For any $r\in\Z_+$ and for any mapping $X$ from $\{1,2,\ldots\,r\}$ to the set of all subsets of $Z$ we also write
$$X(1)+X(2)+\cdots+X(r)= \sum_{j\in \{1,2,\ldots\,r\}}X(j).$$

For any subset $X$ of $Z$, we denote $-X=\{z\in Z|z=-x$ for some $x\in X\}$.

Let $X$ and $Y$ be any sets. The set of mappings from $X$ to $Y$ is denoted by $\Map(X,Y)$. The set $\Map(X,Y)$ has a natural structure of an abelian group, if $Y$ is an abelian group. It has a natural structure of an abelian semigroup, if $Y$ is an abelian semigroup. It has a natural structure of a vector space over $\R$, if $Y$ is a vector space over $\R$. In addition, let $Z$ be a set containing $Y$. Note that the inclusion mapping $Y\rightarrow Z$ induces the inclusion mapping $\Map(X,Y)\rightarrow\Map(X,Z)$, and we can regard $\Map(X,Y)$ as a subset of $\Map(X,Z)$. If $Z$ is an abelian group and $Y$ is a subgroup of $Z$, then  $\Map(X,Y)$ is a subgroup of $\Map(X,Z)$. If $Z$ is an abelian semigroup and $Y$ is a semisubgroup of $Z$, then  $\Map(X,Y)$ is a semisubgroup of $\Map(X,Z)$. If $Z$ is a vector space over $\R$ and $Y$ is a vector subspace of $Z$ over $\R$, then $\Map(X,Y)$ is a vector subspace of $\Map(X,Z)$ over $\R$.

Let $X$ be any set; let $Z$ be any abelian group, and let $Y$ be any subset of $Z$ with $0\in Y$. For any $a\in\Map(X,Y)$ we denote
$$\Supp(a)=\{x\in X|a(x)\neq 0\},$$
and we call $\Supp(a)$ the \emph{support} of $a$. It is a subset of $X$. We denote
$$\Map'(X,Y)=\{a\in\Map(X,Y)| \Supp(a)\text{ is a finite set.}\}.$$
$\Map'(X,Y)\subset\Map(X,Y)$. If $Y$ is a subgroup of $Z$, then $\Map'(X,Y)$ is a subgroup of $\Map(X,Y)$. If $Y$ is a semisubgroup of $Z$, then $\Map'(X,Y)$ is a semisubgroup of $\Map(X,Y)$. If $Z$ is a vector space over $\R$ and $Y$ is a vector subspace of $Z$ over $\R$, then $\Map'(X,Y)$ is a vector subspace of $\Map(X,Y)$ over $\R$. If $X$ is a finite set, we have $\Map'(X,Y)=\Map(X,Y)$.

In the case where $Y$ is a semisubgroup of $Z$ containing $0$, for any $a\in\Map'(X,Y)$ we denote
$$\sum_{x\in X}a(x)=\sum_{x\in\Supp(a)}a(x)\in Y.$$

We call $\sum_{x\in X}a(x)$ the \emph{sum} of $a(x), x\in X$.

Let $V$ be any finite dimensional vector space over $\R$, and let $X$ be any subset of $V$. The subset $X$ is called \emph{convex}, if $X\neq\emptyset$ and for any two different points $x,y$ of $X$, the \emph{segment} $\{a\in V |a=(1-t)x+ty \text{ for some }t\in\R\text{ with }0\leq t\leq 1\}$ joining $x$ and $y$ is contained in $X$. It is called an \emph{affine space}, if $X\neq\emptyset$ and for any two different points $x,y$ of $X$, the \emph{line} $\{a\in V |a=(1-t)x+ty \in\R\text{ for some }t\in\R\}$ joining $x$ and $y$ is contained in $X$. It is called a \emph{cone}, if $0\in X$ and for any $x\in X$ and any $t\in\R_0$, we have $tx\in X$. It is called a \emph{convex cone}, if $0\in X$ and for any $x,y\in X$ and any $t,u\in\R_0$, we have $tx+uy\in X$. It is called a \emph{vector space over} $\R$, or simply a \emph{vector space}, if $0\in X$ and for any $x,y\in X$ and any $t,u\in\R$, we have $tx+uy\in X$. It is called a \emph{vector space over} $\Q$, if $0\in X$ and for any $x,y\in X$ and any $t,u\in\Q$, we have $tx+uy\in X$. It is called \emph{closed}, if the limit of any convergent sequence of elements in $X$ with respect to the natural Hausdorff topology of $V$ belongs to $X$ again. In case $X\neq\emptyset$ the minimum convex subset (respectively, minimum affine space) with respect to the inclusion relation containing $X$ is denoted by $\Conv(X)$ (respectively, $\Affi(X)$). We define $\Conv(\emptyset)=\Affi(\emptyset)=\emptyset$. 
The minimum cone (respectively, minimum convex cone, minimum vector space over $\R$, minimum vector space over $\Q$, minimum closed subset) with respect to the inclusion relation containing $X$ is denoted by
$\Cone(X)$ (respectively, $\Convcone(X), \Vect(X), \QVect(X), \Clos(X)$). 

The subset $X$ is called a \emph{convex polyhedron},(respectively, \emph{convex polyhedral cone}), if there exists a \emph{finite} subset $Y$ of $V$ satisfying $X=\Conv(Y)$ and $Y\neq\emptyset$ (respectively, $X=\Convcone(Y)$). The subset $X$ is called a \emph{convex pseudo polyhedron}, if there exist \emph{finite} subsets $Y, Z$ of $V$ satisfying $X=\Conv(Y)+\Convcone(Z)$ and $Y\neq\emptyset$. The subset $X$ is called a \emph{lattice}, if there exists a basis $B$ of $V$ over $\R$ such that $X=\{a\in V|a=\sum_{b\in B}\lambda(b)b\text{ for some }\lambda\in\Map(B, \Z)\}$.
Any lattice $N$ of $V$ is a free $\Z$-submodule of $V$ with $\Rank N=\dim V$. For any $t\in\R$ we write $tX=\{a\in V|a=tx\text{ for some }x\in X\}$. We  know $(-1)X=-X$, and $0X=\{0\}$ if $X\neq\emptyset$. We write
$\Stab(X)=\{a\in V|X+\{a\}\subset X\}$,
and call it the \emph{stabilizer} of $X$ in $V$. The stabilizer of $X$ in $V$ is a semisubgroup of $V$ containing $0$. In addition, we consider any lattice $N$ in $V$. The subset $X$ is called a \emph{simplicial cone} over N, if $X=\Convcone(C)$ for some basis $B$ of $N$ over $\Z$ and a subset $C$ of $B$. Any simplicial cone is a convex polyhedral cone. Affine spaces, vector spaces, convex polyhedrons, convex polyhedral cones, and convex pseudo polyhedrons are non-empty closed convex subsets of $V$. If $X$ is convex (respectively, a cone, a convex cone), then $\Clos(X)$ is again convex (respectively, a cone, a convex cone). 

For any subset $T$ of $\R$ and for any $a\in V$ we denote $Ta=\{b\in V|b=ta\text{ for some }t\in T\}$, and it is a subset of $V$.

The dual vector space $V^*=\mathrm{Hom}_\R(V,\R)$ is a vector space over $\R$ with $\dim V^*=\dim V$. We define the \emph{canonical bilinear form}
$$\langle\quad,\quad\rangle:V^*\times V\rightarrow \R,$$
by putting $\langle\omega, a\rangle=\omega(a)\in\R$ for any $\omega\in\mathrm{Hom}_\R(V,\R)=V^*$ and any $a\in V$. The dual vector space $V^{**}$ of $V^*$ is identified with $V$ by the natural isomorphism $V\rightarrow V^{**}$ of vector spaces over $\R$.

We consider any finite dimensional vector space $W$ over $\R$ and any homomorphism $\pi:V\rightarrow W$ of vector spaces over $\R$.
Putting
$$\pi^*(\A)=\A \pi\in\mathrm{Hom}_\R(V,\R)=V^*,$$
for any $\A\in\mathrm{Hom}_\R(W,\R)=W^*$,
we define a mapping $\pi^*:W^*\rightarrow V^*$, and we call $\pi^*$ the \emph{dual homomorphism} of $\pi$.
The dual homomorphism $\pi^*$ is a homomorphism of vector spaces over $\R$.
For any $\omega\in W^*$ and for any $a\in V$ the equality $\langle\pi^*(\omega),a\rangle=\langle\omega,\pi(a)\rangle$ holds.

The dual homomorphism $\pi^{**}$ of $\pi^*$ is equal to $\pi$.

If $\pi$ is injective, then $\pi^*$ is surjective.
If $\pi$ is surjective, then $\pi^*$ is injective.

We have $\Id_V^*=\Id_{V^*}$, and for any finite dimensional vector spaces $V', V''$ and for any homomorphisms $\pi:V\rightarrow V', \pi':V'\rightarrow V''$ we have $(\pi'\pi)^*=\pi^*\pi^{\prime*}$.

Let $N$ be a lattice in $V$. We denote 
$$N^*=\{\omega\in V^*|\langle\omega, a\rangle\in\Z \text{ for any } a\in N\},$$
and call $N^*$ the \emph{dual lattice} of $N$. Indeed, $N^*$ is a lattice in $V^*$. The dual lattice $N^{**}$ of $N^*$ is equal to $N$. Let $S$ be any convex cone in $V$. We denote
$$S^\vee|V=\{\omega\in V^*|\langle\omega, a\rangle\geq 0 \text{ for any } a\in S\},$$
and call $ S^\vee|V$ the \emph{dual cone} of $S$ over $V$. Indeed, $ S^\vee|V$ is a closed convex cone in $V^*$. The dual cone $ S^\vee|V^\vee|V^*$ of $ S^\vee|V$ is equal to the closure $\Clos(S)$ of $S$ in $V$. $S^\vee|V^\vee|V^*=S$ if and only if $S$ is closed in $V$. When we need not refer to $V$, we also write simply $S^\vee$, instead of $ S^\vee|V$.

The number of elements of a finite set $P$ is denoted by $\sharp P$. Let $P$ be any non-empty finite set. Note that $\Map(P,\R)$ is a finite dimensional vector space over $\R$ with $\dim \Map(P,\R)=\sharp P$,  $\Map(P,\Z)$ is a lattice in $\Map(P,\R)$,  $\Map(P,\R_0)$ is a simplicial cone over $\Map(P,\Z)$ in $\Map(P,\R)$ with $\Vect(\Map(P,\R_0))= \Map(P,\R)$, and $\Map(P,\Z_0)= \Map(P,\Z)\cap \Map(P,\R_0)$. Let $x\in P$. Let $y\in P$. Putting
$$f^P_x(y)=
\begin{cases}
1&\text{if $y=x$},\\
0&\text{if $y\neq x$},
\end{cases}$$
we define an element $f^P_x\in \Map(P,\Z_0)$. Note that the subset $\{f^P_x|x\in P\}$ of $\Map(P,\Z_0)$ is an $\R$-basis of $\Map(X,\R)$,  it is a $\Z$-basis of $\Map(P,\Z)$, and  $\Map(P,\R_0)$\hfill\break$=\Convcone(\{f^P_x|x\in P\})$. The dual basis of $\{f^P_x|x\in P\}$ is denoted by $\{f^{P\vee}_x|x\in P\}$. For any $x,y\in P$
$$\langle f^{P\vee}_x, f^P_y\rangle=
\begin{cases}
1&\text{if $x=y$},\\
0&\text{if $x\neq y$}.
\end{cases}$$
Indeed, $\{f^{P\vee}_x|x\in P\}$ is a $\R$-basis of the dual vector space $\Map(P,\R)^*$ of $\Map(P,\R)$, it is a $\Z$-basis of the dual lattice  $\Map(P,\Z)^*$ of $\Map(P,\Z)$, and $\Map(P,\R_0)^{\vee}=\Convcone(\{f^{P\vee}_x|x\in P\})$.

A commutative ring  with the identity element is called simply a \emph{ring}. The identity element and the zero element of a ring are denoted $1$ and $0$ respectively.
We assume that any ring homomorphism $\lambda$ preserves the identity elements, in other words, the equality $\lambda(1)=1$ holds.

Let $R$ be any ring.
We assume that for any $R$-module $L$ and any element $x\in L$, the equality $1x=x$ holds.
The equality $1=0$ holds, if and only if, $R=\{0\}$.
We say that a subset $S$ of $R$ is a \emph{subring} of $R$, if $1\in S$, $a-b\in S$ for any $a\in S$ and any $b\in S$, and $ab\in S$ for any $a\in S$ and any $b\in S$. We say that a subset $I$ of $R$ is an \emph{ideal} of $R$, if $0\in I$, $a-b\in I$ for any $a\in I$ and any $b\in I$, and $ab\in I$ for any $a\in I$ and any $b\in R$.
For any ideal $I$ of $R$, $1\in I$, if and only if, $I=R$.
We say that an ideal $I$ of $R$ is \emph{prime}, if $1\not\in I$ and $ab\not\in I$ for any $a\in R-I$ and any $b\in R-I$.
We say that an ideal $I$ of $R$ is \emph{maximal}, if $1\not\in I$ and $I=J$ for any ideal $J$ of $R$ satisfying $1\not\in J$ and $I\subset J$.
Any maximal ideal of $R$ is a prime ideal of $R$. $R$ has at least one maximal ideal, if and only if, $R$ has at least one prime ideal, if and only if, $1\neq 0$.
Let $X$ be any subset of $R$ and let $S$ be any subring of $R$. The minimum ideal of $R$ with respect the inclusion relation containing $X$ is denoted by $XR$ or $RX$. The minimum subring of $R$ with respect to the inclusion relation containing $S$ and $X$ is denoted by $S[X]$. In the case where $X$ contains only one element $x$, we also write simply $xR$, $Rx$, $S[x]$, instead of $\{x\}R$, $R\{x\}$, $S[\{x\}]$ respectively.
We say that $R$ is \emph{noetherian}, if for any ideal $I$ of $R$, there exists a \emph{finite} subset $X$ of $I$ with $I=XR$.
We say an element $a\in R$ is \emph{invertible}, if there exists an element $b\in R$ with $ab=1$. The set of all invertible elements in $R$ is denoted by $R^{\times}$. $R^{\times}\subset R$ and $R^{\times}$ is an abelian group with respect the multiplication.
We say that $R$ is \emph{reduced}, if $a=0$ for any $a\in R$ and any $i\in\Z_+$ satisfying $a^i=0$.
We say that $R$ is an \emph{integral domain}, if $1\neq 0$ and $a=0$ or $b=0$ for any $a\in R$ and any $b\in R$ satisfying $ab=0$.
We say that $R$ is a \emph{field}, if $R$ is an integral domain and $R^\times=R-\{0\}$.
Any ring with a unique maximal ideal is called a \emph{local ring}.

We consider any $r\in\Z_0$ and any mapping $x:\{1,2,\ldots,r\}\rightarrow R$.
For any $s\in\{0,1,\ldots,r\}$ we define inductively
\begin{equation*}
\prod_{i=1}^sx(i)=
\begin{cases}
1&\text{if $s=0$},\\
(\prod_{i=1}^{s-1}x(i))x(s)&\text{if $s\neq 0$}.
\end{cases}
\end{equation*}
Note that $\prod_{i=1}^rx(i)\in R$, $\prod_{i=1}^rx(i)=1$ if $r=0$ and $\prod_{i=1}^rx(i)= \prod_{i=1}^rx(\nu(i))$ for any bijective mapping $\nu: \{1,2,\ldots,r\}\rightarrow\{1,2,\ldots,r\}$.

We consider any finite set $J$ and any mapping $x:J\rightarrow R$.
Let $r\in\Z_0$ be the number of elements in $J$.
Choosing a bijective mapping $\nu: \{1,2,\ldots,r\}\rightarrow J$, we define
$$\prod_{j\in J}x(j)= \prod_{i=1}^rx(\nu(i))\in R.$$
Note that $\prod_{j\in J}x(j)$ does not depend on the choice of the bijective mapping $\nu: \{1,2,\ldots,r\}\rightarrow J$ we used for the definition.
If $J=\emptyset$, then $\prod_{j\in J}x(j)=1$.

We call $\prod_{j\in J}x(j)$ the \emph{product} of $x(j), j\in J$.

Let $R$ be any ring and let $I$ be any ideal of $R$. There exist a ring $S$ and a surjective ring homomorphism $\lambda:R\rightarrow S$ such that $I=\lambda^{-1}(0)$. 
When a pair $(S,\lambda)$ satisfies this condition, we denote $S$ by a symbol $R/I$, we call the ring $R/I$ a \emph{residue ring} of $R$ by $I$ and we call $\lambda:R\rightarrow R/I$ the \emph{canonical homomorphism}. If $T$ is a ring, $\mu:R\rightarrow T$ is a ring homomorphism satisfying $I\subset \mu^{-1}(0)$ and $\lambda:R\rightarrow R/I$ is the canonical homomorphism, then there exists uniquely a ring homomorphism $\nu:R/I\rightarrow T$ satisfying $\nu\lambda=\mu$.
The ideal $I$ is prime, if and only if, the residue ring $R/I$ of $R$ by $I$ is an integral domain. The ideal $I$ is maximal, if and only if, the residue ring $R/I$ of $R$ by $I$ is a field.

A ring $R$ is an integral domain, if and only if, the subset $\{0\}$ of $R$ is a prime ideal of $R$.

A ring $R$ is a field, if and only if, the subset $\{0\}$ of $R$ is a maximal ideal of $R$, if and only if,  the subset $\{0\}$ of $R$ is a prime ideal of $R$ and any ideal $I$ of $R$ satisfies $I=R$ or $I=\{0\}$. Any field is an integral domain, it is a local ring and it is noetherian.

Let $R$ be any integral domain. The ring $R$ is reduced. There exist a field $K$ and an injective ring homomorphism $\lambda:R\rightarrow K$
such that for any $c\in K$ there exist $a\in R-\{0\}$ and $b\in R$ satisfying $\lambda(a)c=\lambda(b)$. When the pair $(K,\lambda)$ satisfies this condition, we call $K$ a \emph{quotient field} of $R$ and we call $\lambda:R\rightarrow K$ the \emph{canonical homomorphism}. If $L$ is a field,  $\mu:R\rightarrow L$ is an injective ring homomorphism, $K$ is a quotient field of $R$ and $\lambda:R\rightarrow K$ the canonical homomorphism, then there exists uniquely an injective ring homomorphism $\nu:K\rightarrow L$ satisfying $\nu\lambda=\mu$.

A ring $R$ is local, if and only if, $R- R^{\times}$ is an ideal of $R$.

Let $R$ be any local ring.
The unique maximal ideal of $R$ is denoted by $M(R)$. We have $R=R^\times\cup M(R), R^\times\cap M(R)=\emptyset$ and $1\neq 0$.
If $R$ is noetherian, then the Krull dimension $\dim R\in\Z_0$ of $R$ is defined. 
A local ring $R$ is called \emph{regular}, if $R$ is noetherian, and the dimension $\dim R$ of $R$ is equal to the dimension of the residue module $M(R)/M(R)^2$ as a vector space over the residue field $R/M(R)$. It is known that a regular local ring is a unique factorization domain. A finite subset $P$ of a regular local ring $R$ is called a \emph{parameter system} of R, if $P\subset M(R)$, $PR=M(R)$, and $\sharp P=\dim R$.

Let $S$ be a ring. The pair $(R,\lambda)$ of a ring $R$ and a ring homomorphism $\lambda:S\rightarrow R$ is called an \emph{algebra over} $S$, or an $S$\emph{-algebra}. Let $(R,\lambda)$ and $(R',\lambda')$ be algebras over $S$.
A ring homomorphism $\mu:R\rightarrow R'$ satisfying $\mu\lambda=\lambda'$ is called
a \emph{homomorphism over} $S$, an $S$\emph{-homomorphism} or a \emph{homomorphism of} $S$\emph{-algebras}.
A ring isomorphism $\mu:R\rightarrow R'$ satisfying $\mu\lambda=\lambda'$ is called an \emph{isomorphism over} $S$, an $S$\emph{-isomorphism} or an \emph{isomorphism of} $S$\emph{-algebras}.
We say that $S$-algebras $(R,\lambda)$ and $(R',\lambda')$ are \emph{isomorphic as} $S$\emph{-algebras}, if there exists an isomorphism  $\mu:R\rightarrow R'$ over $S$.

Consider $S$-algebras $(S,\Id_S)$ and $(R,\lambda)$.
Note that if $\mu:S\rightarrow R$ is an isomorphism of $S$-algebras, then we have $\mu=\lambda$.

To avoid complication, often we avoid mentioning a ring homomorphism $\lambda:S\rightarrow R$ explicitly for an algebra $R$ over $S$. When $S$ is a subring of $R$, we consider the inclusion homomorphism $S\rightarrow R$. When an ideal $I$ of a ring $R$ is given, we consider the canonical surjective homomorphism $R\rightarrow R/I$ to the residue ring $R/I$ and we regard $R/I$ as an $R$-algebra. If $(R,\lambda)$ is an $S$-algebra and $(Q,\kappa)$ is an $R$-algebra, we consider the composition $\kappa\lambda:S\rightarrow Q$ and we regard $Q$ as an $S$-algebra.

Let $R$ be any noetherian ring and let $I$ be any ideal of $R$. We assume either $I$ is contained in any maximal ideal of $R$, or $R$ is an integral domain and $I\neq R$. We call the projective limit
$\varprojlim_{i\in\Z_+} R/I^i$ the \emph{completion} of $R$ with respect to $I$.
It is a ring containing $R$ as a subring. On the completion of $R$ we can define a Hausdorff topology called an $I$\emph{-adic topology}. 

Let $R$ be any noetherian local ring. We call the completion of $R$ with respect to $M(R)$ simply the completion of $R$, and we denote it by $R^c$. $R^c=\varprojlim_{i\in\Z_+} R/M(R)^i$.
The ring $R^c$ is a noetherian local ring, it contains $R$ as a subring, $M(R)=M(R^c)\cap R$, $M(R^c)=M(R)R^c$, $\dim R^c=\dim R$, the induced homomorphism $R/M(R)\rightarrow R^c/M(R^c)$ by the inclusion homomorphism $R\rightarrow R^c$ is an isomorphism, $R^c$ is faithfully flat over $R$, for any prime ideal $\p$ of $R$, there exists a prime ideal $\q$ of $R^c$ satisfying $\p=\q\cap R$, and $R^c=(R^c)^c$.
If, moreover, $R$ is regular, then $R^c$ is also regular, and any parameter system of $R$ is a parameter system of $R^c$.

We say that any noetherian local ring $R$ is $\emph{complete}$, if $R=R^c$.

Consider any complete regular local rings $S$ and $S'$ containing a field $k$ as a subring. Rings $S$ and $S'$ are isomorphic as $k$-algebras, if and only if, $\dim S=\dim S'$ and residue fields $S/M(S)$ and $S'/M(S')$ are isomorphic as $k$-algebras.
Assume that $\dim S=\dim S'$, $\bar{\rho}:S/M(S)\rightarrow S'/M(S')$ is an isomorphism of $k$-algabras, $P$ is a parameter system of $S$, $P'$ is a parameter system of $S'$ and $\sigma:P\rightarrow P'$ is a bijective mapping. Then, there exists uniquely an isomorphism $\rho:S\rightarrow S'$ of $k$-algebras such that the morphism $S/M(S)\rightarrow S'/M(S')$ induced by $\rho$ coincides with $\bar{\rho}$ and $\rho(x)=\sigma(x)$ for any $x\in P$.

See Matsumura~\cite{M}.

Let $k$ be any field. Let $A$ be any complete regular local ring such that $\dim A\geq 1$, $A$ contains $k$ as a subring, and the residue field  $A/M(A)$ is isomorphic to $k$ as algebras over $k$. Let $P$ be any parameter system of $A$. We have $PA=M(A)\supset P$ and $\sharp P=\dim A$.

We fix the above notations $k$, $A$ and $P$ throughout this article. 

Let $\phi$ be any element of $A$. Then, there exists a unique element\hfill\break$c\in \Map(\Map(P,\Z_0),k)$ with
$$\phi=\sum_{\Lambda\in \Map(P,\Z_0)} c(\Lambda)\prod_{x\in P}x^{\Lambda(x)}.$$
The infinite sum in the right-hand side is the limit with respect to the $M(A)$-adic topology on $A$.
We take the unique element $c\in \Map(\Map(P,\Z_0), k)$ satisfying the above  equality. The element $c$ depends on $\phi$ and $P$.
Let $\Lambda\in\Map(P,\Z_0)$. We call $\Lambda$ the \emph{index}, $\prod_{x\in P}x^{\Lambda(x)}\in A$ a \emph{monomial} over $P$,  $c(\Lambda)\in k$ a \emph{coefficient} of $\phi$, $c(\Lambda)\prod_{x\in P}x^{\Lambda(x)}\in A$ a \emph{term} of $\phi$, and $\sum_{x\in P}\Lambda(x)\in\Z_0$ the \emph{degree} of the index $\Lambda$, of the monomial $\prod_{x\in P}x^{\Lambda(x)}$, or of the term $ c(\Lambda)\prod_{x\in P}x^{\Lambda(x)}$. Note that $0\in\Map(P,\Z_0)$. We denote $\phi(0)=c(0)$ and we call $\phi(0)\in k$, \emph{the constant term} of $\phi$. $\phi-\phi(0)\in M(A)$. $\phi(0)=0\Leftrightarrow \phi\in M(A)$.
We denote
$$\Supp(P, \phi)=\Supp(c)=\{\Lambda\in \Map(P,\Z_0)| c(\Lambda)\neq 0\},$$
and we call $\Supp(P,\phi)$ the \emph{support} of $\phi$ over $P$.
It is a subset of $ \Map(P,\Z_0)$.
Note that $\phi=0\Leftrightarrow c=0\Leftrightarrow \Supp(P,\phi)=\emptyset$. 

Below, we consider the case $\phi\neq 0$ for a while. 

We say that $\phi$ has \emph{normal crossings} over $P$, if $\phi=u\prod_{x\in P}x^{\Lambda(x)}$ for some $\Lambda\in\Map(P,\Z_0)$ and some invertible element $u\in A^{\times}$. We say that $\phi$ has \emph{normal crossings}, if  $\phi$ has normal crossings over $Q$ for some parameter system $Q$ of $A$.

We define
$$\Gamma_+(P, \phi)= \Conv(\Supp(P,\phi))+ \Map(P,\R_0),$$
and  call  $\Gamma_+(P, \phi)$ the \emph{Newton polyhedron} of $\phi$ over $P$. By definition we have $\Gamma_+(P, \phi)\subset \Map(P,\R_0)\subset \Map(P,\R)$. We can show that there exists a non-empty \emph{finite} subset $Y$ of $\Supp(P,\phi)$ with $\Gamma_+(P, \phi)=\Conv(Y)+ \Map(P,\R_0)$, and $\Gamma_+(P, \phi)$ is a convex pseudo polyhedron in $\Map(P,\R)$.
(Lemma~\ref{Newton1}, Lemma~\ref{Newton2}.)
By $\mathcal{V}(\Gamma_+(P, \phi))$ we denote the union of all \emph{vertices} (in other words, \emph{faces} with dimension zero. See Definition~\ref{faces of cpp} for the definition of vertices and faces.) of $\Gamma_+(P, \phi)$. By definition we have
\begin{equation*}\begin{split}
\mathcal{V}(\Gamma_+(P, \phi))&=\{a\in\Gamma_+(P, \phi)|
\text{ There exists }\omega\in\Map(P,\R_0)^\vee\text{ such that for any }\\
&\qquad\qquad\quad b\in\Gamma_+(P, \phi)\text{ with }\langle\omega,b\rangle=\langle\omega,a\rangle\text{, we have }b=a\}.\\
\end{split}\end{equation*}
We call $ \mathcal{V}(\Gamma_+(P, \phi))$ the \emph{skeleton} of $\Gamma_+(P, \phi)$. The set $\mathcal{V}(\Gamma_+(P, \phi))$ is a non-empty finite subset of $\Supp(P,\phi)$, and $\Gamma_+(P, \phi)= \Conv(\mathcal{V}(\Gamma_+(P, \phi)))+ \Map(P, \R_0)$.
We denote $c(\Gamma_+(P, \phi))=\sharp\mathcal{V}(\Gamma_+(P, \phi))\in\Z_+$,
and we call $c(\Gamma_+(P, \phi))$ the \emph{characteristic number} of $\Gamma_+(P, \phi)$.

We know that $\Gamma_+(P, \phi)$ has only one vertex$\Leftrightarrow c(\Gamma_+(P, \phi))=1\Leftrightarrow\phi$ has normal crossings over $P$, and that these equivalent conditions always hold, if $\dim A=1$.

Let $\omega\in \Map(P, \R_0)^{\vee}$ be any element. We know that $\{\langle\omega, a\rangle|a\in \Supp(P, \phi)\}$\break $\subset\R_0$,  the minimum element $\min\{\langle\omega, a\rangle|a\in \Supp(P, \phi)\}$ of $\{\langle\omega, a\rangle|a\in \Supp(P, \phi)\}$ exists, and   $\min\{\langle\omega, a\rangle|a\in \Supp(P, \phi)\}= \min\{\langle\omega, a\rangle|a\in \mathcal{V}(\Gamma_+(P, \phi))\}$.
We define
\begin{equation*}\begin{split}
\Ord(P,\omega,\phi)&=\min\{\langle\omega, a\rangle|a\in \Supp(P, \phi)\}\in\R_0,\\
\Supp(P, \omega, \phi)&=\{a\in \Supp(P, \phi)| \langle\omega, a\rangle=\Ord(P,\omega,\phi)\}\subset \Supp(P, \phi),\\
\In(P,\omega,\phi)&=
\sum_{\Lambda\in \Supp(P, \omega, \phi)} c(\Lambda)\prod_{x\in P}x^{\Lambda(x)}\in A.\\
\end{split}\end{equation*}

We consider the case $\phi=0$. We introduce a symbol $\infty$ satisfying the following conditions: for any $t\in\R$, we have $\infty>t, \infty\geq t, \infty\neq t, t<\infty, t\leq\infty, t\neq\infty, \infty+t=t+\infty=\infty$, and moreover $\infty+\infty=\infty$. 
Let $\omega\in \Map(P, \R_0)^{\vee}$ be any element. We define
\begin{equation*}\begin{split}
\Ord(P,\omega,0)&= \infty,\\
\In(P,\omega,0)&=0.\\
\end{split}\end{equation*}

Let $\omega\in \Map(P, \R_0)^{\vee}$ be any element. 
In the general case including the case of $\phi=0$, we have defined $\Ord(P,\omega,\phi)\in\R_0\cup\{\infty\}$ and $\In(P,\omega,\phi)\in A$.
We call $\Ord(P,\omega,\phi)\in\R_0\cup\{\infty\}$ the \emph{order} of $\phi$ over $P$ with respect to $\omega$. By definition $\Ord(P,\omega,\phi)=\infty$ if and only if $\phi=0$.
We call $\In(P,\omega,\phi)$ the \emph{initial sum} of $\phi$ over $P$ with respect to $\omega$. By definition $\In(P,\omega,\phi)=0$ if and only if $\phi=0$.

Let $F$ be any subset of $\Map(P,\R)$. We denote
$$\Ps(P,F,\phi)=
\begin{cases}
\sum_{\Lambda\in \Supp(P, \phi)\cap F} c(\Lambda)\prod_{x\in P}x^{\Lambda(x)}
&\text{if $ \Supp(P, \phi)\cap F \neq \emptyset$},\\
0&\text{if $ \Supp(P, \phi)\cap F =\emptyset$},
\end{cases}$$
and we call $\Ps(P,F,\phi)\in A$ the \emph{partial sum} of $\phi$ over $P$ with respect to $F$.

Here we assume $\dim A\geq 2$ for a while. Again, we assume $\phi\neq 0$. In addition, let $z$ be any element of $P$. 

Note that for any $a\in \mathcal{V}(\Gamma_+(P, \phi))$, we have $\langle f^{P\vee}_z, a\rangle\in\Z_0$. We define
\begin{equation*}\begin{split}
&\Ht(z,\Gamma_+(P, \phi)) \\
= &\max\{\langle f^{P\vee}_z, a\rangle|a\in \mathcal{V}(\Gamma_+(P, \phi))\}
-\min\{\langle f^{P\vee}_z, a\rangle|a\in \mathcal{V}(\Gamma_+(P, \phi))\}\in\Z_0,
\end{split}\end{equation*}
and we call $\Ht(z,\Gamma_+(P, \phi)) $ the \emph{height} of $\Gamma_+(P, \phi)$ with respect to $z$, or simply $z$-\emph{height} of  $\Gamma_+(P, \phi)$. It is a non-negative integer. By definition, $\Ht(z,\Gamma_+(P, \phi)) =0$ if and only if the value $\langle f^{P\vee}_z, a\rangle$ does not depend on $a\in \mathcal{V}(\Gamma_+(P, \phi))$.

Let $a\in \mathcal{V}(\Gamma_+(P, \phi))$. We say that $\{a\}$ is a $z$-\emph{top vertex} of $\Gamma_+(P,\phi)$, if $\langle f^{P\vee}_z, a\rangle= \max\{\langle f^{P\vee}_z, b\rangle|b\in \mathcal{V}(\Gamma_+(P, \phi))\}$. We say that $\{a\}$ is a $z$-\emph{bottom vertex} of $\Gamma_+(P,\phi)$, if $\langle f^{P\vee}_z, a\rangle= \min\{\langle f^{P\vee}_z, b\rangle|b\in \mathcal{V}(\Gamma_+(P, \phi))\}$.

Any element $\phi$ in $A$ such that $\phi=z^h+\sum_{i=0}^{h-1} \phi'(i) z^i$ for some $h\in\Z_0$ and some mappng $\phi':\{0, 1,\ldots, h-1\}\rightarrow M(A')$ is called a $z$-\emph{Weierstrass polynomial} over $P$, and the integer $h$ is called \emph{degree} of $\phi$.

We say that $\Gamma_+(P,\phi)$ is \emph{of $z$-Weierstrass type}, if there exists $a\in \Gamma_+(P, \phi)$ satisfying the equality $\langle f^{P\vee}_x, a\rangle=\Ord(P, f^{P\vee}_x,\phi)$ for any $x\in P - \{z\}$.

Let $b=\Ord(P, f^{P\vee}_z,\phi)\in\Z_0$ and let $h=\Ht(z,\Gamma_+(P, \phi)) \in\Z_0$.
Let $A'$ denote the completion of $k[P - \{z\}]$ with respect to the maximal ideal $k[P - \{z\}]\cap M(A)$. The ring $A'$ is a local subring of $A$ and $M(A')=M(A)\cap A' =(P -\{z\})A'$. The completion of $A'[z]$ with respect to the prime ideal $zA'[z]$ is isomorphic to $A$ as $A'[z]$-algebras. The set $P-\{z\}$ is a parameter system of $A'$.

Under the assumption that $\Gamma_+(P,\phi)$ is of $z$-Weierstrass type, by Weierstrass'  preparation theorem we know the following (Lemma~\ref{Newton2}.8):
\begin{enumerate}
\item
$\Ht(z,\Gamma_+(P, \phi))=0\Leftrightarrow \Gamma_+(P,\phi)$ has only one vertex $\Leftrightarrow c(\Gamma_+(P, \phi))=1\Leftrightarrow \phi$ has normal crossings over $P$.
\item 
The Newton polyhedron $\Gamma_+(P,\phi)$ has a unique $z$-top vertex.
\end{enumerate}

Below, by $\{a_1\}$ we denote the unique $z$-top vertex of $\Gamma_+(P,\phi)$.
\begin{enumerate}
\setcounter{enumi}{2}
\item
Consider any $a\in \Gamma_+(P, \phi)$. The equality $\langle f^{P\vee}_x, a\rangle=\Ord(P, f^{P\vee}_x,\phi)$ holds for any $x\in P - \{z\}\Leftrightarrow a-a_1\in\R_0f^P_z$.
\item
$\langle f^{P\vee}_z, a_1\rangle=b+h$.
\item
There exist uniquely an invertible element $u\in A^{\times}$ and a mapping $\phi':\{0, 1,\ldots, h-1\}\rightarrow M(A')$ satisfying
$$\phi=u z^b\prod_{x\in P -\{z\}}x^{\langle f^{P\vee}_x, a_1\rangle} (z^h+\sum_{i=0}^{h-1} \phi'(i) z^i),$$
and $\phi'(0)\neq 0$ if $h>0$
\end{enumerate}

Under the above notations we know that the following two conditions are equivalent (Lemma~\ref{Newton2}.9):
\begin{enumerate}
\item The Newton polyhedron $\Gamma_+(P,\phi)$ is of $z$-Weierstrass type.
\item There exist uniquely an invertible element $u\in A^{\times}$, a mapping $c:P\rightarrow\Z_0$, a non-negative integer $g\in\Z_0$ and a mapping $\phi':\{0, 1,\ldots, g-1\}\rightarrow M(A')$ satisfying
$$\phi=u \prod_{x\in P}x^{c(x)} (z^g+\sum_{i=0}^{g-1} \phi'(i) z^i),$$
and $\phi'(0)\neq 0$ if $g>0$.
\end{enumerate}

The concept of $z$-removable faces is very important.

Assume that $\Gamma_+(P,\phi)$ is of $z$-Weierstrass type. Let $b$, $h$ and $A'$ be the same as above. By $\{a_1\}$ we denote the unique $z$-top vertex of $\Gamma_+(P,\phi)$. Assume moreover that $h>0$. Under these assumptions we can give the definition of $z$-removable faces. 

A subset $F$ of  $\Map(P,\R)$ is a \emph{face} of $\Gamma_+(P,\phi)$, if and only if, there exists $\omega \in \Map(P,\R_0)^\vee$ such that $F=\{a \in \Gamma_+(P,\phi)|\langle\omega,a\rangle=\Ord(P,\omega,\phi)\}$. Any face of $\Gamma_+(P,\phi)$ is a non-empty closed subset of $\Gamma_+(P,\phi)$, and is a convex pseudo polyhedron. (See Definition~\ref{faces of cpp}.)

Let $F$ be a face of $\Gamma_+(P,\phi)$. We say that $F$ is $z$-\emph{removable}, if $a_1\in F$ and there exist an invertible element $u\in A^{\times}$ and an element $\chi\in M(A')$ satisfying $\chi\neq 0$ and
$$\Ps(P,F,\phi)=u z^b\prod_{x\in P -\{z\}}
x^{\langle f^{P\vee}_x, a_1\rangle}(z+\chi)^h.$$

We would like to explain the relation betwen the concept of $z$-removable faces and Hironaka's maximal contact here. We assume that the field $k$ has characteristic zero, and consider any $z$-Weierstrass polynomial $\psi\in A$ of positive degree. We take the unique pair of a positive integer $h$ and 
a mapping $\psi':\{0,1,\ldots,h-1\}\rightarrow M(A')$
satisfying the equality $\psi=z^h+\sum_{i=0}^{h-1} \psi'(i) z^i$. Let $\hat{z}=z+(\psi'(h-1)/h)\in M(A)$ and let $\hat{P}=\{\hat{z}\}\cup(P-\{z\})$. We know that $\hat{P}$ is a parameter system of $A$ and the Newton polyhedron $\Gamma_+(\hat{P},\psi)$ is of $\hat{z}$-Weierstrass type and has no $\hat{z}$-removable faces. Now, we assume moreover that $\psi$ has multiplicity $h$, in other words, $\psi\in M(A)^h- M(A)^{h+1}$. This condition is equivalent to that $\psi'(i)\in M(A')^{h-i}$ for any $i\in\{0,1,\ldots,h-1\}$. We know that the smooth subscheme $\Spec(A/\hat{z}A)$ of $\Spec(A)$ is Hironaka's maximal contact of the subscheme $\Spec(A/\psi A)$. (See Giraud~\cite{G75}.)

Note that we cannot define the element $\hat{z}=z+(\psi'(h-1)/h) \in A$, if the characteristic of $k$ is positive and the characteristic divides $h$.

The concept of $z$-simple is also very important.

We say that $\Gamma_+(P,\phi)$ is \emph{$z$-simple}, if $\Gamma_+(P,\phi)$ is of $z$-Weierstrass type and any compact face $F$ of $\Gamma_+(P,\phi)$ satisfies $\dim F\leq 1$.

If $\dim A=2$, then always $\Gamma_+(P,\phi)$ is $z$-simple.
If $\Gamma_+(P,\phi)$ is $z$-simple, then $\Gamma_+(P,\phi)$ is of $z$-Weierstrass type.

\section{Basic scheme theory}
\label{scheme}

We develop the basic scheme theory. By $k$ we denote any field.

Let $\Sigma$ be a scheme. Any pair $(\Gamma, \gamma)$ where $\Gamma$ is a scheme and $\gamma:\Gamma\rightarrow\Sigma$ is a morphism of schemes is called a \emph{scheme over} $\Sigma$, or a $\Sigma$\emph{-scheme}, and $\gamma$ is called the \emph{structure morphism} of $\Sigma$-scheme $\Gamma$. 
Let $(\Gamma, \gamma)$ and $(\Gamma', \gamma')$ be $\Sigma$-schemes. 
A morphism $\tau:\Gamma\rightarrow\Gamma'$ of schemes satisfying $\gamma'\tau=\gamma$ is called a \emph{morphism over} $\Sigma$, a $\Sigma$\emph{-morphism} or a \emph{morphism of} $\Sigma$\emph{-schemes}.
An isomorphism $\tau:\Gamma\rightarrow\Gamma'$ of schemes satisfying $\gamma'\tau=\gamma$ is called an \emph{isomorphism over} $\Sigma$, a $\Sigma$\emph{-isomorphism} or an \emph{isomorphism of} $\Sigma$\emph{-schemes}.
We say that two $\Sigma$-schemes $(\Gamma, \gamma)$ and $(\Gamma', \gamma')$ are \emph{isomorphic}, if there exists an isomorphism $\tau:\Gamma\rightarrow\Gamma'$ over $\Sigma$.

In case where a ring $R$ is given, we say that a scheme over $R$, an $R$-scheme, a morphism over $R$, an $R$-morphism, a morphism of $R$-schemes, an isomorphism over $R$, an $R$-isomorphism, an isomorphism of $R$-schemes, instead of, a scheme over $\Spec(R)$, a $\Spec(R)$-scheme, a morphism over $\Spec(R)$, a $\Spec(R)$-morphism, a morphism of $\Spec(R)$-schemes, an isomorphism over $\Spec(R)$, a $\Spec(R)$-iso-\hfill\break morphism, an isomorphism of $\Spec(R)$-schemes, respectively.

Let $\Sigma$ be a scheme, and let $\A\in\Sigma$ be a point. For any open subset $U$ of $\Sigma$ containing $\A$ we have the restriction homomorphism $\mathcal{O}_\Sigma(U)\rightarrow \mathcal{O}_{\Sigma(U),\A}$ from the ring of regular functions $\mathcal{O}_\Sigma(U)$ on $U$ to the local ring $\mathcal{O}_{\Sigma,\A}$ of $\Sigma$ at $\A$. The restriction homomorphisms define the \emph{canonical morphism} $\Spec(\mathcal{O}_{\Sigma,\A})\rightarrow\Sigma$ of schemes.

Let $\Sigma$ be an irreducible scheme. There exists a unique point $\A\in\Sigma$ such that $\{\A\}$ is dense in $\Sigma$. The unique point satisfying this condition is denoted by $[\Sigma]$, and we call $[\Sigma]$ the \emph{generic point} of $\Sigma$. If moreover, $\Sigma$ is irreducible and reduced, then the local ring $\mathcal{O}_{\Sigma,[\Sigma]}$ of $\Sigma$ at $[\Sigma]$ is a field, and we call $\mathcal{O}_{\Sigma,[\Sigma]}$ the \emph{function field} of $\Sigma$. 

Let $\Gamma$ and $\Sigma$ be an irreducible schemes, $\gamma:\Gamma\rightarrow\Sigma$ a morphism of schemes. We say that $\gamma$ is \emph{dominant}, if the image $\gamma(\Gamma)$ is dense in $\Sigma$. If $\gamma$ is dominant, then for any non-empty open set $V$ of $\Gamma$ and any non-empty open set $U$ of $\Sigma$ with $\gamma(V)\subset U$ the homomorphism $\gamma^*:\mathcal{O}_\Sigma(U)\rightarrow\mathcal{O}_\Gamma(V)$ is injective, and an injective homomorphism $\gamma^*:\mathcal{O}_{\Sigma,[\Sigma]}\rightarrow
\mathcal{O}_{\Gamma,[\Gamma]}$ between the local rings at the generic point is induced. We say that $\gamma$ is \emph{birational}, if it is dominant and the induced homomorphism $\gamma^*:\mathcal{O}_{\Sigma,[\Sigma]}\rightarrow
\mathcal{O}_{\Gamma,[\Gamma]}$ is an isomorphism. For any $\A\in\Sigma$
the canonical morphism $\Spec(\mathcal{O}_{\Sigma,\A})\rightarrow\Sigma$ is dominant and birational, since $\Sigma$ is irreducible.

A scheme over $k$ which is separated, irreducible, reduced and of finite type over $k$ is called a \emph{variety over} $k$, or $k$\emph{-variety}. Any $k$-variety is a noetherian scheme.

Let $\Sigma$ be any $k$-scheme, and let $\A\in \Sigma$ be any point. Note that the structure morphism defines an injective ring homomorphism $k\rightarrow\mathcal{O}_{\Sigma,\A}$, and $\mathcal{O}_{\Sigma,\A}$ is a $k$-algebra. The point $\A\in \Sigma$ is called $k$\emph{-valued}, if the residue field $\mathcal{O}_{\Sigma,\A}/M(\mathcal{O}_{\Sigma,\A})$ is isomorphic to $k$ as $k$-algebras. The set of all $k$-valued points on $\Sigma$ is denoted by $\Sigma(k)$. $\Sigma(k)\subset \Sigma$, and the topology on $\Sigma$ defines the relative topology on $\Sigma(k)$. For any $\A\in \Sigma(k)$ and any $\phi\in \mathcal{O}_{\Sigma,\A}$ we can define the value $\phi(\A)\in k$ of a function $\phi$ at a point $\A$ belonging to $k$, and $\phi-\phi(\A)\in M(\mathcal{O}_{\Sigma,\A})$.

 If the local ring $\mathcal{O}_{\Sigma,\A}$ of a scheme $\Sigma$ at a point $\A\in\Sigma$ is noetherian and regular, then we say that $\Sigma$ is \emph{smooth} at $\A\in \Sigma$. We say that a scheme $\Sigma$ is smooth, if $\Sigma$ is \emph{smooth} at any point $\A\in \Sigma$.

 Let $\Sigma$ be any scheme, and let $\mathcal{I}$ be any ideal sheaf in the structure sheaf $\mathcal{O}_\Sigma$, in other words, any sheaf of $\mathcal{O}_\Sigma$-modules which is a subsheaf of $\mathcal{O}_\Sigma$. The ideal sheaf $\mathcal{I}$ is called \emph{locally principal}, if for any $\A\in\Sigma$ there exists $\phi\in \mathcal{O}_{\Sigma,\A}$ such that $\phi$ is not a zero-divisor of $\mathcal{O}_{\Sigma,\A}$ and $\mathcal{I}_\A=\phi \mathcal{O}_{\Sigma,\A}$, where $\mathcal{I}_\A$ denotes the stalk of $\mathcal{I}$ at $\A$. Note that for any scheme $\Gamma$ and for any morphism $\gamma:\Gamma\rightarrow\Sigma$ of schemes, the pull-back $\gamma^*\mathcal{I}$ of $\mathcal{I}$ as an ideal sheaf is defined, and $\gamma^*\mathcal{I}$ is a sheaf of $\mathcal{O}_{\Gamma}$-modules which is a subsheaf of $\mathcal{O}_{\Gamma}$.

Grothendieck showed that there exists a scheme $\Sigma'$ and a morphism $\sigma:\Sigma'\rightarrow\Sigma$ satisfying the following universal mapping property:  
\begin{enumerate}
\item The ideal sheaf $\sigma^*\mathcal{I}$ is locally principal. 
\item If $\Gamma$ is a scheme, $\gamma:\Gamma\rightarrow\Sigma$ is a morphism, and the ideal sheaf $\gamma^*\mathcal{I}$ is locally principal, then there exists a unique morphism $\tau:\Gamma\rightarrow\Sigma'$ with $\sigma\tau=\gamma$.
\end{enumerate}

By the universal mapping property we know that the pair $(\Sigma', \sigma)$ satisfying the above conditions is unique up to isomorphism of schemes over $\Sigma$. The pair $(\Sigma', \sigma)$ satisfying the above conditions is called the \emph{blowing-up with center in} an ideal sheaf $\mathcal{I}$, or the \emph{blowing-up with center in} $\Phi$, where $\Phi$ denotes the closed subscheme of $\Sigma$ defined by the ideal sheaf $\mathcal{I}$. Note that any closed subscheme of $\Sigma$ has a unique ideal sheaf in $\mathcal{O}_{\Sigma}$ defining it. If $\mathcal{I}$ is locally principal, then $\sigma$ is an isomorphism.

Let $(\Sigma', \sigma)$ be the  blowing-up with center in $\mathcal{I}$. By $\Phi$ we denote the closed subscheme of $\Sigma$ defined by the ideal sheaf $\mathcal{I}$. We call the inverse image $\sigma^{-1}(\Phi)$ the \emph{exceptional divisor} of $\sigma$. By Grothendieck's description we know moreover the following: 
\begin{enumerate}
\item The morphism $\sigma$ is surjective. The exceptional divisor of $\sigma$ is a subscheme of $\Sigma'$ of codimension one defined by the locally principal ideal sheaf $\sigma^*\mathcal{I}$, and the induced morphism $\sigma:\Sigma'-\sigma^{-1}(\Phi)\rightarrow\Sigma-\Phi$ is an isomorphism. 
\item If $\Sigma$ is locally noetherian, then $\sigma$ is proper. 
\item If $\Sigma$ is separated, then $\Sigma'$ is also separated. If $\Sigma$ is noetherian, then $\Sigma'$ is also noetherian. 
\item If $\Sigma$ is irreducible, then $\Sigma'$ is also irreducible and $\sigma$ is birational. 
\item If $\Sigma$ is a $k$-variety, then $\Sigma'$ is also a $k$-variety. 
\item If $\Sigma$ and $\Phi$ are smooth, then $\Sigma'$ and $\sigma^{-1}(\Phi)$ are also smooth. If $\Sigma$ and $\Phi$ are smooth and irreducible, then $\Sigma'$ and $\sigma^{-1}(\Phi)$ are also smooth and irreducible.
\end{enumerate}

Let $\Sigma$ be any separated irreducible noetherian smooth scheme with $\dim\Sigma\geq 1$. We denote
\begin{equation*}
\begin{split}
\Prm(\Sigma)&=\text{ The set of non-empty closed irreducible reduced subschemes}\\
&\qquad\qquad \text{of codimension one of }\Sigma,\\
\Div(\Sigma) &= \Map'(\Prm(\Sigma),\Z)=\{\Delta\in\Map(\Prm(\Sigma),\Z)|\Supp(\Delta)\text{ is a finite set.}\},\\
\Div(\Sigma)_0&=\Map'(\Prm(\Sigma),\Z_0) =\{\Delta\in\Map(\Prm(\Sigma),\Z_0)|\Supp(\Delta)\text{ is a finite set.}\}, \\
\Div(\Sigma)'_0&=\text{ The union of the set of non-empty closed subschemes}\\
&\qquad\qquad\text{of codimension one of }\Sigma\text{ and }\{\text{the empty subscheme}\}.\\
\end{split}
\end{equation*}

We know that $\Div(\Sigma)$ is a subgroup of $\Map(\Prm(\Sigma),\Z)$, and $\Div(\Sigma)_0$ is a semisubgroup of $\Div(\Sigma)$ containing $0$ and generating $\Div(\Sigma)$. We call $\Div(\Sigma)$ the \emph{divisor group} of $\Sigma$. Any element of $\Div(\Sigma)$ is called a \emph{divisor} of $\Sigma$, any element of  $\Div(\Sigma)_0$ is called an \emph{effective} divisor of $\Sigma$, and any element of  $\Prm(\Sigma)$ is called a \emph{prime} divisor of $\Sigma$. Let $\Delta$ be a divisor of $\Sigma$. Any element $\Lambda\in \Prm(\Sigma)$ with $\Delta(\Lambda)\neq 0$ is called a \emph{component} of $\Delta$. The set of all components of a divisor $\Delta$ is denoted by $\Comp(\Delta)$, which is a finite set of prime divisors. For any $\Lambda\in \Prm(\Sigma)$ the integer  $\Delta(\Lambda)$ is called the \emph{multiplicity} of $\Lambda$ in $\Delta$. The divisor $\Delta$ is effective, if and only if, for any prime divisor $\Lambda$ of $\Sigma$ the multiplicity $\Delta(\Lambda)$ of $\Lambda$ in $\Delta$ is non-negative.
The union of all components of $\Delta$ is denoted by $\Supp(\Delta)$, and we call $\Supp(\Delta)$ the \emph{support} of $\Delta$.

For any $\Gamma\in\Div(\Sigma)'_0$, the ideal sheaf in $\mathcal{O}_\Sigma$ defining $\Gamma$ is locally principal. Conversely, for any locally principal ideal sheaf $\mathcal{I}$ in $\mathcal{O}_\Sigma$, there exists a unique element $\Gamma\in\Div(\Sigma)'_0$ whose defining ideal sheaf is equal to $\mathcal{I}$. Thus, the set $\Div(\Sigma)'_0$ is identified with the set of locally principal ideal sheaves in $\mathcal{O}_\Sigma$. The empty subscheme in $\Div(\Sigma)'_0$ is identified with $\mathcal{O}_\Sigma$ itself.

We have a unique one-to-one correspondence $\Phi: \Div(\Sigma)'_0\rightarrow\Div(\Sigma)_0$ satisfying the following conditions: Let $\Gamma\in\Div(\Sigma)'_0$. We write $\Delta=\Phi(\Gamma)\in\Div(\Sigma)_0$. By $\mathcal{I}$ we denote the ideal sheaf in $\mathcal{O}_\Sigma$ defining $\Gamma$. For any $\Lambda\in\Prm(\Sigma)$ by $\mathcal{J}_\Lambda$ we denote the ideal sheaf in $\mathcal{O}_\Sigma$ defining $\Lambda$. Then, we have
$\mathcal{I}=\prod_{\Lambda\in\Comp(\Delta)}\mathcal{J}_\Lambda^{\Delta(\Lambda)}$.
If $\Gamma\in\Div(\Sigma)'_0$ and $\Delta=\Phi(\Gamma)\in\Div(\Sigma)_0$, then the set of irreducible reduced components of $\Gamma$ is equal to $\Comp(\Delta)$, and $\Gamma_\mathrm{red}=\Supp(\Delta)$, where $\Gamma_\mathrm{red}$ denotes the reduced subscheme corresponding to $\Gamma$.
Using  $\Phi: \Div(\Sigma)'_0\rightarrow\Div(\Sigma)_0$ we identify $ \Div(\Sigma)'_0$ and $\Div(\Sigma)_0$. The empty scheme is identified with $0$.

Let $\Gamma$ and $\Sigma$ be separated irreducible noetherian smooth schemes with $\dim\Gamma\geq 1$ and $\dim\Sigma\geq 1$, and let $\gamma:\Gamma\rightarrow \Sigma$ be a dominant morphism.
Since $\gamma$ is dominant, we know that the pull-back $\gamma^*\mathcal{I}$ of any locally principal ideal sheaf $\mathcal{I}$ in $\mathcal{O}_\Sigma$ by $\gamma$ is a locally principal ideal sheaf in $\mathcal{O}_\Gamma$. Thus, we can define a semigroup homomorphism
$$\gamma^*:\Div(\Sigma)_0\rightarrow\Div(\Gamma)_0$$
such that if $\Delta\in\Div(\Sigma)_0$, and the ideal sheaf defining $\Delta$ is $\mathcal{I}$, then the ideal sheaf $\gamma^*\mathcal{I}$ defines  $\gamma^*\Delta\in\Div(\Gamma)_0$. We know $\gamma^*0=0$. We have a unique group homomorphism
$$\gamma^*:\Div(\Sigma)\rightarrow\Div(\Gamma),$$
extending  $\gamma^*:\Div(\Sigma)_0\rightarrow\Div(\Gamma)_0$. For any divisor $\Delta\in\Div(\Sigma)$, the divisor $\gamma^*\Delta\in\Div(\Gamma)$ is called the \emph{pull-back} of $\Delta$ by $\gamma$. If $\Delta$ is effective, then $\gamma^*\Delta$ is also effective. 

Let $\Gamma$ and $\Sigma$ be separated irreducible noetherian smooth schemes with $\dim\Gamma\geq 1$ and $\dim\Sigma\geq 1$; and let $\gamma:\Gamma\rightarrow \Sigma$ be a surjective birational morphism, and let $\Lambda$ be any prime divisor of $\Sigma$. Since $\gamma$ is surjective and birational, we have a unique component $\Lambda'$ of $\gamma^*\Lambda$ with $\gamma(\Lambda')= \Lambda$. This unique component $\Lambda'$ is called the \emph{strict transform} of $\Lambda$ by $\gamma$. The multiplicity of $\Lambda'$ in $\gamma^*\Lambda$ is always equal to one.

Let $\Sigma$ be a separated irreducible noethrian smooth scheme with $\dim\Sigma\geq 1$; let $\Delta$ be an effective divisor of $\Sigma$; and let $\A\in\Sigma$ be a point. Since $\Sigma$ is irreducible, the canonical morphism $\Spec(\mathcal{O}_{\Sigma,\A})\rightarrow\Sigma$ is dominant. The inclusion homomorphism from $\mathcal{O}_{\Sigma,\A}$ to its completion $\mathcal{O}_{\Sigma,\A}^c$ is faithfully flat, and the induced morphism $\Spec(\mathcal{O}_{\Sigma,\A}^c)\rightarrow\Spec(\mathcal{O}_{\Sigma,\A})$ of affine schemes is surjective. Their composition $\delta:\Spec(\mathcal{O}_{\Sigma,\A}^c)\rightarrow\Sigma$ is defined, it is dominant, and the pull-back $\delta^*\Delta$ of $\Delta$ by this composition morphism $\delta$ is defined.
We say that $\Delta$ has \emph{normal crossings at} $\A\in\Sigma$, if there exist a parameter system $P$ of the completion $\mathcal{O}_{\Sigma,\A}^c$ of the local ring $\mathcal{O}_{\Sigma,\A}$ of $\Sigma$ at $\A$, and an element $\Lambda\in\Map(P,\Z_0)$ such that $$\delta^*\Delta=
\Spec(\mathcal{O}_{\Sigma,\A}^c/\prod_{x\in P}x^{\Lambda(x)}\mathcal{O}_{\Sigma,\A}^c).$$
We say that $\Delta$ has \emph{normal crossings} or $\Delta$ is a \emph{normal crossing divisor}, if it has normal crossings at any point of $\Sigma$.

Here we give the definition of the concept of \emph{normal crossing schemes} over an algebraically closed field  and introduce some notations associated with it. We assume that the field $k$ is algebraically closed below in this section.

A pair
$$(\Sigma, \Delta),$$
satisfying the following five conditions is called a  \emph{normal crossing scheme} over $k$. 
\begin{enumerate}
\item The first item $\Sigma$ is a separated irreducible noetherian smooth scheme over $k$ with $\dim\Sigma\geq 1$ such that any closed point $\A\in\Sigma$ is a $k$-valued point.
\item The second item $\Delta$ is a non-zero effective normal crossing divisor of $\Sigma$.
\end{enumerate}

We use the following notations: The set of components of $\Delta$ is denoted by $\Comp(\Delta)$. For any point $\A\in\Sigma$ we denote $\Comp(\Delta)(\A)=\{\Lambda\in \Comp(\Delta)|\A\in \Lambda\}$, and $(\Delta)_0=\{\A\in\Sigma|\sharp\Comp(\Delta)(\A)=\dim\Sigma\}$. For any $\A\in(\Delta)_0$ we write
$$U(\Sigma,\Delta,\A)=\Sigma-(\bigcup_{\Lambda\in \Comp(\Delta)- \Comp(\Delta)(\A)}\Lambda).$$
We write simply $U(\A)$, instead of $U(\Sigma,\Delta,\A)$, when we need not refer to the pair $(\Sigma,\Delta)$.

\begin{enumerate}
\setcounter{enumi}{2}
\item For any non-empty subset $Q$ of $\Comp(\Delta)$ with $\bigcap_{\Lambda\in Q}\Lambda\neq\emptyset$, the intersection scheme $\bigcap_{\Lambda\in Q}\Lambda$ is irreducible and smooth.
\item For any non-empty subset $Q$ of $\Comp(\Delta)$ with $\bigcap_{\Lambda\in Q}\Lambda\neq\emptyset$, there exists $\A\in(\Delta)_0$ such that $Q\subset\Comp(\Delta)(\A)$.
\item For any $\A\in(\Delta)_0$, $U(\A)$ is an affine open subset of $\Sigma$.
\end{enumerate}

Let $(\Sigma, \Delta)$ be a normal crossing scheme over $k$. We call $\Sigma$ the \emph{support} of $(\Sigma, \Delta)$. For any $\A\in(\Delta)_0$, we consider a mapping
$$\xi_\A:\Comp(\Delta)(\A)\rightarrow \mathcal{O}_\Sigma(U(\A)).$$
Let $\A\in(\Delta)_0$. If $\xi_\A$ satisfies the following two conditions, then we call $\xi_\A$ a \emph{coordinate system} of $(\Sigma, \Delta)$ at $\A$:

\begin{enumerate} 
\item For any $\Lambda\in \Comp(\Delta)(\A)$ we have
$$\Lambda\cap U(\A)=
\Spec(\mathcal{O}_\Sigma(U(\A))/\xi_\A(\Lambda)\mathcal{O}_\Sigma(U(\A))).$$
\item 
For any $k$-valued point $\beta\in U(\A)(k)$, the set $\{\xi_{\A}(\Lambda)- \xi_{\A}(\Lambda)(\B)|\Lambda\in\Comp(\Delta)(\A)\}$ is a parameter system of the local ring $\mathcal{O}_{U(\A),\B}$.
\end{enumerate}

If $\xi_\A$ is a coordinate system of $(\Sigma, \Delta)$ at $\A$ for any $\A\in(\Delta)_0$, then we call the collection $\xi=\{\xi_\A|\A\in(\Delta)_0\}$ a \emph{coordinate system} of $(\Sigma, \Delta)$. For a coordinate system $\xi$ of $(\Sigma, \Delta)$ we denote the element of $\xi$ corresponding to $\A\in(\Delta)_0$ by $\xi_\A$.

A triplet $(\Sigma, \Delta, \xi)$ such that the pair $(\Sigma,\Delta)$ is a normal crossing scheme over $k$ and $\xi$ is a coordinate system of $(\Sigma,\Delta)$ is called a \emph{coordinated normal crossing scheme} over $k$.

\begin{example}
Let $A$ be any complete regular local ring such that $A$ contains $k$ as a subring, the residue field $A/M(A)$ is isomorphic to $k$ as $k$-algebras, and $\dim A\geq 1$; let $P$ be any parameter system of $A$; and let $\Lambda\in\Map(P,\Z_+)$.

Note that $M(A)\in\Spec(A)$ and $M(A)$ is the unique closed point of $\Spec(A)$. Let $\Delta=\Spec(A/\prod_{x\in P}x^{\Lambda(x)}A)$. The pair $(\Spec(A),\Delta)$ is a normal crossing scheme over $k$. We have $(\Delta)_0=\{M(A)\}$, $\Comp(\Delta)=\Comp(\Delta)(M(A))=\{\Spec(A/xA)|x\in P\}$, and $U(\Spec(A),\Delta, M(A))=\Spec(A)$. For any $x\in P$, we put $\xi_{M(A)}(\Spec(A/xA$\break$))=x$. We obtain a mapping $\xi_{M(A)}: \Comp(\Delta)(M(A))\rightarrow\mathcal{O}_{\Spec(A)}( U(\Spec(A),\Delta, $\break$M(A)))$.
The mapping $\xi_{M(A)}$ is a coordinate system of $(\Spec(A),\Delta)$ at $M(A)$, and the triplet $(\Spec(A),\Delta,\{\xi_{M(A)}\})$ is a coordinated normal crossing scheme over $k$.

Note that $P$ is algebraically independent over $k$. We consider the subring $k[P]$ of $A$. We denote $M_0= k[P]\cap M(A)$. Note that $M_0=Pk[P]$, $M_0\in\Spec(k[P])$ and $M_0$ is a closed point of $\Spec(k[P])$. Let $\bar{\Delta}=\Spec(k[P]/\prod_{x\in P}x^{\Lambda(x)} k[P])$. The pair $(\Spec(k[P]), \bar{\Delta})$ is a normal crossing scheme over $k$. We have $(\bar{\Delta})_0=\{M_0\}$, $\Comp(\bar{\Delta})=\Comp(\bar{\Delta})(M_0)=\{\Spec(k[P]/xk[P])|x\in P\}$, and $U(\Spec(k[P]), \bar{\Delta}, $\break$M_0)=\Spec(k[P])$. For any $x\in P$, we put $\xi_{M_0}(\Spec(k[P]/xk[P]))=x$. We obtain a mapping $\xi_{M_0}: \Comp(\bar{\Delta})(M_0)\rightarrow\mathcal{O}_{\Spec(k[P])}( U(\Spec(k[P]), \bar{\Delta}, M_0))$.
The mapping $\xi_{M_0}$ is a coordinate system of $(\Spec(k[P]), \bar{\Delta})$ at $M_0$, and the triplet $(\Spec(k[P]), \bar{\Delta},\{\xi_{M_0}\})$ is a coordinated normal crossing scheme over $k$.
\end{example}

The four lemmas below easily follow from definitions.

\begin{lemma}
\label{ncv}
Let $(\Sigma, \Delta)$ be a normal crossing scheme over $k$.
\begin{enumerate}
\item
The set $\Comp(\Delta)$ is non-empty and finite.
\item
For any non-empty subset $Q$ of $\Comp(\Delta)$ with $\bigcap_{\Lambda\in Q}\Lambda\neq\emptyset$, $\bigcap_{\Lambda\in Q}\Lambda$ is a closed irreducible smooth subscheme of $\Sigma$, and $\dim\bigcap_{\Lambda\in Q}\Lambda=\dim\Sigma-\sharp Q$.
\item
The set $(\Delta)_0$ is a non-empty finite set of $k$-valued points of $\Sigma$.
\item
For $\A\in(\Delta)_0$ and $\B\in(\Delta)_0$, $\Comp(\Delta)(\A)= \Comp(\Delta)(\B)$, if and only if, $\A=\B$.
\item
$$\Sigma=\bigcup_{\A\in(\Delta)_0}U(\A).$$
\item
For any open set $U$ of $\Sigma$ with $\Sigma(k)\subset U$, we have $U=\Sigma$.
\end{enumerate}

Let $Q$ be any subset of $\Comp(\Delta)$ with $\sharp Q\geq 2$ and $\bigcap_{\Lambda\in Q}\Lambda\neq\emptyset$. We denote $\Phi=\bigcap_{\Lambda\in Q}\Lambda$, and the blowing-up with center in $\Phi$ by $\sigma:\Sigma'\rightarrow\Sigma$. Furthermore, by $\Theta'$ we denote the exceptional divisor of $\sigma$, and by $\Lambda'$ we denote the strict transform of $\Lambda\in\Comp(\Delta)$ by $\sigma$ for any $\Lambda\in\Comp(\Delta)$.

\begin{enumerate}
\setcounter{enumi}{6}
\item
The pair $(\Sigma', \sigma^*\Delta)$ is a normal crossing scheme over $k$.
\item
$\Theta'\in\Comp(\sigma^*\Delta)$. $\Comp(\sigma^*\Delta)-\{\Theta'\}=\{\Lambda'|\Lambda\in\Comp(\Delta)\}$.
$\sharp\Comp(\sigma^*\Delta)=\sharp\Comp(\Delta)+1$.
\item
For any $\Lambda\in\Comp(\Delta)- Q$, we have $\sigma^*\Lambda=\Lambda'$.
For any $\Lambda\in Q$, we have $\sigma^*\Lambda=\Lambda'+\Theta'$.
\item
$\sigma((\sigma^*\Delta)_0)=(\Delta)_0$. 
\item
For any $\A\in(\Delta)_0$ with $\A\not\in\Phi$, we have
$\sharp\sigma^{-1}(\A)\cap(\sigma^*\Delta)_0=1$, and the unique element $\A'$ in $\sigma^{-1}(\A)\cap(\sigma^*\Delta)_0$ satisfies 
$\{\A'\}=
\bigcap_{\Lambda\in\Comp(\Delta)(\A)}\Lambda'$,
$\Comp(\sigma^*\Delta)(\A')=\{\Lambda'|\Lambda\in\Comp(\Delta)(\A)\}$,
and $U(\Sigma', \sigma^*\Delta, \A')=\sigma^{-1}(U (\Sigma, \Delta, $\hfill\break$\A))$.

If moreover, a mapping
$\xi_\A:\Comp(\Delta)(\A)\rightarrow \mathcal{O}_\Sigma(U(\Sigma,\Delta,\A))$ is a coordinate system of $(\Sigma,\Delta)$ at $\A$, then there exists a unique coordinate system $\xi'_{\A'}:\Comp(\sigma^*\Delta)(\A')\rightarrow \mathcal{O}_{\Sigma'} (U(\Sigma',\sigma^*\Delta,\A'))$ of $(\Sigma',\sigma^*\Delta)$ at $\A'$ satisfying $\sigma^*(\xi_\A(\Lambda))=\xi'_{\A'}(\Lambda')$ for any $\Lambda\in\Comp(\Delta)(\A)$, where 
$\sigma^*: \mathcal{O}_\Sigma(U(\Sigma,\Delta,\A))\rightarrow \mathcal{O}_{\Sigma'} (U(\Sigma',\sigma^*\Delta,\A'))$ denotes the ring homomorphism induced by $\sigma$.
\item
For any $\A\in(\Delta)_0$ with $\A\in\Phi$, we have
$\sharp\sigma^{-1}(\A)\cap(\sigma^*\Delta)_0=\sharp Q\geq 2$, and
there exists a unique one-to-one mapping $\A':Q\rightarrow\sigma^{-1}(\A)\cap(\sigma^*\Delta)_0$ such that for any $\Xi\in Q$ we have
$\{\A'(\Xi)\}=
\Theta'\cap\bigcap_{\Lambda\in\Comp(\Delta)(\A)-\{\Xi\}}\Lambda'$,
$\Comp(\sigma^*\Delta)(\A'(\Xi))=\{\Theta'\}\cup\{\Lambda'|\Lambda\in\Comp(\Delta)(\A)-\{\Xi\}\}$, and
\hfill\break$U(\Sigma', \sigma^*\Delta, \A'(\Xi))= \sigma^{-1}(U (\Sigma, \Delta, \A))-\Xi'$.

If moreover, a mapping
$\xi_\A:\Comp(\Delta)(\A)\rightarrow \mathcal{O}_\Sigma(U(\Sigma,\Delta,\A))$ is a coordinate system of $(\Sigma,\Delta)$ at $\A$, then for any $\Xi\in Q$, there exists a unique coordinate system $\xi'_{\A'(\Xi)}:\Comp(\sigma^*\Delta)(\A'(\Xi))\rightarrow \mathcal{O}_{\Sigma'} (U(\Sigma',\sigma^*\Delta,\A'(\Xi)))$ of $(\Sigma',\sigma^*\Delta)$ at $\A'(\Xi)$ satisfying 
\begin{equation*}
\sigma^*(\xi_\A(\Lambda))=
\begin{cases}
\xi'_{\A'(\Xi)}(\Theta')&\text{if $\Lambda=\Xi$},\\
\xi'_{\A'(\Xi)}(\Lambda')\xi'_{\A'(\Xi)}(\Theta')&\text{if $\Lambda\in Q-\{\Xi\}$},\\
\xi'_{\A'(\Xi)}(\Lambda')&\text{if $\Lambda\in\Comp(\Delta)(\A)- Q$},
\end{cases}
\end{equation*}
for any $\Lambda\in\Comp(\Delta)(\A)$, where 
$\sigma^*: \mathcal{O}_\Sigma(U(\Sigma,\Delta,\A))\rightarrow$\hfill\break $\mathcal{O}_{\Sigma'} (U(\Sigma',\sigma^*\Delta,\A'(\Xi)))$ denotes the ring homomorphism induced by $\sigma$.
\end{enumerate}
\end{lemma}

Let $(\Sigma, \Delta)$ be a normal crossing scheme over $k$. 

We call a non-empty closed subscheme $\Phi$ of $\Sigma$ such that there exists a non-empty subset $Q$ of $\Comp(\Delta)$ satisfying $\Phi=\bigcap_{\Lambda\in Q}\Lambda$ a \emph{stratum} of $\Delta$. We call a blowing-up whose center is a stratum of $\Delta$ an \emph{admissible} blowing-up over $\Delta$.

Let $Q$ be any subset of $\Comp(\Delta)$ with $\sharp Q\geq 2$ and $\bigcap_{\Lambda\in Q}\Lambda\neq\emptyset$. Let $\sigma:\Sigma'\rightarrow\Sigma$ denote the admissible blowing-up with center in $\bigcap_{\Lambda\in Q}\Lambda$. We call the normal crossing scheme $(\Sigma', \sigma^*\Delta)$ over $k$ the \emph{pull-back} of $(\Sigma,\Delta)$ by $\sigma$. 
Let $\A'\in(\sigma^*\Delta)_0$ be any element, and let $\xi_{\sigma(\A')}:\Comp(\Delta)({\sigma(\A')})\rightarrow \mathcal{O}_\Sigma(U(\Sigma,\Delta,{\sigma(\A')}))$ be a coordinate system of $(\Sigma,\Delta)$ at ${\sigma(\A')}$. We have the coordinate system $\xi'_{\A'}:\Comp(\Delta)({\A'})\rightarrow \mathcal{O}_{\Sigma'} (U(\Sigma',\sigma^*\Delta,{\A'}))$  of $(\Sigma',\sigma^*\Delta)$ at $\A'$ described in Lemma~\ref{ncv}.$11$ or Lemma~\ref{ncv}.$12$. The coordinate system $\xi'_{\A'}$ is called the \emph{transformed coordinate system} of $\xi_{\sigma(\A')}$ at $\A'$ by $\sigma$. Let $\xi=\{\xi_\A|\A\in(\Delta)_0\}$ be a coordinate system of $(\Sigma,\Delta)$. We denote 
$$\sigma^*\xi=\{\xi'_{\A'}|\A'\in(\sigma^*\Delta)_0\},$$
and call $\sigma^*\xi$ the \emph{transformed coordinate system} of $\xi$ by $\sigma$. Note that triplets $(\Sigma,\Delta,\xi)$ and $(\Sigma',\sigma^*\Delta,\sigma^*\xi)$ are coordinated normal crossing scheme over $k$. We call the coordinated normal crossing scheme $(\Sigma',\sigma^*\Delta,\sigma^*\xi)$ over $k$ the \emph{pull-back} of  $(\Sigma,\Delta,\xi)$ by $\sigma$.

Let $\Sigma'$ be a scheme, and let $\sigma:\Sigma'\rightarrow \Sigma$ be a morphism. We call $\sigma$ an \emph{admissible composition of blowing-ups} over $\Delta$, if there exist a non-negative integer $m$, $(m+1)$ of normal crossing schemes $(\Sigma(i), \Delta(i)), i\in\{0,1,\ldots,m\}$, and $m$ of morphisms $\sigma(i):\Sigma(i)\rightarrow\Sigma(i-1), i\in\{1,2,\ldots,m\}$ satisfying the following two conditions:
\begin{enumerate}
\item $\Sigma(0)=\Sigma,\Delta(0)=\Delta,\Sigma(m)=\Sigma'$ and $\sigma=\sigma(1)\sigma(2)\cdots\sigma(m)$.
\item For any $i\in\{1,2,\ldots\,m\}$, $\sigma(i)$ is an admissible blowing-up over $\Delta(i-1)$ and $\Delta(i)=\sigma(i)^*\Delta(i-1)$.
\end{enumerate}
If moreover, the center of $\sigma(i)$ has codimension two for any $i\in\{1,2,\ldots\,m\}$, then we call $\sigma$ an admissible composition of blowing-ups \emph{with centers of codimension two} over $\Delta$.

\begin{lemma}
\label{admissible}
\mbox{\rm{1.}}
Let $(\Sigma,\Delta)$ be a normal crossing scheme over $k$; let $\Sigma'$ be a scheme, and let $\sigma: \Sigma'\rightarrow\Sigma$ be an admissible composition of blowing-ups over $\Delta$. Then, the pair $(\Sigma',\sigma^*\Delta)$ is a normal crossing scheme over $k$.

\noindent \mbox{\rm{2.}}
Let $(\Sigma,\Delta,\xi)$ be a coordinated normal crossing scheme over $k$; let $\Sigma'$ be a scheme, and let $\sigma: \Sigma'\rightarrow\Sigma$ be an admissible composition of blowing-ups over $\Delta$. 
Assume that $m\in\Z_0$, $(m+1)$ of normal crossing schemes $(\Sigma(i), \Delta(i)), i\in\{0,1,\ldots,m\}$, and $m$ of morphisms $\sigma(i):\Sigma(i)\rightarrow\Sigma(i-1), i\in\{1,2,\ldots,m\}$ satisfy the above two conditions. We write
$\sigma^*\xi=\sigma(m)^*\sigma(m-1)^*\cdots\sigma(1)^*\xi$.
Then, the triplet $(\Sigma',\sigma^*\Delta,\sigma^*\xi)$ is a coordinated normal crossing scheme over $k$, and the coordinate system $\sigma^*\xi$ of $(\Sigma',\sigma^*\Delta)$ does not depend on the choice of $m\in\Z_0$, $(m+1)$ of normal crossing schemes $(\Sigma(i), \Delta(i)), i\in\{0,1,\ldots,m\}$, and $m$ of morphisms $\sigma(i):\Sigma(i)\rightarrow\Sigma(i-1), i\in\{1,2,\ldots,m\}$ satisfying the two conditions.
\end{lemma} 

Let $(\Sigma,\Delta)$ be a normal crossing scheme over $k$; let $\Sigma'$ be a scheme, and let $\sigma: \Sigma'\rightarrow\Sigma$ be an admissible composition of blowing-ups over $\Delta$. We call the normal crossing scheme $(\Sigma',\sigma^*\Delta)$ over $k$ the \emph{pull-back} of $(\Sigma,\Delta)$ by $\sigma$.

Let $(\Sigma,\Delta,\xi)$ be a coordinated normal crossing scheme over $k$; let $\Sigma'$ be a scheme, and let $\sigma: \Sigma'\rightarrow\Sigma$ be an admissible composition of blowing-ups over $\Delta$. 
Choosing $m\in\Z_0$, $(m+1)$ of normal crossing schemes $(\Sigma(i), \Delta(i)), i\in\{0,1,\ldots,m\}$, and $m$ of morphisms $\sigma(i):\Sigma(i)\rightarrow\Sigma(i-1), i\in\{1,2,\ldots,m\}$ satisfying the above two conditions, we define the coordinate system $\sigma^*\xi$ of $(\Sigma',\sigma^*\Delta)$ by putting
$\sigma^*\xi=\sigma(m)^*\sigma(m-1)^*\cdots\sigma(1)^*\xi$.
The coordinate system $\sigma^*\xi$ does not depend on the choice of $m\in\Z_0$, $(m+1)$ of normal crossing schemes $(\Sigma(i), \Delta(i)), i\in\{0,1,\ldots,m\}$, and $m$ of morphisms $\sigma(i):\Sigma(i)\rightarrow\Sigma(i-1), i\in\{1,2,\ldots,m\}$ satisfying the above two conditions. We call $\sigma^*\xi$ the \emph{transformed coordinate system} of $\xi$ by $\sigma$. We call the coordinated normal crossing scheme $(\Sigma',\sigma^*\Delta, \sigma^*\xi)$ over $k$ the \emph{pull-back} of $(\Sigma,\Delta, \xi)$ by $\sigma$.

\begin{lemma}
\label{ext}
Let $(\Sigma, \Delta)$ and $(\Sigma, \Delta')$ be normal crossing schemes over $k$ with the same support $\Sigma$ such that $\Supp(\Delta)\subset\Supp(\Delta')$. 
\begin{enumerate}
\item $\Comp(\Delta)\subset\Comp(\Delta')$.
\item $\Comp(\Delta)(\A)\subset\Comp(\Delta')(\A)$ for any $\A\in\Sigma$.
\item $(\Delta)_0\subset(\Delta')_0$.
\item For any $\B\in(\Delta')_0$ there exists $\A\in(\Delta)_0$ with $\B\in U(\Sigma,\Delta,\A)$.
\item If $\B\in(\Delta')_0, \A\in(\Delta)_0$ and $\B\in U(\Sigma,\Delta,\A)$, then $U(\Sigma, \Delta',\B)\subset U(\Sigma,\Delta,\A)$.
\item If $\Supp(\Delta)=\Supp(\Delta')$, then $\Comp(\Delta)=\Comp(\Delta')$, $\Comp(\Delta)(\A)=$\hfill\break$\Comp(\Delta')(\A)$ for any $\A\in\Sigma$, $(\Delta)_0=(\Delta')_0$, and $U(\Sigma, \Delta',\A)=U(\Sigma,\Delta,\A)$ for any $\A\in(\Delta)_0$.
\end{enumerate}

Let $\xi=\{\xi_\A|\A\in(\Delta)_0\}$ be a coordinate system of $(\Sigma,\Delta)$. 
\begin{enumerate}
\setcounter{enumi}{6}
\item If $\Supp(\Delta)= \Supp(\Delta')$, then $(\Sigma, \Delta', \xi)$ is a coordinated normal crossing scheme over $k$.
\end{enumerate}
\end{lemma}

Let $(\Sigma, \Delta)$ and $(\Sigma, \Delta')$ be normal crossing schemes over $k$ with the same support $\Sigma$ such that $\Supp(\Delta)\subset\Supp(\Delta')$. Let $\xi=\{\xi_\A|\A\in(\Delta)_0\}$ be a coordinate system of $(\Sigma,\Delta)$. For any $\A\in(\Delta)_0$ we have
$\xi_\A:\Comp(\Delta)(\A)\rightarrow\mathcal{O}_\Sigma(U(\Sigma,\Delta,\A))$.

If for any $\B\in(\Delta')_0$, there exists $\A\in(\Delta)_0$
and a coordinate system $\xi'_\B:\Comp(\Delta')(\B)\rightarrow\mathcal{O}_\Sigma(U(\Sigma,\Delta',\B))$
of $(\Sigma,\Delta')$ at $\B$
such that $\B\in U(\Sigma,\Delta,\A)$ and 
for any $\Lambda'\in\Comp(\Delta')(\B)\cap \Comp(\Delta)(\A)$, 
$\xi'_\B(\Lambda)=\Res^{U(\Sigma,\Delta,\A)}_{U(\Sigma,\Delta',\B)}(\xi_\A(\Lambda))$ where $\Res^{U(\Sigma,\Delta,\A)}_{U(\Sigma,\Delta',\B)}:\mathcal{O}_\Sigma(U(\Sigma,\Delta,\A))\rightarrow\mathcal{O}_\Sigma(U(\Sigma,\Delta',\A))$ denotes the restriction homomorphism, then we say that the coordinate system $\xi$ is \emph{extendable} to $(\Sigma, \Delta')$. In the case where $\xi$ is extendable to $(\Sigma, \Delta')$, 
choosing an element $\A\in(\Delta)_0$
and a coordinate system $\xi'_\B:\Comp(\Delta')(\B)\rightarrow\mathcal{O}_\Sigma(U(\Sigma,\Delta',\B))$
of $(\Sigma,\Delta')$ at $\B$
such that $\B\in U(\Sigma,\Delta,\A)$ and 
$\xi'_\B(\Lambda)=\Res^{U(\Sigma,\Delta,\A)}_{U(\Sigma,\Delta',\B)}(\xi_\A(\Lambda))$ for any $\Lambda'\in\Comp(\Delta')(\B)\cap \Comp(\Delta)(\A)$, we define a coordinate system $\xi'_\B$ of $(\Sigma,\Delta')$ at $\B$ for any $\B\in(\Delta')$.
We call $\xi'=\{\xi'_\B|\B\in(\Delta')_0\}$ an \emph{extension} of $\xi$ to $(\Sigma,\Delta')$.

Let $(\Sigma,\Delta,\xi)$ be a coordinated normal crossing scheme over $k$.
Let $\Sigma'$ be a scheme, and let $\sigma:\Sigma'\rightarrow \Sigma$ be a morphism. We call $\sigma$ a \emph{weakly admissible composition of blowing-ups} over the pair $(\Delta, \xi)$, if there exist a non-negative integer $m$, $(m+1)$ of coordinated normal crossing schemes $(\Sigma(i), \Delta(i),\xi(i)), i\in\{0,1,\ldots,m\}$, and $m$ of morphisms $\sigma(i):\Sigma(i)\rightarrow\Sigma(i-1), i\in\{1,2,\ldots,m\}$ satisfying the following three conditions:
\begin{enumerate}
\item $\Sigma(0)=\Sigma,\Delta(0)=\Delta,\xi(1)=\xi,\Sigma(m)=\Sigma'$ and $\sigma=\sigma(1)\sigma(2)\cdots\sigma(m)$.
\item For any $i\in\{1,2,\ldots\,m\}$, $\sigma(i)$ is an admissible blowing-up over $\Delta(i-1)$ and $\Supp(\Delta(i))\supset\Supp(\sigma(i)^*\Delta(i-1))$.
\item For any $i\in\{1,2,\ldots\,m\}$, the coordinate $\sigma^*\xi(i-1)$ of $(\Sigma(i),\sigma(i)^*\Delta(i-1))$ is extendable to $(\Sigma(i),\Delta(i))$, and $\xi(i)$ is an extension of $\sigma^*\xi(i-1)$ to $(\Sigma(i),\Delta(i))$.
\end{enumerate}
If moreover, the center of $\sigma(i)$ has codimension two for any $i\in\{1,2,\ldots\,m\}$, then we call $\sigma$ a weakly admissible composition of blowing-ups \emph{with centers of codimension two} over $(\Delta, \xi)$.

Let $(\Sigma,\Delta,\xi)$ and let $(\Sigma',\Delta',\xi')$ be a coordinated normal crossing schemes over $k$.
Let $\sigma:\Sigma'\rightarrow \Sigma$ a weakly admissible composition of blowing-ups over $(\Delta, \xi)$. We say that $(\Sigma',\Delta',\xi')$  is an \emph{extended pull-back} of $(\Sigma,\Delta,\xi)$ by $\sigma$, if there exist a non-negative integer $m$, $(m+1)$ of coordinated normal crossing schemes $(\Sigma(i), \Delta(i),\xi(i)), i\in\{0,1,\ldots,m\}$, and $m$ of morphisms $\sigma(i):\Sigma(i)\rightarrow\Sigma(i-1), i\in\{1,2,\ldots,m\}$ satisfying the following three conditions:
\begin{enumerate}
\item $\Sigma(0)=\Sigma,\Delta(0)=\Delta,\xi(1)=\xi,\Sigma(m)=\Sigma',\Delta(m)=\Delta',\xi(m)=\xi'$ and $\sigma=\sigma(1)\sigma(2)\cdots\sigma(m)$.
\item For any $i\in\{1,2,\ldots\,m\}$, $\sigma(i)$ is an admissible blowing-up over $\Delta(i-1)$ and $\Supp(\Delta(i))\supset\Supp(\sigma(i)^*\Delta(i-1))$.
\item For any $i\in\{1,2,\ldots\,m\}$, the coordinate $\sigma^*\xi(i-1)$ of $(\Sigma(i),\sigma(i)^*\Delta(i-1))$ is extendable to $(\Sigma(i),\Delta(i))$, and $\xi(i)$ is an extension of $\sigma^*\xi(i-1)$ to $(\Sigma(i),\Delta(i))$.
\end{enumerate}

\begin{lemma}
\label{pull back blowing-ups}
Recall that $k$ denotes any algebraically closed field. Let $A$ be any complete regular local ring such that $A$ contains $k$ as a subring, the residue field  $A/M(A)$ is isomorphic to $k$ as algebras over $k$, and $\dim A\geq 2$, let $P$ be any parameter system of $A$, and let $z\in P$ be any element. 

Let $A'$ denote the completion of $k[P - \{z\}]$ with respect to the maximal ideal $k[P - \{z\}]\cap M(A)$. The ring $A'$ is a local subring of $A$ and $M(A')=M(A)\cap A' =(P -\{z\})A'$. The completion of $A'[z]$ with respect to the prime ideal $zA'[z]$ is isomorphic to $A$ as $A'[z]$-algebras. The set $P-\{z\}$ is a parameter system of $A'$.

Let $\sigma':\Sigma'\rightarrow\Spec(A')$ be any composition of finite blowing-ups with centers in closed irreducible smooth subschemes. We consider a morphism $\Spec(A)\rightarrow\Spec(A')$ induced by the inclusion ring homomorphism $A'\rightarrow A$, the product scheme $\Sigma=\Sigma'\times_{\Spec(A')}\Spec(A)$, the projection $\sigma:\Sigma\rightarrow\Spec(A)$, and the projection $\pi:\Sigma\rightarrow\Sigma'$. We know the following:
\begin{enumerate}
\item The morphism $\sigma$ is a composition of finite blowing-ups with centers in closed irreducible smooth subschemes.
\item The pull-back $\sigma^*\Spec(A/zA)$ of the prime divisor $\Spec(A/zA)$ of $\Spec(A)$ by $\sigma$ is a smooth prime divisor of $\Sigma$, and $\sigma^*\Spec(A/zA)\supset\sigma^{-1}(M(A))$.
\item The projection $\pi:\Sigma\rightarrow\Sigma'$ induces an isomorphism $\sigma^*\Spec(A/zA)\rightarrow\Sigma'$.
\end{enumerate}

Let $\Delta=\Spec(A/\prod_{x\in P}xA)$, and let $\Delta'=\Spec(A'/\prod_{x\in P-\{z\}}xA')$. For any $x\in P$, we put $\xi_{M(A)}(\Spec(A/xA))=x$. We obtain a coordinate system $\xi_{M(A)}:$\hfill\break$
\Comp(\Delta)(M(A))\rightarrow\mathcal{O}_{\Spec(A)}( U(\Spec(A),\Delta, M(A)))$ of $(\Spec(A),\Delta)$ at $M(A)$, and a coordinate system $\xi=\{\xi_{M(A)}\}$ of $(\Spec(A),\Delta)$.
For any $x\in P-\{z\}$ we put $\xi_{M(A')}'(\Spec(A'/xA'))=x$. We obtain a coordinate system $\xi_{M(A')}':
\Comp(\Delta')(M($\break$A'))\rightarrow\mathcal{O}_{\Spec(A')}( U(\Spec(A'),\Delta', M(A')))$ of $(\Spec(A'),\Delta')$ at $M(A')$, and a coordinate system $\xi'=\{\xi_{M(A')}'\}$ of $(\Spec(A'),\Delta')$.
\begin{enumerate}
\setcounter{enumi}{3}
\item If $\sigma'$ is an admissible composition of blowing-ups over $\Delta'$, then $\sigma$ is an admissible composition of blowing-ups over $\Delta$.
\item If $\sigma'$ is a weakly admissible composition of blowing-ups over $(\Delta',\xi')$, then $\sigma$ is a weakly admissible composition of blowing-ups over $(\Delta,\xi)$.
\end{enumerate}

Below we assume that $\sigma'$ is a weakly admissible composition of blowing-ups over $(\Delta',\xi')$, and we consider any extended pull-back $(\Sigma',\bar{\Delta}',\bar{\xi}')$ of a coordinated normal crossing scheme $(\Spec(A'),\Delta',\xi')$ by $\sigma'$.
We denote $\bar{\Delta}=\pi^*\bar{\Delta}'+\sigma^*\Spec(A/zA)\in\Div(\Sigma)$.

\begin{enumerate}
\setcounter{enumi}{5}
\item
The pair $(\Sigma,\bar{\Delta})$ is a normal crossing scheme over $k$. 
$(\bar{\Delta})_0=\pi^{-1}((\bar{\Delta}')_0)\cap\sigma^*\Spec(A/zA)$. For any $\A\in(\bar{\Delta})_0$ we have $\pi(\A)\in (\bar{\Delta}')_0$. The mapping $\pi: (\bar{\Delta})_0\rightarrow(\bar{\Delta}')_0$ induced by $\pi$ is bijective. 
\item
For any $\A\in(\bar{\Delta})_0$, we have $\Comp(\bar{\Delta})(\A)=\{\pi^*\Lambda|\Lambda\in \Comp(\bar{\Delta}')(\pi(\A))\}\cup\{\sigma^*\Spec(A/zA)\}$, and $U(\Sigma, \bar{\Delta},\A)=\pi^{-1}(U(\Sigma', \bar{\Delta}',\pi(\A)))$.
\end{enumerate}

Consider any $\A\in(\bar{\Delta})_0$. We put
$\bar{\xi}_\A(\pi^*\Lambda)=\pi^*(\bar{\xi}'_{\pi(\A)}(\Lambda))$
for any $\Lambda\in \Comp(\bar{\Delta}')($\break$\pi(\A))$, and we put
$\bar{\xi}_\A(\sigma^*\Spec(A/zA))=\sigma^*(z)$.
We have a mapping $\bar{\xi}_\A:\Comp(\bar{\Delta})(\A)$\break$\rightarrow\mathcal{O}_\Sigma(U(\Sigma, \bar{\Delta},\A))$. Let $\bar{\xi}=\{\bar{\xi}_\A|\A\in(\bar{\Delta})_0\}$.
\begin{enumerate}
\setcounter{enumi}{7}
\item
For any $\A\in(\bar{\Delta})_0$ the mapping $\bar{\xi}_\A$ is a coordinate system of $(\Sigma,\bar{\Delta})$ at $\A$, and $\bar{\xi}$ is a coordinate system of $(\Sigma,\bar{\Delta})$.
\end{enumerate}
\end{lemma}

\section{Main results}
\label{mainmain}
We state our main results. Their proofs will be given in Section~\ref{main proof} and Section~\ref{submain proofs}. 

We fix notations used throughout this section.

Let $k$ be any \emph{algebraically closed} field; let $A$ be any complete regular local ring such that $A$ contains $k$ as a subring, the residue field $A/M(A)$ is isomorphic to $k$ as $k$-algebras, and $\dim A\geq 2$; let $P$ be any parameter system of $A$, and let $z\in P$ be any element.

Let $A'$ denote the completion of $k[P - \{z\}]$ with respect to the maximal ideal $k[P - \{z\}]\cap M(A)$. The ring $A'$ is a local subring of $A$ and $M(A')=M(A)\cap A' =(P -\{z\})A'$. The completion of $A'[z]$ with respect to the prime ideal $zA'[z]$ is isomorphic to $A$ as $A'[z]$-algebras. The set $P - \{z\}$ is a parameter system of $A'$.

Let $\Delta=\Spec(A/\prod_{x\in P}xA)$, and $\Delta'=\Spec(A'/\prod_{x\in P -\{z\}}xA')$. We define a coordinate system $\xi_{M(A)}:\Comp(\Delta)\rightarrow A$ of the normal crossing scheme $(\Spec(A),\Delta)$ at $M(A)$ by putting $\xi_{M(A)}(\Spec(A/xA))=x$ for any $x\in P$. Let $\xi=\{ \xi_{M(A)}\}$. We define a coordinate system $\xi'_{M(A')}:\Comp(\Delta')\rightarrow A'$ of the normal crossing scheme $(\Spec(A'),\Delta')$ at $M(A')$ by putting $\xi_{M(A')}'(\Spec(A'/xA'))=x$ for any $x\in P-\{z\}$. Let $\xi'=\{ \xi_{M(A')} '\}$. The triplets $(\Spec(A),\Delta, \xi)$ and $(\Spec(A'),\Delta', \xi')$ are coordinated normal crossing schemes over $k$.

For any $\phi\in A$ with $\phi\neq 0$ the Newton polyhedron $\Gamma_+(P,\phi)$ of $\phi$ over $P$ is defined.

Furthermore, we denote
\begin{equation*}\begin{split}
&PW(1)= \{\phi\in A|\phi= u\prod_{\chi\in\mathcal{X}}(z+\chi)^{a(\chi)}\prod_{x\in P-\{z\}} x^{b(x)} \\
&\quad\text{for some }u\in A^\times, \text{ some finite subset } \mathcal{X} \text{ of } M(A'),\\
&\quad \text{some mapping }a: \mathcal{X}\rightarrow \Z_+, \text{ and some mapping }b:P-\{z\}\rightarrow\Z_0\}.\\
\end{split}\end{equation*}

For any $h\in\Z_+$ with $h\geq 2$, we denote
\begin{equation*}\begin{split}
W(h)&=\{\phi\in A|\phi= z^h+\sum_{i=0}^{h-1} \phi'(i)z^i \\
&\qquad\quad \text{for some mapping }\phi':\{0,1,\ldots,h-1\}\rightarrow M(A') \text{ satisfying }\\
&\qquad\quad
\chi^h+\sum_{i=0}^{h-1} \phi'(i)\chi^i\neq 0 \text{ for any }
\chi\in M(A').\},\\
\end{split}\end{equation*}\begin{equation*}\begin{split}
PW(h)&=\{\phi\in A|\phi=\psi\psi'\text{ for some } \psi\in W(h) \text{ and some }\psi'\in PW(1).\},\\
RW(h)&=\{\phi\in A|\text{ with }\phi=\psi\psi'\text{ for some } \psi\in W(h) \text{ and some }\psi'\in PW(1)\\
&\qquad\quad \text{such that }\Gamma_+(P,\psi) \text{ has no }z \text{-removable faces.}\},\\
SW(h)&=\{\phi\in A|\phi=\psi\psi'\text{ for some } \psi\in W(h) \text{ and some }\psi'\in PW(1)\text{ such that}\\
&\qquad\quad \Gamma_+(P,\psi) \text{ has no }z \text{-removable faces, and }\Gamma_+(P,\phi) \text{ is }z \text{-simple.}\}.\\
\end{split}\end{equation*}

Note that $1\in PW(1)\neq\emptyset$, $A=\{0\}\cup\cup_{h\in\Z_+}PW(h)$ if $\dim A=2$, and $A$ is a unique factorization domain. We consider any integer $h$ with $h\geq 2$. $\emptyset\neq W(h)\subset PW(h) \not\ni 0$, $\emptyset\neq SW(h)\subset RW(h)\subset PW(h)$. For any $\psi\in W(h)$, the Newton polyhedron $\Gamma_+(P,\psi)$ is of $z$-Weierstrass type, and the integer $h$ is equal to the $z$-height $\Ht(z,\Gamma_+(P,\psi))$ of $\Gamma_+(P,\psi)$. For any $\phi\in PW(h)$, the Newton polyhedron $\Gamma_+(P,\phi)$ is of $z$-Weierstrass type, and there exist uniquely $\psi\in W(h)$ and $\psi'\in PW(1)$ with $\phi=\psi\psi'$.

For the proof of our main theorem below, we apply the theory of convex sets and the theory of torus embeddings. By our main theorem any element in $SW(h)$ is reduced to an element in $PW(g)$ with $g<h$.

\begin{theorem}
\label{main}
For any $h\in\Z_+$ with $h\geq 2$ and any $\phi\in SW(h)$, there exist a smooth scheme $\Sigma$ over $\Spec(A)$ and an admissible composition of blowing-ups $\sigma:\Sigma\rightarrow \Spec(A)$ with centers of codimension two over $\Delta$ such that for any closed point $\A\in \Sigma$ with $\sigma(\A)=M(A)$, there exist $\B\in(\sigma^*\Delta)_0$, an isomorphism $\rho:\mathcal{O}_{\Sigma,\A}^c\rightarrow A$ of $k$-algebras, and $g\in\Z_+$ satisfying the following four conditions:
\begin{enumerate}
\item
$\A\in U(\Sigma,\sigma^*\Delta,\B)$. 
\end{enumerate}

We denote $\bar{P}=\{(\sigma^*\xi)_\B(\Lambda)- (\sigma^*\xi)_\B(\Lambda)(\A)|\Lambda\in \Comp(\sigma^*\Delta)(\B)\}$, which is a parameter system of $\mathcal{O}_{\Sigma,\A}^c$, and we consider the ring homomorphism $\sigma^*:A\rightarrow \mathcal{O}_{\Sigma,\A}^c$ induced by $\sigma$.
\begin{enumerate}
\setcounter{enumi}{1}
\item
$\rho(\bar{P})=P$.
\item $\rho\sigma^*(\phi)\in PW(g)$.
\item $g<h$.
\end{enumerate}
\end{theorem}

By the theorem below any element in $PW(h)$ is reduced to an element in $RW(h)$.

\begin{theorem}
\label{erase faces}
For any $h\in\Z_+$ with $h\geq 2$ and any $\phi\in PW(h)$, there exists an element $\chi\in M(A')$ satisfying $\rho(\phi)\in RW(h)$ where $\rho$ denotes the unique isomorphism $\rho:A\rightarrow A$ of $k$-algebras satisfying $\rho(z+\chi)=z$ and $\rho(x)=x$ for any $x\in P-\{z\}$.
\end{theorem}

We would like to solve the following problem:

\begin{problem}
\label{goal one}
Show that for any $\phi\in A$ with $\phi\neq 0$, there exists a weakly admissible composition of blowing-ups $\sigma:\Sigma\rightarrow\Spec(A)$ over $(\Delta,\xi)$ and an extended pull-back $(\Sigma,\bar{\Delta},\bar{\xi})$ of the coordinated normal crossing scheme $(\Spec(A),\Delta,\xi)$ by $\sigma$ satisfying $\Supp(\sigma^*(\Spec(A/\phi A)+\Delta))\subset \Supp(\bar{\Delta})$.
\end{problem}

Note here that $\dim A'=\dim A-1<\dim A$, and any $\phi'\in A'$ with $\phi'\neq 0$ has normal crossings over $P'$ if $\dim A=2$. Therefore, we decide that we use induction on $\dim A$, and we can assume the following claim $(*)$:

\begin{description}
\item[$(*)$]
For any  $\phi'\in A'$ with $\phi'\neq 0$, there exists a weakly admissible composition of blowing-ups $\sigma':\Sigma'\rightarrow\Spec(A')$ over $(\Delta',\xi')$ and an extended pull-back $(\Sigma',\bar{\Delta}',\bar{\xi}')$ of the coordinated normal crossing scheme $(\Spec(A'),\Delta',\xi')$ by $\sigma'$ satisfying $\Supp(\sigma^{\prime*}(\Spec(A'/\phi' A')+ \Delta'))\subset \Supp(\bar{\Delta}')$.
\end{description}

In the case of $\dim A=2$, putting $(\Sigma',\bar{\Delta}',\bar{\xi}')=(\Spec(A'),\Delta',\xi')$ and considering the identity morphism $\sigma':\Sigma'\rightarrow\Spec(A')=\Sigma'$, we know that $\sigma'$ is a weakly admissible composition of blowing-ups over $(\Delta',\xi')$ and $\Supp(\sigma^{\prime*}(\Spec(A'/\phi' A')+ \Delta'))\subset \Supp(\bar{\Delta}')$ for any $\phi'\in A'$ with $\phi'\neq 0$.

Let $\sigma':\Sigma'\rightarrow\Spec(A')$ be any weakly admissible composition of blowing-ups over $(\Delta',\xi')$, and let $(\Sigma',\bar{\Delta}',\bar{\xi}')$ be an extended pull-back of the coordinated normal crossing scheme $(\Spec(A'),\Delta',\xi')$ by $\sigma'$. We consider a morphism $\Spec(A)\rightarrow\Spec(A')$ induced by the inclusion ring homomorphism $A'\rightarrow A$, the product scheme $\Sigma=\Sigma'\times_{\Spec(A')}\Spec(A)$, the projection $\sigma:\Sigma\rightarrow\Spec(A)$, and the projection $\pi:\Sigma\rightarrow\Sigma'$. We know the following (See Lemma~\ref{pull back blowing-ups}.):
\begin{enumerate}
\item The morphism $\sigma$ is a weakly admissible composition of blowing-ups over $(\Delta,\xi)$.
\item The pull-back $\sigma^*\Spec(A/zA)$ of the prime divisor $\Spec(A/zA)$ of $\Spec(A)$ by $\sigma$ is a smooth prime divisor of $\Sigma$, and $\sigma^*\Spec(A/zA)\supset \sigma^{-1}(M(A))$.
\item The projection $\pi:\Sigma\rightarrow\Sigma'$ induces an isomorphism $\sigma^*\Spec(A/zA)\rightarrow\Sigma'$.
\end{enumerate}

Let $\bar{\Delta}=\pi^*\bar{\Delta}'+\sigma^*\Spec(A/zA)$. 

\begin{enumerate}
\setcounter{enumi}{3}
\item
The pair $(\Sigma,\bar{\Delta})$ is a normal crossing scheme over $k$. 
$(\bar{\Delta})_0=\pi^{-1}((\bar{\Delta}')_0)\cap\sigma^*\Spec(A/zA)$. For any $\A\in(\bar{\Delta})_0$, we have $\pi(\A)\in (\bar{\Delta}')_0$. The mapping $\pi: (\bar{\Delta})_0\rightarrow(\bar{\Delta}')_0$ induced by $\pi$ is bijective. 
\item
For any $\A\in(\bar{\Delta})_0$, we have $\Comp(\bar{\Delta})(\A)=\{\pi^*\Lambda|\Lambda\in \Comp(\bar{\Delta}')(\pi(\A))\}\cup\{\sigma^*\Spec(A/zA)\}$, and $U(\Sigma, \bar{\Delta},\A)=\pi^{-1}(U(\Sigma', \bar{\Delta}',\pi(\A)))$.
\end{enumerate}

Consider any $\A\in(\bar{\Delta})_0$. We put
$\bar{\xi}_\A(\pi^*\Lambda)=\pi^*(\bar{\xi}'_{\pi(\A)}(\Lambda))$
for any $\Lambda\in \Comp(\bar{\Delta}')($\break$\pi(\A))$, and we put
$\bar{\xi}_\A(\sigma^*\Spec(A/zA))=\sigma^*(z)$.
We have a mapping $\bar{\xi}_\A:\Comp(\bar{\Delta})(\A)$\break$\rightarrow\mathcal{O}_\Sigma(U(\Sigma, \bar{\Delta},\A))$.
Let $\bar{\xi}=\{\bar{\xi}_\A|\A\in(\bar{\Delta})_0\}$.

\begin{enumerate}
\setcounter{enumi}{5}
\item
For any $\A\in(\bar{\Delta})_0$, $\bar{\xi}_\A$ is a coordinate system of $(\Sigma,\bar{\Delta})$ at $\A$, and $\bar{\xi}$ is a coordinate system of $(\Sigma,\bar{\Delta})$.
\end{enumerate}

By the theorem below any element in $RW(h)$ is reduced to an element in $SW(h)$ or an element in $PW(g)$ with $g<h$.

\begin{theorem}
\label{make simple}
Assume the above $(*)$. For any $h\in\Z_+$ with $h\geq 2$ and any $\phi\in RW(h)$, there exist a smooth scheme $\Sigma'$ over $\Spec(A')$, a weakly admissible composition of blowing-ups $\sigma':\Sigma'\rightarrow\Spec(A')$ over $(\Delta',\xi')$ and an extended pull-back $(\Sigma',\bar{\Delta}',\bar{\xi}')$ of the coordinated normal crossing scheme $(\Spec(A'),\Delta',\xi')$ by $\sigma'$ with the following properties:

We consider the product scheme $\Sigma=\Sigma'\times_{\Spec(A')}\Spec(A)$, the projection $\sigma:\Sigma\rightarrow\Spec(A)$ and the projection $\pi:\Sigma\rightarrow\Sigma'$.
Let $\bar{\Delta}=\pi^*\bar{\Delta}'+\sigma^*\Spec(A/zA)$. We have a normal crossing scheme $(\Sigma,\bar{\Delta})$ with $(\bar{\Delta})_0\subset\sigma^*\Spec(A/zA)$.

Consider any $\A\in(\bar{\Delta})_0$. We put
$\bar{\xi}_\A(\pi^*\Lambda)=\pi^*(\bar{\xi}'_{\pi(\A)}(\Lambda))$
for any $\Lambda\in \Comp(\bar{\Delta}')($\break$\pi(\A))$, and we put
$\bar{\xi}_\A(\sigma^*\Spec(A/zA))=\sigma^*(z)$.
We have a coordinate system $\bar{\xi}_\A: $\hfill\break$\Comp(\bar{\Delta})(\A)\rightarrow\mathcal{O}_\Sigma(U(\Sigma, \bar{\Delta},\A))$ of $(\Sigma,\bar{\Delta})$ at $\A$, and a coordinate system $\bar{\xi}=\{\bar{\xi}_\A|\A\in(\bar{\Delta})_0\}$ of $(\Sigma,\bar{\Delta})$.

Under the above notations, at any closed point $\A\in\Sigma$ with $\sigma(\A)=M(A)$ there exist $\B\in(\bar{\Delta})_0$ and an isomorphism $\rho:\mathcal{O}_{\Sigma,\A}^c\rightarrow A$ of $k$-algebras satisfying the following three conditions:
\begin{enumerate}
\item
$\A\in U(\Sigma, \bar{\Delta},\B)$. 
\end{enumerate}

We denote $\bar{P}=\{\bar{\xi}_\B(\Lambda)- \bar{\xi}_\B(\Lambda)(\A)|\Lambda\in \Comp(\bar{\Delta})(\B)\}$, which is a parameter system of $\mathcal{O}_{\Sigma,\A}^c$, and we consider the ring homomorphism $\sigma^*:A\rightarrow \mathcal{O}_{\Sigma,\A}^c$ induced by $\sigma$.
\begin{enumerate}
\setcounter{enumi}{1}
\item
$\rho(\bar{P})=P$ and $\rho\sigma^*(z)=z$.
\item $\rho\sigma^*(\phi)\in SW(h)$ or $\rho\sigma^*(\phi)\in PW(g)$ for some  positive integer $g\in \Z_+$ with $g<h$.
\end{enumerate}
\end{theorem}

By the theorem below any non-zero element in $A$ is reduced to an element in $\cup_{h\in\Z_+}PW(h)$.

\begin{theorem}
\label{make Weierstrass type}
Assume the above $(*)$. For any $\phi\in A$ with $\phi\neq 0$, there exist a smooth scheme $\Sigma'$ over $\Spec(A')$, a weakly admissible composition of blowing-ups $\sigma':\Sigma'\rightarrow\Spec(A')$ over $(\Delta',\xi')$ and an extended pull-back $(\Sigma',\bar{\Delta}',\bar{\xi}')$ of the coordinated normal crossing scheme $(\Spec(A'),\Delta',\xi')$ by $\sigma'$ with the following properties:

We consider the product scheme $\Sigma=\Sigma'\times_{\Spec(A')}\Spec(A)$, the projection $\sigma:\Sigma\rightarrow\Spec(A)$ and the projection $\pi:\Sigma\rightarrow\Sigma'$.
Let $\bar{\Delta}=\pi^*\bar{\Delta}'+\sigma^*\Spec(A/zA)$. We have a normal crossing scheme $(\Sigma,\bar{\Delta})$ with $(\bar{\Delta})_0\subset\sigma^*\Spec(A/zA)$.

Consider any $\A\in(\bar{\Delta})_0$. We put
$\bar{\xi}_\A(\pi^*\Lambda)=\pi^*(\bar{\xi}'_{\pi(\A)}(\Lambda))$
for any $\Lambda\in \Comp(\bar{\Delta}')($\break$\pi(\A))$, and we put
$\bar{\xi}_\A(\sigma^*\Spec(A/zA))=\sigma^*(z)$.
We have a coordinate system $\bar{\xi}_\A: $\hfill\break$\Comp(\bar{\Delta})(\A)\rightarrow\mathcal{O}_\Sigma(U(\Sigma, \bar{\Delta},\A))$ of $(\Sigma,\bar{\Delta})$ at $\A$, and a coordinate system $\bar{\xi}=\{\bar{\xi}_\A|\A\in(\bar{\Delta})_0\}$ of $(\Sigma,\bar{\Delta})$.

Under the above notations, at any closed point $\A\in\Sigma$ with $\sigma(\A)=M(A)$, there exist $\B\in(\bar{\Delta})_0$ and an isomorphism $\rho:\mathcal{O}_{\Sigma,\A}^c\rightarrow A$ of $k$-algebras satisfying the following three conditions:
\begin{enumerate}
\item
$\A\in U(\Sigma, \bar{\Delta},\B)$. 
\end{enumerate}

We denote $\bar{P}=\{\bar{\xi}_\B(\Lambda)- \bar{\xi}_\B(\Lambda)(\A)|\Lambda\in \Comp(\bar{\Delta})(\B)\}$, which is a parameter system of $\mathcal{O}_{\Sigma,\A}^c$, and we consider the ring homomorphism $\sigma^*:A\rightarrow \mathcal{O}_{\Sigma,\A}^c$ induced by $\sigma$.
\begin{enumerate}
\setcounter{enumi}{1}
\item
$\rho(\bar{P})=P$ and $\rho\sigma^*(z)=z$.
\item $$\rho\sigma^*(\phi)\in \bigcup_{h\in\Z_+}PW(h).$$
\end{enumerate}
\end{theorem}

By the theorem below any element in $PW(1)$ is reduced to an element with normal crossings.

\begin{theorem}
\label{make normal crossings}
Assume the above $(*)$. For any $\phi\in PW(1)$, there exist a smooth scheme $\Sigma$ over $\Spec(A)$, a weakly admissible composition of blowing-ups $\sigma:\Sigma\rightarrow\Spec(A)$ over $(\Delta,\xi)$ and an extended pull-back $(\Sigma,\bar{\Delta},\bar{\xi})$ of the coordinated normal crossing scheme $(\Spec(A),\Delta,\xi)$ by $\sigma$ satisfying $\Supp(\sigma^*(\Spec(A/\phi A)+\Delta))\subset\Supp(\bar{\Delta})$.
\end{theorem}

\begin{conjecture}
\label{goal}
Let $(\Sigma_0,\Delta_0,\xi_0)$ be any coordinated normal crossing scheme over $k$, and let $\Phi_0$ be any effective divisor of $\Sigma_0$.
There exist a weakly admissible composition of blowing-ups $\sigma:\Sigma\rightarrow\Sigma_0$ over $(\Delta_0,\xi_0)$ and an extended pull-back $(\Sigma,\bar{\Delta},\bar{\xi})$ of the coordinated normal crossing scheme $(\Sigma_0,\Delta_0,\xi_0)$ by $\sigma$ satisfying $\Supp(\sigma^*(\Phi_0+\Delta_0))\subset \Supp(\bar{\Delta})$.
\end{conjecture}

Obviously solution of Conjecture~\ref{goal} implies solution of Problem~\ref{goal one}. 

Now, to solve Conjecture~\ref{goal} we have to glue up local blowing-ups obtained by repeated application of theorems in this section, and have to construct global blowing-ups. In case $\dim\Sigma_0\leq 2$ it is easy to glue up them. We would like to solve Conjecture~\ref{goal} in case $\dim\Sigma_0\geq 3$ and would like to write forthcoming articles, cooperating with Professor Heisuke Hironaka and young mathematicians.

\section{Ring theory}
\label{ring}

We explain some basic theory on rings. (Matsumura~\cite{M}, Hironaka et al~\cite{HU81}.)

\begin{lemma}
\label{basic ord}
Let $k$ be any field. Let $A$ be any complete regular local ring such that $\dim A\geq 1$, $A$ contains $k$ as a subring, and the residue field  $A/M(A)$ is isomorphic to $k$ as algebras over $k$. Let $P$ be any parameter system of $A$.
\begin{enumerate}
\item
The ring $A$ is a unique factorization domain.
\item
The parameter system $P$ is a non-empty finite subset of $A$ with $\sharp P=\dim A\in\Z_+$.
For any element $\phi$ of $A$, there exists a unique element $c\in \Map(\Map(P,$\break$\Z_0),k)$ satisfying
$$\phi=\sum_{\Lambda\in \Map(P,\Z_0)} c(\Lambda)\prod_{x\in P}x^{\Lambda(x)}.$$
The infinite sum in the right-hand side is the limit with respect to the $M(A)$-adic topology on $A$.
$P$ is algebraically independent over $k$.
\end{enumerate}

In Section~\ref{concept} we defined the constant term $\phi(0)\in k$ of $\phi$ and the support $\Supp(P,\phi)$ of $\phi$ over $P$ for any $\phi\in A$. 
\begin{enumerate}
\setcounter{enumi}{2}
\item
Consider any $\phi\in A$. $\phi(0)\in k$ and $\phi-\phi(0)\in M(A)$.
$\phi\in M(A)$, if and only if, $\phi(0)=0$.
\item
If $\phi\in A$, $c\in \Map(\Map(P,\Z_0),k)$ and $\phi=\sum_{\Lambda\in \Map(P,\Z_0)} c(\Lambda)\prod_{x\in P}x^{\Lambda(x)}$, then $\phi(0)=c(0)$ and $\Supp(P,\phi)=\Supp(c)=\{\Lambda\in\Map(P,\Z_0)|c(\Lambda)\neq 0\}$.
\item
Consider any $\phi\in A$. $\Supp(P,\phi)\subset\Map(P,\Z_0)$.
$\Supp(P,\phi)=\emptyset$, if and only if $\phi=0$.
$0\not\in \Supp(P,\phi)$, if and only if, $\phi\in M(A)$.
\item
Consider any $\phi\in A$, any $\psi\in A$ and any $\A\in k$ with $\A\neq 0$.
\begin{equation*}
\begin{split}
\Supp(P,-\phi)&=\Supp(P,\phi)=\Supp(P,\A\phi),\\
\Supp(P, \phi+\psi)&\subset\Supp(P, \phi)\cup\Supp(P,\psi),\\
\Supp(P,\phi\psi)&\subset\Supp(P, \phi)+\Supp(P,\psi).
\end{split}
\end{equation*}
\item
The set $\Map(P,\R)$ is a finite dimensional vector space over $\R$ with $\dim\Map(P,$\break$\R)=\sharp P$.
The set $\Map(P,\Z)$ is a lattice of $\Map(P,\R)$.
The set $\Map(P,\R_0)$ is a simplicial cone over $\Map(P,\Z)$ in $\Map(P,\R)$ with $\Vect(\Map(P,\R_0))= \Map(P,\R)$.
$\Map(P,\Z_0)=\Map(P,\R_0)\cap\Map(P,\Z)$.
The set $\Map(P,\Z_0)$ is a semisubgroup of $\Map(P,\R)$ containing $0$.
\end{enumerate}

In Section~\ref{concept} we defined an element $f^P_x\in\Map(P,\Z_0)$ for any $x\in P$.

\begin{enumerate}
\setcounter{enumi}{7}
\item
The set $\{f^P_x|x\in P\}$ is a basis of $\Map(P,\R)$ over $\R$.
It is a basis of $\Map(P,\Z)$ over $\Z$.
$\Map(P,\R_0) =\Convcone(\{f^P_x|x\in P\})=\sum_{x\in P}\R_0f^P_x$.
$\Map(P,\Z_0)=\sum_{x\in P}\Z_0f^P_x$.
\end{enumerate}

We denote the dual basis of $\{f^P_x|x\in P\}$ by $\{f^{P\vee}_x|x\in P\}$.
For any convex cone $S$ in $\Map(P,\R)$ we denote the dual cone $S^\vee|\Map(P,\R)$ in $\Map(P,\R)^*$ simply by $S^\vee$.
\begin{enumerate}
\setcounter{enumi}{8}
\item
The set $\{f^{P\vee}_x|x\in P\}$ is a basis of $\Map(P,\R)^*$ over $\R$.
It is a basis of $\Map(P,\Z)^*$ over $\Z$.
$\Map(P,\R_0)^\vee=\Convcone(\{f^{P\vee}_x|x\in P\})=\sum_{x\in P}\R_0 f^{P\vee}_x$.

\item
Consider any $\phi\in A$ with $\phi\neq 0$ and any $\omega\in\Map(P,\R_0)^\vee$.

$\emptyset\neq\{\langle\omega, \Lambda\rangle|\Lambda\in\Supp(P,\phi)\}\subset\R_0$.
For any $t\in\R_0$, the intersection $\{\langle\omega, \Lambda\rangle|\Lambda\in\Supp(P,\phi)\}\cap\{u\in\R_0|u\leq t\}$ is a finite set.
There exists the minimum element
$\min\{\langle\omega, \Lambda\rangle|\Lambda\in\Supp(P,\phi)\}$ of the subset $\{\langle\omega, \Lambda\rangle|\Lambda\in$\break$\Supp(P,\phi)\}$ of $\R_0$, and $\min\{\langle\omega, \Lambda\rangle|\Lambda\in\Supp(P,\phi)\}\in\R_0$.
\end{enumerate}

According to definitions in Section~\ref{concept}, for any $\phi\in A$ with $\phi\neq 0$ and any $\omega\in\Map(P,\R_0)^\vee$
\begin{equation*}
\begin{split}
\Ord(P,\omega,\phi)=&\min\{\langle\omega, \Lambda\rangle|\Lambda\in\Supp(P,\phi)\}\in\R_0,\\
\Supp(P,\omega,\phi)=&\{\Lambda\in\Supp(P,\phi)|\langle\omega,\Lambda\rangle=\Ord(P,\omega,\phi)\}\subset\Supp(P,\phi),\\
\In(P,\omega,\phi)=&\sum_{\Lambda\in\Supp(P,\omega,\phi)} c(\Lambda)\prod_{x\in P}x^{\Lambda(x)}\in A,
\end{split}
\end{equation*}
where  $c\in\Map(\Map(P,\Z_0),k)$ is the unique element satisfying $\phi=\sum_{\Lambda\in \Map(P,\Z_0)} c(\Lambda)$\break$\prod_{x\in P}x^{\Lambda(x)}$.

Furthermore, for any $\omega\in\Map(P,\R_0)^\vee$,
$\Ord(P,\omega,0)=\infty$ and $\In(P,\omega,0)=0\in A$,
where $\infty$ is a symbol satisfying $\infty+\infty=\infty$, $\infty+t=t+\infty=\infty$, $t\leq\infty$, $\infty\geq t$, $t<\infty$, $\infty>t$, $t\neq\infty$ and $\infty\neq t$ for any $t\in\R$.
\begin{enumerate}
\setcounter{enumi}{10}
\item
Consider any $\phi\in A$ and any $\omega\in\Map(P,\R_0)^\vee$.

$\In(P,\omega,\In(P,\omega,\phi))=\In(P,\omega,\phi)$.
$\Ord(P,\omega,\In(P,\omega,\phi))=\Ord(P,\omega,\phi)$.\hfill\break
$\Ord(P,\omega, \phi-\In(P,\omega,\phi))>\Ord(P,\omega,\phi)$, if $\phi\neq 0$.

If $\psi\in A$, $\In(P,\omega, \psi)=\psi$, and $\Ord(P,\omega, \phi-\psi)>\Ord(P,\omega,\phi)$, then $\psi\neq 0$, $\phi\neq 0$ and $\psi=\In(P,\omega,\phi)$.
\item
Consider any $\phi\in A$ and any $\omega\in\Map(P,\R_0)^\vee$.

If $\phi\in A^\times$, then $\Ord(P,\omega,\phi)=0$.

If $\Ord(P,\omega, \phi)=0$ and $\langle\omega, f^P_x\rangle>0$ for any $x\in P$, then $\phi\in A^\times$ and $\In(P,\omega, \phi)=\phi(0)$.
\item
Consider any $\phi\in A$, any $\omega\in\Map(P,\R_0)^\vee$ and any $\A\in k$ with $\A\neq 0$.
$\Ord(P,\omega,-\phi)= \Ord(P,\omega,\phi)$. $\In(P,\omega,-\phi)= -\In(P,\omega,\phi)$. \hfill\break$\Ord(P,\omega,\A\phi)= \Ord(P,\omega,\phi)$. $\In(P,\omega,\A\phi)= \A\In(P,\omega,\phi)$.

\item
Consider any $\phi\in A$, any $\psi\in A$ and any $\omega\in\Map(P,\R_0)^\vee$.
$$\Ord(P,\omega, \phi+\psi)\geq\min\{\Ord(P,\omega,\phi), \Ord(P,\omega,\psi)\}.$$

$\Ord(P,\omega, \phi+\psi)=\min\{\Ord(P,\omega,\phi), \Ord(P,\omega,\psi)\}$, if and only if, \hfill\break$\Ord(P,\omega,\phi)\neq\Ord(P,\omega,\psi)$ or $\In(P,\omega,\phi)+ \In(P,\omega,\psi)\neq 0$.

If $\Ord(P,\omega,\phi)<\Ord(P,\omega,\psi)$, then $\In(P,\omega, \phi+\psi)=\In(P,\omega,\phi)$.

If $\Ord(P,\omega,\phi)>\Ord(P,\omega,\psi)$, then $\In(P,\omega, \phi+\psi)=\In(P,\omega,\psi)$.

If $\Ord(P,\omega,\phi)=\Ord(P,\omega,\psi)$ and $\In(P,\omega,\phi)+ \In(P,\omega,\psi)\neq 0$, then \hfill\break $\In(P,\omega, \phi+\psi)=\In(P,\omega,\phi)+ \In(P,\omega,\psi)$.
\item
Consider any $\phi\in A$, any $\psi\in A$ and any $\omega\in\Map(P,\R_0)^\vee$.
\begin{equation*}
\begin{split}
\Ord(P,\omega, \phi\psi)=&\Ord(P,\omega,\phi)+\Ord(P,\omega,\psi).\\
\In(P,\omega, \phi\psi)=&\In(P,\omega,\phi)\In(P,\omega,\psi).
\end{split}
\end{equation*}
\item
Consider any parameter system $\bar{P}$ of $A$ and any bijective mapping $\sigma:P\rightarrow \bar{P}$.

Take the unique isomorphism $\rho:A\rightarrow A$ of $k$-algebras satisfying $\rho(x)=\sigma(x)$ for any $x\in P$.

Note that $\sigma$ induces an isomorphism $\sigma^*:\Map(\bar{P},\R)\rightarrow \Map(P,\R)$ of vector spaces over $\R$, and it induces an isomorphism $(\sigma^*)^*:\Map(P,\R)^*\rightarrow \Map(\bar{P},\R)^*$ of vector spaces over $\R$. 
$(\sigma^*)^*(\Map(P,\R_0)^\vee)= \Map(\bar{P},\R_0)^\vee$\break$|\Map(\bar{P},\R)$.

Consider any $\omega\in\Map(P,\R_0)^\vee$.

If $\Ord(P, \omega, \sigma(x))=\Ord(P, \omega, x)$ for any $x\in P$, then, $\Ord(P,\omega,\phi)=\Ord(\bar{P}, (\sigma^*)^*(\omega), \phi)$ and
$\rho(\In(P,\omega,\phi))=\In(\bar{P}, (\sigma^*)^*(\omega), \phi)$ for any $\phi\in A$.
\item
Consider any $\phi\in A$ with $\phi\neq 0$, any $\omega\in\Map(P,\R_0)^\vee$, any $\chi\in\Map(P,\R_0)^\vee$ and any $t\in\R$ with $0\leq t\leq 1$.
$(1-t)\omega+t\chi\in\Map(P,\R_0)^\vee$ and
$\Ord(P,(1-t)\omega+t\chi, \phi)\geq(1-t)\Ord(P,\omega,\phi)+t\Ord(P,\chi,\phi)$.
\end{enumerate}
\end{lemma}
\begin{proof}
See Matsumura~\cite{M} for the proofs of claim 1 and 2.
Below we give the proof only to claim 10 and 15. (Hironaka et al~\cite{HU81}.)
Other claims follow easily.

\noindent 10.
Consider any $\phi\in A$ with $\phi\neq 0$ and any $\omega\in\Map(P,\R_0)^\vee$.

By 5 $\emptyset\neq\Supp(P,\phi)\subset\Map(P,\Z_0)\subset\Map(P,\R_0)$.
Since $\omega\in\Map(P,\R_0)^\vee$, we know $\emptyset\neq\{\langle\omega,\Lambda\rangle|\Lambda\in\Supp(P,\phi)\}\subset\R_0$.

Put $\bar{P}=\{x\in P|\langle\omega, f^P_x\rangle\neq 0\}$.
$\bar{P}\subset P$.
Let $\pi:\Map(P,\R)\rightarrow\Map(\bar{P},\R)$ denote the surjective homomorphism of vector spaces over $\R$ induced by the inclusion mapping $\bar{P}\rightarrow P$.
The dual homomorphism $\pi^*:\Map(\bar{P},\R)^*\rightarrow\Map(P,\R)^*$ is injective.
Since $\omega=\sum_{x\in\bar{P}}\langle\omega, f^P_x\rangle f^{P\vee}_x$,
there exists a unique element $\bar{\omega}\in\Map(\bar{P},\R)^*$ with $\omega=\pi^*(\bar{\omega})$.
We take the unique element $\bar{\omega}\in\Map(\bar{P},\R)^*$ with $\omega=\pi^*(\bar{\omega})$.
Since $\omega\in\Map(P,\R_0)^\vee$, $\langle\bar{\omega},f^{\bar{P}}_x\rangle=\langle\omega, f^P_x\rangle>0$ for any $x\in\bar{P}$.

Consider any $t\in\R_0$.
If $\bar{P}=\emptyset$, then $\omega=0$.
If $\omega=0$, then $\{\langle\omega,\Lambda\rangle|\Lambda\in\Supp(P,\phi)\}=\{0\}$, and 
$\{\langle\omega,\Lambda\rangle|\Lambda\in\Supp(P,\phi)\}\cap\{u\in\R_0|u\leq t\}=\{0\}$ is a finite set.

Below we assume $\omega\neq 0$. $\bar{P}\neq\emptyset$.
Put $e=\min\{\langle\bar{\omega},f^{\bar{P}}_x\rangle|x\in\bar{P}\}\in\R_+$.
We know 
$\pi(\Supp(P,\phi))\subset\pi(\Map(P,\Z_0))=\Map(\bar{P},\Z_0)$, and
$\{\langle\omega,\Lambda\rangle|\Lambda\in\Supp(P,\phi)\}=\{\langle\bar{\omega},\bar{\Lambda}\rangle|
\bar{\Lambda}\in\pi(\Supp(P,\phi))\}\subset \{\langle\bar{\omega},\bar{\Lambda}\rangle|
\bar{\Lambda}\in\Map(\bar{P},\Z_0)\}$.

Assume $\bar{\Lambda}\in\Map(\bar{P},\Z_0)$ and $\langle\bar{\omega},\bar{\Lambda}\rangle\leq t$.
We have
$t\geq\langle\bar{\omega},\bar{\Lambda}\rangle=\sum_{x\in\bar{P}}\langle\bar{\omega},f^{\bar{P}}_x\rangle
$\hfill\break$
\langle f^{\bar{P}\vee}_x,\bar{\Lambda}\rangle\geq e\sum_{x\in\bar{P}}\langle f^{\bar{P}\vee}_x,\bar{\Lambda}\rangle$.

We know
$\{\langle\omega,\Lambda\rangle|\Lambda\in\Supp(P,\phi)\}\cap\{u\in\R_0|u\leq t\}\subset
\{\langle\bar{\omega},\bar{\Lambda}\rangle|\bar{\Lambda}\in\Map(\bar{P},\Z_0),
$\hfill\break$
\sum_{x\in\bar{P}}\langle f^{\bar{P}\vee}_x,\bar{\Lambda}\rangle\leq t/e\}$.
Since $\{\bar{\Lambda}\in\Map(\bar{P},\Z_0)| \sum_{x\in\bar{P}}\langle f^{\bar{P}\vee}_x,\bar{\Lambda}\rangle\leq t/e\}$ is a finite set, we know that $\{\langle\omega,\Lambda\rangle|\Lambda\in\Supp(P,\phi)\}\cap\{u\in\R_0|u\leq t\}$ is a finite set.

Since $\{\langle\omega,\Lambda\rangle|\Lambda\in\Supp(P,\phi)\}\cap\{u\in\R_0|u\leq t\}$ is finite for any $t\in\R_0$, we know that the minimum element $\min\{\langle\omega,\Lambda\rangle|\Lambda\in\Supp(P,\phi)\}$ of $\{\langle\omega,\Lambda\rangle|\Lambda\in\Supp(P,\phi)\}$ exists and $\min\{\langle\omega,\Lambda\rangle|\Lambda\in\Supp(P,\phi)\}\in\R_0$.

\noindent 15.
Consider any $\omega\in\Map(P,\R_0)^\vee$.

Put $\bar{P}=\{x\in P|\langle\omega, f^P_x\rangle\neq 0\}$.
$\bar{P}\subset P$.
Let $\pi:\Map(P,\R)\rightarrow\Map(\bar{P},\R)$ denote the surjective homomorphism of vector spaces over $\R$ induced by the inclusion mapping $\bar{P}\rightarrow P$.
The dual homomorphism $\pi^*:\Map(\bar{P},\R)^*\rightarrow\Map(P,\R)^*$ is injective.
We take the unique element $\bar{\omega}\in\Map(\bar{P},\R)^*$ with $\omega=\pi^*(\bar{\omega})$.
Since $\omega\in\Map(P,\R_0)^\vee$, $\langle\bar{\omega},f^{\bar{P}}_x\rangle=\langle\omega, f^P_x\rangle>0$ for any $x\in\bar{P}$.
For any $t\in\R_0$, sets $\{\bar{\Lambda}\in\Map(\bar{P},\Z_0)|\langle\bar{\omega},\bar{\Lambda}\rangle\leq t\}$ and 
$\{\bar{\Lambda}\in\Map(\bar{P},\Z_0)|\langle\bar{\omega},\bar{\Lambda}\rangle=t\}$ are finite.

We define a total order called the \emph{lexicographic order} on the vector space $\Map(\bar{P},$\break$\R)$.
Let $r=\sharp\bar{P}\in\Z_0$.
Take any bijective mapping $x:\{1,2,\ldots, r\}\rightarrow\bar{P}$.
Let $\bar{\Lambda}\in\Map(\bar{P},\R)$ and $\bar{\Gamma}\in\Map(\bar{P},\R)$ be arbitrary elements.
We write $\bar{\Lambda}<\bar{\Gamma}$ or $\bar{\Gamma}>\bar{\Lambda}$, if there exists $i\in\{1,2,\ldots,r\}$ such that $\langle f^{\bar{P}}_{x(j)}, \bar{\Gamma}-\bar{\Lambda}\rangle=0$ for any $j\in\{1,2,\ldots,i-1\}$ and $\langle f^{\bar{P}}_{x(i)}, \bar{\Gamma}-\bar{\Lambda}\rangle>0$.
We write  $\bar{\Lambda}\leq\bar{\Gamma}$ or $\bar{\Gamma}\geq\bar{\Lambda}$, if
$\bar{\Lambda}<\bar{\Gamma}$ or $\bar{\Lambda}=\bar{\Gamma}$.
We know that the relation $\leq$ is a total order on the abelian group $\Map(\bar{P},\R)$, in other words, the following five conditions hold for any $\bar{\Lambda}\in\Map(\bar{P},\R)$, any $\bar{\Gamma}\in\Map(\bar{P},\R)$ and any $\bar{\Delta}\in\Map(\bar{P},\R)$.
\begin{enumerate}
\item
$\bar{\Lambda}\leq\bar{\Gamma}$ or $\bar{\Gamma}\leq\bar{\Lambda}$.
\item
If $\bar{\Lambda}\leq\bar{\Gamma}$ and $\bar{\Gamma}\leq\bar{\Delta}$,
then $\bar{\Lambda}\leq\bar{\Delta}$.
\item
If $\bar{\Lambda}\leq\bar{\Gamma}$ and $\bar{\Gamma}\leq\bar{\Lambda}$,
then $\bar{\Lambda}=\bar{\Gamma}$.
\item
If $\bar{\Lambda}\leq\bar{\Gamma}$, then $\bar{\Lambda}+\bar{\Delta}\leq\bar{\Gamma}+\bar{\Delta}$.
\item
$\bar{\Lambda}<\bar{\Gamma}$, if and only if, $\bar{\Lambda}\leq\bar{\Gamma}$ and $\bar{\Lambda}\neq\bar{\Gamma}$.
\end{enumerate}

Consider any $\bar{\Lambda}\in\Map(\bar{P},\Z_0)$.
We denote
\begin{equation*}
\begin{split}
\delta(\bar{\Lambda})=\{(\bar{\Gamma}, \bar{\Delta})|&\bar{\Gamma}\in\Map(\bar{P},\Z_0), \bar{\Delta}\in\Map(\bar{P},\Z_0),
\bar{\Gamma}+\bar{\Delta}=\bar{\Lambda}\}\\
&\quad\subset \Map(\bar{P},\Z_0)\times\Map(\bar{P},\Z_0).
\end{split}
\end{equation*}
We know that $\delta(\bar{\Lambda})$ is a non-empty finite subset of $\Map(\bar{P},\Z_0)\times\Map(\bar{P},\Z_0)$.

Let $B$ denote the completion of $k[P-\bar{P}]$ with respect to the prime ideal $k[P-\bar{P}]\cap M(A)$.
$B$ and $B[\bar{P}]$ are subrings of $A$.
$B$ is a complete regular local ring such that $\dim B\leq\dim A$, $B$ contains $k$ as a subring, and the residue field  $B/M(B)$ is isomorphic to $k$ as algebras over $k$.
$P-\bar{P}$ is a parameter system of $B$.
$B$ is an integral domain.
$\bar{P}$ is algebraically independent over $B$.
The completion of $B[\bar{P}]$ with respect to the prime ideal $\bar{P} B[\bar{P}]$ is equal to $A$.

Consider any $\phi\in A$.
By 2 we know that there exists uniquely an element $\bar{c}\in\Map(\Map(\bar{P},\Z_0), B)$ satisfying
$$\phi=\sum_{\bar{\Lambda}\in\Map(\bar{P},\Z_0)} \bar{c}(\bar{\Lambda})\prod_{x\in\bar{P}}x^{\bar{\Lambda}(x)}.$$
We take the unique element $\bar{c}\in\Map(\Map(\bar{P},\Z_0), B)$ satisfying the above equality.
We denote
$$\Supp(\bar{P},\phi)=\Supp(\bar{c})
=\{\bar{\Lambda}\in\Map(\bar{P},\Z_0)| \bar{c}(\bar{\Lambda})\neq 0\}\subset\Map(\bar{P},\Z_0).$$

We know that the following claims hold:
\begin{enumerate}
\item
$\Supp(\bar{P},\phi)=\pi(\Supp(P,\phi))$.
\item
$\{\langle\bar{\omega},\bar{\Lambda}\rangle|\bar{\Lambda}\in\Supp(\bar{P},\phi)\}\subset\R_0$.
\item
$\{\langle\bar{\omega},\bar{\Lambda}\rangle|\bar{\Lambda}\in\Supp(\bar{P},\phi)\}=\emptyset$, if and only if, $\phi=0$.
\item
If $\phi\neq 0$, then 
the minimum element $\min\{\langle\bar{\omega},\bar{\Lambda}\rangle|\bar{\Lambda}\in\Supp(\bar{P},\phi)\}$ of the non-empty subset $\{\langle\bar{\omega},\bar{\Lambda}\rangle|\bar{\Lambda}\in\Supp(\bar{P},\phi)\}$ of $\R_0$ exists.
\item
If $\phi\neq 0$, then 
$\Ord(P,\omega,\phi)= \min\{\langle\bar{\omega},\bar{\Lambda}\rangle|\bar{\Lambda}\in\Supp(\bar{P},\phi)\}$.
\end{enumerate}

Consider any $\phi\in A$ and any $\psi\in A$.
We would like to show the equality $\Ord(P,\omega,\phi\psi)=\Ord(P,\omega,\phi)+\Ord(P,\omega,\psi)$.

If $\phi=0$ or $\psi=0$, then we have $\phi\psi=0$, $\Ord(P,\omega,\phi)=\infty$ or $\Ord(P,\omega,\psi)=\infty$, and $\Ord(P,\omega,\phi\psi)=\infty=\Ord(P,\omega,\phi)+\Ord(P,\omega,\psi)$.

Below we assume $\phi\neq 0$ and $\psi\neq 0$.

Take elements $\bar{c}\in\Map(\Map(\bar{P},\Z_0), B)$, $\bar{d}\in\Map(\Map(\bar{P},\Z_0), B)$, and
$\bar{e}\in\Map(\Map(\bar{P},\Z_0), B)$ satisfying
\begin{equation*}
\begin{split}
\phi=&\sum_{\bar{\Lambda}\in\Map(\bar{P},\Z_0)} \bar{c}(\bar{\Lambda})\prod_{x\in\bar{P}}x^{\bar{\Lambda}(x)},\\
\psi=&\sum_{\bar{\Lambda}\in\Map(\bar{P},\Z_0)} \bar{d}(\bar{\Lambda})\prod_{x\in\bar{P}}x^{\bar{\Lambda}(x)},\\
\phi\psi=&\sum_{\bar{\Lambda}\in\Map(\bar{P},\Z_0)} \bar{e}(\bar{\Lambda})\prod_{x\in\bar{P}}x^{\bar{\Lambda}(x)}.
\end{split}
\end{equation*}

We know that $\bar{e}(\bar{\Lambda})=
\sum_{(\bar{\Gamma},\bar{\Delta})\in\delta(\bar{\Lambda})}\bar{c}(\bar{\Gamma})\bar{d}(\bar{\Delta})$ for any $\bar{\Lambda}\in\Map(\bar{P},\Z_0)$.

If $\phi\psi=0$, then $\Ord(P,\omega,\phi\psi)=\infty\geq\Ord(P,\omega,\phi)+ \Ord(P,\omega,\psi)$. We consider the case $\phi\psi\neq 0$.

We take any $\bar{\Lambda}\in\Supp(\bar{P},\phi)$ with $\Ord(P,\omega,\phi\psi)=\langle\bar{\omega},\bar{\Lambda}\rangle$.
$0\neq \bar{e}(\bar{\Lambda})=$\hfill\break$
\sum_{(\bar{\Gamma},\bar{\Delta})\in\delta(\bar{\Lambda})}\bar{c}(\bar{\Gamma})\bar{d}(\bar{\Delta})$.
We know that there exists $(\bar{\Gamma},\bar{\Delta})\in\delta(\bar{\Lambda})$ satisfying $\bar{c}(\bar{\Gamma})\neq 0$ and $\bar{d}(\bar{\Delta})\neq 0$.
We take an element $(\bar{\Gamma},\bar{\Delta})\in\delta(\bar{\Lambda})$ satisfying $\bar{c}(\bar{\Gamma})\neq 0$ and $\bar{d}(\bar{\Delta})\neq 0$.
We know $\bar{\Lambda}=\bar{\Gamma}+\bar{\Delta}$, $\bar{\Lambda}\in\Supp(\bar{P},\phi)$, $\bar{\Delta}\in\Supp(\bar{P},\psi)$,
$\langle\bar{\omega},\bar{\Lambda}\rangle\geq\Ord(P,\omega,\phi)$, and $\langle\bar{\omega},\bar{\Delta}\rangle\geq\Ord(P,\omega,\psi)$.
Therefore,
$\Ord(P,\omega,\phi\psi)=\langle\bar{\omega},\bar{\Lambda}\rangle
=\langle\bar{\omega},\bar{\Gamma}+\bar{\Delta}\rangle
=\langle\bar{\omega},\bar{\Gamma}\rangle+\langle\bar{\omega},\bar{\Delta}\rangle
\geq\Ord(P,\omega,\phi)+ \Ord(P,\omega,\psi)$.

We know $\Ord(P,\omega,\phi\psi) \geq\Ord(P,\omega,\phi)+ \Ord(P,\omega,\psi)$.

Now, the set $\{\bar{\Gamma}\in\Supp(\bar{P},\phi)|\langle\bar{\omega},\bar{\Gamma}\rangle=\Ord(P,\omega,\phi)\}$ is a non-empty finite subset of $\Supp(\bar{P},\phi)$. 
Let $\bar{\Gamma}_0$ denote the minimum element in this set with respect to the lexicographic order $\leq$.
$\bar{\Gamma}_0\in\Supp(\bar{P},\phi)$
$\langle\bar{\omega},\bar{\Gamma}_0\rangle=\Ord(P,\omega,\phi)$.
If $\bar{\Gamma}\in\Supp(\bar{P},\phi)$ and $\langle\bar{\omega},\bar{\Gamma}\rangle=\Ord(P,\omega,\phi)$, then $\bar{\Gamma}_0\leq\bar{\Gamma}$.
The set $\{\bar{\Delta}\in\Supp(\bar{P},\psi)|\langle\bar{\omega},\bar{\Delta}\rangle=\Ord(P,\omega,\psi)\}$ is a non-empty finite subset of $\Supp(\bar{P},\psi)$. 
Let $\bar{\Delta}_0$ denote the minimum element in this set with respect to the lexicographic order $\leq$.
$\bar{\Delta}_0\in\Supp(\bar{P},\psi)$
$\langle\bar{\omega},\bar{\Delta}_0\rangle=\Ord(P,\omega,\psi)$.
If $\bar{\Delta}\in\Supp(\bar{P},\psi)$ and $\langle\bar{\omega},\bar{\Delta}\rangle=\Ord(P,\omega,\psi)$, then $\bar{\Delta}_0\leq\bar{\Delta}$.

Let $\bar{\Lambda}_0=\bar{\Gamma}_0+\bar{\Delta}_0\in\Map(\bar{P},\Z_0)$.

$\bar{e}(\bar{\Lambda}_0)=
\sum_{(\bar{\Gamma},\bar{\Delta})\in\delta(\bar{\Lambda}_0)}\bar{c}(\bar{\Gamma})\bar{d}(\bar{\Delta})$.
$(\bar{\Gamma}_0,\bar{\Delta}_0)\in\delta(\bar{\Lambda}_0)$, $\bar{c}(\bar{\Gamma}_0)\neq 0$, and $\bar{d}(\bar{\Delta}_0)\neq 0$.

Consider any $(\bar{\Gamma},\bar{\Delta})\in\delta(\bar{\Lambda}_0)$ satisfying 
$\bar{c}(\bar{\Gamma})\neq 0$, and $\bar{d}(\bar{\Delta})\neq 0$.

We know $\bar{\Gamma}\in\Supp(\bar{P},\phi)$, $\bar{\Delta}\in\Supp(\bar{P},\psi)$,
$\langle\bar{\omega},\bar{\Gamma}\rangle\geq\Ord(P,\omega,\phi)$,
$\langle\bar{\omega},\bar{\Delta}\rangle\geq\Ord(P,\omega,\psi)$, and
$\bar{\Gamma}+\bar{\Delta}=\bar{\Lambda}_0=\bar{\Gamma}_0+\bar{\Delta}_0$.

Therefore, $\langle\bar{\omega},\bar{\Gamma}\rangle+\langle\bar{\omega},\bar{\Delta}\rangle
=\langle\bar{\omega},\bar{\Gamma}+\bar{\Delta}\rangle
=\langle\bar{\omega},\bar{\Gamma}_0+\bar{\Delta}_0\rangle
=\langle\bar{\omega},\bar{\Gamma}_0\rangle+\langle\bar{\omega},\bar{\Delta}_0\rangle
=\Ord(P,\omega,\phi)+\Ord(P,\omega,\psi)$, and
$\langle\bar{\omega},\bar{\Gamma}\rangle=\Ord(P,\omega,\phi)$,
$\langle\bar{\omega},\bar{\Delta}\rangle=\Ord(P,\omega,\psi)$.

We know $\bar{\Gamma}\geq\bar{\Gamma}_0$ and $\bar{\Delta}\geq\bar{\Delta}_0$.
Since $\bar{\Gamma}=\bar{\Gamma}+\bar{\Delta}-\bar{\Delta}=\bar{\Gamma}_0+\bar{\Delta}_0-\bar{\Delta}\leq\bar{\Gamma}_0+\bar{\Delta}-\bar{\Delta}=\bar{\Gamma}_0$, we know $\bar{\Gamma}=\bar{\Gamma}_0$.
Therefore, $\bar{\Delta}=\bar{\Gamma}+\bar{\Delta}-\bar{\Gamma}
=\bar{\Gamma}_0+\bar{\Delta}_0-\bar{\Gamma}_0=\bar{\Delta}_0$.

We conclude $\bar{\Gamma}=\bar{\Gamma}_0$ and $\bar{\Delta}=\bar{\Delta}_0$.

We know $\{(\bar{\Gamma},\bar{\Delta})\in\delta(\bar{\Lambda}_0)|\bar{c}(\bar{\Gamma})\neq 0, \bar{d}(\bar{\Delta})\neq 0\}=\{(\bar{\Gamma}_0,\bar{\Delta}_0)\}$, 
$\bar{e}(\bar{\Lambda}_0)= $\hfill\break$\sum_{(\bar{\Gamma},\bar{\Delta})\in\delta(\bar{\Lambda}_0)} \bar{c}(\bar{\Gamma})\bar{d}(\bar{\Delta})=\bar{c}(\bar{\Gamma}_0) \bar{d}(\bar{\Delta}_0)\neq 0$, and
$\bar{\Lambda}_0\in\Supp(\bar{P},\phi\psi)$.
Therefore,
$\Ord(P,\omega, \phi\psi)\leq\langle\bar{\omega},\bar{\Lambda}_0\rangle
=\langle\bar{\omega},\bar{\Gamma}_0+\bar{\Delta}_0\rangle
=\langle\bar{\omega},\bar{\Gamma}_0\rangle+\langle\bar{\omega},\bar{\Delta}_0\rangle
=\Ord(P,\omega,\phi)+ \Ord(P,\omega,\psi)$.

We conclude $\Ord(P,\omega, \phi\psi)\leq\Ord(P,\omega,\phi)+ \Ord(P,\omega,\psi)$ and $\Ord(P,\omega, \phi\psi)=\Ord(P,\omega,\phi)+ \Ord(P,\omega,\psi)$.

We know that for any $\phi\in A$, any $\psi\in A$ and any $\omega\in\Map(P,\R_0)^\vee$, the equality $\Ord(P,\omega, \phi\psi)=\Ord(P,\omega,\phi)+ \Ord(P,\omega,\psi)$ holds.

Consider any $\phi\in A$, any $\psi\in A$ and any $\omega\in\Map(P,\R_0)^\vee$.
We would like to show the equality
$\In(P,\omega, \phi\psi)=\In(P,\omega,\phi)\In(P,\omega,\psi)$.

If $\phi=0$ or $\psi=0$, then $\phi\psi=0$, $\In(P,\omega,\phi)=0$ or $\In(P,\omega,\psi)=0$, and $\In(P,\omega, \phi\psi)=0=\In(P,\omega,\phi)\In(P,\omega,\psi)$.

Below we assume $\phi\neq 0$ and $\psi\neq 0$.
By 11 we have $\Ord(P,\omega, \phi-\In(P,\omega,\phi))>\Ord(P,\omega, \phi)$, $\Ord(P,\omega, \In(P,\omega,\phi))=\Ord(P,\omega, \phi)$,
$\Ord(P,\omega, \psi-\In(P,\omega,\psi))>$\hfill\break$\Ord(P,\omega, \psi)$, $\Ord(P,\omega, \In(P,\omega,\psi))=\Ord(P,\omega, \psi)$.

Therefore, $\Ord(P,\omega, (\phi-\In(P,\omega,\phi))(\psi-\In(P,\omega,\psi)))
=\Ord(P,\omega, \phi-\In(P,\omega,\phi))+\Ord(P,\omega, \psi-\In(P,\omega,\psi))
>\Ord(P,\omega, \phi)+ \Ord(P,\omega, \psi)
=\Ord(P,\omega, \phi\psi)$,\hfill\break
$\Ord(P,\omega, \In(P,\omega,\phi)(\psi-\In(P,\omega,\psi)))
=\Ord(P,\omega, \In(P,\omega,\phi))+\Ord(P,\omega, \psi-$\hfill\break$\In(P,\omega,\psi))
>\Ord(P,\omega, \phi)+ \Ord(P,\omega, \psi)
=\Ord(P,\omega, \phi\psi)$, and
$\Ord(P,\omega, $\hfill\break$(\phi-\In(P,\omega,\phi))\In(P,\omega,\psi))
=\Ord(P,\omega, \phi-\In(P,\omega,\phi))+\Ord(P,\omega, \In(P,\omega,\psi))
>\Ord(P,\omega, \phi)+ \Ord(P,\omega, \psi)
=\Ord(P,\omega, \phi\psi)$.

Since $\phi\psi-\In(P,\omega,\phi) \In(P,\omega,\psi)
=(\phi-\In(P,\omega,\phi))(\psi-\In(P,\omega,\psi))+ \In(P,\omega,\phi)$\hfill\break$(\psi-\In(P,\omega,\psi))+ (\phi-\In(P,\omega,\phi))\In(P,\omega,\psi)$,
we know $\Ord(P,\omega, \phi\psi-\In(P,\omega,\phi) $\hfill\break$\In(P,\omega,\psi))
>\Ord(P,\omega, \phi\psi)$ by 14.

Take the elements $c\in\Map(\Map(P,\Z_0),k)$, $d\in\Map(\Map(P,\Z_0),k)$ satisfying $\phi=\sum_{\Gamma\in\Map(P,\Z_0)}c(\Gamma)\prod_{x\in P}x^{\Gamma(x)}$ and $\psi=\sum_{\Delta\in\Map(P,\Z_0)}d(\Delta)\prod_{x\in P}x^{\Delta(x)}$.
We know
$\Supp(P,\omega,\phi)=\{\Gamma\in\Map(P,\Z_0)| c(\Gamma)\neq 0, \langle\omega,\Gamma\rangle=\Ord(P,\omega,\phi)\}\neq\emptyset$,
$\In(P,\omega,\phi)= \sum_{\Gamma\in\Supp(P,\omega,\phi)}c(\Gamma)\prod_{x\in P}x^{\Gamma(x)}\neq 0$,
$\Supp(P,\omega,\psi)=$\hfill\break$\{\Delta\in\Map(P,\Z_0)| d(\Delta)\neq 0, \langle\omega,\Delta\rangle=\Ord(P,\omega,\psi)\}\neq\emptyset$,
$\In(P,\omega,\psi)= $\hfill\break$\sum_{\Delta\in\Supp(P,\omega,\psi)}d(\Delta)\prod_{x\in P}x^{\Delta(x)}\neq 0$.

We know
\begin{equation*}\begin{split}
0&\neq\In(P,\omega,\phi)\In(P,\omega,\psi)\\
&=
(\sum_{\Gamma\in\Supp(P,\omega,\phi)}c(\Gamma)\prod_{x\in P}x^{\Gamma(x)}) ( \sum_{\Delta\in\Supp(P,\omega,\psi)}d(\Delta)\prod_{x\in P}x^{\Delta(x)})\\
&=
\sum_{\Lambda\in\Map(P,\Z_0)}\sum_{(\Gamma,\Delta)\in\Supp(P,\omega,\phi)\times\Supp(P,\omega,\psi), \:\Gamma+\Delta=\Lambda}c(\Gamma)d(\Delta)\prod_{x\in P}x^{\Lambda(x)}.
\end{split}\end{equation*}

Conside any $\Lambda\in\Supp(P, \In(P,\omega,\phi)\In(P,\omega,\psi))$.
$\Lambda\in\Map(P,\Z_0)$ and we know $$\sum_{(\Gamma,\Delta)\in\Supp(P,\omega,\phi)\times\Supp(P,\omega,\psi), \:\Gamma+\Delta=\Lambda}c(\Gamma)d(\Delta)\neq 0.$$
We know that there exist $\Gamma\in\Supp(P,\omega,\phi)$ and $\Delta\in\Supp(P,\omega,\psi)$ satisfying $\Gamma+\Delta=\Lambda$.
We take $\Gamma\in\Supp(P,\omega,\phi)$ and $\Delta\in\Supp(P,\omega,\psi)$ satisfying $\Gamma+\Delta=\Lambda$.
We know $\langle\omega,\Gamma\rangle=\Ord(P,\omega,\phi)$ and  $\langle\omega,\Delta\rangle=\Ord(P,\omega,\psi)$.
Therefore,
$\langle\omega,\Lambda\rangle
=\langle\omega, \Gamma+\Delta\rangle
=\langle\omega, \Gamma\rangle+\langle\omega, \Delta\rangle
=\Ord(P,\omega,\phi)+ \Ord(P,\omega,\psi)
=\Ord(P,\omega,\In(P,\omega,\phi))+ \Ord(P,\omega,\In(P,\omega,\psi))
=\Ord(P,\omega, \In(P,\omega,\phi)\In(P,\omega,\psi))$.

We know that for any $\Lambda\in\Supp(P, \In(P,\omega,\phi)\In(P,\omega,\psi))$, 
$\langle\omega,\Lambda\rangle=\Ord(P,\omega, $\hfill\break$\In(P,\omega,\phi) \In(P,\omega,\psi))$,
$\Supp(P, \In(P,\omega,\phi)\In(P,\omega,\psi))= \Supp(P,\omega, \In(P,\omega,\phi) $\hfill\break$\In(P,\omega,\psi))$, and
$\In(P,\omega, \In(P,\omega,\phi)\In(P,\omega,\psi))= \In(P,\omega,\phi)\In(P,\omega,\psi)$.

By 11 we conclude $\In(P,\omega,\phi\psi)= \In(P,\omega,\phi)\In(P,\omega,\psi)$.

We know that for any $\phi\in A$, any $\psi\in A$ and any $\omega\in\Map(P,\R_0)^\vee$, the equality $\In(P,\omega, \phi\psi)=\In(P,\omega,\phi)\In(P,\omega,\psi)$ holds.
\end{proof}

Recall that we denote the set of all mappings from $X$ to $Y$ by $\Map(X,Y)$ for any sets $X$ and $Y$.
When $Y$ is a subset of an abelian group $Z$ with $0\in Y$, we denote
$$\Map'(X,Y)=\{a\in\Map(X,Y)| \Supp(a)\text{ is a finite set.}\}.$$
(Section~\ref{concept}.)

Note that for any set $X$ and any ring $R$,
the set $\Map(X,R)$ has the natural structure of an $R$-module, and $\Map'(X,R)$ is an $R$-submodule of  $\Map(X,R)$.

Let $R$ be any ring; let $S$ be any subring of $R$, and let $P$ be any finite subset of $R$. 

We call $P$ a \emph{variable system} of $R$ over $S$, if for any $\phi\in R$ there exists uniquely an element $c\in\Map'(\Map(P,\Z_0),S)$ satisfying
$$\phi=\sum_{\Lambda\in\Map(P,\Z_0)}c(\Lambda)\prod_{x\in P}x^{\Lambda(x)}.$$
Any element of a variable system of $R$ over $S$ is called a \emph{variable} of $R$ over $S$.
We call $R$ \emph{a polynomial ring over} $S$, if there exists a variable system $P$ of $R$ over $S$.
We call any element $\phi$ in $R$ a \emph{polynomial over} $S$, if $R$ is a polynomial ring over $S$.

Assume that $R$ is a polynomial ring over $S$ and $P$ is a variable system of $R$ over $S$.
We consider any element $\phi\in R$.
We take the unique element $c\in\Map'(\Map(P,\Z_0),S)$ satisfying
$\phi=\sum_{\Lambda\in\Map(P,\Z_0)}c(\Lambda)\prod_{x\in P}x^{\Lambda(x)}$.
We denote
$$\Supp(P,\phi)=\Supp(c)=\{\Lambda\in\Map(P,\Z_0)|c(\Lambda)\neq 0\}\subset\Map(P,\Z_0),$$
and we call $\Supp(P,\phi)$ the \emph{support} of $\phi$ over $P$.
$\Supp(P,\phi)$ is a \emph{finite} subset of $\Map(P,\Z_0)$.

Note that $\phi=0\Leftrightarrow c=0\Leftrightarrow \Supp(P,\phi)=\emptyset$.

We introduce a symbol $-\infty$ satisfying $(-\infty)+(-\infty)=-\infty$, $(-\infty)=-(\infty)$, $-(-\infty)=\infty$, $-\infty<\infty$, $\infty>-\infty$, $-\infty\neq\infty$, $\infty\neq-\infty$, $-\infty\leq\infty$, $\infty\geq-\infty$ and satisfying
$(-\infty)+t=t+(-\infty)=-\infty$,  $-\infty<t$, $t>-\infty$, $-\infty\neq t$, $t\neq-\infty$, $-\infty\leq t$, $t\geq-\infty$ for any $t\in\R$.

Putting $\delta_0(\Lambda)=\sum_{x\in P}\Lambda(x)\in\R$ for any $\Lambda\in\Map(P,\R)$, we define an element $\delta_0\in
(\Map(P,\R_0)^\vee|\Map(P,\R))^\circ\cap\Map(P,\Z)^*\subset\Map(P,\R)^*$.
The element $\delta_0$ is equal to the barycenter of the simplicial cone $\Map(P,\R_0)^\vee|\Map(P,\R)$ over $\Map(P,\Z)^*$.
(See Definition~\ref{barycenter}.)

Let $\omega\in\Map(P,\R)^*$ be any element. We denote
\begin{equation*}
\deg(P,\omega,\phi)=
\begin{cases}
\max\{\langle \omega,\Lambda\rangle|\Lambda\in\Supp(P,\phi)\}& \text{ if $\phi\neq0$ },\\
-\infty& \text{ if $\phi=0$ },
\end{cases}
\end{equation*}
and
$$\deg(P,\phi)=\deg(P,\delta_0,\phi).$$
We call $\deg(P,\omega,\phi)\in\R\cup\{-\infty\}$ the \emph{degree} of $\phi$ over $P$ with respect to $\omega$, and we call $\deg(P, \phi)\in\Z_0\cup\{-\infty\}$
the \emph{degree} of $\phi$ over $P$.

\begin{lemma}
\label{polynomial ring}
Let $R$ be any ring; let $S$ be any subring of $R$, and let $P$ be any finite subset of $R$. 
\begin{enumerate}
\item
The following two conditions are equivalent:
\begin{enumerate}
\item
$P$ is a variable system of $R$ over $S$.
\item
For any ring $T$ such that there exists a ring homomorphism from $S$ to $T$, any ring homomorphism $\mu:S\rightarrow T$ and any $\kappa\in\Map(P,T)$, there exists uniquely a ring homomorphism $\lambda:R\rightarrow T$ satisfying $\lambda(s)=\mu(s)$ for any $s\in S$ and $\lambda(x)=\kappa(x)$ for any $x\in P$.
\end{enumerate}
\end{enumerate}

Below, we assume that $P$ is a variable system of $R$ over $S$ and $R$ is a polynomial ring over $S$.
\begin{enumerate}
\setcounter{enumi}{1}
\item
Let $(\bar{R},\bar{\lambda})$ be an $S$-algebra isomorphic to $R$ and let $\mu:R\rightarrow\bar{R}$ be any $S$-isomorphism.

Then, $\bar{\lambda}:S\rightarrow \bar{R}$ is injective, $\bar{R}$ is a polynomial ring over $\bar{\lambda}(S)$, and $\mu(P)$ is a variable system of $\bar{R}$ over $\bar{\lambda}(S)$.
\item
For any variable system $\bar{P}$ of $R$ over $S$, $\sharp\bar{P}=\sharp P$.
\item
Let $\mathcal{I}$ denote the set of all ideals $I$ in $R$ such that the residue ring $R/I$ is isomorphic to $S$ as $S$-algebras.

For any $a\in\Map(P,S)$, the subset $\{x-a(x)|x\in P\}$ of $R$ is a variable system of $R$ over $S$ and  $\{x-a(x)|x\in P\}R\in\mathcal{I}$.

The mapping from $\Map(P,S)$ to $\mathcal{I}$ sending $a\in\Map(P,S)$ to $\{x-a(x)|x\in P\}R\in\mathcal{I}$ is bijective.

If $S$ is an integral domain, then $\mathcal{I}\subset\Spec(R)$. If $S$ is a field, then $\mathcal{I}$ is equal to the set of $S$-valued points $\Spec(R)(S)$ on $\Spec(R)$. (Section~\ref{scheme}.)
\item
Consider any $\omega\in\Map(P,\R)^*$, any $\phi\in R$, any $\psi\in R$ and any non-zero element $a\in S-\{0\}$.
\begin{enumerate}
\item
$\Supp(P,\phi)=\emptyset\Leftrightarrow \deg(P,\omega,\phi)=-\infty\Leftrightarrow \phi=0$.
\item
$\Supp(P, \phi+\psi)\subset\Supp(P,\phi)\cup\Supp(P,\psi)$.
\item
$\Supp(P, a\phi)\subset\Supp(P,\phi)$.
If $S$ is an integral domain, then $\Supp(P, a\phi)=\Supp(P,\phi)$.
\item
$\Supp(P, \phi\psi)\subset\Supp(P,\phi)+\Supp(P,\psi)$.
\item
$\deg(P,\omega,\phi+\psi)\leq\max\{\deg(P,\omega,\phi), \deg(P,\omega,\psi)\}$.
\item
$\deg(P,\omega,\phi)=\deg(P,\omega,-\phi)$.
\item
If $\deg(P,\omega,\phi)\neq\deg(P,\omega,\psi)$, then 
$$\deg(P,\omega,\phi+\psi)=\max\{\deg(P,\omega,\phi), \deg(P,\omega,\psi)\}.$$
\item
$\deg(P,\omega,a\phi)\leq\deg(P,\omega,\phi)$.
If $S$ is an integral domain, then \hfill\break$\deg(P,\omega,a\phi)=\deg(P,\omega,\phi)$.
\item
$\deg(P,\omega,\phi\psi)\leq\deg(P,\omega,\phi)+\deg(P,\omega,\psi)$.
If $S$ is an integral domain, then $\deg(P,\omega,\phi\psi)=\deg(P,\omega,\phi)+\deg(P,\omega,\psi)$.
\end{enumerate}
\item
$R$ is an integral domain, if and only if, $S$ is an integral domain.

If these equivalent conditions are satisfied, then $R^\times=S^\times$.
\item
$R$ is noetherian, if and only if, $S$ is noetherian.
\end{enumerate}

Below, furthermore, we assume that $S$ is a field. 
\begin{enumerate}
\setcounter{enumi}{7}
\item
$R$ is a noetherian integral domain. $\dim R=\sharp P$. 
\end{enumerate}

Let $K$ denote the quotient field of $R$. Let $\iota:R\rightarrow K$ denote the canonical homomorphism. The homomorphism $\iota$ is an injective ring homomorphism. 
Using $\iota$, we regard $R$ as a subring of $K$.
\begin{enumerate}
\setcounter{enumi}{8}
\item
$PR\in\Spec(R)$ and the local ring $R_{PR}$ of $R$ at $PR$ is defined.

The ring $R_{PR}$ is a subring of $K$, it is a regular local ring, it contains $S$ and $R$ as subrings, the residue field $R_{PR}/M(R_{PR})$ is isomorphic to $S$ as $S$-algebras, and $P$ is a parameter system of $R_{PR}$.
\item
$\Map(P,\R_0)^\vee|\Map(P,\R)\subset\Map(P,\R)^*$.
\item
$\Ord(P,\omega,\phi)=-\deg(P,-\omega,\phi)$ for any $\omega\in\Map(P,\R_0)^\vee|\Map(P,\R)$ and any $\phi\in R$.
\item
Consider any $\omega\in\Map(P,\R_0)^\vee|\Map(P,\R)$ and any $\phi\in R$.

If $\phi\neq 0$, then $\Ord(P,\omega,\phi)\leq\deg(P,\omega,\phi)$. 

$\Ord(P,\omega,\phi)=\deg(P,\omega,\phi)$, if and only if, $\phi\neq 0$ and $\phi=\In(P,\omega,\phi)$.
\end{enumerate}
\end{lemma}

Let $R$ be any ring and let $S$ be any subring of $R$. Assume that $R$ is a polynomial ring over $S$.

If $n\in\Z_0$ and there exists a variable system $P$ of $R$ over $S$ with $n=\sharp P\in\Z_0$, then we say that $R$ is a polynomial ring over $S$ \emph{with} $n$ \emph{variables}. By the claim 3 of the above lemma, if $n\in\Z_0$ and $R$ is a polynomial ring over $S$ with $n$ variables, then $n=\sharp\bar{P}$ for any variable system $\bar{P}$ of $R$ over $S$.

\section{Basic theory of convex sets}
\label{btcs}

In this section to develop the theory of torus embeddings we begin the study of convex sets. The theory of convex sets will be applied to the proof of our main theorem, Therem~\ref{main} in Section~\ref{main proof}.

Let $V$ be any finite dimensional vector space over $\R$. 

In Section~\ref{concept} we defined eight mappings
$$\Conv,\Affi,\Cone,\Convcone,\Vect,\QVect,\Clos,\Stab:2^V\rightarrow 2^V.$$

\begin{lemma}
\label{a}
Let $X$ be any subset of $V$.
\begin{enumerate}
\item
\begin{equation*}
\begin{split}
\Conv(\emptyset)=&\Affi(\emptyset)=\Clos(\emptyset)=\emptyset,\\
\Cone(\emptyset)=&\Convcone(\emptyset)=\Vect(\emptyset)=\QVect(\emptyset)=\{0\},\\
\Stab(\emptyset)=&V.\\
\end{split}
\end{equation*}
\item
\begin{equation*}
\begin{split}
\Conv(X)=&\{a\in V|a=\sum_{x\in\Supp(\lambda)}\lambda(x)x
\text{ for some }\lambda\in\Map'(X,\R_0)\text{ with}\\
&\qquad\qquad\sum_{x\in\Supp(\lambda)}\lambda(x)=1\},\\
\Affi(X)= &\{a\in V|a=\sum_{x\in\Supp(\lambda)}\lambda(x)x
\text{ for some }\lambda\in\Map'(X,\R)\text{ with}\\
&\qquad\qquad\sum_{x\in\Supp(\lambda)}\lambda(x)=1\},
\end{split}\end{equation*}
\begin{equation*}
\Cone(X)=
\begin{cases}
\{a\in V|a=\lambda x\text{ for some }\lambda\in\R_0\text{ and some }x\in X\}&
\text{ if $X\neq\emptyset$},\\
\{0\}&\text{ if $X=\emptyset$},
\end{cases}\end{equation*}
\begin{equation*}\begin{split}
\Convcone(X)=&\{a\in V|a=\sum_{x\in\Supp(\lambda)}\lambda(x)x
\text{ for some }\lambda\in\Map'(X,\R_0)\},\\
\Vect(X)=&\{a\in V|a=\sum_{x\in\Supp(\lambda)}\lambda(x)x
\text{ for some }\lambda\in\Map'(X,\R)\},\\
\QVect(X)=&\{a\in V|a=\sum_{x\in\Supp(\lambda)}\lambda(x)x
\text{ for some }\lambda\in\Map'(X,\Q)\}.\\
\end{split}
\end{equation*}
\item
If $X$ is a finite set, then we have
$\Convcone(X)=\sum_{x\in X}\R_0 x$, $\Vect(X)=\sum_{x\in X}\R x$, and $\QVect(X)=\sum_{x\in X}\Q x$.
\item
For any finite dimensional vector space $W$ over $\R$ and any homomorphism $\pi:V\rightarrow W$ of vector spaces over $\R$, we have
$\pi(\Conv(X))=\Conv(\pi(X))$,
$\pi(\Affi(X))=\Affi(\pi(X))$,
$\pi(\Cone(X))=\Cone(\pi(X))$,
$\pi(\Convcone(X))=$\hfill\break$\Convcone(\pi(X))$,
$\pi(\Vect(X))=\Vect(\pi(X))$, and
$\pi(\QVect(X))=\QVect(\pi(X))$.
\item
For any $a\in V$, we have
$\Conv(X+\{a\})=\Conv(X)+\{a\}$,
$\Affi(X+\{a\})=\Affi(X)+\{a\}$.
\item
\begin{equation*}
\begin{gathered}
X\subset\Conv(X)\subset\Affi(X)\subset\Vect(X),\\
X\cup\{0\}\subset\Cone(X)\subset\Convcone(X)\subset\Vect(X),\\
\Conv(X)\subset\Convcone(X),\\
X\cup\{0\}\subset\QVect(X)\subset\Vect(X),\\
X\subset\Clos(X).\\
\end{gathered}
\end{equation*}
\item
For any subset $Y$ of $V$ with $X\subset Y$, we have $\Conv(X)\subset\Conv(Y)$, $\Affi(X)\subset\Affi(Y)$, $\Cone(X)\subset\Cone(Y)$, $\Convcone(X)\subset\Convcone(Y)$, $\Vect(X)\subset\Vect(Y)$, $\QVect(X)\subset\QVect(Y)$, and $\Clos(X)\subset\Clos(Y)$.
\item
$\Conv(\Conv(X))=\Conv(X)$, $\Affi(\Affi(X))=\Affi(X)$, $\Cone(\Cone(X))=\Cone(X)$, $\Convcone(\Convcone(X))=\Convcone(X)$,
$\Vect(\Vect(X))=\Vect(X)$,\hfill\break$\QVect(\QVect(X))=\QVect(X)$, and $\Clos(\Clos(X))=\Clos(X)$.
\item
\begin{equation*}
\begin{split}
\Vect(\Conv(X))=&\Conv(\Vect(X))=\Vect(X),\\
\Vect(\Affi(X))=&\Affi(\Vect(X))=\Vect(X),\\
\Vect(\Cone(X))=&\Cone(\Vect(X))=\Vect(X),\\
\Vect(\Convcone(X))=&\Convcone(\Vect(X))=\Vect(X).\\
\end{split}
\end{equation*}
\item
\begin{equation*}
\begin{split}
\Convcone(\Conv(X))=&\Conv(\Convcone(X))=\Convcone(X),\\
\Convcone(\Cone(X))=&\Cone(\Convcone(X))=\Convcone(X).\\
\Cone(\Conv(X))=&\Conv(\Cone(X))=\Convcone(X),\\
\Convcone(\Affi(X))=\Cone(\Affi(X))\subset&\Affi(\Convcone(X))= \Affi(\Cone(X))=\Vect(X),\\
\end{split}
\end{equation*}
\item
$\Affi(\Conv(X))=\Conv(\Affi(X))=\Affi(X)$.
\item
\begin{equation*}
\begin{split}
\Clos(\Affi(X))=&\Affi(\Clos(X))=\Affi(X),\\
\Clos(\QVect(X))=&\Clos(\Vect(X))=\Vect(\Clos(X))=\Vect(X),\\
\Clos(\Conv(X))=&\Conv(\Clos(\Conv(X))),\\
\Clos(\Cone(X))=&\Cone(\Clos(\Cone(X))),\\
\Clos(\Convcone(X))=&\Convcone(\Clos(\Convcone(X))).\\
\end{split}
\end{equation*}
\end{enumerate}
\end{lemma}

\begin{lemma}
\label{b}
Let $X$ and $Y$ be any subsets of $V$.
\begin{enumerate}
\item
$\Conv(X)+\Conv(Y)=\Conv(X+Y)$.
\item
$\Affi(X)+\Affi(Y)=\Affi(X+Y)$.
\item
$\Cone(X\cup Y)\subset\Cone(X)+\Cone(Y)\subset\Conv(\Cone(X\cup Y))$.
\item
$\Convcone(X)+\Convcone(Y)=\Convcone(X\cup Y)$.
\item
$\Vect(X)+\Vect(Y)=\Vect(X\cup Y)$.
\item 
For any finite dimensional vector space $W$ over $\R$ and any homomorphism $\pi:V\rightarrow W$ of vector spaces over $\R$, we have
$\pi(X)+\pi(Y)=\pi(X+Y)$.
\end{enumerate}
\end{lemma}

In Section~\ref{concept} we defined concepts of segments, lines, convex sets, affine spaces, cones, convex cones, vector spaces, vector spaces over $\Q$, and closed subsets.

\begin{lemma}
\label{c}
Let $S$ be any non-empty subset of $V$.
\begin{enumerate}
\item
The following three conditions are equivalent;
\begin{enumerate}
\item $S$ is convex.
\item $S\supset\Conv(S)$.
\item $S=\Conv(X)$ for some non-empty subset $X$ of $V$.
\end{enumerate}
\item
If $S$ is convex, then $\Clos(S)$ is also convex, and $\Affi(S)=\Affi(\Clos(S))$.
\item
The following three conditions are equivalent;
\begin{enumerate}
\item $S$ is an affine space.
\item $S\supset\Affi(S)$.
\item $S=\Affi(X)$ for some non-empty subset $X$ of $V$.
\end{enumerate}
\item
The following three conditions are equivalent;
\begin{enumerate}
\item $S$ is an affine space containing $0$.
\item $S$ is an affine space with $S=\Stab(S)$.
\item $S$ is a vector space.
\end{enumerate}
\item
Assume that $S$ is an affine space. Then, $\Stab(S)$ is a vector space, and for any $x\in S$ we have $S=\Stab(S)+\{x\}$ and $\Stab(S)=S+\{-x\}$.
\item
Any affine space is closed and convex.
\item
$0\in\Stab(S)\subset\Stab(\Affi(S))\subset\Vect(S)$. $\Stab(S)+\Stab(S)=\Stab(S)$.
\item
The following three conditions are equivalent;
\begin{enumerate}
\item $S$ is a cone.
\item $S\supset\Cone(S)$.
\item $S=\Cone(X)$ for some non-empty subset $X$ of $V$.
\end{enumerate}
\item Any cone contains $0$.
\item
If $S$ is a cone, then $\Clos(S)$ is also a cone, and $\Vect(S)=\Vect(\Clos(S))$.
\item
The following four conditions are equivalent;
\begin{enumerate}
\item $S$ is a convex cone.
\item $S$ is convex and $S$ is a cone.
\item $S\supset\Convcone(S)$.
\item $S=\Convcone(X)$ for some non-empty subset $X$ of $V$.
\end{enumerate}
\item
Any convex cone contains $0$.
\item
If $S$ is a convex cone, then $\Clos(S)$ is also a convex cone, and $\Vect(S)=\Vect(\Clos(S))$.
\item
If $S$ is a convex cone, then $S\cap(-S)$ is the maximal vector space contained in $S$ with respect to the inclusion relation.
\item
The following three conditions are equivalent;
\begin{enumerate}
\item $S$ is a vector space.
\item $S\supset\Vect(S)$.
\item $S=\Vect(X)$ for some non-empty subset $X$ of $V$.
\end{enumerate}
\item
Any vector space contains $0$.
\item
Any vector space is closed, it is an affine space containing $0$, and it is a convex cone.
\item
The following three conditions are equivalent;
\begin{enumerate}
\item $S$ is a vector space over $\Q$.
\item $S\supset\QVect(S)$.
\item $S=\QVect(X)$ for some non-empty subset $X$ of $V$.
\end{enumerate}
\item
Any vector space over $\Q$ contains $0$.
\item
The following three conditions are equivalent;
\begin{enumerate}
\item $S$ is closed.
\item $S\supset\Clos(S)$.
\item $S=\Clos(X)$ for some non-empty subset $X$ of $V$.
\end{enumerate}
\end{enumerate}
\end{lemma}

\begin{lemma}
\label{d}
Let $S$ and $T$ be any subsets of $V$.
\begin{enumerate}
\item
If $S$ and $T$ are convex, then $S+T$ is convex. If $S$ and $T$ are convex and $S\cap T\neq\emptyset$, then $S\cap T$ is convex.
\item
If $S$ and $T$ are affine spaces, then $S+T$ is an affine space. If $S$ and $T$ are affine spaces and $S\cap T\neq\emptyset$, then $S\cap T$ is an affine space.
\item
If $S$ and $T$ are cones, then $S+T$ and $S\cap T$ are cones.
\item
If $S$ and $T$ are convex cones, then $S+T$ and $S\cap T$ are convex cones.
\item
If $S$ and $T$ are vector spaces, then $S+T$ and $S\cap T$ are vector spaces.
\item
If $S$ and $T$ are vector spaces over $\Q$, then $S+T$ and $S\cap T$ are vector spaces over $\Q$.
\item
If $S$ and $T$ are closed, then $S\cap T$ is closed.
\end{enumerate}
\end{lemma}

For any vector space $S$ over $\R$ of $V$, the dimension $\dim S$ of $S$ over $\R$ is defined. $\dim S\in\Z_0$ and $0\leq\dim S\leq \dim V$. 

\begin{definition}
\label{affine dimension}
For any affine space $S$ of $V$, we define
$$\dim S=\dim\Stab(S)\in\Z_0,$$
and we call $\dim S$ the \emph{dimension} of $S$.
\end{definition}

\begin{lemma}
\label{e}
\begin{enumerate}
\item
Let $S$ be an affine space of $V$. $\dim S\in\Z_0$, and $0\leq\dim S\leq\dim V$.
If $S$ contains $0$, then the dimension of $S$ as an affine space and the dimension of $S$ as a vector space are equal.
\item
Let $S$ and $T$ be affine spaces of $V$ with $S\subset T$. We have $\dim S\leq\dim T$, and $\dim S=\dim T$ if and only if $S=T$.
\end{enumerate}
\end{lemma}

\begin{definition}
\label{dimension}
 Let $S$ be any convex subset of $V$.
We define
$$\dim S=\dim\Affi(S)\in\Z_0,$$
and we call $\dim S$ the \emph{dimension} of $S$.

We define
\begin{equation*}
\begin{split}
\partial S&=S\cap\Clos(\Affi(S)- S),\\
S^\circ&=S-\Clos(\Affi(S)- S),
\end{split}
\end{equation*}
we call $\partial S$ the \emph{boundary} of $S$, and we call $S^\circ$ the \emph{interior} of $S$.
\end{definition}

\begin{lemma}
\label{f}
\begin{enumerate}
\item
Let $S$ be a convex subset of $V$. $\dim S\in\Z_0$, and $0\leq\dim S\leq\dim V$.
If $S$ is an affine space, then the dimension of $S$ as a convex set and the dimension of $S$ as an affine space are equal.
\item
For any convex set $S$ of $V$, we have $\dim S=\dim\Affi(S)=\dim \Clos(S)$.
\item
Let $S$ and $T$ be convex subsets of $V$ with $S\subset T$. We have $\dim S\leq\dim T$.
\item
Let $S$ be a convex subset of $V$.
$\partial S\cup S^\circ=S$.
$\partial S\cap S^\circ=\emptyset$.
$S^\circ$ is a non-empty open subset of $\Affi(S)$.
If $S$ is closed, then $\partial S$ is also closed.
\item
For any convex cone $S$ of $V$, we have $\Affi(S)=\Vect(S)$, $\dim S=\dim\Vect(S)$,
$\partial S=S\cap\Clos(\Vect(S)- S)$, $S^\circ=S-\Clos(\Vect(S)- S)$ and $S^\circ$ is a non-empty open subset of $\Vect(S)$.
\end{enumerate}
\end{lemma}

\begin{remark}
It does not follow $S=T$ from only the assumptions $S\subset T$ and $\dim S=\dim T$ for convex subsets $S$, $T$ of $V$.
\end{remark}

In Section~\ref{concept} we defined concepts of convex polyhedrons, convex polyhedral cones, and convex pseudo polyhedrons.

\begin{lemma}
\label{g}
\begin{enumerate}
\item
Any convex polyhedron in $V$ is convex, compact and closed.
\item
Any convex polyhedral cone in $V$ is a closed convex cone.
\item
Any convex pseudo polyhedron in $V$ is convex and closed.
\item
Any vector space in $V$ is a convex polyhedral cone.
\item
Any affine space in $V$ is a convex pseudo polyhedron.
Any convex polyhedral cone in $V$ is a convex pseudo polyhedron.
Any convex polyhedron in $V$ is a convex pseudo polyhedron.
\end{enumerate}
\end{lemma}

\begin{proof}
\noindent 1. We consider any convex polyhedron $S$ in $V$. We take any non-empty finite subset $X$ of $V$ with $S=\Conv(X)$.
By Lemma~\ref{a}.6 we have $S=\Conv(X)\supset X\neq\emptyset$. Therefore, $S\neq\emptyset$.
By Lemma~\ref{a}.8 we have $\Conv(S)=\Conv(\Conv(X))=\Conv(X)=S$.
By Lemma~\ref{c}.1 we know that $S$ is convex.

Let $\Lambda=\{\lambda\in\Map(X,\R_0)|\sum_{x\in X}\lambda(x)=1\}$, which is a compact subset of the finite dimensional vector space $\Map(X,\R)$.
Putting $\pi(\lambda)=\sum_{x\in X}\lambda(x)x\in V$ for any $\Lambda\in\Map(X,\R)$, we define a homomorphism $\pi:\Map(X,\R)\rightarrow V$ of vector spaces over $\R$.
By Lemma~\ref{a}.2 we know $S=\Conv(X)=\pi(\Lambda)$.
Since $\Lambda$ is compact and $\pi$ is continuous, we know that $S$ is compact and closed.

\noindent 2. We consider any convex polyhedral cone $S$ in $V$. We take any finite subset $X$ of $V$ with $S=\Convcone(X)$.
By Lemma~\ref{a}.6 we have $S=\Convcone(X)\ni 0$.
Therefore, $S\neq\emptyset$.
By Lemma~\ref{a}.8 we have $\Convcone(S)=\Convcone(\Convcone(X))=\Convcone(X)=S$.
By Lemma~\ref{c}.11 we know that $S$ is a convex cone.

Let $W=S\cap(-S)$.
By Lemma~\ref{c}.14 we know that $W$ is the maximal vector space contained in $S$.
Let $\pi:V\rightarrow V/W$ denote the canonical surjective homomorphism of vector spaces to the residue vector space $V/W$, and $\sigma:V/W\rightarrow V$ be an injective homomorphism of vector spaces satisfying $\pi\sigma(\bar{x})=\bar{x}$ for any $\bar{x}\in V/W$. Putting $\phi(x)=(x-\sigma\pi(x),\pi(x))\in W\times (V/W)$ for any $x\in V$, we have an isomorphism of vector spaces $\phi:V\rightarrow W\times (V/W)$.
We have $\phi(S)=W\times\pi(S)$, and $\pi(S)\cap(-\pi(S))=\{0\}$.
Therefore, in order to show that $S$ is closed, it suffice to show that $\pi(S)$ is closed, applying the condition $\pi(S)\cap(-\pi(S))=\{0\}$.
By Lemma~\ref{a}.4 we know $\pi(S)=\pi(\Convcone(X))=\Convcone(\pi(X))$ and $\pi(X)$ is a finite subset of $V/W$.

We show that $\pi(S)$ is closed.
If $\pi(S)=\{0\}$, then obviously  $\pi(S)$ is closed.
Below we assume $\pi(S)\neq\{0\}$ and denote $\bar{X}=\pi(X)-\{0\}$.
We know that $\bar{X}$ is a non-empty finite subset of $V/W$, $0\not\in\bar{X}$, and $\pi(S)=\Convcone(\bar{X})=\{x\in V/W|x=\sum_{x\in\bar{X}}\lambda(x)x\text{ for some }\lambda\in\Map(\bar{X},\R_0).\}$ by Lemma~\ref{a}.2.

For any $\lambda\in \Map(\bar{X},\R_0)$, we denote $|\lambda|=\sum_{x\in\bar{X}}\lambda(x)\in\R_0$.
For any $\lambda\in \Map(\bar{X},\R_0)$, we have $|\lambda|>0$ if and only if $\lambda\neq 0$.
We denote $\Lambda=\{\lambda\in\Map(\bar{X},\R_0)|$\hfill\break$|\lambda| =1\}$.
The set $\Lambda$ is a non-empty compact subset of $\Map(\bar{X},\R_0)$.
For any $\lambda\in \Map(\bar{X}, \R_0)-\{0\}$, we have $\lambda/|\lambda|\in\Lambda$.

Let $a\in\Clos(\pi(S))$ be any point. 
If $a=0$, then obviously we have $a=0\in\pi(S)$.
Below we assume $a\neq 0$.
We take an infinite sequence $\lambda(i), i\in\Z_0$ of elements of $\Map(\bar{X},\R_0)$ satisfying $a=\lim_{i\rightarrow\infty}\sum_{x\in\bar{X}}\lambda(i)(x)x$.

Assume that there exists $i_0\in\Z_0$ such that $\lambda(i)=0$ for any $i\in\Z_0$ with $i\geq i_0$.
We have $0\neq a=0$, a contradiction.
We know that for any $i_0\in\Z_0$ there exists $i\in\Z_0$ with $i\geq i_0$ and $\lambda(i)\neq 0$.

Below we assume that $\lambda(i)\neq 0$ for any $i\in\Z_0$, replacing $\lambda(i), i\in\Z_0$ by some subsequence.
We have an infinite sequence $\lambda(i)/|\lambda(i)|, i\in\Z_0$ of elements of $\Lambda$.

Below we assume that there exists the limit $\lim_{i\rightarrow\infty}\lambda(i)/|\lambda(i)|\in \Lambda$, replacing $\lambda(i), i\in\Z_0$ by some subsequence.
We denote $\bar{\lambda}=\lim_{i\rightarrow\infty}\lambda(i)/|\lambda(i)|\in \Lambda $. We know that there exists the limit $\lim_{i\rightarrow\infty}\sum_{x\in\bar{X}}(\lambda(i)(x) /|\lambda(i)|)x\in V/W$, and\hfill\break
$\lim_{i\rightarrow\infty}\sum_{x\in\bar{X}}(\lambda(i)(x) /|\lambda(i)|)x=
\sum_{x\in\bar{X}}\bar{\lambda}(x)x$.

We have 
$$a=\lim_{i\rightarrow\infty}|\lambda(i)|\sum_{x\in\bar{X}}(\frac{\lambda(i)(x)} {|\lambda(i)|})x.$$

We have two cases, the case where $\{|\lambda(i)||i\in\Z_0\}\subset\R_0$ is not bounded, and the case where $\{|\lambda(i)||i\in\Z_0\}\subset\R_0$ is bounded.

We consider the case where $\{|\lambda(i)||i\in\Z_0\}\subset\R_0$ is not bounded.

Below we assume that $\lim_{i\rightarrow\infty}|\lambda(i)|=\infty$, replacing $\lambda(i), i\in\Z_0$ by some subsequence.
We have
$$0=\frac{a}{\lim_{i\rightarrow\infty}|\lambda(i)|}=
\sum_{x\in\bar{X}}\bar{\lambda}(x)x.$$
Since $\bar{\lambda}\in \Lambda$, there exists $\bar{x}\in\bar{X}$ with
$\bar{\lambda}(\bar{x})>0$.
We take $\bar{x}\in\bar{X}$ with $\bar{\lambda}(\bar{x})>0$, and putting
$\mu(x)= \bar{\lambda}(x)/ \bar{\lambda}(\bar{x})\in\R_0$ for any $x\in\bar{X}-\{\bar{x}\}$, we define an element $\mu\in\Map(\bar{X}-\{\bar{x}\},\R_0)$.
We know that
$$\bar{x}=-\sum_{x\in\bar{X}-\{\bar{x}\}}\mu(x)x,$$
$\bar{x}\in\bar{X}\subset\Convcone(\bar{X})=\pi(S)$, and $-\sum_{x\in\bar{X}-\{\bar{x}\}}\mu(x)x\in-\Convcone(\bar{X})=-\pi(S)$ by Lemma~\ref{a}.2.
We conclude $\bar{x}\in\pi(S)\cap(-\pi(S))=\{0\}$ and $\bar{x}=0$.
On the other hand, since $\bar{x}\in\bar{X}\not\ni 0$, we have $\bar{x}\neq 0$, which is a contradiction. 
The case under consideration never occurs.

We consider the case where $\{|\lambda(i)||i\in\Z_0\}\subset\R_0$ is bounded.

Below we assume that there exists the limit $\lim_{i\rightarrow\infty}|\lambda(i)|\in\R_0$, replacing $\lambda(i), i\in\Z_0$ by some subsequence.
We have
$$a=\sum_{x\in\bar{X}}(\lim_{i\rightarrow\infty}|\lambda(i)|)\bar{\lambda}(x)x\in\Convcone(\bar{X},\R_0)=\pi(S),$$
by Lemma~\ref{a}.2.

We know that $\Clos(\pi(S))\subset\pi(S)$.
By Lemma~\ref{c}.20 we conclude that $\pi(S)$ is closed, and $S$ is closed.

\noindent 3. We consider any convex pseudo polyhedron $S$ in $V$. We take any finite subsets $X,Y$ of $V$ with $S=\Conv(X)+\Convcone(Y)$ and $X\neq\emptyset$.
By Lemma~\ref{a}.6 we know $\Convcone(Y)\ni 0$ and $S=\Conv(X)+\Convcone(Y)\supset\Conv(X)\supset X\neq\emptyset$.
Therefore, $S\neq\emptyset$.
By Lemma~\ref{b}.1, Lemma~\ref{a}.8 and Lemma~\ref{a}.10 we have
$\Conv(S)=\Conv(\Conv(X)+\Convcone(Y))=\Conv(\Conv(X))+\Conv(\Convcone(Y))=\Conv(X)+\Convcone(Y)=S$.
By Lemma~\ref{c}.1 we know that $S$ is convex.

We consider any element $a\in\Clos(S)$.
We take an infinite sequence $b(i),i\in\Z_0$ of elements of $S$ with $a=\lim_{i\rightarrow\infty}b(i)$.
For any $i\in\Z_0$ we take $c(i)\in\Conv(X)$ and $d(i)\in\Convcone(Y)$ with $b(i)=c(i)+d(i)$.
Note that $\Conv(X)$ is compact by $1$.

Below we assume that there exists the limit  $\lim_{i\rightarrow\infty}c(i)\in\Conv(X)$, replacing $b(i), i\in\Z_0$ by some subsequence.

Since $d(i)=b(i)-c(i)$ for any $i\in\Z_0$, we know that there exists the limit
$\lim_{i\rightarrow\infty}d(i)\in\Clos(\Convcone(Y))$.

By $2$ we know $\Clos(\Convcone(Y))=\Convcone(Y)$.
We know that $a=\lim_{i\rightarrow\infty}$\hfill\break$b(i)=\lim_{i\rightarrow\infty}(c(i)+d(i))=\lim_{i\rightarrow\infty}c(i)+\lim_{i\rightarrow\infty}d(i)
\in\Conv(X)+\Convcone(Y)=S$.

By Lemma~\ref{c}.20 we conclude that $\Clos(S)\subset S$ and $S$ is closed.

\noindent 4, 5. Easy.
\end{proof}

In Section~\ref{concept} we defined concepts of lattices, dual lattices and simplicial cones.

By definition we know that there exists a lattice $N$ of $V$.
Let $N$ be any lattice of $V$.

\begin{definition}
Let $S$ be any subset of $V$.
\begin{enumerate}
\item
We say that $S$ is a \emph{rational convex polyhedral cone} over $N$, or a convex polyhedral cone $S$ is \emph{rational} over $N$, if there exists a finite subset $X$ of $\QVect(N)$ with $S=\Convcone(X)$.
\item
We say that $S$ is a \emph{rational convex pseudo polyhedron} over $N$, or a convex pseudo polyhedron $S$ is \emph{rational} over $N$, if there exist finite subsets $X, Y$ of $\QVect(N)$ with $S=\Conv(X)+\Convcone(Y)$ and $X\neq\emptyset$.
\end{enumerate}
\end{definition}

The dual lattice $N^*$ is defined. We have
$$N^*=\{\omega\in V^*|\langle\omega, a\rangle\in\Z \text{ for any } a\in N\}\subset V^*,$$
by definition.

\begin{lemma}
\label{h}
\begin{enumerate}
\item
$N$ is a submodule of $V$, $N$ is a free module of finite rank over $\Z$ with $\Rank N=\dim V$, and $\Vect(N)=V$.
\item
Any basis of $N$ over $\Z$ is a basis of $\QVect(N)$ over $\Q$, and is a basis of $V$ over $\R$.
\item
For any non-empty finite subset $F$ of $N$, the following three conditions are equivalent:
\begin{enumerate}
\item
$F$ is linearly independent over $\Z$.
\item
$F$ is linearly independent over $\Q$.
\item
$F$ is linearly independent over $\R$.
\end{enumerate}
\item
The dual lattice $N^*$ of $N$ is a lattice of the dual vector space $V^*$ of $V$.
The dual lattice $N^{**}$ of $N^*$ is equal to $N$.
\item
A convex polyhedral cone $S$ in $V$ is rational over $N$, if and only if, $S=\Convcone(X)$ for some finite subset $X$ of $N$.
\item
For any vector space $S$ in $V$ the following three conditions are equivalent:
\begin{enumerate}
\item
$S$ is a rational polyhedral cone over $N$.
\item
$S=\Vect(X)$ for some finite subset $X$ of $N$.
\item
$N\cap S$ is a lattice of $S$.
\end{enumerate}
\item
For any rational polyhedral cone $S$ over $N$ in $V$, $\Vect(S)$ is a rational vector space over $N$ in $V$.
\item
For any non-empty subset $S$ of $V$, the following two conditions are equivalent:
\begin{enumerate}
\item
$S$ is a simplicial cone over $N$ in $V$ and $\dim S=1$.
\item
$S$ is a rational convex polyhedral cone over $N$ in $V$, $\dim S=1$, and $S\neq\Vect(S)$.
\end{enumerate}
\item
A convex pseudo polyhedron $S$ in $V$ is rational over $N$, if and only if, $S=\Conv (X)+\Convcone(Y)$ for some non-empty finite subset $X$ of $\QVect(N)$ and some finite subset $Y$ of $N$.
\item
For any affine space $S$ in $V$, the following two conditions are equivalent:
\begin{enumerate}
\item
$S$ is a rational convex pseudo polyhedron over $N$.
\item
$S=\{x\}+\Vect(Y)$ for some $x\in\QVect(N)$ and some finite subset $Y$ of $N$.
\end{enumerate}
\item
For any rational affine space $S$ over $N$ in $V$, $\Stab(S)$ is a rational vector space over $N$ in $V$.
\item
For any rational convex pseudo polyhedron $S$ over $N$ in $V$, $\Affi(S)$ is a rational affine space over $N$ in $V$.
\end{enumerate}
We consider any finite dimensional vector space $W$ over $\R$ and any homomorphism $\pi:V\rightarrow W$ of vector spaces over $\R$.
The dual homomorphism $\pi^*:W^*\rightarrow V^*$ is defined, and is a homomorphism of vector spaces over $\R$.
The kernel $\pi^{-1}(0)$ of $\pi$ is a vector subspace of $V$, the image $\pi(V)$ is a vector subspace of $W$, the image $\pi^*(W^*)$ of $\pi^*$ is a vector subspace of $V^*$, and the kernel $\pi^{*-1}(0)$ of $\pi^*$ is a vector subspace of $W^*$.
\begin{enumerate}
\setcounter{enumi}{12}
\item
The following seven conditions are equivalent;
\begin{enumerate}
\item
$\pi^{-1}(0)$ is rational over $N$.
\item
$N\cap\pi^{-1}(0)$ is a lattice of $\pi^{-1}(0)$.
\item
$\pi(N)$ is a lattice of $\pi(W)$.
\item
There exists a lattice $Q$ of $W$ satisfying $Q\cap \pi(V)=\pi(N)$.
\item
$\pi^*(W^*)$ is rational over $N^*$.
\item
$N^*\cap\pi^*(W^*)$ is a lattice of $\pi^*(W^*)$.
\item
There exists a lattice $\bar{Q}$ of $W^*$ satisfying $\pi^*(\bar{Q})=N^*\cap\pi^*(W^*)$.
\end{enumerate}
\item
Assume that equivalent seven conditions in the above $14$ hold.
For any lattice $\bar{Q}$ of $W^*$ satisfying $\pi^*(\bar{Q})=N^*\cap\pi^*(W^*)$,
$\bar{Q}\cap \pi^{*-1}(0)$ is a lattice of $\pi^{*-1}(0)$.
\item
For any lattice $Q$ of $W$, the following two conditions are equivalent:
\begin{enumerate}
\item
$Q\cap\pi(V)=\pi(N)$.
\item
$\pi^*(Q^*)=N^*\cap \pi^*(W^*)$.
\end{enumerate}
\item
For any lattice $\bar{Q}$ of $W^*$, the following two conditions are equivalent:
\begin{enumerate}
\item
$\pi^*(\bar{Q})=N^*\cap\pi^*(W^*)$.
\item
$\bar{Q}^*\cap\pi(V)=\pi(N)$.
\end{enumerate}
\end{enumerate}
\end{lemma}

\section{Convex cones and convex polyhedral cones}
\label{cones}

We study convex cones and convex polyhedral cones.

Let $V$ be any finite dimensional vector space over $\R$, and let $N$ be any lattice of $V$.

In Section~\ref{concept} we defined the dual cone $S^\vee|V$ of any convex cone $S$ in $V$.
By definition
$$S^\vee|V=\{\omega\in V^*|\langle\omega, a\rangle\geq 0 \text{ for any } a\in S\}\subset V^*,$$
for any convex cone $S$ in $V$.

\begin{lemma}
\label{i}
Let $S$ be any convex cone in $V$.
\begin{enumerate}
\item
The dual cone $S^\vee|V$ of $S$ is a closed convex cone in the dual vector space $V^*$ of $V$.
\item
Let $W$ be any vector subspace in $V$ with $S\subset W$.
$S$ is a convex cone in $W$ and the dual cone $S^\vee|W$ of $S$ in $W^*$ is defined.

Let $\iota:W\rightarrow V$ denote the inclusion homomorphism.
The dual homomorphism $\iota^*:V^*\rightarrow W^*$ is defined, which is surjective.

$S^\vee|V=\iota^{*-1}( S^\vee|W)$.
\item
$S^\vee|V^\vee|V^*=\Clos(S)$.
\item
$S^\vee|V^\vee|V^*=S$, if and only if, $S$ is closed.
\end{enumerate}
\end{lemma}

\begin{proof}
\noindent 1,2. Easy.

\noindent 3.
We know $\langle\omega, a\rangle\geq 0$ for any $\omega\in S^\vee|V$ and for any $a\in\Clos(S)$,
and $S^\vee|V^\vee|V^*\supset\Clos(S)$.

We have to show the opposite inclusion relation.

We take any positive definite symmetric bilinear form $(\;,\;):V\times V\rightarrow \R$, and putting $|x|=\sqrt{(x,x)}\in\R_0$ for any $x\in V$, we define a norm $|\;|:V\rightarrow\R_0$.

We consider any point $a\in V- \Clos(S)$.

It is easy to show that there exists a point $b\in\Clos(S)$ satisfying
$|a-b|=\min\{|a-c||c\in\Clos(S)\}$.
We take a point $b\in\Clos(S)$ with $|a-b|=\min\{|a-c||c\in\Clos(S)\}$.
Since $a\not\in\Clos(S)$, we have $a-b\neq 0$, and $|a-b|>0$.

We consider any point $c\in\Clos(S)$.
Furthermore, we consider any $t\in\R$ with $0\leq t\leq 1$ and any $u\in\R_0$.
Since $\Clos(S)$ is a convex cone, we have $(1-t)b+tuc\in\Clos(S)$ and $|a-b|\leq|a-(1-t)b-tuc|=|a-b+t(b-uc)|$.
Therefore, $|a-b|^2\leq|a-b+t(b-uc)|^2=|a-b|^2+2t(a-b,b-uc)+t^2|b-uc|^2$, and
$0\leq 2t(a-b,b-uc)+t^2|b-uc|^2$.
Since $t\in\R$ is any element with $0\leq t\leq 1$, we know
$0\leq (a-b,b-uc)=(a-b,b)-u(a-b,c)$.
Since $u\in\R_0$ is any element, we know $(a-b,b)\geq 0$ and $(a-b,c)\leq 0$.
Furthermore, considering the case $c=b$, we know $(a-b,b)\leq 0$.
Therefore, we conclude $(b-a, b)=0$ and $(b-a, c)\geq 0$ for any $c\in\Clos(S)$.

We take the point $\omega\in V^*$ satisfying $\langle\omega, x\rangle=(b-a,x)$ for any $x\in V$.
We have $\langle\omega, c\rangle=(b-a,c)\geq 0$.
Since $c\in\Clos(S)$ is any point and $\Clos(S)\supset S$, we know $\omega\in S^\vee|V$.

Since $\langle\omega, a\rangle=(b-a,a)=(b-a, a-b)+(b-a, b)=-|a-b|^2+0=-|a-b|^2<0$,
we conclude $a\not\in S^\vee|V^\vee|V^*$, and $a\in V- (S^\vee|V^\vee|V^*)$.

Since $a\in V- \Clos(S)$ is any point, we conclude $V- \Clos(S)\subset V- (S^\vee|V^\vee|V^*)$, $S^\vee|V^\vee|V^*\subset\Clos(S)$, and $ S^\vee|V^\vee|V^*=\Clos(S)$.

\noindent 4. It follows from $3$.
\end{proof}

When we need not refer to $V$, we also write simply $S^\vee$, instead of $ S^\vee|V$.

\begin{lemma}
\label{j}
Let $S$ and $T$ be any convex cones in $V$.
\begin{enumerate}
\item
If $S\subset T$, then $S^\vee\supset T^\vee$.
\item
Assume that $S$ and $T$ are closed. $S\subset T$, if and only if, $S^\vee\supset T^\vee$.
\item
$(S+T)^\vee=S^\vee\cap T^\vee$.
\item
Assume that $S$ and $T$ are closed.
$(S\cap T)^\vee=\Clos(S^\vee+T^\vee)$.
\item
Assume that $S$ is a vector space.
By $\iota:S\rightarrow V$ we denote the inclusion homomorphism.
$\iota^*:V^*\rightarrow S^*$.
$$S^\vee =\{\omega\in V^*|\langle\omega, a\rangle=0 \text{ for any } a\in S\}=\iota^{*-1}(0)\subset V^*,$$
$S^\vee$ is also a vector space, and $\dim S+\dim S^\vee =\dim V$.
\item
$\Vect(S)^\vee=S^\vee\cap (-S^\vee)$.
$\Vect(S)^\vee$ is the maximal vector space contained in $S^\vee$ with respect to the inclusion relation.
\item
Assume that $S$ is closed.
$\Vect(S^\vee)=(S\cap(-S))^\vee$ and $\dim S^\vee =\dim V-\dim(S\cap(-S))$.
\item
$\dim S+\dim S^\vee\geq \dim V$.  $\dim S+\dim S^\vee=\dim V$, if and only if, $S$ is a vector space.
\item
We denote $n=\dim V\in\Z_0$. Let $e: \{1,2,\ldots,n\}\rightarrow V$ be any mapping such that the image $e(\{1,2,\ldots,n\})$ of $e$ is a basis of $V$, and let $I, J, K, L$ be any subset of $\{1,2,\ldots,n\}$ such that $I\cup J\cup K\cup L=\{1,2,\ldots,n\}$ and the intersection of any two of $I, J, K, L$ is empty.
We denote the dual basis of  $\{e(i)|i\in\{1,2,\ldots,n\}\}$ by $\{e^\vee(i)|i\in\{1,2,\ldots,n\}\}$.
We assume
$$\langle e^\vee(i), e(j)\rangle=
\begin{cases}
1&\text{if $i=j$},\\
0&\text{if $i\neq j$},
\end{cases}$$
for any $i\in\{1,2,\ldots,n\}$ and any $j\in\{1,2,\ldots,n\}$.

If
$$S=\sum_{i\in I}\{0\}e(i)+\sum_{j\in J}\R e(j)+\sum_{k\in K}\R_0 e(k)+\sum_{\ell\in L}\R_0(-e(\ell)),$$
then
$$S^\vee=\sum_{i\in I}\R e^\vee (i)+\sum_{j\in J}\{0\}
e^\vee (j)+\sum_{k\in K}\R_0 e^\vee (k)+\sum_{\ell\in L}\R_0(-e^\vee (\ell)).$$
\item
If $S$ is a simplicial cone over $N$ with $\dim S=\dim V$, then $S^\vee$ is a simplicial cone over $N^*$ with $\dim S^\vee=\dim V^*$.
\end{enumerate}
\end{lemma}

\begin{remark}
Assume $V=\R^3$, $S=\{(x, y, z)\in V|x\geq 0, y\geq 0, z\geq -2\sqrt{xy}\}$ and $T=\{(x, y, z)\in V|x=0, z\geq 0\}$. $S$ and $T$ are closed convex cones in the vector space $V$ with $\dim V=3$. $(0,0,-1)\in\Clos(S+T)$ and $(0,0,-1)\not\in S+T$.
Therefore $S+T$ is not closed.
\end{remark}

\begin{definition}
\label{faces of cones}
Let $S$ be any convex polyhedral cone in $V$.
We consider the dual cone $S^\vee=S^\vee|V\subset V^*$ of $S$.
\begin{enumerate}
\item
For any $\omega\in S^\vee$, we denote
$$\Delta(\omega,S|V)=\{x\in S|\langle\omega, x\rangle=0\}\subset S.$$

When we need not refer to $V$ or to the pair $(S, V)$, we also write simply  $\Delta(\omega, S)$ or $\Delta(\omega)$, instead of $\Delta(\omega, S|V)$.
\item
Let $F$ be any subset of $S$.
We say that $F$ is a \emph{face} of $S$, if $F=\Delta(\omega, S|V)$ for some $\omega\in S^\vee.$

It is easy to see that any face $F$ of $S$ is a closed convex cone, and the dimension $\dim F\in\Z_0$ of $F$, the boundary $\partial F$ of $F$ and the interior $F^\circ$ of $F$ are defined.
\item
By $\mathcal{F}(S)$ we denote the set of all faces of $S$.

For any $i\in\Z$, the set of all faces $F$ with $\dim F=i$ is denoted by $\mathcal{F} (S)_i$, and the set of all faces $F$ with $\dim F=\dim S-i$ is denoted by $\mathcal{F}(S)^i$.
\item
Let $F$ be any face of $S$.
We denote
\begin{equation*}
\begin{split}
\Delta^\circ(F,S|V)&=\{\omega\in S^\vee|F=\Delta(\omega, S|V)\}\subset S^\vee\subset V^*,\\
\Delta(F,S|V)&=\{\omega\in S^\vee|F\subset\Delta(\omega, S|V)\}\subset S^\vee\subset V^*.\\
\end{split}
\end{equation*}

We call $\Delta^\circ(F,S|V)$ the \emph{open face cone} of $F$, and we call $\Delta (F,S|V)$ the \emph{face cone} of $F$.

When we need not refer to $V$ or to the pair $(S, V)$, we also write simply  $\Delta^\circ (F, S)$ or $\Delta^\circ(F)$, $\Delta(F, S)$ or $\Delta(F)$ respectively, instead of $\Delta^\circ(F, S|V)$, $\Delta(F, S|V)$.
\end{enumerate}
\end{definition}

\begin{prop}
\label{property of cones}
Let $S$ be any convex polyhedral cone in $V$, and let $X$ be any finite subset of $V$ with $S=\Convcone(X)$.
We consider the dual cone $S^\vee=S^\vee|V\subset V^*$ of $S$.
For simplicity we denote
$s=\dim S\in\Z_0$, $L=S\cap(-S)\subset S$, $\ell=\dim L\in\Z_0$, $M=S^\vee\cap(-S^\vee)\subset S^\vee$.
\begin{enumerate}
\item
We consider any finite dimensional vector space $U$ over $\R$ with $\dim S\leq\dim U\leq\dim V$, any injective homomorphism $\nu:U\rightarrow V$ of vector spaces over $\R$ such that $S\subset\nu(U)$, and any subset $F$ of $S$.
The inverse image $\nu^{-1}(S)$ is a convex polyhedral cone in $U$.
The set $F$ is a face of $S$, if and only if $\nu^{-1}(F)$ is a face of $\nu^{-1}(S)$.
\item
We consider any finite dimensional vector space $W$ over $\R$ with $\dim V\leq\dim W$, any injective homomorphism $\pi:V\rightarrow W$ of vector spaces over $\R$, and any subset $F$ of $S$.
The image $\pi(S)$ is a convex polyhedral cone in $W$.
The set $F$ is a face of $S$, if and only if $\pi(F)$ is a face of $\pi(S)$.
\item
$\ell\leq s$. $\ell=s\Leftrightarrow L=S\Leftrightarrow S=\Vect(S)$.
\item
Let $F$ be any face of $S$.
\begin{enumerate}
\item
$F=\Convcone(X\cap F)=S\cap\Vect(F)$. $\Vect(F)=\Vect(X\cap F)$.
\item
$F$ is a convex polyhedral cone in $V$.
\item
If $S$ is rational over $N$, then $F$ is also rational over $N$.
If $S$ is a simplicial cone over $N$, then $F$ is also a simplicial cone over $N$.
\item
$L=F\cap(-F)\subset F$. $\ell\leq\dim F\leq s$.
\item
Let $G$ be any face of $S$ with $G\subset F$. We have $\dim G\leq\dim F$.  $\dim G=\dim F$, if and only if, $G=F$.
\item
Let $G$ be any subset of $F$. $G$ is a face of the convex polyhedral cone $F$, if and only if, $G$ is a face of $S$ with $G\subset F$.
\end{enumerate}
\item
Assume $\ell<s$.
For any face $G$ of $S$ with $G\neq S$ there exists a face $F$ of $S$ with $\dim F=s-1$ and $G\subset F$.
There exists a face $F$ of $S$ with $\dim F=s-1$.
\item
$\mathcal{F}(S)$ is a finite set.
$S\in\mathcal{F}(S)_s$ and $\mathcal{F}(S)_s=\{S\}$.
$S$ contains any face of $S$.
$L\in\mathcal{F}(S)_\ell$ and $\mathcal{F}(S)_\ell=\{L\}$.
$L$ is contained in any face of $S$.
$L=\Convcone(X\cap L)=\Vect(X\cap L)$. 
For any $i\in\Z_0$, $\mathcal{F}(S)_i\neq\emptyset$ if and only if $\ell\leq i\leq s$.
\item
Let $F$ and $G$ be any face of $S$ with $F\subset G$.
We denote $f=\dim F$ and $g=\dim G$.
$\ell\leq f\leq g\leq s$.
There exist $(s-\ell+1)$ of faces $F(\ell), F(\ell+1),\ldots, F(s)$ satisfying the following three conditions:
\begin{enumerate}
\item
For any $i\in\{\ell,\ell+1,\ldots,s-1\}$, $F(i)\subset F(i+1)$.
\item
For any $i\in\{\ell,\ell+1,\ldots,s\}$, $\dim F(i)=i$.
\item
$F(\ell)=L, F(f)=F, F(g)=G, F(s)=S$.
\end{enumerate}
\item
Let $F$ be any face of $S$.
\begin{enumerate}
\item
$F=\partial F\cup F^\circ$. $\partial F\cap F^\circ=\emptyset$.
\item
$F^\circ=F\Leftrightarrow \partial F=\emptyset\Leftrightarrow F=L$.
\item
$$\partial F=\bigcup_{G\in\mathcal{F}(F)-\{F\}}G.$$
\item
$F^\circ $ is a non-empty open subset of $\Vect(F)$.
For any $a\in F^\circ$ and for any $b\in F$, $\Conv(\{a,b\})-\{b\}\subset F^\circ$.
$F^\circ$ is convex.
$\Clos(F^\circ)=F$.
\end{enumerate}
\item
Consider any $m\in\Z_+$ and any mapping $F:\{1,2,\ldots,m\}\rightarrow\mathcal{F}(S)$. 
The intersection $\cap_{i\in\{1,2,\ldots,m\}}F(i)$ is a face of $S$.
\item
We consider any two faces $F, G$ of $S$.
$F^\circ\cap G\neq\emptyset$, if and only if, $F\subset G$.
$F^\circ\cap G^\circ\neq\emptyset$, if and only if, $F=G$.
\item
$M=\Delta(S)=\Vect(S)^\vee$. If $S$ is rational over $N$, then $M$ is rational over $N^*$.
\item
Assume $\ell<s$. Let $F\in\mathcal{F}(S)^1$ be any element.
\begin{enumerate}
\item
$M\subset\Delta(F)$. $M\neq\Delta(F)$.
\item
For any $\omega_F\in\Delta(F)- M$ we have $\Delta(F)=\R_0\omega_F+M$.
\item
If $S$ is rational over $N$, then $(\Delta(F)- M)\cap N^*\neq\emptyset$.
\end{enumerate}
\item
Note that $\ell<s$, if and only if, $\mathcal{F}(S)^1\neq\emptyset$.
In case $\ell<s$ we take any element $\omega_F\in\Delta(F)- M$ for any $F\in\mathcal{F}(S)^1$.
$$
S^\vee=
\Convcone(\{\omega_F|F\in\mathcal{F}(S)^1\})+M,\ 
S=
\bigcap_{F\in\mathcal{F}(S)^1}(\R_0\omega_F)^\vee\cap\Vect(S).
$$
\item
$S^\vee$ is a convex polyhedral cone in $V^*$.
If $S$ is rational over $N$, then $S^\vee$ is rational over $N^*$.
\item
Let $F$ be any face of $S$.
\begin{enumerate}
\item
$\Delta(F)$ is a face of $S^\vee$.
\item
$\Delta(F)=\Vect(F)^\vee\cap S^\vee$. $\Vect(\Delta(F))=\Vect(F)^\vee$.
\item
$\Delta^\circ(F)=\Delta(F)^\circ$. $\Delta(F)=\Clos(\Delta^\circ(F))$.
\end{enumerate}
\item
For any face $F$ of $S$, $\Delta(F,S|V)$ is a face of $S^\vee$, and $\dim\Delta(F,S|V)=\dim V-\dim F$. For any two faces $F$, $G$ of $S$ with $F\subset G$, $\Delta(F,S|V)\supset\Delta(G,S|V)$.

For any face $\bar{F}$ of $S^\vee$, $\Delta(\bar{F}, S^\vee |V^*)$ is a face of $S $, and $\dim\Delta(\bar{F}, S^\vee |V^*)=\dim V-\dim \bar{F}$. For any two faces $\bar{F}$, $\bar{G}$ of $S^\vee$ with $\bar{F}\subset \bar{G}$, $\Delta(\bar{F}, S^\vee |V^*)\supset\Delta(\bar{G}, S^\vee |V^*)$.

The mapping from $\mathcal{F}(S)$ to $\mathcal{F}(S^\vee)$ sending $F\in\mathcal{F}(S)$ to $\Delta(F,S|V)\in\mathcal{F}(S^\vee)$ and the mapping from $\mathcal{F}(S^\vee)$ to $\mathcal{F}(S)$ sending $\bar{F}\in\mathcal{F}(S^\vee)$ to\hfill\break $\Delta(\bar{F}, S^\vee |V^*)\in\mathcal{F}(S)$ are bijective mappings reversing the inclusion relation between $\mathcal{F}(S)$ and $\mathcal{F}(S^\vee)$, and they are the inverse mappings of each other.
Furthermore, if $F\in\mathcal{F}(S)$ and $\bar{F}\in\mathcal{F}(S^\vee)$ correspond to each other by them, $\dim F+\dim\bar{F}=\dim V$.

\item
Assume $\ell<s$. Let $F\in\mathcal{F}(S)_{\ell+1}$ be any element.
\begin{enumerate}
\item
$L\subset F$. $L\neq F$.
\item
For any $t_F\in F- L$ we have $F=\R_0t_F+L$.
\end{enumerate}
\item
In case $\ell<s$ we take any element $t_F\in F- L$ for any $F\in\mathcal{F}(S)_{\ell+1}$.
Note that $\ell<s$, if and only if, $\mathcal{F}(S)_{\ell+1}\neq\emptyset$.
$$S =
\Convcone(\{t_F|F\in\mathcal{F}(S)_{\ell+1}\})+L,\ 
S^\vee=
\bigcap_{F\in\mathcal{F}(S)_{\ell+1}}(\R_0t_F)^\vee\cap\Vect(S^\vee).
$$
\item
The family $\{F^\circ|F\in\mathcal{F}(S)\}$ of subsets of $S$ gives the equivalence class decomposition of $S$.
In other words, the following three conditions hold:
\begin{enumerate}
\item
$F^\circ\neq\emptyset$ for any $F\in\mathcal{F}(S)$.
\item
If $F\in\mathcal{F}(S)$, $G\in\mathcal{F}(S)$, and $F^\circ\cap G^\circ\neq\emptyset$, then $F^\circ=G^\circ$.
\item
$$S=\bigcup_{ F\in\mathcal{F}(S)}F^\circ.$$
\end{enumerate}
\item
$\mathcal{F}(S)$ is a convex polyhedral cone decomposition in $V$, and the support of  $\mathcal{F}(S)$ is equal to $S$.
In other words, the following four conditions hold:
\begin{enumerate}
\item
$\mathcal{F}(S)$ is a non-empty finite set whose elements are convex polyhedral cones in $V$.
\item
For any $F\in\mathcal{F}(S)$ and for any $G\in\mathcal{F}(S)$, $F\cap G$ is a face of $F$, and $F\cap G$ is a face of $G$.
\item
If $F\in\mathcal{F}(S)$ and $G$ is a face of $F$, then $G\in\mathcal{F}(S)$.
\item
$$S=\bigcup_{ F\in\mathcal{F}(S)}F.$$
\end{enumerate}
\item
If $S$ is rational over $N$, then any $F\in\mathcal{F}(S)$ is rational over $N$.
If $S$ is a simplicial cone over $N$, then any $F\in\mathcal{F}(S)$ is a simplicial cone over $N$.
\item
Consider any finite dimensional vector space $W$ over $\R$ and any homomorphism $\pi:V\rightarrow W$ of vector spaces over $\R$.
The image $\pi(S)$ is a convex polyhedral cone in $W$, and it satisfies $\pi(S)^\circ=\pi(S^\circ)$.
\end{enumerate}
\end{prop}

\begin{proof}
We give proofs only to claims 4, 5, 8, 13 and 15. Other claims follow easily.

\noindent 4.
Let $F$ be any face of $S$.
We take $\omega\in S^\vee$ with $F=\Delta(\omega)=\{x\in S|\langle\omega, x\rangle=0\}$.

\noindent $(a)$.
We consider any point $a\in F$.
$a\in F\subset S=\Convcone(X)=\{\sum_{x\in X}\lambda(x)x|\lambda\in\Map(X,\R_0)\}$.
We take $\lambda\in\Map(X,\R_0)$ with $a=\sum_{x\in X}\lambda(x)x$.
Since $X\subset\Convcone(X)=S$ and $\omega\in S^\vee$, $\langle\omega,x\rangle\geq 0$ for any $x\in X$, and $\langle\omega,x\rangle>0$ for any $x\in X- F$.
Since $\sum_{x\in X}\lambda(x)\langle\omega, x\rangle=\langle\omega,a\rangle=0$, $\lambda(x)=0$ for any $x\in X- F$.
We know $a=\sum_{x\in X\cap F}\lambda(x)x\in\Convcone(X\cap F)$, and $F\subset\Convcone(X\cap F)$.

Since $F$ is a convex cone and $F\supset X\cap F$, $F=\Convcone(F)\supset\Convcone(X\cap F)$.
We know $F=\Convcone(X\cap F)$, and $\Vect(F)=\Vect(\Convcone(X\cap F))=\Vect(X\cap F)$.

Obviously $F\subset S\cap\Vect(F)$.

We consider any $a\in S\cap\Vect(F)$. 
$a\in\Vect(F)=\{\sum_{x\in\Supp(\lambda)}\lambda(x)x|\lambda\in\Map'(F,\R)\}$.
We take $\lambda\in\Map'(F,\R)$ with $a=\sum_{x\in\Supp(\lambda)}\lambda(x)x$.
We have $\langle\omega, x\rangle=0$ for any $x\in\Supp(\lambda)\subset F$.
Therefore, $\langle\omega,a\rangle=\sum_{x\in\Supp(\lambda)}\lambda(x)\langle\omega,x\rangle=0$ and $a\in F$, since $a\in S$.
We know $F\supset S\cap\Vect(F)$ and $F=S\cap\Vect(F)$.

\noindent $(b)$, $(c)$. They follow from $(a)$.

\noindent $(d)$.
Obviously $L=S\cap(-S)\supset F\cap(-F)$.

We consider any $a\in L$. Since $\{a,-a\}\subset L\subset S$ and $\omega\in S^\vee$,
$\langle\omega, a\rangle\geq 0$ and $-\langle\omega,a\rangle=\langle\omega,-a\rangle\geq 0$.
Therefore, $\langle\omega,a\rangle=0$ and $a\in F$.
We know $L\subset F$.
Since $L=-L\subset -F$, we know $L\subset F\cap(-F)$ and $L=F\cap(-F)\subset F$.

Since $L\subset F\subset S$, $\ell=\dim L\leq\dim F\leq s=\dim S$.

\noindent $(e)$.
Let $G$ be any face of $S$ with $G\subset F$.
We have $\Vect(G)\subset\Vect(F)$ and $\dim G=\dim\Vect(G)\leq\dim F=\dim\Vect(F)$.

Assume $\dim G=\dim F$.
We have $\dim\Vect(G)=\dim\Vect(F)$, and $\Vect(G)=\Vect(F)$.
By $(a)$ we have $G=S\cap\Vect(G)=S\cap\Vect(F)=F$.

Conversely, assume $G=F$. Obviously, we have $\dim G=\dim F$.

\noindent $(f)$.
Let $G$ be any face of the convex polyhedral cone $F$.
$G\subset F\subset S$. We take $\chi\in F^\vee\subset V^*$ with $G=\Delta(\chi, F)$.

Since $X\subset\Conv(X)=S$, $\langle\omega, x\rangle\geq 0$ for any $x\in X$.
$\langle\omega,x\rangle=0$, if and only if, $x\in X\cap F$.
$\langle\omega,x\rangle>0$ for any $x\in X- F$.

Since $X\cap F\subset F$, $\langle\chi, x\rangle\geq 0$ for any $x\in X\cap F$.
$\langle\chi,x\rangle=0$, if and only if, $x\in X\cap G$.
$\langle\chi,x\rangle>0$ for any $x\in(X\cap F)-(X\cap G)$.

We consider any $\lambda\in\R$.
$\langle\chi+\lambda\omega, x\rangle=\langle\chi,x\rangle+\lambda\langle\omega,x\rangle$ for any $x\in X$.
For any $x\in X\cap G$, $\langle\chi, x\rangle=0, \langle\omega,x\rangle=0, \langle\chi,x\rangle+\lambda\langle\omega,x\rangle=0$,
For any $x\in(X\cap F)-(X\cap G)$, $\langle\chi, x\rangle>0, \langle\omega,x\rangle=0, \langle\chi,x\rangle+\lambda\langle\omega,x\rangle=\langle\chi,x\rangle>0$.
For any $x\in X- F$, $\langle\chi, x\rangle\in\R, \langle\omega,x\rangle>0, \langle\chi,x\rangle+\lambda\langle\omega,x\rangle>0$, if and only if, $\lambda>-\langle\chi, x\rangle/\langle\omega,x\rangle$.

Note that $X- F$ is a finite set. We take any $\lambda\in\R$ satisfying
$$\lambda>\max\{-\frac{\langle\chi,x\rangle}{\langle\omega,x\rangle}|x\in X- F\},$$
if $X- F\neq\emptyset$.

We know that $\langle\chi+\lambda\omega, x\rangle=0$ for any $x\in X\cap G$, and $\langle\chi+\lambda\omega, x\rangle>0$ for any $x\in X- G$.

We consider any point $a\in S=\Convcone(X)=\{\sum_{x\in X}\mu(x)x|\mu\in\Map(X,\R_0)\}$. We take any $\mu\in\Map(X,\R_0)$ with $a=\sum_{x\in X}\mu(x)x$.
$\langle\chi+\lambda\omega, a\rangle=\sum_{x\in X}\mu(x)\langle \chi+\lambda\omega,x\rangle\geq 0$, and if $\langle\chi+\lambda\omega, a\rangle=0$, then $\mu(x)=0$ for any $x\in X- G$.
We know $\chi+\lambda\omega\in S^\vee$. If $a\in\Delta(\chi+\lambda\omega,S)$, then $a=\sum_{x\in X\cap G}\mu(x)x\in\Convcone(X\cap G)=G$. We know $\Delta(\chi+\lambda\omega,S)\subset G$.

We consider any point $a\in G$. $a\in G\subset F\subset S$. $a\in G=\Convcone(X\cap G)=\{\sum_{x\in X\cap G}\mu(x)x|\mu\in\Map(X\cap G,\R_0)\}$. We take any $\mu\in\Map(X\cap G,\R_0)$ with $a=\sum_{x\in X\cap G}\mu(x)x$.
$\langle\chi+\lambda\omega, a\rangle=\sum_{x\in X\cap G}\mu(x)\langle \chi+\lambda\omega,x\rangle=0$, and $a\in\Delta(\chi+\lambda\omega,S)$. We know $G\subset \Delta(\chi+\lambda\omega,S)$, and $G=\Delta(\chi+\lambda\omega,S)$.

We know that $G$ is a face of $S$ with $G\subset F$.

Conversely, we consider any face $G$ of $S$ with $G\subset F$.
We take any $\chi\in S^\vee$ with $G=\Delta(\chi, S)$.
Since $S\supset F$, we have $\chi\in S^\vee\subset F^\vee$. 
$G=G\cap F=\Delta(\chi, S)\cap F=\{x\in S|\langle\chi,x\rangle=0\}\cap F=\{x\in F|\langle\chi,x\rangle=0\}=\Delta(\chi, F)$.

We know that $G$ is a face of the convex polyhedral cone $F$.

\noindent 5. Assume $\ell<s$.
We consider any face $G$ of $S$ with $G\neq S$.

Assume moreover, that a face $F$ of $S$ satisfying $G\subset F\subset S$ and $\dim F=\dim S-1$ does not exist. 
We will deduce a contradiction from this assumption.

$G\in\{F\in\mathcal{F}(S)|G\subset F\subset S, F\neq S\}\neq\emptyset$.

We consider any $ F\in\mathcal{F}(S)$ with $ G\subset F\subset S$ and $F\neq S$.
By 4.(e) we know $\dim G\leq\dim F\leq s-1$. By assumption we know $\dim F\leq s-2$.
Put
$$\delta=\max\{\dim F| F\in\mathcal{F}(S), G\subset F\subset S, F\neq S\}.$$
We know $\dim G\leq\delta\leq s-2$.

Let $F$ be any face of $S$ with $\dim F=\delta$ and $G\subset F\subset S$.
Let $\omega\in S^\vee$ be any element with $F=\Delta(\omega)$.

$F=\Convcone(X\cap F)$.
Since $F\neq S$, $X- F\neq\emptyset$.
$X- F$ is a non-empty finite set.
$\langle\omega,x\rangle=0$ for any $x\in X\cap F$, and $\langle\omega,x\rangle>0$ for any $x\in X- F$.

We denote $D_\omega=\{x\in\Vect(S)|\langle\omega, x\rangle=0\}$.
$D_\omega$ is a vector subspace of $\Vect(S)$ with $\dim D_\omega\geq s-1$.
Take any element $x_0\in X- F\neq\emptyset$. $x_0\in X- F\subset X\subset S\subset\Vect(S)$ and $\langle\omega, x_0\rangle>0$.
We know that $D_\omega\neq\Vect(S)$ and $\dim D_\omega=s-1$.

$\Vect(F)\subset\Vect(S)$.
$\dim\Vect(F)=\dim F=\delta\leq s-2$.
$\dim\Vect(S)=s$.

We know that there exists a vector subspace $E$ of $\Vect(S)$ of dimension $s-1$ such that $\Vect(F)\subset E\subset\Vect(S)$ and $E\neq D_\omega$.
We take a vector subspace $E$ of $\Vect(S)$ of dimension $s-1$ such that $\Vect(F)\subset E\subset\Vect(S)$ and $E\neq D_\omega$.
We take $\chi\in V^*$ with $E=\{x\in\Vect(S)|\langle\chi,x\rangle=0\}$.
Since $X\cap F\subset F\subset\Vect(F)\subset E=\{x\in \Vect(S)|\langle\chi,x\rangle=0\}$, $\langle\chi,x\rangle=0$ for any $x\in X\cap F$.

Let
$$\lambda=\max\{-\frac{\langle\chi,x\rangle}{\langle\omega,x\rangle}|x\in X- F\}\in\R.$$
For any $x\in X- F$, $\langle\omega, x\rangle>0$, $\lambda\geq-\langle\chi,x\rangle/\langle\omega,x\rangle$ and $\langle\chi+\lambda\omega, x\rangle=\langle\chi,x\rangle+\lambda\langle\omega,x\rangle\geq 0$. 
There exists $ x\in X- F$ with $\langle\chi+\lambda\omega, x\rangle=0$.
For any $x\in X\cap F$ $\langle\chi,x\rangle=0$, $\langle\omega,x\rangle=0$, and $\langle\chi+\lambda\omega, x\rangle=\langle\chi,x\rangle+\lambda\langle\omega,x\rangle=0$.

Let
$$Y=(X\cap F)\cup\{x\in X- F|\langle\chi+\lambda\omega, x\rangle=0\}.$$
$X\cap F\subset Y\subset X$ and $X\cap F\neq Y$.
For any $x\in X$, $\langle\chi+\lambda\omega, x\rangle\geq 0$, and $\langle\chi+\lambda\omega, x\rangle=0$, if and only if, $x\in Y$.

We consider any point $a\in S=\{\sum_{x\in X}\mu(x)x|\mu\in\Map(X,\R_0)\}$.
We take $\mu\in\Map(X,\R_0)$ with $a=\sum_{x\in X}\mu(x)x$.
We have $\langle\chi+\lambda\omega,a\rangle=\sum_{x\in X}\mu(x)\langle \chi+\lambda\omega,x\rangle\geq 0$, and if $\langle\chi+\lambda\omega,a\rangle=0$, then $\mu(x)=0$ for any $x\in X- Y$.

We know $\chi+\lambda\omega\in S^\vee$ and $\Convcone(Y)\supset\Delta(\chi+\lambda\omega, S)$.

We consider any point $a\in\Convcone(Y)=\{\sum_{x\in Y}\mu(x)x|\mu\in\Map(Y,\R_0)\}$.
$a\in\Convcone(Y)\subset\Convcone(X)=S$.
We take $\mu\in\Map(Y,\R_0)$ with $a=\sum_{x\in Y}\mu(x)x$.
We have $\langle\chi+\lambda\omega,a\rangle=\sum_{x\in Y}\mu(x)\langle \chi+\lambda\omega,x\rangle=0$, and $a\in\Delta(\chi+\lambda\omega, S)$.

We know $\Convcone(Y)\subset\Delta(\chi+\lambda\omega, S)$ and $\Convcone(Y)=\Delta(\chi+\lambda\omega, S)\in\mathcal{F}(S)$.

Since $X\cap F\subset Y$, $F=\Convcone(X\cap F)\subset\Convcone(Y)$.
Since $Y-(X\cap F)\subset Y\subset \Convcone(Y)$ and $\emptyset\neq Y-(X\cap F)\not\subset F$, $F\neq\Convcone(Y)$.
By 4.$(e)$ we have $\delta=\dim F<\dim\Convcone(Y)$, $\Convcone(Y)=S$, and $\Vect(Y)=\Vect(\Convcone(Y))=\Vect(S)$.

We consider any $a\in\Vect(S)=\Vect(Y)=\{\sum_{x\in Y}\mu(x)x|\mu\in\Map(Y,\R)\}$.
We take $\mu\in\Map(Y,\R)$ with $a=\sum_{x\in Y}\mu(x)x$.
$\langle\chi+\lambda\omega,a\rangle=\sum_{x\in Y}\mu(x)\langle\chi+\lambda\omega,x\rangle=0$.
We know that $\langle\chi+\lambda\omega,a\rangle=0$ for any $a\in\Vect(S)$.

We consider any point $a\in D_\omega$. $a\in D_\omega\subset\Vect(S)$ and $\langle\omega,a\rangle=0$.
Therefore $\langle\chi,a\rangle=\langle\chi,a\rangle+\lambda\langle\omega,a\rangle=\langle\chi+\lambda\omega,a\rangle=0$, and $a\in E$.

We know $D_\omega\subset E$. Since $\dim D_\omega=\dim E=s-1$, we know $D_\omega=E$, which contradicts $D_\omega\neq E$.

We know that there exists a face $F$ of $S$ with $G\subset F\subset S$ and $\dim F=s-1$.

Assume that $\Delta(\omega)=S$ for any $\omega\in S^\vee$.
For any $a\in S$ and for any $\omega\in S^\vee$ we have $\langle\omega, a\rangle=0$.

We consider any $a\in S$ and any $\chi\in\Vect(S^\vee)=\{\sum_{\omega\in\Supp(\lambda)}\lambda(\omega)\omega|\lambda\in\Map'(S^\vee,\R)\}$.
We take $\lambda\in \Map'(S^\vee,\R)$ with $\chi=\sum_{\omega\in\Supp(\lambda)}\lambda(\omega)\omega$.
For any $\omega\in\Supp(\lambda)\subset S^\vee $ we have $\langle\omega,a\rangle=0$, and
$\langle\chi,a\rangle=\sum_{\omega\in\Supp(\lambda)}\lambda(\omega)\langle\omega,a\rangle=0$.

We know that for any $a\in S$ and any $\chi\in\Vect(S^\vee)$,  $\langle\chi,a\rangle=0$ and $S\subset\Vect(S^\vee)^\vee=L^{\vee\vee}=L$.
Since $S\supset L$, we know that $S =L$ and $\ell=s$ by 2, which contradicts the assumption $\ell<s$.

We know that there exists $\omega\in S^\vee$ with $\Delta(\omega)\neq S$. We take $\omega\in S^\vee$ with $\Delta(\omega)\neq S$. Putting $G=\Delta(\omega)$, we apply the above result.

We know that there exists a face $F$ of $S$ with $\dim F=s-1$.

\noindent 8.

\noindent $(a)$. It follows from definitions.

\noindent $(b)$, $(c)$. By $(a)$ $F^\circ=F\Leftrightarrow \partial F=\emptyset$.

First, we consider the case $F=L$. We have $F=L=\Vect(L)=\Vect(F)$, and $\partial F=F\cap\Clos(\Vect(F)- F)=F\cap\Clos(\emptyset)=\emptyset$.
For any $\omega\in F^\vee=\Vect(F)^\vee$ we have $\Delta(\omega)=\{x\in F|\langle\omega,x\rangle=0\}=F$. Therefore, $\mathcal{F}(F)=\{F\}$, $\mathcal{F}(F)-\{F\}=\emptyset$ and $\cup_{G\in\mathcal{F}(F)-\{F\}}G=\emptyset=\partial F$.

We consider the case $F\neq L$. By 5 we know $\mathcal{F}(F)^1\neq \emptyset$, and $\cup_{G\in\mathcal{F}(F)-\{F\}}G=\cup_{G\in\mathcal{F}(F)^1}G \neq \emptyset$.
We consider any point $a\in \cup_{G\in\mathcal{F}(F)^1}G$.
We take $G\in\mathcal{F}(F)^1$ with $a\in G$, and $\omega\in F^\vee$ with $G=\Delta(\omega, F)$.
$a\in G\subset F$, and $a\in G\subset\Vect(G)\subset\Vect(F)$.
Since $\dim\Vect(G)=\dim G=\dim F-1=\dim\Vect(F)-1$, $\Vect(G)=\{x\in\Vect(F)|\langle\omega,x\rangle=0\}$ and $F\subset\{x\in\Vect(F)|\langle\omega,x\rangle\geq 0\}$.
We know that there exists an infinite sequence $b(i),i\in\Z_0$ of elements in $\Vect(F)$ such that $a=\lim_{i\rightarrow\infty}b(i)$ and $\langle\omega, b(i)\rangle<0$ for any $i\in\Z_0$.
We take an infinite sequence $b(i),i\in\Z_0$ of elements in $\Vect(F)$ such that $a=\lim_{i\rightarrow\infty}b(i)$ and $\langle\omega, b(i)\rangle<0$ for any $i\in\Z_0$.
Since $\langle\omega, b(i)\rangle<0$ for any $i\in\Z_0$ and $F\subset\{x\in\Vect(F)|\langle\omega,x\rangle\geq 0\}$, $b(i)\in\Vect(F)- F$ for any $i\in\Z_0$.
Therefore, $a\in F\cap\Clos(\Vect(F)- F)=\partial F$.
We know $\partial F\supset \cup_{G\in\mathcal{F}(F)^1}G$, and $\partial F\neq\emptyset$.

We consider any point $a\in \partial F= F\cap\Clos(\Vect(F)- F)$.
$a\in F$. We take infinite sequence $a(i),i\in\Z_0$ of elements of $\Vect(F)- F$ with $a=\lim_{i\rightarrow\infty}a(i)$.

We take any positive definite symmetric bilinear form $(\;,\;):V\times V\rightarrow\R$, and putting $|x|=\sqrt{(x,x)}\in\R_0$ for any $x\in V$, we define a norm $|\;|:V\rightarrow\R_0$.
We know $\lim_{i\rightarrow\infty}|a-a(i)|=0$.

We consider any $i\in\Z_0$.
We take any point $b(i)\in F$ with $|a(i)-b(i)|=\min\{|a(i)-c||c\in F\}$.
By the proof of Lemma~\ref{i}.4, we know $|a(i)-b(i)|>0$,$(b(i)-a(i), b(i))=0$, and $(b(i)-a(i),c)\geq 0$ for any $c\in F$.
We take $\omega(i)\in V^*$ such that $\langle\omega(i), x\rangle=(b(i)-a(i),x)$ for any $x\in V$.
Since $\langle\omega(i), c\rangle=(b(i)-a(i),c)\geq 0$ for any $c\in F$, $\omega(i)\in F^\vee$.
Since $b(i)\in F$ and $\langle\omega(i), b(i)\rangle=(b(i)-a(i),b(i))=0$, we know $b(i)\in\Delta(\omega(i), F)$.

Now, $\Delta(\omega(i), F)=\{x\in F|\langle\omega(i),x\rangle=0\}=\{x\in F|(b(i)-a(i),x)=0\}\subset\{x\in\Vect(F) |(b(i)-a(i),x)=0\}$.
Since $b(i)\in F\subset\Vect(F), a(i)\in\Vect(F)- F\subset\Vect(F)$, we know $b(i)-a(i)\in\Vect(F)$ and $\{x\in\Vect(F) |(b(i)-a(i),x)=0\}$ is a vector subspace of $\Vect(F)$ of codimension one. 
Therefore $\dim \Delta(\omega(i), F)\leq\dim \{x\in\Vect(F) |(b(i)-a(i),x)=0\}=\dim F-1$ and $\Delta(\omega(i), F)\neq F$.

Since $a\in F$, $ |a(i)-b(i)|\leq|a(i)-a|=|a-a(i)|$.
$0\leq|a-b(i)|\leq|a-a(i)|+|a(i)-b(i)|\leq 2|a-a(i)|$.

Now, since $\lim_{i\rightarrow\infty}|a-a(i)|=0$, we know $\lim_{i\rightarrow\infty}|a-b(i)|=0$ and $a=\lim_{i\rightarrow\infty}b(i)$.
Note that $\mathcal{F}(F)$ is a finite set. 
Below we assume that $\Delta(\omega(i), F)\in\mathcal{F}(F)-\{F\}$ does not depend on $i\in\Z_0$, replacing $b(i),i\in\Z_0$ by some subsequence.
Put $G_0=\Delta(\omega(0), F)\in\mathcal{F}(F)-\{F\}$.
For any $i\in\Z_0$ we have $b(i)\in \Delta(\omega(i), F)=G_0$.
Therefore $a=\lim_{i\rightarrow\infty}b(i)\in G_0\subset\cup_{G\in\mathcal{F}(F)-\{F\}}G$.

We know $\partial F\subset\cup_{G\in\mathcal{F}(F)-\{F\}}G$, and $\partial F=\cup_{G\in\mathcal{F}(F)-\{F\}}G$.

\noindent $(d)$.
By Lemma~\ref{f}.4 we know $F^\circ\neq\emptyset$.

We consider any $a\in F^\circ$ and any $b\in F$.
We would like to show that \hfill\break$\Conv(\{a,b\})-\{b\}\subset F^\circ$.
If $a=b$, then $\Conv(\{a,b\})-\{b\}=\emptyset\subset F^\circ$.
Below we assume $a\neq b$. We consider any point $c\in\Conv(\{a,b\})-\{b\}$ and we take $t\in\R$ such that $c=ta+(1-t)b$ and $0<t\leq 1$.

Since $\{a,b\}\subset F$ and $F$ is convex, $c\in\Conv(\{a,b\})\subset\Conv(F)=F$.
Assume $c\not\in F^\circ$. $c\in\partial F$. We take $G\in\mathcal{F}(F)$ with $c\in G$ and $G\neq F$, and we take $\omega\in F^\vee$ with $G=\Delta(\omega, F)$.
We know that $c\in\Delta(\omega, F)$ and $\langle\omega, c\rangle=0$.
Since $\{a,b\}\subset F$, $\langle\omega,a\rangle\geq 0$ and $\langle\omega,b\rangle\geq 0$.
Since $a\in F^\circ$, $a\not\in\partial F$, $a\not\in G$ and $\langle\omega,a\rangle>0$.
We have
$0=\langle\omega, c\rangle=\langle\omega, ta+(1-t)b\rangle=t\langle\omega,a\rangle+(1-t)\langle\omega, b\rangle>0$, which is a contradiction. We know $c \in F^\circ$ and $\Conv(\{a,b\})-\{b\}\subset F^\circ$.

Consider any $a,b\in F^\circ$. By the above $\Conv(\{a,b\})=(\Conv(\{a,b\})-\{b\})\cup\{b\}\subset F^\circ$.
We know that $F^\circ$ is convex.

Since $F^\circ\subset F$ and $F$ is closed, $\Clos(F^\circ)\subset\Clos(F)=F$.
Consider any $b\in F$. We take an $a\in F^\circ\neq\emptyset$ with $a\neq b$.
Since $\Conv(\{a,b\})-\{b\}\subset F^\circ$, $ta+(1-t)b\in F^\circ$ for any $t\in\R$ with $0<t\leq 1$.
Therefore, $b=\lim_{t\rightarrow 0}ta+(1-t)b\in\Clos(F^\circ)$.
We know $\Clos(F^\circ)\supset F$ and $\Clos(F^\circ)=F$.

\noindent 13.
In case $\ell<s$ we take any element $\omega_F\in\Delta(F)- M$ for any $F\in\mathcal{F}(S)^1$.
Note that $\ell<s$, if and only if, $\mathcal{F}(S)^1\neq\emptyset$.

Consider any $F\in\mathcal{F}(S)^1$.
Since $\omega_F\in\Delta(F)\subset S^\vee$, we have $\R_0\omega_F\subset S^\vee$ and $S=S^{\vee\vee}\subset(\R_0\omega_F)^\vee$.

Since $S\subset\Vect(S)$, we know
$$S\subset\bigcap_{F\in\mathcal{F}(S)^1}(\R_0\omega_F)^\vee\cap\Vect(S).$$

If $\ell=s$, then we have $S=\Vect(S)$ by 2, $\mathcal{F}(S)^1=\emptyset$ by 5, and
$S=\Vect(S)= \cap_{F\in\mathcal{F}(S)^1}(\R_0\omega_F)^\vee\cap\Vect(S)$.

Assume
$$(\bigcap_{F\in\mathcal{F}(S)^1}(\R_0\omega_F)^\vee\cap\Vect(S))- S\neq\emptyset.$$
We will deduce a contradiction from this assumption.

It follows $\ell<s$ and $\mathcal{F}(S)^1\neq\emptyset$.

Take any $a\in (\cap_{F\in\mathcal{F}(S)^1}(\R_0\omega_F)^\vee\cap\Vect(S))- S$,
and take any $b\in S^\circ\neq\emptyset$.
$a\in\Vect(S)- S$, $b\in S^\circ\subset S\subset\Vect(S)$ and $a\neq b$.

We denote $[0,1]=\{t\in\R|0\leq t\leq 1\}$.
We know $\{(1-t)a+tb|t\in[0,1]\}=\Conv(\{a,b\})\subset\Vect(S)$.
Let $I=\{t\in[0,1]| (1-t)a+tb\in S\}$. $0\not\in I$, $1\in I$ and $I$ is a closed interval contained in $[0,1]$, since $S$ is a closed convex cone.
Since $b\in S^\circ$ and $S^\circ$ is an open subset of $\Vect(S)$, we know that there exists $t_1\in[0,1]$ such that $t_1\neq 1$ and $\{t\in[0,1]|t_1<t\leq 1\}\subset I$.
We know that there exists $t_0\in[0,1]$ such that $t_0\neq 1$, $t_0\neq 0$ and $\{t\in[0,1]|t_0\leq t\leq 1\}=I$.
Let $c=(1-t_0)a+t_0b\in\Vect(S)$.
Since $t_0\in I$, $c\in S$.
For any $t\in\R$ with $0\leq t<t_0$, $t\not\in I$ and $(1-t)a+tb\in\Vect(S)- S$.
We know $c=\lim_{t\rightarrow t_0+(-0)}\in\Clos(\Vect(S)- S)$, and $c\in S\cap\Clos(\Vect(S)- S)=\partial S$. 
By 8 we know that there exists a face $G_0$ of $S$ satisfying $G_0\neq S$ and $c\in G_0$.
We take a face $G_0$ of $S$ satisfyingt $G_0\neq S$ and $c\in G_0$.
By 5 we know that there exists $F_0\in\mathcal{F}(S)^1$ with $G_0\subset F_0$.
We take $F_0\in\mathcal{F}(S)^1$ with $G_0\subset F_0$.
$c\in G_0\subset F_0\subset\Delta(\omega_{F_0})$, since $\omega_{F_0}\in\Delta(F_0)$.
We know $\langle\omega_{F_0},c\rangle=0$.

Now, since $a\in\cap_{F\in\mathcal{F}(S)^1}(\R_0\omega_F)^\vee\cap\Vect(S)\subset(\R_0\omega_{F_0})^\vee$, $\langle\omega_{F_0},a\rangle\geq 0$.
Since $b\in S^\circ\subset S$ and $\omega_{F_0}\in\Delta(F_0)\subset S^\vee$,$\langle\omega_{F_0},b\rangle\geq 0$.
Since $b\not\in \partial S=\cup_{G\in\mathcal{F}(S)-\{S\}}G$, $b\not\in F_0$.
$F_0\subset \Delta(\omega_{F_0})\subset S$.
If $\Delta(\omega_{F_0})=S$, then $\omega_{F_0}\in\Delta(S)=M$, which contradicts $\omega_{F_0}\not\in M$. 
Therefore, $\Delta(\omega_{F_0})\neq S$ and $\dim\Delta(\omega_{F_0})\leq s-1$ by 4.$(e)$.
Since $\dim F_0=s-1$, we have $F_0=\Delta(\omega_{F_0})$ by 4.$(e)$ and $b\not\in\Delta(\omega_{F_0})$. 
We know $\langle\omega_{F_0},b\rangle>0$.
Since $0<t_0<1$, we have
$0=\langle\omega_{F_0},c\rangle=\langle\omega_{F_0},(1-t_0)a+t_0b\rangle=(1-t_0) \langle\omega_{F_0},a\rangle+t_0\langle\omega_{F_0},b\rangle>0$, which is a contradiction.

We know
$(\cap_{F\in\mathcal{F}(S)^1}(\R_0\omega_F)^\vee\cap\Vect(S))- S=\emptyset$ and
$$S=\bigcap_{F\in\mathcal{F}(S)^1}(\R_0\omega_F)^\vee\cap\Vect(S).$$

By Lemma~\ref{j}.4 and Lemma~\ref{g}.2 we know
\begin{equation*}
\begin{split}
S^\vee&=(\bigcap_{F\in\mathcal{F}(S)^1}(\R_0\omega_F)^\vee\cap\Vect(S))^\vee\\
&=\Clos(\sum_{F\in\mathcal{F}(S)^1} \R_0\omega_F+\Vect(S)^\vee)\\
&=\sum_{F\in\mathcal{F}(S)^1} \R_0\omega_F+\Vect(S)^\vee\\
&=\Convcone(\{\omega_F| F\in\mathcal{F}(S)^1\})+M.
\end{split}
\end{equation*}

\noindent 15.
Let $F$ be any face of $S$.

\noindent $(a)$.
Take any $a\in F^\circ$.
$a\in F^\circ\subset F\subset S=S^{\vee\vee}$.
Consider any $\omega\in\Delta(F,S)$.
$\omega\in S^\vee$, $a\in F\subset\Delta(\omega)$ and $\langle\omega, a\rangle=0$.
We know $\omega\in\Delta(a,S^\vee)$ and $\Delta(F,S)\subset\Delta(a,S^\vee)$.

Consider any $\omega\in \Delta(a,S^\vee)$.
$\omega\in S^\vee$, $\langle\omega,a\rangle=0$ and $a\in\Delta(\omega,S)$.
Since $a\in F^\circ\cap\Delta(\omega,S)\neq\emptyset$, we know $F\subset\Delta(\omega,S)$ by 9, $\omega\in\Delta(F,S)$, and $\Delta(F,S)\supset\Delta(a,S^\vee)$.

We know $\Delta(F)=\Delta(F,S)= \Delta(a,S^\vee)$ is a face of $S^\vee$.

\noindent $(b)$.
By definition $\Delta(F)\subset S^\vee$.
Consider any $\omega\in\Delta(F)$.
By definition $F\subset\Delta(\omega)$ and $\langle\omega, a\rangle=0$ for any $a\in F$.
Consider any $b\in\Vect(F)=\{\sum_{a\in\Supp(\lambda)}\lambda(a)a|\lambda\in\Map'(F,\R)\}$.
We take any $\lambda\in\Map'(F,\R)$ with $b=\sum_{a\in\Supp(\lambda)}\lambda(a)a$.
For any $a\in\Supp(\lambda)$, $a\in\Supp(\lambda)\subset F$ and
$\langle\omega,b\rangle=\sum_{a\in\Supp(\lambda)}\lambda(a)\langle\omega,a\rangle=0$.
We know $\omega\in\Vect(F)^\vee$, and $\Delta(F)\subset\Vect(F)^\vee\cap S^\vee$.

Consider any $\omega\in \Vect(F)^\vee\cap S^\vee$.
$\omega\in S^\vee$.
Since $\omega\in\Vect(F)^\vee$, $\langle\omega,a\rangle=0$ for any $a\in\Vect(F)$.
Since $F\subset\Vect(F)\cap S$, $a\in S$ and $\langle\omega,a\rangle=0$ for any $a\in F$.
Therefore, $F\subset\Delta(\omega)$, $\omega\in\Delta(F)$, and $\Delta(F)\supset\Vect(F)^\vee\cap S^\vee$.

We know $\Delta(F)=\Vect(F)^\vee\cap S^\vee$.

Now, $\Delta(F)=\Vect(F)^\vee\cap S^\vee=(\Vect(F)+S)^\vee$ by Lemma~\ref{j}.3.
We denote $W=(\Vect(F)+S)\cap(-(\Vect(F)+S))$.
$W$ is a vector subspace of $V$ with $\Vect(F)\subset W$.
Consider any $a\in W$.
$a\in \Vect(F)+S$. We take $b\in\Vect(F)$ and $c\in S$ with $a=b+c$.
$-a\in W\subset \Vect(F)+S$. We take $b'\in\Vect(F)$ and $c'\in S$ with $-a=b'+c'$.
$0=a+(-a)=(b+b')+(c+c')$ and $c+c'=-( b+b')\in S\cap \Vect(F)=F$.
We take any $\omega\in S^\vee$ with $F=\Delta(\omega)$.
$\langle\omega, c+c'\rangle=0$.
We have $\langle\omega, c\rangle\geq 0$, and $\langle\omega, c'\rangle\geq 0$.
Since $\langle\omega, c\rangle+\langle\omega, c'\rangle=\langle\omega, c+c'\rangle=0$, we know $\langle\omega, c\rangle=\langle\omega, c'\rangle=0$, $c\in\Delta(\omega)=F\subset\Vect(F)$.
Therefore, $a=b+c\in\Vect(F)$, $\Vect(F)\supset W$ and $\Vect(F)=W$.
$\Vect(\Delta(F))=\Vect((\Vect(F)+S)^\vee)= W^\vee=\Vect(F)^\vee$.

We know $\Vect(\Delta(F))= \Vect(F)^\vee$.

\noindent $(c)$.
By $(b)$, $\dim\Delta(F)=\dim\Vect(\Delta(F))=\dim\Vect(F)^\vee=\dim V-\dim\Vect(F)=\dim V-\dim F$.
Replacing the pair $(F,S)$ by $(\Delta(F), S^\vee)$, we know $\Delta(\Delta(F))\in\mathcal{F}(S)$ by $(a)$.
By $(b)$, $\dim\Delta(\Delta(F))=\dim\Vect(\Delta(\Delta(F)))=\dim\Vect(\Delta(F))^\vee =\dim V-\dim\Vect(\Delta(F))=\dim V-\dim\Delta(F)=\dim V-(\dim V-\dim F)=\dim F$.

Consider any $a\in F$.
For any $\omega\in\Delta(F)$, $a\in F\subset\Delta(\omega)$, $\langle\chi,a\rangle=0$, and $\omega\in\Delta(a)$.
We know $\Delta(F)\subset\Delta(a)$, and $a\in\Delta(\Delta(F))$.
We know $F\subset\Delta(\Delta(F))$.

Since $\dim F=\dim\Delta(\Delta(F))$, we know $\Delta(\Delta(F))=F$.

Consider any $\omega\in\Delta^\circ(F)$.
$\omega\in S^\vee$ and $F=\Delta(\omega)$.
We know $\omega\in\Delta(F)$.

Consider any $\Lambda\in\mathcal{F}(\Delta(F))$ with $\Lambda\neq \Delta(F)$.
Assume $\omega\in\Lambda$.
We will deduce a contradiction from this assumption.
Replacing the pair $(F,S)$ by $(\Lambda, S^\vee)$, we know $\Delta(\Lambda)\in\mathcal{F}(S)$ by $(a)$.
We denote $G=\Delta(\Lambda)$.
Furthermore, by $(a)$ we know $\Delta(G)\in\mathcal{F}(S^\vee)$.

By $(b)$ we know $\Vect(G)=\Vect(\Lambda)^\vee$ and $\Vect(\Delta(G))=\Vect(G)^\vee=\Vect(\Lambda)$.
Therefore, $\dim G=\dim\Vect(G)=\dim\Vect(\Lambda)^\vee=\dim V-\dim\Vect(\Lambda)=\dim V-\dim\Lambda>\dim V-\dim\Delta(F)=\dim F$, and
$\dim\Delta(G)=\dim\Vect(\Delta(G))=\dim\Vect($\hfill\break$\Lambda) =\dim\Lambda$.

We consider any $\chi\in\Lambda$.
For any $a\in G=\Delta(\Lambda)$, $\chi\in\Lambda\subset\Delta(a)$.
Therefore, $\langle\chi,a\rangle=0$, $G\subset\Delta(\chi)$, and $\chi\in\Delta(G)$.
We know $\Lambda\subset\Delta(G)$.
Since $\dim\Delta(G)= \dim\Lambda$, we know $\omega\in\Lambda=\Delta(G)$,
and $G\subset\Delta(\omega)=F$.
We have $\dim F<\dim G\leq\dim F$, which is a contradiction.

We know $\omega\not\in\Lambda$.
$\omega\not\in\partial\Delta(F)$ by 8.$(c)$.
Therefore, $\omega\in\Delta(F)-\partial\Delta(F)=\Delta(F)^\circ$.

We know $\Delta^\circ(F)\subset\Delta(F)^\circ$.

Consider any $\omega\in\Delta(F)^\circ$.
Since $\omega\in\Delta(F)^\circ\subset\Delta(F)$, $F\subset\Delta(\omega)$.
Consider any $a\in\Delta(\omega)$. $\langle\omega,a\rangle=0$ and $\omega\in\Delta(a)\cap\Delta(F)^\circ\neq\emptyset$. By 10 we know $\Delta(F)\subset\Delta(a)$ and $a\in\Delta(\Delta(F))=F$.
We know $F\supset\Delta(\omega)$, $F=\Delta(\omega)$ and
$\omega\in\Delta^\circ(F)$.

We know $\Delta^\circ(F)\supset\Delta(F)^\circ$, and $\Delta^\circ(F)=\Delta(F)^\circ$.

By 8.$(d)$ we know $\Delta(F)=\Clos(\Delta(F)^\circ)=\Clos(\Delta^\circ(F))$.
\end{proof}

\begin{cor}
\label{good category}
\begin{enumerate}
\item
For any convex polyhedral cone $S$ in $V$, the dual cone $S^\vee$ is a convex polyhedral cone in $V^*$.
If moreover, $S$ is rational over $N$, then $S^\vee$ is rational over $N^*$.
\item
For any convex polyhedral cones $S$ and $T$ in $V$, $S+T$ and $S\cap T$ are convex polyhedral cones in $V$.
If moreover, $S$ and $T$ are rational over $N$, then $S+T$ and $S\cap T$ are rational over $N$.
\item
For any convex polyhedral cones $S$ and $T$ in $V$, $(S+T)^\vee=S^\vee\cap T^\vee$ and
$(S\cap T)^\vee=S^\vee+T^\vee$.
\item
For any convex polyhedral cone $S$ in $V$, any finite dimensional vector space $W$ over $\R$ and any homomorphism $\pi:V\rightarrow W$ of vector spaces over $\R$, $\pi(S)$ is a convex polyhedral cone in $W$.
If moreover, $S$ and $\pi^{-1}(0)$ are rational over $N$, then $\pi(S)$ is rational over $Q$ for any lattice $Q$ of $W$ with $\pi(N)=Q\cap\pi(V)$.
\item
For any convex polyhedral cone $S$ in $V$, any finite dimensional vector space $U$ over $\R$ and any homomorphism $\nu:U\rightarrow V$ of vector spaces over $\R$, $\nu^{-1}(S)$ is a convex polyhedral cone in $U$.
If moreover, $S$ and $\nu(U)$ are rational over $N$, then $\nu^{-1}(S)$ is rational over $K$ for any lattice $K$ of $U$ with $\nu(K)=N\cap\nu(U)$.

\end{enumerate}
\end{cor}

\begin{lemma}
\label{several cones}
Let $m\in \Z_+$ be any positive integer, and let $S$ be any mapping from the set $\{1,2,\ldots,m\}$ to the set of all convex polyhedral cones in $V$.
We denote
$$\bar{S}=\bigcap_{i\in\{1,2,\ldots,m\}}S(i)\subset V.$$
\begin{enumerate}
\item
$\bar{S}$ is a convex polyhedral cone in $V$.
$\bar{S}^\vee=\sum_{i\in\{1,2,\ldots,m\}} S(i)^\vee$.
If $S(i)$ is rational over $N$ for any $i\in\{1,2,\ldots,m\}$, then $\bar{S}$ is rational over $N$.
\item
If $\cap_{i\in\{1,2,\ldots,m\}}S(i)^\circ\neq\emptyset$, then $\bar{S}^\circ=\cap_{i\in\{1,2,\ldots,m\}}S(i)^\circ$.
\end{enumerate}

Let $\bar{F}$ be any face of $\bar{S}$.
\begin{enumerate}
\setcounter{enumi}{2}
\item
There exists uniquely a face $F(i)$ of $S(i)$ with $\bar{F}^\circ\subset F(i)^\circ$ for any $i\in\{1,2,\ldots, m\}$.
\end{enumerate}

Below, we assume that $F(i)\in\mathcal{F}(S(i))$ and $\bar{F}^\circ\subset F(i)^\circ$ for any $i\in\{1,2,\ldots,m\}$.

\begin{enumerate}
\setcounter{enumi}{3}
\item
$\bar{F}\subset F(i)$ for any $i\in\{1,2,\ldots,m\}$.
\item
$$\bar{F}=\bigcap_{i\in\{1,2,\ldots,m\}}F(i).$$
\item
$$\bar{F}^\circ=\bigcap_{i\in\{1,2,\ldots,m\}}F(i)^\circ.$$
\item
$$\Vect(\bar{F})=\bigcap_{i\in\{1,2,\ldots,m\}}\Vect(F(i)).$$
\item
$$\Delta(\bar{F},\bar{S})=\sum_{i\in\{1,2,\ldots,m\}}\Delta(F(i),S(i)).$$
\end{enumerate}

Let $G(i)$ be any face of $S(i)$ for any $i\in\{1,2,\ldots,m\}$.
\begin{enumerate}
\setcounter{enumi}{8}
\item
The intersection $\cap_{i\in\{1,2,\ldots,m\}}G(i)$ is a face of $\bar{S}$.
\item
If $\bar{F}\subset G(i)$ then $F(i)\subset G(i)$, for any $i\in\{1,2,\ldots,m\}$.
\item
If
$$\bar{F}=\bigcap_{i\in\{1,2,\ldots,m\}}G(i),$$
then $F(i)\subset G(i)$ for any $i\in\{1,2,\ldots,m\}$ and the following three conditions are equivalent:
\begin{enumerate}
\item
$$\bigcap_{i\in\{1,2,\ldots,m\}}G(i)^\circ\neq\emptyset.$$
\item
$F(i)=G(i)$ for any $i\in\{1,2,\ldots,m\}$.
\item
$$\bar{F}^\circ=\bigcap_{i\in\{1,2,\ldots,m\}}G(i)^\circ$$
\end{enumerate}
\end{enumerate}
\end{lemma}

\section{Simplicial cones}
\label{simplex}

We study simplicial cones.

Let $V$ be any finite dimensional vector space over $\R$, and let $N$ be any lattice of $V$.

\begin{lemma}
\label{simplex1}
Let $S$ be any simplicial cone over $N$ in $V$.
\begin{enumerate}
\item
The cone $S$ is a rational polyhedral cone over $N$ in $V$ satisfying $S\cap(-S)=\{0\}$.
\item
$\sharp\mathcal{F}(S)_1=\dim S$.
$0\leq\dim S\leq\dim V$.
\item
The intersection $N\cap\Vect(S)$ is a lattice of $\Vect(S)$.
$S$ is a simplicial cone over $N\cap\Vect(S)$ in $\Vect(S)$.
The residue module $N/(N\cap\Vect(S))$ is a free module over $\Z$ of finite rank.
$\Rank(N/(N\cap\Vect(S)))=\dim V-\dim S$.
\item
Any face of $S$ is a simplicial cone over $N$ in $V$.
The set $\{0\}$ is a face of $S$.
For any face $F$ of $S$, $\mathcal{F}(F)\subset\mathcal{F}(S)$ and $\mathcal{F}(F)_1\subset\mathcal{F}(S)_1$.
\item
If $\dim S=1$, then there exists uniquely an element $b_{S/N}$ of $S\cap N$ satisfying $S\cap N=\Z_0b_{S/N}$.
\end{enumerate}
\end{lemma}

\begin{definition}
\label{barycenter}
Let $S$ be any simplicial cone over $N$ in $V$.

If $\dim S=1$, we take the unique element $b_{S/N}$ of $S\cap N$ satisfying $S\cap N=\Z_0b_{S/N}$.

If $\dim S\neq 1$, we put
$$b_{S/N}=\sum_{E\in\mathcal{F}(S)_1} b_{E/N}\in S\cap N.$$

We call $b_{S/N}\in S\cap N$ the \emph{barycenter} of $S$ over $N$.
When we need not refer to $N$, we also write simply $b_S$, instead of $b_{S/N}$.
\end{definition}

\begin{lemma}
\label{simplex2}
Let $S$ be any simplicial cone over $N$ in $V$.
\begin{enumerate}
\item
The set $\{b_E|E\in\mathcal{F}(S)_1\}$ is a basis over $\Z$ of the lattice $N\cap\Vect(S)$ of $\Vect(S)$, and it is a basis over $\R$ of the vector space $\Vect(S)$.
\begin{equation*}
\begin{split}
S&=\sum_{E\in\mathcal{F}(S)_1}\R_0b_E=\Convcone(\{b_E|E\in\mathcal{F}(S)_1\}),\\
S^\circ&=\sum_{E\in\mathcal{F}(S)_1}\R_+b_E.
\end{split}
\end{equation*}
For any $E\in\mathcal{F}(S)_1$, $E=\R_0b_E$.
$\mathcal{F}(S)_1=\{\R_0b_E| E\in\mathcal{F}(S)_1\}$.
\item
If a basis $B$ over $\Z$ of $N$ and a subset $C$ of $B$ satisfies $S=\Convcone(C)$, then $\sharp C=\dim S$ and $C=\{b_E|E\in\mathcal{F}(S)_1\}$.
\item
$\mathcal{F}(F)_1\in 2^{\mathcal{F}(S)_1}$ and $\sharp \mathcal{F}(F)_1=\dim F$  for any $F\in\mathcal{F}(S)$.
$\mathcal{F}(F)_1\subset\mathcal{F}(G)_1$ for any $F\in\mathcal{F}(S)$ and any $G\in\mathcal{F}(S)$ satisfying $F\subset G$.

$\sum_{E\in X}E\in\mathcal{F}(S)$ and $\dim(\sum_{E\in X}E)=\sharp X$ for any $X\in 2^{\mathcal{F}(S)_1}$.
$\sum_{E\in X}E\subset\sum_{E\in Y}E$ for any $X\in 2^{\mathcal{F}(S)_1}$ and any $Y\in 2^{\mathcal{F}(S)_1}$ satisfying $X\subset Y$.

The mapping from $\mathcal{F}(S)$ to $2^{\mathcal{F}(S)_1}$ sending $F\in\mathcal{F}(S)$ to $\mathcal{F}(F)_1\in 2^{\mathcal{F}(S)_1}$ and the mapping from $2^{\mathcal{F}(S)_1}$ to $\mathcal{F}(S)$ sending $X\in 2^{\mathcal{F}(S)_1}$ to $\sum_{E\in X}E\in\mathcal{F}(S)$ are bijective mappings preserving the inclusion relation between  $\mathcal{F}(S)$ and $2^{\mathcal{F}(S)_1}$, and they are the inverse mappings of each other.

Furthermore, if $F\in\mathcal{F}(S)$ corresponds to $X\in 2^{\mathcal{F}(S)_1}$ by them, then $\dim F=\sharp X$.
The element $\{0\}\in\mathcal{F}(S)$ corresponds to $\emptyset\in 2^{\mathcal{F}(S)_1}$ by them, and $S\in\mathcal{F}(S)$ corresponds to $\mathcal{F}(S)_1\in 2^{\mathcal{F}(S)_1}$ by them.
\item
For any $X\in 2^{\mathcal{F}(S)_1}$ and any $Y\in 2^{\mathcal{F}(S)_1}$,
\begin{equation*}
\begin{split}
(\sum_{E\in X}E)\cap(\sum_{E\in Y}E)&=\sum_{E\in X\cap Y}E,\\
(\sum_{E\in X}E)+(\sum_{E\in Y}E)&=\sum_{E\in X\cup Y}E.
\end{split}
\end{equation*}
\item
For any $F\in\mathcal{F}(S)$ and any $G\in\mathcal{F}(S)$ the following claims hold:
\begin{enumerate}
\item
$F\cap G\in\mathcal{F}(S)$ and $\mathcal{F}(F\cap G)_1=\mathcal{F}(F)_1\cap\mathcal{F}(G)_1$.
\item
$F+G\in\mathcal{F}(S)$ and $\mathcal{F}(F+ G)_1=\mathcal{F}(F)_1\cup\mathcal{F}(G)_1$.
$F\subset F+G$ and $G\subset F+G$.
If $H\in \mathcal{F}(S)$ satisfies $F\subset H$ and $G\subset H$, then $F+G\subset H$.
$(F+G)^\circ=F^\circ+G^\circ$.
\item
$F\cap G=\{0\}$ and $F+G=S$, if and only if, $\mathcal{F}(F)_1\cap\mathcal{F}(G)_1=\emptyset$ and $\mathcal{F}(F)_1\cup\mathcal{F}(G)_1=\mathcal{F}(S)_1$
\end{enumerate}
\end{enumerate}
\end{lemma}

\begin{definition}
\label{opposite}
Let $S$ be any simplicial cone over $N$ in $V$, and let $F$ be any face of $S$.
We denote
$$F\Op|S=\sum_{E\in\mathcal{F}(S)_1-\mathcal{F}(F)_1}E\in\mathcal{F}(S),$$
and we call $F\Op|S$ the \emph{opposite face} of $F$ over $S$.
When we need not refer to $S$, we also write simply $F\Op$, instead of $F\Op|S$.
\end{definition}

\begin{lemma}
\label{opposite2}
Let $S$ be any simplicial cone over $N$ in $V$.
\begin{enumerate}
\item
For any face $F$ of $S$, the following claims hold:
\begin{enumerate}
\item
$F\Op=F\Op|S$ is a face of $S$.
$\dim F+\dim F\Op=\dim S$.
\item
$F\cap F\Op=\{0\}$ and $F+F\Op=S$. If $G\in\mathcal{F}(S)$, $F\cap G=\{0\}$ and $F+G=S$, then $G=F\Op$.
\item
$(F\Op)\Op=F$.
\item
$\mathcal{F}(F\Op)_1=\mathcal{F}(S)_1-\mathcal{F}(F)_1$.
If $G\in\mathcal{F}(S)$ and $\mathcal{F}(G)_1=\mathcal{F}(S)_1-\mathcal{F}(F)_1$,
then $G= F\Op$
\end{enumerate}
\item
$\{0\}\Op=S$. $S\Op=\{0\}$.
\item
For any $F\in\mathcal{F}(S)$ and any $G\in\mathcal{F}(S)$, the following claims hols:
\begin{enumerate}
\item
$F\subset G$, if and only if, $F\Op\supset G\Op$.
\item
$(F\cap G)\Op=F\Op+G\Op$.
\item
$(F+G)\Op=F\Op\cap G\Op$.
\end{enumerate}
\item
Consider any $F\in\mathcal{F}(S)$ and any $G\in\mathcal{F}(S)$ with $F\subset G$.
$F\in\mathcal{F}(G)$, $F\Op|G=(F\Op|S)\cap G$, $F\Op|S=(F\Op|G)+(G\Op|S)$ and $(F\Op|G)\cap(G\Op|S)=\{0\}$.
\item
The mapping from $\mathcal{F}(S)$ to itself sending $F\in \mathcal{F}(S)$ to $F\Op\in\mathcal{F}(S)$ is a bijective mapping reversing the inclusion relation. Its inverse mapping is equal to itself.
\end{enumerate}
\end{lemma}

\begin{lemma}
\label{mirror}
\begin{enumerate}
\item
For any simplicial cone $S$ over $N$ in $V$ with $\dim S=\dim V$, the set $\{b_E|E\in\mathcal{F}(S)_1\}$ is a basis of $N$ over $\Z$.

For any basis $B$ of $N$ over $\Z$, $\Convcone(B)$ is a simplicial cone over $N$ in $V$ with $\dim \Convcone(B)=\dim V$.

The mapping sending any simplicial cone $S$ over $N$ in $V$ with $\dim S=\dim V$ to $\{b_E|E\in\mathcal{F}(S)_1\}$ and the mapping sending any basis $B$ of $N$ over $\Z$ to $\Convcone(B)$ are bijective mappings between the set of all simplicial cones $S$ over $N$ in $V$ with $\dim S=\dim V$ and the set of all bases $B$ of $N$ over $\Z$, and they are the inverse mappings of each other.
\end{enumerate}

Below we consider any simplicial cone $S$ over $N$ in $V$ with $\dim S=\dim V$. Note that $\{b_E|E\in\mathcal{F}(S)_1\}$ is a basis of $N$ over $\Z$ and it is a basis of $V$ over $\R$. By $\{b_E^\vee|E\in\mathcal{F}(S)_1\}$ we denote the dual basis of $\{b_E|E\in\mathcal{F}(S)_1\}$, which is a basis of $V^*$ over $\R$. We assume that for any $D\in\mathcal{F}(S)_1$ and any $E\in\mathcal{F}(S)_1$
\begin{equation*}
\langle b_D^\vee, b_E\rangle=
\begin{cases}
1&\text{if $D=E$},\\
0&\text{if $D\neq E$}.
\end{cases}
\end{equation*}

\begin{enumerate}
\setcounter{enumi}{1}
\item
The set $\{b_E^\vee|E\in\mathcal{F}(S)_1\}$ is a basis of $N^*$ over $\Z$.
\item
$S^\vee$ is a simplicial cove over $N^*$ in $V^*$ with $\dim S^\vee=\dim V^*$.
$$S^\vee=\Convcone(\{b_E^\vee|E\in\mathcal{F}(S)_1\}).$$
For any $E\in\mathcal{F}(S)_1$, $\R_0 b_E^\vee\in\mathcal{F}(S^\vee)_1$ and $b_{\R_0 b_E^\vee}= b_E^\vee$.
$\mathcal{F}(S^\vee)_1=\{\R_0b_E^\vee| E\in\mathcal{F}(S)_1\}$, and
$\{b_D|D\in \mathcal{F}(S^\vee)_1\}=\{b_E^\vee|E\in\mathcal{F}(S)_1\}$.
\item
For any subset $X$ of $\mathcal{F}(S)_1$
\begin{equation*}
\begin{split}
\Delta((\sum_{E\in X}\R_0b_E)\Op,S)&=\sum_{E\in X}\R_0b_E^\vee,\\
\Delta((\sum_{E\in X}\R_0b_E^\vee)\Op,S^\vee)&=\sum_{E\in X}\R_0b_E.
\end{split}
\end{equation*}
\end{enumerate}
\end{lemma}

\begin{lemma}
\label{barycenter property}
Let $S$ be any simplicial cone over $N$ in $V$.
\begin{enumerate}
\item
Let $W$ be any finite dimensional vector space over $\R$; let $Q$ be any lattice of $W$, and let $T$ be any simplicial cone over $Q$ in $W$.
$\dim S=\dim T$, if and only if, there exists an isomorphism $\phi:\Vect(S)\rightarrow\Vect(T)$ of vector spaces over $\R$ satisfying $\phi(S)=T$ and $\phi(N\cap\Vect(S))=Q\cap\Vect(T)$.
\item
$b_S=b_{S/N}\in S^\circ\cap N$.
$\R_0b_S\cap N=\Z_0b_S$.
\item
Consider any $a\in S\cap N$ with $\R_0a\cap N=\Z_0a$.
$\phi(a)=a$ for any homomorphism $\phi:\Vect(S)\rightarrow\Vect(S)$ of vector spaces over $\R$ satisfying $\phi(S)=S$ and $\phi(N\cap\Vect(S))=N\cap\Vect(S)$, if and only if, $a=0$ or $a=b_S$.
\item
$b_S=0\Leftrightarrow\dim S=0$.
$\R_0b_S\subset S$.
$\R_0b_S=S\Leftrightarrow \R_0b_S$ is a face of $S\Leftrightarrow \dim S\leq 1$
\item
Assume $F\in\mathcal{F}(S), \dim F\geq 1, \Lambda\in\mathcal{F}(S)$, and $F\not\subset\Lambda$.
\begin{enumerate}
\item
$\Lambda+\R_0b_F$ is a simplicial cone over $N$ in $V$.
$\dim(\Lambda+\R_0b_F)=\dim\Lambda+1$.
\item
$\R_0b_F\in\mathcal{F}(\Lambda+\R_0b_F)_1$.
$\Lambda\in\mathcal{F}(\Lambda+\R_0b_F)^1$.
$\R_0b_F\cap\Lambda=\{0\}$.
$\R_0b_F=\Lambda\Op|(\Lambda+\R_0b_F)$.
$\Lambda=(\R_0b_F)\Op|(\Lambda+\R_0b_F)$.
\item
$\mathcal{F}(\Lambda)\subset\{\Lambda'\in\mathcal{F}(S)|F\not\subset\Lambda'\}$.
$\mathcal{F}(\Lambda)=\{\Lambda'\in\mathcal{F}(\Lambda+\R_0b_F)| \R_0b_F\not\subset\Lambda'\}$.
$\{\Lambda'+\R_0b_F|\Lambda'\in\mathcal{F}(\Lambda)\} =\{\Lambda'\in\mathcal{F}(\Lambda+\R_0b_F)| \R_0b_F\subset\Lambda'\}$.
\item
$\Lambda+\R_0b_F\subset \Lambda+F\in\mathcal{F}(S)$.
$(\Lambda+\R_0b_F)^\circ\subset (\Lambda+F)^\circ$.
\item
If $\dim F\geq 2$, then $\Lambda+\R_0b_F\not\in\mathcal{F}(S)$.
If $\dim F=1$, then $\Lambda+\R_0b_F=\Lambda+F\in\mathcal{F}(S)$.
\end{enumerate}
\item
Assume $F\in\mathcal{F}(S), \dim F\geq 1, \Lambda\in\mathcal{F}(S), F\not\subset\Lambda, \Lambda'\in\mathcal{F}(S)$, and $F\not\subset\Lambda'$.
\begin{enumerate}
\item
$(\Lambda+\R_0b_F)\cap\Lambda'=\Lambda\cap\Lambda'\in\mathcal{F}(\Lambda)\subset\mathcal{F}(\Lambda+\R_0b_F)$.
\item
$(\Lambda+\R_0b_F)\cap(\Lambda'+\R_0b_F)=(\Lambda\cap\Lambda')+ \R_0b_F\in\mathcal{F}(\Lambda+\R_0b_F)$.
\item
$\Lambda+\R_0b_F\subset\Lambda'+\R_0b_F$, if and only if, $\Lambda\subset\Lambda'$.
\item
$\Lambda+\R_0b_F=\Lambda'+\R_0b_F$, if and only if, $\Lambda=\Lambda'$.
\end{enumerate}
\item
Assume $F\in\mathcal{F}(S)$ and $\dim F\geq 1$.
\begin{enumerate}
\item
$\{E\Op|S\:|\:E\in\mathcal{F}(F)_1\}$ is equal to the set of maximal elements in $\{\Lambda\in\mathcal{F}(S)|F\not\subset\Lambda\}$ with respect to the inclusion relation.
\item
$$S=\bigcup_{\Lambda\in\mathcal{F}(S),F\not\subset\Lambda}(\Lambda+\R_0b_F)
=\bigcup_{E\in\mathcal{F}(F)_1}((E\Op|S)+\R_0b_F).$$
\item
$$S-(\bigcup_{\Lambda\in\mathcal{F}(S),F\not\subset\Lambda}\Lambda)
=\bigcup_{\Lambda\in\mathcal{F}(S),F\subset\Lambda}\Lambda^\circ
=\bigcup_{\Lambda\in\mathcal{F}(S),F\not\subset\Lambda}(\Lambda+\R_0b_F)^\circ.$$
\item
If $\dim F=1$, then
$\{\Lambda+\R_0b_F|\Lambda\in\mathcal{F}(S), F\not\subset\Lambda\}
=\{\Lambda\in\mathcal{F}(S)| F\subset\Lambda\}$.
\end{enumerate}
\end{enumerate}
\end{lemma}

\section{Convex polyhedral cone decompositions}
\label{decomposition}

We begin the study of convex polyhedral cone decompositions. We define notations and concepts to develop our theory.

Let $V$ be any finite dimensional vector space over $\R$, and let $N$ be any lattice of $V$.
Let $\mathcal{E}$ be any finite set whose elements are convex polyhedral cones in $V$.

We say that $\mathcal{E}$ is \emph{rational} over $N$, if any $\Delta\in\mathcal{E}$ is rational over $N$.
We say that $\mathcal{E}$ is \emph{simplicial} over $N$, if any $\Delta\in\mathcal{E}$ is a simplicial cone over $N$.

We denote
\begin{equation*}
\begin{split}
|\mathcal{E}|&=\bigcup_{\Delta\in \mathcal{E}}\Delta\subset V,\\
|\mathcal{E}|^\circ&=\bigcup_{\Delta\in \mathcal{E}}\Delta^\circ\subset V.
\end{split}
\end{equation*}
We call $|\mathcal{E}|$ and $|\mathcal{E}|^\circ$ the \emph{support} of $\mathcal{E}$ and the \emph{open support} of $\mathcal{E}$ respectively.
We denote
$$\mathcal{E}\Mx=\{\Delta\in \mathcal{E}|\text{ If }\Lambda\in\mathcal{E}\text{ and }\Delta\subset\Lambda\text{, then }\Delta=\Lambda\}\subset\mathcal{E}.$$
We call any element in $\mathcal{E}\Mx$ a \emph{maximal element} of $\mathcal{E}$
and we call $\mathcal{E}\Mx$ the set of \emph{maximal elements} of $\mathcal{E}$.

Note that $\mathcal{F}(\Delta)$ is a non-empty finite set whose elements are convex polyhedral cones in $V$ for any $\Delta\in\mathcal{E}$.
We denote
$$\mathcal{E}\Fc=\bigcup_{\Delta\in\mathcal{E}}\mathcal{F}(\Delta)\subset 2^V,$$
and we call $\mathcal{E}\Fc$ the \emph{face closure} of $\mathcal{E}$.

In case $\mathcal{E}\neq\emptyset$ we define
$$\dim\mathcal{E}=\max\{\dim\Delta|\Delta\in\mathcal{E}\}\in\Z_0,$$
and we call $\dim\mathcal{E}$ the \emph{dimension} of $\mathcal{E}$.
In case  $\mathcal{E}=\emptyset$ we do not define $\dim\mathcal{E}$.
For any $i\in\Z$ we denote
$$\mathcal{E}_i=\{\Delta\in\mathcal{E}|\dim\Delta=i\},$$
\begin{equation*}
\mathcal{E}^i=
\begin{cases}
\{\Delta\in\mathcal{E}|\dim\Delta=\dim\mathcal{E}-i\}&\text{if $\mathcal{E}\neq\emptyset$},\\
\emptyset&\text{if $\mathcal{E}=\emptyset$}.
\end{cases}
\end{equation*}
$\mathcal{E}_i$ and $\mathcal{E}^i$ are subsets of $\mathcal{E}$.

Consider any subset $F$ of $V$.
We denote
\begin{equation*}
\begin{split}
\mathcal{E}\backslash F&=\{\Delta\in\mathcal{E}|\Delta\subset F\}\subset \mathcal{E},\\
\mathcal{E}/F&=\{\Delta\in\mathcal{E}|\Delta\supset F\}\subset\mathcal{E}.
\end{split}
\end{equation*}

Consider any finite dimensional vector space $W$ and any homomorphism $\pi:V\rightarrow W$ of vector spaces over $\R$.
We denote
$$\pi_*\mathcal{E}=\{\pi(\Delta)|\Delta\in\mathcal{E}\}\subset 2^W,$$
and we call $\pi_*\mathcal{E}$ the \emph{push-down} of $\mathcal{E}$ by $\pi$.

Consider any finite dimensional vector space $U$ and any homomorphism $\nu:U\rightarrow V$ of vector spaces over $\R$.
We denote
$$\nu^*\mathcal{E}=\{\nu^{-1}(\Delta)|\Delta\in\mathcal{E}\}\subset 2^U,$$
and we call $\nu^*\mathcal{E}$ the \emph{pull-back} of $\mathcal{E}$ by $\nu$.

\begin{lemma}
\label{first step}
Let $\mathcal{E}$ be any finite set whose elements are convex polyhedral cones in $V$.
\begin{enumerate}
\item
$|\mathcal{E}|$ is a closed subset of $V$. If $|\mathcal{E}|\neq\emptyset$, $|\mathcal{E}|$ is a cone in $V$.
\item
$|\mathcal{E}|^\circ\subset|\mathcal{E}|$.
$\Clos(|\mathcal{E}|^\circ)=|\mathcal{E}|$.
\item
$\mathcal{E}\Mx\subset\mathcal{E}$.
$(\mathcal{E}\Mx)\Mx=\mathcal{E}\Mx$.
$|\mathcal{E}\Mx|=|\mathcal{E}|$.
\item
$\mathcal{E}\Fc$ is a finite set whose elements are convex polyhedral cones in $V$.
$\mathcal{E}\subset\mathcal{E}\Fc$.
$(\mathcal{E}\Fc)\Fc=\mathcal{E}\Fc$.
$|\mathcal{E}\Fc|=|\mathcal{E}|$.
If $\mathcal{E}$ is rational over $N$ in $V$, then $\mathcal{E}\Fc$ is also rational over $N$ in $V$.
If $\mathcal{E}$ is simplicial over $N$ in $V$, then $\mathcal{E}\Fc$ is also simplicial over $N$ in $V$.
\item
Consider any finite dimensional vector space $W$ and any homomorphism $\pi:V\rightarrow W$ of vector spaces over $\R$. The set $\pi_*\mathcal{E}$ is a finite set whose elements are convex polyhedral cones in $W$.
$|\pi_*\mathcal{E}|=\pi(|\mathcal{E}|)$.
If $\mathcal{E}$ and $\pi^{-1}(0)$ are rational over $N$, then $\pi_*\mathcal{E}$ is rational over $Q$ for any lattice $Q$ of $W$ with $\pi(N)=Q\cap\pi(V)$.
\item
$\Id_{V*}\mathcal{E}=\mathcal{E}$.
For any finite dimensional vector spaces $W$, $W'$ and any homomorphisms $\pi:V\rightarrow W$, $\pi':W\rightarrow W'$ of vector spaces over $\R$,
$(\pi'\pi)_*\mathcal{E}=\pi'_*\pi_*\mathcal{E}$.
\item
Consider any finite dimensional vector space $U$ and any homomorphism $\nu:U\rightarrow V$ of vector spaces over $\R$. The set $\nu^*\mathcal{E}$ is a finite set whose elements are convex polyhedral cones in $U$.
$|\nu^*\mathcal{E}|=\nu^{-1}(|\mathcal{E}|)$.
If $\mathcal{E}$ and $\nu(U)$ are rational over $N$, then $\nu^*\mathcal{E}$ is rational over $K$ for any lattice $K$ of $U$ with $\nu(K)=N\cap\nu(U)$.
\item
$\Id_V^*\mathcal{E}=\mathcal{E}$.
For any finite dimensional vector spaces $U$, $U'$ and any homomorphisms $\nu:U\rightarrow V$, $\nu':U'\rightarrow U$ of vector spaces over $\R$,
$(\nu\nu')^*\mathcal{E}=\nu^{\prime*}\nu^*\mathcal{E}$.
\item
Consider any finite dimensional vector spaces $W$, $U$ and any homomorphisms $\pi:V\rightarrow W$, $\nu:U\rightarrow V$ of vector spaces over $\R$.

$\mathcal{E}=\emptyset\Leftrightarrow|\mathcal{E}|=\emptyset\Leftrightarrow0\not\in|\mathcal{E}|\Leftrightarrow|\mathcal{E}|^\circ=\emptyset\Leftrightarrow\mathcal{E}\Mx=\emptyset\Leftrightarrow\mathcal{E}\Fc=\emptyset\Leftrightarrow\pi_*\mathcal{E}=\emptyset\Leftrightarrow\nu^*\mathcal{E}=\emptyset$.
\item
Consider any subset $\mathcal{D}$ of $\mathcal{E}$, any $i\in\Z$, any subset $F$ of $V$, any finite dimensional vector spaces $W$, $U$ and any homomorphisms $\pi:V\rightarrow W$, $\nu:U\rightarrow V$ of vector spaces over $\R$.

$|\mathcal{D}|\subset|\mathcal{E}|$, $|\mathcal{D}|^\circ\subset|\mathcal{E}|^\circ$,
$\mathcal{D}\Fc\subset\mathcal{E}\Fc$,
$\mathcal{D}_i\subset\mathcal{E}_i$,
$\mathcal{D}\backslash F\subset\mathcal{E}\backslash F$,
$\mathcal{D}/F\subset\mathcal{E}/F$,
$\pi_*\mathcal{D}\subset\pi_*\mathcal{E}$, and
$\nu^*\mathcal{D}\subset\nu^*\mathcal{E}$.
\item
Let $F$ and $G$ be any subsets of $V$.
If $F\supset G$, $(\mathcal{E}\backslash F)\backslash G=\mathcal{E}\backslash G$.
If $F\subset G$, $(\mathcal{E}/F)/G=\mathcal{E}/G$.
$\mathcal{E}\backslash F=\mathcal{E}\Leftrightarrow |\mathcal{E}|\subset F$.
\end{enumerate}
\end{lemma}

\begin{definition}
\label{cpcd}
\begin{enumerate}
\item
Any subset $\mathcal{D}$ of $2^V$ satisfying the following three conditions is called a \emph{convex polyhedral cone decomposition} in $V$.
\begin{enumerate}
\item
The set $\mathcal{D}$ is a non-empty finite set whose elements are convex polyhedral cones in $V$.
\item
For any $\Delta\in\mathcal{D}$ and any $\Lambda\in\mathcal{D}$, $\Delta\cap\Lambda$ is a face of $\Delta$, and $\Delta\cap\Lambda$ is a face of $\Lambda$.
\item
For any $\Delta\in\mathcal{D}$ and any face $\Lambda$ of $\Delta$, $\Lambda\in\mathcal{D}$.
\end{enumerate}
\item
We say that $\mathcal{D}$ is a \emph{rational convex polyhedral cone decomposition} over $N$ in $V$, if $\mathcal{D}$ is a convex polyhedral cone decomposition in $V$ and $\mathcal{D}$ is rational over $N$.
\item
We say that $\mathcal{D}$ is a \emph{simplicial cone decomposition} over $N$ in $V$, if $\mathcal{D}$ is a convex polyhedral cone decomposition in $V$ and $\mathcal{D}$ is simplicial over $N$.
\item
Let $\mathcal{D}$ be any convex polyhedral cone decomposition in $V$. We call an element $L\in\mathcal{D}$ the \emph{mimimum element} of $\mathcal{D}$, if $\dim L\leq\dim\Delta$ for any $\Delta\in\mathcal{D}$.
\item
Let $\mathcal{D}$ and $\mathcal{E}$ be any finite sets whose elements are convex polyhedral cones in $V$.

If for any $\Delta\in\mathcal{D}$ there exists $\Lambda\in\mathcal{E}$ with $\Delta\subset\Lambda$, then we say that $\mathcal{D}$ is a \emph{subdivision} of $\mathcal{E}$.

If $\mathcal{D}$ is a subdivision of $\mathcal{E}$ and $|\mathcal{D}|=|\mathcal{E}|$, we say that $\mathcal{D}$ is a \emph{full subdivision} of $\mathcal{E}$.
\item
Let $J$ be any finite set and let $\mathcal{E}$ be any mapping from $J$ to the set of all finite sets whose elements are convex polyhedral cones in $V$.
For any $j\in J$, $\mathcal{E}(j)$ is a finite set whose elements are convex polyhedral cones in $V$.
We denote
\begin{equation*}
\begin{split}
\bigcap_{j\in J}^\wedge\mathcal{E}(j)
=\{\bar{\Delta}\in 2^V|&\bar{\Delta}=\cap_{j\in J}\Delta(j)\text{ for some mapping }\Delta:J\rightarrow 2^V\\
&\quad\text{such that }\Delta(j)\in\mathcal{E}(j)\text{ for any }j\in J\}\subset 2^V,
\end{split}
\end{equation*}
and we call $\hat{\cap}_{j\in J}\mathcal{E}(j)$ the \emph{real intersection} of $\mathcal{E}(j), j\in J$.

Note that $\hat{\cap}_{j\in J}\mathcal{E}(j)$ is different from the intersection $\cap_{j\in J}\mathcal{E}(j)$ of $\mathcal{E}(j), j\in J$.

When $J=\{1,2,\ldots,m\}$ for some $m\in\Z_+$, we also denote
$$\mathcal{E}(1)\hat{\cap}\mathcal{E}(2) \hat{\cap}\cdots\hat{\cap} \mathcal{E}(m)
=\bigcap_{j\in J}^\wedge\mathcal{E}(j).$$
\item
Let $\mathcal{D}$ be any convex polyhedral cone decomposition in $V$; let $W$ be any finite dimensional vector space over $\R$; let $T$ be any subset of $W$; and let $\phi:| \mathcal{D}|\rightarrow T$ be any mapping.

We say that $\phi$ is \emph{piecewise linear}, if for any $\Delta\in\mathcal{D}$ there exists a homomorphism $\phi_\Delta:\Vect(\Delta)\rightarrow W$ of vector spaces over $\R$ such that $\phi(a)= \phi_\Delta(a)$ for any $a\in\Delta$.

\item
Let $\mathcal{D}$ be any convex polyhedral cone decomposition in $V$, and let $\phi:| \mathcal{D}|\rightarrow \R$ be any piecewise linear function.

Assume that $\mathcal{D}$ is rational over $N$. We say that $\phi$ is \emph{rational} over $N$, if for any $\Delta\in\mathcal{D}$, there exists a linear function $\phi_\Delta:\Vect(\Delta)\rightarrow \R$ such that $\phi(a)= \phi_\Delta(a)$ for any $a\in\Delta$ and $\phi_\Delta(N\cap\Vect(\Delta))\subset\Q$.

Assume that the support $|\mathcal{D}|$ of $\mathcal{D}$ is convex.
We say that $\phi$ is \emph{convex} over $\mathcal{D}$, if the following two conditions are satisfied:
\begin{enumerate}
\item
For any $a\in |\mathcal{D}|$, any $b\in |\mathcal{D}|$ and any $t\in\R$ with $0\leq t\leq 1$, $\phi((1-t)a+tb)\geq (1-t)\phi(a)+t\phi(b)$.
\item
If $a\in |\mathcal{D}|$, $b\in |\mathcal{D}|$, $t\in\R$, $0<t<1$ and $\phi((1-t)a+tb)=(1-t)\phi(a)+t\phi(b)$, then $\{a,b\}\subset\Delta$ for some $\Delta\in\mathcal{D}$.
\end{enumerate}
\end{enumerate}
\end{definition}

\begin{example}
Let $S$ be any convex polyhedral cone in $V$.
$\mathcal{F}(S)$ is a convex polyhedral cone decomposition in $V$.
$|\mathcal{F}(S)|=S$.
If $S$ is a simplicial cone over $N$ in $V$, then $\mathcal{F}(S)$ is a simplicial cone decomposition over $N$ in $V$.
\end{example}

\begin{lemma}
\label{basic cpcd}
Let $\mathcal{D}$ be any convex polyhedral cone decomposition in $V$.
\begin{enumerate}
\item
$\Delta\cap\Lambda\in\mathcal{D}$ for any $\Delta\in\mathcal{D}$ and any  $\Lambda\in\mathcal{D}$.
\item
Consider any $\Delta\in\mathcal{D}$ and any $\Lambda\in\mathcal{D}$.
The following three conditions are equivalent:
\begin{enumerate}
\item
$\Delta\subset\Lambda$.
\item
$\Delta$ is a face of $\Lambda$.
\item
$\Delta^\circ\cap\Lambda\neq\emptyset$.
\end{enumerate}
\item
Consider any $\Delta\in\mathcal{D}$ and any $\Lambda\in\mathcal{D}$.
The following three conditions are equivalent:
\begin{enumerate}
\item
$\Delta=\Lambda$.
\item
$\dim\Delta=\dim\Lambda$ and $\Delta\subset\Lambda$, or $\dim\Delta=\dim\Lambda$ and $\Delta\supset\Lambda$.
\item
$\Delta^\circ\cap\Lambda^\circ\neq\emptyset$.
\end{enumerate}
\item
$\mathcal{D}=\mathcal{D}\Fc=(\mathcal{D}\Mx)\Fc$.
$|\mathcal{D}|=|\mathcal{D}|^\circ=|\mathcal{D}\Mx|$.
The suppart $|\mathcal{D}|$ is a closed subset of $V$.
For any subset $\mathcal{E}$ of $\mathcal{D}$, $\mathcal{E}\Fc=\mathcal{D}\backslash|\mathcal{E}|\subset\mathcal{D}$.
For any $\Delta\in\mathcal{D}$, $\mathcal{F}(\Delta)=\{\Delta\}\Fc=\mathcal{D}\backslash\Delta\subset\mathcal{D}$.
\item
The family $\{\Delta^\circ|\Delta\in\mathcal{D}\}$ of subsets of $V$ gives an equivalence class decomposition of $|\mathcal{D}|$, in other words, the following three conditions hold:
\begin{enumerate}
\item
$\Delta^\circ\neq\emptyset$ for any $\Delta\in\mathcal{D}$.
\item
If $\Delta^\circ\cap\Lambda^\circ\neq\emptyset$, then $\Delta^\circ=\Lambda^\circ$ for any $\Delta\in\mathcal{D}$ and any $\Lambda\in\mathcal{D}$.
\item
$$|\mathcal{D}|=\bigcup_{\Delta\in\mathcal{D}}\Delta^\circ.$$
\end{enumerate}
\item
For any subset $\mathcal{E}$ of $\mathcal{D}$, the following three conditions are equivalent:
\begin{enumerate}
\item
$\mathcal{E}$ is a convex polyhedral cone decomposition in $V$.
\item
$\mathcal{E}\neq\emptyset$ and $\mathcal{D}\backslash\Delta\subset\mathcal{E}$ for any $\Delta\in\mathcal{E}$.
\item
$\mathcal{E}\neq\emptyset$ and $\mathcal{E}=\mathcal{E}\Fc$.
\end{enumerate}
\item
For any non-empty subset $\mathcal{E}$ of $\mathcal{D}$, $\mathcal{E}\Fc$ is a convex polyhedral cone decomposition in $V$.
\item
Consider any subset $F$ of $V$.
If $\mathcal{D}\backslash F\neq\emptyset$, then $\mathcal{D}\backslash F$ is a convex polyhedral cone decomposition in $V$.
If $\mathcal{D}-(\mathcal{D}/F)\neq\emptyset$, then $\mathcal{D}-(\mathcal{D}/F)$ is a convex polyhedral cone decomposition in $V$.
\item
For any subset $X$ of $|\mathcal{D}|$, the following three conditions are equivalent:
\begin{enumerate}
\item
$X$ is a closed subset of $V$.
\item
$X$ is a closed subset of $|\mathcal{D}|$.
\item
$X\cap\Delta$ is a closed subset of $\Delta$ for any $\Delta\in\mathcal{D}$.
\end{enumerate}
\item
For any subset $Y$ of $|\mathcal{D}|$, the following two conditions are equivalent:
\begin{enumerate}
\item
$Y$ is an open subset of $|\mathcal{D}|$.
\item
$Y\cap\Delta$ is an open subset of $\Delta$ for any $\Delta\in\mathcal{D}$.
\end{enumerate}
\item
$\Delta\in\mathcal{D}\Mx$, if and only if, $\Delta^\circ$ is an open subset of $|\mathcal{D}|$ for any $\Delta\in\mathcal{D}$.
\item
$\mathcal{D}\Mx\supset\mathcal{D}^0\neq\emptyset$.
$\mathcal{D}\Mx=\mathcal{D}^0$, if and only if, $\dim\Delta=\dim\mathcal{D}$ for any $\Delta\in\mathcal{D}\Mx$.

The following four conditions are equivalent:
\begin{enumerate}
\item
$\dim\mathcal{D}=\dim\Vect(|\mathcal{D}|)$ and $\mathcal{D}\Mx=\mathcal{D}^0$.
\item
$\dim\Delta=\dim\Vect(|\mathcal{D}|)$ for any $\Delta\in\mathcal{D}\Mx$.
\item
$\Vect(\Delta)= \Vect(|\mathcal{D}|)$ for any $\Delta\in\mathcal{D}\Mx$.
\item
$\Vect(\Delta)= \Vect(\Lambda)$ for any $\Delta\in\mathcal{D}\Mx$ and any $\Lambda\in\mathcal{D}\Mx$.
\end{enumerate}

\item
Consider any $L\in\mathcal{D}$. 
The following three conditions are equivalent:
\begin{enumerate}
\item
$L$ is the minimum element of $\mathcal{D}$, in other words, $\dim L\leq\dim\Delta$ for any $\Delta\in\mathcal{D}$.
\item
$L\subset\Delta$ for any $\Delta\in\mathcal{D}$.
\item
$L=\Delta\cap(-\Delta)$ for any $\Delta\in\mathcal{D}$.
\end{enumerate}

There exists a unique element $L\in\mathcal{D}$ satisfying the above three conditions.
\end{enumerate}

Below, we assume that $L\in\mathcal{D}$ is the minimum element of $\mathcal{D}$.
\begin{enumerate}
\setcounter{enumi}{13}
\item
The minimum element $L$ is a vector subspace over $\R$ of $V$.
If $\Delta\in\mathcal{D}$ and $\Delta$ is a vector subspace over $\R$ of $V$, then $\Delta=L$.
$\dim L\leq\dim\mathcal{D}$.
\item
$\mathcal{D}_{\dim L}=\{L\}$.
$\mathcal{D}_i\neq\emptyset$, if and only if, $\dim L\leq i\leq\dim\mathcal{D}$ for any $i\in\Z$.
$\mathcal{D}^i\neq\emptyset$, if and only if, $0\leq i\leq\dim\mathcal{D}-\dim L $ for any $i\in\Z$.
\item
$\Delta+L=\Delta$, $\Delta^\circ+L=\Delta^\circ$ and  $\Vect(\Delta)+L=\Vect(\Delta)$ for any  $\Delta\in\mathcal{D}$.
\item
$\mathcal{D}$ is rational over $N$, if and only if, one of the following two conditions hold:
\begin{enumerate}
\item
$\dim L=\dim\mathcal{D}$ and $L$ is rational over $N$.
\item
$\dim L<\dim\mathcal{D}$ and any $\Delta\in\mathcal{D}$ with $\dim\Delta=\dim L+1$ is rational over $N$.
\end{enumerate}
\end{enumerate}
Below, we consider any finite dimensional vector space $W$ over $\R$ and any homomorphism $\pi:V\rightarrow W$ of vector spaces over $\R$ satisfying $\pi^{-1}(0)\subset L$.

\begin{enumerate}
\setcounter{enumi}{17}
\item
The push-down $\pi_*\mathcal{D}$ is a convex polyhedral cone decomposition in $W$, and $\pi^*\pi_*\mathcal{D}=\mathcal{D}$.
\item
For any $\Delta\in\mathcal{D}$, $\pi(\Delta)\in\pi_*\mathcal{D}$.

For any $\bar{\Delta}\in\pi_*\mathcal{D}$, $\pi^{-1}(\bar{\Delta})\in\mathcal{D}$.

The mapping from $\mathcal{D}$ to $\pi_*\mathcal{D}$ sending any $\Delta\in\mathcal{D}$ to $\pi(\Delta)\in\pi_*\mathcal{D}$ and the mapping from $\pi_*\mathcal{D}$ to $\mathcal{D}$ sending any $\bar{\Delta}\in\pi_*\mathcal{D}$ to $\pi^{-1}(\bar{\Delta})\in\mathcal{D}$ are bijective mappings preserving the inclusion relation between $\mathcal{D}$ and $\pi_*\mathcal{D}$, and they are the inverse mappings of each other.

Furthermore, if $\Delta\in\mathcal{D}$ and $\bar{\Delta}\in\pi_*\mathcal{D}$ correspond to each other by them, the following equalities hold:
\begin{enumerate}
\item
$\dim\bar{\Delta}=\dim\Delta-\dim\pi^{-1}(0)$.
\item
$\Vect(\bar{\Delta})=\pi(\Vect(\Delta))$,
$\pi^{-1}(\Vect(\bar{\Delta}))=\Vect(\Delta)$.
\item
$\bar{\Delta}^\circ=\pi(\Delta^\circ)$,
$\pi^{-1}(\bar{\Delta}^\circ)= \Delta^\circ$.
\item
$\mathcal{F}(\bar{\Delta})=\pi_*\mathcal{F}(\Delta)$,
$\pi^*\mathcal{F}(\bar{\Delta})=\mathcal{F}(\Delta)$.
\end{enumerate}
\end{enumerate}
\end{lemma}

\begin{lemma}
\label{construction}
Let $\mathcal{E}$ be any non-empty finite set whose elements are convex polyhedral cones in $V$ satisfying the following two conditions \emph{Z:}
\begin{description}
\item[\emph{(a)}]
$\Delta\cap\Lambda$ is a face of $\Delta$, and $\Delta\cap\Lambda$ is a face of $\Lambda$ for any $\Delta\in\mathcal{E}$ and any $\Lambda\in\mathcal{E}$.
\item[\emph{(b)}]
$\Delta\cap(-\Delta)=\Lambda\cap(-\Lambda)$ for any $\Delta\in\mathcal{E}$ and any $\Lambda\in\mathcal{E}$.
\end{description}

Choosing any element $\Delta\in\mathcal{E}$, we put $L=\Delta\cap(-\Delta)\subset V$.
$L$ does not depend on the choice of $\Delta\in\mathcal{E}$.
Put $\mathcal{D}=\mathcal{E}\Fc$.
\begin{enumerate}
\item
$\mathcal{D}$ is a convex polyhedral cone decomposition in $V$.
\item
$\mathcal{D}\supset\mathcal{E}$.
$|\mathcal{D}|=|\mathcal{E}|$.
$\mathcal{D}\Mx=\mathcal{E}\Mx$.
\item
$L\in\mathcal{D}$.
$L$ is the minimum element of $\mathcal{D}$.
\item
If $\mathcal{E}$ is rational over $N$, then $\mathcal{D}$ is rational over $N$.
If $\mathcal{E}$ is simplicial over $N$, then $\mathcal{D}$ is a simplicial cone decomposition over $N$.
\end{enumerate}

Assume $\dim V\geq 2$. Let $S$ be any convex polyhedral cone in $V$ with $\dim S=\dim V$; let $m\in\Z_0$; let $H$ be any mapping from $\{1,2,\ldots,m\}$ to the set of all vector subspaces of $V$ of codimension one satisfying the following three conditions:
\begin{description}
\item[\emph{(c)}]
$H(i)\neq H(j)$ for any $i\in\{1,2,\ldots,m\}$ and any $j\in\{1,2,\ldots,m\}$ with $i\neq j$.
\item[\emph{(d)}]
$H(i)\cap S^\circ\neq\emptyset$ for any $i\in\{1,2,\ldots,m\}$.
\item[\emph{(e)}]
$H(i)\cap H(j)\cap S^\circ=\emptyset$ for any $i\in\{1,2,\ldots,m\}$ and any $j\in\{1,2,\ldots,m\}$ with $i\neq j$.
\end{description}
\begin{enumerate}
\setcounter{enumi}{4}
\item
The difference $S^\circ-(\cup_{i\in\{1,2,\ldots,m\}}H(i))$ is a non-empty open set of $V$.
It has $(m+1)$ connected components.
The closure of any connected component of it is a convex polyhedral cone in $V$ whose dimension is equal to $\dim V$.
\end{enumerate}

Let $\bar{\mathcal{E}}$ denote the finite set whose elements are $(m+1)$ of closures of connected components of $S^\circ-(\cup_{i\in\{1,2,\ldots,m\}}H(i))$.
Let $\bar{\mathcal{D}}=\bar{\mathcal{E}}\Fc$, and let $\bar{L}=S\cap(-S)\cap(\cap_{i\in\{1,2,\ldots,m\}}H(i))$.
\begin{enumerate}
\setcounter{enumi}{5}
\item
$\bar{\mathcal{E}}$ satisfies the above two conditions \emph{Z}.
\item 
$\bar{\mathcal{D}}$ is a convex polyhedral cone decomposition in $V$.
$\dim\bar{\mathcal{D}}=\dim V$.
$|\bar{\mathcal{D}}|=S$.
$\bar{\mathcal{D}}\Mx=\bar{\mathcal{D}}^0=\bar{\mathcal{E}}$.
$\bar{L}$ is the miminum element of $\bar{\mathcal{D}}$.
$\{\Delta\in \bar{\mathcal{D}}^1|\Delta^\circ\subset S^\circ\}
=\{H(i)\cap S|i\in\{1,2,\ldots,m\}\}$.
$\sharp\bar{\mathcal{D}}^0=\sharp\{\Delta\in \bar{\mathcal{D}}^1|\Delta^\circ\subset S^\circ\}+1=m+1$.
For any $\Delta\in\bar{\mathcal{D}}$ with $\dim\Delta\leq\dim V-2$, $\Delta\subset\partial S$.
If $S$ is rational over $N$ and $H(i)$ is rational over $N$ for any $i\in\{1,2,\ldots,m\}$, then $\bar{\mathcal{D}}$ is rational over $N$.
\end{enumerate}
\end{lemma}

\begin{lemma}
Let $\mathcal{D}$, $\mathcal{E}$ and $\mathcal{F}$ be convex polyhedral cone decompositions in $V$.
\begin{enumerate}
\item
The following three conditions are equivalent:
\begin{enumerate}
\item
$\mathcal{D}$ is a subdivision of $\mathcal{E}$, in other words, for any $\Delta\in\mathcal{D}$ there exists $\Lambda\in\mathcal{E}$ with $\Delta\subset\Lambda$.
\item
For any $\Delta\in\mathcal{D}$, there exists uniquely $\Lambda\in\mathcal{E}$ with $\Delta^\circ\subset\Lambda^\circ$.
\item
$|\mathcal{D}|\subset|\mathcal{E}|$, and if $\Delta\in\mathcal{D}$, $\Lambda\in\mathcal{E}$ and $\Delta^\circ\cap\Lambda^\circ\neq\emptyset$ then $\Delta^\circ\subset\Lambda^\circ$.
\end{enumerate}
\item
If $\mathcal{D}$ is a subdivision of $\mathcal{E}$ and $\mathcal{E}$ is a subdivision of $\mathcal{D}$, then $\mathcal{D}=\mathcal{E}$.
\item
If $\mathcal{D}$ is a subdivision of $\mathcal{E}$ and $\mathcal{D}$ is a subdivision of $\mathcal{F}$, then $\mathcal{D}$ is a subdivision of $\mathcal{F}$.
\item
If $\mathcal{D}$ is a subdivision of $\mathcal{E}$, then $|\mathcal{D}|\subset|\mathcal{E}|$.
\item
Assume that $\mathcal{D}$ is a subdivision of $\mathcal{E}$.
For any $\Delta\in\mathcal{D}$ and any $\Lambda\in\mathcal{E}$ the following three conditions are equivalent:
\begin{enumerate}
\item
$\Delta^\circ\subset\Lambda^\circ$.
\item
$\Delta\subset\Lambda$, and for any $\Lambda'\in\mathcal{E}$ with $\Delta\subset\Lambda'$ we have $\Lambda\subset\Lambda'$.
\item
$\Delta^\circ\cap\Lambda^\circ\neq\emptyset$
\end{enumerate}
\item
If $\Delta\in\mathcal{D}$, $\Lambda\in\mathcal{E}$ and $\Delta^\circ\subset\Lambda^\circ$,
then $\Delta\subset\Lambda$, and $\dim\Delta\leq\dim\Lambda$.
\item
The following three conditions are equivalent:
\begin{enumerate}
\item
$\mathcal{D}$ is a full subdivision of $\mathcal{E}$, in other words, $\mathcal{D}$ is a subdivision of $\mathcal{E}$ and $|\mathcal{D}|=|\mathcal{E}|$.
\item
$|\mathcal{D}|=|\mathcal{E}|$ and for any $\Lambda\in\mathcal{E}$, 
$\Lambda^\circ=\cup_{\Delta\in\mathcal{D}, \Delta^\circ\subset\Lambda^\circ}\Delta^\circ$.
\item
$|\mathcal{D}|=|\mathcal{E}|$ and $|\mathcal{D}-\mathcal{D}\Mx|\supset|\mathcal{E}-\mathcal{E}\Mx|$.
\end{enumerate}
\item
Assume that $\mathcal{D}$ is a full subdivision of $\mathcal{E}$.
For any $\Lambda\in\mathcal{E}$, there exists $\Delta\in\mathcal{D}$ with $\Delta^\circ\subset\Lambda^\circ$.
For any $\Lambda\in\mathcal{E}\Mx$, there exists $\Delta\in\mathcal{D}\Mx$ with $\Delta^\circ\subset\Lambda^\circ$.
\item
Assume that $\mathcal{D}$ is a full subdivision of $\mathcal{E}$, $\Delta\in\mathcal{D}$, $\Lambda\in\mathcal{E}$ and $\Delta^\circ\subset\Lambda^\circ$.
$\Delta\in\mathcal{D}\Mx$, if and only if, $\Lambda\in\mathcal{E}\Mx$ and $\dim\Delta=\dim\Lambda$.
$\dim\mathcal{D}=\dim\mathcal{E}$.
\item
If $\mathcal{D}$ is a full subdivision of $\mathcal{E}$,
$\dim \mathcal{E}=\dim\Vect(|\mathcal{E}|)$ and $\mathcal{E}\Mx=\mathcal{E}^0$, then $\dim \mathcal{D}=\dim\Vect(|\mathcal{D}|)$ and $\mathcal{D}\Mx=\mathcal{D}^0$
\end{enumerate}
\end{lemma}

\begin{lemma}
\label{scpc}
Let $\mathcal{D}$ be any convex polyhedral cone decomposition in $V$ such that the support $|\mathcal{D}|$ of $\mathcal{D}$ is a convex polyhedral cone in $V$.
By $L\in\mathcal{D}$ we denote the minimum element of $\mathcal{D}$.
\begin{enumerate}
\item
$\mathcal{D}$ is a full subdivision of $\mathcal{F}(|\mathcal{D}|)$.
\item
$\dim\mathcal{D}=\dim |\mathcal{D}|=\dim\Vect(|\mathcal{D}|)$.
$\mathcal{D}\Mx=\mathcal{D}^0$.
\item
For any $\Delta\in\mathcal{D}\Mx$, $\Delta^\circ\subset|\mathcal{D}|^\circ$.
\item
For any $\Delta\in\mathcal{D}$, $\Delta\not\subset\partial|\mathcal{D}|$, if and only if, $\Delta^\circ\subset|\mathcal{D}|^\circ$.
\item
Consider any $\Lambda\in\mathcal{F}(|\mathcal{D}|)$.
$|\mathcal{D}\backslash\Lambda|=\Lambda$.
$\mathcal{D}\backslash\Lambda$ is a full subdivision of $\mathcal{F}(\Lambda)$.
$\dim(\mathcal{D}\backslash\Lambda)=\dim\Lambda$.
$(\mathcal{D}\backslash\Lambda)\Mx=(\mathcal{D}\backslash\Lambda)^0=
\{\Delta\cap\Lambda|\Delta\in\mathcal{D}^0, \dim(\Delta\cap\Lambda)=\dim\Lambda\}$.
\item
Consider any $\Delta\in\mathcal{D}$.
Take the unique $\Lambda\in\mathcal{F}(|\mathcal{D}|)$ with $\Delta^\circ\subset\Lambda^\circ$.
Then, $\Delta\in\mathcal{D}\backslash\Lambda$,
$(\mathcal{D}\backslash\Lambda)\Mx/\Delta\neq\emptyset$,
$\Delta=\cap_{\bar{\Delta}\in(\mathcal{D}\backslash\Lambda)\Mx/\Delta}\bar{\Delta}$, and $|\mathcal{D}/\Delta|+\Vect(\Delta)=|\mathcal{D}|+\Vect(\Delta)= |\mathcal{D}|+\Vect(\Lambda)$.
\item
$L\subset|\mathcal{D}|\cap(-|\mathcal{D}|)$.
$\dim L\leq\dim(|\mathcal{D}|\cap(-|\mathcal{D}|))\leq\dim|\mathcal{D}|$.
\item
Assume $\dim|\mathcal{D}|-\dim L\geq 1$.
Consider any $\Delta\in\mathcal{D}^1$.

$\Vect(\Delta)\subset\Vect(|\mathcal{D}|)$, and
$\dim\Vect(|\mathcal{D}|)=\dim\Vect(\Delta)+1$.
Let $H^{\circ\prime}$ and $H^{\circ\prime\prime}$ denote two connected components of $\Vect(|\mathcal{D}|)-\Vect(\Delta)$.
Let $H'=\Clos(H^{\circ\prime})$, and let $H''=\Clos(H^{\circ\prime\prime})$.
$H'\cup H''=\Vect(|\mathcal{D}|)$.
$H'\cap H''=\Vect(\Delta)$.

We consider the case where $\Delta\not\subset\partial|\mathcal{D}|$.
$\sharp(\mathcal{D}^0/\Delta)=2$.
Let $\Delta'$ and $\Delta''$ denote the two elements of $\mathcal{D}^0/\Delta$.
$\{\Delta'+\Vect(\Delta), \Delta''+\Vect(\Delta)\}=\{H', H''\}$, and $\Delta'+\Vect(\Delta)\neq\Delta''+\Vect(\Delta)$.

We consider the case where $\Delta\subset\partial|\mathcal{D}|$. $\sharp(\mathcal{D}^0/\Delta)=1$.
Let $\Delta'$ denote the unique element of $\mathcal{D}^0/\Delta$.
$\Delta'+\Vect(\Delta)=|\mathcal{D}|+\Vect(\Delta)$.
$|\mathcal{D}|+\Vect(\Delta)=H'$ or $|\mathcal{D}|+\Vect(\Delta)=H''$.
\end{enumerate}
\end{lemma}

\begin{lemma}
\label{ris}
Let $J$ be any finite set and let $\mathcal{E}$ be any mapping from $J$ to the set of all finite sets whose elements are convex polyhedral cones in $V$.
For any $j\in J$, $\mathcal{E}(j)$ is a finite set whose elements are convex polyhedral cones in $V$.
We consider the real intersection $\hat{\cap}_{j\in J}\mathcal{E}(j)$ of $\mathcal{E}(j), j\in J$.
By definition
\begin{equation*}
\begin{split}
\bigcap_{j\in J}^\wedge\mathcal{E}(j)
=\{\bar{\Delta}\in 2^V|&\bar{\Delta}=\cap_{j\in J}\Delta(j)\text{ for some mapping }\Delta:J\rightarrow 2^V\\
&\quad\text{such that }\Delta(j)\in\mathcal{E}(j)\text{ for any }j\in J\}\subset 2^V.
\end{split}
\end{equation*}
\begin{enumerate}
\item
$\hat{\cap}_{j\in J}\mathcal{E}(j)$ is a finite set whose elements are convex polyhedral cones in $V$.
\item
$\hat{\cap}_{j\in J}\mathcal{E}(j)$ is a subdivision of $\mathcal{E}(j)$ for any $j\in J$.

Let $\mathcal{D}$ be any finite set whose elements are convex polyhedral cones in $V$.
If $\mathcal{D}$ is a subdivision of $\mathcal{E}(j)$ for any $j\in J$, then $\mathcal{D}$ is a subdivision of $\hat{\cap}_{j\in J}\mathcal{E}(j)$.
\item
$$|\bigcap_{j\in J}^\wedge\mathcal{E}(j)|=\bigcap_{j\in J}|\mathcal{E}(j)|.$$
\item
If $J=\emptyset$, then $\hat{\cap}_{j\in J}\mathcal{E}(j)=\{V\}$.
$\hat{\cap}_{j\in J}\mathcal{E}(j)=\emptyset$, if and only if, $J\neq\emptyset$ and $\mathcal{E}(j)=\emptyset$ for some $j\in J$.
\item
If $\mathcal{E}(j)$ is a convex polyhedral cone decomposition for any $j\in J$, then $\hat{\cap}_{j\in J}\mathcal{E}(j)$ is also a convex polyhedral cone decomposition.

\item
If $\mathcal{E}(j)$ is rational over $N$ for any $j\in J$, then $\hat{\cap}_{j\in J}\mathcal{E}(j)$ is also rational over $N$.
\item
For any subsets $J',J''$ of $J$ with $J'\cup J''=J$ and $J'\cap J''=\emptyset$,
$$\bigcap_{j\in J}^\wedge\mathcal{E}(j)
=(\bigcap_{j\in J'}^\wedge\mathcal{E}(j))\hat{\cap}(\bigcap_{j\in J''}^\wedge\mathcal{E}(j)).$$
\item
For any bijective mapping $\sigma:J\rightarrow J$
$$\bigcap_{j\in J}^\wedge\mathcal{E}(\sigma(j))= \bigcap_{j\in J}^\wedge\mathcal{E}(j).$$
\end{enumerate}
\end{lemma}

\begin{lemma}
\label{several cone decompositions}
Let $m\in \Z_+$ be any positive integer, and let $\mathcal{D}$ be any mapping from the set $\{1,2,\ldots,m\}$ to the set of all convex polyhedral cone decompositions in $V$.
For any $i\in\{1,2,\ldots,m\}$, $\mathcal{D}(i)$ is a convex polyhedral cone decomposition in $V$.

We denote
$$\bar{\mathcal{D}}=\bigcap_{i\in\{1,2,\ldots,m\}}^\wedge\mathcal{D}(i)\subset 2^V.$$

Let $\bar{\Delta}$ be any element of $\bar{\mathcal{D}}$.
\begin{enumerate}
\item
There exists uniquely an element $\Delta(i)\in\mathcal{D}(i)$ with $\bar{\Delta}^\circ\subset \Delta(i)^\circ$ for any $i\in\{1,2,\ldots,m\}$.
\end{enumerate}

Below, we assume that $\Delta(i)\in\mathcal{D}(i)$ and $\bar{\Delta}^\circ\subset \Delta(i)^\circ$ for any $i\in\{1,2,\ldots,m\}$.

\begin{enumerate}
\setcounter{enumi}{1}
\item
$\bar{\Delta}\subset \Delta(i)$ for any $i\in\{1,2,\ldots,m\}$.
\item
$$\bar{\Delta}=\bigcap_{i\in\{1,2,\ldots,m\}}\Delta(i).$$
\item
$$\bar{\Delta}^\circ=\bigcap_{i\in\{1,2,\ldots,m\}}\Delta(i)^\circ.$$
\item
$$\Vect(\bar{\Delta})=\bigcap_{i\in\{1,2,\ldots,m\}}\Vect(\Delta(i)).$$
\item
Consider any subset $\bar{\Lambda}$ of $V$.
$\bar{\Lambda}$ is a face of $\bar{\Delta}$, if and only if, there exists a mapping $\Lambda: \{1,2,\ldots,m\}\rightarrow 2^V$ such that $\bar{\Lambda}=\cap_{i\in \{1,2,\ldots,m\}}\Lambda(i)$ and $\Lambda(i)$ is a face of $\Delta(i)$ for any $i\in\{1,2,\ldots,m\}$.
\end{enumerate}

Let $\Lambda(i)$ be any element of $\mathcal{D}(i)$ for any $i\in\{1,2,\ldots,m\}$.
\begin{enumerate}
\setcounter{enumi}{6}
\item
The intersection $\cap_{i\in\{1,2,\ldots,m\}}\Lambda(i)$ is an element of $\bar{\mathcal{D}}$.
\item
If $\bar{\Delta}\subset \Lambda(i)$ then $\Delta(i)\subset \Lambda(i)$, for any $i\in\{1,2,\ldots,m\}$.
\item
If
$$\bar{\Delta}=\bigcap_{i\in\{1,2,\ldots,m\}}\Lambda(i),$$
then $\Delta(i)\subset \Lambda(i)$ for any $i\in\{1,2,\ldots,m\}$ and the following three conditions are equivalent:
\begin{enumerate}
\item
$$\bigcap_{i\in\{1,2,\ldots,m\}}\Lambda(i)^\circ\neq\emptyset.$$
\item
$\Delta(i)=\Lambda(i)$ for any $i\in\{1,2,\ldots,m\}$.
\item
$$\bar{\Delta}^\circ=\bigcap_{i\in\{1,2,\ldots,m\}}\Lambda(i)^\circ$$
\end{enumerate}
\item
$\bar{\mathcal{D}}$ is a convex polyhedral cone decomposition in $V$.
$|\bar{\mathcal{D}}|=\cap_{i\in\{1,2,\ldots,m\}}|\mathcal{D}(i)|$.
If $\mathcal{D}(i)$ is rational over $N$ for any $i\in\{1,2,\ldots,m\}$, then $\bar{\mathcal{D}}$ is rational over $N$.
\item
$\bar{\mathcal{D}}$ is a subdivision of $\mathcal{D}(i)$ for any $i\in\{1,2,\ldots,m\}$.

Let $\bar{\mathcal{E}}$ be any convex polyhedral cone decomposition in $V$.
If $\bar{\mathcal{E}}$ is a subdivision of $\mathcal{D}(i)$ for any $i\in\{1,2,\ldots,m\}$, then $\bar{\mathcal{E}}$ is a subdivision of $\bar{\mathcal{D}}$.
\end{enumerate}
\end{lemma}

\begin{lemma}
\label{pbcpcd}
Let $\mathcal{D}$ be any convex polyhedral cone decomposition in $V$; let $U$ be any finite dimensional vector space over $\R$; and let $\nu:U\rightarrow V$ be any homomorphism of vector spaces over $\R$.

The pull back $\nu^*\mathcal{D}$ of $\mathcal{D}$ by $\nu$ is a convex polyhedral cone decomposition in $U$.
\end{lemma}

\section{Convex pseudo polyhedrons}
\label{cpp}

We study convex pseudo polyhedrons.

Let $V$ be any finite dimensional vector space over $\R$, and let $N$ be any lattice of $V$.

\begin{lemma}
\label{cpp1}
Let $S$ be any convex pseudo polyhedron in $V$, and let $X$ and $Y$ be any finite subset of $V$ satisfying $S=\Conv(X)+\Convcone(Y)$ and $X\neq\emptyset$.
\begin{enumerate}
\item
$\Stab(S)=\Convcone(Y)$.
$\Stab(S)$ is a convex polyhedral cone in $V$.
If $S$ is rational over $N$, then $\Stab(S)$ is also rational over $N$.
\item
$\Vect(\Stab(S))\subset\Stab(\Affi(S))$.
\item
The following four conditions are equivalent:
\begin{enumerate}
\item
$S$ is a convex polyhedron.
\item
$S=\Conv(X)$.
\item
$\Stab(S)=\{0\}$.
\item
$S$ is compact.
\end{enumerate}
\item
The following three conditions are equivalent:
\begin{enumerate}
\item
$S$ is a convex polyhedral cone.
\item
$\Conv(X)\cap\Stab(S)\cap(-\Stab(S))\neq\emptyset$ and $\Conv(X)\subset\Stab(S)$.
\item
$S=\Stab(S)$.
\end{enumerate}
\item
We consider the dual vector space $V^*$ of $V$ and the dual cone $\Stab(S)^\vee=\Stab(S)^\vee|V\subset V^*$ of $\Stab(S)$.

For any $\omega\in V^*$, the following three conditions are equivalent:
\begin{enumerate}
\item
$\omega\in\Stab(S)^\vee$.
\item
There exists the minimum element $\min\{\langle\omega, x\rangle|x\in S\}$ of the subset $\{\langle\omega, x\rangle|$\break$x\in S\}$ of $\R$.
\item
The subset $\{\langle\omega, x\rangle|x\in S\}$ of $\R$ is bounded below.
\end{enumerate}
\end{enumerate}
\end{lemma}

\begin{definition}
\label{faces of cpp}
Let $S$ be any convex pseudo polyhedron in $V$.
We consider the dual cone $\Stab(S)^\vee=\Stab(S)^\vee|V\subset V^*$ of $\Stab(S)$.
\begin{enumerate}
\item
For any $\omega\in \Stab(S)^\vee$, we denote
\begin{equation*}
\begin{split}
\Ord(\omega,S|V)=&\min\{\langle\omega,x\rangle|x\in S\}\in\R\\
\Delta(\omega,S|V)=&\{x\in S|\langle\omega, x\rangle=\Ord(\omega,S|V)\}\subset S.
\end{split}
\end{equation*}

When we need not refer to $V$ or to the pair $(S, V)$, we also write simply  $\Ord(\omega, S)$ or $\Ord(\omega)$, $\Delta(\omega, S)$ or $\Delta(\omega)$ respectively, instead of $\Ord(\omega,S|V)$, $\Delta(\omega, S|V)$.
\item
Let $F$ be any subset of $S$.
We say that $F$ is a \emph{face} of $S$, if $F=\Delta(\omega, S|V)$ for some $\omega\in \Stab(S)^\vee.$

It is easy to see that any face $F$ of $S$ is a closed convex subset of $V$, and the dimension $\dim F\in\Z_0$ of $F$, the boundary $\partial F$ of $F$, and the interior $F^\circ$ of $F$ are defined.

Any face $F$ of $S$ with $\dim F=0$ is called a \emph{vertex} of $S$.
Any vertex of $S$ is a subset of $S$ with only one element.
Any face $F$ of $S$ with $\dim F=1$ is called an \emph{edge} of $S$.
\item
By $\mathcal{F}(S)$ we denote the set of all faces of $S$.

For any $i\in\Z$, the set of all faces $F$ with $\dim F=i$ is denoted by $\mathcal{F} (S)_i$, and the set of all faces $F$ with $\dim F=\dim S-i$ is denoted by $\mathcal{F}(S)^i$.
\item
Let $\ell=\dim(\Stab(S)\cap(-\Stab(S)))\in\Z_0$.
We denote
\begin{equation*}
\begin{split}
c(S)=&\sharp\mathcal{F}(S)_\ell\in\Z_0,\\
\mathcal{V}(S)=&\bigcup_{F\in\mathcal{F}(S)_\ell}F\subset S,
\end{split}
\end{equation*}
we call $c(S)$ the \emph{characteristic number} of $S$, and we call $\mathcal{V}(S)$ the \emph{skeleton} of $S$.
We call any face $F$ of $S$ with $\dim F=\ell$ a \emph{minimal} face of $S$.
\item
Let $F$ be any face of $S$.
We denote
\begin{equation*}
\begin{split}
\Delta^\circ(F,S|V)=&\{\omega\in \Stab(S)^\vee|F=\Delta(\omega, S|V)\}\subset \Stab(S)^\vee\subset V^*,\\
\Delta(F,S|V)=&\{\omega\in \Stab(S)^\vee|F\subset\Delta(\omega, S|V)\}\subset \Stab(S)^\vee\subset V^*.
\end{split}
\end{equation*}

We call $\Delta^\circ(F,S|V)$ the \emph{open face cone} of $F$, and we call $\Delta (F,S|V)$ the \emph{face cone} of $F$.

When we need not refer to $V$ or to the pair $(S, V)$, we also write simply  $\Delta^\circ (F, S)$ or $\Delta^\circ(F)$, $\Delta(F, S)$ or $\Delta(F)$ respectively, instead of $\Delta^\circ(F, S|V)$, $\Delta(F, S|V)$.
\item
We denote
$$\mathcal{D}(S|V)=\{\Delta(F, S|V)|F\in\mathcal{F}(S)\}\subset 2^{\Stab(S)^\vee}\subset 2^{V^*},$$
and we call $\mathcal{D}(S|V)$ the \emph{face cone decomposition} of $S$.

When we need not refer to $V$, we also write simply $\mathcal{D}(S)$, instead of $\mathcal{D}(S|V)$.
\end{enumerate}
\end{definition}

Let $W$ be any finite dimensional vector space over $\R$ containing $V$ as a vector subspace over $\R$ with $\dim W=\dim V+1$, and let $z\in W- V$ be any point.

Let $\pi:V\rightarrow W$ denote the inclusion homomorphism.
Putting $\pi'(t)=tz\in W$ for any $t\in\R$ we define an injective homomorphism $\pi':\R\rightarrow W$ of vector spaces over $\R$.

For any $a\in W$, choosing the unique pair $b\in V$ and $t\in\R$ with $a=b+tz$ and putting $\rho(a)=b$ and $\rho'(a)=t$ we define homomorphisms $\rho:W\rightarrow V$ and $\rho':W\rightarrow \R$ of vector spaces over $\R$.

Putting $\iota(\omega)=\omega(1)\in\R$ for any $\omega\in \R^*=\mathrm{Hom}_{\R}(\R,\R)$, we define an isomorphism $\iota:\R^*\rightarrow \R$ of vector spaces over $\R$.
For any $\omega\in\R^*$ and any $t\in\R$ we have $\langle\omega, t\rangle=\iota(\omega)t$.
Below, using this isomorphism $\iota$ we identify $\R^*$ with $\R$.
For any $t\in\R=\R^*$ and any $u\in\R$ we have $\langle t,u\rangle=tu$.

We have eight homomorphisms of vector spaces over $\R$.
\begin{gather*}
\begin{matrix}
\pi:&V\rightarrow W,&\quad& \pi'&:\R\rightarrow W,\\
\rho:&W\rightarrow V,&\quad& \rho'&:W\rightarrow\R,\\
\pi^*:&W^*\rightarrow V^*,&\quad& \pi^{\prime*}&:W^*\rightarrow \R,\\
\rho^*:&V^*\rightarrow W^*,&\quad& \rho^{\prime*}&:\R\rightarrow W^*.
\end{matrix}
\end{gather*}
Four homomorphisms $\pi, \pi', \rho^*, \rho^{\prime*}$ are injective.
The other four $\rho, \rho', \pi^*, \pi^{\prime*}$ are surjective.
We denote
$H=V+\R_0z\subset W$ and $\zeta=\rho^{\prime*}(1)\in W^*$.

\begin{lemma}
\label{eight mappings}
\begin{enumerate}
\item
$\rho\pi=\Id_V$, $\rho'\pi'=\Id_{\R}$, $\pi\rho+\pi'\rho'=\Id_W$,
$V=\pi(V)= \rho^{\prime-1}(0)$, $\R z=\pi'(\R)=\rho^{-1}(0)$, $\pi'(1)=z$.
For any $x\in W$, $\rho'(x)=\langle\zeta,x\rangle$.
\item
$\pi^*\rho^*=\Id_{V^*}$, $\pi^{\prime*}\rho^{\prime*}=\Id_{\R}$, $\rho^*\pi^*+\rho^{\prime*}\pi^{\prime*}=\Id_{W^*}$,
$(\R z)^\vee =\rho^*(V^*)= \pi^{\prime*-1}(0)$, $V^\vee =\R\zeta=\rho^{\prime*}(\R)=\pi^{*-1}(0)$.
For any $\xi\in W^*$, $\pi^{\prime*}(\xi)=\langle\xi,z\rangle$. $\langle\zeta,z\rangle=1$.
\item
$N+\Z x$ is a lattice of $W$. $H$ is a rational convex polyhedral cone over $N+\Z x$ in $W$. 
$\dim H=\dim W=\dim V+1$.
$H=\{x\in W|\rho'(x)\geq 0\}$.
$W=H\cup(-H)=\Vect(H)=\Vect(-H)$.
$V=H\cap(-H)=\partial H=\partial(-H)$.
$z\in H^\circ=V+\R_+z=\{x\in W|\rho'(x)> 0\}$.
\item
$(N+\Z z)^*=\rho^*(N^*)+\Z\zeta$.
$H^\vee$ is a simplicial cone over $(N+\Z z)^*$ in $W^*$ with $\dim H^\vee=1$.
$H^\vee\cap(N+\Z z)^*=\Z_0\zeta$.
$H^\vee=\R_0\zeta$.
$V^\vee =H^\vee\cup(-H^\vee)=\Vect(H^\vee) =\Vect(-H^\vee)$.
$\{0\}=H^\vee\cap(-H^\vee)=\partial H^\vee =\partial(-H^\vee)$.
$-H^\vee=(-H)^\vee$.
\end{enumerate}
\end{lemma}

Putting $\sigma(a)=\rho(a)/\rho'(a)\in V$ for any $a\in H^\circ$, we define a mapping $\sigma:H^\circ\rightarrow V$.
If $a\in H^\circ$, $b\in V$, $t\in\R$ and $a=b+tz$, then $t>0$ and $\sigma(a)=b/t$.

\begin{lemma}
\label{first correspondence}
\begin{enumerate}
\item
$\sigma$ is surjective.
For any $a\in H^\circ$, $\{\sigma(a)+z\}=(\R_+a)\cap(V+\{z\})=(\R a)\cap(V+\{z\})$ and
$\sigma^{-1}(\sigma(a))=\R_+a$.
For any $b\in V$, $\sigma^{-1}(b)=\R_+(b+z)$.
\item
Consider any convex polyhedral cone $\Delta$ in $W$ satisfying $\Delta\subset H$ and $\Delta\cap H^\circ\neq\emptyset$ and any finite subset $Z$ of $H$ satisfying $\Delta=\Convcone(Z)$.
\begin{enumerate}
\item
$\sigma(\Delta\cap H^\circ)=\Conv(\sigma(Z\cap H^\circ))+\Convcone(Z\cap V)$.

$\sigma(\Delta\cap H^\circ)$ is a convex pseudo polyhedral cone in $V$.

If $\Delta$ is rational over $N+\Z z$, then $\sigma(\Delta\cap H^\circ)$ is rational over $N$.

$\dim\Delta=\dim \sigma(\Delta\cap H^\circ)+1$.
$\Delta\cap V=\Stab(\sigma(\Delta\cap H^\circ))$.
\item
$\Delta\cap H^\circ=\sigma^{-1}(\sigma(\Delta\cap H^\circ))$.
$\Delta=\Clos(\sigma^{-1}(\sigma(\Delta\cap H^\circ)))$.
\end{enumerate}
\item
Consider any convex pseudo polyhedron $S$ in $V$ and any finite subsets $X$, $Y$ of $V$ satisfying $S=\Conv(X)+\Convcone(Y)$ and $X\neq\emptyset$.
\begin{enumerate}
\item
$\Clos(\sigma^{-1}(S))=\Convcone((X+\{z\})\cup Y)$.

$\Clos(\sigma^{-1}(S))$ is a convex polyhedral cone in $W$.
$\Clos(\sigma^{-1}(S))\subset H$.
$\Clos(\sigma^{-1}(S))\cap H^\circ\neq\emptyset$.
$\Clos(\sigma^{-1}(S))\cap V=\Stab(S)$.

If $S$ is rational over $N$, then $\Clos(\sigma^{-1}(S))$ is rational over $N+\Z z$.
\item
$\Clos(\sigma^{-1}(S))\cap H^\circ=\sigma^{-1}(S)$.
$\sigma(\Clos(\sigma^{-1}(S))\cap H^\circ)=S$.
\end{enumerate}
\item
For any subsets $\Delta$, $\Lambda$ of $H$ satisfying $\Delta\cap H^\circ\neq\emptyset$, $\Lambda\cap H^\circ\neq\emptyset$ and $\Delta\subset\Lambda$, $\sigma(\Delta\cap H^\circ)\subset\sigma(\Lambda\cap H^\circ)$.

For any non-empty subsets $S$, $T$ of $V$ satisfying $S\subset T$, $\Clos(\sigma^{-1}(S))\subset\Clos(\sigma^{-1}(T))$.

For any non-empty closed subsets $S$, $T$ of $V$, $S\cap T=\emptyset$, if and only if, $\Clos(\sigma^{-1}(S))\cap\Clos(\sigma^{-1}(T))\cap H^\circ=\emptyset$.
\item
For any convex polyhedral cone $\Delta$ in $W$ satisfying $\Delta\subset H$ and $\Delta\cap H^\circ\neq\emptyset$, $\sigma(\Delta\cap H^\circ)$ is a convex pseudo polyhedron in $V$.

For any convex pseudo polyhedron $S$ in $V$, $\Clos(\sigma^{-1}(S))$ is a convex polyhedral cone in $W$, $\Clos(\sigma^{-1}(S))\subset H$ and $\Clos(\sigma^{-1}(S))\cap H^\circ\neq\emptyset$.

The mapping sending any convex polyhedral cone $\Delta$ in $W$ satisfying $\Delta\subset H$ and $\Delta\cap H^\circ\neq\emptyset$ to $\sigma(\Delta\cap H^\circ)$ and the mapping sending any convex pseudo polyhedron $S$ in $V$ to $\Clos(\sigma^{-1}(S))$ are bijective mappings preserving the inclusion relation between the set of all convex plyhedral cones $\Delta$ in $W$ satisfying $\Delta\subset H$ and $\Delta\cap H^\circ\neq\emptyset$ and the set of all convex pseudo polyhedrons in $V$, and they are the inverse mappings of each other.

Furthermore, if a convex polyhedral cone $\Delta$ in $W$ satisfying $\Delta\subset H$ and $\Delta\cap H^\circ\neq\emptyset$ and a convex pseudo polyhedron $S$ in $V$ correspond to each other by them, then $\Delta=\Clos(\sigma^{-1}(S))$, $S=\sigma(\Delta\cap H^\circ)$ and the following claims holds:

$\Delta$ is rational over $N+\Z z$, if and only if, $S$ is rational over $N$.
\begin{equation*}
\begin{split}
\dim\Delta&\geq 1.\\
\dim\Delta&=\dim S+1.\\
\Delta\cap V&=\Stab(S).\\
\Delta\cap H^\circ&=\sigma^{-1}(S).\\
\Delta\cap(V+\{z\})&= S+\{z\}.\\
\Vect(\Delta)\cap(V+\{z\})&=\Affi(S)+\{z\}.\\
\Vect(\Delta)\cap H^\circ&=\sigma^{-1}(\Affi(S)).\\
\sigma(\Vect(\Delta)\cap H^\circ)&= \Affi(S).\\
\Vect(\Delta)&=\Vect(\Affi(S)+\{z\}).\\
\partial\Delta\cap H^\circ&=\sigma^{-1}(\partial S).\\
\Delta^\circ&=\sigma^{-1}(S^\circ).\\
\sigma(\Delta^\circ)&=S^\circ.
\end{split}
\end{equation*}
\end{enumerate}
\end{lemma}

\begin{cor}
\label{intersection}
Consider any convex pseudo polyhedrons $S$, $T$ in $V$.

If $S\cap T\neq\emptyset$, then $S\cap T$ is a convex pseudo polyhedron in $V$ and $\Stab(S\cap T)=\Stab(S)\cap\Stab(T)$.

If $S$ and $T$ are rational over $N$ and $S\cap T\neq\emptyset$, then $S\cap T$ is also rational over $N$.

If $S$ and $T$ are convex polyhedrons and $S\cap T\neq\emptyset$, then $S\cap T$ is also a convex polyhedron.

If $S$ and $T$ are convex polyhedral cones, then $S\cap T$ is also a convex polyhedral cone.

\end{cor}

\begin{prop}
\label{miracle}
Let $\Delta$ be any convex polyhedral cone $\Delta$ in $W$ satisfying $\Delta\subset H$ and $\Delta\cap H^\circ\neq\emptyset$.
We denote
\begin{equation*}
\begin{split}
L&=\Delta\cap(-\Delta)\subset\Delta,\\
\ell&=\dim L\in\Z_0,\\
S&=\sigma(\Delta\cap H^\circ)\subset V,\\
H^\vee&=H^\vee|W\subset W^*,\\
\Delta^\vee&=\Delta^\vee|W\subset W^*,\\
\partial_-\Delta^\vee&=\{\omega\in\Delta^\vee|
(\{\omega\}+\Vect(H^\vee))\cap\Delta^\vee\subset\{\omega\}+H^\vee\}\subset \Delta^\vee,\\
\mathcal{F}(\Delta)_*&=\{\Lambda\in\mathcal{F}(\Delta)|\Lambda\cap H^\circ\neq\emptyset\}\subset\mathcal{F}(\Delta),\\
\mathcal{F}(\Delta^\vee)^*&=\mathcal{F}(\Delta^\vee)\backslash\partial_-\Delta^\vee\subset\mathcal{F}(\Delta^\vee).
\end{split}
\end{equation*}
\begin{enumerate}
\item
$L=(\Delta\cap V)\cap(-(\Delta\cap V))=\Stab(S)\cap(-\Stab(S))\subset V$.
$L\in\mathcal{F}(\Delta)-\mathcal{F}(\Delta)_*$.
$\Delta\in\mathcal{F}(\Delta)_*\neq\emptyset$.
If $\Lambda\in\mathcal{F}(\Delta)_*$, $\Gamma\in\mathcal{F}(\Delta)$ and $\Lambda\subset\Gamma$, then $\Gamma\in\mathcal{F}(\Delta)_*$.
For any $\Lambda\in\mathcal{F}(\Delta)_*$ there exists $\Gamma\in\mathcal{F}(\Delta)_*$ such that $\Lambda\supset\Gamma$ and $\dim\Gamma=\ell+1$.
\item
$H^\vee\subset\Delta^\vee$.
$-H^\vee\not\subset\Delta^\vee$.
$\Delta^\vee+\pi^{*-1}(0)=(\Delta\cap V)^\vee|W=\pi^{*-1}((\Delta\cap V)^\vee|V)$.
$\pi^*(\Delta^\vee)=(\Delta\cap V)^\vee|V=\Stab(S)^\vee|V$.
\item
$\emptyset\neq\partial_-\Delta^\vee\subset\Delta^\vee=\partial_-\Delta^\vee+H^\vee$.
$\pi^*(\partial_-\Delta^\vee)=\Stab(S)^\vee|V$.
The mapping $\pi^*:\partial_-\Delta^\vee\rightarrow\Stab(S)^\vee|V$ induced by $\pi^*$ is bijective.
\item
For any face $\Lambda$ of $\Delta$ the following three conditions are equivalent:
\begin{enumerate}
\item
$\Lambda\in\mathcal{F}(\Delta)_*$.
\item
$\Delta(\Lambda,\Delta|W)\in\mathcal{F}(\Delta^\vee)^*$.
\item
$\Delta^\circ(\Lambda,\Delta|W)\cap\partial_-\Delta^\vee\neq\emptyset$.
\end{enumerate}
\item
$|\mathcal{F}(\Delta^\vee)^*|=\partial_-\Delta^\vee$.
$\Delta^\vee\cap(-\Delta^\vee)\in\mathcal{F}(\Delta^\vee)^*\neq\emptyset$.
$\dim\mathcal{F}(\Delta^\vee)^* =\dim\Delta^\vee-1=\dim\Stab(S)^\vee=\dim V-\ell$.
$(\mathcal{F}(\Delta^\vee)^*)\Mx=(\mathcal{F}(\Delta^\vee)^*)^0$.
\item
For any $\omega\in \partial_-\Delta^\vee$ the following claims hold:
\begin{enumerate}
\item
$\Delta(\omega,\Delta|W)\in\mathcal{F}(\Delta)_*$.
\item
$\Ord(\pi^*(\omega),S|V)=-\langle\omega,z\rangle$.
\item
$\Delta(\omega,\Delta|W)\cap H^\circ=\sigma^{-1}(\Delta(\pi^*(\omega),S|V))$.
\item
$\sigma(\Delta(\omega,\Delta|W)\cap H^\circ)=\Delta(\pi^*(\omega),S|V)$.
\end{enumerate}
\item
For any $\Lambda\in\mathcal{F}(\Delta)_*$, $\sigma(\Lambda\cap H^\circ)\in\mathcal{F}(S)$.

For any $F\in\mathcal{F}(S)$, $\Clos(\sigma^{-1}(F))\in \mathcal{F}(\Delta)_*$.

The mapping from $\mathcal{F}(\Delta)_*$ to $\mathcal{F}(S)$ sending $\Lambda\in\mathcal{F}(\Delta)_*$ to $\sigma(\Lambda\cap H^\circ)\in\mathcal{F}(S)$ and the mapping from $\mathcal{F}(S)$ to $\mathcal{F}(\Delta)_*$ sending $F\in\mathcal{F}(S)$ to $\Clos(\sigma^{-1}(F))\in \mathcal{F}(\Delta)_*$ are bijective mappings preserving the inclusion relation between $\mathcal{F}(\Delta)_*$ and $\mathcal{F}(S)$, and they are the inverse mappings of each other.
\item
Assume $\Lambda\in\mathcal{F}(\Delta)_*$, $F\in\mathcal{F}(S)$, $F=\sigma(\Lambda\cap H^\circ)$ and $\Lambda=\Clos(\sigma^{-1}(F))$.
The following claims hold:
\begin{enumerate}
\item
$F$ is a convex pseudo polyhedron in $V$.
If $S$ is rational over $N$, then $F$ is also rational over $N$.
$\Stab(F)$ is a face of $\Stab(S)$.
\item
$\dim\Lambda=\dim F+1$.
\item
$\Delta^\circ(\Lambda,\Delta|W)\subset\Delta(\Lambda,\Delta|W)\subset\partial_-\Delta^\vee$.
\item
$\Delta(F,S|V)=\pi^*(\Delta(\Lambda,\Delta|W))$.
$\pi^{*-1}(\Delta(F,S|V))\cap\partial_-\Delta^\vee= \Delta(\Lambda,\Delta|W)$.
$\Delta(F,S|V)$ is a convex polyhedral cone in $V^*$.
\item
$\Delta^\circ(F,S|V)=\pi^*(\Delta^\circ(\Lambda,\Delta|W))$.
$\pi^{*-1}(\Delta^\circ(F,S|V))\cap\partial_-\Delta^\vee=\Delta^\circ(\Lambda,\Delta|W)$.
$\Delta^\circ(F,S|V)=\Delta(F,S|V) ^\circ$.
\item
$\Vect(\Delta(\Lambda,\Delta|W))\cap\pi^{*-1}(0)=\{0\}$.
$\Vect(\Delta(F,S|V))=$\hfill\break$\pi^*(\Vect(\Delta(\Lambda,\Delta|W)))=\Stab(\Affi(F))^\vee|V$.
\end{enumerate}
\item
For any $\Lambda^*\in\mathcal{F}(\Delta^\vee)^*$, $\pi^*(\Lambda^*)\in\mathcal{D}(S|V)$.

For any $\bar{\Lambda}^*\in\mathcal{D}(S|V)$, $\pi^{*-1}(\bar{\Lambda}^*)\cap \partial_-\Delta^\vee\in \mathcal{F}(\Delta^\vee)^*$.

The mapping from $\mathcal{F}(\Delta^\vee)^*$ to $\mathcal{D}(S|V)$ sending $\Lambda^*\in\mathcal{F}(\Delta^\vee)^*$ to $\pi^*(\Lambda^*)\in\mathcal{D}(S|V)$ and the mapping from $\mathcal{D}(S|V)$ to $\mathcal{F}(\Delta^\vee)^*$ sending $\bar{\Lambda}^*\in\mathcal{D}(S|V)$ to $\pi^{*-1}(\bar{\Lambda}^*)\cap \partial_-\Delta^\vee\in \mathcal{F}(\Delta^\vee)^*$ are bijective mappings preserving the dimension and the inclusion relation between $\mathcal{F}(\Delta^\vee)^*$ and $\mathcal{D}(S|V)$, and they are the inverse mappings of each other.
\item
The face cone decomposition $\mathcal{D}(S|V)$ of $S$ in $V$ is a convex polyhedral cone decomposition in $V^*$.
$|\mathcal{D}(S|V)|=\Stab(S)^\vee|V$.
If $S$ is rational over $N$, then $\mathcal{D}(S|V)$ is rational over $N^*$.
\item
For any $F\in\mathcal{F}(S)$, $\Vect(\Delta(F,S|V))=\Stab(\Affi(F))^\vee|V$ and $\dim F+\dim \Delta(F,$\hfill\break$S|V)=\dim V$.

For any $F\in\mathcal{F}(S)$ and $G\in\mathcal{F}(S)$, $F\subset G$, if and only if, $\Delta(F,S|V)\supset\Delta(G,S|V)$.

The mapping from $\mathcal{F}(S)$ to $\mathcal{D}(S|V)$ sending $F\in\mathcal{F}(S)$ to $\Delta(F,S|V)\in\mathcal{D}(S|V)$ is a bijective mapping.
\item
The function $\Ord(\;,S|V):\Stab(S)^\vee|V\rightarrow\R$ sending $\bar{\omega}\in\Stab(S)^\vee|V$ to $\Ord(\bar{\omega},S|V)\in\R$ is a piecewise linear convex function over $\mathcal{D}(S|V)$.

If $S$ is rational over $N$, then this function $\Ord(\;,S|V)$ is rational over $N^*$.
\item
Denote
\begin{equation*}\begin{split}
\Sigma(\mathcal{D}(S|V), \Ord(\;,S|V))=
\{\xi \in W^*|&\xi=\rho(\bar{\omega})+t\zeta\text{ for some }\bar{\omega}\in |\mathcal{D}(S|V)|\\
&\quad\text{and some }t\in\R
\text{ with }t\geq -\Ord(\bar{\omega},S|V)\}.
\end{split}\end{equation*}
Then, $\Sigma(\mathcal{D}(S|V), \Ord(\;,S|V))=\Delta^\vee$,
$\Sigma(\mathcal{D}(S|V), \Ord(\;,S|V))^\vee|W^*=\Delta$, and
$\sigma((\Sigma(\mathcal{D}(S|V), \Ord(\;,S|V))^\vee|W^*)\cap H^\circ)=S$.
\end{enumerate}
\end{prop}

\begin{proof}
We give only the proof of 4.

Since $L\subset V=\{x\in W|\langle\zeta,x\rangle=0\}$, $\langle\zeta,b\rangle=0$ for any $b\in L$.

Consider any face $\Gamma$ of $\Delta$ with $\dim\Gamma=\ell+1$.
We take any point $e_\Gamma\in\Gamma- L$.
We have $\Gamma=\R_0e_\Gamma+L$.
Since $e_\Gamma\in\Gamma\subset\Delta\subset H=\{x\in W|\langle\zeta,x\rangle\geq 0\}$, $\langle\zeta,e_\Gamma\rangle\geq 0$.

We know $\Delta=\sum_{\Gamma\in\mathcal{F}(\Delta)_{\ell+1}}\R_0e_\Gamma+L$.
Take any point $a\in\Delta\cap H^\circ\neq\emptyset$.
$a\in\Delta=\sum_{\Gamma\in\mathcal{F}(\Delta)_{\ell+1}}\R_0e_\Gamma+L$.
Take any function $t: \mathcal{F}(\Delta)_{\ell+1}\rightarrow\R_0$ and $b\in L$ with $a=\sum_{\Gamma\in\mathcal{F}(\Delta)_{\ell+1}}t(\Gamma) e_\Gamma+b$.
Since $a\in H^\circ=\{x\in W|\langle\zeta,x\rangle>0\}$,
$0<\langle\zeta,a\rangle=\sum_{\Gamma\in\mathcal{F}(\Delta)_{\ell+1}}t(\Gamma)\langle\zeta, e_\Gamma\rangle+\langle\zeta,b\rangle=\sum_{\Gamma\in\mathcal{F}(\Delta)_{\ell+1}}t(\Gamma)\langle\zeta, e_\Gamma\rangle$.
We know that there exists $\Gamma\in\mathcal{F}(\Delta)_{\ell+1}$ with $\langle\zeta, e_\Gamma\rangle>0$.

Consider any face $\Lambda$ of $\Delta$.
We know $\Lambda=\sum_{\Gamma\in\mathcal{F}(\Delta)_{\ell+1}\backslash\Lambda}\R_0e_\Gamma+L$.

\noindent $(a)\Rightarrow (b)$.
Assume $(a)$. $\Lambda\cap H^\circ\neq\emptyset$. Take any point $a\in\Lambda\cap H^\circ$. Since $a\in H^\circ=\{x\in W|\langle\zeta,x\rangle>0\}$, $\langle\zeta,a\rangle>0$.
$a\in\Lambda=\sum_{\Gamma\in\mathcal{F}(\Delta)_{\ell+1}\backslash\Lambda}\R_0e_\Gamma+L$. Take any function $t: \mathcal{F}(\Delta)_{\ell+1}\backslash\Lambda\rightarrow\R_0$ and $b\in L$ with $a=\sum_{\Gamma\in\mathcal{F}(\Delta)_{\ell+1}\backslash\Lambda}t(\Gamma)e_\Gamma+b$.
$0 <\langle\zeta,a\rangle=\sum_{\Gamma\in\mathcal{F}(\Delta)_{\ell+1}\backslash\Lambda}t(\Gamma)\langle\zeta, e_\Gamma\rangle+\langle\zeta,b\rangle=\sum_{\Gamma\in\mathcal{F}(\Delta)_{\ell+1}\backslash\Lambda}t(\Gamma)\langle\zeta, e_\Gamma\rangle$.
We know that there exists $\Gamma\in\mathcal{F}(\Delta)_{\ell+1}\backslash\Lambda$ with $\langle\zeta,e_\Gamma\rangle>0$.
We take any $\Gamma\in\mathcal{F}(\Delta)_{\ell+1}\backslash\Lambda$ with $\langle\zeta,e_\Gamma\rangle>0$.

Note that $H^\vee=\R_0\zeta$ and $\Vect(H)=\R\zeta$.

Consider any $\omega\in\Delta(\Lambda,\Delta|W)$.
$e_\Gamma\in\Gamma\subset\Lambda\subset\Delta(\omega,\Delta|W)$.
$\langle\omega, e_\Gamma\rangle=0$.

Consider any $\chi\in(\{\omega\}+\Vect(H))\cap\Delta^\vee$.
Take $t\in\R$ with $\chi=\omega+t\zeta$.
Since $\chi\in\Delta^\vee$ and $e_\Gamma\in\Gamma\subset\Lambda\subset\Delta$, we have $0\leq\langle\chi, e_\Gamma\rangle=\langle\omega+t\zeta, e_\Gamma\rangle=\langle\omega, e_\Gamma\rangle+t\langle\zeta, e_\Gamma\rangle= t\langle\zeta, e_\Gamma\rangle$.
Since $\langle\zeta, e_\Gamma\rangle>0$, we know $t\geq 0$ and $\chi=\omega+t\zeta\in\{\omega\}+\R_0\zeta=\{\omega\}+H$.

We know $(\{\omega\}+\Vect(H))\cap\Delta^\vee\subset\{\omega\}+H$ and $\omega\in\partial_-\Delta^\vee$.

We know $\Delta(\Lambda,\Delta|W)\subset\partial_-\Delta^\vee$ and $\Delta(\Lambda,\Delta|W)\in\mathcal{F}(\Delta^\vee)^*$.

\noindent $(c)\Rightarrow (a)$.
Assume $(a)$ does not hold.
$\Lambda\cap H^\circ=\emptyset$.
Since $\Lambda\subset\Delta\subset H$, $\Lambda\subset H- H^\circ=\partial H=V=\{x\in W|\langle\zeta,a\rangle=0\}$, and $\langle\zeta,a\rangle=0$ for any $a\in\Lambda$.
In particular, $\langle\zeta,e_\Gamma\rangle=0$ for any $\Gamma\in\mathcal{F}(\Delta)_{\ell+1}\backslash\Lambda$.
If $\Gamma\in\mathcal{F}(\Delta)_{\ell+1}$ and $\langle\zeta,e_\Gamma\rangle>0$, then $\Gamma\not\subset\Lambda$.

Consider any $\omega\in\Delta^\circ(\Lambda,\Delta|W)$.
$\omega\in\Delta^\vee$.
Since $L\subset\Lambda=\Delta(\omega,\Delta|W)$, $\langle\omega, b\rangle=0$ for any $b\in L$.
Consider any $\Gamma\in\mathcal{F}(\Delta)_{\ell+1}$.
If $\Gamma\subset\Lambda$, then
$e_\Gamma\in\Gamma\subset\Lambda=\Delta(\omega,\Delta|W)$ and $\langle\omega, e_\Gamma\rangle=0$.
It is easy to see that if $\Gamma\not\subset\Lambda$, then $\langle\omega, e_\Gamma\rangle>0$

Consider any $t\in\R$.
For any $b\in L$ we have $\langle\omega+t\zeta, b\rangle=\langle\omega, b\rangle+t\langle\zeta, b\rangle=0+t0=0$.
For any $\Gamma\in\mathcal{F}(\Delta)_{\ell+1}$, $\langle\omega+t\zeta, e_\Gamma\rangle=\langle\omega, e_\Gamma\rangle+t\langle\zeta, e_\Gamma\rangle$.
$\omega+t\zeta\in\Delta^\vee$, if and only if, $\langle\omega, e_\Gamma\rangle+t\langle\zeta, e_\Gamma\rangle\geq 0$ for any $\Gamma\in\mathcal{F}(\Delta)_{\ell+1}$.

Consider any $\Gamma\in\mathcal{F}(\Delta)_{\ell+1}$.
If $\Gamma\subset\Lambda$, then
$\langle\omega, e_\Gamma\rangle+t\langle\zeta, e_\Gamma\rangle=0+t0=0$.
If $\Gamma\not\subset\Lambda$ and $\langle\zeta, e_\Gamma\rangle=0$, then
$\langle\omega, e_\Gamma\rangle+t\langle\zeta, e_\Gamma\rangle=\langle\omega, e_\Gamma\rangle>0$.
We consider the case $\Gamma\not\subset\Lambda$ and $\langle\zeta, e_\Gamma\rangle>0$.
We have $\langle\omega, e_\Gamma\rangle>0$, $-\langle\omega, e_\Gamma\rangle/\langle\zeta, e_\Gamma\rangle<0$ and $\langle\omega, e_\Gamma\rangle+t\langle\zeta, e_\Gamma\rangle\geq0$, if and only if, $t\geq-\langle\omega, e_\Gamma\rangle/\langle\zeta, e_\Gamma\rangle$.
Put
$$t_0=\max\{-\frac{\langle\omega, e_\Gamma\rangle}{\langle\zeta, e_\Gamma\rangle}|\Gamma\in\mathcal{F}(\Delta)_{\ell+1}, \langle\zeta, e_\Gamma\rangle>0\}\in\R.$$
$t_0<0$ and $\omega+t\zeta\in\Delta^\vee$, if and only if, $t\geq t_0$ for any $t\in\R$.

Since $\Vect(H)=\R\zeta$ and $H=\R_0\zeta$, we know
$\{\omega+t\zeta|t\in\R, t\geq t_0\}=(\{\omega\}+\Vect(H))\cap\Delta^\vee\not\subset\{\omega\}+\R_0\zeta=\{\omega\}+H$, and $\omega\not\in\partial_-\Delta^\vee$.

We know $\Delta^\circ(\Lambda,\Delta|W)\cap\partial_-\Delta^\vee=\emptyset$.
Claim $(c)$ does not hold.

\noindent $(b)\Rightarrow (c)$. Trivial.

\end{proof}

\begin{lemma}
\label{another correspondence}
Denote
\begin{equation*}
\begin{split}
\mathcal{X}(V)&=\text{the set of all convex pseudo polyhedrons in }V,\\
\mathcal{X}(V,N)&=\text{the set of all rational convex pseudo polyhedrons over }N\text{ in }V,\\
\mathcal{Y}(V)&=\{(\mathcal{D},\phi)|
\mathcal{D}\text{ is a convex polyhedral cone decomposition in }V^*\text{ such that}\\
&\qquad\quad\text{the support }|\mathcal{D}|\text{ of }\mathcal{D}\text{ is a convex polyhedral cone in }V^*\text{ and there}\\
&\qquad\quad\text{exists a piecewise linear convex function }|\mathcal{D}|\rightarrow\R\text{ over }\mathcal{D},\\
&\qquad\quad 
\phi: |\mathcal{D}|\rightarrow\R \text{ is a piecewise linear convex function over }\mathcal{D}\},\\
\mathcal{Y}(V, N)&=\{(\mathcal{D},\phi)|
\mathcal{D}\text{ is a rational convex polyhedral cone decomposition over }N^*\\
&\qquad\quad\text{in }V^*\text{ such that the support }|\mathcal{D}|\text{ of }\mathcal{D}
\text{ is a convex polyhedral cone in }\\
&\qquad\quad V^*\text{ and there exists a piecewise linear function }|\mathcal{D}|\rightarrow\R\text{ which is}\\
&\qquad\quad\text{convex over }\mathcal{D}\text{ and rational over }N^*,
\phi: |\mathcal{D}|\rightarrow\R \text{ is a piecewise}\\
&\qquad\quad\text{linear function which is convex over }\mathcal{D}\text{ and rational over }N^*\}.
\end{split}
\end{equation*}
$\mathcal{X}(V,N)\subset\mathcal{X}(V)$.
$\mathcal{Y}(V,N)\subset\mathcal{Y}(V)$.
Putting $\Phi(S)=(\mathcal{D}(S|V),\Ord(\;, S|V))\in\mathcal{Y}(V)$ for any $S\in\mathcal{X}$, we define a mapping $\Phi:\mathcal{X}(V)\rightarrow\mathcal{Y}(V)$.
$\Phi$ induces a mapping $\Phi': \mathcal{X}(V, N)\rightarrow\mathcal{Y}(V, N)$.

For any $(\mathcal{D},\phi)\in\mathcal{Y}$ we denote
\begin{equation*}
\begin{split}
\Sigma(\mathcal{D},\phi)=\{\xi\in W^*|&\xi=\rho^*(\bar{\omega})+t\zeta
\text{ for some }\bar{\omega}\in |\mathcal{D}|\text{ and some }t\in\R\\
&\quad\text{with } t\geq-\phi(\bar{\omega})\}.
\end{split}
\end{equation*}
\begin{enumerate}
\item
Consider any $(\mathcal{D},\phi)\in\mathcal{Y}(V)$.

$\Sigma(\mathcal{D},\phi)$ is a convex polyhedral cone in $W^*$.
$H^\vee\subset\Sigma(\mathcal{D},\phi)$.
$-H^\vee\not\subset\Sigma(\mathcal{D},\phi)$.

$\Sigma(\mathcal{D},\phi)^\vee|W^*$ is a convex polyhedral cone in $W$.
$\Sigma(\mathcal{D},\phi)^\vee|W^*\subset H$.\hfill\break
$(\Sigma(\mathcal{D},\phi)^\vee|W^*)\cap H^\circ\neq\emptyset$.
$\sigma((\Sigma(\mathcal{D},\phi)^\vee|W^*)\cap H^\circ)\in\mathcal{X}(V)$.

If $(\mathcal{D},\phi)\in\mathcal{Y}(V, N)$, then $\Sigma(\mathcal{D},\phi)$ is rational over $(N+\Z z)^*$, $\Sigma(\mathcal{D},\phi)^\vee|W^*$ is rational over $N+\Z z$ and $\sigma((\Sigma(\mathcal{D},\phi)^\vee|W^*)\cap H^\circ)\in\mathcal{X}(V, N)$.
\end{enumerate}
Putting
$\Psi(\mathcal{D},\phi)= \sigma((\Sigma(\mathcal{D},\phi)^\vee|W^*)\cap H^\circ)\in\mathcal{X}(V)$ for any $(\mathcal{D},\phi)\in\mathcal{Y}(V)$, we define a mapping $\Psi: \mathcal{Y}(V)\rightarrow \mathcal{X}(V)$.
$\Psi$ induces a mapping $\Psi': \mathcal{Y}(V, N)\rightarrow \mathcal{X}(V, N)$.
\begin{enumerate}
\setcounter{enumi}{1}
\item
$\Phi$ and $\Psi$ are bijective mappings, and they are the inverse mappings of each other.

$\Phi'$ and $\Psi'$ are bijective mappings, and they are the inverse mappings of each other.
\end{enumerate}
\end{lemma}

\begin{remark}
Assume $\dim V=3$. 
There exists a rational convex polyhedral cone decomposition $\mathcal{D}$ over $N^*$ in $V^*$ such that $|\mathcal{D}|=V^*$ and there does \emph{not} exist a piecewise linear convex function $|\mathcal{D}|\rightarrow \R$ over $\mathcal{D}$.

See Fulton~\cite{F93} page 71.
\end{remark}

\begin{prop}
\label{property of polyhedrons}
Let $S$ be any convex pseudo polyhedron in $V$, and let $X, Y$ be any finite subsets of $V$ satisfying $S=\Conv(X)+\Convcone(Y)$ and $X\neq\emptyset$.
We consider the dual cone $\Stab(S)^\vee=\Stab(S)^\vee|V\subset V^*$ of $\Stab(S)$.
For simplicity we denote
$s=\dim S\in\Z_0$, $L=\Stab(S)\cap(-\Stab(S))\subset \Stab(S)$, $\ell=\dim L\in\Z_0$.
\begin{enumerate}
\item
We consider any finite dimensional vector space $U$ over $\R$ with $\dim S\leq\dim U\leq\dim V$, any injective homomorphism $\nu:U\rightarrow V$ of vector spaces over $\R$, any point $a\in V$ such that $S\subset\nu(U)+{a}$, and any subset $F$ of $S$.
Putting $\bar{\nu}(x)=\nu(x)+a\in V$ for any $x\in U$ we define an injective mapping $\bar{\nu}:U\rightarrow V$.
$S\subset\bar{\nu}(U)$.
The inverse image $\bar{\nu}^{-1}(S)$ is a convex polyhedral cone in $U$.
The set $F$ is a face of $S$, if and only if $\bar{\nu}^{-1}(F)$ is a face of $\bar{\nu}^{-1}(S)$.
\item
We consider any finite dimensional vector space $W$ over $\R$ with $\dim V\leq\dim W$, any injective homomorphism $\pi:V\rightarrow W$ of vector spaces over $\R$, any point $b\in W$ and any subset $F$ of $S$.
Putting $\bar{\pi}(x)=\pi(x)+b$ for any $x\in V$ we define an injective mapping $\bar{\pi}:V\rightarrow W$.
The image $\bar{\pi}(S)$ is a convex polyhedral cone in $W$.
The set $F$ is a face of $S$, if and only if $\bar{\pi}(F)$ is a face of $\bar{\pi}(S)$.
\item
$\ell\leq s$. $\ell=s\Leftrightarrow L+\{a\}=S$ for some $a\in V\Leftrightarrow S=\Affi(S)$.
\item
Let $F$ be any face of $S$.
\begin{enumerate}
\item
$F$ is a convex pseudo polyhedron in $V$.
$\Stab(F)$ is a face of $\Stab(S)$.
\item
If $\omega\in\Stab(S)^\vee$ and $F=\Delta(\omega, S)$, then $\Stab(F)=\Delta(\omega,\Stab(S))$.
\item
$\Stab(F)=\Convcone(Y\cap\Stab(F))$.
$\Vect(\Stab(F))=\Vect(Y\cap\Stab(F))$.
\item
$F=\Conv(X\cap F)+\Stab(F)=S\cap\Affi(F)$. $\Affi(F)=\Affi(X\cap F)+\Vect(\Stab(F))$.
\item
If $S$ is rational over $N$, then $F$ is also rational over $N$.
\item
$L=\Stab(F)\cap(-\Stab(F))\subset \Stab(F)\subset\Vect(\Stab(F))\subset\Stab(\Affi(F))$.
$\ell\leq\dim \Stab(F)\leq\dim F\leq s$.
\item
Let $G$ be any face of $S$ with $G\subset F$. We have $\dim G\leq\dim F$.  $\dim G=\dim F$, if and only if, $G=F$.
\item
Let $G$ be any subset of $F$. $G$ is a face of the convex pseudo polyhedron $F$, if and only if, $G$ is a face of $S$ with $G\subset F$.
\end{enumerate}
\item
Assume $\ell<s$.
For any face $G$ of $S$ with $G\neq S$ there exists a face $F$ of $S$ with $\dim F=s-1$ and $G\subset F$.
There exists a face $F$ of $S$ with $\dim F=s-1$.
\item
$\mathcal{F}(S)$ is a finite set.
$S\in\mathcal{F}(S)_s$ and $\mathcal{F}(S)_s=\{S\}$.
$S$ contains any face of $S$.
For any $i\in\Z_0$, $\mathcal{F}(S)_i\neq\emptyset$ if and only if $\ell\leq i\leq s$.
The characteristic number $c(S)$ of $S$ is equal to $\sharp\mathcal{F}(S)_\ell$.
$c(S)$ is a positive integer.
\item
$L=\Convcone(Y\cap L)=\Vect(Y\cap L)$.

Any face $G$ of $S$ with $\dim G=\ell$ is an affine space in $V$ with $\Stab(G)=L$.

For any face $F$ of $S$ and any face $G$ of $S$ with $\dim G=\ell$, $F\supset G$, if and only if, $F\cap G\neq\emptyset$.

For any faces $F$, $G$ of $S$ with $\dim F=\dim G=\ell$, $F=G$, if and only if, $F\cap G\neq\emptyset$.

For any face $F$ of $S$, there exists a face $G$ of $S$ such that $\dim G=\ell$ and $F\supset G$.

Consider any face $G$ of $S$ with $\dim G=\ell$ and any point $\omega\in\Vect(\Stab(S)^\vee)$. The function $\langle\omega,\;\rangle:G\rightarrow\R$ sending $x\in G$ to $\langle\omega,x\rangle\in\R$ is a constant function on $G$.
\item
The skeleton $\mathcal{V}(S)$ of $S$ is a non-empty closed subset of $S$.
Any connected component of $\mathcal{V}(S)$ is an affine space $G$ in $V$ with $\Stab(G)=L$.
The set of connected components of $\mathcal{V}(S)$ is equal to $\mathcal{F}(S)_\ell$.
The number of connected components of $\mathcal{V}(S)$ is equal to $c(S)$.

For any point $\omega\in\Vect(\Stab(S)^\vee)$, the function $\langle\omega,\;\rangle: \mathcal{V}(S)\rightarrow\R$ sending $x\in \mathcal{V}(S)$ to $\langle\omega,x\rangle\in\R$ is constant on each connected component of $\mathcal{V}(S)$, and this function has only finite number of values.

For any face $F$ of $S$, the intersection $F\cap\mathcal{V}(S)$ is non-empty and union of some connected components of $\mathcal{V}(S)$.
\item
Let $F$ and $G$ be any face of $S$ with $F\subset G$.
We denote $f=\dim F$ and $g=\dim G$.
$\ell\leq f\leq g\leq s$.
There exist $(s-\ell+1)$ of faces $F(\ell), F(\ell+1),\ldots, F(s)$ satisfying the following three conditions:
\begin{enumerate}
\item
For any $i\in\{\ell,\ell+1,\ldots,s-1\}$, $F(i)\subset F(i+1)$.
\item
For any $i\in\{\ell,\ell+1,\ldots,s\}$, $\dim F(i)=i$.
\item
$F(f)=F, F(g)=G, F(s)=S$.
\end{enumerate}
\item
Let $F$ be any face of $S$.
\begin{enumerate}
\item
$F=\partial F\cup F^\circ$. $\partial F\cap F^\circ=\emptyset$.
\item
$F^\circ=F\Leftrightarrow\partial F=\emptyset\Leftrightarrow\dim F=\ell$.
\item
$$\partial F=\bigcup_{G\in\mathcal{F}(F)-\{F\}}G.$$
\item
$F^\circ $ is a non-empty open subset of $\Affi(F)$.
For any $a\in F^\circ$ and any $b\in F$ $\Conv(\{a,b\})-\{b\}\subset F^\circ$.
$F^\circ$ is convex.
$\Clos(F^\circ)=F$.
\end{enumerate}
\item
For any face $G$ of $S$ satisfying $\dim G=\ell$ we take any point $a_G\in G$.\begin{enumerate}
\item
$S=\Conv(\{a_G|G\in\mathcal{F}(S)_{\ell}\})+\Stab(S)$.
\item
For any $\omega\in\Vect(\Stab(S)^\vee)$, $\{\langle\omega, x\rangle|x\in\mathcal{V}(S)\}= $\hfill\break$\{\langle\omega, a_G\rangle| G\in\mathcal{F}(S)_{\ell}\}$.
\item
For any $\omega\in\Stab(S)^\vee$, $\Ord(\omega, S|V)=\min\{\langle\omega, a_G\rangle| G\in\mathcal{F}(S)_{\ell}\}$.
\item
For any face $F$ of $S$,
$F=\Conv(\{a_G| G\in\mathcal{F}(S)_{\ell}, G\subset F\})+\Stab(F)$.
\end{enumerate}
\item
Consider any $m\in\Z_+$ and any mapping $F:\{1,2,\ldots,m\}\rightarrow\mathcal{F}(S)$.
If $\cap_{i\in\{1,2,\ldots,m\}}F(i)\neq\emptyset$, then the intersection $\cap_{i\in\{1,2,\ldots,m\}}F(i)$ is a face of $S$.
\item
The face cone decomposition $\mathcal{D}(S|V)$ of $S$ is a convex polyhedral cone decomposition in $V^*$.
$|\mathcal{D}(S|V)|=\Stab(S)^\vee$.
$\dim \mathcal{D}(S|V)=\dim \Stab(S)^\vee=\dim V-\ell$.
$c(S)=\sharp\mathcal{D}(S|V)^0$.
If $S$ is rational over $N$, then $\mathcal{D}(S|V)$ is rational over $N^*$.
The minimum element of $\mathcal{D}(S|V)$ is $\Delta(S, S|V)$.
$\Delta^\circ(S, S|V)=\Delta(S, S|V)=\Stab(\Affi(S))^\vee|V$.
$\dim\Delta(S, S|V)=\dim V-s$.

For any $i\in\Z$, $\mathcal{D}(S|V)_i\neq\emptyset$ if and only if $\dim V-s\leq i\leq\dim V-\ell$, and $\mathcal{D}(S|V)^i\neq\emptyset$ if and only if $0\leq i\leq s-\ell$.
\item
Let $F$ be any face of $S$.
\begin{enumerate}
\item
$\Delta(F,S|V)\in\mathcal{D}(S|V)$.
\item
$\Vect(\Delta(F,S|V))=\Stab(\Affi(F))^\vee|V$.
\item
$\dim F+\dim\Delta(F,S|V)=\dim V$.
\item
$\Delta^\circ(F, S|V)=\Delta(F,S|V)^\circ$.
\item
$\Delta(F,S|V)=\Clos(\Delta^\circ(F, S|V))$.
\item
$\Delta^\circ(F,S|V)\subset\Delta^\circ(\Stab(F),\Stab(S)|V)\in\mathcal{F}(\Stab(S)^\vee)$.
\end{enumerate}
\item
For any faces $F,G$ of $S$, $F\subset G$, if and only if, $\Delta(F,S|V)\supset\Delta(G,S|V)$.

The mapping from $\mathcal{F}(S)$ to $\mathcal{D}(S|V)$ sending $F\in\mathcal{F}(S)$ to $\Delta(F,S|V)\in\mathcal{D}(S|V)$ is a bijective mapping reversing the inclusion relation.
\item
Consider any two faces $F, G$ of $S$.
The following four conditions are equivalent:
\begin{enumerate}
\item
$F\subset G$.
\item
$F^\circ\cap G\neq\emptyset$.
\item
$\Delta(F)\supset\Delta(G)$
\item
$\Delta(F)\cap\Delta^\circ(G)\neq\emptyset$.
\end{enumerate}
The following six conditions are also equivalent:
\begin{enumerate}
\setcounter{enumii}{4}
\item
$F=G$.
\item
$F^\circ=G^\circ$.
\item
$F^\circ\cap G^\circ\neq\emptyset$.
\item
$\Delta(F)=\Delta(G)$.
\item
$\Delta^\circ(F)=\Delta^\circ(G)$.
\item
$\Delta^\circ(F)\cap\Delta^\circ(G)\neq\emptyset$.
\end{enumerate}
\item
$c(S)=\sharp\mathcal{D}(S|V)^0=\sharp\mathcal{F}(S)_\ell=$ the number of connected components of $\mathcal{V}(S)\in\Z_+$.
\item
The following three conditions are equivalent:
\begin{enumerate}
\item
$c(S)=1$.
\item
$\mathcal{D}(S|V)=\mathcal{F}(\Stab(S)^\vee)$.
\item
There exists $a\in V$ satisfying $S=\{a\}+\Stab(S)$.
\end{enumerate}
\item
The family $\{F^\circ|F\in\mathcal{F}(S)\}$ of subsets of $S$ gives the equivalence class decomposition of $S$, in other words, the following three conditions hold:
\begin{enumerate}
\item
$F^\circ\neq\emptyset$ for any $F\in\mathcal{F}(S)$.
\item
$F^\circ=G^\circ$, if and only if, $F^\circ\cap G^\circ\neq\emptyset$ for any $F\in\mathcal{F}(S)$ and any $G\in\mathcal{F}(S)$.
\item
$$S=\bigcup_{ F\in\mathcal{F}(S)}F^\circ.$$
\end{enumerate}
\item
The family $\{\Delta^\circ(F)|F\in\mathcal{F}(S)\}$ of subsets of $\Stab(S)^\vee$ gives the equivalence class decomposition of $\Stab(S)^\vee $, in other words, the following three conditions hold:
\begin{enumerate}
\item
$\Delta^\circ(F)\neq\emptyset$ for any $F\in\mathcal{F}(S)$.
\item
$\Delta^\circ(F)=\Delta^\circ(G)$, if and only if, $\Delta^\circ(F)\cap\Delta^\circ(G)\neq\emptyset$ for any $F\in\mathcal{F}(S)$ and any $G\in\mathcal{F}(S)$.
\item
$$\Stab(S)^\vee =\bigcup_{ F\in\mathcal{F}(S)}\Delta^\circ(F).$$
\end{enumerate}
\item
The function $\Ord(\;,S|V):\Stab(S)^\vee\rightarrow\R$ sending $\omega\in\Stab(S)^\vee$ to $\Ord(\omega,S|$\hfill\break$V)\in\R$ is a piecewise linear convex function over $\mathcal{D}(S|V)$.

If $S$ is rational over $N$, then this function $\Ord(\;,S|V)$ is rational over $N^*$.
\item
Let $m=\dim\Delta(S)\in\Z_0$.
$m=\dim V-s$.
For any $\Lambda\in\mathcal{D}(S|V)_{m+1}$ we take any point $\omega_\Lambda\in\Lambda-\Delta(S)$.
Then, 
\begin{equation*}\begin{split}
S&=\bigcap_{\omega\in\Stab(S)^\vee}\{x\in V|\langle\omega,x\rangle\geq\Ord(\omega,S|V)\}\cap\Affi(S)\\
&=\bigcap_{\Lambda\in\mathcal{D}(S|V)_{m+1}}\{x\in V|\langle\omega_\Lambda,x\rangle\geq\Ord(\omega_\Lambda,S|V)\}\cap\Affi(S).
\end{split}\end{equation*}
\item
Consider any finite dimensional vector space $W$ over $\R$ and any homomorphism $\pi:V\rightarrow W$ of vector spaces over $\R$.
The image $\pi(S)$ is a convex pseudo polyhedron in $W$, and $\pi(S)^\circ=\pi(S^\circ)$.

If $S$ is a convex polyhedron in $V$, then $\pi(S)$ is a convex polyhedron in $W$.
If $S$ is a convex polyhedral cone in $V$, then $\pi(S)$ is a convex polyhedral cone in $W$.
\end{enumerate}
\end{prop}

\begin{lemma}
\label{sum}
Consider any convex pseudo polyhedrons $S$, $T$ in $V$.

$S+T$ is a convex pseudo polyhedron in $V$, $\Stab(S+T)=\Stab(S)+\Stab(T)$,
$\Stab(S+T)^\vee=\Stab(S)^\vee\cap\Stab(T)^\vee$, and
$\mathcal{D}(S+T|V)=\mathcal{D}(S|V)\hat{\cap}\mathcal{D}(T|V)$.

If $S$ and $T$ are rational over $N$, then $S+T$ is also rational over $N$.

If $S$ and $T$ are convex polyhedrons, then $S+T$ is also a convex polyhedron.

If $S$ and $T$ are convex polyhedral cones, then $S+T$ is also a convex polyhedral cone.
\end{lemma}

\begin{cor}
\label{sum2}
Consider any convex pseudo polyhedron $S$ in $V$ and any convex polyhedral cone $\Delta$ in $V$.

$S+\Delta$ is a convex pseudo polyhedron in $V$, $\Stab(S+\Delta)=\Stab(S)+\Delta$, $\Stab(S+\Delta)^\vee=\Stab(S)^\vee\cap\Delta^\vee$ and
$\mathcal{D}(S+\Delta|V)=\mathcal{D}(S|V)\hat{\cap}\mathcal{F}(\Delta^\vee|V)$.

If $S$ and $\Delta$ are rational over $N$, then $S+\Delta$ is also rational over $N$.

$S\subset S+\Delta$.
For any $\omega\in\Stab(S+\Delta)^\vee$, we have $\omega\in\Stab(S)^\vee$, $\Ord(\omega, S)=\Ord(\omega, S+\Delta)$, and
$\Delta(\omega, S)=\Delta(\omega, S+\Delta)\cap S$.
\end{cor}

\begin{definition}
\label{Newton}
We call any convex pseudo polyhedron $S$ in $V$ satisfying the following three conditions a \emph{Newton polyhedron} over $N$ in $V$:
\begin{enumerate}
\item
The stabilizer $\Stab(S)$ of $S$ is a simplicial cone over $N$.
\item
$\dim\Stab(S)=\dim V$.
\item
$\mathcal{V}(S)\subset N$.
\end{enumerate}
\end{definition}

Let $S$ be a Newton polyhedron $S$ over $N$ in $V$. 
$\Stab(S)\cap(-\Stab(S))=\{0\}$, and $S$ has a vertex.
The skeleton $\mathcal{V}(S)$ of $S$ is a non-empty finite subset of $S$, and  $\mathcal{V}(S)$ is the union of all vertices of $S$.

\begin{lemma}
\label{Newton1}
Consider any simplicial cone $\Delta$ over $N$ in $V$ with $\Delta=\dim V$ and any subset $X$ of $N$ such that $X\subset\{a\}+\Delta$ for some $a\in V$.

There exists a \emph{finite} subset $Y$ of $X$ satisfying $\Conv(X)+\Delta=\Conv(Y)+\Delta$, and $\Conv(X)+\Delta$ is a Newton polyhedron over $N$ in $V$.
\end{lemma}

\begin{remark}
The subset $X$ of $N$ above is not necessarily finite.
\end{remark}

\begin{lemma}
\label{Newton2}
Let $k$ be any field.
Let $A$ be any complete regular local ring such that $\dim A\geq 1$, $A$ contains $k$ as a subring, and the residue field  $A/M(A)$ is isomorphic to $k$ as algebras over $k$.
Let $P$ be any parameter system of $A$.
Let $\phi$ be any non-zero element of $A$.

\begin{enumerate}
\item
The Newton polyhedron $\Gamma_+(P,\phi)$ of $\phi$ over $P$ is a Newton polyhedron over \hfill\break$\Map(P,\Z)$ in $\Map(P,\R)$ in the meaning of Definition~\ref{Newton}.
$\Stab(\Gamma_+(P,\phi))=\Map(P,\R_0)$. $\Gamma_+(P,\phi)\subset\Map(P,\R_0)$.
\item
The face cone decomposition $\mathcal{D}(\Gamma_+(P,\phi)| \Map(P,\R))$ of $\Gamma_+(P,\phi)$ is a rational convex polyhedral cone decomposition over $\Map(P,\Z)^*$ in $\Map(P,\R)^*$.
\hfill\break
$|\mathcal{D}(\Gamma_+(P,\phi)| \Map(P,\R))|=\Map(P,\R_0)^\vee|\Map(P,\R)$.
\item
The Newton polyhedron $\Gamma_+(P,\phi)$ has a vertex.
The skeleton $\mathcal{V}(\Gamma_+(P,\phi))$ of $\Gamma_+(P,\phi)$ is a non-empty finite subset of $\Map(P,\Z_0)$, and $\mathcal{V}(\Gamma_+(P,\phi))$ is the union of all vertices of $\Gamma_+(P,\phi)$.
\begin{equation*}\begin{split}
\mathcal{V}(\Gamma_+(P, \phi))= &\{a\in\Gamma_+(P, \phi)|
\text{ There exists }\omega\in\Map(P,\R_0)^\vee|\Map(P,\R)\text{ such that}\\
&\qquad\text{for any } b\in\Gamma_+(P, \phi)\text{ with }\langle\omega,b\rangle=\langle\omega,a\rangle\text{, we have }b=a\}.\\
\end{split}\end{equation*}
\item
$c(\Gamma_+(P,\phi)) =\sharp \mathcal{D}(\Gamma_+(P,\phi)|\Map(P,\R))^0=\sharp\mathcal{V}(\Gamma_+(P,\phi))=\sharp\mathcal{F}(\Gamma_+(P,\phi))_0$\hfill\break$=\text{the number of vertices of }\Gamma_+(P,\phi)$.
\item
For any $\omega\in\Map(P,\R_0)^\vee|\Map(P,\R)$, we have
\begin{equation*}
\begin{split}
\Ord(P,\omega,\phi)=&\Ord(\omega, \Gamma_+(P,\phi)|\Map(P,\R)), \text{ and}\\
\In(P,\omega,\phi)=&\Ps(P,\Delta(\omega, \Gamma_+(P,\phi)|\Map(P,\R)),\phi).
\end{split}
\end{equation*}
\item
$c(\Gamma_+(P,\phi))=1$, if and only if, $\phi$ has normal crossings over $P$.
\item
If $\dim A=1$, then $c(\Gamma_+(P,\phi))=1$.
\end{enumerate}

Let $z\in P$ be any element.

Let $b=\Ord(P, f^{P\vee}_z,\phi)\in\Z_0$ and let $h=\Ht(z,\Gamma_+(P, \phi)) \in\Z_0$.
Let $A'$ denote the completion of $k[P - \{z\}]$ with respect to the maximal ideal $k[P - \{z\}]\cap M(A)$. The ring $A'$ is a local subring of $A$ and $M(A')=M(A)\cap A' =(P -\{z\})A'$. The completion of $A'[z]$ with respect to the prime ideal $zA'[z]$ is isomorphic to $A$ as $A'[z]$-algebras. The set $P-\{z\}$ is a parameter system of $A'$.

\begin{enumerate}
\setcounter{enumi}{7}
\item
Assume that $\Gamma_+(P, \phi)$ is of $z$-Weierstrass type.
\begin{enumerate}
\item
$\Ht(z,\Gamma_+(P, \phi))=0\Leftrightarrow \Gamma_+(P,\phi)$ has only one vertex $\Leftrightarrow c(\Gamma_+(P, \phi))=1\Leftrightarrow \phi$ has normal crossings over $P$.
\item 
The Newton polyhedron $\Gamma_+(P,\phi)$ has a unique $z$-top vertex.
\end{enumerate}

Below, by $\{a_1\}$ we denote the unique $z$-top vertex of $\Gamma_+(P,\phi)$.
\begin{enumerate}
\setcounter{enumii}{2}
\item
Consider any $a\in \Gamma_+(P, \phi)$. The equality $\langle f^{P\vee}_x, a\rangle=\Ord(P, f^{P\vee}_x,\phi)$ holds for any $x\in P - \{z\}\Leftrightarrow a-a_1\in\R_0f^P_z$.
\item
$\langle f^{P\vee}_z, a_1\rangle=b+h$.
\item
There exist uniquely an invertible element $u\in A^{\times}$ and a mapping $\phi':\{0, 1,\ldots, h-1\}\rightarrow M(A')$ satisfying
$$\phi=u z^b\prod_{x\in P -\{z\}}x^{\langle f^{P\vee}_x, a_1\rangle} (z^h+\sum_{i=0}^{h-1} \phi'(i) z^i),$$
and $\phi'(0)\neq 0$ if $h>0$
\end{enumerate}
\item
The following two conditions are equivalent:
\begin{enumerate}
\item The Newton polyhedron $\Gamma_+(P,\phi)$ is of $z$-Weierstrass type.
\item There exist uniquely an invertible element $u\in A^{\times}$, a mapping $c:P\rightarrow\Z_0$, a non-negative integer $g\in\Z_0$ and a mapping $\phi':\{0, 1,\ldots, g-1\}\rightarrow M(A')$ satisfying
$$\phi=u \prod_{x\in P}x^{c(x)} (z^g+\sum_{i=0}^{g-1} \phi'(i) z^i),$$
and $\phi'(0)\neq 0$ if $g>0$.
\end{enumerate}
\item
If $\dim A=2$, then $\Gamma_+(P,\phi)$ is $z$-simple.
\item
If $\Gamma_+(P,\phi)$ is $z$-simple, then $\Gamma_+(P,\phi)$ is of $z$-Weierstrass type.
\item
Let $r=c(\Gamma_+(P, \phi))\in\Z_+$. The Newton polyhedron $\Gamma_+(P,\phi)$ is $z$-simple, if and only if, the following three conditions are satisfied:
\begin{enumerate}
\item
For any $a\in \mathcal{V}(\Gamma_+(P, \phi))$ and any $b\in \mathcal{V}(\Gamma_+(P, \phi))$, $\langle f^{P\vee}_z, a\rangle \neq \langle f^{P\vee}_z, b\rangle$.
\end{enumerate}

Below we take the unique bijective mapping $a:\{1,2,\ldots,r\}\rightarrow \mathcal{V}(\Gamma_+(P, \phi))$ satisfying $\langle f^{P\vee}_z, a(i)\rangle > \langle f^{P\vee}_z, a(i+1)\rangle$ for any $i\in\{1,2,\ldots,r-1\}$, if $r\geq 2$.
\begin{enumerate}
\setcounter{enumii}{1}
\item For any $x\in P- \{z\}$, $\langle f^{P\vee}_x, a(2)-a(1)\rangle \geq 0$, if $r\geq 2$.
\item For any $i\in\{1,2,\ldots,r-2\}$ and any  $x\in P - \{z\}$,
$$\frac{\langle f^{P\vee}_x, a(i+1)-a(i)\rangle}{\langle f^{P\vee}_z, a(i)-a(i+1)\rangle} \leq
\frac{\langle f^{P\vee}_x, a(i+2)-a(i+1)\rangle}{\langle f^{P\vee}_z, a(i+1)-a(i+2)\rangle},$$
if $r\geq 3$.
\end{enumerate}

Furthermore, if the above equivalent conditions are satisfied, then the following claims hold:
\begin{enumerate}
\setcounter{enumii}{3}
\item There exists $x\in P- \{z\}$ with  $\langle f^{P\vee}_x, a(2)-a(1)\rangle > 0$.
\item For any $i\in\{1,2,\ldots,r-2\}$, there exists  $x\in P - \{z\}$ with
$$\frac{\langle f^{P\vee}_x, a(i+1)-a(i)\rangle}{\langle f^{P\vee}_z, a(i)-a(i+1)\rangle} < 
\frac{\langle f^{P\vee}_x, a(i+2)-a(i+1)\rangle}{\langle f^{P\vee}_z, a(i+1)-a(i+2)\rangle},$$
if $r\geq 3$.
\end{enumerate}
\item
Assume that $\Gamma_+(P,\phi)$ is of $z$-Weierstrass type. Let $\{a_1\}$ denote the unique $z$-top vertex of $\Gamma_+(P,\phi)$. We take an invertible element $u\in A^{\times}$ and a mapping $\phi':\{0,1,\ldots,h-1\}\rightarrow M(A')$ satisfying
$$\phi=u z^b\prod_{x\in P -\{z\}}x^{\langle f^{P\vee}_x, a_1\rangle} (z^h+\sum_{i=0}^{h-1} \phi'(i) z^i),$$
and $\phi'(0)\neq 0$ if $h>0$.
Then, $\Gamma_+(P,\phi)$ is $z$-simple, if and only if, there exist positive integer $r$, and a mapping $c:\{1,2,\ldots,r\}\rightarrow \Map(P,\Z_0)$ satisfying the following conditions:
\begin{enumerate}
\item $1\leq r\leq h+1$. $r=1\Leftrightarrow h=0$.
\item $c(1)=hf_z^P$. $\langle f_z^{P\vee}, c(r)\rangle=0$.
\item For any $i\in\{1,2,\ldots,r-1\}$, we have $\langle f_z^{P\vee}, c(i)-c(i+1)\rangle>0$, if $r\geq 2$.
\item For any $x\in P-\{z\}$, we have
$$\langle f_x^{P\vee},c(2)-c(1)\rangle\geq 0,$$ if $r\geq 2$.
\item There exists $x\in P-\{z\}$ with
$$\langle f_x^{P\vee}, c(2)-c(1)\rangle> 0,$$ if $r\geq 2$.
\item For any $i\in\{1,2,\ldots,r-1\}$ and any $x\in P-\{z\}$, we have
$$\frac{\langle f^{P\vee}_x, c(i+1)-c(i)\rangle}{\langle f^{P\vee}_z, c(i)-c(i+1)\rangle} \leq 
\frac{\langle f^{P\vee}_x, c(i+2)-c(i+1)\rangle}{\langle f^{P\vee}_z, c(i+1)-c(i+2)\rangle},$$
if $r\geq 3$.
\item For any $i\in\{1,2,\ldots,r-1\}$, there exists $x\in P-\{z\}$ with
$$\frac{\langle f^{P\vee}_x, c(i+1)-c(i)\rangle}{\langle f^{P\vee}_z, c(i)-c(i+1)\rangle} < 
\frac{\langle f^{P\vee}_x, c(i+2)-c(i+1)\rangle}{\langle f^{P\vee}_z, c(i+1)-c(i+2)\rangle},$$
if $r\geq 3$.
\item For any $i\in\{2,3,\ldots,r\}$ and any $x\in P-\{z\}$, we have
$$\Ord(P, f^{P\vee}_x, \phi'(\langle f^{P\vee}_z, c(i)\rangle))=\langle f^{P\vee}_x,c(i)\rangle,$$
if $r\geq 2$.
\item For any $i\in\{1,2,\ldots,r-1\}$, any $j\in\Z$
with
$$\langle f^{P\vee}_z, c(i+1)\rangle<j<\langle f^{P\vee}_z, c(i)\rangle$$
and any $x\in P-\{z\}$, we have
\begin{equation*}
\begin{split}
&\Ord(P, f^{P\vee}_x, \phi'(j))\geq\\
&\frac{\langle f^{P\vee}_z, c(i)\rangle-j}{\langle f^{P\vee}_z, c(i)-c(i+1)\rangle} \langle f^{P\vee}_x, c(i+1)\rangle
 +\frac{j-\langle f^{P\vee}_z, c(i+1)\rangle}{\langle f^{P\vee}_z, c(i)-c(i+1)\rangle} \langle f^{P\vee}_x, c(i)\rangle,\\
\end{split}
\end{equation*}
if $r\geq 2$.
\end{enumerate}
\item
For any non-zero element $\phi\in A$ and any non-zero element $\psi\in A$,
\begin{equation*}
\begin{split}
\Gamma_+(P,\phi\psi)=&\Gamma_+(P,\phi)+\Gamma_+(P,\psi),\\
\mathcal{D}(\Gamma_+(P,\phi\psi)|\Map(P,\R))=&
\mathcal{D}(\Gamma_+(P,\phi)|\Map(P,\R))\hat{\cap}\mathcal{D}(\Gamma_+(P, \psi)|\Map(P,\R)).
\end{split}
\end{equation*}
\end{enumerate}
\end{lemma}

\section{Barycentric subdivisions}
\label{bcd}
We study barycentric subdivisions of simplicial cone decompositions.

Let $V$ be any finite dimensional vector space over $\R$; let $N$ be any lattice of $V$; let $\mathcal{D}$ be any simplicial cone decomposition over $N$ in $V$ with $\dim\mathcal{D}\geq 1$; and let $F\in\mathcal{D}$ be any element with $\dim F\geq 1$. For simplicity we denote the barycenter $b_{F/N}$ of $F$ over $N$ by $b$.
$b\in F^\circ\cap N$.

\begin{lemma}
\label{prep1}
Consider any element $\Lambda\in\mathcal{D}$ satisfying $\Lambda+F\in\mathcal{D}$ and $F\not\subset\Lambda$.
\begin{enumerate}
\item
$\Lambda+\R_0b$ is a simplicial cone over $N$ in $V$.
$\R_0b\in\mathcal{F}(\Lambda+\R_0b)_1$. $\Lambda\in\mathcal{F}(\Lambda+\R_0b)^1$.
$\R_0b\cap\Lambda=\{0\}$.
$\R_0b=\Lambda\Op|(\Lambda+\R_0b)$.
$\Lambda=(\R_0b)\Op|(\Lambda+\R_0b)$.
\item
$\Lambda+\R_0b\subset\Lambda+F\in\mathcal{D}/F$.
$(\Lambda+\R_0b)^\circ\subset(\Lambda+F)^\circ$.
If $\dim F=1$, then $\R_0b=F$, and $\Lambda+\R_0b=\Lambda+F$.
If $\dim F\geq 2$, then $\Lambda+\R_0b\neq\Lambda+F$.
\item
$\dim(\Lambda+\R_0b)=\dim\Lambda+1\leq\dim(\Lambda+F)$.
$\Lambda\Op|(\Lambda+F)\in\mathcal{F}(F)\subset\mathcal{F}(\Lambda+F)$.
$\dim(\Lambda\Op|(\Lambda+F))\geq 1$.
$\dim(\Lambda\Op|(\Lambda+F))=1\Leftrightarrow\dim(\Lambda+\R_0b)= \dim(\Lambda+F)$.
\item
For any $\Lambda'\in\mathcal{F}(\Lambda)$, we have $\Lambda'+F\in\mathcal{D}$, and $F\not\subset\Lambda'$. $\mathcal{F}(\Lambda)\subset(\mathcal{D}/F)\Fc-(\mathcal{D}/F)$.
\item
$\{\Lambda'+\R_0b|\Lambda'\in\mathcal{F}(\Lambda)\}=\mathcal{F}(\Lambda+\R_0b)/\R_0b\subset\mathcal{F}(\Lambda+\R_0b)$.
$\mathcal{F}(\Lambda)= \mathcal{F}(\Lambda+\R_0b)-(\mathcal{F}(\Lambda+\R_0b)/\R_0b)\subset\mathcal{F}(\Lambda+\R_0b)$.
\item
For any $\Delta'\in\mathcal{D}-(\mathcal{D}/F)$, we have
$(\Lambda+\R_0b)\cap\Delta'=\Lambda\cap\Delta'\in\mathcal{F}(\Delta')$, and
$\Lambda\cap\Delta'\in\mathcal{F}(\Lambda)\subset\mathcal{F}(\Lambda+\R_0b)$.
\end{enumerate}
\end{lemma}

\begin{lemma}
\label{prep2}
Consider any element $\Lambda\in\mathcal{D}$ satisfying $\Lambda+F\in\mathcal{D}$ and $F\not\subset\Lambda$ and any element $\Lambda'\in\mathcal{D}$ satisfying $\Lambda'+F\in\mathcal{D}$ and $F\not\subset\Lambda'$.
\begin{enumerate}
\item
$(\Lambda\cap \Lambda')+F\in\mathcal{D}$.
$F\not\subset\Lambda\cap \Lambda'$.
\item
$(\Lambda+\R_0b)\cap(\Lambda'+\R_0b)= (\Lambda\cap \Lambda')+\R_0b\in\mathcal{F}(\Lambda+\R_0b)/ \R_0b$.
\item
$(\Lambda+F)\cap(\Lambda'+F)= (\Lambda\cap \Lambda')+F$.
\item
$\Lambda+\R_0b\subset\Lambda'+\R_0b\Leftrightarrow\Lambda\subset\Lambda'$.
\item
$\Lambda+\R_0b=\Lambda'+\R_0b\Leftrightarrow\Lambda=\Lambda'$.
\end{enumerate}
\end{lemma}

\begin{lemma}
\label{prep3}
Consider any element $\Delta\in\mathcal{D}/F$.
\begin{enumerate}
\item
$F\in\mathcal{F}(\Delta)$.
$\mathcal{F}(F)_1\subset\mathcal{F}(\Delta)_1$.
\item
$\{E\Op|\Delta\:|\:E\in\mathcal{F}(F)_1\}\subset\{\Lambda\in\mathcal{F}(\Delta)|F\not\subset\Lambda, \Lambda+F=\Delta\}\subset\{\Lambda\in\mathcal{F}(\Delta)|F\not\subset\Lambda\}\subset\{\Lambda\in\mathcal{D}| F\not\subset\Lambda, \Lambda+F\in\mathcal{D}\}$.
\item
$$\bigcup_{E\in\mathcal{F}(F)_1}(( E\Op|\Delta)+\R_0b)=
\bigcup_{\Lambda\in\mathcal{F}(\Delta), F\not\subset\Lambda, \Lambda+F=\Delta}(\Lambda+\R_0b)=
\bigcup_{\Lambda\in\mathcal{F}(\Delta), F\not\subset\Lambda}(\Lambda+\R_0b)=\Delta$$
\item
$$\bigcup_{\Lambda\in\mathcal{F}(\Delta), F\not\subset\Lambda, \Lambda+F=\Delta}(\Lambda+\R_0b)^\circ=\Delta^\circ$$
\end{enumerate}
\end{lemma}

\begin{definition}
\label{defbcd}
We denote
\begin{equation*}
\begin{split}
\mathcal{D}*F=
(\mathcal{D}-(\mathcal{D}/F))\cup
\{\Delta\in 2^V|& \Delta=\Lambda+\R_0b\text{ for some }\Lambda\in\mathcal{D}\\
&\quad\text{satisfying }\Lambda+F\in\mathcal{D}\text{ and } F\not\subset\Lambda\}\subset 2^V,
\end{split}
\end{equation*}
and we call $\mathcal{D}*F$ the \emph{barycentric subdivision} of $\mathcal{D}$ with \emph{center} in $F$.
\end{definition}

\begin{lemma}
\label{probcd}
\begin{enumerate}
\item
$\mathcal{D}*F$ is a simplicial cone decomposition over $N$ in $V$.
$\mathcal{D}*F$ is a full subdivision of $\mathcal{D}$.
$|\mathcal{D}*F|=|\mathcal{D}|$.
$\dim \mathcal{D}*F=\dim \mathcal{D}$.
\item
$\R_0b\in(\mathcal{D}*F)_1$.
$|\mathcal{D}*F/\R_0b|^\circ=|\mathcal{D}/F|^\circ$.
\item
$(\mathcal{D}*F)-(\mathcal{D}*F/\R_0b)= \mathcal{D}-(\mathcal{D}/F)$.
$\mathcal{D}*F/\R_0b=\{\Delta\in 2^V|\Delta=\Lambda+\R_0b\text{ for some }\Lambda\in\mathcal{D}
\text{ satisfying }\Lambda+F\in\mathcal{D}\text{ and }F\not\subset\Lambda\}$.
\item
If $\dim F=1$, then $\R_0b=F\in\mathcal{D}_1$ and $\mathcal{D}*F=\mathcal{D}$.
If $\dim F\geq 2$, then $\R_0b\not\in\mathcal{D}$, $\mathcal{D}*F\neq\mathcal{D}$,
$(\mathcal{D}*F)_1=\mathcal{D}_1\cup\{\R_0b\}$, and $\sharp(\mathcal{D}*F)_1=\sharp\mathcal{D}_1+1$.
\item
Consider any $\Delta\in \mathcal{D}*F/\R_0b$.
We denote $\Lambda=(\R_0b)\Op|\Delta\in\mathcal{F}(\Delta)$.
\begin{enumerate}
\item
$\Delta=\Lambda+\R_0b$. $\Lambda\cap\R_0b=\{0\}$.
$\R_0b\in\mathcal{F}(\Delta)_1$. $\Lambda\in\mathcal{F}(\Delta)^1$.
\item
$\Lambda+F=\Delta+F\in\mathcal{D}$.
$\Delta^\circ\subset(\Lambda+F)^\circ=(\Delta+F )^\circ$.
\end{enumerate}
\item
$(\mathcal{D}*F)\Mx-(\mathcal{D}*F/\R_0b)= \mathcal{D}\Mx-(\mathcal{D}/F)$.
$(\mathcal{D}*F)\Mx\cap(\mathcal{D}*F/\R_0b)=\{\Delta\in 2^V|\Delta=(E\Op|\Lambda)+ \R_0b\text{ for some }
E\in\mathcal{F}(F)_1\text{ and some }\Lambda\in\mathcal{D}\Mx/F\}$.
\end{enumerate}
\end{lemma}

\begin{example}
Assume $\dim V\geq 3$.
Consider any simplicial cone $S$ over $N$ in $V$ with $\dim S=3$.
Let $E(1), E(2), E(3)$ denote the three edges of $S$.
We denote $b(i)=b_{E(i)/N}\in E(i)^\circ\cap N$ for any $i\in\{1,2,3\}$ for simplicity.
$E(i)=\R_0b(i)$ for any $i\in\{1,2,3\}$.

Put
\begin{equation*}\begin{split}
T(1)&=\R_0b(1)+\R_0(b(1)+b(2))+\R_0(b(1)+b(3)), \\
T(2)&=\R_0b(2)+\R_0(b(2)+b(3))+\R_0(b(2)+b(1)), \\
T(3)&=\R_0b(3)+\R_0(b(3)+b(1))+\R_0(b(3)+b(2)), \\
T(4)&=\R_0(b(1)+b(2)+b(3))+\R_0(b(1)+b(2))+\R_0(b(1)+b(3)), \\
T(5)&=\R_0(b(1)+b(2)+b(3))+\R_0(b(2)+b(3))+\R_0(b(2)+b(1)), \\
T(6)&=\R_0(b(1)+b(2)+b(3))+\R_0(b(3)+b(1))+\R_0(b(3)+b(2)).
\end{split}\end{equation*}
For any $i\in\{1,2,\ldots,6\}$, $T(i)$ is a simplicial cone over $N$ in $V$ with $\dim T(i)=3$.

Let $\mathcal{E}=\cup_{i\in\{1,2,\ldots,6\}}\mathcal{F}(T(i))\subset 2^V$.
$\mathcal{E}$ is a simplicial cone decomposition over $N$ in $V$.
$\dim \mathcal{E}=3$ and $|\mathcal{E}|=S$.
$\mathcal{E}$ is a full subdivision of the simplicial cone decomposition $\mathcal{F}(S)$.

For any $\bar{F}\in \mathcal{F}(S)_2$, $\R_0b_{\bar{F}/N}\in\mathcal{E}$, and $\mathcal{E}$ is not a subdivision of $\mathcal{F}(S)*\bar{F}$.

$\R_0b_{S/N}\in\mathcal{E}$, and $\mathcal{E}$ is not a subdivision of $\mathcal{F}(S)*S$.
\end{example}

\section{Iterated barycentric subdivisions}
\label{ibcd}
We study iterated barycentric subdivisions of simplicial cone decompositions.

Let $V$ be any finite dimensional vector space over $\R$; let $N$ be any lattice of $V$; and let $\mathcal{D}$ be any simplicial cone decomposition over $N$ in $V$ with $\dim\mathcal{D}\geq 1$.

\begin{definition}
\label{defibcd}
Let $m\in\Z_0$ be any non-negative integer.
We call a mapping $F$ from $\{1,2,\ldots,m\}$ to the set $2^V$ of all subsets of $V$ satisfying the following two conditions a \emph{center sequence} of $\mathcal{D}$ of \emph{length} $m$:
\begin{enumerate}
\item
$F(i)$ is a simplicial cone over $N$ in $V$ and $\dim F(i)\geq 2$ for any $i\in\{1,2,\ldots,m\}$.
\item
There exists uniquely a mapping $\bar{\mathcal{D}}$ from $\{0,1,\ldots,m\}$ to the set of all simplicial cone decompositions over $N$ in $V$ satisfying the following two conditions:
\begin{enumerate}
\item
$\bar{\mathcal{D}}(0)=\mathcal{D}$.
\item
$F(i)\in\bar{\mathcal{D}}(i-1)$ and $\bar{\mathcal{D}}(i)=\bar{\mathcal{D}}(i-1)*F(i)$ for any $i\in\{1,2,\ldots,m\}$.
\end{enumerate}
\end{enumerate}

Consider any $m\in\Z_0$ and any center sequence $F$ of $\mathcal{D}$ of length $m$.
There exists uniquely a mapping $\bar{\mathcal{D}}$ from $\{0,1,\ldots,m\}$ to the set of all simplicial cone decompositions over $N$ in $V$ satisfying the above two conditions $(a)$ and $(b)$.
Since simplicial cone decomposition $\bar{\mathcal{D}}(m)$ is uniquely determined by $\mathcal{D}$ and the center sequence $F$ of $\mathcal{D}$, we denote $\bar{\mathcal{D}}(m)$ by the symbol
$$\mathcal{D}*F(1)*F(2)*\cdots*F(m),$$
and we call $\mathcal{D}*F(1)*F(2)*\cdots*F(m)$ the \emph{iterated barycentric subdivision} of $\mathcal{D}$ \emph{along} the center sequence $F$ of $\mathcal{D}$.

Consider any simplicial cone decomposition $\mathcal{E}$ over $N$ in $V$.
If $\mathcal{E}=\mathcal{D}*F(1)*F(2)*\cdots*F(m)$ for some $m\in\Z_0$ and some center sequence $F$ of $\mathcal{D}$ of length $m$, then we call $\mathcal{E}$ an \emph{iterated barycentric subdivision} of $\mathcal{D}$.
\end{definition}

\begin{lemma}
\label{propibcd1}
Consider any $m\in\Z_0$ and any center sequence $F$ of $\mathcal{D}$ of length $m$.

\begin{enumerate}
\item
$\mathcal{D}*F(1)*F(2)*\cdots*F(m)$ is a simplicial cone decomposition over $N$ in $V$.
$\dim\mathcal{D}*F(1)*F(2)*\cdots*F(m)=\dim\mathcal{D}$.
$\mathcal{D}*F(1)*F(2)*\cdots*F(m)$ is a full subdivision of $\mathcal{D}$.
$|\mathcal{D}*F(1)*F(2)*\cdots*F(m)|=|\mathcal{D}|$.
If $\mathcal{D}\Mx=\mathcal{D}^0$, then $(\mathcal{D}*F(1)*F(2)*\cdots*F(m))\Mx=(\mathcal{D}*F(1)*F(2)*\cdots*F(m))^0$.
\item
$\mathcal{D}*F(1)*F(2)*\cdots*F(m)=\mathcal{D}$, if $m=0$.

If $m=1$, then $\mathcal{D}*F(1)*F(2)*\cdots*F(m)$ is equal to the barycentric subdivision $\mathcal{D}*F(1)$ of $\mathcal{D}$ with center in $F(1)$.
\item
If $\dim\mathcal{D}=1$, then $m=0$ and $\mathcal{D}*F(1)*F(2)*\cdots*F(m)=\mathcal{D}$.
\item
For any $i\in\{0,1,\ldots,m\}$, the composition of the inclusion mapping\hfill\break $\{1,2,\ldots,i\}\rightarrow \{1,2,\ldots,m\}$ and $F: \{1,2,\ldots,m\}\rightarrow 2^V$ is a center sequence of $\mathcal{D}$ of lemgth $i$.
\item
Assume $m\geq 1$ and consider any $i\in\{1,2,\ldots,m\}$. 
$\dim F(i)\geq 2$.
$F(i)\in\mathcal{D}*F(1)*F(2)*\cdots*F(i-1)$.
$F(i)\subset|\mathcal{D}|$.
$(\mathcal{D}*F(1)*F(2)*\cdots*F(i-1))*F(i)= \mathcal{D}*F(1)*F(2)*\cdots*F(i)$.
\item
Assume $m\geq 2$ and consider any $i\in\{2,3,\ldots,m\}$.
The mapping $G:\{1,2,\ldots, m-i+1\}\rightarrow 2^V$ satisfying $G(j)=F(i+j-1)$ for any $j\in\{1,2,\ldots, m-i+1\}$ is a center sequence of $\mathcal{D}*F(1)*F(2)*\cdots*F(i-1)$ of length $m-i+1$, and
$(\mathcal{D}*F(1)*F(2)*\cdots*F(i-1))*F(i)*F(i+1)*\cdots*F(m)= \mathcal{D}*F(1)*F(2)*\cdots*F(m)$.
\item
$(\mathcal{D}*F(1)*F(2)*\cdots*F(m))_1=\mathcal{D}_1\cup\{\R_0 b_{F(i)/N}|i\in\{1,2,\ldots,m\}\}$.
$\mathcal{D}_1\cap\{\R_0 b_{F(i)/N}|i\in\{1,2,\ldots,m\}\}=\emptyset$.
For any $i\in\{1,2,\ldots,m\}$ and any $j\in\{1,2,\ldots,m\}$, $\R_0 b_{F(i)/N}=\R_0 b_{F(j)/N}$, if and only if, $i=j$.

$\sharp(\mathcal{D}*F(1)*F(2)*\cdots*F(m))_1=\sharp\mathcal{D}_1+m$.
\item
For any $\ell\in\Z_0$ and any center sequence $G$ of $\mathcal{D}*F(1)*F(2)*\cdots*F(m)$ of length $\ell$,
the mapping $H:\{1,2,\ldots, m+\ell\}\rightarrow 2^V$ satisfying $H(i)=F(i)$ for any $i\in\{1,2,\ldots,m\}$ and $H(i)=G(i-m)$ for any $i\in\{m+1,m+2,\ldots,m+\ell\}$
is a center sequence of $\mathcal{D}$ of length $m+\ell$ and
$\mathcal{D}*F(1)*F(2)*\cdots*F(m)*G(1)*G(2)*\cdots*G(\ell)
=(\mathcal{D}*F(1)*F(2)*\cdots*F(m))*G(1)*G(2)*\cdots*G(\ell)$.
\end{enumerate}

Consider any non-empty subset $\mathcal{E}$ of $\mathcal{D}$ satisfying $\mathcal{E}\Fc=\mathcal{E}$.
$\mathcal{E}$ is a simplicial cone decomposition over $N$ in $V$.
$|\mathcal{E}|\subset|\mathcal{D}|$.

\begin{enumerate}
\setcounter{enumi}{8}
\item
Let $\ell=\sharp\{i\in\{1,2,\ldots,m\}|F(i)\subset|\mathcal{E}|\}\in\Z_0$ and
let $\nu:\{1,2,\cdots,\ell\}\rightarrow\{1,2,\ldots,m\}$ be the unique injective mapping preserving the order and satisfying
$\nu(\{1,2,\cdots,\ell\})=\{i\in\{1,2,\ldots,m\}|F(i)\subset|\mathcal{E}|\}$.

The composition $F\nu$ is a center sequence of $\mathcal{E}$ of length $\ell$, and
$\mathcal{E}* F\nu(1)*F\nu(2)*\cdots*F\nu(\ell)
=(\mathcal{D}*F(1)*F(2)*\cdots*F(m))\backslash |\mathcal{E}|
\subset \mathcal{D}*F(1)*F(2)*\cdots*F(m)$.
\item
If $F(i)\subset |\mathcal{E}|$ for any $i\in\{1,2,\ldots,m\}$, then the sequence $F$ is a center sequence of $\mathcal{E}$ of length $m$, and
$\mathcal{E}* F(1)*F(2)*\cdots*F(m)
=(\mathcal{D}*F(1)*F(2)*\cdots*F(m))\backslash |\mathcal{E}|
\subset \mathcal{D}*F(1)*F(2)*\cdots*F(m)$.
\item
For any $n\in\Z_0$ and any center sequence $G$ of $\mathcal{E}$ of length $n$, the sequence $G$ is a center sequence of $\mathcal{D}$ of length $n$, and
$\mathcal{E}* G(1)*G(2)*\cdots*G(n)
=(\mathcal{D}*G(1)*G(2)*\cdots*G(n))\backslash |\mathcal{E}|
\subset \mathcal{D}*G(1)*G(2)*\cdots*G(n)$.
\end{enumerate}
\end{lemma}

\begin{example}
Assume $\dim V\geq 3$.
Consider any simplicial cone $S$ over $N$ in $V$ with $\dim S=3$.
Let $E(1), E(2), E(3)$ denote the three edges of $S$.
We denote $b(i)=b_{E(i)/N}\in E(i)^\circ\cap N$ for any $i\in\{1,2,3\}$ for simplicity.
$E(i)=\R_0b(i)$ for any $i\in\{1,2,3\}$.

Let
\begin{gather*}\begin{matrix}
F(1)=\R_0b(1)+\R_0b(3), & F(2)=\R_0b(2)+\R_0b(3), \\
G(1)=\R_0(b(1)+b(3))+\R_0b(2), & G(2)=\R_0(b(2)+b(3))+\R_0b(1).
\end{matrix}\end{gather*}

$\dim F(1)=\dim F(2)=\dim G(1)=\dim G(2)=2$.

We consider two mappings $H$ and $\bar{H}$ from $\{1,2,3\}$ to $2^V$ satisfying $(H(1),H(2),$\hfill\break$H(3))=(F(1), F(2), G(1))$ and $(\bar{H}(1),\bar{H}(2),\bar{H}(3))=(F(2), F(1), G(2))$.
Mappings $H$ and $\bar{H}$ are center sequences of the simplicial cone decomposition $\mathcal{F}(S)$ of length $3$.

$$\mathcal{F}(S)*F(1)*F(2)*G(1)=\mathcal{F}(S)*F(2)*F(1)*G(2).$$
\end{example}

\section{Simpleness and semisimpleness}
\label{simple}

Simpleness and semisimpleness are very important concepts.

Let $V$ be any finite dimensional vector space over $\R$ with $\dim V\geq 1$; let $N$ be any lattice of $V$; let $\mathcal{D}$ be any convex polyhedral cone decomposition in the dual vector space $V^*$ of $V$ such that the support $|\mathcal{D}|$ of $\mathcal{D}$ is a simplicial cone over the dual lattice $N^*$ of $N$ in $V^*$ and $\dim |\mathcal{D}|\geq 1$; let $H\in\mathcal{F}(|\mathcal{D}|)_1$ be any edge of the simplicial cone $|\mathcal{D}|$; let $S$ be any convex pseudo polyhedron in $V$ such that $|\mathcal{D}(S|V)|=\Stab(S)^\vee|V$ is a simplicial cone over $N^*$ in $V^*$ and $\dim |\mathcal{D}(S|V)|\geq 1$, and let $G\in\mathcal{F}(|\mathcal{D}(S|V)|)_1$ be any edge of the simplicial cone $|\mathcal{D}(S|V)|$.

We denote $L=\Stab(S)\cap(-\Stab(S))=\Vect(|\mathcal{D}(S|V)|)^\vee|V^*\subset V$ and $\ell=\dim L\in\Z_0$. $L$ is the maximum vector subspace over $\R$ in $V$ contained in $\Stab(S)$.

Note that for any $E\in\mathcal{F}(|\mathcal{D}(S|V)|)_1$, any $F\in\mathcal{F}(S)_\ell$, any $a\in F$ and any $b\in F$,
$E\subset|\mathcal{D}(S|V)|\subset\Vect(|\mathcal{D}(S|V)|)$, and we have $\langle b_{E/N^*},a\rangle=\langle b_{E/N^*},b\rangle$. 
(Proposition~\ref{property of polyhedrons}.7.)

\begin{definition}
\label{defsimple}
\begin{enumerate}
\item
We say that $\mathcal{D}$ is \emph{semisimple}, if $\dim\Delta\geq \dim \mathcal{D}-1$ for any  $\Delta\in\mathcal{D}$ satisfying $\Delta^\circ\subset|\mathcal{D}|^\circ$.

\item
We say that $\mathcal{D}$ is of $H$-\emph{Weierstrass type}, if $\mathcal{D}\backslash(H\Op||\mathcal{D}|)=\mathcal{F}(H\Op||\mathcal{D}|)$.

\item
We say that $\mathcal{D}$ is $H$-\emph{simple}, if $\mathcal{D}$ is semisimple and $\mathcal{D}$ is of $H$-Weierstrass type.

\item
We say that $S$ is \emph{semisimple}, if $\dim F\leq \ell+1$ for any face $F$ of $S$ satisfying $\Stab(F)=L$.

\item
We say that $S$ is of $G$-\emph{Weierstrass type}, if there exists only one face $F$ of $S$ satisfying $\Stab(F)=\Delta(G\Op||\mathcal{D}(S|V)|, |\mathcal{D}(S|V)||V^*)$.

\item
We say that $S$ is $G$-\emph{simple}, if $S$ is semisimple and $S$ is of $G$-Weierstrass type.

\item
Let $F\in\mathcal{F}(S)_\ell$ be any minimal face of $S$.

We say that $F$ is $G$-\emph{top}, if $\langle b_{G/N^*}, a\rangle=\max\{\langle b_{G/N^*}, c\rangle|c\in\mathcal{V}(S)\}$ for some $a\in F$.

We say that $F$ is $G$-\emph{bottom }, if $\langle b_{G/N^*}, a\rangle=\min\{\langle b_{G/N^*}, c\rangle|c\in\mathcal{V}(S)\}$ for some $a\in F$.
\item
We define
$$\Ht(G,S)= \max\{\langle b_{G/N^*}, c\rangle|c\in\mathcal{V}(S)\}- \min\{\langle b_{G/N^*}, c\rangle|c\in\mathcal{V}(S)\}\in\R_0,$$
and we call $\Ht(G,S)$ $G$-\emph{height} of $S$.
\end{enumerate}
\end{definition}

\begin{lemma}
\label{propsimple}
\begin{enumerate}
\item
$S$ is semisimple, if and only if, $\mathcal{D}(S|V)$ is semisimple.
\item
$S$ is of $G$-Weierstrass type, if and only if, $\mathcal{D}(S|V)$ is of $G$-Weierstrass type.
\item
Note that $L\subset\Vect(G\Op||\mathcal{D}(S|V)|)^\vee|V^*\subset V$ and
$\dim \Vect(G\Op||\mathcal{D}(S|V)|)^\vee$\hfill\break$|V^*=\ell+1$.
Let $W=V/\Vect(G\Op||\mathcal{D}(S|V)|)^\vee|V^*$ denote the residue vector space, and let $\rho:V\rightarrow W$ denote the canonical surjective homomorphism of vector spaces over $\R$ to $W$.

$\rho(N)$ is a lattice in $W$, $\rho(S)$ is a convex pseudo polyhedron in $W$, $\Stab(\rho(S))$ is a simplicial cone over $\rho(N)$ in $W$, and $\dim \Stab(\rho(S))=\dim W=\dim V-\ell-1$.

$S$ is of $G$-Weierstrass type $\Leftrightarrow c(\rho(S))=1\Leftrightarrow$ there exists $c\in S$ satisfying $\Ord(b_{E/N^*},S|V)=\langle b_{E/N^*},c\rangle$ for any $E\in\mathcal{F}(|\mathcal{D}(S|V)|)_1-\{G\}$.
\item
Assume that $S$ is of $G$-Weierstrass type.
$S$ has a unique $G$-top minimal face.
$c(S)=1$, if and only if, $\Ht(G,S)=0$.
\item
$S$ is $G$-simple, if and only if, $\mathcal{D}(S|V)$ is $G$-simple.
\item
If $\dim |\mathcal{D}|=1$, then  $\mathcal{D}=\mathcal{F}(|\mathcal{D}|)$.

If $\dim |\mathcal{D}|\leq 2$, then $\mathcal{D}$ is $H$-simple.

If $\dim|\mathcal{D}(S|V)|=1$, then $S=\{a\}+\Stab(S)$ for some $a\in S$.

If $\dim|\mathcal{D}(S|V)|\leq 2$, then $S$ is $G$-simple.
\item
Let $k$ be any field. Let $A$ be any complete regular local ring such that $\dim A\geq 1$, $A$ contains $k$ as a subring, and the residue field  $A/M(A)$ is isomorphic to $k$ as algebras over $k$. Let $P$ be any parameter system of $A$. Let $z\in P$ be any element. Let $\phi\in A$ be any non-zero element. We consider the Newton polyhedron $\Gamma_+(P,\phi)$ over the lattice $\Map(P,\Z)$ in the vector space $\Map(P,\R)$. (See Section~\ref{concept}.)
Let 
$G_z=\R_0f^{P\vee}_z \in\mathcal{F}(\Map(P,\R_0)^\vee|\Map(P,\R))_1$.

$\Gamma_+(P,\phi)$ is of $z$-Weierstrass type, if and only if, it is of $G_z$-Weierstrass type.

$\Gamma_+(P,\phi)$ is $z$-simple, if and only if, it is $G_z$-simple.

$\Ht(z,\Gamma_+(P, \phi))= \Ht(G_z,\Gamma_+(P,\phi))$.
\item
Let $F\in\mathcal{F}(S)_\ell$ be any mimimal face of $S$. 

$F$ is $G$-bottom $\Leftrightarrow F\subset\Delta(b_{G/N^*},S|V)\Leftrightarrow G\subset\Delta(F,S|V)$.

If $F$ is $G$-top, then $\dim(\Delta(F,S|V)\cap(G\Op||\mathcal{D}(S|V)|))=\dim |\mathcal{D}(S|V)|-1$ and $\Delta(F,S|V)\subset (\Delta(F,S|V)\cap(G\Op||\mathcal{D}(S|V)|))+G$.
\item
If $\mathcal{D}$ is semisimple, then $\sharp\mathcal{D}^0=\sharp\{\Delta\in\mathcal{D}^1|\Delta^\circ\subset|\mathcal{D}|^\circ\}+1$.
\item
Let $\Lambda$ be any simplicial cone over $N^*$ in $V^*$ satisfying $\Lambda\subset|\mathcal{D}|$ and $\dim\Lambda\geq 1$.
If $\mathcal{D}$ is semisimple, then $\mathcal{D}\hat{\cap}\mathcal{F}(\Lambda)$ is also semisimple.
\item
If $\mathcal{D}$ is semisimple, then $\mathcal{D}\backslash\Lambda$ is also semisimple for any $\Lambda\in\mathcal{F}(|\mathcal{D}|)$ with $\dim\Lambda\geq 1$.
\item
Assume that $\mathcal{D}$ is of $H$-Weierstrass type.
For any $\Delta\in\mathcal{D}^0$,\hfill\break $\dim(\Delta\cap(H\Op||\mathcal{D}|))=\dim |\mathcal{D}|-1$, if and only if, $\Delta\supset H\Op||\mathcal{D}|$, and there exists only one element $\Delta\in\mathcal{D}^0$ satisfying these equivalent conditions.
\item
If $\mathcal{D}$ is of $H$-Weierstrass type, then $\mathcal{D}\backslash\Lambda$ is also of $H$-Weierstrass type for any $\Lambda\in\mathcal{F}(|\mathcal{D}|)/H$.
\item
If $\mathcal{D}$ is $H$-simple, then $\mathcal{D}\backslash\Lambda$ is also $H$-simple for any $\Lambda\in\mathcal{F}(|\mathcal{D}|)/H$.
\item
Assume that $\mathcal{D}$ is $H$-simple.
We denote $\bar{\mathcal{D}}^1=\{\Delta\in\mathcal{D}^1|\Delta^\circ\subset|\mathcal{D}|^\circ\}\cup\{ H\Op||\mathcal{D}|\}$.
\begin{enumerate}
\item
$ H\Op||\mathcal{D}|\in \mathcal{D}^1$.
$\bar{\mathcal{D}}^1\subset\mathcal{D}^1$.
$\sharp\mathcal{D}^0=\sharp\bar{\mathcal{D}}^1$.
\item
We denote $\Delta\leq\Lambda$, if $\Delta+H\supset\Lambda+H$ for any $\Delta\in\mathcal{D}^0$ and any $\Lambda\in\mathcal{D}^0$.
Then, the relation $\leq$ is a total order on $\mathcal{D}^0$.
\item
We denote $\Delta\leq\Lambda$, if $\Delta+H\supset\Lambda+H$ for any $\Delta\in\bar{\mathcal{D}}^1$ and any $\Lambda\in\bar{\mathcal{D}}^1$.
Then, the relation $\leq$ is a total order on $\bar{\mathcal{D}}^1$.
\end{enumerate}

Let $r=\sharp\mathcal{D}^0=\sharp\bar{\mathcal{D}}^1\in\Z_+$.

We consider the total order on $\mathcal{D}^0$ described in $(b)$. Let $\Delta:\{1,2,\ldots,r\}\rightarrow \mathcal{D}^0$ denote the unique bijective mapping preserving the order. 

We consider the total order on $\bar{\mathcal{D}}^1$ described in $(c)$.
Let $\bar{\Delta}:\{1,2,\ldots,r\}\rightarrow \bar{\mathcal{D}}^1$ denote the unique bijective mapping preserving the order.
\begin{enumerate}
\setcounter{enumii}{3}
\item
Consider any $i\in\{1,2,\ldots,r\}$ and any $E\in\mathcal{F}(|\mathcal{D}|)_1-\{H\}$.
There exists a unique real number $c(\mathcal{D},i, E)\in\R$ depending on the pair $(i, E)$ satisfying $b_{E/N^*}+ c(\mathcal{D},i,E)b_{H/N^*}\in\Vect(\bar{\Delta}(i))$.
\end{enumerate}

Below we assume $c(\mathcal{D},i,E)\in\R$ and $b_{E/N^*}+ c(\mathcal{D},i,E)b_{H/N^*}\in\Vect(\bar{\Delta}(i))$ for any $i\in\{1,2,\ldots,r\}$ and any $E\in\mathcal{F}(|\mathcal{D}|)_1-\{H\}$.

\begin{enumerate}
\setcounter{enumii}{4}
\item
For any $i\in\{1,2,\ldots,r\}$,
$\bar{\Delta}(i)=\Convcone(\{b_{E/N^*}+ c(\mathcal{D},i,E)b_{H/N^*}|E\in\mathcal{F}(|\mathcal{D}|)_1-\{H\}\})$.
\item
For any $\Gamma\in\bar{\mathcal{D}}^1$, $\Gamma=\Vect(\Gamma)\cap|\mathcal{D}|$.
\item
For any $E\in\mathcal{F}(|\mathcal{D}|)_1-\{H\}$, $ c(\mathcal{D},1,E) =0$.
\item
If $r\geq 2$, then $c(\mathcal{D},i,E)\leq c(\mathcal{D},i+1,E)$ for any $i\in\{1,2,\ldots,r-1\}$ and any $E\in\mathcal{F}(|\mathcal{D}|)_1-\{H\}$.
\item
If $r\geq 2$ and $i\in\{1,2,\ldots,r-1\}$, then $c(\mathcal{D},i,E)< c(\mathcal{D},i+1,E)$ for some $E\in\mathcal{F}(|\mathcal{D}|)_1-\{H\}$.
\item
$\mathcal{D}$ is rational over $N^*$, if and only if, $c(\mathcal{D},i,E)\in\Q$
for any $i\in\{2,3,\ldots,r\}$ and any $E\in\mathcal{F}(|\mathcal{D}|)_1-\{H\}$.
\item
If $r\geq 2$, then $\Delta(i)= \bar{\Delta}(i)+ \bar{\Delta}(i+1)$ for any $i\in\{1,2,\ldots,r-1\}$.
\item
$\Delta(r)= \bar{\Delta}(r)+H$.
\item
$\{\Lambda\in\mathcal{D}|\Lambda\not\subset|\mathcal{F}(|\mathcal{D}|)/H|\}=
\mathcal{D}^0\cup\smash{\bar{\mathcal{D}}}^1$.

\item
$\bar{\Delta}(1)=H\Op||\mathcal{D}|\subset\partial|\mathcal{D}|$.
$\mathcal{D}^0/\bar{\Delta}(1)=\{\Delta(1)\}$.

For any $i\in\{2,3,\ldots,r\}$,
$\bar{\Delta}(i)^\circ\subset|\mathcal{D}|^\circ$,
$\bar{\Delta}(i)\not\subset\partial|\mathcal{D}|$, and
$\mathcal{D}^0/\bar{\Delta}(i)=\{\Delta(i-1), \Delta(i)\}$.
\item
$H\subset \Delta(r)$.
$\smash{\bar{\mathcal{D}}}^1\backslash\Delta(r)=\{\bar{\Delta}(r)\}$.

For any $i\in\{1,2,\ldots,r-1\}$,
$H\not\subset\Delta(i)$,
$\smash{\bar{\mathcal{D}}}^1\backslash\Delta(i)=\{\bar{\Delta}(i), \bar{\Delta}(i+1)\}$.
\item
If $r\geq 2$, then $\Delta(i)\cap\Delta(j)=\bar{\Delta}(i+1)\cap\bar{\Delta}(j)$ for any $i\in\{1,2,\ldots,$\hfill\break$r-1\}$ and any $j\in\{2,3,\ldots,r\}$ with $i<j$.
\item
Consider any $\omega\in\Vect(H\Op||\mathcal{D}|)$.

Take the unique function $\bar{\omega}:\mathcal{F}(H\Op||\mathcal{D}|)_1\rightarrow \R$ satisfying
$\omega=$\hfill\break$\sum_{E\in\mathcal{F}(H\Op||\mathcal{D}|)_1}\bar{\omega}(E)b_{E/N^*}$.
For any $i\in\{1,2,\ldots,r\}$, put $t(i)=$\hfill\break$ \sum_{E\in\mathcal{F}(H\Op||\mathcal{D}|)_1}\bar{\omega}(E)c(\mathcal{D},i, E)\in\R$.
\begin{enumerate}
\item
$t(1)=0$.
\item
For any $i\in\{1,2,\ldots,r\}$, the following claims hold:
\begin{enumerate}
\item
$(\{\omega\}+\R b_{H/N^*})\cap\Vect(\bar{\Delta}(i))=\{\omega+t(i)b_{H/N^*}\}$.
\item
$\omega+t(i)b_{H/N^*}\in\bar{\Delta}(i)\Leftrightarrow \omega\in H\Op||\mathcal{D}|$.
\item
$\omega+t(i)b_{H/N^*}\in\bar{\Delta}(i)^\circ\Leftrightarrow \omega\in (H\Op||\mathcal{D}|)^\circ$.
\end{enumerate}
\item
The following claims hold for any $i\in\{1,2,\ldots,r-1\}$, if $r\geq 2$.
\begin{enumerate}
\item
If $\omega\in H\Op||\mathcal{D}|$, then $t(i)\leq t(i+1)$.
\item
$(\{\omega\}+\R b_{H/N^*})\cap\Delta(i)$\hfill\break
\begin{equation*}
=
\begin{cases}
\{\omega+t b_{H/N^*}|t\in\R, t(i)\leq t\leq t(i+1)\}&\text{if $\omega\in H\Op||\mathcal{D}|$},\\
\emptyset&\text{if $\omega\not\in H\Op||\mathcal{D}|$}.
\end{cases}
\end{equation*}
\item
If $\omega\in (H\Op||\mathcal{D}|)^\circ$, then $t(i)< t(i+1)$.
\item
$(\{\omega\}+\R b_{H/N^*})\cap\Delta(i)^\circ$\hfill\break
\begin{equation*}
=
\begin{cases}
\{\omega+t b_{H/N^*}|t\in\R, t(i)< t< t(i+1)\}&\text{if $\omega\in (H\Op||\mathcal{D}|)^\circ$},\\
\emptyset&\text{if $\omega\not\in (H\Op||\mathcal{D}|)^\circ$}.
\end{cases}
\end{equation*}
\end{enumerate}
\item
$(\{\omega\}+\R b_{H/N^*})\cap\Delta(r)$\hfill\break
\begin{equation*}
=
\begin{cases}
\{\omega+t b_{H/N^*}|t\in\R, t(r)\leq t\}&\text{if $\omega\in H\Op||\mathcal{D}|$},\\
\emptyset&\text{if $\omega\not\in H\Op||\mathcal{D}|$}.
\end{cases}
\end{equation*}
\item
$(\{\omega\}+\R b_{H/N^*})\cap\Delta(r)^\circ$\hfill\break
\begin{equation*}
=
\begin{cases}
\{\omega+t b_{H/N^*}|t\in\R, t(r)< t\}&\text{if $\omega\in (H\Op||\mathcal{D}|)^\circ$},\\
\emptyset&\text{if $\omega\not\in (H\Op||\mathcal{D}|)^\circ$}.
\end{cases}\end{equation*}
\end{enumerate}
\item
Let $\pi:\Vect(|\mathcal{D}|)\rightarrow \Vect(H\Op||\mathcal{D}|)$ denote the unique surjective homomorphism of vector spaces over $\R$ satisfying $\pi^{-1}(0)=\Vect(H)$ and $\pi(x)=x$ for any $x\in\Vect(H\Op||\mathcal{D}|)$.

$\pi(|\mathcal{D}|)= H\Op||\mathcal{D}|$.
For any $\Delta\in\mathcal{D}$, $\pi(\Delta)\in\mathcal{F}(H\Op||\mathcal{D}|)$,
$\pi^{-1}(\pi(\Delta))\cap|\mathcal{D}|\in\mathcal{F}(|\mathcal{D}|)/H$.

For any $\Delta\in\mathcal{D}$ with $\Vect(H)\subset\Vect(\Delta)$, $\dim \pi(\Delta)=\dim\Delta-1$.

For any $\Delta\in\mathcal{D}$ with $\Vect(H)\not\subset\Vect(\Delta)$, $\dim \pi(\Delta)=\dim\Delta$.

For any $\Delta\in\mathcal{D}$ with $\Delta\not\subset H\Op||\mathcal{D}|$,
$\Delta^\circ\subset (\pi^{-1}(\pi(\Delta)) \cap|\mathcal{D}|)^\circ$.
\item
If $\Delta\in\mathcal{D}$, $\Lambda\in\mathcal{F}(|\mathcal{D}|)$, $\Delta^\circ\subset\Lambda^\circ$ and $\Delta\not\subset H\Op||\mathcal{D}|$, then $H\subset\Lambda$, and $\dim \Delta=\dim\Lambda$ or $\dim \Delta=\dim\Lambda-1$.

If $\Delta\in\mathcal{D}$, $\Lambda\in\mathcal{F}(|\mathcal{D}|)$, $\Delta^\circ\subset\Lambda^\circ$ and $\Delta\subset H\Op||\mathcal{D}|$, then $\Delta=\Lambda$ and $\dim\Delta=\dim\Lambda$.
\end{enumerate}
\item
Assume that $\mathcal{D}$ is $H$-simple.
Consider any $\Lambda\in\mathcal{F}(|\mathcal{D}|)/H$.

We use the same notations $\smash{\bar{\mathcal{D}}}^1$, $r$, $\Delta$ and $\bar{\Delta}$ as above.
We denote $\smash{\overline{\mathcal{D}\backslash\Lambda}}^1=\{\Gamma\in\mathcal{D}\backslash\Lambda|\Gamma^\circ\subset\Lambda^\circ\}\cup\{H\Op|\Lambda\}$.
\begin{enumerate}
\item
$r\in\{i\in\{1,2,\ldots,r\}|\dim(\Delta(i)\cap\Lambda)=\dim\Lambda\}\neq\emptyset$.
\end{enumerate}

Put $\bar{r}=\sharp\{i\in\{1,2,\ldots,r\}|\dim(\Delta(i)\cap\Lambda)=\dim\Lambda\}\in\Z_+$.
Let $\nu:\{1,2,\ldots,\bar{r}\}\rightarrow\{1,2,\ldots,r\}$ be the unique injective mapping preserving the order and satisfying $\nu(\{1,2,\ldots,\bar{r}\})=\{i\in\{1,2,\ldots,r\}|\dim(\Delta(i)\cap\Lambda)=\dim\Lambda\}$.
\begin{enumerate}
\setcounter{enumii}{1}
\item
$1\leq\bar{r}\leq r$.
$\nu(\bar{r})=r$.
\item
$\sharp(\mathcal{D}\backslash\Lambda)^0=\sharp\smash{\overline{\mathcal{D}\backslash\Lambda}}^1=\bar{r}$.
\item
$(\mathcal{D}\backslash\Lambda)^0=\{\Delta\nu(i)\cap\Lambda|i\in\{1,2,\ldots,\bar{r}\}\}$.

We consider the total order on $(\mathcal{D}\backslash\Lambda)^0$ described in $15.(b)$.

The bijective mapping $\{1,2,\ldots,\bar{r}\}\rightarrow(\mathcal{D}\backslash\Lambda)^0$ sending $ i\in\{1,2,\ldots,\bar{r}\}$ to $\Delta\nu(i)\cap\Lambda\in(\mathcal{D}\backslash\Lambda)^0$ preserves the order.
\item
$\smash{\overline{\mathcal{D}\backslash\Lambda}}^1=\{\bar{\Delta}\nu(i)\cap\Lambda|i\in\{1,2,\ldots,\bar{r}\}\}$.

We consider the total order on $\smash{\overline{\mathcal{D}\backslash\Lambda}}^1$ described in $15.(c)$.
The bijective mapping $\{1,2,\ldots,\bar{r}\}\rightarrow\smash{\overline{\mathcal{D}\backslash\Lambda}}^1$ sending $ i\in\{1,2,\ldots,\bar{r}\}$ to $\bar{\Delta}\nu(i)\cap\Lambda\in\smash{\overline{\mathcal{D}\backslash\Lambda}}^1$ preserves the order.
\item
For any $j\in\Z$ with $1\leq j<\nu(1)$, $\Delta(j)\cap\Lambda=\bar{\Delta}\nu(1)\cap\Lambda$.

For any $i\in\{2,3,\ldots, \bar{r}\}$ and any $j\in\Z$ with $\nu(i-1)<j<\nu(i)$,
$\Delta(j)\cap\Lambda=\bar{\Delta}\nu(i)\cap\Lambda$.
\item
For any $j\in\Z$ with $1\leq j\leq\nu(1)$, $\bar{\Delta}(j)\cap\Lambda=\bar{\Delta}\nu(1)\cap\Lambda$.

For any $i\in\{2,3,\ldots, \bar{r}\}$ and any $j\in\Z$ with $\nu(i-1)<j\leq\nu(i)$,
$\bar{\Delta}(j)\cap\Lambda=\bar{\Delta}\nu(i)\cap\Lambda$.
\end{enumerate}
\end{enumerate}

Consider any $A\in\mathcal{F}(S)_\ell$ and any $E\in \mathcal{F}(|\mathcal{D}(S|V)|)_1$.
For any $a\in A$, the real number $\langle b_{E/N^*}, a\rangle$ does not depend on the choice of $a\in A$ and it depends only on $A$ and $E$.
We take any $a\in A$ and we define $\langle b_{E/N^*},A\rangle=\langle b_{E/N^*},a\rangle$.

\begin{enumerate}
\setcounter{enumi}{16}
\item
$S$ is $G$-simple, if and only if, the following three conditions are satisfied:
\begin{enumerate}
\item
For any $A\in \mathcal{F}(S)_\ell$ and any $\bar{A}\in \mathcal{F}(S)_\ell$, $\langle b_{G/N^*}, A\rangle=\langle b_{G/N^*}, \bar{A}\rangle$, if and only if, $A=\bar{A}$.
\end{enumerate}

We assume that the first condition is satisfied. Let $r=c(S)$. Let $A:\{1, 2,\ldots, r\}\rightarrow \mathcal{F}(S)_\ell$ be the unique bijective mapping satisfying $\langle b_{G/N^*}, A(i-1)\rangle > \langle b_{G/N^*}, A(i)\rangle$ for any $i\in\{2,3,\ldots,r\}$, if $r\geq 2$. 
\begin{enumerate}
\setcounter{enumii}{1}
\item $\langle b_{E/N^*}, A(2)\rangle \geq \langle b_{E/N^*},A(1)\rangle$ for any  $E\in \mathcal{F}(|\mathcal{D}(S|V)|)_1-\{G\}$, if $r\geq 2$.
\item
$$\frac{\langle b_{E/N^*}, A(i)\rangle- \langle b_{E/N^*}, A(i-1)\rangle}{\langle b_{G/N^*}, A(i-1)\rangle- \langle b_{G/N^*}, A(i)\rangle} \leq 
\frac{\langle b_{E/N^*}, A(i+1)\rangle- \langle b_{E/N^*},A(i)\rangle}{\langle b_{G/N^*}, A(i)\rangle- \langle b_{G/N^*},A(i+1)\rangle},$$
for any $i\in\{2,3,\ldots,r-1\}$ and any  $E\in \mathcal{F}(|\mathcal{D}(S|V)|)_1-\{G\}$, if $r\geq 3$.
\end{enumerate}
\item
Assume that $S$ is $G$-simple.
Let $r=c(S)\in\Z_+$.
Let $A:\{1,2,\ldots,r\}\rightarrow \mathcal{F}(S)_\ell$ be the unique bijective mapping satisfying  $\langle b_{G/N^*}, A(i-1)\rangle > \langle b_{G/N^*}, A(i)\rangle$ for any $i\in\{2,3,\ldots,r\}$, if $r\geq 2$. 

We denote $\bar{\mathcal{D}}(S|V)^1=\{\Delta\in\mathcal{D}(S|V)^1|\Delta^\circ\subset|\mathcal{D}(S|V)|^\circ\}\cup\{G\Op||\mathcal{D}(S|V)|\}\subset\mathcal{D}(S|V)^1$, and $\Delta_G=\Delta(G\Op||\mathcal{D}(S|V)|, |\mathcal{D}(S|V)||V^*)\in\mathcal{F}(\Stab(S))_{\ell+1}$.
The following claims hold:
\begin{enumerate}
\item There exists $E\in \mathcal{F}(|\mathcal{D}(S|V)|)_1-\{G\}$ with $\langle b_{E/N^*}, A(2)\rangle > \langle b_{E/N^*},$\hfill\break$A(1)\rangle$, if $r\geq 2$.
\item For any $i\in\{2,3,\ldots,r-1\}$ there exists  $E\in \mathcal{F}(\Stab(S)^\vee|V)_1-\{G\}$ with
$$\frac{\langle b_{E/N^*}, A(i)\rangle-\langle b_{E/N^*}, A(i-1)\rangle}{\langle b_{G/N^*}, A(i-1)\rangle- \langle b_{G/N^*},A(i)\rangle} < 
\frac{\langle b_{E/N^*}, A(i+1)\rangle-\langle b_{E/N^*},A(i)\rangle}{\langle b_{G/N^*}, A(i)\rangle- \langle b_{G/N^*},A(i+1)\rangle},$$
if $r\geq 3$.
\item
$A(1)+\Delta_G\in\mathcal{F}(S)_ {\ell+1}$.
$\Stab(A(1)+\Delta_G)= \Delta_G$.
$\Delta(A(1)+\Delta_G, S|V)=G\Op||\mathcal{D}(S|V)|\in\mathcal{D}(S|V)^1$.

If $\hat{A}\in\mathcal{F}(S)_{\ell+1}$ and $\Stab(\hat{A})=\Delta_G$, then $\hat{A}=A(1)+\Delta_G$.
\item
$\Conv(A(i-1)\cup A(i))\in\mathcal{F}(S)_{\ell+1}$,
$\Stab(\Conv(A(i-1)\cup A(i)))=L$,
$\Delta(\Conv(A(i-1)\cup A(i)),S|V)\in\mathcal{D}(S|V)^1$, and
$\Delta(\Conv(A(i-1)\cup A(i)),$\hfill\break$S|V)^\circ\subset|\mathcal{D}(S|V)|^\circ$
for any $i\in\{2,3,\ldots,r\}$, if $r\geq 2$.

If $\hat{A}\in\mathcal{F}(S)_{\ell+1}$ and $\Stab(\hat{A})=L$, then $r\geq 2$ and $\hat{A}=\Conv(A(i-1)\cup A(i))$ for some $i\in\{2,3,\ldots,r\}$.
\item
$\bar{\mathcal{D}}(S|V)^1=\{\Delta(A(1)+\Delta_G, S|V)\}\cup\{\Delta(\Conv(A(i-1)\cup A(i)),S|V)| i\in\{2,3,\ldots,r\}\}$.

\item
We define a bijective mapping $\bar{\Delta}:\{1,2,\ldots,r\}\rightarrow\bar{\mathcal{D}}(S|V)^1$ by putting 
$\bar{\Delta}(1)= \Delta(A(1)+\Delta_G, S|V)$ and
$\bar{\Delta}(i)= \Delta(\Conv(A(i-1)\cup A(i)),S|V)$ for any $i\in\{2,3,\ldots,r\}$.
For any $i\in\{1,2,\ldots,r\}$ and any $E\in$\hfill\break$\mathcal{F}(|\mathcal{D}(S|V)|)_1-\{G\}$, we take a unique real number $c(\mathcal{D}(S|V),i,E)\in\R$ satisfying $b_{E/N^*}+ c(\mathcal{D}(S|V),i,E)b_{G/N^*}\in\Vect(\bar{\Delta}(i))$.

For any $i\in\{2,3,\ldots,r\}$, $\bar{\Delta}(i-1)+G\supset\bar{\Delta}(i)+G$, if $r\geq 2$.

If we define a total order described in $15.(c)$ on $\bar{\mathcal{D}}(S|V)^1$, the mapping $\bar{\Delta}$ preserves the order.

For any $i\in\{2,3,\ldots,r\}$ and any $E\in\mathcal{F}(|\mathcal{D}(S|V)|)_1-\{G\}$,
$$c(\mathcal{D}(S|V),i,E)= \frac{\langle b_{E/N^*}, A(i)\rangle- \langle b_{E/N^*}, A(i-1)\rangle}{\langle b_{G/N^*}, A(i-1)\rangle- \langle b_{G/N^*}, A(i)\rangle},$$
if $r\geq 2$.
\end{enumerate}
\end{enumerate}
\end{lemma}

\begin{definition}
\label{defsimple}
Assume that $\mathcal{D}$ is $H$-simple.
We denote $\bar{\mathcal{D}}^1=\{\Delta\in\mathcal{D}^1|\Delta^\circ\subset|\mathcal{D}|^\circ\}\cup\{ H\Op||\mathcal{D}|\}$.
\begin{enumerate}
\item
We call $\bar{\mathcal{D}}^1$ the $H$-\emph{skeleton} of $\mathcal{D}$.
\item
We call the total order on $\mathcal{D}^0$ described in Lemma~\ref{propsimple}$.15.(b)$ the $H$-\emph{order}.
\item
We call the total order on $\bar{\mathcal{D}}^1$ described in Lemma~\ref{propsimple}$.15.(c)$ the $H$-\emph{order}.
\item
We consider the $H$-order on $\bar{\mathcal{D}}^1$.
Let $r=\sharp\bar{\mathcal{D}}^1\in\Z_+$.
Let $\bar{\Delta}:\{1,2,\ldots,r\}\rightarrow \bar{\mathcal{D}}^1$ be the unique bijective mapping preserving the order.
Consider any $i\in\{1,2,\ldots, r\}$ and any $E\in\mathcal{F}(|\mathcal{D}|)_1-\{H\}$.
By Lemma~\ref{propsimple}$.15.(d)$ there exists uniquely a real number $c(\mathcal{D},i,E)\in\R$ depending on the pair $(i, E)$ satisfying $b_{E/N^*}+ c(\mathcal{D},i,E)b_{H/N^*}\in\Vect(\bar{\Delta}(i))$.

We call $c(\mathcal{D},i,E)$ the \emph{structure constant} of $\mathcal{D}$ corresponding to the pair $(i,E)$.
\end{enumerate}
\end{definition}

\section{Basic subdivisions}
\label{basic}
We define the concept of basic subdivisions.

Let $V$ be any finite dimensional vector space over $\R$ with $\dim V\geq 2$; let $N$ be any lattice of $V$; let $H$ be any one-dimensional simplicial cone over $N$ in $V$; let $\mathcal{C}$ be any simplicial cone decomposition over $N$ in $V$ satisfying 
$\dim \mathcal{C}=\dim \Vect(|\mathcal{C}|)\geq 2$, $\mathcal{C}\Mx=\mathcal{C}^0$, $H\in\mathcal{C}_1$ and $\mathcal{C}=(\mathcal{C}/H)\Fc$.

\begin{example}
Let $\Delta$ be any simplicial cone over $N$ in $V$ such that $\dim\Delta\geq 2$ and $H\in\mathcal{F}(\Delta)_1$.
Let $\mathcal{C}=\mathcal{F}(\Delta)$.
$\mathcal{C}$ is a simplicial cone decomposition over $N$ in $V$, and it satisfies 
$\dim \mathcal{C}=\dim \Vect(|\mathcal{C}|)=\dim\Delta\geq 2$, $\mathcal{C}\Mx=\mathcal{C}^0$, $H\in\mathcal{C}_1$ and $\mathcal{C}=(\mathcal{C}/H)\Fc$.
Furthermore, we know $\mathcal{C}-(\mathcal{C}/H)=\mathcal{F}(H\Op|\Delta)$.
\end{example}

\begin{lemma}
\label{property of C}
\begin{enumerate}
\item
$\{0\}\in\mathcal{C}-(\mathcal{C}/H)\subset\mathcal{C}$. 
$\mathcal{C}-(\mathcal{C}/H)$ is a simplicial cone decomposition over $N$ in $V$.
$(\mathcal{C}-(\mathcal{C}/H))_1\neq\emptyset$.
\item
For any $\Delta\in\mathcal{C}/H$, $H\in\mathcal{F}(\Delta)_1$ and $H\Op|\Delta\in\mathcal{C}-(\mathcal{C}/H)$.
For any $\bar{\Delta}\in\mathcal{C}-(\mathcal{C}/H)$, $\bar{\Delta}+H\in\mathcal{C}/H$.

The mapping from $\mathcal{C}/H$ to $\mathcal{C}-(\mathcal{C}/H)$ sending $\Delta\in\mathcal{C}/H$ to $H\Op|\Delta\in\mathcal{C}-(\mathcal{C}/H)$ and the mapping from $\mathcal{C}-(\mathcal{C}/H)$ to $\mathcal{C}/H$ sending $\bar{\Delta}\in\mathcal{C}-(\mathcal{C}/H)$ to $\bar{\Delta}+H\in\mathcal{C}/H$ are bijective mappings preserving the inclusion relation between $\mathcal{C}/H$ and $\mathcal{C}-(\mathcal{C}/H)$, and they are the inverse mappings of each other.

Furthermore, if $\Delta\in\mathcal{C}/H$ and $\bar{\Delta}\in\mathcal{C}-(\mathcal{C}/H)$ correspond to each other by them, then $\dim\Delta=\dim\bar{\Delta}+1$.
\item
$|\mathcal{C}|=|\mathcal{C}-(\mathcal{C}/H)|\cup|\mathcal{C}/H|^\circ$.
$|\mathcal{C}-(\mathcal{C}/H)|\cap|\mathcal{C}/H|^\circ=\emptyset$.
\item
$\mathcal{C}\Mx\subset\mathcal{C}/H$.
For any $\Delta\in\mathcal{C}\Mx$, $\Vect(\Delta)=\Vect(|\mathcal{C}|)$.
\end{enumerate}

Let $\pi_H:V\rightarrow V/\Vect(H)$ denote the canonical surjective homomorphism of vector spaces over $\R$ to the residue vector space $V/\Vect(H)$.
\begin{enumerate}
\setcounter{enumi}{4}
\item
$\pi_H(N)$ is a lattice of $V/\Vect(H)$.
\item
$\pi_{H*}\mathcal{C}=\pi_{H*}(\mathcal{C}-(\mathcal{C}/H))$.
$\pi_{H*}\mathcal{C}$ is a simplicial cone decomposition over $\pi_H(N)$ in $V/\Vect(H)$.
$\dim \pi_{H*}\mathcal{C}=\dim \Vect(|\pi_{H*}\mathcal{C}|)=\dim\mathcal{C}-1$.
$\Vect(|\pi_{H*}\mathcal{C}|)=\pi_H(\Vect(|\mathcal{C}|))$.
$(\pi_{H*}\mathcal{C})\Mx=(\pi_{H*}\mathcal{C})^0$.
\item
For any $\bar{\Delta}\in\mathcal{C}-(\mathcal{C}/H)$,
$\pi_H(\bar{\Delta})\in \pi_{H*}\mathcal{C}$.

For any $\hat{\Delta}\in\pi_{H*}\mathcal{C}$, $\pi_H^{-1}(\hat{\Delta})\cap |\mathcal{C}-(\mathcal{C}/H)|\in\mathcal{C}-(\mathcal{C}/H)$.

The mapping from $\mathcal{C}-(\mathcal{C}/H)$ to $\pi_{H*}\mathcal{C}$ sending $\bar{\Delta}\in\mathcal{C}-(\mathcal{C}/H)$ to $\pi_H(\bar{\Delta})\in \pi_{H*}\mathcal{C}$ and the mapping from $\pi_{H*}\mathcal{C}$ to $\mathcal{C}-(\mathcal{C}/H)$ sending $\hat{\Delta}\in\pi_{H*}\mathcal{C}$ to $\pi_H^{-1}(\hat{\Delta})\cap |\mathcal{C}-(\mathcal{C}/H)|\in\mathcal{C}-(\mathcal{C}/H)$ are bijective mappings preserving the inclusion relation and the dimension between $\mathcal{C}-(\mathcal{C}/H)$ and $\pi_{H*}\mathcal{C}$, and they are the inverse mappings of each other.
\item
$\pi_H(|\mathcal{C}|)=\pi_H(|\mathcal{C}-(\mathcal{C}/H)|)=|\pi_{H*}\mathcal{C}|$.
The mapping $\pi_H: |\mathcal{C}-(\mathcal{C}/H)|\rightarrow
|\pi_{H*}\mathcal{C}|$ induced by $\pi_H$ is a continuous bijective mapping whose inverse mapping is also continuous.
\end{enumerate}
\end{lemma}

In addition, we consider any non-negative integer $m\in\Z_0$ and any mapping
$$E:\{1,2,\ldots,m\}\rightarrow (\mathcal{C}-(\mathcal{C}/H))_1.$$
For any $i\in\{1,2,\ldots,m\}$, $E(i)\in(\mathcal{C}-(\mathcal{C}/H))_1$.

For any $i\in\{0,1,\ldots,m\}$ and any $\bar{E}\in(\mathcal{C}-(\mathcal{C}/H))_1$, putting
$$s(i, \bar{E})=\sharp(\{1,2,\ldots,i\}\cap E^{-1}(\bar{E}))\in\Z_0,$$
we define a mapping
$$s:\{0,1,\ldots,m\}\times(\mathcal{C}-(\mathcal{C}/H))_1\rightarrow \Z_0.$$
The mapping $s$ is uniquely determined depending on the mapping $E$.

For any $i\in\{1,2,\ldots,m\}$, we put
\begin{equation*}\begin{split}
F(i)=&\R_0(b_{E(i)/N}+s(i-1, E(i))b_{H/N})+H\subset V,\\
G(i)=&\R_0(b_{E(i)/N}+s(i-1, E(i))b_{H/N})\subset V,\\
H(i)=&\R_0(b_{E(i)/N}+s(i, E(i))b_{H/N})\subset V.
\end{split}\end{equation*}
We put
$$H(m+1)=H\subset V.$$

Three mappings are defined.
$$F, G:\{1,2,\ldots,m\}\rightarrow 2^V.$$
$$H:\{1,2,\ldots,m, m+1\}\rightarrow 2^V.$$
They are uniquely determined depending on the mapping $E$.
\begin{lemma}
\label{basic1}
\begin{enumerate}
\item
For any $\bar{E}\in(\mathcal{C}-(\mathcal{C}/H))_1$, $s(0,\bar{E})=0$.

For any $i\in\{1,2,\ldots,m\}$ and any $\bar{E}\in(\mathcal{C}-(\mathcal{C}/H))_1$,
\begin{equation*}
s(i, \bar{E})=
\begin{cases}
s(i-1,\bar{E})&\text{if $\bar{E}\neq E(i)$},\\
s(i-1,\bar{E})+1&\text{if $\bar{E}=E(i)$}.
\end{cases}\end{equation*}
\item
For any $i\in\{1,2,\ldots,m\}$, $E(i), F(i), G(i), H(i), G(i)+H(i)$ and $E(i)+H$ are simplicial cones over $N$ in $V$,
$\dim E(i)=\dim G(i)=\dim H(i)=1$, $\dim F(i)=\dim(G(i)+H(i))=\dim(E(i)+H)=2$, and
$G(i)+H(i)\subset F(i)\subset E(i)+H\in\mathcal{C}_2/H$.
\item
For any $i\in\{1,2,\ldots,m\}$, $\mathcal{F}(G(i)+H(i))_1=\{G(i), H(i)\}$,
$\mathcal{F}(F(i))_1=\{G(i), H\}$,
$\mathcal{F}(E(i)+H)_1=\{E(i), H\}$,
$b_{G(i)/N}=b_{E(i)/N}+s(i-1,E)b_{H/N}$,
$b_{H(i)/N}=b_{E(i)/N}+s(i,E)b_{H/N}= b_{G(i)/N}+ b_{H/N}=b_{F(i)/N}$, and
$H(i)=\R_0 b_{F(i)/N}\subset F(i)$.
\item
$F$ is a center sequence of $\mathcal{C}$ of length $m$ such that $\dim F(i)=2$
and $F(i)\not\subset|\mathcal{C}-(\mathcal{C}/H)|$
for any $i\in\{1,2,\ldots,m\}$, and it is determined by the sextuplet $(V, N, H, \mathcal{C}, m, E)$ uniquely.

$\mathcal{C}*F(1)*F(2)*\cdots*F(m)$ is an iterated barycentric subdivision of $\mathcal{C}$, and it is a simplicial cone decomposition, and it is determined by the sextuplet $(V, N, H, \mathcal{C}, m, E)$ uniquely.
\end{enumerate}
\end{lemma}

Below, for simplicity we denote
$$\mathcal{B}= \mathcal{C}*F(1)*F(2)*\cdots*F(m)\subset 2^V.$$

For any $i\in\{1,2,\ldots,m\}$, we put
$$\mathcal{B}(i)=(\mathcal{B}/(G(i)+H(i)))\Fc\subset \mathcal{B}\subset 2^V.$$

For $m+1$, we put
$$\mathcal{B}(m+1)=(\mathcal{B}/H(m+1))\Fc\subset \mathcal{B}\subset 2^V.$$

We obtain a mapping
$$\mathcal{B}:\{1,2,\ldots,m+1\}\rightarrow 2^{2^V}.$$

\begin{lemma}
\label{property of basic subdivisions}
\begin{enumerate}
\item
$\mathcal{B}$ is a simplicial cone decomposition over $N$ in $V$.
$\dim \mathcal{B}=\dim\Vect(|\mathcal{B}|)=\dim \mathcal{C}$. 
$\Vect(|\mathcal{B}|)=\Vect(|\mathcal{C}|)$.
$\mathcal{B}\Mx=\mathcal{B}^0$.
$\mathcal{B}$ is an iterated barycentric subdivision of $\mathcal{C}$.
$|\mathcal{B}|=|\mathcal{C}|$.
\item
For any $i\in\{1,2,\ldots,m\}$, $G(i)\in\mathcal{B}_1$, 
$G(i)+H(i)\in\mathcal{B}_2$, and $F(i)^\circ\subset|\mathcal{C}/H|^\circ$.
For any $i\in\{1,2,\ldots,m+1\}$, $H(i)\in\mathcal{B}_1$.
\item
For any $\Delta\in\mathcal{C}/H$, $|\mathcal{B}\backslash\Delta|=\Delta$ and $\mathcal{B}\backslash\Delta$ is $H$-simple.
\item
$\mathcal{B}\backslash|\mathcal{C}-(\mathcal{C}/H)|= \mathcal{C}-(\mathcal{C}/H)$.
\item
For any $i\in\{1,2,\ldots,m+1\}$, $\mathcal{B}(i)$ is a simplicial cone decomposition over $N$ in $V$,
$\dim \mathcal{B}(i)=\dim \Vect(|\mathcal{B}(i)|)=\dim\mathcal{C}$, 
$\Vect(|\mathcal{B}(i)|)=\Vect(|\mathcal{C}|)$,
$\mathcal{B}(i)\Mx=\mathcal{B}(i)^0$, $H(i)\in\mathcal{B}(i)_1$, and $\mathcal{B}(i)=(\mathcal{B}(i)/H(i))\Fc$.
\item
For any $i\in\{1,2,\ldots,m\}$,
$G(i)+H(i)=|\mathcal{B}(i)|\cap(E(i)+H)\in\mathcal{B}(i)_2$, 
$G(i)=| \mathcal{B}(i)-(\mathcal{B}(i)/H(i))|\cap(G(i)+H(i))\in(\mathcal{B}(i)-(\mathcal{B}(i)/H(i)))_1\subset\mathcal{B}(i)_1$,
$\mathcal{B}(i)=(\mathcal{B}(i)/(G(i)+H(i)))\Fc=(\mathcal{B}(i)/G(i))\Fc$, and
$H(i)=| \mathcal{B}(i)-(\mathcal{B}(i)/G(i))|\cap(G(i)+H(i))\in(\mathcal{B}(i)-(\mathcal{B}(i)/G(i)))_1\subset\mathcal{B}(i)_1$
\item
For any $i\in\{1,2,\ldots, m\}$ any $j\in\{2,3,\ldots, m+1\}$ with $i<j$,
$\mathcal{B}(i)\cap\mathcal{B}(j)=(\mathcal{B}(i)-(\mathcal{B}(i)/G(i)))\cap(\mathcal{B}(j)-(\mathcal{B}(j)/H(j)))$.
\item
$$|\mathcal{C}-(\mathcal{C}/H)|\cup(\bigcup_{i\in\{1,2,\ldots,m+1\}}|\mathcal{B}(i)/H(i)|^\circ)=|\mathcal{C}|$$
\item
For any $i\in\{0,1,\ldots,m+1\}$, $|\mathcal{C}-(\mathcal{C}/H)|\cap|\mathcal{B}(i)/H(i)|^\circ=\emptyset$.

For any $i\in\{0,1,\ldots,m+1\}$ and any $j\in\{0,1,\ldots,m+1\}$ with $i\neq j$,
$|\mathcal{B}(i)/H(i)|^\circ\cap|\mathcal{B}(j)/H(j)|^\circ=\emptyset$.
\item
$$(\mathcal{C}-(\mathcal{C}/H))\cup(\bigcup_{i\in\{1,2,\ldots,m+1\}}(\mathcal{B}(i)/H(i)))=\mathcal{B}$$
\item
For any $i\in\{0,1,\ldots,m+1\}$, $(\mathcal{C}-(\mathcal{C}/H))\cap(\mathcal{B}(i)/H(i))=\emptyset$.

For any $i\in\{0,1,\ldots,m+1\}$ and any $j\in\{0,1,\ldots,m+1\}$ with $i\neq j$,
$(\mathcal{B}(i)/H(i))\cap(\mathcal{B}(j)/H(j))=\emptyset$.
\item
$$\bigcup_{i\in\{1,2,\ldots,m+1\}}\mathcal{B}(i)\Mx=\mathcal{B}\Mx$$
\item
For any $i\in\{0,1,\ldots,m+1\}$ and any $j\in\{0,1,\ldots,m+1\}$ with $i\neq j$,
$\mathcal{B}(i)\Mx\cap\mathcal{B}(j)\Mx=\emptyset$.
\end{enumerate}

For any $i\in\{0,1,\ldots,m+1\}$, we denote
$$X(i)=|\mathcal{C}-(\mathcal{C}/H)|\cup(\bigcup_{j\in\{1,2,\ldots,i\}}|\mathcal{B}(j)/H(j)|^\circ)\subset|\mathcal{C}|.$$
\begin{enumerate}
\setcounter{enumi}{13}
\item
$X(0)= |\mathcal{C}-(\mathcal{C}/H)|$. $X(m+1)=|\mathcal{C}|$.

For any $i\in\{1,2,\ldots, m+1\}$, $X(i-1)\subset X(i)$, $X(i-1)\neq X(i)$, and $|\mathcal{B}(i)|\cap X(i-1)=|\mathcal{B}(i)-(\mathcal{B}(i)/H(i))|$.
\item
For any $i\in\{0,1,\ldots,m+1\}$,
$$X(i)= |\mathcal{C}-(\mathcal{C}/H)|\cup(\bigcup_{j\in\{1,2,\ldots,i\}}|\mathcal{B}(j)|),$$
$$|\mathcal{B}\backslash X(i)|=X(i),$$

and $X(i)$ is a closed subset of $|\mathcal{C}|$.
\end{enumerate}
\end{lemma}

For any $i\in\{1,2,\ldots,m\}$, we put
$$\mathcal{B}^\circ(i)=\mathcal{B}(i)/G(i) \subset \mathcal{B}(i)\subset 2^V.$$

For $m+1$, we put
$$\mathcal{B}^\circ(m+1)=\mathcal{B}(m+1) \subset 2^V.$$

We obtain a mapping
$$\mathcal{B}^\circ:\{1,2,\ldots,m+1\}\rightarrow 2^{2^V}.$$

\begin{lemma}
\label{property of basic subdivisions2}
\begin{enumerate}
\item
For any $i\in\{1,2,\ldots,m+1\}$, the following claims hold:
\begin{enumerate}
\item
$\mathcal{B}(i)\Mx\subset\mathcal{B}^\circ(i)\subset\mathcal{B}(i)$.
$(\mathcal{B}(i)\Mx)\Fc=\mathcal{B}^\circ(i)\Fc=\mathcal{B}(i)$.
$|\mathcal{B}(i)\Mx|=|\mathcal{B}^\circ(i)|=|\mathcal{B}(i)|\supset|\mathcal{B}^\circ(i)|^\circ$.
\item
$\mathcal{B}^\circ(i)=\mathcal{B}(i)\Leftrightarrow|\mathcal{B}^\circ(i)|^\circ=|\mathcal{B}(i)|\Leftrightarrow i=m+1$.
\item
If $i\neq m+1$, then
$\mathcal{B}(i)=\mathcal{B}^\circ(i)\cup(\mathcal{B}(i)-(\mathcal{B}(i)/G(i)))$, $\mathcal{B}^\circ(i)\cap(\mathcal{B}(i)-(\mathcal{B}(i)/G(i)))=\emptyset$,
$|\mathcal{B}(i)|=|\mathcal{B}^\circ(i)|^\circ\cup|\mathcal{B}(i)-(\mathcal{B}(i)/G(i))|$, and $|\mathcal{B}^\circ(i)|^\circ\cap|\mathcal{B}(i)-(\mathcal{B}(i)/G(i))|=\emptyset$.
\item
For any $\Theta\in \mathcal{B}^\circ(i)/H(i)$, $H(i)\Op|\Theta\in\mathcal{B}^\circ(i)-(\mathcal{B}^\circ(i)/H(i))$.
For any $\Lambda\in\mathcal{B}^\circ(i)-(\mathcal{B}^\circ(i)/H(i))$, $\Lambda+H(i)\in \mathcal{B}^\circ(i)/H(i)$.

The mapping from $\mathcal{B}^\circ(i)/H(i)$ to $\mathcal{B}^\circ(i)-(\mathcal{B}^\circ(i)/H(i))$ sending\hfill\break$\Theta\in \mathcal{B}^\circ(i)/H(i)$ to $H(i)\Op|\Theta\in\mathcal{B}^\circ(i)-(\mathcal{B}^\circ(i)/H(i))$ and the mapping from $\mathcal{B}^\circ(i)-(\mathcal{B}^\circ(i)/H(i))$ to $\mathcal{B}^\circ(i)/H(i)$ sending $\Lambda\in\mathcal{B}^\circ(i)-(\mathcal{B}^\circ(i)/H(i))$ to $\Lambda+H(i)\in \mathcal{B}^\circ(i)/H(i)$ are bijective mapping preserving the inclusion relation between $\mathcal{B}^\circ(i)/H(i)$ and $\mathcal{B}^\circ(i)-(\mathcal{B}^\circ(i)/H(i))$, and they are the inverse mappings of each other.

Furthermore, if $\Theta\in \mathcal{B}^\circ(i)/H(i)$ and $\Lambda\in\mathcal{B}^\circ(i)-(\mathcal{B}^\circ(i)/H(i))$ correspond to each other by them, then $\dim\Theta=\dim\Lambda+1$.
\end{enumerate}
\item
$$\bigcup_{i\in\{1,2,\ldots,m+1\}}|\mathcal{B}^\circ(i)|^\circ=|\mathcal{C}|$$
\item
For any $i\in\{0,1,\ldots,m+1\}$ and any $j\in\{0,1,\ldots,m+1\}$ with $i\neq j$,$|\mathcal{B}^\circ(i) |^\circ\cap|\mathcal{B}^\circ(j)|^\circ=\emptyset$.
\item
$$\bigcup_{i\in\{1,2,\ldots,m+1\}}\mathcal{B}^\circ(i)=\mathcal{B}$$
\item
For any $i\in\{0,1,\ldots,m+1\}$ and any $j\in\{0,1,\ldots,m+1\}$ with $i\neq j$,
$\mathcal{B}^\circ(i)\cap\mathcal{B}^\circ(j)=\emptyset$.
\end{enumerate}

For any $i\in\{0,1,\ldots,m+1\}$, we denote
$$Y(i)= \bigcup_{j\in\{i+1,i+2,\ldots,m+1\}}|\mathcal{B}^\circ(j)|^\circ\subset|\mathcal{C}|.$$
\begin{enumerate}
\setcounter{enumi}{5}
\item
$Y(0)= |\mathcal{C}|$. $Y(m+1)=\emptyset$.

For any $i\in\{1,2,\ldots, m+1\}$, $Y(i-1)\supset Y(i)$, $Y(i-1)\neq Y(i)$. For any $i\in\{1,2,\ldots, m\}$, $|\mathcal{B}(i)|\cap Y(i)=|\mathcal{B}(i)-(\mathcal{B}(i)/G(i))|$.
\item
For any $i\in\{0,1,\ldots,m+1\}$,
$$Y(i)= \bigcup_{j\in\{i+1,i+2,\ldots,m+1\}}|\mathcal{B}(j)|,$$
$$|\mathcal{B}\backslash Y(i)|=Y(i),$$

and $Y(i)$ is a closed subset of $|\mathcal{C}|$.

\end{enumerate}
\end{lemma}

\begin{definition}
\label{defbasic}
We denote $\mathcal{B}=\mathcal{E}*F(1)*F(2)*\cdots*F(m)$ above by the symbol
$$\mathcal{B}(V, N, H, \mathcal{C}, m, E)\subset 2^V,$$
and we call it the \emph{basic subdivision} associated with the sextuplet $(V, N, H, \mathcal{C}, m, E)$, because $\mathcal{B}$ is uniquely determined depending on the sextuplet $(V, N, H, \mathcal{C}, m, E)$.

$\mathcal{B}(V, N, H, \mathcal{C}, m, E)$ is an iterated barycentric subdivision of $\mathcal{C}$, it is a simplicial cone decomposition over $N$ in $V$, and for any $\Delta\in\mathcal{C}/H$, $\mathcal{B}(V, N, H, \mathcal{C}, m, E) \backslash\Delta$ is $H$-simple.
$|\mathcal{B}(V, N, H, \mathcal{C}, m, E)|=|\mathcal{C}|$.

Note that depending on the sextuplet $(V, N, H, \mathcal{C}, m, E)$, six mappings
$$s: \{0,1,\ldots,m\}\times(\mathcal{C}-(\mathcal{C}/H))_1\rightarrow \Z_0,$$
$$F, G:\{1,2,\ldots,m\}\rightarrow 2^V,$$
$$H:\{1,2,\ldots,m, m+1\}\rightarrow 2^V,$$
$$\mathcal{B}:\{1,2,\ldots,m, m+1\}\rightarrow 2^{2^V},$$
$$\mathcal{B}^\circ:\{1,2,\ldots,m, m+1\}\rightarrow 2^{2^V},$$
are defined. 

We denote these six mappings by $s(V, N, H, \mathcal{C}, m, E)$, $F(V, N, H, \mathcal{C}, m, E)$, \hfill\break$G(V, N, H, \mathcal{C}, m, E)$, $H(V, N, H, \mathcal{C}, m, E)$, $\mathcal{B}(V, N, H, \mathcal{C}, m, E)$ and\hfill\break
$\mathcal{B}^\circ(V, N, H, \mathcal{C}, m, E)$ respectively, and we express the dependence explicitly.
\end{definition}

\begin{remark}
We denote two different objects by the same symbol $\mathcal{B}(V, N, H, \mathcal{C}, m, E)$.
One satisfies $\mathcal{B}(V, N, H, \mathcal{C}, m, E)\in 2^{2^V}$ and the other $\mathcal{B}(V, N, H, \mathcal{C}, m, E)$ is a mapping from $\{1,2,\ldots,m+1\}$ to $2^{2^V}$.
It is easy to distinguish them.
\end{remark}

\begin{lemma}
\label{basic2}
Consider any subset $\hat{\mathcal{C}}$ of $\mathcal{C}$ satisfying $\dim \hat{\mathcal{C}}=\dim \Vect(|\hat{\mathcal{C}}|)\geq 2$, $\hat{\mathcal{C}}\Mx=\hat{\mathcal{C}}^0$, $H\in\hat{\mathcal{C}}_1$ and $\hat{\mathcal{C}}=(\hat{\mathcal{C}}/H)\Fc$.

We know that $\emptyset\neq\hat{\mathcal{C}}=\hat{\mathcal{C}}\Fc$ and $\hat{\mathcal{C}}$ is a simplicial cone decomposition over $N$ in $V$.

Note that $\emptyset\neq(\hat{\mathcal{C}}-(\hat{\mathcal{C}}/H))_1\subset(\mathcal{C}-(\mathcal{C}/H))_1$ and $E^{-1}((\hat{\mathcal{C}}-(\hat{\mathcal{C}}/H))_1)\subset\{1,2,\ldots,m\}$.

Let $\hat{m}=\sharp E^{-1}((\hat{\mathcal{C}}-(\hat{\mathcal{C}}/H))_1)\in\Z_0$.
Let $\hat{\tau}:\{1,2,\ldots,\hat{m}\}\rightarrow\{1,2,\ldots,m\}$ be the unique injective mapping preserving the order and satisfying $\hat{\tau}(\{1,2,\ldots,\hat{m}\})= E^{-1}((\hat{\mathcal{C}}-(\hat{\mathcal{C}}/H))_1)$.
Putting $\hat{\tau}(0)=0$ and  $\hat{\tau}(\hat{m}+1)=m+1$, we define an extension
$\hat{\tau}:\{0,1,2,\ldots,\hat{m},\hat{m}+1\}\rightarrow\{0,1,2,\ldots,m,m+1\}$ of $\hat{\tau}:\{1,2,\ldots,\hat{m}\}\rightarrow\{1,2,\ldots,m\}$.
Let $\hat{E}:\{1,2,\ldots,\hat{m}\}\rightarrow(\hat{\mathcal{C}}-(\hat{\mathcal{C}}/H))_1$ be the unique mapping satisfying $\iota\hat{E}=E\hat{\tau}$, where $\iota: (\hat{\mathcal{C}}-(\hat{\mathcal{C}}/H))_1\rightarrow(\mathcal{C}-(\mathcal{C}/H))_1$ denotes the inclusion mapping.

\begin{enumerate}
\item
$\mathcal{B}(V, N, H, \mathcal{C}, m, E)\backslash|\hat{\mathcal{C}}|=
\mathcal{B}(V, N, H, \hat{\mathcal{C}}, \hat{m}, \hat{E})$.
\item
Let $s=s(V, N, H, \mathcal{C}, m, E)$ and $\hat{s}=s(V, N H, \hat{\mathcal{C}}, \hat{m}, \hat{E})$.

For any $\bar{E}\in(\hat{\mathcal{C}}-(\hat{\mathcal{C}}/H))_1$, any $i\in\{0,1,\ldots,\hat{m}\}$ and any $j\in\{0,1,\ldots,m\}$ with $\hat{\tau}(i)\leq j<\hat{\tau}(i+1)$,
$\hat{s}(i, \bar{E})=s(j,\bar{E})$, $\hat{s}(i, \bar{E})=s(\hat{\tau}(i), \bar{E})$, and $\hat{s}(i, \bar{E})=s(\hat{\tau}(i+1)-1, \bar{E})$.
\item
Let $F=F(V, N, H, \mathcal{C}, m, E)$, $G=G(V, N, H, \mathcal{C}, m, E)$,\hfill\break $\hat{F}=F(V, N, H, \hat{\mathcal{C}}, \hat{m}, \hat{E})$, and
$\hat{G}=G(V, N, H, \hat{\mathcal{C}}, \hat{m}, \hat{E})$.

$\hat{F}=F\hat{\tau}$, and $\hat{G}=G\hat{\tau}$.
\item
Let $H=H(V, N, H, \mathcal{C}, m, E)$, $\mathcal{B}=\mathcal{B}(V, N, H, \mathcal{C}, m, E)$,\hfill\break$\mathcal{B}^\circ=\mathcal{B}^\circ(V, N, H, \mathcal{C}, m, E)$, $\hat{H}=H(V, N, H, \hat{\mathcal{C}}, \hat{m}, \hat{E})$, 
$\hat{\mathcal{B}}=\mathcal{B}(V, N, H, \hat{\mathcal{C}}, \hat{m}, \hat{E})$, and $\smash{\hat{\mathcal{B}}}^\circ=\mathcal{B}^\circ(V, N, H, \hat{\mathcal{C}}, \hat{m}, \hat{E})$.
\begin{enumerate}
\item
$\hat{H}=H\hat{\tau}$.
\item
$\hat{\mathcal{B}}(\hat{m}+1) =\mathcal{B}\hat{\tau} (\hat{m}+1)\backslash|\hat{\mathcal{C}}|=\mathcal{B}(m+1)\backslash|\hat{\mathcal{C}}|$.
\item
For any $i\in\{1,2,\ldots,\hat{m}\}$,
\begin{equation*}\begin{split}
&\hat{\mathcal{B}}(i)\subset\mathcal{B}\hat{\tau}(i)\backslash|\hat{\mathcal{C}}|, \text{ and}\\
&(\mathcal{B}\hat{\tau}(i)\backslash|\hat{\mathcal{C}}|)-\hat{\mathcal{B}}(i) \\
&\quad\subset(\mathcal{B}\hat{\tau}(i)-(\mathcal{B}\hat{\tau}(i)/G\hat{\tau}(i))\cap(\mathcal{B}\hat{\tau}(i)-(\mathcal{B}\hat{\tau}(i)/H\hat{\tau}(i)).
\end{split}\end{equation*}
\item
For any $j\in\{1,2,\ldots, m\}-\hat{\tau}(\{1,2,\ldots,\hat{m}\})$,
$$\mathcal{B}(j)\backslash|\hat{\mathcal{C}}|\subset(\mathcal{B}(j)-(\mathcal{B}(j)/G(j))\cap(\mathcal{B}(j))-(\mathcal{B}(j)/H(j)).$$
\item
For any $i\in\{1,2,\ldots,\hat{m}, \hat{m}+1\}$, $\smash{\hat{\mathcal{B}}}^\circ(i)=\mathcal{B}^\circ\hat{\tau}(i)\backslash|\hat{\mathcal{C}}|$, and $|\smash{\hat{\mathcal{B}}}^\circ(i)|^\circ=|\mathcal{B}^\circ\hat{\tau}(i)|^\circ\cap|\hat{\mathcal{C}}|$.
\end{enumerate}
\item
For any $j\in\{1,2,\ldots,m\}$, $j\in\hat{\tau}(\{1,2,\ldots,\hat{m}\})\Leftrightarrow E(j)\in(\hat{\mathcal{C}}-(\hat{\mathcal{C}}/H))_1\Leftrightarrow
F(j)\subset|\hat{\mathcal{C}}|\Leftrightarrow
G(j)\subset|\hat{\mathcal{C}}|\Leftrightarrow
H(j)\subset|\hat{\mathcal{C}}|\Leftrightarrow
G(j)+H(j)\subset|\hat{\mathcal{C}}|$.

$\{\hat{F}(i)|i\in\{1,2,\ldots,\hat{m}\}\}=\{F\hat{\tau}(i)|i\in\{1,2,\ldots,\hat{m}\}\}=$\hfill\break$\{F(j)|j\in\{1,2,\ldots,m\}, F(j)\subset|\hat{\mathcal{C}}|\}$.

$\{\hat{G}(i)|i\in\{1,2,\ldots,\hat{m}\}\}=\{G\hat{\tau}(i)|i\in\{1,2,\ldots,\hat{m}\}\}=$\hfill\break$\{G(j)|j\in\{1,2,\ldots,m\}, G(j)\subset|\hat{\mathcal{C}}|\}$.

$\{\hat{H}(i)|i\in\{1,2,\ldots,\hat{m}, \hat{m}+1\}\}=\{H\hat{\tau}(i)|i\in\{1,2,\ldots,\hat{m}, \hat{m}+1\}\}=\{H(j)|j\in\{1,2,\ldots,m, m+1\}, H(j)\subset|\hat{\mathcal{C}}|\}$.
\item
For any $j\in\{1,2,\ldots,m+1\}$, $j\in\hat{\tau}(\{1,2,\ldots,\hat{m}+1\})
\Leftrightarrow
|\mathcal{B}^\circ(j)|^\circ\cap|\hat{\mathcal{C}}|\neq\emptyset$.

For any $i\in\{1,2,\ldots,\hat{m}+1\}$ and any $\Theta\in\mathcal{B}^\circ\hat{\tau}(i)/H\hat{\tau}(i)$, $\Theta\subset|\hat{\mathcal{C}}|$, if and only if, $H\hat{\tau}(i)\Op|\Theta\subset|\hat{\mathcal{C}}|$.
\item
If $i\in\{1,2,\ldots,\hat{m}\}$, $j\in\{1,2,\ldots,m\}$ and $\hat{\tau}(i)\leq j<\hat{\tau}(i+1)$, then
\begin{equation*}\begin{split}
X(j)\cap|\hat{\mathcal{C}}|&=|\hat{\mathcal{C}}-(\hat{\mathcal{C}}/H)|\cup
(\bigcup_{k\in\{1,2,\ldots, i\}}|\hat{\mathcal{B}}(k)|),\\
Y(j)\cap|\hat{\mathcal{C}}|&=
\bigcup_{k\in\{i+1,i+2,\ldots, \hat{m}+1\}}|\hat{\mathcal{B}}(k)|),
\end{split}\end{equation*}
where $X(j)$ and $Y(j)$ are subsets of $|\mathcal{C}|$ defined in Lemma~\ref{property of basic subdivisions}.14 and in\hfill\break Lemma~\ref{property of basic subdivisions2}.6 respectively.
\end{enumerate}
\end{lemma}

\begin{lemma}
\label{basic7}
Consider any $\Delta\in \mathcal{C}/H$ with $\dim\Delta\geq 2$.

Note that $\mathcal{F}(\Delta)$ is a simplicial cone decomposition over $N$ in $V$,
$\mathcal{F}(\Delta)\subset \mathcal{C}$, 
$|\mathcal{F}(\Delta)|=\Delta$,
$\dim \mathcal{F}(\Delta)=\dim \Vect(|\mathcal{F}(\Delta)|)=\dim\Delta\geq 2$, $\mathcal{F}(\Delta)\Mx=\mathcal{F}(\Delta)^0$, $H\in\mathcal{F}(\Delta)_1$, and $\mathcal{F}(\Delta)=(\mathcal{F}(\Delta)/H)\Fc$.

Note that $\emptyset\neq(\mathcal{F}(\Delta)-(\mathcal{F}(\Delta)/H))_1\subset(\mathcal{C}-(\mathcal{C}/H))_1$ and \hfill\break
$E^{-1}((\mathcal{F}(\Delta)-(\mathcal{F}(\Delta)/H))_1)
\subset\{1,2,\ldots,m\}$.

Let $\hat{m}=\sharp E^{-1}((\mathcal{F}(\Delta)-(\mathcal{F}(\Delta)/H))_1)\in\Z_0$.
Let $\hat{\tau}:\{1,2,\ldots,\hat{m}\}\rightarrow\{1,2,\ldots,m\}$ be the unique injective mapping preserving the order and satisfying $\hat{\tau}(\{1,2,\ldots,$\hfill\break$\hat{m}\})= E^{-1}((\mathcal{F}(\Delta)-(\mathcal{F}(\Delta)/H))_1)$.
Putting $\hat{\tau}(0)=0$ and  $\hat{\tau}(\hat{m}+1)=m+1$, we define an extension
$\hat{\tau}:\{0,1,2,\ldots,\hat{m},\hat{m}+1\}\rightarrow\{0,1,2,\ldots,m,m+1\}$ of $\hat{\tau}:\{1,2,\ldots,\hat{m}\}\rightarrow\{1,2,\ldots,m\}$.
Let $\hat{E}:\{1,2,\ldots,\hat{m}\}\rightarrow(\mathcal{F}(\Delta)-(\mathcal{F}(\Delta)/H))_1$ be the unique mapping satisfying $\iota\hat{E}=E\hat{\tau}$, where $\iota: (\mathcal{F}(\Delta)-(\mathcal{F}(\Delta)/H))_1\rightarrow(\mathcal{C}-(\mathcal{C}/H))_1$ denotes the inclusion mapping.

Let
$H\Op=H\Op|\Delta\in\mathcal{F}(\Delta)^1$,
$\hat{\mathcal{B}}=\mathcal{B}(V, N, H, \mathcal{F}(\Delta), \hat{m}, \hat{E})$,\hfill\break
$\hat{s}=s(V, N, H, \mathcal{F}(\Delta), \hat{m}, \hat{E})$,
$\hat{F}=F(V, N, H, \mathcal{F}(\Delta), \hat{m}, \hat{E})$,\hfill\break
$\hat{G}=G(V, N, H, \mathcal{F}(\Delta), \hat{m}, \hat{E})$,
$\hat{H}=H(V, N, H, \mathcal{F}(\Delta), \hat{m}, \hat{E})$,\hfill\break and
$\hat{\mathcal{B}}=\mathcal{B}(V, N, H, \mathcal{F}(\Delta), \hat{m}, \hat{E})$.

For any $i\in\{1,2,\ldots,\hat{m}+1\}$, we denote
$$\Theta(i)=|\hat{\mathcal{B}}(i)|\subset\Delta,\text{ and}$$
$$\bar{\Theta}(i)= |\hat{\mathcal{B}}(i)-(\hat{\mathcal{B}}(i)/\hat{H}(i))|\subset\Theta(i).$$

\begin{enumerate}
\item
If $\dim\Delta=2$, then $\sharp\mathcal{F}(H\Op)_1=1$, and for the unique element $\bar{E}\in\mathcal{F}(H\Op)_1$ and any $i\in\{0,1,\ldots,\hat{m}\}$, $\hat{s}(i, \bar{E})=i$.
\item
For any $i\in\{1,2,\ldots,\hat{m}, \hat{m}+1\}$, $\bar{\Theta}(i)$, $\bar{\Theta}(i)+H$ and $\Theta(i)$ are simplicial cones over $N$ in $V$,
$\dim \bar{\Theta}(i)=\dim \Delta-1$, $\dim(\bar{\Theta}(i)+H)=\dim \Theta(i)=\dim\Delta$,
$\bar{\Theta}(i)\subset \Theta(i)\subset\bar{\Theta}(i)+H\subset\Delta$,
$\Theta(i)+H =\bar{\Theta}(i)+H$,
$\hat{H}(i)\in\mathcal{F}(\Theta(i))_1$,
$\bar{\Theta}(i)=\hat{H}(i)\Op|\Theta(i)\in\mathcal{F}(\Theta(i))^1$,
$\Theta(i)=\bar{\Theta}(i)+\hat{H}(i)\subset\Delta$,
$\Vect(\Theta(i))=\Vect(\Delta)$,
$\Vect(H)\cap\Vect(\bar{\Theta}(i))=\{0\}$,
$\Vect(H)+\Vect(\bar{\Theta}(i))=\Vect(\Delta)$,
$\bar{\Theta}(i)=\Vect(\bar{\Theta}(i))\cap\Delta$, 
$\Theta(i)+\Vect(H) =\bar{\Theta}(i)+\Vect(H)=\Delta+\Vect(H)$,
$\hat{\mathcal{B}}(i)=\mathcal{F}(\Theta(i))$,
$\hat{\mathcal{B}}(i)-(\hat{\mathcal{B}}(i)/\hat{H}(i)) =\mathcal{F}(\bar{\Theta}(i))$,
$\bar{\Theta}(i)=\Convcone(\{b_{E/N^*}+\hat{s}(i-1, E)b_{H/N^*}| E\in\mathcal{F}(H\Op)_1\})\subset\Delta$, and
$(\mathcal{F}(\Delta)*\hat{F}(1)*\hat{F}(2)*\cdots*\hat{F}(i-1))\Mx=
\{\Theta(j)|j\in\{1,2,\ldots,i-1\}\}\cup\{\bar{\Theta}(i)+H\}$.
\item
Consider any $i\in\{1,2,\ldots,\hat{m}, \hat{m}+1\}$.

For any $\bar{E}\in\mathcal{F}(H\Op)_1$, $\bar{E}+H\in\mathcal{F}(\Delta)_2/H$,
$\bar{\Theta}(i)\cap(\bar{E}+H)\in\mathcal{F}(\bar{\Theta}(i))_1$.
The mapping from $\mathcal{F}(H\Op)_1$ to $\mathcal{F}(\bar{\Theta}(i))_1$ sending  $\bar{E}\in\mathcal{F}(H\Op)_1$ to $\bar{\Theta}(i)\cap(\bar{E}+H)\in\mathcal{F}(\bar{\Theta}(i))_1$ is a bijective mapping.

$H\subset\Theta(i)\Leftrightarrow i=\hat{m}+1$.

$\bar{\Theta}(i)^\circ\not\subset\Delta^\circ\Leftrightarrow\bar{\Theta}(i)\subset\partial\Delta\Leftrightarrow\bar{\Theta}(i)=H\Op\Leftrightarrow i=1$.
\item
For any $i\in\{1,2,\ldots,\hat{m}\}$ and any $j\in\{2,3,\ldots,\hat{m}+1\}$ with $i<j$, 
$\Theta(i)\cap\Theta(j)=\Theta(i)\cap(\bar{\Theta}(j)+H)=\bar{\Theta}(i+1)\cap\bar{\Theta}(j)$.
\item
Consider any $i\in\{1,2,\ldots,\hat{m}\}$.
$\hat{F}(i)=(\bar{\Theta}(i)+H)\cap(\hat{E}(i)+H)\in\mathcal{F}(\bar{\Theta}(i)+H)_2/H$.
$\bar{\Theta}(i+1)=\hat{G}(i)\Op|\Theta(i)\in\mathcal{F}(\Theta(i))^1$.
$\hat{H}(i)=\bar{\Theta}(i+1)\cap(\hat{G}(i)+\hat{H}(i)) \in\mathcal{F}(\Theta(i))_1$.
$\bar{\Theta}(i)\cap\bar{\Theta}(i+1)=(\hat{G}(i)+\hat{H}(i))\Op|\Theta(i)$.
$\{\Lambda\in\mathcal{F}(\Theta(i))|$\hfill\break$\Lambda^\circ\subset\Delta^\circ\cup (H\Op)^\circ \}=\{\Theta(i), \bar{\Theta}(i), \bar{\Theta}(i+1)\}$.
For any $\bar{E}\in\mathcal{F}(H\Op)_1-\{\hat{E}(i)\}$.
$\Theta(i)\cap(\bar{E}+H)= \bar{\Theta}(i)\cap\bar{\Theta}(i+1)\cap(\bar{E}+H)\in\mathcal{F}(\bar{\Theta}(i)\cap\bar{\Theta}(i+1))_1=\mathcal{F}(\Theta(i))_1-\{\hat{G}(i), \hat{H}(i)\}$.
The mapping from $\mathcal{F}(H\Op)_1-\{\hat{E}(i)\}$ to $\mathcal{F}(\Theta(i))_1-\{\hat{G}(i), \hat{H}(i)\}$ sending $\bar{E}\in\mathcal{F}(H\Op)_1-\{\hat{E}(i)\}$ to $\Theta(i)\cap(\bar{E}+H)\in\mathcal{F}(\Theta(i))_1-\{\hat{G}(i), \hat{H}(i)\}$ is a bijective mapping.
$\Theta(i)\cap(\hat{E}(i)+H)= \hat{G}(i)+\hat{H}(i) \in\mathcal{F}(\Theta(i))_2$.
\item
$H=\hat{H}(\hat{m}+1)\in\mathcal{F}(\Theta(\hat{m}+1))_1$.
$\bar{\Theta}(\hat{m}+1)=H\Op|\Theta(\hat{m}+1)\in\mathcal{F}(\Theta(\hat{m}+1))^1$.
$\{\Lambda\in\mathcal{F}(\Theta(\hat{m}+1))|\Lambda^\circ\subset\Delta^\circ\cup (H\Op)^\circ \}=\{\Theta(\hat{m}+1), \bar{\Theta}(\hat{m}+1)\}$.
For any $\bar{E}\in\mathcal{F}(H\Op)_1$, $\Theta(\hat{m}+1)\cap(\bar{E}+H)\in\mathcal{F}(\Theta(\hat{m}+1))_2/H$.
The mapping from $\mathcal{F}(H\Op)_1$ to $\mathcal{F}(\Theta(\hat{m}+1))_2/H$ sending $\bar{E}\in\mathcal{F}(H\Op)_1$ to $\Theta(\hat{m}+1)\cap(\bar{E}+H)\in\mathcal{F}(\Theta(\hat{m}+1))_2/H$ is a bijective mapping.
\item
$\hat{\mathcal{B}}$ is the iterated barycentric subdivision of $\mathcal{F}(\Delta)$, it is a simplicial cone decomposition, and it is determined by the sextuplet $(V, N, H, \mathcal{F}(\Delta), \hat{m}, \hat{E})$ uniquely.
$|\hat{\mathcal{B}}|=\Delta$.

$\hat{\mathcal{B}}$ is $H$-simple.
Let $\smash{\bar{\hat{\mathcal{B}}}}^1=\{\Lambda\in\hat{\mathcal{B}}^1|\Lambda^\circ\subset\Delta^\circ\}\cup\{H\Op\}$ denote the $H$-skeleton of $\hat{\mathcal{B}}$.

$\sharp\hat{\mathcal{B}}^0=\sharp\smash{\bar{\hat{\mathcal{B}}}}^1=\hat{m}+1$.
$\hat{\mathcal{B}}^0=\{\Theta(i)|i\in\{1,2,\ldots,\hat{m},\hat{m}+1\}\}$.
$\smash{\bar{\hat{\mathcal{B}}}}^1=\{\bar{\Theta}(i)|i\in\{1,2,\ldots,\hat{m},\hat{m}+1\}\}$.

We consider the $H$-order on $\hat{\mathcal{B}}^0$. The bijective mapping from $\{1,2,\ldots,\hat{m},$\hfill\break$\hat{m}+1\}$ to $\hat{\mathcal{B}}^0$ sending $i\in\{1,2,\ldots,\hat{m},\hat{m}+1\}$ to $\Theta(i)\in \hat{\mathcal{B}}^0$ preserves the $H$-order.

We consider the $H$-order on $\smash{\bar{\hat{\mathcal{B}}}}^1$. The bijective mapping from $\{1,2,\ldots,\hat{m},$\hfill\break$\hat{m}+1\}$ to $\smash{\bar{\hat{\mathcal{B}}}}^1$ sending $i\in\{1,2,\ldots,\hat{m},\hat{m}+1\}$ to $\bar{\Theta}(i)\in\smash{\bar{\hat{\mathcal{B}}}}^1$ preserves the $H$-order.

Consider any $\bar{E}\in\mathcal{F}(H\Op)_1$ and any $i\in\{1,2,\ldots,\hat{m},\hat{m}+1\}$.
The structure constant of $\hat{\mathcal{B}}$ corresponding to the pair $(i,\bar{E})$ is equal to $\hat{s}(i-1, \bar{E})\in\Z_0$.
\end{enumerate}
\end{lemma}

\section{Upper boundaries and lower boundaries}
\label{upper}

Let $V$ be any finite dimensional vector space over $\R$ with $\dim V\geq 1$; let $N$ be any lattice of $V$; let $H$ be any simplicial cone over $N$ in $V$ with $\dim H=1$; let $\pi_H:V\rightarrow V/\Vect(H)$ denote the canonical surjective homomorphism of vector spaces over $\R$ to the residue vector space $V/\Vect(H)$; and let $\Delta$ be any convex polyhedral cone in $V$ such that $\Vect(H)\subset\Vect(\Delta)$.

\begin{definition}
Let $X$ be any subset of $V$. We denote
\begin{equation*}\begin{split}
\partial^H_+X&=\{a\in X|(\{a\}+\Vect(H))\cap X\subset\{a\}+(-H)\},\\
\partial^H_-X&=\{a\in X|(\{a\}+\Vect(H))\cap X\subset\{a\}+H\},
\end{split}\end{equation*}
and we call $\partial^H_+X$ and $\partial^H_-X$ the $H$-\emph{upper boundary} of $X$ and the $H$-\emph{lower boundary} of $X$ respectively.
\end{definition}

\begin{lemma}
\label{upper and lower}
\begin{enumerate}
\item
\begin{equation*}\begin{split}
\partial^H_+\Delta&=|\{\Lambda\in\mathcal{F}(\Delta)|H\not\subset\Delta+\Vect(\Lambda)\}|\\
&=|\{\Lambda\in\mathcal{F}(\Delta)|H\not\subset\Delta+\Vect(\Lambda), \dim\Lambda=\dim\Delta-1\}|,\\
\partial^H_-\Delta&=|\{\Lambda\in\mathcal{F}(\Delta)|-H\not\subset\Delta+\Vect(\Lambda)\}|\\
&=|\{\Lambda\in\mathcal{F}(\Delta)|-H\not\subset\Delta+\Vect(\Lambda), \dim\Lambda=\dim\Delta-1\}|,\\
\partial^H_+\Delta\cup\partial^H_-\Delta&=|\{\Lambda\in\mathcal{F}(\Delta)|\Vect(H)\not\subset\Vect(\Lambda)\}|\\
&=|\{\Lambda\in\mathcal{F}(\Delta)|\Vect(H)\not\subset\Vect(\Lambda), \dim\Lambda=\dim\Delta-1\}|\\
&\subset\partial\Delta.
\end{split}\end{equation*}
\item
$\partial^H_+\Delta=\emptyset\Leftrightarrow H\subset\Delta$. 
$\partial^H_-\Delta=\emptyset\Leftrightarrow -H\subset\Delta$.
\item
If $\partial^H_+\Delta\neq\emptyset$, then $\pi_H(\partial^H_+\Delta)=\pi_H(\Delta)$ and the mapping $\pi_H: \partial^H_+\Delta\rightarrow\pi_H(\Delta)$ induced by $\pi_H$ is a continuous bijective mapping whose inverse mapping is also continous.

If $\partial^H_-\Delta\neq\emptyset$, then $\pi_H(\partial^H_-\Delta)=\pi_H(\Delta)$ and the mapping $\pi_H: \partial^H_-\Delta\rightarrow\pi_H(\Delta)$ induced by $\pi_H$ is a continuous bijective mapping whose inverse mapping is also continous.
\item
$$\partial^H_+\Delta\cap\partial^H_-\Delta=
|\{\Lambda\in\mathcal{F}(\Delta)|\Lambda\subset\partial^H_+\Delta\cap\partial^H_-\Delta, \dim\Lambda\leq\dim\Delta-2\}|.$$
\item
Consider any $a\in V$.
Note that $\pi_H(a)\in V/\Vect(H)$ and $\pi_H^{-1}(\pi_H(a))=\{a\}+\R b_{H/N}$.
\begin{enumerate}
\item
If $\pi_H(a)\in\pi_H(\Delta)$ and $\partial^H_+\Delta\neq\emptyset$, then there exists uniquely a real number $t_+\in\R$ satisfying $a+t_+b_{H/N}\in\partial^H_+\Delta$.
\end{enumerate}

In the case where $\pi_H(a)\in\pi_H(\Delta)$ and $\partial^H_+\Delta\neq\emptyset$, we take the unique $t_+\in\R$ satisfying $a+t_+b_{H/N}\in\partial^H_+\Delta$.

\begin{enumerate}
\setcounter{enumii}{1}
\item
If $\pi_H(a)\in\pi_H(\Delta)$ and $\partial^H_-\Delta\neq\emptyset$, then there exists uniquely a real number $t_-\in\R$ satisfying $a+t_-b_{H/N}\in\partial^H_-\Delta$.
\end{enumerate}

In the case where $\pi_H(a)\in\pi_H(\Delta)$ and $\partial^H_-\Delta\neq\emptyset$, we take the unique $t_-\in\R$ satisfying $a+t_-b_{H/N}\in\partial^H_-\Delta$.

\begin{enumerate}
\setcounter{enumii}{2}
\item
If $\pi_H(a)\in\pi_H(\Delta)$, $\partial^H_+\Delta\neq\emptyset$ and $\partial^H_-\Delta\neq\emptyset$, then $t_-\leq t_+$.
\item
$(\{a\}+\R b_{H/N})\cap\Delta$
\begin{equation*}
=
\begin{cases}
\emptyset&\text{if $\pi_H(a)\not\in\pi_H(\Delta)$},\\
\{a+t b_{H/N}|t\in\R, t_-\leq t\leq t_+\}&\text{if $\pi_H(a)\in\pi_H(\Delta)$, $\partial^H_+\Delta\neq\emptyset$ and $\partial^H_-\Delta\neq\emptyset$},\\
\{a+t b_{H/N}|t\in\R, t\leq t_+\}&\text{if $\pi_H(a)\in\pi_H(\Delta)$, $\partial^H_+\Delta\neq\emptyset$ and $\partial^H_-\Delta=\emptyset$},\\
\{a+t b_{H/N}|t\in\R, t_-\leq t\}&\text{if $\pi_H(a)\in\pi_H(\Delta)$, $\partial^H_+\Delta=\emptyset$ and $\partial^H_-\Delta\neq\emptyset$},\\
\{a+t b_{H/N}|t\in\R \}&\text{if $\pi_H(a)\in\pi_H(\Delta)$, $\partial^H_+\Delta=\emptyset$ and $\partial^H_-\Delta=\emptyset$}.
\end{cases}
\end{equation*}
\item
If $\pi_H(a)\in\pi_H(\Delta)^\circ$, $\partial^H_+\Delta\neq\emptyset$ and $\partial^H_-\Delta\neq\emptyset$, then $t_-< t_+$.
\item
$(\{a\}+\R b_{H/N})\cap\Delta^\circ$
\begin{equation*}
=
\begin{cases}
\emptyset&\text{if $\pi_H(a)\not\in\pi_H(\Delta)^\circ$},\\
\{a+t b_{H/N}|t\in\R, t_-< t< t_+\}&\text{if $\pi_H(a)\in\pi_H(\Delta)^\circ$, $\partial^H_+\Delta\neq\emptyset$ and $\partial^H_-\Delta\neq\emptyset$},\\
\{a+t b_{H/N}|t\in\R, t< t_+\}&\text{if $\pi_H(a)\in\pi_H(\Delta)^\circ$, $\partial^H_+\Delta\neq\emptyset$ and $\partial^H_-\Delta=\emptyset$},\\
\{a+t b_{H/N}|t\in\R, t_-< t\}&\text{if $\pi_H(a)\in\pi_H(\Delta)^\circ$, $\partial^H_+\Delta=\emptyset$ and $\partial^H_-\Delta\neq\emptyset$},\\
\{a+t b_{H/N}|t\in\R \}&\text{if $\pi_H(a)\in\pi_H(\Delta)^\circ$, $\partial^H_+\Delta=\emptyset$ and $\partial^H_-\Delta=\emptyset$}.
\end{cases}\end{equation*}
\end{enumerate}
\end{enumerate}
\end{lemma}

\section{Height, characteristic functions and compatible mappings}
\label{compatible}
In this section we consider the following objects: Let $V$ be any finite dimensional vector space over $\R$ with $\dim V\geq 2$; let $N$ be any lattice of $V$; let $H$ be any one-dimensional simplicial cone over the dual lattice $N^*$ of $N$ in the dual vector space $V^*$ of $V$; let $\mathcal{C}$ be any simplicial cone decomposition over $N^*$ in $V^*$ satisfying 
$\dim \mathcal{C}=\dim \Vect(|\mathcal{C}|)\geq 2$, $\mathcal{C}\Mx=\mathcal{C}^0$, $H\in\mathcal{C}_1$ and $\mathcal{C}=(\mathcal{C}/H)\Fc$; let $S$ be any rational convex pseudo polyhedron over $N$ in $V$ satisfying $\dim(|\mathcal{D}(S|V)|)\geq 2$ and $|\mathcal{C}|\subset |\mathcal{D}(S|V)|$; and let $T$ be any convex pseudo polyhedron in $V$ satisfying $\dim(|\mathcal{D}(T|V)|)\geq 1$ and $H\subset\Vect(|\mathcal{D}(T|V)|)$.

\begin{lemma}
\label{finiteness rationality}
\begin{enumerate}
\item
$\{\langle b_{H/N^*}, a\rangle|a\in\mathcal{V}(T)\}$ is a non-empty finite subset of $\R$.
\item
$0<\sharp\{\langle b_{H/N^*}, a\rangle|a\in\mathcal{V}(T)\}\leq c(T)$.
\item
If $T$ is rational over $N$, then the subset
$$\{m\in\Z|ma\in (N+(\Vect(|\mathcal{D}(T|V)|)^\vee|V^*))\text{ for any }a\in\mathcal{V}(T)\}$$
of $\Z$ is an ideal of the ring $\Z$ containing a positite integer.
\end{enumerate}
\end{lemma}

\begin{proof}
It follows from Proposition~\ref{property of polyhedrons}.7.
\end{proof}

Note that
$H\subset|\mathcal{C}|\subset|\mathcal{D}(S|V)|\subset
\Vect(|\mathcal{D}(S|V)|)$.

We denote $\ell=\dim\Vect(|\mathcal{D}(S|V)|)^\vee|V^*=\dim(\Stab(S)\cap(-\Stab(S)))\in\Z_0$.

\begin{definition}
\label{relative}
\begin{enumerate}
\item
We define functions
$$\lfloor\ \rfloor,\  \lceil\ \rceil:\R\rightarrow\Z,$$
by putting
$$\lfloor r\rfloor=\min\{i\in\Z|r\leq i\},\quad
\lceil r\rceil=\max\{i\in\Z|r\geq i\},$$
for any $r\in\R$.
\item
We define
$$\mathcal{H}(H, T)=\{\langle b_{H/N^*},a\rangle|a\in \mathcal{V}(T)\}\subset\R,$$
$$\Ht(H,T)=\max\mathcal{H}(H, T)-\min\mathcal{H}(H, T)\in\R_0,$$
and we call
$\mathcal{H}(H, T)$ and
$\Ht(H,T)$,
the $H$-\emph{height set} of $T$ and
the $H$-\emph{height} of $T$ respectively.
\item
Assume that $T$ is rational over $N$. 
By $\Den(T/N)$ we denote the minimum positive integer in the ideal
$$\{m\in\Z|ma\in (N+(\Vect(|\mathcal{D}(T|V)|)^\vee|V^*))\text{ for any }a\in\mathcal{V}(T)\},$$
and we call $\Den(T/N)\in\Z_+$ the \emph{denominator} of $T$ over $N$.
\item
We define
\begin{equation*}\begin{split}
(\mathcal{C}, \mathcal{F}(S)_\ell)=
\{F\in\mathcal{F}(S)_\ell|&\dim(\Delta(F,S|V)\cap\Delta)=\dim \Delta\text{ for some }\Delta\in\mathcal{C}\Mx.\}\\
&\qquad\qquad\subset \mathcal{F}(S)_\ell,
\end{split}\end{equation*}
$$\mathcal{V}(\mathcal{C}, S)=
\bigcup_{F\in(\mathcal{C}, \mathcal{F}(S)_\ell)}F\subset\mathcal{V}(S),$$
$$\mathcal{H}(H, \mathcal{C}, S)=\{\langle b_{H/N^*},a\rangle|a\in \mathcal{V}(\mathcal{C}, S)\}\subset\R,$$
$$\Ht(H,\mathcal{C},S)=\max\mathcal{H}(H, \mathcal{C}, S)-\min\mathcal{H}(H, \mathcal{C}, S)\in\R_0,$$
and we call $(\mathcal{C}, \mathcal{F}(S)_\ell)$,
$\mathcal{V}(\mathcal{C}, S)$,
$\mathcal{H}(H, \mathcal{C}, S)$ and
$\Ht(H,\mathcal{C},S)$,
the \emph{set of minimal faces} of the pair $(\mathcal{C}, S)$,
the \emph{skeleton} of the pair $(\mathcal{C}, S)$,
the $H$-\emph{height set} of the pair $(\mathcal{C}, S)$ and
the $H$-\emph{height} of the pair $(\mathcal{C}, S)$ respectively.
\item
For any $h\in \mathcal{H}(H, \mathcal{C}, S)$ we denote
\begin{equation*}\begin{split}
\mathcal{E}(h)&=
\{\Delta(F, S|V)\cap\Delta|F\in(\mathcal{C}, \mathcal{F}(S)_\ell),
\Delta\in\mathcal{C}\Mx,
\dim(\Delta(F, S|V)\cap\Delta)=\dim \Delta,\\
&\qquad\quad\langle b_{H/N^*}, a\rangle=h\text{ for some }a\in F\}\Fc\subset\mathcal{D}(S|V)\hat{\cap}\mathcal{C},\\
\mathcal{D}(h)&=
\{\Delta(F, S|V)\cap\Delta|F\in(\mathcal{C}, \mathcal{F}(S)_\ell),
\Delta\in\mathcal{C}\Mx,
\dim(\Delta(F, S|V)\cap\Delta)=\dim \Delta,\\
&\qquad\quad\langle b_{H/N^*}, a\rangle\geq h\text{ for some }a\in F\}\Fc\subset\mathcal{D}(S|V)\hat{\cap}\mathcal{C}.
\end{split}\end{equation*}
\end{enumerate}
\end{definition}

Consider any $\Delta\in\mathcal{C}/H$.
$S+(\Delta^\vee|V^*)$ is a rational conex pseudo polyhedron over $N$ in $V$.
$\Stab(S+(\Delta^\vee|V^*))=\Delta^\vee|V^*$.
$H\subset |\mathcal{D}(S+(\Delta^\vee|V^*)|V)|=\Stab(S+(\Delta^\vee|V^*))^\vee|V=\Delta\subset|\mathcal{C}|\subset|\mathcal{D}(S|V)|\subset\Vect(|\mathcal{D}(S|V)|)$.
$\mathcal{D}(S+(\Delta^\vee|V^*)|V)= \mathcal{D}(S|V)\hat{\cap}\mathcal{F}(\Delta)$.

\begin{lemma}
\label{compatible1}
\begin{enumerate}
\item
The set $\mathcal{H}(H, T)$ is a non-empty finite subset of $\R$.
$0<\sharp\mathcal{H}(H, T)\leq c(T)$.

If $T$ is rational over $N$, then $\mathcal{H}(H, T)\subset(1/\Den(T/N))\Z\subset\Q$ and\hfill\break
$\Ht(H,T)\in(1/\Den(T/N))\Z_0\subset\Q_0$.

If $T$ is rational over $N$ and $\mathcal{V}(T)\subset N+(\Vect(|\mathcal{D}(T|V)|)^\vee|V^*)$, then $\Den(T/N)=1$, $\mathcal{H}(H, T)\subset\Z$ and $\Ht(H,T)\in\Z_0$.
\item
$\emptyset\neq(\mathcal{C},\mathcal{F}(S)_\ell)\subset\mathcal{F}(S)_\ell$.

$\emptyset\neq\mathcal{V}(\mathcal{C}, S)\subset\mathcal{V}(S)$. The set $\mathcal{V}(\mathcal{C}, S)$ is the union of some connected components of $\mathcal{V}(S)$.

$\emptyset\neq\mathcal{H}(H, \mathcal{C}, S)\subset\mathcal{H}(H, S)\subset 
(1/\Den(S/N))\Z\subset\Q$. The set $\mathcal{H}(H, \mathcal{C}, S)$ is a non-empty finite subset of $\Q$.

$\min \mathcal{H}(H, \mathcal{C}, S)=\min\mathcal{H}(H, S)$.
$\Ht(H,\mathcal{C},S)\leq \Ht(H,S)$.

$\Ht(H,S)\in(1/\Den(S/N))\Z_0\subset\Q_0$.

$\Ht(H,\mathcal{C},S)\in(1/\Den(S/N))\Z_0\subset\Q_0$.

If $\mathcal{V}(S)\subset N+(\Vect(|\mathcal{D}(S|V)|)^\vee|V^*)$, then
$\Den(S/N)=1$,
$\mathcal{H}(H, \mathcal{C}, S)\subset\mathcal{H}(H, S)\subset \Z$, $\Ht(H,S)\in\Z_0$ and $\Ht(H,\mathcal{C},S)\in\Z_0$.
\item
Consider any $\Delta\in\mathcal{C}/H$. 

$\mathcal{H}(H, S+(\Delta^\vee|V^*))\subset \mathcal{H}(H, \mathcal{C}, S)\subset\Q$.

$\min\mathcal{H}(H, S+(\Delta^\vee|V^*))=\min\mathcal{H}(H, \mathcal{C}, S)$.

$\Ht(H,S+(\Delta^\vee|V^*))\leq\Ht(H,\mathcal{C},S)$.

$\Den(S/N)$ is a multiple of $\Den(S+(\Delta^\vee|V^*)/N)$.

$\mathcal{H}(H, S+(\Delta^\vee|V^*))\subset(1/\Den(S+(\Delta^\vee|V^*)/N))\Z\subset(1/\Den(S/N))\Z\subset \Q$.

$\Ht(H,S+(\Delta^\vee|V^*))\in(1/\Den(S+(\Delta^\vee|V^*)/N))\Z_0\subset (1/\Den(S/N))\Z_0$\break$\subset \Q_0$.

If $\mathcal{V}(S)\subset N+(\Vect(|\mathcal{D}(S|V)|)^\vee|V^*)$, then 
$\Den(S+(\Delta^\vee|V^*)/N)=1$,
$\mathcal{H}(H, S+(\Delta^\vee|V^*))\subset\Z$ and $\Ht(H,S+(\Delta^\vee|V^*))\in\Z_0$.
\item
$\max\mathcal{H}(H, \mathcal{C}, S)=\max\{\max\mathcal{H}(H, S+(\Delta^\vee|V^*))|\Delta\in \mathcal{C}\Mx\}=\max\{\max\mathcal{H}(H,$\break$ S+(\Delta^\vee|V^*))|\Delta\in \mathcal{C}/H\}$.

$\Ht(H,\mathcal{C},S)= \max\{\Ht(H,S+(\Delta^\vee|V^*))|\Delta\in \mathcal{C}\Mx\}= \max\{\Ht($\break$H,S+(\Delta^\vee|V^*))|\Delta\in \mathcal{C}/H\}$.
\item
Consider any subset $\hat{\mathcal{C}}$ of $\mathcal{C}$ satisfying 
$\dim\hat{\mathcal{C}}=\dim \Vect(|\hat{\mathcal{C}}|)\geq 2$, $\hat{\mathcal{C}}\Mx=\hat{\mathcal{C}}^0$,
$H\in\hat{\mathcal{C}}_1$
and $\hat{\mathcal{C}}=(\hat{\mathcal{C}}/H)\Fc$.

Note that $\hat{\mathcal{C}}$ is a simplicial cone decomposition over $N^*$ in $V^*$, and
$|\hat{\mathcal{C}}|\subset|\mathcal{C}|\subset|\mathcal{D}(S|V)|$.

$\mathcal{V}(\hat{\mathcal{C}}, S)\subset\mathcal{V}(\mathcal{C}, S)$.
$\mathcal{H}(H, \hat{\mathcal{C}},S)\subset\mathcal{H}(H, \mathcal{C}, S)$.
$\min\mathcal{H}(H, \hat{\mathcal{C}},S)=\min\mathcal{H}(H, $\hfill\break$\mathcal{C}, S)$.
$\Ht(H, \hat{\mathcal{C}}, S)\leq\Ht(H,\mathcal{C},S)$.

\item
$\mathcal{D}(S|V)\hat{\cap}\mathcal{C}$ is a rational convex polyhedral cone decomposition over $N^*$ in $V^*$.
$\dim(\mathcal{D}(S|V)\hat{\cap}\mathcal{C})=
\dim(\Vect(|\mathcal{D}(S|V)\hat{\cap}\mathcal{C}|))=
\dim \mathcal{C}$.
$\Vect(|\mathcal{D}(S|V)\hat{\cap}\mathcal{C}|)=\Vect(|\mathcal{C}|)$.
$(\mathcal{D}(S|V)\hat{\cap}\mathcal{C})\Mx=(\mathcal{D}(S|V)\hat{\cap}\mathcal{C})^0$.
\item
Consider any $h\in\mathcal{H}(H, \mathcal{C}, S)$.
$\mathcal{E}(h)$ and $\mathcal{D}(h)$ are rational convex polyhedral cone decompositions over $N^*$ in $V^*$.
$\dim \mathcal{E}(h)=\dim\Vect(|\mathcal{E}(h)|)=\dim\mathcal{C}$.
$\Vect(|\mathcal{E}(h)|)=\Vect(|\mathcal{C}|)$.
$\mathcal{E}(h)\Mx=\mathcal{E}(h)^0$.
$\dim \mathcal{D}(h)=\dim\Vect(|\mathcal{D}(h)|)=\dim\mathcal{C}$.
$\Vect(|\mathcal{D}(h)|)= \Vect(|\mathcal{C}|)$.
$\mathcal{D}(h)\Mx=\mathcal{D}(h)^0$.
$\mathcal{E}(h)\subset\mathcal{D}(h)\subset$
\hfill\break $\mathcal{D}(S|V)\hat{\cap}\mathcal{C}$.

\item
Consider any $g\in\mathcal{H}(H, \mathcal{C}, S)$ and any $h\in\mathcal{H}(H, \mathcal{C}, S)$.
$g\leq h\Leftrightarrow \mathcal{D}(g)\supset\mathcal{D}(h)\Leftrightarrow |\mathcal{D}(g)|\supset|\mathcal{D}(h)|$.
$g=h\Leftrightarrow \mathcal{D}(g)=\mathcal{D}(h)\Leftrightarrow |\mathcal{D}(g)|=|\mathcal{D}(h)|$.
\item
We consider any $F\in\mathcal{F}(S)_\ell$ and any $G\in\mathcal{F}(S)_\ell$
such that $\langle b_{H/N^*}, a\rangle=\langle b_{H/N^*}, b\rangle$ for some $a\in F$ and some $b\in G$.
\begin{equation*}\begin{split}
\partial^H_+(\Delta(F,S|V)\cup\Delta(G,S|V))
&=\partial^H_+\Delta(F,S|V)\cup\partial^H_+\Delta(G,S|V),\text{ and}\\
\partial^H_-(\Delta(F,S|V)\cup\Delta(G,S|V))
&=\partial^H_-\Delta(F,S|V)\cup\partial^H_-\Delta(G,S|V).
\end{split}\end{equation*}
\end{enumerate}
\end{lemma}

Below in this section, we assume moreover that $\mathcal{D}(S|V)\hat{\cap}\mathcal{F}(\Delta)$ is $H$-simple for any $\Delta\in\mathcal{C}\Mx$.
We denote
$$\max=\max\mathcal{H}(H, \mathcal{C}, S)\in(1/\Den(S/N))\Z, \ \min=\min\mathcal{H}(H, \mathcal{C}, S)\in(1/\Den(S/N))\Z.$$
By $\pi_H:V^*\rightarrow V^*/\Vect(H)$ we denote the canonical surjective homomorphism of vector spaces over $\R$ to the residue vector space $V^*/\Vect(H)$.
\begin{lemma}
\label{compatible2}
Consider any $h\in\mathcal{H}(H, \mathcal{C}, S)$.
\begin{enumerate}
\item
$\max\geq\min$. $\Ht(H,\mathcal{C},S)=\max-\min$.
$\Ht(H,\mathcal{C},S)=0\Leftrightarrow \max=\min\Leftrightarrow \mathcal{C}$ is a subdivision of $\mathcal{D}(S|V)$.
\item
$\partial^H_+|\mathcal{E}(h)|\neq\emptyset\Leftrightarrow h\neq\min$.
$\partial^H_-|\mathcal{E}(h)|\neq\emptyset$.
\item
If $h\neq\min$, then $\pi_H(\partial^H_+|\mathcal{E}(h)|)=\pi_H(|\mathcal{E}(h)|)$ and the mapping $\pi_H:$\hfill\break $\partial^H_+|\mathcal{E}(h)|\rightarrow \pi_H(|\mathcal{E}(h)|)$ induced by $\pi_H$ is a continuous bijective mapping whose inverse mapping is also continuous.

$\pi_H(\partial^H_-|\mathcal{E}(h)|)=\pi_H(|\mathcal{E}(h)|)$ and the mapping $\pi_H: \partial^H_-|\mathcal{E}(h)|\rightarrow \pi_H(|\mathcal{E}(h)|)$ induced by $\pi_H$ is a continuous bijective mapping whose inverse mapping is also continuous.
\item
Consider any $a\in V^*$.
Note that $\pi_H(a)\in V^*/\Vect(H)$ and $\pi_H^{-1}(\pi_H(a))=\{a\}+\R b_{H/N^*}$.
\begin{enumerate}
\item
If $\pi_H(a)\in\pi_H(|\mathcal{E}(h)|)$ and $h\neq\min$, then there exists uniquely a real number $t_+\in\R$ satisfying $a+t_+b_{H/N^*}\in\partial^H_+|\mathcal{E}(h)|$.
\end{enumerate}

In the case where $\pi_H(a)\in\pi_H(|\mathcal{E}(h)|)$ and $h\neq\min$, we take the unique $t_+\in\R$ satisfying $a+t_+b_{H/N^*}\in\partial^H_+|\mathcal{E}(h)|$.

\begin{enumerate}
\setcounter{enumii}{1}
\item
If $\pi_H(a)\in\pi_H(|\mathcal{E}(h)|)$, then there exists uniquely a real number $t_-\in\R$ satisfying $a+t_-b_{H/N^*}\in\partial^H_-|\mathcal{E}(h)|$.
\end{enumerate}

In the case where $\pi_H(a)\in\pi_H(|\mathcal{E}(h)|)$, we take the unique $t_-\in\R$ satisfying $a+t_-b_{H/N^*}\in\partial^H_-|\mathcal{E}(h)|$.

\begin{enumerate}
\setcounter{enumii}{2}
\item
If $\pi_H(a)\in\pi_H(|\mathcal{E}(h)|)$ and $h\neq\min$, then $t_-\leq t_+$.
\item
$(\{a\}+\R b_{H/N^*})\cap|\mathcal{E}(h)|$
\begin{equation*}
=
\begin{cases}
\emptyset&\text{if $\pi_H(a)\not\in\pi_H(|\mathcal{E}(h)|)$},\\
\{a+t b_{H/N^*}|t\in\R, t_-\leq t\leq t_+\}&\text{if $\pi_H(a)\in\pi_H(|\mathcal{E}(h)|)$ and $h\neq\min$},\\
\{a+t b_{H/N^*}|t\in\R, t_-\leq t\}&\text{if $\pi_H(a)\in\pi_H(|\mathcal{E}(h)|)$ and $h=\min$}.
\end{cases}
\end{equation*}
\end{enumerate}
\item
$\partial^H_+|\mathcal{D}(h)|\neq\emptyset\Leftrightarrow h\neq\min$.
$\partial^H_-|\mathcal{D}(h)|= |\mathcal{D}(h)|\cap|\mathcal{C}-(\mathcal{C}/H)|\neq\emptyset$.
\item
If $h\neq\min$, then $\pi_H(\partial^H_+|\mathcal{D}(h)|)=\pi_H(|\mathcal{D}(h)|)$ and the mapping $\pi_H:$\hfill\break $\partial^H_+|\mathcal{D}(h)|\rightarrow \pi_H(|\mathcal{D}(h)|)$ induced by $\pi_H$ is a continuous bijective mapping whose inverse mapping is also continuous.

$\pi_H(\partial^H_-|\mathcal{D}(h)|)=\pi_H(|\mathcal{D}(h)|)$ and the mapping $\pi_H: \partial^H_-|\mathcal{D}(h)|\rightarrow \pi_H(|\mathcal{D}(h)|)$ induced by $\pi_H$ is a continuous bijective mapping whose inverse mapping is also continuous.
\item
Consider any $a\in V^*$.
Note that $\pi_H(a)\in V^*/\Vect(H)$ and $\pi_H^{-1}(\pi_H(a))=\{a\}+\R b_{H/N^*}$.
\begin{enumerate}
\item
If $\pi_H(a)\in\pi_H(|\mathcal{D}(h)|)$ and $h\neq\min$, then there exists uniquely a real number $t_+\in\R$ satisfying $a+t_+b_{H/N^*}\in\partial^H_+|\mathcal{D}(h)|$.
\end{enumerate}

In the case where $\pi_H(a)\in\pi_H(|\mathcal{D}(h)|)$ and $h\neq\min$, we take the unique $t_+\in\R$ satisfying $a+t_+b_{H/N^*}\in\partial^H_+|\mathcal{D}(h)|$.

\begin{enumerate}
\setcounter{enumii}{1}
\item
If $\pi_H(a)\in\pi_H(|\mathcal{D}(h)|)$, then there exists uniquely a real number $t_-\in\R$ satisfying $a+t_-b_{H/N^*}\in\partial^H_-|\mathcal{D}(h)|$.
\end{enumerate}

In the case where $\pi_H(a)\in\pi_H(|\mathcal{D}(h)|)$, we take the unique $t_-\in\R$ satisfying $a+t_-b_{H/N^*}\in\partial^H_-|\mathcal{D}(h)|$.

\begin{enumerate}
\setcounter{enumii}{2}
\item
If $\pi_H(a)\in\pi_H(|\mathcal{D}(h)|)$ and $h\neq\min$, then $t_-\leq t_+$.
\item
$(\{a\}+\R b_{H/N^*})\cap|\mathcal{D}(h)|$
\begin{equation*}
=
\begin{cases}
\emptyset&\text{if $\pi_H(a)\not\in\pi_H(|\mathcal{D}(h)|)$},\\
\{a+t b_{H/N^*}|t\in\R, t_-\leq t\leq t_+\}&\text{if $\pi_H(a)\in\pi_H(|\mathcal{D}(h)|)$ and $h\neq\min$},\\
\{a+t b_{H/N^*}|t\in\R, t_-\leq t\}&\text{if $\pi_H(a)\in\pi_H(|\mathcal{D}(h)|)$ and $h=\min$}.
\end{cases}
\end{equation*}
\end{enumerate}
\item
If $h=\max$, then $\mathcal{D}(h)= \mathcal{E}(h)$.
\item
Assume $h\neq\max$.
Put $g=\min\{f\in\mathcal{H}(H, \mathcal{C}, S)|f>h\}\in\mathcal{H}(H, \mathcal{C}, S)$.

$\mathcal{D}(h)=\mathcal{D}(g)\cup\mathcal{E}(h)$.
$|\mathcal{D}(h)|=|\mathcal{D}(g)|\cup|\mathcal{E}(h)|$.
$\partial^H_+|\mathcal{D}(g)|\neq\emptyset$.
$\partial^H_-|\mathcal{E}(h)|\neq\emptyset$.
$|\mathcal{D}(g)|\cap|\mathcal{E}(h)|= \partial^H_+|\mathcal{D}(g)|\cap\partial^H_-|\mathcal{E}(h)|$.
$\partial^H_-|\mathcal{E}(h)|\subset\partial^H_+|\mathcal{D}(g)|\cup|\mathcal{C}-(\mathcal{C}/H)|$.
$\partial^H_+|\mathcal{D}(h)|=( \partial^H_+|\mathcal{D}(g)|- \partial^H_-|\mathcal{E}(h)|)\cup\partial^H_+|\mathcal{E}(h)|$.
\item
If $h=\min$, then $\mathcal{D}(h)=\mathcal{D}(S|V)\hat{\cap}\mathcal{C}$ and
$|\mathcal{D}(h)|=| \mathcal{C}|$.
\item
$$\pi_H(|\mathcal{D}(h)|)\subset|\pi_{H*}\mathcal{C}|.$$
$$\pi_H(|\mathcal{D}(h)|)=
|\{\bar{\Delta}\in(\pi_{H*}\mathcal{C})\Mx|\bar{\Delta}\subset\pi_H(|\mathcal{D}(h)|)\}|.$$
$$\Clos(|\pi_{H*}\mathcal{C}|-\pi_H(|\mathcal{D}(h)|))=
|\{\bar{\Delta}\in(\pi_{H*}\mathcal{C})\Mx|\bar{\Delta}\subset\Clos(|\pi_{H*}\mathcal{C}|-\pi_H(|\mathcal{D}(h)|))\}|.$$

\item
Consider any $\bar{E}\in(\mathcal{C}-(\mathcal{C}/H))_1$.
If $\Ht(H,\mathcal{C},S)>0$ and $\pi_H(\bar{E})\subset\pi_H(\mathcal{D}(\max))$, then there exists uniquely a real number $\gamma(\bar{E})\in\R$ satisfying 
$b_{\bar{E}/N^*}+\gamma(\bar{E})b_{H/N^*}\in\partial^H_+\mathcal{D}(\max)$.
\end{enumerate}
\end{lemma}

Below we assume moreover that $\Ht(H,\mathcal{C},S)>0$.

\begin{definition}
\label{defchar}
We define a function $\gamma: (\mathcal{C}-(\mathcal{C}/H))_1\rightarrow \R$.

Consider any $\bar{E}\in(\mathcal{C}-(\mathcal{C}/H))_1$.
If $\pi_H(\bar{E})\subset\pi_H(\mathcal{D}(\max))$, then we take the unique real number $\gamma(\bar{E})\in\R$ satisfying 
$b_{\bar{E}/N^*}+\gamma(\bar{E})b_{H/N^*}\in\partial^H_+\mathcal{D}(\max)$.
If $\pi_H(\bar{E})\not\subset\pi_H(\mathcal{D}(\max))$, then we put $\gamma(\bar{E})=0\in\R$.

We call $\gamma$ the \emph{characteristic function} of the triplet $(H,\mathcal{C},S)$.
\end{definition}

\begin{lemma}
\label{propchar}
Let $\gamma: (\mathcal{C}-(\mathcal{C}/H))_1\rightarrow \R$ denote the characteristic function of $(H, \mathcal{C}, S)$.
\begin{enumerate}
\item
For any $\bar{E}\in(\mathcal{C}-(\mathcal{C}/H))_1$, $\gamma(\bar{E})\in\Q_0$.
\item
Consider any $\Delta\in\mathcal{C}\Mx$. 
There exists $\bar{E}\in\mathcal{F}(H\Op|\Delta)_1$ satisfying $\gamma(\bar{E})>0\Leftrightarrow\pi_H(\Delta)\subset\pi_H(\mathcal{D}(\max))\Leftrightarrow \pi_H(\Delta)\not\subset\Clos(|\pi_{H*}\mathcal{C}|-\pi_H(|\mathcal{D}(\max)|))$.
\item
Consider any $\Delta\in\mathcal{C}\Mx$ satisfying $\pi_H(\Delta)\subset\pi_H(\mathcal{D}(\max))$.
Note that $H\in\mathcal{F}(\Delta)_1$ and $\mathcal{F}(H\Op|\Delta)_1\subset (\mathcal{C}-(\mathcal{C}/H))_1$.

$\mathcal{D}(S+(\Delta^\vee|V^*)|V)=\mathcal{D}(S|V)\hat{\cap}\mathcal{F}(\Delta)$ is $H$-simple, and
$c(S+(\Delta^\vee|V^*))\geq 2$.

For any $\bar{E}\in\mathcal{F}(H\Op|\Delta)_1$, $\gamma(\bar{E})=c(\mathcal{D}(S+(\Delta^\vee|V^*)|V), 2, \bar{E})$,
where \hfill\break$c(\mathcal{D}(S+(\Delta^\vee|V^*)|V), 2, \bar{E})$ denotes the structure constant of $\mathcal{D}(S+(\Delta^\vee|V^*)|V)$ corresponding to $(2, \bar{E})$.
\item
There exists $\bar{E}\in(\mathcal{C}-(\mathcal{C}/H))_1$ with $\gamma(\bar{E})>0$.
\item
For any $\bar{E}\in(\mathcal{C}-(\mathcal{C}/H))_1$ satisfying
$\pi_H(\bar{E})\subset\Clos(|\pi_{H*}\mathcal{C}|-\pi_H(|\mathcal{D}(\max)|))$,
$\gamma(\bar{E})=0$.
\item
Consider any $\bar{E}\in(\mathcal{C}-(\mathcal{C}/H))_1$ satisfying $\gamma(\bar{E})\not\in\Z$.
Then, there exists uniquely $h(\bar{E})\in\mathcal{H}(H, \mathcal{C}, S)$ satisfying
\begin{equation*}\begin{split}
\{h\in\mathcal{H}(H, \mathcal{C}, S)&|
b_{\bar{E}/N^*}+\lfloor\gamma(\bar{E})\rfloor b_{H/N^*}\in\mathcal{D}(h)-\partial^H_+\mathcal{D}(h)\}\\
=&
\{h\in\mathcal{H}(H, \mathcal{C}, S)|h\leq h(\bar{E})\},\\
\end{split}\end{equation*}
and if $h(\bar{E})\in\mathcal{H}(H, \mathcal{C}, S)$ satisfies this equality, then $h(\bar{E})\neq\max$.
\end{enumerate}
\end{lemma}

\begin{definition}
\label{defcomp}
Let $\gamma: (\mathcal{C}-(\mathcal{C}/H))_1\rightarrow \R$ denote the characteristic function of $(H, \mathcal{C}, S)$.
For any $\bar{E}\in(\mathcal{C}-(\mathcal{C}/H))_1$ satisfying $\gamma(\bar{E})\not\in\Z$, we take the unique element  $h(\bar{E})\in\mathcal{H}(H, \mathcal{C}, S)$ satisfying
\begin{equation*}\begin{split}
\{h\in\mathcal{H}(H, \mathcal{C}, S)&|
b_{\bar{E}/N^*}+\lfloor\gamma(\bar{E})\rfloor b_{H/N^*}\in\mathcal{D}(h)-\partial^H_+\mathcal{D}(h)\}\\
=&
\{h\in\mathcal{H}(H, \mathcal{C}, S)|h\leq h(\bar{E})\}.\\
\end{split}\end{equation*}

Let
\begin{equation*}\begin{split}
m&=\sum_{\bar{E}\in(\mathcal{C}-(\mathcal{C}/H))_1}\lfloor\gamma(\bar{E})\rfloor,\\
\bar{m}&=\sum_{\bar{E}\in(\mathcal{C}-(\mathcal{C}/H))_1}\lceil\gamma(\bar{E})\rceil,\text{ and}\\
\mathcal{R}&=\{ \bar{E}\in(\mathcal{C}-(\mathcal{C}/H))_1|\gamma(\bar{E})\not\in\Z\}.
\end{split}\end{equation*}
Note that $m\in\Z_+$, $\bar{m}\in\Z_0$, $\bar{m}\leq m$, and
$m-\bar{m}=\sharp\mathcal{R}$.

Consider any mapping $E:\{1,2,\ldots,m\}\rightarrow (\mathcal{C}-(\mathcal{C}/H))_1$.

We say that the mapping $ E$ is \emph{compatible} with $S$, if the following three conditions are satisfied:
\begin{enumerate}
\item
For any $\bar{E}\in(\mathcal{C}-(\mathcal{C}/H))_1$,
$\sharp(\{1,2,\ldots,\bar{m}\}\cap E^{-1}(\bar{E}))= \lceil\gamma(\bar{E})\rceil$.
\item
$E(\{\bar{m}+1,\bar{m}+2,\ldots,m\})=\mathcal{R}$.
\item
If $m-\bar{m}\geq 2$, then $h(E(i))\geq h(E(i+1))$ for any $i\in\{\bar{m}+1,\bar{m}+2,\ldots,m-1\}$.
\end{enumerate}
\end{definition}

\begin{lemma}
\label{existcomp}
Let $\gamma: (\mathcal{C}-(\mathcal{C}/H))_1\rightarrow\R$ denote the characteristic function of $(H, \mathcal{C}, S)$.
Let $m=\sum_{\bar{E}\in(\mathcal{C}-(\mathcal{C}/H))_1}\lfloor\gamma(\bar{E})\rfloor\in\Z_+$,
$\bar{m}=\sum_{\bar{E}\in(\mathcal{C}-(\mathcal{C}/H))_1}\lceil\gamma(\bar{E})\rceil\in\Z_0$, and
$\mathcal{R}=\{\bar{E}\in(\mathcal{C}-(\mathcal{C}/H))_1|\gamma(\bar{E})\not\in\Z\}\subset (\mathcal{C}-(\mathcal{C}/H))_1$.
\begin{enumerate}
\item
There exists a compatible mapping $E:\{1,2,\ldots,m\}\rightarrow (\mathcal{C}-(\mathcal{C}/H))_1$
with $S$.
\item
Assume that a mapping $E:\{1,2,\ldots,m\}\rightarrow (\mathcal{C}-(\mathcal{C}/H))_1$
is compatible with $S$.
\begin{enumerate}
\item
For any $\bar{E}\in(\mathcal{C}-(\mathcal{C}/H))_1$,
$\sharp E^{-1}(\bar{E})=\lfloor \gamma(\bar{E})\rfloor$, and
$\sharp(\{1,2,\ldots, \bar{m}\}\cap E^{-1}(\bar{E}))=\lceil \gamma(\bar{E})\rceil$.

\item
$E(\{\bar{m}+1, \bar{m}+2,\ldots,m\})=\mathcal{R}$, and the mapping $E: \{\bar{m}+1, \bar{m}+2,\ldots,m\}\rightarrow \mathcal{R}$ induced by $E$ is bijective.
\item
$E(\{1,2,\ldots,m\})=\{\bar{E}\in(\mathcal{C}-(\mathcal{C}/H))_1|\gamma(\bar{E})>0\}$.
\item
For any $i\in\{1,2,\ldots,m\}$, $\pi_H(E(i))\subset\pi_H(\mathcal{D}(\max))$ and
$\pi_H(E(i))\not\subset\Clos(|\pi_{H*}\mathcal{C}|-\pi_H(|\mathcal{D}(\max)|))$.
\end{enumerate}
\end{enumerate}

Consider any subset $\hat{\mathcal{C}}$ of $\mathcal{C}$ satisfying $\dim\hat{\mathcal{C}}=\dim \Vect(|\hat{\mathcal{C}}|)\geq 2$, $\hat{\mathcal{C}}\Mx=\hat{\mathcal{C}}^0$,
$H\in\hat{\mathcal{C}}_1$, and $\hat{\mathcal{C}}=(\hat{\mathcal{C}}/H)\Fc$.

Note that
$\hat{\mathcal{C}}$ is a simplicial cone decomposition over $N^*$ in $V^*$,
$\emptyset\neq(\hat{\mathcal{C}}-(\hat{\mathcal{C}}/H))_1\subset(\mathcal{C}-(\mathcal{C}/H))_1$ and
$|\hat{\mathcal{C}}|\subset|\mathcal{C}|\subset|\mathcal{D}(S|V)|$.
\begin{enumerate}
\setcounter{enumi}{2}
\item
$\Ht(H, \hat{\mathcal{C}}, S)\leq \Ht(H,\mathcal{C},S)$.

$\Ht(H, \hat{\mathcal{C}}, S)=\Ht(H,\mathcal{C},S)$, if and only if,
$\gamma(\bar{E})>0$ for some $\bar{E}\in (\hat{\mathcal{C}}-(\hat{\mathcal{C}}/H))_1$
\item
If $\gamma(\bar{E})>0$ for some $\bar{E}\in (\hat{\mathcal{C}}-(\hat{\mathcal{C}}/H))_1$, then $\Ht(H, \hat{\mathcal{C}}, S)>0$ and the composition $(\hat{\mathcal{C}}-(\hat{\mathcal{C}}/H))_1\rightarrow\R$ of the inclusion mapping $(\hat{\mathcal{C}}-(\hat{\mathcal{C}}/H))_1\rightarrow(\mathcal{C}-(\mathcal{C}/H))_1$ and $\gamma: (\mathcal{C}-(\mathcal{C}/H))_1\rightarrow\R$ coincides with the characteristic function of $(H,\hat{\mathcal{C}},S)$.
\item
Consider any compatible mapping: $E:\{1,2,\ldots,m\}\rightarrow (\mathcal{C}-(\mathcal{C}/H))_1$ with $S$.

Let $\hat{m}=\sharp E^{-1}((\hat{\mathcal{C}}-(\hat{\mathcal{C}}/H))_1)\in\Z_0$
and $\hat{\bar{m}}=\sharp(\{1,2,\ldots,\bar{m}\}\cap E^{-1}((\hat{\mathcal{C}}-(\hat{\mathcal{C}}/H))_1))\in\Z_0$.
Let $\hat{\tau}:\{1,2,\ldots,\hat{m}\}\rightarrow\{1,2,\ldots,m\}$ be the unique injective mapping preserving the order and satisfying $\hat{\tau}(\{1,2,\ldots,\hat{m}\})=$\hfill\break$E^{-1}((\hat{\mathcal{C}}-(\hat{\mathcal{C}}/H))_1)$.
Let $\hat{E}:\{1,2,\ldots,\hat{m}\}\rightarrow (\hat{\mathcal{C}}-(\hat{\mathcal{C}}/H))_1$ be the unique mapping satisfying $\iota\hat{E}=E\hat{\tau}$, where $\iota: (\hat{\mathcal{C}}-(\hat{\mathcal{C}}/H))_1\rightarrow(\mathcal{C}-(\mathcal{C}/H))_1$ denotes the inclusion mapping.

If $\gamma(\bar{E})>0$ for some $\bar{E}\in (\hat{\mathcal{C}}-(\hat{\mathcal{C}}/H))_1$, then $\hat{m}=\sum_{\bar{E}\in(\hat{\mathcal{C}}-(\hat{\mathcal{C}}/H))_1}\lfloor\gamma(\bar{E})\rfloor$, $\hat{\bar{m}}=\sum_{\bar{E}\in(\hat{\mathcal{C}}-(\hat{\mathcal{C}}/H))_1}\lceil\gamma(\bar{E})\rceil$ and $\hat{E}$ is compatible with $S$.

If $\gamma(\bar{E})=0$ for any $\bar{E}\in (\hat{\mathcal{C}}-(\hat{\mathcal{C}}/H))_1$, then $\hat{m}=0$.
\end{enumerate}
\end{lemma}

\begin{lemma}
\label{equiv2}
Assume $\sharp\mathcal{C}\Mx=1$.

Let $\Delta\in\mathcal{C}\Mx$ denote the unique element. 
$\Delta$ is a simplicial cone over $N^*$ in $V^*$, $H\in\mathcal{F}(\Delta)_1$ and $\dim\Delta=\dim \mathcal{C}\geq 2$, $\mathcal{C}=\mathcal{F}(\Delta)$, 
$\mathcal{C}-(\mathcal{C}/H)= \mathcal{F}(H\Op|\Delta)$, and
 $\Ht(H,S+(\Delta^\vee|V^*))=\Ht(H,\mathcal{C},S)>0$.

Note that $\mathcal{D}(S+(\Delta^\vee|V^*)|V)=\mathcal{D}(S|V)\hat{\cap}\mathcal{F}(\Delta)$ is $H$-simple and $c(S+(\Delta^\vee|V^*))\geq 2$.

Let $\gamma, m, \bar{m}$ and $\mathcal{R}$ be the same as in above Lemma~\ref{existcomp}.

For any $i\in\{1,2,\ldots, c(S+(\Delta^\vee|V^*))\}$ and any $\bar{E}\in\mathcal{F}(H\Op|\Delta)_1$, we can consider the structure constant
$c(\mathcal{D}(S+(\Delta^\vee|V^*)|V),i, \bar{E})\in\Q_0$ of $\mathcal{D}(S+(\Delta^\vee|V^*)|V)$ corresponding to the pair $(i, \bar{E}).$

Denote
\begin{equation*}\begin{split}
\hat{m}&=\sum_{ \bar{E}\in\mathcal{F}(H\Op|\Delta)_1}\lfloor c(\mathcal{D}(S+(\Delta^\vee|V^*)|V),2, \bar{E})\rfloor\in\Z_+,\\
\hat{\bar{m}}&=\sum_{ \bar{E}\in\mathcal{F}(H\Op|\Delta)_1}\lceil c(\mathcal{D}(S+(\Delta^\vee|V^*)|V),2, \bar{E})\rceil\in\Z_0,\text{ and}\\
\hat{\mathcal{R}}&=\{ \bar{E}\in\mathcal{F}(H\Op|\Delta)_1|
c(\mathcal{D}(S+(\Delta^\vee|V^*)|V),2, \bar{E})\not\in\Z\}
\subset\mathcal{F}(H\Op|\Delta)_1.
\end{split}\end{equation*}
For any $\bar{E}\in \hat{\mathcal{R}}$, denote
\begin{equation*}\begin{split}
\bar{c}(\bar{E})&=\max\{j\in\{2,3,\ldots, c(S+(\Delta^\vee|V^*))\}|\\
&\qquad\quad c(\mathcal{D}(S+(\Delta^\vee|V^*)|V),j, \bar{E})<
\lfloor c(\mathcal{D}(S+(\Delta^\vee|V^*)|V),2, \bar{E})\rfloor\}\\
&\qquad\quad\quad\in\{2,3,\ldots, c(S+(\Delta^\vee|V^*))\}.
\end{split}\end{equation*}

\begin{enumerate}
\item
$\gamma(\bar{E})=c(\mathcal{D}(S+(\Delta^\vee|V^*)|V),2, \bar{E})$ for any $\bar{E}\in \mathcal{F}(H\Op|\Delta)_1$.
\item
$\hat{m}=m$.
$\hat{\bar{m}}=\bar{m}$.
$\hat{\mathcal{R}}=\mathcal{R}$.
\item
A mapping $E:\{1,2,\ldots,m\}\rightarrow (\mathcal{C}-(\mathcal{C}/H))_1=\mathcal{F}(H\Op|\Delta)_1$ is compatible with $S$, if and only if, the following three conditions are satisfied:
\begin{enumerate}
\item
For any $\bar{E}\in\mathcal{F}(H\Op|\Delta)_1$,
$\sharp(\{1,2,\ldots, \bar{m}\}\cap E^{-1}(\bar{E}))=$\hfill\break$\lceil c(\mathcal{D}(S+(\Delta^\vee|V^*)|V),2, \bar{E})\rceil$.
\item
$E(\{\bar{m}+1, \bar{m}+2, \ldots, m\})=\mathcal{R}$.
\item
If $m-\bar{m}\geq 2$, then $\bar{c}(E(i))\leq\bar{c}(E(i+1))$ for any $i\in\{\bar{m}+1, \bar{m}+2, \ldots, m-1\}$.
\end{enumerate}
\item
Assume that a mapping $E:\{1,2,\ldots,m\}\rightarrow \mathcal{F}(H\Op|\Delta)_1$
is compatible with $S$.
\begin{enumerate}
\item
For any $\bar{E}\in\mathcal{F}(H\Op|\Delta)_1$,
$\sharp E^{-1}(\bar{E})=\lfloor c(\mathcal{D}(S+(\Delta^\vee|V^*)|V),2, \bar{E})\rfloor$, and
$\sharp(\{1,2,\ldots, \bar{m}\}\cap E^{-1}(\bar{E}))=\lceil c(\mathcal{D}(S+(\Delta^\vee|V^*)|V),2, \bar{E})\rceil$.

\item
$E(\{\bar{m}+1, \bar{m}+2,\ldots,m\})=\mathcal{R}$, and the mapping $E: \{\bar{m}+1, \bar{m}+2,\ldots,m\}\rightarrow \mathcal{R}$ induced by $E$ is bijective.
\item
$E(\{1,2,\ldots,m\})=\{\bar{E}\in\mathcal{F}(H\Op|\Delta)_1| c(\mathcal{D}(S+(\Delta^\vee|V^*)|V),2, \bar{E})>0\}$.
\end{enumerate}
\end{enumerate}
\end{lemma}

\section{The height inequalities}
\label{height inequalities}
We show the height inequalities.

In this section we consider the following objects: Let $V$ be any finite dimensional vector space over $\R$ with $\dim V\geq 2$; let $N$ be any lattice of $V$; let $H$ be any one-dimensional simplicial cone over the dual lattice $N^*$ of $N$ in the dual vector space $V^*$ of $V$; let $\mathcal{C}$ be any simplicial cone decomposition over $N^*$ in $V^*$ satisfying 
$\dim \mathcal{C}=\dim \Vect(|\mathcal{C}|)\geq 2$, $\mathcal{C}\Mx=\mathcal{C}^0$, $H\in\mathcal{C}_1$ and $\mathcal{C}=(\mathcal{C}/H)\Fc$; and let $S$ be any rational convex pseudo polyhedron over $N$ in $V$ satisfying $\dim(|\mathcal{D}(S|V)|)\geq 2$ and $|\mathcal{C}|\subset |\mathcal{D}(S|V)|$.

In this section we assume that $\mathcal{D}(S|V)\hat{\cap}\mathcal{F}(\Delta)$ is $H$-simple for any $\Delta\in\mathcal{C}\Mx$ and $\Ht(H,\mathcal{C},S)>0$.

Let $\gamma:(\mathcal{C}-(\mathcal{C}/H))_1\rightarrow\R$ denote the characteristic function of $(H, \mathcal{C}, S)$.
Denote
\begin{equation*}\begin{split}
m&=\sum_{\bar{E}\in(\mathcal{C}-(\mathcal{C}/H))_1}\lfloor \gamma(\bar{E})\rfloor\in\Z_+,\\
\bar{m}&=\sum_{\bar{E}\in(\mathcal{C}-(\mathcal{C}/H))_1}\lceil \gamma(\bar{E})\rceil\in\Z_0,\text{ and}\\
\end{split}\end{equation*}\begin{equation*}\begin{split}
\mathcal{R}&=\{ \bar{E}\in(\mathcal{C}-(\mathcal{C}/H))_1| \gamma(\bar{E})\not\in\Z\}\subset(\mathcal{C}-(\mathcal{C}/H))_1.
\end{split}\end{equation*}
@
We know $\Ht(H,S)\in(1/\Den(S/N))\Z_+$, $\Ht(H,\mathcal{C},S)\in(1/\Den(S/N))\Z_+$ and $\Ht(H,\mathcal{C},S)\leq\Ht(H,S)$.

We denote $\ell=\dim(\Vect(|\mathcal{C}|)^\vee|V^*)\in\Z_0$.

Consider any simplicial cone $\Theta$ over $N^*$ in $V^*$ satisfying $\Theta\subset|\mathcal{C}|$ and any $G\in\mathcal{F}(\Theta)_1$.
$\dim\Theta\leq\dim\Vect(|\mathcal{C}|)=\dim V-\ell$.
$S+(\Theta^\vee|V^*)$ is a rational conex pseudo polyhedron over $N$ in $V$.
$\Stab(S+(\Theta^\vee|V^*))=\Theta^\vee|V^*$.
$G\subset |\mathcal{D}(S+(\Theta^\vee|V^*)|V)|=\Stab(S+(\Theta^\vee|V^*))^\vee|V=\Theta\subset|\mathcal{C}|\subset|\mathcal{D}(S|V)|\subset\Vect(|\mathcal{D}(S|V)|)$.
$\mathcal{D}(S+(\Theta^\vee|V^*)|V)= \mathcal{D}(S|V)\hat{\cap}\mathcal{F}(\Theta)$.
$\Stab(S+(\Theta^\vee|V^*))\cap(-\Stab(S+(\Theta^\vee|V^*)))=
\Vect(|\mathcal{D}(S+(\Theta^\vee|V^*)|V)|)^\vee|V^*=\Vect(\Theta)^\vee|V^*\supset\Vect(|\mathcal{C}|)^\vee|V^*$.

If $\dim\Theta=\dim\Vect(|\mathcal{C}|)$, then
$\Vect(|\mathcal{D}(S+(\Theta^\vee|V^*)|V)|)^\vee|V^*=
\Vect(|\mathcal{C}|)^\vee|V^*$ and
$\dim(\Vect(|\mathcal{D}(S+(\Theta^\vee|V^*)|V)|)^\vee|V^*)=\ell$.

\begin{theorem}
\label{heightinequality}
Consider any compatible mapping
$$E:\{1,2,\ldots,m\}\rightarrow (\mathcal{C}-(\mathcal{C}/H))_1$$
with $S$.

Let
$$\mathcal{B}=\mathcal{B}(V^*, N^*, H, \mathcal{C}, m, E)$$
be the basic subdivision associated with the sextuplet $(V^*, N^*, H, \mathcal{C}, m, E)$.

Let $s=s(V^*, N^*, H, \mathcal{C}, m, E)$, 
$H=H(V^*, N^*, H, \mathcal{C}, m, E)$, \hfill\break and
$\mathcal{B}=\mathcal{B}(V^*, N^*, H, \mathcal{C}, m, E)$.
We have three mappings
$$s:\{0,1,\ldots,m\}\times(\mathcal{C}-(\mathcal{C}/H))_1\rightarrow \Z_0,$$
$$H:\{1,2,\ldots,m+1\}\rightarrow 2^{V^*},$$
$$\mathcal{B}:\{1,2,\ldots,m+1\}\rightarrow 2^{2^{V^*}}.$$

\begin{enumerate}
\item
$\bar{m}\in\Z_0$. $m\in\Z_+$. $\bar{m}\leq m$. $m-\bar{m}=\sharp\mathcal{R}$.
\item
For any $\bar{E}\in(\mathcal{C}-(\mathcal{C}/H))_1$, $s(\bar{m}, \bar{E})=\lceil \gamma(\bar{E})\rceil$ and $s(m, \bar{E})=\lfloor \gamma(\bar{E})\rfloor$.
\item
$\mathcal{B}$ is an iterated barycentric subdivision of $\mathcal{C}$, and it is a simplicial cone decomposition over $N^*$ in $V^*$. $|\mathcal{B}|=|\mathcal{C}|\subset\Stab(S)^\vee|V=|\mathcal{D}(S|V)|$. 
$\dim \mathcal{B}=\dim \Vect(|\mathcal{B}|)=\dim\mathcal{C}$. 
$\Vect(|\mathcal{B}|)=\Vect(|\mathcal{C}|)$.
$\mathcal{B}\Mx=\mathcal{B}^0$.
\item
Consider any $i\in\{1,2,\ldots,m+1\}$.
\begin{enumerate}
\item
$\mathcal{B}(i)$ is a simplicial cone decomposition over $N^*$ in $V^*$.
$\mathcal{B}(i)\subset\mathcal{B}$. 
$|\mathcal{B}(i)|\subset|\mathcal{B}|$. 
$\dim\mathcal{B}(i)=\dim \Vect(|\mathcal{B}(i)|)=\dim\mathcal{C}$.
$\Vect(|\mathcal{B}(i)|)= \Vect(|\mathcal{C}|)$.
$\mathcal{B}(i)\Mx=\mathcal{B}(i)^0$.
$H(i)\in\mathcal{B}(i)_1$.
$\mathcal{B}(i)=( \mathcal{B}(i)/H(i))\Fc$.
\item
For any $\Theta\in\mathcal{B}(i)$,  $\Theta^\vee|V^*$ is a rational polyhedral cone over $N$ in $V$, $\dim \Theta^\vee|V^*=\dim V$, $S+(\Theta^\vee|V^*)$ is a rational convex pseudo polyhedron over $N$ in $V$, $\Stab(S+(\Theta^\vee|V^*))=\Theta^\vee|V^*$, $\mathcal{D}(S+(\Theta^\vee|V^*)|V)=\mathcal{D}(S|V)\hat{\cap}\mathcal{F}(\Theta)$, and $\mathcal{D}(S|V)\hat{\cap}\mathcal{F}(\Theta)$ is semisimple.
\item
\emph{[The height inequality]}
$\Ht(H(i), \mathcal{B}(i), S)<\Ht(H,\mathcal{C},S)$.
\end{enumerate}
\item
Consider any $i\in\{1,2,\ldots,\bar{m}\}$.

For any $\Theta\in\mathcal{B}(i)$,
$\mathcal{D}(S|V)\hat{\cap}\mathcal{F}(\Theta)= \mathcal{F}(\Theta)$ and
$S+(\Theta^\vee|V^*)=\{a\}+(\Theta^\vee|V^*)$ for any $a\in\mathcal{V}(\mathcal{C}, S)$ satisfying $\langle b_{H/N^*},a\rangle=\max\mathcal{H}(H, \mathcal{C}, S)$.

$\Ht(H(i), \mathcal{B}(i), S)=0$.
\item
For any $\Theta\in\mathcal{B}(\bar{m}+1)\Mx$,
$\mathcal{D}(S|V)\hat{\cap}\mathcal{F}(\Theta)$
is $H(\bar{m}+1)$-simple.
\item
Consider any $\Delta\in\mathcal{C}\Mx$. 
$\Ht(H,S+(\Delta^\vee|V^*))=\Ht(H,\mathcal{C},S)$, if and only if, $\gamma(E)>0$ for some $E\in\mathcal{F}(H\Op|\Delta)_1$
\item
Consider any $\Delta\in\mathcal{C}\Mx$ such that $\gamma(\bar{E})=0$ for any $\bar{E}\in\mathcal{F}(H\Op|\Delta)_1$.

$\Ht(H,S+(\Delta^\vee|V^*))<\Ht(H,\mathcal{C},S)$.

$\Delta\subset|\mathcal{B}(m+1)|$.
$\mathcal{B}\backslash\Delta=\mathcal{B}(m+1)\backslash\Delta=\mathcal{F}(\Delta)$.

For any $i\in\{1,2,\ldots,m\}$,
$|\mathcal{B}(i)|\cap\Delta\in \mathcal{F}(H\Op|\Delta)$, $|\mathcal{B}(i)|\cap\Delta\neq H\Op|\Delta$, and 
$\mathcal{B}(i)\backslash\Delta=\mathcal{F}(|\mathcal{B}(i)|\cap\Delta)$.
\end{enumerate}

Below we consider any $\Delta\in\mathcal{C}\Mx$ such that $\gamma(E)>0$ for some $E\in\mathcal{F}(H\Op|\Delta)_1$.
By $7$ we know $\Ht(H,S+(\Delta^\vee|V^*))=\Ht(H,\mathcal{C},S)>0$.

Note that $\mathcal{D}(S+(\Delta^\vee|V^*)|V)=\mathcal{D}(S|V)\hat{\cap}\mathcal{F}(\Delta)$ is $H$-simple and the characteristic number $c(S+(\Delta^\vee|V^*))$ of $S+(\Delta^\vee|V^*)$ satisfies $c(S+(\Delta^\vee|V^*))\geq 2$.
Let $\bar{\mathcal{D}}( S+(\Delta^\vee|V^*)|V)^1=\{\bar{\Lambda}\in\mathcal{D}( S+(\Delta^\vee|V^*)|V)^1|\bar{\Lambda}^\circ\subset\Delta^\circ\}\cup\{H\Op|\Delta\}$ denote the $H$-skeleton of $\mathcal{D}( S+(\Delta^\vee|V^*)|V)$.
$\sharp\mathcal{D}( S+(\Delta^\vee|V^*)|V)^0=\sharp\bar{\mathcal{D}}( S+(\Delta^\vee|V^*)|V)^1=c(S+(\Delta^\vee|V^*))\geq 2$.

We consider the $H$-order on $\mathcal{D}(S+(\Delta^\vee|V^*)|V)^0$.
Let
$$\Lambda:\{1,2,\ldots, c(S+(\Delta^\vee|V^*))\}\rightarrow \mathcal{D}( S+(\Delta^\vee|V^*)|V)^0,$$
denote the unique bijective mapping preserving the $H$-order.

We consider the $H$-order on $\bar{\mathcal{D}}( S+(\Delta^\vee|V^*)|V)^1$.
Let
$$\bar{\Lambda}:\{1,2,\ldots, c(S+(\Delta^\vee|V^*))\}\rightarrow \bar{\mathcal{D}}(S+(\Delta^\vee|V^*)|V)^1,$$
denote the unique bijective mapping preserving the $H$-order.

There exists uniquely a bijective mapping $A: \{1,2,\ldots, c(S+(\Delta^\vee|V^*))\}\rightarrow\mathcal{F}(S+(\Delta^\vee|V^*)|V)_\ell$ satisfying $\Lambda(i)=\Delta(A(i), S+(\Delta^\vee|V^*)|V)$ for any $i\in\{1,2,\ldots, c(S+(\Delta^\vee|V^*))\}$.
We take the bijective mapping $A: \{1,2,\ldots, c(S+(\Delta^\vee|V^*))\}\rightarrow\mathcal{F}(S+(\Delta^\vee|V^*)|V)_\ell$ satisfying $\Lambda(i)=\Delta(A(i), S+(\Delta^\vee|V^*)|V)$ for any $i\in\{1,2,\ldots, c(S+(\Delta^\vee|V^*))\}$.

For any $i\in\{1,2,\ldots, c(S+(\Delta^\vee|V^*))\}$ and any $\bar{E}\in\mathcal{F}(H\Op|\Delta)_1$, we can consider the structure constant
$c(\mathcal{D}( S+(\Delta^\vee|V^*)|V),i, \bar{E})\in\Q_0$ of $\mathcal{D}( S+(\Delta^\vee|V^*)|V)$ corresponding to the pair $(i, \bar{E}).$

Denote
\begin{equation*}\begin{split}
\hat{m}&=\sum_{ \bar{E}\in\mathcal{F}(H\Op|\Delta)_1}\lfloor c(\mathcal{D}( S+(\Delta^\vee|V^*)|V),2, \bar{E})\rfloor\in\Z_+,\\
\hat{\bar{m}}&=\sum_{ \bar{E}\in\mathcal{F}(H\Op|\Delta)_1}\lceil c(\mathcal{D}( S+(\Delta^\vee|V^*)|V),2, \bar{E})\rceil\in\Z_0,\text{ and}\\
\hat{\mathcal{R}}&=\{ \bar{E}\in\mathcal{F}(H\Op|\Delta)_1|
c(\mathcal{D}( S+(\Delta^\vee|V^*)|V),2, \bar{E})\not\in\Z\}
\subset\mathcal{F}(H\Op|\Delta)_1.
\end{split}\end{equation*}

\begin{enumerate}
\setcounter{enumi}{9}
\item
$\hat{m}\in\Z_+$.
$\hat{\bar{m}}\in\Z_0$.
$\hat{\bar{m}}\leq \hat{m}$.
$\hat{m}=\sharp E^{-1}(\mathcal{F}(H\Op|\Delta)_1)$,
$\hat{\bar{m}}=\sharp(\{1,2,\ldots,\bar{m}\}\cap E^{-1}(\mathcal{F}(H\Op|\Delta)_1))$, and 
$\hat{\mathcal{R}}=\mathcal{R}\cap \mathcal{F}(H\Op|\Delta)_1$.
\end{enumerate}

Let $\hat{\tau}:\{1,2,\ldots, \hat{m}\}\rightarrow\{1,2,\ldots,m\}$ be the unique injective mapping preserving the order and satisfying 
$\hat{\tau}(\{1,2,\ldots, \hat{m}\})= E^{-1}(\mathcal{F}(H\Op|\Delta)_1)$.
Let $\hat{E}:\{1,2,\ldots, \hat{m}\}\rightarrow \mathcal{F}(H\Op|\Delta)_1$ be the unique mapping satisfying $\iota\hat{E}=E\hat{\tau}$ where $\iota: \mathcal{F}(H\Op|\Delta)_1)\rightarrow (\mathcal{C}-(\mathcal{C}/H))_1$ denotes the inclusion mapping.

Let
$\hat{\mathcal{B}}=\mathcal{B}(V^*, N^*, H, \mathcal{F}(\Delta), \hat{m}, \hat{E})\subset 2^{V^*}$, 
$\hat{G}=G(V^*, N^*, H, \mathcal{F}(\Delta), \hat{m}, \hat{E})$,\hfill\break 
$\hat{H}=H(V^*, N^*, H, \mathcal{F}(\Delta), \hat{m}, \hat{E})$, and $\hat{\mathcal{B}}=\mathcal{B}(V^*, N^*, H, \mathcal{F}(\Delta), \hat{m}, \hat{E})$.

$$\hat{G}:\{1,2,\ldots,\hat{m}\}\rightarrow 2^{V^*},$$
$$\hat{H}:\{1,2,\ldots,\hat{m}+1\}\rightarrow 2^{V^*},$$
$$\hat{\mathcal{B}}:\{1,2,\ldots,\hat{m}+1\}\rightarrow 2^{2^{V^*}}.$$

For any $i\in\{1,2,\ldots,\hat{m}+1\}$, we denote
\begin{equation*}\begin{split}
\Theta(i)&=|\hat{\mathcal{B}}(i)|\subset\Delta,\text{ and}\\
\bar{\Theta}(i)&= |\hat{\mathcal{B}}(i)-(\hat{\mathcal{B}}(i)/\hat{H}(i))|\subset\Theta(i).
\end{split}\end{equation*}

Denote
\begin{equation*}\begin{split}
\bar{c}(i)&=\max\{j\in\{2,3,\ldots,c(S+(\Delta^\vee|V^*))\}|\\
&\qquad\quad c(\mathcal{D}( S+(\Delta^\vee|V^*)|V),j, \hat{E}(i))<
\lfloor c(\mathcal{D}( S+(\Delta^\vee|V^*)|V),2, \hat{E}(i))\rfloor\}\\
&\qquad\quad\quad\in\{2,3,\ldots,c(S+(\Delta^\vee|V^*))\}
\end{split}\end{equation*}
for any $i\in\{\hat{\bar{m}}+1, \hat{\bar{m}}+2,\ldots,\hat{m}\}$.

\begin{enumerate}
\setcounter{enumi}{10}
\item
$\hat{E}$ is compatible with $S$.
\item
If $\hat{m}-\hat{\bar{m}}\geq 1$, then
$\bar{c}(i)\in\{2,3,\ldots,c(S+(\Delta^\vee|V^*))\}$ for any $i\in\{\hat{\bar{m}}+1, \hat{\bar{m}}+2,\ldots,\hat{m}\}$.
If $\hat{m}-\hat{\bar{m}}\geq 2$, then 
$\bar{c}(i)\leq\bar{c}(i+1)$  for any $i\in\{\hat{\bar{m}}+1, \hat{\bar{m}}+2,\ldots, \hat{m}-1\}$.
\item
$(H\Op|\Delta)\cap\bar{\Lambda}(2)\in\mathcal{F}(H\Op|\Delta)$.
$\mathcal{F}((H\Op|\Delta)\cap\bar{\Lambda}(2))_1=\{\bar{E}\in\mathcal{F}(H\Op|\Delta)_1| c(\mathcal{D}(S+(\Delta^\vee|V^*)|V),2, \bar{E})=0\}$.
$\mathcal{F}((H\Op|\Delta)\cap\bar{\Lambda}(2))_1\cup \hat{E}(\{1,2,\ldots,\hat{m}\})=\mathcal{F}(H\Op|\Delta)_1$.
$\mathcal{F}((H\Op|\Delta)\cap\bar{\Lambda}(2))_1\cap \hat{E}(\{1,2,\ldots,\hat{m}\})=\emptyset$.
\item
For any $i\in\{1,2,\ldots, \hat{m}+1\}$,
$\hat{H}(i)\in\mathcal{F}(\Theta(i))_1$,
$\bar{\Theta}(i)=\hat{H}(i)\Op|\Theta(i) \in\mathcal{F}(\Theta(i))^1$,
$\dim\Theta(i)=\dim\Vect(|\mathcal{C}|)$, and\hfill\break
$\dim\Vect(|\mathcal{D}(S+(\Theta(i)^\vee|V^*)|V)|)^\vee|V^*=\ell$.

\item
Consider any $i\in\{1,2,\ldots, \hat{\bar{m}}\}$.
$\mathcal{D}(S|V)\hat{\cap}\mathcal{F}(\Theta(i))= \mathcal{F}(\Theta(i))$.
$S+(\Theta(i)^\vee$\hfill\break$|V^*)=A(1)+(\Theta(i)^\vee|V^*)$.
$\Ht(\hat{H}(i),S+(\Theta(i)^\vee|V^*))=0<\Ht(H,S+(\Delta^\vee|V^*))$.
For any $j\in\{2,3,\ldots,c(S+(\Delta^\vee|V^*))\}$, $\bar{\Lambda}(j)\cap(\Theta(i)^\circ\cup\bar{\Theta}(i)^\circ)=\emptyset$.
\item
$\mathcal{D}(S|V)\hat{\cap}\mathcal{F}(\Theta(\hat{\bar{m}}+1))$
is $\hat{H}(\hat{\bar{m}}+1)$-simple.
\item
Assume $\hat{m}\neq\hat{\bar{m}}$.
Consider any $i\in\{\hat{\bar{m}}+1, \hat{\bar{m}}+2, \ldots, \hat{m}\}$.

$\mathcal{D}(S|V)\hat{\cap}\mathcal{F}(\Theta(i))$ is semisimple.

$\mathcal{F}(S+(\Theta(i)^\vee|V^*))_\ell=
\{A(j)|j\in\{1,2,\ldots, \bar{c}(i)\}\}$.

Take any point $a(j)\in A(j)$ for any $j\in\{1,2,\ldots, \bar{c}(i)\}$.
For any $j\in\{2,3,\ldots, \bar{c}(i)\}$,
$0<\langle b_{\hat{H}(i)/N^*},a(j-1)-a(j)\rangle<\langle b_{H/N^*},a(j-1)-a(j)\rangle$.

$0<\Ht(\hat{H}(i), S+(\Theta(i)^\vee|V^*))<\Ht(H, S+(\Delta^\vee|V^*))$.

For any $j\in\{2,3,\ldots,c(S+(\Delta^\vee|V^*))\}$,
$j\in\{2,3,\ldots, \bar{c}(i)\}
\Leftrightarrow
c(\mathcal{D}(S+(\Delta^\vee|V^*)|V), j,\hat{E}(i))<
\lfloor c(\mathcal{D}(S+(\Delta^\vee|V^*)|V),2,\hat{E}(i))\rfloor
\Leftrightarrow
\lceil c(\mathcal{D}(S+(\Delta^\vee|V^*)$\hfill\break$|V),2,\hat{E}(i))\rceil<
c(\mathcal{D}(S+(\Delta^\vee|V^*)|V),j,\hat{E}(i))<
\lfloor c(\mathcal{D}(S+(\Delta^\vee|V^*)|V),2,\hat{E}(i))\rfloor
$\break$\Leftrightarrow
\bar{\Lambda}(j)\cap(\hat{G}(i)+\hat{H}(i))^\circ\neq\emptyset
\Leftrightarrow
\bar{\Lambda}(j)\cap\Theta(i)^\circ\neq\emptyset
\Leftrightarrow
\bar{\Lambda}(j)\cap(\Theta(i)^\circ\cup\bar{\Theta}(i)^\circ)\neq\emptyset
$.

\item
$\mathcal{D}(S|V)\hat{\cap}\mathcal{F}(\Theta(\hat{m}+1))$ is semisimple.

$\mathcal{F}(S+(\Theta(\hat{m}+1)^\vee|V^*))_\ell=
\{A(j)|j\in\{k,k+1,\ldots, \hat{m}+1\}\}$ for some $k\in\{2,3,\ldots, \hat{m}+1\}$.

If $\hat{m}=\hat{\bar{m}}$, then
$\mathcal{F}(S+(\Theta(\hat{m}+1)^\vee|V^*))_\ell=
\{A(j)|j\in\{2,3,\ldots, \hat{m}+1\}\}$.

$\Ht(\hat{H}(\hat{m}+1), S+(\Theta(\hat{m}+1)^\vee|V^*))<\Ht(H, S+(\Delta^\vee|V^*))$.
\item The following three conditions are equivalent:
\begin{enumerate}
\item
$\hat{\mathcal{B}}$ is a subdivision of $\mathcal{D}(S+(\Delta^\vee|V^*)|V)$.
\item
$c(S+(\Delta^\vee|V^*)|V)=2$ and $\hat{m}=\bar{\hat{m}}$.
\item
$c(S+(\Delta^\vee|V^*)|V)=2$ and $c(\mathcal{D}(S+(\Delta^\vee|V^*)|V),2, \bar{E})\in\Z$ for any $\bar{E}\in\mathcal{F}(H\Op|\Delta)_1$.
\end{enumerate}
\end{enumerate}

Below we assume that $\hat{\mathcal{B}}$ is a subdivision of $\mathcal{D}(S+(\Delta^\vee|V^*)|V)$.

\begin{enumerate}
\setcounter{enumi}{19}

\item
$\hat{m}=\hat{\bar{m}}\geq 1$. $c(S+(\Delta^\vee|V^*))=2$. $\bar{\Lambda}(1)=\bar{\Theta}(1)=H\Op|\Delta$.
$\Lambda(1)\cap\Lambda(2)=\bar{\Lambda}(2)=\bar{\Theta}(\hat{m}+1)$.
$\Lambda(1)=\cup_{i\in\{1,2,\ldots,\hat{m}\}}\Theta(i)$.
$\Lambda(2)=\Theta(\hat{m}+1)$.
\item
$\{\Theta\in\hat{\mathcal{B}}|\Theta^\circ\subset\Delta^\circ\}$\hfill\break
$=\{\Theta(i)|i\in\{1,2,\ldots,\hat{m}+1\}\}\cup\{\bar{\Theta}(i)|i\in\{2,3,\ldots,\hat{m}\}\}\cup\{\bar{\Theta}(\hat{m}+1)\}$.

\noindent $\{\Theta\in\hat{\mathcal{B}}|\Theta^\circ\subset\Delta^\circ,$ 
The unique element $\Lambda\in \mathcal{D}(S+(\Delta^\vee|V^*)|V)$ satisfying $\Theta^\circ\subset\Lambda^\circ$ satisfies $\dim\Lambda=\dim\Delta-1\}
=\{\bar{\Theta}(\hat{m}+1)\}$.
\item
$H=\hat{H}(\hat{m}+1)\in\mathcal{F}(\Theta(\hat{m}+1))_1$.
$H\in\mathcal{F}(\Delta)_1$.
$\bar{\Theta}(\hat{m}+1)\in\mathcal{F}(\Theta(\hat{m}+1))^1$.
$\bar{\Theta}(\hat{m}+1)+ H= \Theta(\hat{m}+1) \in\smash{\hat{\mathcal{B}}}\Mx$.
$\bar{\Theta}(\hat{m}+1)\cap H=\{0\}$.

$\bar{\Theta}(\hat{m}+1)\subset\Lambda(1)$.
$\bar{\Theta}(\hat{m}+1)\subset\Lambda(2)$.
$\Theta(\hat{m}+1)\not\subset\Lambda(1)$.
$\Theta(\hat{m}+1)\subset\Lambda(2)$.
\end{enumerate}

Let $a(1)\in A(1)$ and $a(2)\in A(2)$ be any points.
Let $F=A(2)+(\Vect(\bar{\Theta}(\hat{m}+1))^\vee|V^*)$. 
\begin{enumerate}
\setcounter{enumi}{22}
\item
$\bar{\Theta}(\hat{m}+1)\subset\Theta(\hat{m}+1)\subset\Delta\subset|\mathcal{D}(S|V)|$, and $S+(\bar{\Theta}(\hat{m}+1)^\vee|V^*)\supset S+(\Theta(\hat{m}+1)^\vee|V^*)\supset S+(\Delta^\vee|V^*)\supset S$.

$S+(\bar{\Theta}(\hat{m}+1)^\vee|V^*)=A(2)+(\bar{\Theta}(\hat{m}+1)^\vee|V^*)=\{a(2)\}+(\bar{\Theta}(\hat{m}+1)^\vee|V^*)$.
$S+(\Theta(\hat{m}+1)^\vee|V^*)=A(2)+ (\Theta(\hat{m}+1)^\vee|V^*)=\{a(2)\}+ (\Theta(\hat{m}+1)^\vee|V^*)$.
$S+(\Delta^\vee|V^*)=\Conv(A(1)\cup A(2))+(\Delta^\vee|V^*)=\Conv(\{a(1),a(2)\})+(\Delta^\vee|V^*)$.
\item
$F\in\mathcal{F}(S+(\bar{\Theta}(\hat{m}+1)^\vee|V^*))$.
$F\cap(S+(\Theta(\hat{m}+1)^\vee|V^*))\in\mathcal{F}(S+(\Theta(\hat{m}+1)^\vee|V^*))$.
$F\cap(S+(\Delta^\vee|V^*))\in\mathcal{F}(S+(\Delta^\vee|V^*))$.

$\Delta(F, S+(\bar{\Theta}(\hat{m}+1)^\vee|V^*)|V)=
\Delta(F\cap(S+(\Theta(\hat{m}+1)^\vee|V^*)), S+(\Theta(\hat{m}+1)^\vee|V^*)|V)=
\Delta(F\cap(S+(\Delta^\vee|V^*)), S+(\Delta^\vee|V^*)|V)=
\bar{\Lambda}(2)$.

$F=\{a(2)\}+ (\Vect(\bar{\Theta}(\hat{m}+1))^\vee|V^*)=\Affi(\{a(1),a(2)\})+(\Vect(\Delta)^\vee|V^*)$.
$F\cap(S+(\Theta(\hat{m}+1)^\vee|V^*))=A(2)+\Delta(\bar{\Theta}(\hat{m}+1), \Theta(\hat{m}+1)|V^*)=\{a(2)\}+ \R_0(a(1)-a(2))+(\Vect(\Delta)^\vee|V^*)$.
$F\cap(S+(\Delta^\vee|V^*))=\Conv(A(1)\cup A(2))= \Conv(\{a(1), a(2)\})+ (\Vect(\Delta)^\vee|V^*)$.
\item
$\langle b_{H/N^*},a(1)\rangle> \langle b_{H/N^*},a(2)\rangle$.

\noindent$\{\langle b_{H/N^*}, a\rangle|a\in F\}=\R$.

\noindent$\{\langle b_{H/N^*}, a\rangle|a\in F\cap(S+(\Theta(\hat{m}+1)^\vee|V^*))\}=
\{t\in\R|\langle b_{H/N^*},a(2)\rangle\leq t \}$.

\noindent$\{\langle b_{H/N^*}, a\rangle|a\in F\cap (S+(\Delta^\vee|V^*))\}=
\{t\in\R|\langle b_{H/N^*},a(2)\rangle\leq t\leq \langle b_{H/N^*},$
\hfill\break$a(1)\rangle\}$.
\item
\begin{equation*}\begin{split}
&\max\{\langle b_{H/N^*}, a\rangle|a\in F\cap (S+(\Delta^\vee|V^*))\}\\
&\qquad -\min\{\langle b_{H/N^*}, a\rangle|a\in F\cap (S+(\Delta^\vee|V^*))\}\\
=\:&\langle b_{H/N^*},a(1)\rangle- \langle b_{H/N^*},a(2)\rangle\\
=\:& \Ht(H, S+(\Delta^\vee|V^*))\\
\end{split}\end{equation*}
\end{enumerate}

Below we consider any rational convex pseudo polyhedrons $T$ and $U$ over $N$ in $V$ satisfying $T+U=S$.
\begin{enumerate}\setcounter{enumi}{26}\item
$\mathcal{D}(T|V)\hat{\cap}\mathcal{D}(U|V)= \mathcal{D}(S|V)$.
$\mathcal{D}(T+(\Delta^\vee|V^*)|V)= \mathcal{D}(T|V) \hat{\cap}\mathcal{F}(\Delta)$ is $H$-simple.
$\mathcal{D}(U+(\Delta^\vee|V^*)|V)= \mathcal{D}(U|V) \hat{\cap}\mathcal{F}(\Delta)$ is $H$-simple.
$\Ht(H, T+(\Delta^\vee|V^*))+\Ht(H, U+(\Delta^\vee|V^*))=\Ht(H, S+(\Delta^\vee|V^*))$.
\end{enumerate}

Let $\bar{\mathcal{D}}(T+(\Delta^\vee|V^*)|V)^1$ and $\bar{\mathcal{D}}(U+(\Delta^\vee|V^*)|V)^1$ denote the $H$-skeleton of $\mathcal{D}(T+(\Delta^\vee|V^*)|V)$ and $\mathcal{D}(U+(\Delta^\vee|V^*)|V)$ respectively.
\begin{enumerate}\setcounter{enumi}{27}\item
$\bar{\mathcal{D}}(T+(\Delta^\vee|V^*)|V)^1\cup\bar{\mathcal{D}}(U+(\Delta^\vee|V^*)|V)^1=\bar{\mathcal{D}}(S+(\Delta^\vee|V^*)|V)^1$.
\item
For any $i\in\{1,2,\ldots,\hat{\bar{m}}\}$, $\Ht(\hat{H}(i), T+(\Theta(i)^\vee|V^*))=0$.
\item
Assume $\hat{m}\neq\hat{\bar{m}}$. Consider any $i\in\{\hat{\bar{m}}+1, \hat{\bar{m}}+2,\ldots,\hat{m}\}$. If $0<\Ht(\hat{H}(i),$\break$T+(\Theta(i)^\vee|V^*))$, then $\Ht(\hat{H}(i), T+(\Theta(i)^\vee|V^*))<\Ht(H, T+(\Delta^\vee|V^*))$.
\item
$\Ht(\hat{H}(\hat{m}+1), T+(\Theta(\hat{m}+1)^\vee|V^*))\leq\Ht(H, T+(\Delta^\vee|V^*))$.
\item
Let $r_U=c(U+(\Delta^\vee|V^*))\in\Z_+$. Assume that the structure constant of $\mathcal{D}(U+(\Delta^\vee|V^*)|V)$ corresponding to the pair $(i,\bar{E})$ is an integer for any $i\in\{1,2,\ldots,r_U\}$ and any $\bar{E}\in\mathcal{F}(H\Op|\Delta)_1$. Then, $\Ht(\hat{H}(\hat{m}+1), T+(\Theta(\hat{m}+1)^\vee|V^*))=\Ht(H, T+(\Delta^\vee|V^*))$, if and only if, $\bar{\Lambda}(2)\not\in\bar{\mathcal{D}}(T+(\Delta^\vee|V^*)|V)^1$.
\end{enumerate}
\end{theorem}

\begin{proof}
We show only claim $17$. This is the most important. Claim $4$.(c) follows from claims $7, 8, 15, 17$ and $18$.

Assume $\hat{m}\neq\hat{\bar{m}}$. $\hat{m}-\hat{\bar{m}}\geq 1$.
Consider any $i\in\{\hat{\bar{m}}+1, \hat{\bar{m}}+2,\ldots, \hat{m}\}$.
Since $\Theta(i)\subset\Delta$ and $\mathcal{D}(S|V)\hat{\cap}\mathcal{F}(\Delta)$ is $H$-simple, $\mathcal{D}(S|V)\hat{\cap}\mathcal{F}(\Delta)$ is semisimple and we know $\mathcal{D}(S|V)\hat{\cap}\mathcal{F}(\Theta(i)) =\mathcal{D}(S|V)\hat{\cap}\mathcal{F}(\Delta) \hat{\cap}\mathcal{F}(\Theta(i))$ is semisimple by Lemma~\ref{propsimple}.10.

Let $\hat{s}=s(V^*, N^*,H,\mathcal{F}(\Delta),\hat{m},\hat{E})$.
$\hat{s}:\mathcal{F}(H\Op|\Delta)_1\times\{0,1,\ldots,\hat{m}\}\rightarrow \Z_0$.

Note that $\mathcal{D}(S+(\Theta(i)^\vee|V^*)|V)^0=(\mathcal{D}(S|V)\hat{\cap}\mathcal{F}(\Theta(i)))^0=\{\Lambda(j)\cap\Theta(i)|j\in\{1,2,\ldots,c(S+(\Delta^\vee|V^*))\}, \Lambda(j)^\circ\cap\Theta(i)^\circ\neq\emptyset\}$,

\begin{equation*}\begin{split}
\partial^H_-\Theta(i)&=\bar{\Theta}(i) =\sum_{\bar{E}\in\mathcal{F}(H\Op|\Delta)_1}\R_0(b_{\bar{E}/N^*}+\hat{s}(i-1, \bar{E})b_{H/N^*})\\
&=\sum_{\bar{E}\in\mathcal{F}(H\Op|\Delta)_1-\{\hat{E}(i)\}}\R_0(b_{\bar{E}/N^*}
 +\hat{s}(i-1, \bar{E})b_{H/N^*})\\
&\qquad\qquad\quad +\R_0(b_{\hat{E}(i)/N^*}+\hat{s}(i-1, \hat{E}(i))b_{H/N^*}),\\
\partial^H_+\Theta(i)&= \bar{\Theta}(i+1)=\sum_{\bar{E}\in\mathcal{F}(H\Op|\Delta)_1}\R_0(b_{\bar{E}/N^*}+\hat{s}(i, \bar{E})b_{H/N^*})\\
&=\sum_{\bar{E}\in\mathcal{F}(H\Op|\Delta)_1-\{\hat{E}(i)\}}\R_0(b_{\bar{E}/N^*}
 +\hat{s}(i-1,\bar{E})b_{H/N^*})\\
&\qquad\qquad\quad +\R_0(b_{\hat{E}(i)/N^*}+\hat{s}(i, \hat{E}(i))b_{H/N^*}),\text{ and}
\end{split}\end{equation*}
\begin{equation*}\begin{split}
\hat{s}(i-1, \hat{E}(i))&=\lceil c(\mathcal{D}(S+(\Delta^\vee|V^*)|V^*),2,\hat{E}(i))\rceil\\
&< c(\mathcal{D}(S+(\Delta^\vee|V^*)|V^*),2,\hat{E}(i))\leq
c(\mathcal{D}(S+(\Delta^\vee|V^*)|V^*),\bar{c}(i),\hat{E}(i))\\
&< \lfloor c(\mathcal{D}(S+(\Delta^\vee|V^*)|V^*),2,\hat{E}(i))\rfloor
= \hat{s}(i, \hat{E}(i)).
\end{split}\end{equation*} 

Consider any $j\in\{1,2,\cdots, \bar{c}(i)\}$ with $j\neq c(S+(\Delta^\vee|V^*))$. There exists a real number $r_j$ satisfying $ \hat{s}(i-1, \hat{E}(i))<r_j< \hat{s}(i, \hat{E}(i))$ and $ c(\mathcal{D}(S+(\Delta^\vee|V^*)|V^*),j,$\hfill\break$\hat{E}(i))\leq r_j\leq c(\mathcal{D}(S+(\Delta^\vee|V^*)|V^*),j+1,\hat{E}(i))$, and there exist a mapping
$t_j:\mathcal{F}(H\Op|\Delta)_1\rightarrow\R_+$ and a real number $u_j$ satisfying
$$\sum_{\bar{E}\in\mathcal{F}(H\Op|\Delta)_1}t_j(\bar{E}) \hat{s}(i-1,\bar{E})<u_j
<\sum_{\bar{E}\in\mathcal{F}(H\Op|\Delta)_1}t_j(\bar{E}) \hat{s}(i,\bar{E})\text { and}$$
\begin{equation*}\begin{split}
&\sum_{\bar{E}\in\mathcal{F}(H\Op|\Delta)_1}t_j(\bar{E}) c(\mathcal{D}(S+(\Delta^\vee|V^*)|V^*),j, \bar{E})<u_j\\
&\qquad<\sum_{\bar{E}\in\mathcal{F}(H\Op|\Delta)_1}t_j(\bar{E}) c(\mathcal{D}(S+(\Delta^\vee|V^*)|V^*),j+1, \bar{E}).
\end{split}\end{equation*} 

We take a mapping
$t_j:\mathcal{F}(H\Op|\Delta)_1\rightarrow\R_+$ and a real number $u_j$ satisfying
the above conditions.
We know
$(\sum_{\bar{E}\in\mathcal{F}(H\Op|\Delta)_1}t_j(\bar{E})b_{\bar{E}/N^*})+u_jb_{H/N^*}\in\Lambda(j)^\circ\cap\Theta(i)^\circ \neq\emptyset$.

We know that if $ \bar{c}(i)\neq c(S+(\Delta^\vee|V^*))$, then for any $j\in\{1,2,\cdots, \bar{c}(i)\}$, $\Lambda(j)^\circ\cap\Theta(i)^\circ \neq\emptyset$.

Consider the case $ \bar{c}(i)= c(S+(\Delta^\vee|V^*))$.
We denote $j=\bar{c}(i)= c(S+(\Delta^\vee|V^*))$.
There exists a real number $r_j$ satisfying $ \hat{s}(i-1, \hat{E}(i))<r_j< \hat{s}(i, \hat{E}(i))$ and $ c(\mathcal{D}(S+(\Delta^\vee|V^*)|V^*),j,\hat{E}(i))< r_j$, and there exist a mapping
$t_j:\mathcal{F}(H\Op|\Delta)_1\rightarrow\R_+$ and a real number $u_j$ satisfying
$$\sum_{\bar{E}\in\mathcal{F}(H\Op|\Delta)_1}t_j(\bar{E}) \hat{s}(i-1, \bar{E})<u_j
<\sum_{\bar{E}\in\mathcal{F}(H\Op|\Delta)_1}t_j(\bar{E}) \hat{s}(i,\bar{E})\text { and}$$
$$\sum_{\bar{E}\in\mathcal{F}(H\Op|\Delta)_1}t_j(\bar{E}) c(\mathcal{D}(S+(\Delta^\vee|V^*)|V^*),j, \bar{E})<u_j.$$
We take a mapping
$t_j:\mathcal{F}(H\Op|\Delta)_1\rightarrow\R_+$ and a real number $u_j$ satisfying
the above conditions.
We know
$(\sum_{\bar{E}\in\mathcal{F}(H\Op|\Delta)_1}t_j(\bar{E})b_{\bar{E}/N^*})+u_jb_{H/N^*}\in\Lambda(j)^\circ\cap\Theta(i)^\circ \neq\emptyset$.

We know that if $ \bar{c}(i)= c(S+(\Delta^\vee|V^*))$, then for any $j\in\{1,2,\cdots, \bar{c}(i)\}$, $\Lambda(j)^\circ\cap\Theta(i)^\circ \neq\emptyset$.

We know that for any $j\in\{1,2,\cdots, \bar{c}(i)\}$, $\Lambda(j)^\circ\cap\Theta(i)^\circ \neq\emptyset$.

Consider any $j\in\{1,2,\cdots,c(S+(\Delta^\vee|V^*))\}$ satisfying $\Lambda(j)^\circ\cap\Theta(i)^\circ \neq\emptyset$.
We know that there exists $\bar{E}\in\mathcal{F}(H\Op|\Delta)_1$ satisfying
$ c(\mathcal{D}(S+(\Delta^\vee|V^*)|V^*),j,\bar{E})<\hat{s}(i, \bar{E})$.
We take $\bar{E}\in\mathcal{F}(H\Op|\Delta)_1$ satisfying
$ c(\mathcal{D}(S+(\Delta^\vee|V^*)|V^*),j,\bar{E})<\hat{s}(i,\bar{E})$.

If $\bar{E}\not\in\hat{\mathcal{R}}$, then $ c(\mathcal{D}(S+(\Delta^\vee|V^*)|V^*),j,\bar{E})<\hat{s}(i,\bar{E})= c(\mathcal{D}(S+(\Delta^\vee|V^*)|V^*),2,\bar{E})$ and $j=1\leq\bar{c}(i)$.

If $\bar{E}\in\hat{\mathcal{R}}$, then there exists $k\in\{\hat{\bar{m}}+1, \hat{\bar{m}}+2,\ldots,\hat{m}\}$ satisfying $\bar{E}=\hat{E}(k)$. We take $k\in\{\hat{\bar{m}}+1, \hat{\bar{m}}+2,\ldots,\hat{m}\}$ satisfying $\bar{E}=\hat{E}(k)$.

Consider the case $k\leq i$. 
We have
$c(\mathcal{D}(S+(\Delta^\vee|V^*)|V^*),j, \hat{E}(k))= c(\mathcal{D}(S+(\Delta^\vee|V^*)|V^*),j,\bar{E})<\hat{s}(i, \bar{E})= \hat{s}(i, \hat{E}(k))=\lfloor c(\mathcal{D}(S+(\Delta^\vee|V^*)|V^*),2, \hat{E}(k))\rfloor $ and
$1\leq j\leq\bar{c}(k)\leq\bar{c}(i)$.

Consider the case $k>i$. We have
$c(\mathcal{D}(S+(\Delta^\vee|V^*)|V^*),j, \hat{E}(k))= c(\mathcal{D}(S+(\Delta^\vee|V^*)|V^*),j,\bar{E})<\hat{s}(i, \bar{E})= \hat{s}(i, \hat{E}(k))=\lceil c(\mathcal{D}(S+(\Delta^\vee|V^*)|V^*),2,\hat{E}(k))\rceil<
c(\mathcal{D}(S+(\Delta^\vee|V^*)|V^*),2, \hat{E}(k))$ and
$j=1\leq\bar{c}(i)$.

We know that for any $j\in\{1,2,\cdots,c(S+(\Delta^\vee|V^*))\}$ satisfying $\Lambda(j)^\circ\cap\Theta(i)^\circ \neq\emptyset$, $j\in\{1,2,\cdots, \bar{c}(i)\}$.

Therefore,
$\mathcal{D}(S+(\Theta(i)^\vee|V^*)|V)^0=\{\Lambda(j)\cap\Theta(i)|j\in\{1,2,\ldots,c(S+(\Delta^\vee|V^*))\},$\hfill\break$ \Lambda(j)^\circ\cap\Theta(i)^\circ\neq\emptyset\}
=\{\Lambda(j)\cap\Theta(i)|j\in\{1,2,\ldots,\bar{c}(i)\}\}$.
Since $\dim\Theta(i)=\dim\Delta=\dim\mathcal{C}$, we know
$\mathcal{F}( S+(\Theta(i)^\vee|V^*))_\ell=\{A(j) |j\in\{1,2,\ldots,\bar{c}(i)\}\}$.

For any $j\in\{1,2,\ldots,c(S+(\Delta^\vee|V^*))\}$ we take any point $a(j)\in A(j)$.

Consider any $ j\in\{2,3,\ldots,\bar{c}(i)\}$.

$\langle b_{H/N^*}, a(j-1)-a(j) \rangle>0$ and 
$ c(S+(\Delta^\vee|V^*), 2, \hat{E}(i)))\leq
c(S+(\Delta^\vee|V^*), j,$\hfill\break$ \hat{E}(i)))\leq
c(S+(\Delta^\vee|V^*), \bar{c}(i), \hat{E}(i)))<
\lfloor c(S+(\Delta^\vee|V^*), 2, \hat{E}(i))\rfloor$.
$0<\lfloor c(S+(\Delta^\vee|V^*), 2, \hat{E}(i))\rfloor- c(S+(\Delta^\vee|V^*), 2, \hat{E}(i))<1$.

By Lemma~\ref{propsimple}.$18$.(f) we have
\begin{equation*}\begin{split}
&\langle b_{\hat{H}(i)/N^*},a(j-1) \rangle-\langle b_{\hat{H}(i)/N^*}, a(j)\rangle
=\langle b_{\hat{H}(i)/N^*},a(j-1)-a(j)\rangle\\
=\:&\langle b_{\hat{E}(i)/N^*}+\hat{s}(i, \hat{E}(i))b_{H/N^*},a(j-1)-a(j)\rangle\\
=\:&\langle b_{\hat{E}(i)/N^*}, a(j-1) -a(j)\rangle
+\hat{s}(i, \hat{E}(i), i)\langle b_{H/N^*}, a(j-1)-a(j) \rangle\\
=\:&-c(S+(\Delta^\vee|V^*), j, \hat{E}(i))\langle b_{H/N^*}, a(j-1)-a(j) \rangle\\
\:&\qquad\qquad+\lfloor c(S+(\Delta^\vee|V^*), 2, \hat{E}(i))\rfloor\langle b_{H/N^*}, a(j-1)-a(j) \rangle\\
=\:&(\lfloor c(S+(\Delta^\vee|V^*), 2, \hat{E}(i))\rfloor- c(S+(\Delta^\vee|V^*), j, \hat{E}(i))) \langle b_{H/N^*}, a(j-1)-a(j) \rangle\\
>\:&0\text{, and}
\end{split}\end{equation*}
\begin{equation*}\begin{split}
&\langle b_{\hat{H}(i)/N^*},a(j-1)-a(j)\rangle\\
=\:&(\lfloor c(S+(\Delta^\vee|V^*), 2, \hat{E}(i))\rfloor- c(S+(\Delta^\vee|V^*), j, \hat{E}(i))) \langle b_{H/N^*}, a(j-1)-a(j) \rangle\\
\leq\:&(\lfloor c(S+(\Delta^\vee|V^*), 2, \hat{E}(i))\rfloor- c(S+(\Delta^\vee|V^*), 2, \hat{E}(i))) \langle b_{H/N^*}, a(j-1)-a(j) \rangle\\
<\:& \langle b_{H/N^*}, a(j-1)-a(j) \rangle.
\end{split}\end{equation*}

We know
\begin{equation*}\begin{split}
&\max\{\langle b_{\hat{H}(i)/N^*},a(j)\rangle|j\in\{1,2,\ldots,\bar{c}(i)\}\}
=\langle b_{\hat{H}(i)/N^*},a(1)\rangle,\\
&\min\{\langle b_{\hat{H}(i)/N^*},a(j)\rangle|j\in\{1,2,\ldots,\bar{c}(i)\}\}
=\langle b_{\hat{H}(i)/N^*},a(\bar{c}(i))\rangle,\\
&\Ht(\hat{H}(i), S+(\Theta(i)^\vee|V^*))
=\langle b_{\hat{H}(i)/N^*},a(1)\rangle-\langle b_{\hat{H}(i)/N^*}, a(\bar{c}(i))\rangle\\
=\:&\langle b_{\hat{H}(i)/N^*},a(1)- a(\bar{c}(i))\rangle
=\sum_{j=2}^{\bar{c}(i)} \langle b_{\hat{H}(i)/N^*},a(j-1)- a(j)\rangle\\
>\:&0\text{, and}\\
&\Ht(\hat{H}(i), S+(\Theta(i)^\vee|V^*))
=\sum_{j=2}^{\bar{c}(i)} \langle b_{\hat{H}(i)/N^*},a(j-1)- a(j)\rangle\\
<\:&\sum_{j=2}^{\bar{c}(i)} \langle b_{H/N^*},a(j-1)- a(j)\rangle
\leq\sum_{j=2}^{c(S+(\Delta^\vee|V^*))} \langle b_{H/N^*},a(j-1)- a(j)\rangle\\
=\:&\langle b_{H/N^*},a(1)- a(c(S+(\Delta^\vee|V^*)))\rangle
=\langle b_{H/N^*},a(1) \rangle-\langle b_{H/N^*}, a(c(S+(\Delta^\vee|V^*)))\rangle\\
=\:&\Ht(H,S+(\Delta^\vee|V^*)).
\end{split}\end{equation*}

Consider any $j\in\{2,3,\ldots,c(S+(\Delta^\vee|V^*))\}$.

By definition of $\bar{c}(i)$, we know
$j\in\{2,3,\ldots, \bar{c}(i)\}
\Leftrightarrow
c(\mathcal{D}(S+(\Delta^\vee|V^*)|V), j,$\hfill\break$\hat{E}(i))<
\lfloor c(\mathcal{D}(S+(\Delta^\vee|V^*)|V),2,\hat{E}(i))\rfloor$.
Since $\hat{E}(i)\in\hat{\mathcal{R}}$, 
$\lceil c(\mathcal{D}(S+(\Delta^\vee|V^*)|V),2,$\hfill\break$\hat{E}(i))\rceil<
c(\mathcal{D}(S+(\Delta^\vee|V^*)|V),2,\hat{E}(i))\leq
c(\mathcal{D}(S+(\Delta^\vee|V^*)|V),j,\hat{E}(i))$.
Therefore,
$c(\mathcal{D}(S+(\Delta^\vee|V^*)|V), j,\hat{E}(i))<
\lfloor c(\mathcal{D}(S+(\Delta^\vee|V^*)|V),2,\hat{E}(i))\rfloor
\Leftrightarrow
\lceil c(\mathcal{D}(S+(\Delta^\vee|V^*)|V),2,\hat{E}(i))\rceil<
c(\mathcal{D}(S+(\Delta^\vee|V^*)|V),j,\hat{E}(i))<
\lfloor c(\mathcal{D}(S+(\Delta^\vee|V^*)|V),2, $\break$\hat{E}(i))\rfloor$.
Since 
$\hat{G}(i)+\hat{H}(i)=
\R_0(b_{\hat{E}(i)/N^*}+\lceil c(\mathcal{D}(S+(\Delta^\vee|V^*)|V),2,\hat{E}(i))\rceil b_{H/N^*})+
\R_0(b_{\hat{E}(i)/N^*}+\lfloor c(\mathcal{D}(S+(\Delta^\vee|V^*)|V),2,\hat{E}(i))\rfloor b_{H/N^*})$,
we know
$\lceil c(\mathcal{D}(S+(\Delta^\vee|V^*)|V),$\hfill\break$2,\hat{E}(i))\rceil<
c(\mathcal{D}(S+(\Delta^\vee|V^*)|V),j,\hat{E}(i))<
\lfloor c(\mathcal{D}(S+(\Delta^\vee|V^*)|V),2,\hat{E}(i))\rfloor
\Leftrightarrow
\bar{\Lambda}(j)\cap(\hat{G}(i)+\hat{H}(i))^\circ\neq\emptyset$.
Note that $\{k\in\{1,2,\ldots,c(S+(\Delta^\vee|V^*))\}|\Lambda(k)\cap\Theta(i)^\circ\neq\emptyset\}=\{1,2,\ldots,\bar{c}(i)\}$, since $\mathcal{F}(S+(\Theta(i)^\vee|V^*))_\ell=
\{A(j)|j\in\{1,2,\ldots, \bar{c}(i)\}\}$.
Therefore, $j\in\{2,3,\ldots, \bar{c}(i)\}
\Leftrightarrow
\Lambda(j-1)\cap \Theta(i)^\circ\neq\emptyset$ and $\Lambda(j)\cap \Theta(i)^\circ\neq\emptyset$.
Since $\bar{\Lambda}(j)= \Lambda(j-1)\cap\Lambda(j)$, we know 
$\Lambda(j-1)\cap \Theta(i)^\circ\neq\emptyset$ and $\Lambda(j)\cap \Theta(i)^\circ\neq\emptyset
\Leftrightarrow
\bar{\Lambda}(j)\cap \Theta(i)^\circ\neq\emptyset$.

Assume $\bar{\Lambda}(j)\cap\bar{\Theta}(i)^\circ\neq\emptyset$.
Take any point $\omega\in\bar{\Lambda}(j)\cap\bar{\Theta}(i)^\circ$.
Let $\chi= b_{\hat{E}(i)/N^*}+ c(\mathcal{D}(S+(\Delta^\vee|V^*)|V),j,\hat{E}(i))b_{H/N^*}\in\bar{\Lambda}(j)$.
We know that there exists a real number $\epsilon$ with $0<\epsilon\leq 1$ such that
$(1-t)\omega+t\chi\in\bar{\Lambda}(j)\cap\Theta(i)^\circ$ for any real number $t$ with $0<t<\epsilon$, since 
$\lceil c(\mathcal{D}(S+(\Delta^\vee|V^*)|V),2,\hat{E}(i))\rceil<
c(\mathcal{D}(S+(\Delta^\vee|V^*)|V),j,\hat{E}(i))$.
It follows $\bar{\Lambda}(j)\cap\Theta(i)^\circ\neq\emptyset$.
Therefore,
$\bar{\Lambda}(j)\cap\Theta(i)^\circ\neq\emptyset
\Leftrightarrow
\bar{\Lambda}(j)\cap(\Theta(i)^\circ\cup\bar{\Theta}(i)^\circ)\neq\emptyset
$.
\end{proof}

\section{Upward subdivisions and the hard height inequalities}
\label{upward}

Let $V$ be any finite dimensional vector space over $\R$ with $\dim V\geq 2$; let $N$ be any lattice of $V$; and let $S$ be any rational convex pseudo polyhedron over $N$ in $V$ such that $\dim|\mathcal{D}(S|V)|\geq 2$.

By $\mathcal{HC}(V,N,S)$ we denote the set of all pairs $(H,\mathcal{C})$ of a one-dimensional simplicial cone $H$ over the dual lattice $N^*$ of $N$ in the dual vector space $V^*$ of $V$ and a simplicial cone decomposition $\mathcal{C}$ over $N^*$ in $V^*$ such that $\dim \mathcal{C}=\dim \Vect(|\mathcal{C}|)\geq 2$, $\mathcal{C}\Mx=\mathcal{C}^0$, $H\in\mathcal{C}_1$, $\mathcal{C}=(\mathcal{C}/H)\Fc$, $|\mathcal{C}|\subset |\mathcal{D}(S|V)|$ and $\mathcal{D}(S|V)\hat{\cap}\mathcal{F}(\Delta)$ is $H$-simple for any $\Delta\in\mathcal{C}\Mx$.

In this section we consider the case $\mathcal{HC}(V,N,S)\neq\emptyset$. We assume $\mathcal{HC}(V,N,S)\neq\emptyset$ below.

Note that $\Ht(H,\mathcal{C},S)\in(1/\Den(S/N))\Z_0$ for any $(H,\mathcal{C})\in\mathcal{HC}(V,N,S)$.
Therefore, for any infinite sequence $(H(i),\mathcal{C}(i)), i\in\Z_0$ of elements of $\mathcal{HC}(V,N,S)$ such that $\Ht(H(i),\mathcal{C}(i),S)\geq\Ht(H(i+1),\mathcal{C}(i+1),S)$ for any $i\in\Z_0$, there exists $i_0\in\Z_0$ such that $\Ht(H(i),\mathcal{C}(i),S)= \Ht(H(i_0),\mathcal{C}(i_0),S)$ for any $i\in\Z_0$ with $i\geq i_0$.

Consider any $(H,\mathcal{C})\in\mathcal{HC}(V,N,S)$.

By $\mathcal{SD}(H,\mathcal{C},S)$ we denote the set of all pairs $(M, F)$ of a non-negative integer $M\in\Z_0$ and a center sequence $F$ of $\mathcal{C}$ of length $M$ such that $\dim F(i)=2$ for any $i\in\{1,2,\ldots,M\}$, $F(i)\not\subset|\mathcal{C}-(\mathcal{C}/H)|$ for any $i\in\{1,2,\ldots,M\}$, and $\mathcal{C}* F(1)*F(2)*\cdots *F(M)$ is a subdivision of $\mathcal{D}(S|V)\hat{\cap}\mathcal{C}$.

Below, we use induction on $\Ht(H,\mathcal{C},S)$, we will show that $\mathcal{SD}(H,\mathcal{C},S)\neq\emptyset$, and we will define a non-empty subset $\mathcal{USD}(H,\mathcal{C},S)$ of $\mathcal{SD}(H,\mathcal{C},S)$. 

Consider any $(H,\mathcal{C})\in\mathcal{HC}(V,N,S)$ satisfying $\Ht(H,\mathcal{C},S)=0$.
By $0_{\mathcal{C}}$ we denote the unique center sequence of $\mathcal{C}$ of length $0$.
By Lemma~\ref{compatible2}.1 we know that $\mathcal{C}$ is a subdivision of $\mathcal{D}(S|V)\hat{\cap}\mathcal{C}$ and therefore $(0, 0_{\mathcal{C}})\in\mathcal{SD}(H,\mathcal{C},S)\neq\emptyset$.
We define $\{(0, 0_{\mathcal{C}})\}=\mathcal{USD}(H,\mathcal{C},S)$.
Obviously $\emptyset\neq\mathcal{USD}(H,\mathcal{C},S)\subset\mathcal{SD}(H,\mathcal{C},S)$.

Consider any $(H,\mathcal{C})\in\mathcal{HC}(V,N,S)$ satisfying $\Ht(H,\mathcal{C},S)>0$.

By induction hypothesis we can assume that a non-empty subset $\mathcal{USD}(\bar{H},\bar{\mathcal{C}},S)$ of $\mathcal{SD}(\bar{H},\bar{\mathcal{C}},S)$ is defined for any $(\bar{H},\bar{\mathcal{C}})\in \mathcal{HC}(V,N,S)$ satisfying $\Ht(\bar{H}, \bar{\mathcal{C}}, S)<\Ht(H,\mathcal{C},S)$.
Below we assume this claim.

The characteristic function $\gamma:(\mathcal{C}-(\mathcal{C}/H))_1\rightarrow\Q_0$ of $(H,\mathcal{C}, S)$ is defined.
Let $m=\sum_{\bar{E}\in(\mathcal{C}-(\mathcal{C}/H))_1}\lfloor\gamma(\bar{E})\rfloor\in\Z_+$ and $\bar{m}=\sum_{\bar{E}\in(\mathcal{C}-(\mathcal{C}/H))_1}\lceil\gamma(\bar{E})\rceil\in\Z_0$.
Obviously $\bar{m}\leq m$.

Let $\mathcal{E}$ denote the set of all compatible mappings $E:\{1,2,\ldots,m\}\rightarrow (\mathcal{C}-(\mathcal{C}/H))_1$ with $S$.
By Lemma~\ref{existcomp}.1 we know $\mathcal{E}\neq\emptyset$.
Consider any $E\in\mathcal{E}$.

Let $F=F(V^*, N^*, H,\mathcal{C},m,E)$ and  $\mathcal{B}=\mathcal{B}(V^*, N^*, H,\mathcal{C},m,E)$.
We have a mapping
$$F:\{1,2,\ldots,m\}\rightarrow 2^{V^*}.$$
By Lemma~\ref{basic1}.4 and 2 we know that $F$ is a center sequence of $\mathcal{C}$ of length $m$ such that $\dim F(i)=2$, $F(i)\subset E(i)+H$ and $F(i)\not\subset|\mathcal{C}-(\mathcal{C}/H)|$ for any $i\in\{1,2,\ldots,m\}$.
By definition $\mathcal{B}=\mathcal{C}* F(1)* F(2)*\ldots*F(m)$.
By Lemma~\ref{property of basic subdivisions}.1 we know that
$\mathcal{B}$ is a simplicial cone decomposition over $N$ in $V$,
$\dim \mathcal{B}=\dim \Vect(|\mathcal{B}|)=\dim\mathcal{C}$,
$\Vect(|\mathcal{B}|)=\Vect(|\mathcal{C}|)$,
$\mathcal{B}\Mx=\mathcal{B}^0$, 
$\mathcal{B}$ is an iterated barycentric subdivision of $\mathcal{C}$, and
$|\mathcal{B}|=|\mathcal{C}|$.

We have mappings
$$H=H(V^*, N^*, H,\mathcal{C},m,E):\{1,2,\ldots,m+1\}\rightarrow 2^{V^*},\text{ and}$$
$$\mathcal{B}=\mathcal{B}(V^*, N^*, H,\mathcal{C},m,E):\{1,2,\ldots,m+1\}\rightarrow 2^{2^{V^*}}.$$

By Lemma~\ref{basic1}.2, Lemma~\ref{property of basic subdivisions}.4 and Theorem~\ref{heightinequality}.4.$(c)$ we know that $H(i)$ is a one-dimensional simplicial cone over $N^*$ in $V^*$, $\mathcal{B}(i)$ is a simplicial cone decomposition over $N^*$ in $V^*$,
$\dim \mathcal{B}(i)=\dim \Vect(|\mathcal{B}(i)|)=\dim\mathcal{C}$,
$ \Vect(|\mathcal{B}(i)|)= \Vect(|\mathcal{C}|)$,
$\mathcal{B}(i)\Mx=\mathcal{B}(i)^0$, $H(i)\in\mathcal{B}(i)_1$, $\mathcal{B}(i)=(\mathcal{B}(i)/H(i))\Fc$, $\mathcal{B}(i)\subset\mathcal{B}$, $|\mathcal{B}(i)|\subset|\mathcal{B}|=|\mathcal{C}|\subset|\mathcal{D}(S|V)|$, and $\Ht(H(i), \mathcal{B}(i), S)<\Ht(H,\mathcal{C},S)$ for any $i\in\{1,2,\ldots,m+1\}$.

By Theorem~\ref{heightinequality}.5 we know that $\mathcal{B}\backslash |\mathcal{B}(i)|=\mathcal{B}(i)=\mathcal{D}(S|V)\hat{\cap}\mathcal{B}(i)$ and $\mathcal{B}\backslash |\mathcal{B}(i)|$ is a subdivision of $\mathcal{D}(S|V)\hat{\cap}\mathcal{B}(i)$ for any $i\in\{1,2,\ldots,\bar{m}\}$.

Consider any element $\mu\in\{\bar{m}, \bar{m}+1,\ldots, m+1\}$.

By induction we will show that there exist mappings
$$M: \{\bar{m},\bar{m}+1,\ldots,\mu\}\rightarrow \Z_+,\text{ and}$$
$$\bar{F}:\{1,2,\ldots, M(\mu)\}\rightarrow 2^{V^*}$$ satisfying the following conditions:
\begin{enumerate}
\item
$M(\bar{m})=m$ and $M(i-1)\leq M(i)$ for any $i\in\{\bar{m}+1,\bar{m}+2,\ldots, \mu\}$.
\item
$\bar{F}(j)=F(j)$ for any $j\in\{1,2,\ldots, m\}$
\item
$\bar{F}(j)$ is a simplicial cone over $N^*$ in $V^*$ and $\dim \bar{F}(j)=2$ for and any $j\in\{1,2,\ldots, M(\mu)\}$, and 
$\bar{F}$ is a center sequence of $\mathcal{C}$ of length $M(\mu)$.
\item
$\bar{F}(j)\subset|\mathcal{B}(i)|$ and $\bar{F} (j)\not\subset|\mathcal{B}(i)-(\mathcal{B}(i)/H(i))|$ for any $i\in\{\bar{m}+1,\bar{m}+2,\ldots, \mu\}$ and any $j\in\{M(i-1)+1,M(i-1)+2,\ldots, M(i)\}$.
\item
$\mathcal{C}*\bar{F}(1)*\bar{F}(2)*\cdots*\bar{F}(M(\mu))\backslash |\mathcal{B}(i)|$ is a subdivision of $\mathcal{D}(S|V)\hat{\cap}\mathcal{B}(i)$ for any $i\in\{1,2,\ldots,\mu\}$.
\item
Consider any  $i\in\{\bar{m}+1,\bar{m}+2,\ldots,\mu\}$.

Denote $\bar{\mathcal{C}}(i)= \mathcal{C}*\bar{F}(1)*\bar{F}(2)*\cdots*\bar{F}(M(i-1))\backslash |\mathcal{B}(i)|$.

Let $\bar{\bar{F}}(i): \{1,2,\ldots, M(i)-M(i-1)\}\rightarrow 2^{V^*}$ denote the mapping satisfying $\bar{\bar{F}}(i)(j)= \bar{F}(M(i-1)+j)$ for any $j\in\{1,2,\ldots, M(i)-M(i-1)\}$.

$|\bar{\mathcal{C}}(i)|=| \mathcal{B}(i)|$,
$(H(i), \bar{\mathcal{C}}(i))\in\mathcal{HC}(V,N,S)$, 
$\Ht(H(i), \bar{\mathcal{C}}(i) , S)<\Ht($\hfill\break$H, \mathcal{C}, S)$ and $(M(i)-M(i-1), \bar{\bar{F}}(i))\in\mathcal{USD}(H(i), \bar{\mathcal{C}}(i),S)$.
\end{enumerate}

Consider the case $\mu=\bar{m}$.
We define $M(\bar{m})=m$ and $\bar{F}(j)=F(j)$ for any $j\in\{1,2,\ldots, m\}$.
We know that mappings
$$M: \{\bar{m},\bar{m}+1,\ldots,\mu\}=\{\bar{m}\}\rightarrow\Z_+,$$
$$\bar{F}=F:\{1,2,\ldots, M(\mu)\}=\{1,2,\ldots, m\}\rightarrow 2^{V^*}$$ satisfy the above conditions.

Assume that $\mu>\bar{m}$ and there exist mappings
$$M: \{\bar{m},\bar{m}+1,\ldots,\mu-1\}\rightarrow\Z_+,$$
$$\bar{F}:\{1,2,\ldots, M(\mu-1)\}\rightarrow 2^{V^*}$$ satisfying the above conditions in which $\mu$ is replaced by $\mu-1$.
We take mappings $M: \{\bar{m},\bar{m}+1,\ldots,\mu-1\}\rightarrow\Z_+$, $\bar{F}:\{1,2,\ldots, M(\mu-1)\}\rightarrow 2^{V^*}$ satisfying the above conditions in which $\mu$ is replaced by $\mu-1$.

Let $\bar{\mathcal{C}}(\mu)=\mathcal{C}*\bar{F}(1)* \bar{F}(2)*\cdots*\bar{F}(M(\mu-1))\backslash|\mathcal{B}(\mu)|$.
We know that $\bar{\mathcal{C}}(\mu)=\mathcal{B}*\bar{F}(M(\bar{m})+1)* \bar{F}( M(\bar{m})+2)*\cdots*\bar{F}(M(\mu-1))\backslash|\mathcal{B}(\mu)|$,  $\bar{\mathcal{C}}(\mu)$ is a simplicial cone decomposition over $N^*$ in $V^*$,
$\dim \bar{\mathcal{C}}(\mu)=\dim \Vect(|\bar{\mathcal{C}}(\mu)|)=\dim\mathcal{C}$,
$\Vect(|\bar{\mathcal{C}}(\mu)|)= \Vect(|\mathcal{C}|)$,
$\bar{\mathcal{C}}(\mu)\Mx=\bar{\mathcal{C}}(\mu)^0$, 
$\bar{\mathcal{C}}(\mu)$ is a subdivision of $\mathcal{B}(\mu)$ and 
$|\bar{\mathcal{C}}(\mu)|= |\mathcal{B}(\mu)|$.
Since $\mathcal{D}(S|V)\hat{\cap}\mathcal{F}(\Theta)$ is semisimple for any $\Theta\in \mathcal{B}(\mu)\Mx$, we know that $\mathcal{D}(S|V)\hat{\cap}\mathcal{F}(\Theta)$ is semisimple for any $\Theta\in \bar{\mathcal{C}} (\mu)\Mx$

Note that
$$\bar{F}(j)\subset|\mathcal{C}-(\mathcal{C}/H)|\cup(\bigcup_{i\in\{1,2,\ldots,\mu-1\}}|\mathcal{B}(i)|)$$
for any $j\in\{ M(\bar{m})+1, M(\bar{m})+2,\ldots, M(\mu-1)\}$, 
$$(|\mathcal{C}-(\mathcal{C}/H)|\cup(\bigcup_{i\in\{1,2,\ldots,\mu-1\}}|\mathcal{B}(i)|))\cap|\mathcal{B}(\mu)|=| \mathcal{B}(\mu)-( \mathcal{B}(\mu)/H(\mu))|,$$
and $\mathcal{B}*\bar{F}(M(\bar{m})+1)* \bar{F}( M(\bar{m})+2)*\cdots*\bar{F}(M(\mu-1))\backslash(|\mathcal{C}-(\mathcal{C}/H)|\cup$\hfill\break$(\bigcup_{i\in\{1,2,\ldots,\mu-1\}}|\mathcal{B}(i)|))$ is a subdivision of $\mathcal{D}(S|V)$.

Let $\ell=\sharp\{j\in\{ M(\bar{m})+1, M(\bar{m})+2,\ldots, M(\mu-1)\}| \bar{F}(j)\subset| \mathcal{B}(\mu)-( \mathcal{B}(\mu)/H(\mu))|\}$ and let
$\tau:\{1,2,\ldots,\ell\}\rightarrow\{ M(\bar{m})+1, M(\bar{m})+2,\ldots, M(\mu-1)\}$ denote the unique injective mapping preserving the order and satisfying $\tau(\{1,2,\ldots,\ell\})=\{j\in\{ M(\bar{m})+1, M(\bar{m})+2,\ldots, M(\mu-1)\}| \bar{F}(j)\subset| \mathcal{B}(\mu)-( \mathcal{B}(\mu)/H(\mu))|\}$.
We know that $\bar{\mathcal{C}}(\mu)=\mathcal{B}(\mu)*\bar{F}\tau(1)* \bar{F}\tau(2)*\cdots*\bar{F}\tau(\ell)$.
Since $\bar{F}\tau(j)\subset|\mathcal{B}(\mu)-(\mathcal{B}(\mu)/H(\mu))|$ for any $j\in\{1,2,\ldots,\ell\}$, we know that $H(\mu)\in \bar{\mathcal{C}}(\mu)_1$, $\bar{\mathcal{C}}(\mu)=(\bar{\mathcal{C}}(\mu)/H(\mu))\Fc$, 
$|\bar{\mathcal{C}}(\mu)|=| \mathcal{B}(\mu)|$,
$|\bar{\mathcal{C}}(\mu)- (\bar{\mathcal{C}}(\mu)/H(\mu))|=
|\mathcal{B}(\mu)- (\mathcal{B}(\mu)/H(\mu))|$ and
$\Ht(H(\mu), \bar{\mathcal{C}}(\mu), S)=\Ht(H(\mu), \mathcal{B}(\mu), S)<\Ht(H,\mathcal{C},S)$.

Consider any $\Theta\in \bar{\mathcal{C}}(\mu)\Mx$.
$\mathcal{D}(S|V)\hat{\cap}\mathcal{F}(\Theta)$ is semisimple.
$\Theta\in \bar{\mathcal{C}}(\mu)/H(\mu)$, and $H(\mu)\Op|\Theta\in \bar{\mathcal{C}}(\mu)-( \bar{\mathcal{C}}(\mu)/H(\mu))$.
On the other hand we know
$\bar{\mathcal{C}}(\mu)-( \bar{\mathcal{C}}(\mu)/H(\mu))
=\bar{\mathcal{C}}(\mu)\backslash|\mathcal{B}(\mu)-( \mathcal{B}(\mu)/H(\mu))|
=\mathcal{B}*\bar{F}(M(\bar{m})+1)* \bar{F}( M(\bar{m})+2)*\cdots*\bar{F}(M(\mu-1))\backslash|\mathcal{B}(\mu)-( \mathcal{B}(\mu)/H(\mu))|
=(\mathcal{B}*\bar{F}(M(\bar{m})+1)* \bar{F}( M(\bar{m})+2)*\cdots*\bar{F}(M(\mu-1))\backslash(|\mathcal{C}-(\mathcal{C}/H)|\cup(\bigcup_{i\in\{1,2,\ldots,\mu-1\}}|\mathcal{B}(i)|)))\backslash|\mathcal{B}(\mu)-( \mathcal{B}(\mu)/H(\mu))|$, and $\bar{\mathcal{C}}(\mu)-( \bar{\mathcal{C}}(\mu)/H(\mu))$ is a subdivision of $\mathcal{D}(S|V)$.
Therefore $\mathcal{D}(S|V)\hat{\cap}\mathcal{F}( H(\mu)\Op|\Theta)= \mathcal{F}( H(\mu)\Op|\Theta)$, and we know that $\mathcal{D}(S|V)\hat{\cap}\mathcal{F}(\Theta)$ is $H(\mu)$-simple.

We know that $(H(\mu), \bar{\mathcal{C}}(\mu))\in\mathcal{HC}(V,N,S)$, $\Ht(H(\mu), \bar{\mathcal{C}}(\mu), S)<\Ht(H, \mathcal{C},$\hfill\break$ S)$, and a non-empty subset $\mathcal{USD}( H(\mu), \bar{\mathcal{C}}(\mu),S)$ of $\mathcal{SD}( H(\mu), \bar{\mathcal{C}}(\mu),S)$ is defined.

Take any element $(L, G)\in \mathcal{USD}( H(\mu), \bar{\mathcal{C}}(\mu),S)$. 
Put $M(\mu)=M(\mu-1)+L$, and put $\bar{F}(j)=G(j-M(\mu-1))$ for any $j\in\{M(\mu-1)+1,M(\mu-1)+2,\ldots,M(\mu)\}$.
We obtain extended mappings
$$M: \{\bar{m},\bar{m}+1,\ldots,\mu\}\rightarrow\Z_+,$$
$$\bar{F}:\{1,2,\ldots, M(\mu)\}\rightarrow 2^{V^*}$$ satisfying the above conditions.

By induction on $\mu$ we know that there exist mappings 
$$M: \{\bar{m},\bar{m}+1,\ldots,m+1\}\rightarrow\Z_+,$$
$$\bar{F}:\{1,2,\ldots, M(m+1)\}\rightarrow 2^{V^*}$$ satisfying the following conditions:
\begin{enumerate}
\item
$M(\bar{m})=m$ and $M(i-1)\leq M(i)$ for any $i\in\{\bar{m}+1,\bar{m}+2,\ldots, m+1\}$.
\item
$\bar{F}(j)=F(j)$ for any $j\in\{1,2,\ldots, m\}$
\item
$\bar{F}(j)$ is a simplicial cone over $N^*$ in $V^*$ and $\dim \bar{F}(j)=2$ for and any $j\in\{1,2,\ldots, M(m+1)\}$, and 
$\bar{F}$ is a center sequence of $\mathcal{C}$ of length $M(m+1)$.
\item
$\bar{F}(j)\subset|\mathcal{B}(i)|$ and $\bar{F} (j)\not\subset|\mathcal{B}(i)-(\mathcal{B}(i)/H(i))|$ for any $i\in\{\bar{m}+1,\bar{m}+2,\ldots, m+1\}$ and any $j\in\{M(i-1)+1,M(i-1)+2,\ldots, M(i)\}$.
\item
$\mathcal{C}*\bar{F}(1)*\bar{F}(2)*\cdots*\bar{F}(M(m+1))\backslash |\mathcal{B}(i)|$ is a subdivision of $\mathcal{D}(S|V)\hat{\cap}\mathcal{B}(i)$ for any $i\in\{1,2,\ldots,m+1\}$.
\item
Consider any $i\in\{\bar{m}+1,\bar{m}+2,\ldots,m+1\}$.

Denote $\bar{\mathcal{C}}(i)= \mathcal{C}*\bar{F}(1)*\bar{F}(2)*\cdots*\bar{F}(M(i-1))\backslash |\mathcal{B}(i)|$.

Let $\bar{\bar{F}}(i): \{1,2,\ldots, M(i)-M(i-1)\}\rightarrow 2^{V^*}$ denote the mapping satisfying $\bar{\bar{F}}(i)(j)= \bar{F}(M(i-1)+j)$ for any $j\in\{1,2,\ldots, M(i)-M(i-1)\}$.

$|\bar{\mathcal{C}}(i)|=|\mathcal{B}(i)|$,
$(H(i), \bar{\mathcal{C}}(i))\in\mathcal{HC}(V,N,S)$, 
$\Ht(H(i), S, \bar{\mathcal{C}}(i))<$\hfill\break$\Ht(H,\mathcal{C},S)$ and $(M(i)-M(i-1), \bar{\bar{F}}(i))\in\mathcal{USD}(H(i), \bar{\mathcal{C}}(i),S)$.
\end{enumerate}

Put
\begin{equation*}\begin{split}
\mathcal{USD}(E, H,\mathcal{C},S)&=
\{(M(m+1),\bar{F})|\\
&\qquad\quad M: \{\bar{m},\bar{m}+1,\ldots,m+1\}\rightarrow\Z_+,\\
&\qquad\quad \bar{F}:\{1,2,\ldots, M(m+1)\}\rightarrow 2^{V^*},\\
&\qquad\quad M\text{ and }\bar{F}\text{ satisfy the above conditions.}\}.
\end{split}\end{equation*}
We know that $\emptyset\neq\mathcal{USD}(E, H,\mathcal{C},S)\subset\mathcal{SD}(H,\mathcal{C},S)$.

Recall that $E\in\mathcal{E}$ is an arbitrary element.
Define
$$\mathcal{USD}(H,\mathcal{C},S)=\bigcup_{ E\in\mathcal{E}}\mathcal{USD}(E, H,\mathcal{C},S).$$
We know that $\emptyset\neq\mathcal{USD}(H,\mathcal{C},S)\subset\mathcal{SD}(H,\mathcal{C},S)$.

Consider any $(H,\mathcal{C})\in\mathcal{HC}(V,N,S)$ and any $(M, F)\in\mathcal{USD}(H,\mathcal{C},S)$.
We call $F$ an \emph{upward center sequence} of $(H,\mathcal{C}, S)$, and
we call $\mathcal{C}* F(1)*F(2)*\cdots *F(M)$ an \emph{upward subdivision} of $(H,\mathcal{C}, S)$.

\begin{theorem}
\label{usd1}
Assume $\mathcal{HC}(V,N,S)\neq\emptyset$, and consider any $(H,\mathcal{C})\in\mathcal{HC}(V,N,S)$.
\begin{enumerate}
\item
$\mathcal{USD}(H,\mathcal{C},S)\neq\emptyset$.
\item
For any $(M, F)\in\mathcal{USD}(H,\mathcal{C},S)$, 
$F$ is a center sequence of $\mathcal{C}$ of length $M$,
$\dim F(i)=2$ and $F(i)\not\subset|\mathcal{C}-(\mathcal{C}/H)|$ for any $i\in\{1,2,\ldots,M\}$, and $\mathcal{C}* F(1)*F(2)*\cdots*F(M)$ is a subdivision of $\mathcal{D}(S|V)\hat{\cap}\mathcal{C}$.
\item
By $0_{\mathcal{C}}$ we denote the unique center sequence of $\mathcal{C}$ of length $0$.
The following three conditions are equivalent:
\begin{enumerate}
\item
$\Ht(H,\mathcal{C},S)=0$.
\item
$(0, 0_{\mathcal{C}})\in\mathcal{USD}(H,\mathcal{C},S)$.
\item
$\{(0, 0_{\mathcal{C}})\}=\mathcal{USD}(H,\mathcal{C},S)$.
\end{enumerate}
\end{enumerate}

Consider any $(M, F)\in\mathcal{USD}(H,\mathcal{C},S)$.
We denote $\widetilde{\mathcal{C}}=\mathcal{C}*F(1)*F(2)*\cdots*F(M)$ for simplicity.

\begin{enumerate}
\setcounter{enumi}{3}
\item
For any $\Delta\in\mathcal{C}_0\cup\mathcal{C}_1$,
$\widetilde{\mathcal{C}}\backslash\Delta=\mathcal{D}(S|V)\hat{\cap}\mathcal{F}(\Delta)=\mathcal{F}(\Delta)$.
\item
$\widetilde{\mathcal{C}}\backslash|\mathcal{C}-(\mathcal{C}/H)|=(\mathcal{D}(S|V)\hat{\cap}\mathcal{C})\backslash|\mathcal{C}-(\mathcal{C}/H)|= \mathcal{C}-(\mathcal{C}/H)$.
\item
Consider any subset $\hat{\mathcal{C}}$ of $\mathcal{C}$ satisfying $\dim \hat{\mathcal{C}}=\dim \Vect(|\hat{\mathcal{C}}|)\geq 2$, $\hat{\mathcal{C}}\Mx=\hat{\mathcal{C}}^0$, $H\in\hat{\mathcal{C}}_1$ and $\hat{\mathcal{C}}=(\hat{\mathcal{C}}/H)\Fc$.

Let $\hat{M}=\sharp\{i\in\{1,2,\ldots,M\}|F(i)\subset|\hat{\mathcal{C}}|\}$, and $\tau:\{1,2,\ldots,\hat{M}\}\rightarrow \{1,2,\ldots,M\}$ denote the unique injective mapping preserving the order and satisfying $\tau(\{1,2,\ldots,\hat{M}\})=\{i\in\{1,2,\ldots,M\}|F(i)\subset|\hat{\mathcal{C}}|\}$.

Then, $(H, \hat{\mathcal{C}})\in\mathcal{HC}(V,N,S)$, $(\hat{M},F\tau)\in\mathcal{USD}(H, \hat{\mathcal{C}},S)$, and
$\widetilde{\mathcal{C}}\backslash|\hat{\mathcal{C}}|=\hat{\mathcal{C}}*F\tau (1)*F\tau(2)*\cdots*F\tau(\hat{M})$.
\item
Consider any $\Delta\in\mathcal{C}_2/H$.

$|\widetilde{\mathcal{C}}\backslash\Delta|=|\mathcal{D}(S|V)\hat{\cap}\mathcal{F}(\Delta)|=\Delta$.

$\widetilde{\mathcal{C}}\backslash\Delta$ is the mimimum simplicial cone subdivision of $\mathcal{D}(S|V)\hat{\cap}\mathcal{F}(\Delta)$ over $N^*$ in $V^*$, in other words, the following three conditions hold:
\begin{enumerate}
\item
$\widetilde{\mathcal{C}}\backslash\Delta$ is a simplicial cone decomposition over $N^*$ in $V^*$.
\item
$\widetilde{\mathcal{C}}\backslash\Delta$ is a subdivision of $\mathcal{D}(S|V)\hat{\cap}\mathcal{F}(\Delta)$.
\item
If $\mathcal{E}$ is a simplicial cone decomposition over $N^*$ in $V^*$ and $\mathcal{E}$ is a subdivision of $\mathcal{D}(S|V)\hat{\cap}\mathcal{F}(\Delta)$, then $\mathcal{E}$ is a subdivision of $\widetilde{\mathcal{C}}\backslash\Delta$.
\end{enumerate}
\end{enumerate}

Below, we consider the case $\Ht(H,\mathcal{C},S)>0$.
Assume $\Ht(H,\mathcal{C},S)>0$.

The characteristic function $\gamma:(\mathcal{C}-(\mathcal{C}/H))_1\rightarrow\Q_0$ of $(H,\mathcal{C}, S)$ is defined.
Let $m=\sum_{\bar{E}\in(\mathcal{C}-(\mathcal{C}/H))_1}\lfloor\gamma(\bar{E})\rfloor\in\Z_+$ and $\bar{m}=\sum_{\bar{E}\in(\mathcal{C}-(\mathcal{C}/H))_1}\lceil\gamma(\bar{E})\rceil\in\Z_0$.
Obviously $\bar{m}\leq m$.
\begin{enumerate}
\setcounter{enumi}{7}
\item
$m\leq M$.
\item
There exists uniquely a pair $(E,M)$ of a compatible mapping
$E:\{1,2,\ldots,$\hfill\break$m\}\rightarrow (\mathcal{C}-(\mathcal{C}/H))_1$
with $S$ and a mapping 
$M:\{\bar{m},\bar{m}+1,\ldots,m+1\}\rightarrow\Z_+$
satisfying the following three conditions.
We denote 
$$F_E=F(V^*,N^*,H,\mathcal{C},m,E):\{1,2,\ldots,m\}\rightarrow 2^{ V^*}$$ 
$$H_E=H(V^*,N^*,H,\mathcal{C},m,E):\{1,2,\ldots,m+1\}\rightarrow 2^{ V^*},\text{ and}$$
$$\mathcal{B}_E=\mathcal{B}(V^*,N^*,H,\mathcal{C},m,E):\{1,2,\ldots,m+1\}\rightarrow 2^{2^{ V^*}}:$$
\begin{enumerate}
\item
$F(j)=F_E(j)$ for any $j\in\{1,2,\ldots, m\}$.
\item
$M(\bar{m})=m$, $M(m+1)=M$ and $M(i-1)\leq M(i)$ for any $i\in\{\bar{m}+1,\bar{m}+2,\ldots, m+1\}$.
\item
$F(j)\subset|\mathcal{B}_E(i)|$ and $F(j)\not\subset|\mathcal{B}_E(i)-(\mathcal{B}_E(i)/H_E(i))|$ for any $i\in\{\bar{m}+1,\bar{m}+2,\ldots, m+1\}$ and any $j\in\{M(i-1)+1,M(i-1)+2,\ldots, M(i)\}$.
\end{enumerate}
\end{enumerate}

We take the unique pair $(E,M)$ of a compatible mapping
$E:\{1,2,\ldots,m\}\rightarrow (\mathcal{C}-(\mathcal{C}/H))_1$
with $S$ and a mapping 
$M:\{\bar{m},\bar{m}+1,\ldots,m+1\}\rightarrow\Z_+$
satisfying the above three conditions, and we denote 
$$F_E=F(V^*,N^*,H,\mathcal{C},m,E):\{1,2,\ldots,m\}\rightarrow 2^{V^*}$$ 
$$H_E=H(V^*,N^*,H,\mathcal{C},m,E):\{1,2,\ldots,m+1\}\rightarrow 2^{V^*},\text{ and}$$
$$\mathcal{B}_E=\mathcal{B}(V^*,N^*,H,\mathcal{C},m,E):\{1,2,\ldots,m+1\}\rightarrow 2^{2^{V^*}}.$$
We put $M(i)=m$ for any $i\in\{0,1,\ldots,\bar{m}-1\}$. We obtain an extension $M:\{0,1,\ldots,m+1\}\rightarrow\Z_+$ of $M:\{\bar{m},\bar{m}+1,\ldots,m+1\}\rightarrow\Z_+$.
$M(0)=m$, $M(m+1)=M$, $M(i-1)\leq M(i)$ for any $i\in\{1,2,\ldots,m+1\}$.

Consider any $i\in\{1, 2,\ldots,m+1\}$.
We denote $\bar{\mathcal{C}}(i)= \mathcal{C}*F(1)*F(2)*\cdots*F(M(i-1))\backslash |\mathcal{B}_E(i)|$.
Let $\bar{F}(i): \{1,2,\ldots, M(i)-M(i-1)\}\rightarrow 2^{V^*}$ denote the mapping satisfying $\bar{F}(i)(j)= F(M(i-1)+j)$ for any $j\in\{1,2,\ldots, M(i)-M(i-1)\}$.

\begin{enumerate}
\setcounter{enumi}{9}
\item
For any $i\in\{1,2,\ldots,m+1\}$,
$|\bar{\mathcal{C}}(i)|=|\mathcal{B}_E(i)|$,
$(H_E(i), \bar{\mathcal{C}}(i))\in\mathcal{HC}(V,N,S)$, 
$\Ht(H_E(i), \bar{\mathcal{C}}(i),S)<\Ht(H,\mathcal{C},S)$ and $(M(i)-M(i-1), \bar{F}(i))\in$\hfill\break$\mathcal{USD}(H_E(i), \bar{\mathcal{C}}(i),S)$.
\item
For any $i\in\{1, 2,\ldots,m+1\}$,
$\widetilde{\mathcal{C}}\backslash|\mathcal{B}_E(i)|=
(\mathcal{C}*F(1)*F(2)*\cdots*F(M(i))) $\hfill\break$\backslash|\mathcal{B}_E(i)|=
\bar{\mathcal{C}}(i)*\bar{F}(1)*\bar{F}(2)*\cdots*\bar{F}(M(i)- M(i-1))$.
\item
For any $i\in\{1,2,\ldots,\bar{m}\}$, $\Ht(H_E(i), \bar{\mathcal{C}}(i),S)=0$, $ M(i)-M(i-1)=0$, and $\widetilde{\mathcal{C}}\backslash|\mathcal{B}_E(i)|=
(\mathcal{C}*F(1)*F(2)*\cdots*F(m))\backslash|\mathcal{B}_E(i)|=\bar{\mathcal{C}}(i)*\bar{F}(1)*\bar{F}(2)*\cdots*\bar{F}(M(i)- M(i-1))
=\bar{\mathcal{C}}(i)= \mathcal{B}_E(i)$.
\item

$$\widetilde{\mathcal{C}}=
(\mathcal{C}-(\mathcal{C}/H))\cup(
\bigcup_{i\in\{1,2,\ldots,m+1\}}
(
(\widetilde{\mathcal{C}}\backslash|\mathcal{B}_E(i)|)
-(\widetilde{\mathcal{C}}\backslash|\mathcal{B}_E(i)-(\mathcal{B}_E(i)/H_E(i))|)
)).$$
$$(\mathcal{C}-(\mathcal{C}/H))\cap(
(\widetilde{\mathcal{C}}\backslash|\mathcal{B}_E(i)|)
-(\widetilde{\mathcal{C}}\backslash|\mathcal{B}_E(i)-(\mathcal{B}_E(i)/H_E(i))|)
)=\emptyset,$$
for any $i\in\{1,2,\ldots,m+1\}$.

\begin{equation*}\begin{split}(
&(\widetilde{\mathcal{C}}\backslash|\mathcal{B}_E(i)|)
-(\widetilde{\mathcal{C}}\backslash|\mathcal{B}_E(i)-(\mathcal{B}_E(i)/H_E(i))|)
)\\
&\qquad\qquad\cap(
(\widetilde{\mathcal{C}}\backslash|\mathcal{B}_E(j)|)
-(\widetilde{\mathcal{C}}\backslash|\mathcal{B}_E(j)-(\mathcal{B}_E(j)/H_E(j))|)
)=\emptyset,
\end{split}\end{equation*}
for any $i\in\{1,2,\ldots,m+1\}$ and any $j\in\{1,2,\ldots,m+1\}$ with $i\neq j$.
\item
If $\Gamma\in\widetilde{\mathcal{C}}_1$ and $\Gamma\not\subset|\mathcal{C}-(\mathcal{C}/H)|$,
then there exists uniquely $i\in\{1,2,\ldots,m+1\}$ satisfying
$\Gamma\subset|\mathcal{B}_E(i)|$ and $\Gamma\not\subset|\mathcal{B}_E(i)-(\mathcal{B}_E(i)/H_E(i))|$.

If $\Gamma\in\widetilde{\mathcal{C}}_1$, $i\in\{1,2,\ldots,m+1\}$, $\Gamma\subset|\mathcal{B}_E(i)|$ and $\Gamma\not\subset|\mathcal{B}_E(i)-(\mathcal{B}_E(i)/H_E(i))|$, then
$\Gamma\not\subset|\mathcal{C}-(\mathcal{C}/H)|$.

For any $i\in\{1,2,\ldots,\bar{m}\}$, 
$$\{\Gamma\in\widetilde{\mathcal{C}}_1|\Gamma\subset|\mathcal{B}_E(i)|, \Gamma\not\subset|\mathcal{B}_E(i)-(\mathcal{B}_E(i)/H_E(i))|\}=
\{H_E(i)\}.$$

For any $i\in\{\bar{m}+1, \bar{m}+2,\ldots, m+1\}$, 
\begin{equation*}\begin{split}
&\{\Gamma\in\widetilde{\mathcal{C}}_1|\Gamma\subset|\mathcal{B}_E(i)|, \Gamma\not\subset|\mathcal{B}_E(i)-(\mathcal{B}_E(i)/H_E(i))|\}=\\
&\qquad\qquad\{H_E(i)\}\cup\{\R_0b_{F(j)/N^*}|j\in\Z, M(i-1)<j\leq M(i)\},\\
&\sharp\{\Gamma\in\widetilde{\mathcal{C}}_1|\Gamma\subset|\mathcal{B}_E(i)|, \Gamma\not\subset|\mathcal{B}_E(i)-(\mathcal{B}_E(i)/H_E(i))|\}=M(i)-M(i-1)+1.
\end{split}\end{equation*}
\end{enumerate}
\end{theorem}

Consider any $(H,\mathcal{C})\in\mathcal{HC}(V,N,S)$ and any $(M,F)\in\mathcal{USD}(H,\mathcal{C},S)$.
We denote $\widetilde{\mathcal{C}}=\mathcal{C}*F(1)*F(2)*\cdots*F(M)$ and
$\widetilde{\mathcal{C}}_1^\circ=\{\Gamma\in\widetilde{\mathcal{C}}_1|
\Gamma\not\subset|\mathcal{C}-(\mathcal{C}/H)|\}$ for simplicity.
We know $\sharp\widetilde{\mathcal{C}}_1^\circ=M+1$.

By induction on $\Ht(H,\mathcal{C},S)$ we define three mappings
$$I(V,N, H,\mathcal{C},S, M,F):\widetilde{\mathcal{C}}_1^\circ\rightarrow
\{1,2,\ldots,M+1\},$$
$$\mathcal{A}(V,N, H,\mathcal{C},S, M,F):
\widetilde{\mathcal{C}}_1^\circ\rightarrow 2^{2^{V^*}},\text{ and}$$
$$\mathcal{A}^\circ(V,N, H,\mathcal{C},S, M,F):
\widetilde{\mathcal{C}}_1^\circ\rightarrow 2^{2^{V^*}}$$
such that $\mathcal{A}(V,N, H,\mathcal{C},S, M,F)(\Gamma)$ is a simplicial cone decomposition over $N^*$ in $V^*$ and $\mathcal{A}^\circ(V,N, H,\mathcal{C},S, M,F) (\Gamma)\subset\mathcal{A}(V,N, H,\mathcal{C},S, M,F)(\Gamma)\subset\widetilde{\mathcal{C}}$ for any $\Gamma\in \widetilde{\mathcal{C}}_1^\circ$.

Consider the case $\Ht(H,\mathcal{C},S)=0$.

We know $M=0$, $\widetilde{\mathcal{C}}=\mathcal{C}$, and
$\widetilde{\mathcal{C}}_1^\circ=\{H\}$.
We define
\begin{equation*}\begin{split}
I(V,N, H,\mathcal{C},S, M,F)(H)&=1,\\
\mathcal{A}(V,N, H,\mathcal{C},S, M,F)(H)&= \mathcal{C},\\
\mathcal{A}^\circ(V,N, H,\mathcal{C},S, M,F)(H)&=\mathcal{C}.
\end{split}\end{equation*}

Consider the case $\Ht(H,\mathcal{C},S)>0$.

The characteristic function $\gamma:(\mathcal{C}-(\mathcal{C}/H))_1\rightarrow\Q_0$ of $(H,\mathcal{C}, S)$ is defined.
Let $m=\sum_{\bar{E}\in(\mathcal{C}-(\mathcal{C}/H))_1}\lfloor\gamma(\bar{E})\rfloor\in\Z_+$ and $\bar{m}=\sum_{\bar{E}\in(\mathcal{C}-(\mathcal{C}/H))_1}\lceil\gamma(\bar{E})\rceil\in\Z_0$.
We know $\bar{m}\leq m\leq M$.

We take the unique pair $(E,M)$ of a compatible mapping
$E:\{1,2,\ldots,m\}\rightarrow (\mathcal{C}-(\mathcal{C}/H))_1$
with $S$ and a mapping 
$M:\{\bar{m},\bar{m}+1,\ldots,m+1\}\rightarrow\Z_+$
satisfying the following three conditions.
We denote 
$$F_E=F(V^*,N^*,H,\mathcal{C},m,E):\{1,2,\ldots,m\}\rightarrow 2^{ V^*}$$ 
$$H_E=H(V^*,N^*,H,\mathcal{C},m,E):\{1,2,\ldots,m+1\}\rightarrow 2^{ V^*},\text{ and}$$
$$\mathcal{B}_E=\mathcal{B}(V^*,N^*,H,\mathcal{C},m,E):\{1,2,\ldots,m+1\}\rightarrow 2^{2^{ V^*}}.$$
\begin{enumerate}
\item
$F(j)=F_E(j)$ for any $j\in\{1,2,\ldots, m\}$.
\item
$M(\bar{m})=m$, $M(m+1)=M$ and $M(i-1)\leq M(i)$ for any $i\in\{\bar{m}+1,\bar{m}+2,\ldots, m+1\}$.
\item
$F(j)\subset|\mathcal{B}_E(i)|$ and $F(j)\not\subset|\mathcal{B}_E(i)-(\mathcal{B}_E(i)/H_E(i))|$ for any $i\in\{\bar{m}+1,\bar{m}+2,\ldots, m+1\}$ and any $j\in\{M(i-1)+1,M(i-1)+2,\ldots, M(i)\}$.
\end{enumerate}

For any $i\in\{0,1,\ldots, \bar{m}-1\}$, we put $M(i)=m$.
We obtain an extension $M:\{0,1,\ldots,m+1\}\rightarrow\Z_+$ of $M:\{\bar{m},\bar{m}+1,\ldots,m+1\}\rightarrow\Z_+$.
For any $i\in\{1,2,\ldots,m+1\}$, we denote
$\bar{\mathcal{C}}(i)=(\mathcal{C}*F(1)*F(2)*\cdots*F(M(i-1)))\backslash|\mathcal{B}_E(i)|$, and we take the mapping $\bar{F}(i):\{1,2,\ldots, M(i)-M(i-1)\}\rightarrow 2^{V^*}$ satisfying $\bar{F}(i)(j)=F(M(i-1)+j)$ for any $j\in\{1,2,\ldots, M(i)-M(i-1)\}$.
We know $(H_E(i), \bar{\mathcal{C}}(i))\in\mathcal{HC}(V,N,S)$,
$|\bar{\mathcal{C}}(i)|=|\mathcal{B}_E(i)|$,
$|\bar{\mathcal{C}}(i)-( \bar{\mathcal{C}}(i)/H_E(i))|=
|\mathcal{B}_E(i)-( \mathcal{B}_E(i)/H_E(i))|$,
$\Ht(H_E(i), \bar{\mathcal{C}}(i),S)<\Ht(H,\mathcal{C},S)$, $(M(i)-M(i-1),\bar{F}(i))\in\mathcal{USD}( H_E(i), \bar{\mathcal{C}}(i),S)$, and
$\bar{\mathcal{C}}(i)* \bar{F}(1)* \bar{F}(2)*\cdots*\bar{F}(M(i)-M(i-1))=
\widetilde{\mathcal{C}}\backslash|\mathcal{B}_E(i)|$
for any $i\in\{1,2,\ldots,m+1\}$.
We denote
$$G_E=G(V^*,N^*,H,\mathcal{C},m,E):\{1,2,\ldots,m\}\rightarrow 2^{V^*}.$$

Consider any $\Gamma\in\widetilde{\mathcal{C}}_1^\circ$.
Take the unique $i\in\{1,2,\ldots,m+1\}$ satisfying $\Gamma\subset|\mathcal{B}_E(i)|$ and $\Gamma\not\subset|\mathcal{B}_E(i)-( \mathcal{B}_E(i)/H_E(i))|$.
We know $\Gamma\in(\bar{\mathcal{C}}(i)* \bar{F}(1)* \bar{F}(2)*\cdots*\bar{F}(M(i)-M(i-1)))_1$ and $\Gamma\not\subset|\bar{\mathcal{C}}(i)-( \bar{\mathcal{C}}(i)/H_E(i))|$.
By induction hypothesis we know that 
$I(V,N, H_E(i),\bar{\mathcal{C}},S, M(i)-M(i-1),\bar{F}(i))(\Gamma)$,
$\mathcal{A}(V,N, H_E(i),\bar{\mathcal{C}},S, M(i)-M(i-1),\bar{F}(i))(\Gamma)$ and
$\mathcal{A}^\circ(V,N, H_E(i),\bar{\mathcal{C}},S, M(i)-M(i-1),\bar{F}(i))(\Gamma)$ are defined.

We define
\begin{equation*}\begin{split}
&I(V,N, H,\mathcal{C},S, M,F)(\Gamma)=\\
&\qquad i-1+M(i-1)-m+ I(V,N, H_E(i),\bar{\mathcal{C}}(i),S, M(i)-M(i-1),\bar{F}(i))(\Gamma),\\
&\mathcal{A}(V,N, H,\mathcal{C},S, M,F)(\Gamma)=\mathcal{A}(V,N, H_E(i),\bar{\mathcal{C}}(i),S, M(i)-M(i-1),\bar{F}(i))(\Gamma).
\end{split}\end{equation*}
In case $i\neq m+1$, we define
\begin{equation*}\begin{split}
&\mathcal{A}^\circ(V,N, H,\mathcal{C},S, M,F)(\Gamma)=\\
&\qquad \{\Theta\in\mathcal{A}^\circ(V,N, H_E(i),\bar{\mathcal{C}}(i),S, M(i)-M(i-1),\bar{F}(i))(\Gamma)|\Theta^\circ\subset|\mathcal{B}_E/G_E(i)|^\circ\}.
\end{split}\end{equation*}
In case $i=m+1$, we define
$$\mathcal{A}^\circ(V,N, H,\mathcal{C},S, M,F)(\Gamma)=
\mathcal{A}^\circ(V,N, H_E(i),\bar{\mathcal{C}}(i),S, M(i)-M(i-1),\bar{F}(i))(\Gamma).$$

We call the mapping $I(V,N, H,\mathcal{C},S, M,F)$ the $H$\emph{-ordered enumeration} of $\widetilde{\mathcal{C}}_1^\circ$.
Let $\Gamma\in\widetilde{\mathcal{C}}_1^\circ$.
We call the simplicial cone decomposition $\mathcal{A}(V,N, H,\mathcal{C},S, M,F)(\Gamma)$ the $H$\emph{-lower part} of  $\widetilde{\mathcal{C}}$ below $\Gamma$, and we call the subset $\mathcal{A}^\circ(V,N, H,\mathcal{C},S, M,F)(\Gamma)$ the $H$\emph{-lower main part} of  $\widetilde{\mathcal{C}}$ below $\Gamma$.

\begin{lemma}
\label{lower parts}
Assume $\mathcal{HC}(V,N,S)\neq\emptyset$.
Consider any $(H,\mathcal{C})\in\mathcal{HC}(V,N,S)$ and any $(M,F)\in\mathcal{USD}(H,\mathcal{C},S)$.
We denote $\widetilde{\mathcal{C}}=\mathcal{C}*F(1)*F(2)*\cdots*F(M)$,
$\widetilde{\mathcal{C}}_1^\circ=\{\Gamma\in\widetilde{\mathcal{C}}_1|\Gamma\not\subset|\mathcal{C}-(\mathcal{C}/H)|\}$,
$$I= I(V,N, H,\mathcal{C},S, M,F): \widetilde{\mathcal{C}}_1^\circ
\rightarrow\{1,2,\ldots, M+1\},$$
$$\mathcal{A}=\mathcal{A}(V,N, H,\mathcal{C},S, M,F): \widetilde{\mathcal{C}}_1^\circ
\rightarrow 2^{2^{V^*}},\text{ and}$$
$$\mathcal{A}^\circ=\mathcal{A}^\circ(V,N, H,\mathcal{C},S, M,F): \widetilde{\mathcal{C}}_1^\circ
\rightarrow 2^{2^{V^*}}.$$
\begin{enumerate}
\item
The mapping $I$ is bijective. $H\in\widetilde{\mathcal{C}}_1^\circ$.
$I(H)=M+1$.
\item
Consider any $\Gamma\in \widetilde{\mathcal{C}}_1^\circ$.

$\mathcal{A}(\Gamma)$ is a simplicial cone decomposition over $N^*$ in $V^*$.
$\mathcal{A}(\Gamma)\subset \widetilde{\mathcal{C}}$.
$\dim \mathcal{A}(\Gamma)=\dim\Vect(|\mathcal{A}(\Gamma)|)=\dim\mathcal{C}\geq 2$.
$\Vect(|\mathcal{A}(\Gamma)|)=\Vect(|\mathcal{C}|)$.
$\Gamma\in\mathcal{A}(\Gamma)_1$.
$\mathcal{A}(\Gamma)=( \mathcal{A}(\Gamma)/\Gamma)\Fc$.
$\mathcal{A}(\Gamma)\Mx\subset\mathcal{A}^\circ(\Gamma)\subset\mathcal{A}(\Gamma)$.

$\Gamma\in\mathcal{A}^\circ(\Gamma)\Leftrightarrow\mathcal{A}^\circ(\Gamma)=\mathcal{A}(\Gamma)\Leftrightarrow\Gamma=H$.
\item
$$|\mathcal{C}|=|\mathcal{C}-(\mathcal{C}/H)|\cup(
\bigcup_{\Gamma\in\widetilde{\mathcal{C}}_1^\circ}|\mathcal{A}(\Gamma)/\Gamma|^\circ).$$

For any $\Gamma\in\widetilde{\mathcal{C}}_1^\circ$,
$|\mathcal{C}-(\mathcal{C}/H)|\cap|\mathcal{A}(\Gamma)/\Gamma|^\circ=\emptyset$.

For any $\Gamma\in\widetilde{\mathcal{C}}_1^\circ$ and any $\bar{\Gamma}\in\widetilde{\mathcal{C}}_1^\circ$ with $\Gamma\neq\bar{\Gamma}$,
$|\mathcal{A}(\Gamma)/\Gamma|^\circ\cap|\mathcal{A}(\bar{\Gamma})/\bar{\Gamma}|^\circ=\emptyset$.
\item
$$\widetilde{\mathcal{C}}=(\mathcal{C}-(\mathcal{C}/H))\cup(
\bigcup_{\Gamma\in\widetilde{\mathcal{C}}_1^\circ}(\mathcal{A}(\Gamma)/\Gamma)).$$

For any $\Gamma\in\widetilde{\mathcal{C}}_1^\circ$,
$(\mathcal{C}-(\mathcal{C}/H))\cap(\mathcal{A}(\Gamma)/\Gamma)=\emptyset$.

For any $\Gamma\in\widetilde{\mathcal{C}}_1^\circ$ and any $\bar{\Gamma}\in\widetilde{\mathcal{C}}_1^\circ$ with $\Gamma\neq\bar{\Gamma}$,
$(\mathcal{A}(\Gamma)/\Gamma)\cap(\mathcal{A}(\bar{\Gamma})/\bar{\Gamma})=\emptyset$.
\item
$$\widetilde{\mathcal{C}}\Mx=
\bigcup_{\Gamma\in\widetilde{\mathcal{C}}_1^\circ}\mathcal{A}(\Gamma)\Mx.$$

For any $\Gamma\in\widetilde{\mathcal{C}}_1^\circ$ and any $\bar{\Gamma}\in\widetilde{\mathcal{C}}_1^\circ$ with $\Gamma\neq\bar{\Gamma}$,
$\mathcal{A}(\Gamma)\Mx\cap\mathcal{A}(\bar{\Gamma})\Mx=\emptyset$.
\item
$$|\mathcal{C}|=
\bigcup_{\Gamma\in\widetilde{\mathcal{C}}_1^\circ}|\mathcal{A}^\circ(\Gamma) |^\circ.$$

For any $\Gamma\in\widetilde{\mathcal{C}}_1^\circ$ and any $\bar{\Gamma}\in\widetilde{\mathcal{C}}_1^\circ$ with $\Gamma\neq\bar{\Gamma}$,
$|\mathcal{A}^\circ(\Gamma) |^\circ \cap|\mathcal{A}^\circ(\bar{\Gamma}) |^\circ =\emptyset$.
\item
$$\widetilde{\mathcal{C}}=
\bigcup_{\Gamma\in\widetilde{\mathcal{C}}_1^\circ}\mathcal{A}^\circ(\Gamma).$$

For any $\Gamma\in\widetilde{\mathcal{C}}_1^\circ$ and any $\bar{\Gamma}\in\widetilde{\mathcal{C}}_1^\circ$ with $\Gamma\neq\bar{\Gamma}$,
$\mathcal{A}^\circ(\Gamma)\cap\mathcal{A}^\circ(\bar{\Gamma})=\emptyset$.
\end{enumerate}

For any $i\in\{0,1,\ldots,M+1\}$, we denote
\begin{equation*}\begin{split}
X(i)=&|\mathcal{C}-(\mathcal{C}/H)|\cup(
\bigcup_{\Gamma\in\widetilde{\mathcal{C}}_1^\circ, I(\Gamma)\leq i}
|\mathcal{A}(\Gamma)/\Gamma)|^\circ)\subset V^*,\\
Y(i)=&\bigcup_{\Gamma\in\widetilde{\mathcal{C}}_1^\circ, I(\Gamma)> i}
|\mathcal{A}^\circ(\Gamma)|^\circ\subset V^*.
\end{split}\end{equation*}

\begin{enumerate}
\setcounter{enumi}{7}
\item
$X(0)= |\mathcal{C}-(\mathcal{C}/H)|$.
$X(M+1)=|\mathcal{C}|$.
$Y(0)=|\mathcal{C}|$.
$Y(M+1)=\emptyset$.

For any $i\in\{1,2,\ldots,M+1\}$, $X(i-1)\subset X(i)$, $X(i-1)\neq X(i)$,\hfill\break
$Y(i-1)\supset Y(i)$, and $Y(i-1)\neq Y(i)$.
\item
For any $i\in\{1,2,\ldots,M+1\}$,
\begin{equation*}\begin{split}
X(i-1)\cap|\mathcal{A}I^{-1}(i)|&=
|\mathcal{A}I^{-1}(i)|-| \mathcal{A}I^{-1}(i)/ I^{-1}(i)|^\circ,\\
Y(i)\cap|\mathcal{A}I^{-1}(i)|&= |\mathcal{A}I^{-1}(i)|- |\mathcal{A}^\circ I^{-1}(i)|^\circ.
\end{split}\end{equation*}
\item
For any $i\in\{0,1,\ldots,M+1\}$,
\begin{equation*}\begin{split}
X(i)&=|\mathcal{C}-(\mathcal{C}/H)|\cup(
\bigcup_{\Gamma\in\widetilde{\mathcal{C}}_1^\circ, I(\Gamma)\leq i}
|\mathcal{A}(\Gamma)|),\\
Y(i)&=\bigcup_{\Gamma\in\widetilde{\mathcal{C}}_1^\circ, I(\Gamma)> i}
|\mathcal{A}(\Gamma)|,
\end{split}\end{equation*}
and $X(i)$ and $Y(i)$ are closed subsets of $V^*$.

\item
Consider any subset $\hat{\mathcal{C}}$ of $\mathcal{C}$ satisfying $\dim \hat{\mathcal{C}}=\dim \Vect(|\hat{\mathcal{C}}|)\geq 2$, $\hat{\mathcal{C}}\Mx=\hat{\mathcal{C}}^0$, $H\in\hat{\mathcal{C}}_1$ and $\hat{\mathcal{C}}=(\hat{\mathcal{C}}/H)\Fc$.

Let $\hat{M}=\sharp\{i\in\{1,2,\ldots,M\}|F(i)\subset|\hat{\mathcal{C}}|\}$, and $\tau:\{1,2,\ldots,\hat{M}\}\rightarrow \{i\in\{1,2,\ldots,M\}$ denote the unique injective mapping preserving the order and satisfying $\tau(\{1,2,\ldots,\hat{M}\})=\{i\in\{1,2,\ldots,M\}|F(i)\subset|\hat{\mathcal{C}}|\}$.

By Theorem~\ref{usd1}.6 we know that $(H, \hat{\mathcal{C}})\in\mathcal{HC}(V,N,S)$, $(\hat{M},F\tau)\in$\hfill\break$\mathcal{USD}(H, \hat{\mathcal{C}},S)$, and
$\widetilde{\mathcal{C}}\backslash|\hat{\mathcal{C}}|=\hat{\mathcal{C}}*F\tau(1)*F\tau(2)*\cdots*F\tau(\hat{M})$.
\begin{enumerate}
\item
$\{\Gamma\in(\widetilde{\mathcal{C}}\backslash|\hat{\mathcal{C}}|)_1|
\Gamma\not\subset|\hat{\mathcal{C}}-(\hat{\mathcal{C}}/H)|\}
=\widetilde{\mathcal{C}}_1^\circ\backslash|\hat{\mathcal{C}}|$.
\end{enumerate}
We denote
$$\hat{I}= I(V,N, H,\hat{\mathcal{C}},S, \hat{M},F\tau): \widetilde{\mathcal{C}}_1^\circ\backslash|\hat{\mathcal{C}}|
\rightarrow\{1,2,\ldots, \hat{M}+1\},$$
$$\hat{\mathcal{A}}=\mathcal{A}(V,N, H,\hat{\mathcal{C}},S, \hat{M},F\tau): \widetilde{\mathcal{C}}_1^\circ\backslash|\hat{\mathcal{C}}|
\rightarrow 2^{2^{V^*}},$$
$$\smash{\hat{\mathcal{A}}}^\circ=\mathcal{A}^\circ(V,N, H,\hat{\mathcal{C}},S, \hat{M},F\tau): \widetilde{\mathcal{C}}_1^\circ\backslash|\hat{\mathcal{C}}|
\rightarrow 2^{2^{V^*}}.$$
\begin{enumerate}
\setcounter{enumii}{1}
\item
Let $\kappa: \{1,2,\ldots, \hat{M}+1\}\rightarrow\{1,2,\ldots, M+1\}$ denote the composition mapping $I\iota\hat{I}^{-1}$, where $\iota: \widetilde{\mathcal{C}}_1^\circ\backslash|\hat{\mathcal{C}}|\rightarrow\widetilde{\mathcal{C}}_1^\circ$ denotes the inclusion mapping.

The mapping $\kappa$ is injective and preserves the order.
$\kappa(\hat{M}+1)=M+1$.
\item
For any $\Gamma\in\widetilde{\mathcal{C}}_1^\circ\backslash|\hat{\mathcal{C}}|$,
$\hat{\mathcal{A}}(\Gamma)/\Gamma= (\mathcal{A}(\Gamma)/\Gamma)\backslash |\hat{\mathcal{C}}|$ and
$|\hat{\mathcal{A}}(\Gamma)/\Gamma|^\circ= |\mathcal{A}(\Gamma)/\Gamma|^\circ\cap |\hat{\mathcal{C}}|$

For any $\Gamma\in\widetilde{\mathcal{C}}_1^\circ-
(\widetilde{\mathcal{C}}_1^\circ\backslash|\hat{\mathcal{C}}|)$,
$(\mathcal{A}(\Gamma)/\Gamma)\backslash |\hat{\mathcal{C}}|=\emptyset$ and
$|\mathcal{A}(\Gamma)/\Gamma|^\circ\cap |\hat{\mathcal{C}}|=\emptyset$
\item
For any $\Gamma\in\widetilde{\mathcal{C}}_1^\circ\backslash|\hat{\mathcal{C}}|$,
$\smash{\hat{\mathcal{A}}}^\circ (\Gamma)= \mathcal{A}^\circ(\Gamma) \backslash |\hat{\mathcal{C}}|$ and
$|\smash{\hat{\mathcal{A}}}^\circ(\Gamma) |^\circ = |\mathcal{A}^\circ(\Gamma) |^\circ \cap |\hat{\mathcal{C}}|$.

For any $\Gamma\in\widetilde{\mathcal{C}}_1^\circ-
(\widetilde{\mathcal{C}}_1^\circ\backslash|\hat{\mathcal{C}}|)$,
$\mathcal{A}^\circ(\Gamma) \backslash |\hat{\mathcal{C}}|=\emptyset$ and
$|\mathcal{A}^\circ(\Gamma)|^\circ \cap |\hat{\mathcal{C}}|=\emptyset$
\end{enumerate}
\item
Consider any $\Gamma\in \widetilde{\mathcal{C}}_1^\circ$.

If $\Theta\in\mathcal{A}(\Gamma)/\Gamma$, $\Delta\in\mathcal{C}$ and $\Theta \subset\Delta$, then $H\subset\Delta$, and $\Gamma\subset\Delta$.

If $\Theta\in\mathcal{A}^\circ(\Gamma)$, $\Delta\in\mathcal{C}/H$ and $\Theta \subset\Delta$, then $\Gamma\subset\Delta$.

If $\Theta\in\mathcal{A}^\circ(\Gamma) $ and $\Theta\subset|\mathcal{C}-(\mathcal{C}/H)|$, then
$\Gamma\not\subset\Theta$, $\Theta+H\in\mathcal{C}/H$ and
$b_{\Gamma/N^*}-b_{H/N^*}\in N^*\cap\Vect(\Theta)$.

If $\Theta\in\mathcal{A}^\circ(\Gamma) $ and $\Gamma\not\subset\Theta$,
then $\Theta+\Gamma\in\mathcal{A}^\circ(\Gamma)/\Gamma$.
\end{enumerate}

Below, we consider the case $\Ht(H,\mathcal{C},S)>0$.
Assume $\Ht(H,\mathcal{C},S)>0$.

The characteristic function $\gamma:(\mathcal{C}-(\mathcal{C}/H))_1\rightarrow\Q_0$ of $(H,\mathcal{C}, S)$ is defined.
Let $m=\sum_{\bar{E}\in(\mathcal{C}-(\mathcal{C}/H))_1}\lfloor\gamma(\bar{E})\rfloor\in\Z_+$ and $\bar{m}=\sum_{\bar{E}\in(\mathcal{C}-(\mathcal{C}/H))_1}\lceil\gamma(\bar{E})\rceil\in\Z_0$.
We know $\bar{m}\leq m\leq M$.

We take the unique pair $(E,M)$ of a compatible mapping
$E:\{1,2,\ldots,m\}\rightarrow (\mathcal{C}-(\mathcal{C}/H))_1$
with $S$ and a mapping 
$M:\{\bar{m},\bar{m}+1,\ldots,m+1\}\rightarrow\Z_+$
satisfying the following three conditions.
We denote 
$$F_E=F(V^*,N^*,H,\mathcal{C},m,E):\{1,2,\ldots,m\}\rightarrow 2^{ V^*}$$ 
$$H_E=H(V^*,N^*,H,\mathcal{C},m,E):\{1,2,\ldots,m+1\}\rightarrow 2^{ V^*},\text{ and}$$
$$\mathcal{B}_E=\mathcal{B}(V^*,N^*,H,\mathcal{C},m,E):\{1,2,\ldots,m+1\}\rightarrow 2^{2^{ V^*}}.$$
\begin{description}
\item[\emph{(a)}]
$F(j)=F_E(j)$ for any $j\in\{1,2,\ldots, m\}$.
\item[\emph{(b)}]
$M(\bar{m})=m$, $M(m+1)=M$ and $M(i-1)\leq M(i)$ for any $i\in\{\bar{m}+1,\bar{m}+2,\ldots, m+1\}$.
\item[\emph{(c)}]
$F(j)\subset|\mathcal{B}_E(i)|$ and $F(j)\not\subset|\mathcal{B}_E(i)-(\mathcal{B}_E(i)/H_E(i))|$ for any $i\in\{\bar{m}+1,\bar{m}+2,\ldots, m+1\}$ and any $j\in\{M(i-1)+1,M(i-1)+2,\ldots, M(i)\}$.
\end{description}

For any $i\in\{0,1,\ldots, \bar{m}-1\}$, we put $M(i)=m$.
We obtain an extension $M:\{0,1,\ldots,m+1\}\rightarrow\Z_+$ of $M:\{\bar{m},\bar{m}+1,\ldots,m+1\}\rightarrow\Z_+$.
For any $i\in\{1,2,\ldots,m+1\}$, we denote
$\bar{\mathcal{C}}(i)=(\mathcal{C}*F(1)*F(2)*\cdots*F(M(i-1)))\backslash|\mathcal{B}_E(i)|$, and we take the mapping $\bar{F}(i):\{1,2,\ldots, M(i)-M(i-1)\}\rightarrow 2^{V^*}$ satisfying $\bar{F}(i)(j)=F(M(i-1)+j)$ for any $j\in\{1,2,\ldots, M(i)-M(i-1)\}$.
We know $(H_E(i), \bar{\mathcal{C}}(i))\in\mathcal{HC}(V,N,S)$,
$|\bar{\mathcal{C}}(i)|=|\mathcal{B}_E(i)|$,
$|\bar{\mathcal{C}}(i)-( \bar{\mathcal{C}}(i)/H_E(i))|=
|\mathcal{B}_E(i)-( \mathcal{B}_E(i)/H_E(i))|$,
$\Ht(H_E(i), \bar{\mathcal{C}}(i),S)<\Ht(H,\mathcal{C},S)$, $(M(i)-M(i-1),\bar{F}(i))\in\mathcal{USD}( H_E(i), \bar{\mathcal{C}}(i),S)$, and
$\bar{\mathcal{C}}(i)* \bar{F}(1)* \bar{F}(2)*\cdots*\bar{F}(M(i)-M(i-1))=
\widetilde{\mathcal{C}}\backslash|\mathcal{B}_E(i)|$
for any $i\in\{1,2,\ldots,m+1\}$.
We denote
$$G_E=G(V^*,N^*,H,\mathcal{C},m,E):\{1,2,\ldots,m\}\rightarrow 2^{V^*}.$$

For any $i\in\{0,1,\ldots,m+1\}$, we put $L(i)=i+M(i)-m$.
We obtain a mapping $L: \{0,1,\ldots,m+1\}\rightarrow\Z_0$.

\begin{enumerate}
\setcounter{enumi}{12}
\item
$L(0)=0$. $L(m+1)=M+1$. $L(i-1)<L(i)$ and $L(i)-L(i-1)=M(i)-M(i-1)+1$ for any $i\in\{1,2,\ldots,m+1\}$.
$L(i)=i$ for any $i\in\{0,1,\ldots,\bar{m}\}$.
\item
For any $i\in\{1,2,\ldots,m+1\}$, $H_E(i)\in\widetilde{\mathcal{C}}_1^\circ$ and $I(H_E(i))=L(i)$.

For any $i\in\{1,2,\ldots,\bar{m}\}$, $I(H_E(i))=i$.
\item
Consider any $\Gamma\in\widetilde{\mathcal{C}}_1^\circ$ and any $i\in\{1,2,\ldots,m+1\}$.
$\Gamma\subset|\mathcal{B}_E(i)|$ and $\Gamma\not\subset|\mathcal{B}_E(i)-( \mathcal{B}_E(i)/H_E(i))|
\Leftrightarrow L(i-1)<I(\Gamma)\leq L(i)
\Leftrightarrow |\mathcal{A}(\Gamma)|\subset|\mathcal{B}_E(i)|
\Leftrightarrow |\mathcal{A}^\circ(\Gamma)|^\circ\subset|\mathcal{B}_E(i)|$.
\item
If $\Gamma\in\widetilde{\mathcal{C}}_1^\circ$, $i\in\{1,2,\ldots,m+1\}$,
$\Gamma\subset|\mathcal{B}_E(i)|$ and $\Gamma\not\subset|\mathcal{B}_E(i)-( \mathcal{B}_E(i)/H_E(i))|$,
then 
$I(\Gamma)=L(i-1)+ I(V,N, H_E(i),\bar{\mathcal{C}}(i),S, M(i)-M(i-1),\bar{F}(i))(\Gamma)$, and
$\mathcal{A}(\Gamma)= \mathcal{A}(V,N, H_E(i),\bar{\mathcal{C}}(i),S, M(i)-M(i-1),\bar{F}(i))(\Gamma)$.
\item
If $\Gamma\in\widetilde{\mathcal{C}}_1^\circ$, $i\in\{1,2,\ldots,m\}$,
$\Gamma\subset|\mathcal{B}_E(i)|$ and $\Gamma\not\subset|\mathcal{B}_E(i)-( \mathcal{B}_E(i)/H_E(i))|$,
then $\mathcal{A}^\circ(\Gamma)= \{\Theta\in\mathcal{A}^\circ(V,N, H_E(i),\bar{\mathcal{C}}(i),S, M(i)-M(i-1),\bar{F}(i))(\Gamma)|\Theta^\circ\subset|\mathcal{B}_E/G_E(i)|^\circ\}$.

If $\Gamma\in\widetilde{\mathcal{C}}_1^\circ$, $\Gamma\subset|\mathcal{B}_E(m+1)|$ and $\Gamma\not\subset|\mathcal{B}_E(m+1)-( \mathcal{B}_E(m+1)/H_E(m+1))|$,
then $\mathcal{A}^\circ(\Gamma)= \mathcal{A}^\circ(V,N, H_E(m+1),\bar{\mathcal{C}}(m+1),S, M(m+1)-M(m),\bar{F}(m+1))(\Gamma)$.
\item
For any  $i\in\{1,2,\ldots,m\}$,
$\mathcal{A}^\circ(H_E(i))= \{\Theta\in\mathcal{A}(H_E(i))|\Theta^\circ\subset|\mathcal{B}_E/G_E(i)|^\circ\}$.

$\mathcal{A}^\circ(H_E(m+1))=\mathcal{A}(H_E(m+1))$.
\item
Consider any $\Gamma\in \widetilde{\mathcal{C}}_1^\circ$ and any $i\in\{1,2,\ldots,m+1\}$ satisfying $\Gamma\subset|\mathcal{B}_E(i)|$ and $\Gamma\not\subset|\mathcal{B}_E(i)-(\mathcal{B}_E(i)/H_E(i))|$.

If $\Theta\in\mathcal{A}(\Gamma)/\Gamma$, $\Delta\in\mathcal{C}$
and $\Theta \subset\Delta$, then $H_E(i)\subset\Delta$.

If $\Theta\in\mathcal{A}^\circ(\Gamma)$, $\Delta\in\mathcal{C}/H$
and $\Theta \subset\Delta$, then $H_E(i)\subset\Delta$.
\end{enumerate}
\end{lemma}

\begin{theorem}
\label{important}
Assume $\mathcal{HC}(V,N,S)\neq\emptyset$.
Consider any $(H,\mathcal{C})\in\mathcal{HC}(V,N,S)$ and any $(M,F)\in\mathcal{USD}(H,\mathcal{C},S)$.
We denote $\widetilde{\mathcal{C}}=\mathcal{C}*F(1)*F(2)*\cdots*F(M)$,
$\widetilde{\mathcal{C}}_1^\circ=\{\Gamma\in\widetilde{\mathcal{C}}_1|\Gamma\not\subset|\mathcal{C}-(\mathcal{C}/H)|\}$ and
$$\mathcal{A}^\circ=\mathcal{A}^\circ(V,N, H,\mathcal{C},S, M,F): \widetilde{\mathcal{C}}_1^\circ
\rightarrow 2^{2^{V^*}}.$$

Consider any $\Theta\in\widetilde{\mathcal{C}}$ satisfying $\Theta\not\subset|\mathcal{C}-(\mathcal{C}/H)|$.
\begin{enumerate}
\item
There exists uniquely an element $\Lambda\in\mathcal{D}(S|V)\hat{\cap}\mathcal{C}$ satisfying $\Theta^\circ\subset\Lambda^\circ$.
\item
There exists uniquely an element $\Delta\in\mathcal{C}$ satisfying $\Theta^\circ\subset\Delta^\circ$.
\item
There exists uniquely an element $A\in\mathcal{F}(S+(\Theta^\vee|V^*))$ satisfying
$\Delta(A, S+(\Theta^\vee|V^*)|V)=\Theta$.
\item
There exists uniquely an element $\Gamma\in\widetilde{\mathcal{C}}_1^\circ$ satisfying
$\Theta\in\mathcal{A}^\circ(\Gamma)$.
\end{enumerate}

We take the unique element $\Lambda\in\mathcal{D}(S|V)\hat{\cap}\mathcal{C}$ satisfying $\Theta^\circ\subset\Lambda^\circ$, the unique element $\Delta\in\mathcal{C}$ satisfying $\Theta^\circ\subset\Delta^\circ$,
the unique element $A\in\mathcal{F}(S+(\Theta^\vee|V^*))$ satisfying
$\Delta(A, S+(\Theta^\vee|V^*)|V)=\Theta$ and the unique element $\Gamma\in\widetilde{\mathcal{C}}_1^\circ$ satisfying
$\Theta\in\mathcal{A}^\circ(\Gamma)$.
\begin{enumerate}
\setcounter{enumi}{4}
\item
$\Theta^\circ\subset\Lambda^\circ\subset\Delta^\circ$.
$\Delta\in\mathcal{C}/H$.
$\dim\Lambda=\dim\Delta$ or $\dim\Lambda=\dim\Delta-1$.
$\Gamma\subset\Delta$.
\item
$\dim A=\dim V-\dim\Theta$.
$S+(\Delta^\vee|V^*)\subset S+(\Theta^\vee|V^*)$.
$A\cap(S+(\Delta^\vee|V^*))\in\mathcal{F}(S+(\Delta^\vee|V^*))$.
$\Delta(A\cap(S+(\Delta^\vee|V^*)), S+(\Delta^\vee|V^*)|V)= \Lambda$.
\item
If $\dim\Lambda=\dim\Delta$, then $\langle \omega,a\rangle=\langle\omega,b\rangle$ for any $\omega\in\Vect(\Delta)$, 
any $a\in A\cap(S+(\Delta^\vee|V^*))$, and any $b\in A\cap(S+(\Delta^\vee|V^*))$.
\item
If $\dim\Lambda=\dim\Delta-1$, then $\Ht(H, S+(\Delta^\vee|V^*))>0$ and $\Gamma\in\widetilde{\mathcal{C}}_1$ satisfies the following five conditions:
\begin{enumerate}
\item
$\Gamma\subset\Delta$.
$\Gamma\not\subset\Vect(\Lambda)$.
$\Gamma\not\subset\Lambda$.
$\Gamma\not\subset\Theta$.
$\Theta+\Gamma\in \mathcal{A}^\circ(\Gamma)/\Gamma\subset\widetilde{\mathcal{C}}$.
\item
$\Vect(\Lambda)+\Gamma=\Vect(\Lambda)+H$.
\item
The subset $\{\langle b_{\Gamma/N^*},a\rangle|a\in A\cap(S+(\Delta^\vee|V^*))\}$ of $\R$ is a non-empty bounded closed interval.
\item
\emph{[The hard height inequality]}
\begin{equation*}\begin{split}
&\max\{\langle b_{\Gamma/N^*},a\rangle|a\in A\cap(S+(\Delta^\vee|V^*))\}\\
&\qquad\qquad-\min\{\langle b_{\Gamma/N^*},a\rangle|a\in A\cap(S+(\Delta^\vee|V^*))\}\\
\leq\:&\Ht(H, S+(\Delta^\vee|V^*)).
\end{split}\end{equation*}
\item
The equality 
\begin{equation*}\begin{split}
&\max\{\langle b_{\Gamma/N^*},a\rangle|a\in A\cap(S+(\Delta^\vee|V^*))\}\\
&\qquad\qquad-\min\{\langle b_{\Gamma/N^*},a\rangle|a\in A\cap(S+(\Delta^\vee|V^*))\}\\
=\:&\Ht(H, S+(\Delta^\vee|V^*)),
\end{split}\end{equation*}
holds, if and only if, $c(S+(\Delta^\vee|V^*))=2$ and
the structure constant of $\mathcal{D}(S|V)\hat{\cap}\mathcal{F}(\Delta)$ corresponding to the pair $(2,\bar{E})$ is an integer for any $\bar{E}\in\mathcal{F}(\Delta)_1-\{H\}$.
\end{enumerate}

\item
If $\dim\Lambda=\dim\Delta-1$ and the equivalent conditions in $8$.\emph{(e)} are satisfied, then $\Theta=\Lambda$ and $\Gamma=H$.
\end{enumerate}

Below we consider any rational convex pseudo polyhedrons $T$ and $U$ over $N$ in $V$ satisfying $T+U=S$.
\begin{enumerate}\setcounter{enumi}{9}\item
$\mathcal{D}(T|V)\hat{\cap}\mathcal{D}(U|V)= \mathcal{D}(S|V)$.
$(H,\mathcal{C})\in\mathcal{HC}(V,N,T)$.
$(H,\mathcal{C})\in\mathcal{HC}(V,N,U)$.
\item
If $\dim\Lambda=\dim\Delta-1$, then $\Lambda\in\mathcal{D}(T+(\Delta^\vee|V^*)|V)^1$ or $\Lambda\in\mathcal{D}(U+(\Delta^\vee|V^*)|V)^1$.
\item
Assume $\dim\Lambda=\dim\Delta-1$ and $\Lambda\in\mathcal{D}(T+(\Delta^\vee|V^*)|V)^1$.

$\Ht(H, T+(\Delta^\vee|V^*))>0$. There exists uniquely an element $A_T\in\mathcal{F}(T+(\Delta^\vee|V^*))$ satisfying $\Delta(A_T, T+(\Delta^\vee|V^*)|V)=\Lambda$.

We take the unique element $A_T\in\mathcal{F}(T+(\Delta^\vee|V^*))$ satisfying $\Delta(A_T, T+(\Delta^\vee|V^*)|V)=\Lambda$. The element $\Gamma\in\widetilde{\mathcal{C}}_1^\circ$ satisfies the following three conditions:
\begin{enumerate}
\item
The subset $\{\langle b_{\Gamma/N^*},a\rangle|a\in A_T\}$ of $\R$ is a non-empty bounded closed interval.
\item
\emph{[The hard height inequality]}
\begin{equation*}\begin{split}
&\max\{\langle b_{\Gamma/N^*},a\rangle|a\in A_T\}\\
&\qquad\qquad-\min\{\langle b_{\Gamma/N^*},a\rangle|a\in A_T\}\\
\leq\:&\Ht(H, T+(\Delta^\vee|V^*)).
\end{split}\end{equation*}
\item
Let $r_U=c(U+(\Delta^\vee|V^*))\in\Z_+$.
If $\max\{\langle b_{\Gamma/N^*},a\rangle|a\in A_T\} 
-\min\{\langle b_{\Gamma/N^*},$\break$a\rangle|a\in A_T\}
=Ht(H, T+(\Delta^\vee|V^*))$ and the structure constant of $\mathcal{D}(U+(\Delta^\vee|V^*)|V)$ corresponding to the pair $(i,\bar{E})$ is an integer for any $i\in\{1,2,\ldots,r_U\}$ and any $\bar{E}\in\mathcal{F}(H\Op|\Delta)_1$, then
$c(T+(\Delta^\vee|V^*))=2$, the structure constant of $\mathcal{D}(T+(\Delta^\vee|V^*)|V)$ corresponding to the pair $(2,\bar{E})$ is an integer for any $\bar{E}\in\mathcal{F}(H\Op|\Delta)_1$ and $\Theta=\Lambda$.
\end{enumerate}
\end{enumerate}
\end{theorem}

\begin{proof}

We show only claim $8$ and $9$.

Assume $\mathcal{HC}(V,N,S)\neq\emptyset$.
Consider any $(H,\mathcal{C})\in\mathcal{HC}(V,N,S)$ and any $(M,F)\in\mathcal{USD}(H,\mathcal{C},S)$.
We denote $\widetilde{\mathcal{C}}=\mathcal{C}*F(1)*F(2)*\cdots*F(M)$,
$\widetilde{\mathcal{C}}_1^\circ=\{\Gamma\in\widetilde{\mathcal{C}}_1|\Gamma\not\subset|\mathcal{C}-(\mathcal{C}/H)|\}$ and
$\mathcal{A}^\circ=\mathcal{A}^\circ(V,N, H,\mathcal{C},S, M,F): \widetilde{\mathcal{C}}_1^\circ
\rightarrow 2^{2^{V^*}}$.

Consider any $\Theta\in\widetilde{\mathcal{C}}$ satisfying $\Theta\not\subset|\mathcal{C}-(\mathcal{C}/H)|$.

We take the unique element $\Lambda\in\mathcal{D}(S|V)\hat{\cap}\mathcal{C}$ satisfying $\Theta^\circ\subset\Lambda^\circ$, the unique element $\Delta\in\mathcal{C}$ satisfying $\Theta^\circ\subset\Delta^\circ$,
the unique element $A\in\mathcal{F}(S+(\Theta^\vee|V^*))$ satisfying
$\Delta(A, S+(\Theta^\vee|V^*)|V)=\Theta$ and the unique element $\Gamma\in\widetilde{\mathcal{C}}_1^\circ$ satisfying
$\Theta\in\mathcal{A}^\circ(\Gamma)$.

By $5$ and $6$ we know that 
$\Theta^\circ\subset\Lambda^\circ\subset\Delta^\circ$,
$\Delta\in\mathcal{C}/H$,
$\dim\Lambda=\dim\Delta$ or $\dim\Lambda=\dim\Delta-1$,
$\Gamma\subset\Delta$,
$\dim A=\dim V-\dim\Theta$,
$S+(\Delta^\vee|V^*)\subset S+(\Theta^\vee|V^*)$,
$A\cap(S+(\Delta^\vee|V^*))\in\mathcal{F}(S+(\Delta^\vee|V^*))$, and
$\Delta(A\cap(S+(\Delta^\vee|V^*)), S+(\Delta^\vee|V^*)|V)= \Lambda$.

Furthermore, assume $\dim\Lambda=\dim\Delta-1$.

Note that $H\in\mathcal{F}(\Delta)_1$ and $\mathcal{D}( S+(\Delta^\vee|V^*)|V)$ is $H$-simple, since $\Delta\in\mathcal{C}/H$.
Let $\bar{\mathcal{D}}(S+(\Delta^\vee|V^*)|V)^1=\{\bar{\Lambda}\in\mathcal{D}( S+(\Delta^\vee|V^*)|V)^1|\bar{\Lambda}^\circ\subset\Delta^\circ\}\cup\{H\Op|\Delta\}$ denote the $H$-skeleton of $\mathcal{D}( S+(\Delta^\vee|V^*)|V)$.
We know $\Lambda\in\bar{\mathcal{D}}(S+(\Delta^\vee|V^*)|V)^1$, $\Lambda\neq H\Op|\Delta$, $c(S+(\Delta^\vee|V^*))=\sharp \bar{\mathcal{D}}(S+(\Delta^\vee|V^*)|V)^1\geq 2$, and $\Ht(H, S+(\Delta^\vee|V^*))>0$.

We consider the $H$-order on $\mathcal{D}( S+(\Delta^\vee|V^*)|V)^0$.
Let $\hat{\Lambda}:\{1,2,\ldots, c(S+(\Delta^\vee|V^*))\}\rightarrow\mathcal{D}( S+(\Delta^\vee|V^*)|V)^0$ denote the unique bijective mapping preserving the $H$-order.
Let $\ell=\dim\Vect(\Delta)^\vee|V^*\in\Z_0$ and let $\hat{A}:\{1,2,\ldots, c(S+(\Delta^\vee|V^*))\}\rightarrow\mathcal{F}(S+(\Delta^\vee))_\ell$ denote the unique bijective mapping satisfying $\Delta(\hat{A}(i), $\hfill\break$S+(\Delta^\vee|V^*)|V)= \hat{\Lambda}(i)$ for any $i\in\{1,2,\ldots, c(S+(\Delta^\vee|V^*))\}$.
For any $i\in\{1,2,\ldots, c(S+(\Delta^\vee|V^*))\}$, we take any point $\hat{a}(i)\in \hat{A}(i)$.

We consider the $H$-order on $\bar{\mathcal{D}}(S+(\Delta^\vee|V^*)|V)^1$.
Let $\hat{\bar{\Lambda}}:\{1,2,\ldots, c(S+(\Delta^\vee|V^*))\}\rightarrow\bar{\mathcal{D}}(S+(\Delta^\vee|V^*)|V)^1$ denote the unique bijective mapping preserving the $H$-order.
We take the unique element $i_\Lambda\in\{2,3,\ldots, c(S+(\Delta^\vee|V^*))\}$ satisfying $\hat{\bar{\Lambda}}( i_\Lambda)=\Lambda$.

Now, by $5$ we know $\Gamma\subset\Delta$. 

If $\Gamma\subset\Theta$, then $\Theta=\Theta+\Gamma\in\mathcal{A}^\circ(\Gamma)/\Gamma$.
If $\Gamma\not\subset\Theta$, then $\Theta+\Gamma\in\mathcal{A}^\circ(\Gamma)/\Gamma$ by Lemma~\ref{lower parts}.12.
We know $\Theta+\Gamma\in\mathcal{A}^\circ(\Gamma)/\Gamma \subset\mathcal{A}^\circ(\Gamma)\subset\widetilde{\mathcal{C}}$.

Since $\Delta(A\cap(S+(\Delta^\vee|V^*)), S+(\Delta^\vee|V^*)|V)=\Lambda=\hat{\bar{\Lambda}}( i_\Lambda)=
\hat{\Lambda}(i_\Lambda)\cap\hat{\Lambda}(i_\Lambda-1)$ and $\Lambda^\circ\subset\Delta^\circ$,
we know that
$A\cap(S+(\Delta^\vee|V^*))=\Conv(\hat{A}(i_\Lambda)\cup\hat{A}(i_\Lambda-1))=\Conv(\{\hat{a}(i_\Lambda), \hat{a}(i_\Lambda-1)\})+(\Vect(\Delta)^\vee|V^*)$.
Therefore, if $\Vect(\Lambda)+\Gamma=\Vect(\Lambda)+H$, then
$\langle b_{\Gamma/N^*}, \hat{a}(i_\Lambda)\rangle <\langle b_{\Gamma/N^*}, \hat{a}(i_\Lambda-1)\rangle$,
$\{\langle b_{\Gamma/N^*}, a\rangle|a\in A\cap(S+(\Delta^\vee|V^*))\}
=\{t\in\R|\langle b_{\Gamma/N^*}, \hat{a}(i_\Lambda)\rangle \leq t\leq\langle b_{\Gamma/N^*}, \hat{a}(i_\Lambda-1)\rangle\}$, and the subset
$\{\langle b_{\Gamma/N^*}, a\rangle|a\in A\cap(S+(\Delta^\vee|V^*))$ of $\R$ is a non-empty bounded closed interval.
If $\Gamma\subset\Vect(\Lambda)$, then 
$\langle b_{\Gamma/N^*}, \hat{a}(i_\Lambda)\rangle =\langle b_{\Gamma/N^*}, \hat{a}(i_\Lambda-1)\rangle$,
$\{\langle b_{\Gamma/N^*}, a\rangle|a\in A\cap(S+(\Delta^\vee|V^*))\}
=\{\langle b_{\Gamma/N^*}, \hat{a}(i_\Lambda)\rangle\}$, and the subset
$\{\langle b_{\Gamma/N^*}, a\rangle|a\in A\cap(S+(\Delta^\vee|V^*))$ of $\R$ is a non-empty bounded closed interval.
If $\Gamma\not\subset\Vect(\Lambda)$ and $\Vect(\Lambda)+\Gamma\neq\Vect(\Lambda)+H$, then 
$\langle b_{\Gamma/N^*}, \hat{a}(i_\Lambda)\rangle >\langle b_{\Gamma/N^*}, \hat{a}(i_\Lambda-1)\rangle$,
$\{\langle b_{\Gamma/N^*}, a\rangle|a\in A\cap(S+(\Delta^\vee|V^*))\}
=\{t\in\R|\langle b_{\Gamma/N^*}, \hat{a}(i_\Lambda)\rangle \geq t\geq\langle b_{\Gamma/N^*}, \hat{a}(i_\Lambda-1)\rangle\}$, and the subset
$\{\langle b_{\Gamma/N^*}, a\rangle|a\in A\cap(S+(\Delta^\vee|V^*))$ of $\R$ is a non-empty bounded closed interval.

We know that the subset
$\{\langle b_{\Gamma/N^*}, a\rangle|a\in A\cap(S+(\Delta^\vee|V^*))$ of $\R$ is a non-empty bounded closed interval.

Let $\hat{M}=\sharp\{i\in\{1,2,\ldots,M\}|F(i)\subset\Delta\}\in\Z_0$.
We take the unique injective mapping $\tau:\{1,2,\ldots, \hat{M}\}\rightarrow\{1,2,\ldots,M\}$ preserving the order and satisfying $\tau(\{1,2,\ldots,\hat{M}\})=\{i\in\{1,2,\ldots,M\}|F(i)\subset\Delta\}$.
By Therem~\ref{usd1}.6 we know that $(H, \mathcal{F}(\Delta))\in \mathcal{HC}(V, N, S)$, 
$(\hat{M}, F\tau)\in\mathcal{USD}(H,\mathcal{F}(\Delta),S)$, and
$\widetilde{\mathcal{C}}\backslash\Delta=\mathcal{F}(\Delta)*F\tau(1)*F\tau(2)*\cdots*F\tau(\hat{M})$.

We denote $\widetilde{\mathcal{F}}(\Delta)=\mathcal{F}(\Delta)*F\tau(1)*F\tau(2)*\cdots*F\tau(\hat{M})$,
$\smash{\widetilde{\mathcal{F}}(\Delta)}_1^\circ=\{\bar{\Gamma}\in\widetilde{\mathcal{F}}(\Delta)_1|\bar{\Gamma}\not\subset H\Op|\Delta\}$ and
$\smash{\hat{\mathcal{A}}}^\circ=\mathcal{A}^\circ(V,N, H, \mathcal{F}(\Delta),S, \hat{M},F\tau): \smash{\widetilde{\mathcal{F}}(\Delta)}_1^\circ
\rightarrow 2^{2^{V^*}}$.

Since $\Gamma\subset\Delta$, $\Gamma\in\smash{\widetilde{\mathcal{F}}(\Delta)}_1^\circ$ and 
$\smash{\hat{\mathcal{A}}}^\circ(\Gamma)=\mathcal{A}^\circ(\Gamma)\backslash\Delta$.
Since $\Theta\subset\Delta$, $\Theta\in \smash{\hat{\mathcal{A}}}^\circ(\Gamma)$.

For any $j\in\{1,2,\ldots,c(S+(\Delta^\vee|V^*))\}$ and any $\bar{E}\in\mathcal{F}(H\Op|\Delta)_1$, we denote the structure constant of 
$\mathcal{D}( S+(\Delta^\vee|V^*)|V)$ corresponding to the pair $(j, \bar{E})$ by
$c(\mathcal{D}( S+(\Delta^\vee|V^*)|V),j, \bar{E})$.

Let $\hat{m}=\sum_{\bar{E}\in\mathcal{F}(H\Op|\Delta)_1}\lfloor c(\mathcal{D}( S+(\Delta^\vee|V^*)|V),2,\bar{E})\rfloor\in\Z_+$ and\hfill\break $\hat{\bar{m}}=\sum_{\bar{E}\in\mathcal{F}(H\Op|\Delta)_1}\lceil c(\mathcal{D}( S+(\Delta^\vee|V^*)|V),2,\bar{E})\rceil\in\Z_0$.
We know that $\hat{\bar{m}}\leq\hat{m}\leq\hat{M}$.

We take the unique pair $(\hat{E},\hat{M})$ of a compatible mapping
$\hat{E}:\{1,2,\ldots,\hat{m}\}\rightarrow \mathcal{F}(H\Op|\Delta)_1$
with $S$ and a mapping 
$\hat{M}:\{\hat{\bar{m}},\hat{\bar{m}}+1,\ldots,\hat{m}+1\}\rightarrow\Z_+$
satisfying the following three conditions.
We denote 
$$F_{\hat{E}}=F(V^*,N^*,H,\mathcal{F}(\Delta),\hat{m},\hat{E}):\{1,2,\ldots,\hat{m}\}\rightarrow 2^{ V^*}$$ 
$$H_{\hat{E}}=H(V^*,N^*,H,\mathcal{F}(\Delta),\hat{m},\hat{E}):\{1,2,\ldots,\hat{m}+1\}\rightarrow 2^{ V^*},\text{ and}$$
$$\mathcal{B}_{\hat{E}}=\mathcal{B}(V^*,N^*,H,\mathcal{F}(\Delta),\hat{m},\hat{E}):\{1,2,\ldots,\hat{m+1}\}\rightarrow 2^{2^{ V^*}}.$$
\begin{description}
\item[(a)]
$F\tau(j)=F_{\hat{E}}(j)$ for any $j\in\{1,2,\ldots, \hat{m}\}$.
\item[(b)]
$\hat{M}(\hat{\bar{m}})=\hat{m}$, $\hat{M}(\hat{m}+1)=\hat{M}$ and $\hat{M}(i-1)\leq \hat{M}(i)$ for any $i\in\{\hat{\bar{m}}+1,\hat{\bar{m}}+2,\ldots, \hat{m}+1\}$.
\item[(c)]
$F\tau(j)\subset|\mathcal{B}_{\hat{E}}(i)|$ and $F\tau(j)\not\subset|\mathcal{B}_{\hat{E}}(i)-(\mathcal{B}_{\hat{E}}(i)/H_{\hat{E}}(i))|$ for any $i\in\{\hat{\bar{m}}+1,\hat{\bar{m}}+2,\ldots, \hat{m}+1\}$ and any $j\in\{\hat{M}(i-1)+1,\hat{M}(i-1)+2,\ldots, \hat{M}(i)\}$.
\end{description}

For any $i\in\{1,2,\ldots,\hat{m}+1\}$, we denote
$\hat{\Delta}(i)=| \mathcal{B}_{\hat{E}}(i)|\subset\Delta$ and 
$\hat{\bar{\Delta}}(i)= H_{\hat{E}}(i)\Op|\hat{\Delta}(i)\subset\hat{\Delta}(i)$.
For any $i\in\{0,1,\ldots, \hat{\bar{m}}-1\}$ we put $\hat{M}(i)=\hat{m}$.
For any $i\in\{1,2,\ldots,\hat{m}+1\}$, 
let $\bar{\mathcal{C}}(i)=\mathcal{F}(\Delta)*F\tau(1)*F\tau(2)*\cdots*F\tau(\hat{M}(i-1))\backslash \hat{\Delta}(i)$ and 
let $\bar{F}(i):\{1,2,\ldots,\hat{M}(i)-\hat{M}(i-1)\}\rightarrow 2^{V^*}$ denote the mapping satisfying $\bar{F}(i)(j)=F\tau(\hat{M}(i-1)+j)$ for any $j\in \{1,2,\ldots,\hat{M}(i)-\hat{M}(i-1)\}$.
Note that we have 
$|\bar{\mathcal{C}}(i)|= \hat{\Delta}(i)$,
$(H_{\hat{E}}(i), \bar{\mathcal{C}}(i))\in\mathcal{HC}(V,N,S)$,
$(\hat{M}(i)-\hat{M}(i-1), \bar{F}(i))\in\mathcal{USD}(H_{\hat{E}}(i), \bar{\mathcal{C}}(i),S)$,
$\Ht(H_{\hat{E}}(i), \bar{\mathcal{C}}(i),S)=\Ht(H_{\hat{E}}(i), S+(\hat{\Delta}(i)^\vee|V^*))<\Ht(H, S+(\Delta^\vee|V^*))$ for any 
$i\in\{1,2,\ldots,\hat{m}+1\}$ by Theorem~\ref{usd1}.10.

There exists uniquely $i_\Gamma\in\{1,2,\ldots,\hat{m}+1\}$ satisfying
$\Gamma\subset\hat{\Delta}(i_\Gamma)$ and $\Gamma\not\subset \hat{\bar{\Delta}}(i_\Gamma)$.
We take the unique  $i_\Gamma\in\{1,2,\ldots,\hat{m}+1\}$ satisfying
$\Gamma\subset\hat{\Delta}(i_\Gamma)$ and $\Gamma\not\subset \hat{\bar{\Delta}}(i_\Gamma)$.

By Lemma~\ref{lower parts}.17 we know 
$\Theta\in\smash{\hat{\mathcal{A}}}^\circ(\Gamma)\subset
\mathcal{A}^\circ(V,N, H_{\hat{E}}(i_\Gamma), \bar{\mathcal{C}}(i_\Gamma), S, 
\hat{M}(i_\Gamma)- \hat{M}(i_\Gamma-i),\bar{F}(i_\Gamma))(\Gamma)$, and
$\Theta\subset|\smash{\hat{\mathcal{A}}}^\circ(\Gamma)|
\subset|\mathcal{A}^\circ(V,N, H_{\hat{E}}(i_\Gamma), \bar{\mathcal{C}}(i_\Gamma), S, 
\hat{M}(i_\Gamma)- \hat{M}(i_\Gamma-i),\bar{F}(i_\Gamma))(\Gamma)|\subset
|\bar{\mathcal{C}}(i_\Gamma)|=\hat{\Delta}(i_\Gamma)$.
Consider the case $i_\Gamma=\hat{m}+1$. $\Theta^\circ\subset \hat{\Delta}(i_\Gamma)\cap \Delta^\circ=
\hat{\Delta}(i_\Gamma)^\circ\cup\hat{\bar{\Delta}}(i_\Gamma)^\circ $.
Consider the case $i_\Gamma\neq\hat{m}+1$.
We denote  
$$G_{\hat{E}}=G(V^*,N^*,H,\mathcal{F}(\Delta),\hat{m},\hat{E}):\{1,2,\ldots,\hat{m}\}\rightarrow 2^{ V^*}.$$
By Lemma~\ref{lower parts}.17
$\Theta^\circ\subset|\mathcal{F}(\hat{\Delta}(i_\Gamma))/ G_{\hat{E}}(i_\Gamma)|^\circ=
\hat{\Delta}(i_\Gamma)- (G_{\hat{E}}(i_\Gamma)\Op|\Delta(i_\Gamma))$.
Therefore,
$\Theta^\circ\subset(\hat{\Delta}(i_\Gamma)- (G_{\hat{E}}(i_\Gamma)\Op|\Delta(i_\Gamma)))\cap \Delta^\circ=
\hat{\Delta}(i_\Gamma)^\circ\cup\hat{\bar{\Delta}}(i_\Gamma)^\circ $.

We know $\Theta^\circ\subset \hat{\Delta}(i_\Gamma)^\circ\cup\hat{\bar{\Delta}}(i_\Gamma)^\circ$.

Since $\Theta\in\widetilde{\mathcal{F}}(\Delta)$, $\{\hat{\Delta}(i_\Gamma), \hat{\bar{\Delta}}(i_\Gamma)\}\subset \mathcal{F}(\Delta)*F\tau(1)*F\tau(2)*\cdots*F\tau(\hat{m})$, and $\widetilde{\mathcal{F}}(\Delta)$ is a subdivision of $\mathcal{F}(\Delta)*F\tau(1)*F\tau(2)*\cdots*F\tau(\hat{m})$, we know that $\Theta^\circ\subset \hat{\Delta}(i_\Gamma)^\circ$ or $\Theta^\circ\subset \hat{\bar{\Delta}}(i_\Gamma)^\circ$.

We take the unique element $\Delta_0\in\bar{\mathcal{C}}(i_\Gamma)/ H_{\hat{E}}(i_\Gamma)$ satisfying $\Theta^\circ\subset \Delta_0^\circ\cup(H_{\hat{E}}(i_\Gamma)\Op$\hfill\break$|\Delta_0)^\circ$.
$\Theta^\circ\subset \hat{\Delta}(i_\Gamma)^\circ$, if and only if, $\Theta^\circ\subset \Delta_0^\circ$.
$\Theta^\circ\subset \hat{\bar{\Delta}}(i_\Gamma)^\circ$, if and only if,
$\Theta^\circ\subset(H_{\hat{E}}(i_\Gamma)\Op|\Delta_0)^\circ$.

Since $\Theta\in\mathcal{A}^\circ(V,N, H_{\hat{E}}(i_\Gamma), \bar{\mathcal{C}}(i_\Gamma), S, 
\hat{M}(i_\Gamma)- \hat{M}(i_\Gamma-i),\bar{F}(i_\Gamma))(\Gamma)$ and $\Theta\subset\Delta_0$, we know $\Gamma\subset\Delta_0$.
$\Ht(H_{\hat{E}}( i_\Gamma), S+(\Delta_0^\vee|V^*))\leq\Ht(H_{\hat{E}}( i_\Gamma), \bar{\mathcal{C}}( i_\Gamma),S)<\Ht(H,S+(\Delta^\vee|V^*))$.
$\Lambda\cap\Delta_0\in\mathcal{D}(S|V)\hat{\cap}\mathcal{C}\hat{\cap}\bar{\mathcal{C}}(i_\Gamma)= \mathcal{D}(S|V)\hat{\cap}\bar{\mathcal{C}}(i_\Gamma)$.
$\emptyset\neq\Theta^\circ\subset\Lambda\cap(\hat{\Delta}(i_\Gamma)^\circ\cup\hat{\bar{\Delta}}(i_\Gamma)^\circ)$.
$\Theta\subset\Lambda\cap\Delta_0$.

Let $\hat{s}=s(V^*,N^*,H,\mathcal{F}(\Delta),\hat{m},\hat{E}):\{0,1,\ldots,\hat{m}\}\times\mathcal{F}(H\Op|\Delta)_1\rightarrow\Z_0$.

Assume that $i_\Gamma\neq\hat{m}+1$.
By Theorem~\ref{heightinequality}.$15$ we know
$\hat{m}- \hat{\bar{m}}\geq 1$,
$i_\Gamma\in\{\hat{\bar{m}}+1,\hat{\bar{m}}+2,\ldots,\hat{m}\}$, since 
$\Lambda\cap(\hat{\Delta}(i_\Gamma)^\circ\cup\hat{\bar{\Delta}}(i_\Gamma)^\circ)\neq\emptyset$.
Since $\Lambda=\hat{\bar{\Lambda}}(i_\Lambda)$ and
$\Lambda\cap(\hat{\Delta}(i_\Gamma)^\circ\cup\hat{\bar{\Delta}}(i_\Gamma)^\circ)\neq\emptyset$, we know 
$c(\mathcal{D}(S+(\Delta^\vee|V^*)|V),i_\Lambda,\hat{E}(i_\Gamma))<
\lfloor c(\mathcal{D}(S+(\Delta^\vee|V^*)|V),2,\hat{E}(i_\Gamma))\rfloor=
\hat{s}(i_\Gamma,\hat{E}(i_\Gamma))$ by Theorem~\ref{heightinequality}.$17$.

It follows that if $i_\Gamma\neq\hat{m}+1$, then $H_{\hat{E}}(i_\Gamma)\not\subset\Vect(\Lambda)$ and $\Vect(\Lambda)+ H_{\hat{E}}(i_\Gamma)= \Vect(\Lambda)+H$.
Obviously, if $i_\Gamma=\hat{m}+1$, then $H_{\hat{E}}(i_\Gamma)=H$, $H_{\hat{E}}(i_\Gamma)\not\subset\Vect(\Lambda)$ and $\Vect(\Lambda)+ H_{\hat{E}}(i_\Gamma)= \Vect(\Lambda)+H$.
We know that $H_{\hat{E}}(i_\Gamma)\not\subset\Vect(\Lambda)$ and $\Vect(\Lambda)+ H_{\hat{E}}(i_\Gamma)= \Vect(\Lambda)+H$.
Furthermore, it follows that
$\dim \Lambda\cap\Delta_0\leq\dim \Delta_0-1$.

We have two cases.
\begin{enumerate}
\item
$\Theta^\circ\subset \hat{\Delta}(i_\Gamma)^\circ$.
\item
$\Theta^\circ\subset \hat{\bar{\Delta}}(i_\Gamma)^\circ$.
\end{enumerate}

We consider the case $\Theta^\circ\subset \hat{\Delta}(i_\Gamma)^\circ$.
$\Theta^\circ\subset \Delta_0^\circ$.
$\Lambda\cap\Delta_0\in\mathcal{D}(S|V)\hat{\cap}\mathcal{F}(\Delta_0)$.
$\Theta^\circ\subset\Lambda^\circ\cap\Delta_0^\circ=(\Lambda\cap\Delta_0)^\circ$,
$\emptyset\neq\Theta^\circ\subset(\Lambda\cap\Delta_0)^\circ\cap\Delta_0^\circ$.
Since $\mathcal{D}(S|V)\hat{\cap}\mathcal{F}(\Delta_0)$ is $H_{\hat{E}}(i_\Gamma)$-simple, we know $\dim (\Lambda\cap\Delta_0)\geq\dim\Delta_0-1$, $\dim (\Lambda\cap\Delta_0)=\dim\Delta_0-1$ and $\Vect(\Lambda\cap\Delta_0)=\Vect(\Lambda)\cap\Vect(\Delta_0)$.

Since $\Ht(H_{\hat{E}}(i_\Gamma), S+(\Delta_0^\vee|V^*))<\Ht(H, S+(\Delta^\vee|V^*)$, by induction on $\Ht$ we know that the following claims hold:
\begin{description}
\item[(a)] $\Gamma\subset\Delta_0$. $\Gamma\not\subset\Vect(\Lambda\cap\Delta_0)$.
\item[(b)]
$\Vect(\Lambda\cap\Delta_0)+\Gamma=\Vect(\Lambda\cap\Delta_0)+ H_{\hat{E}}(i_\Gamma)$.
\item[(c)]
The subset $\{\langle b_{\Gamma/N^*},a\rangle|a\in A\cap(S+(\Delta_0^\vee|V^*))\}$ of $\R$ is a non-empty bounded closed interval.
\item[(d)]
\begin{equation*}\begin{split}
&\max\{\langle b_{\Gamma/N^*},a\rangle|a\in A\cap(S+(\Delta_0^\vee|V^*))\}\\
&\qquad\qquad -\min\{\langle b_{\Gamma/N^*},a\rangle|a\in A\cap(S+(\Delta_0^\vee|V^*))\}\\
\leq\:&\Ht(H_{\hat{E}}(i_\Gamma), S+(\Delta_0^\vee|V^*))
\end{split}\end{equation*}
\end{description}

Since $\Gamma\subset\Delta_0\subset\Vect(\Delta_0)$ and $\Gamma\not\subset\Vect(\Lambda\cap\Delta_0)=\Vect(\Lambda)\cap\Vect(\Delta_0)$,
we know $\Gamma\not\subset\Vect(\Lambda)$.
Since $\Theta\subset\Lambda\subset\Vect(\Lambda)$, we know $\Gamma\not\subset\Lambda$ and $\Gamma\not\subset\Theta$.

Note that 
$\Lambda=\Vect(\Lambda)\cap\Delta$.
Since $ H_{\hat{E}}(i_\Gamma)\not\subset\Lambda$ and $ H_{\hat{E}}(i_\Gamma)\subset\Delta$, we know $ H_{\hat{E}}(i_\Gamma)\not\subset\Vect(\Lambda)$.
Since $\Vect(\Lambda\cap\Delta_0)+\Gamma=\Vect(\Lambda\cap\Delta_0)+ H_{\hat{E}}(i_\Gamma)$ and $\Vect(\Lambda\cap\Delta_0)\subset\Vect(\Lambda)$, we know
$\Vect(\Lambda)+\Gamma=\Vect(\Lambda)+ H_{\hat{E}}(i_\Gamma) =\Vect(\Lambda)+H$.

Since $\Delta_0\subset\Delta$, we know that
$A\cap(S+(\Delta^\vee|V^*)\subset A\cap(S+(\Delta_0^\vee|V^*)$,\hfill\break
$\max\{\langle b_{\Gamma/N^*},a\rangle|a\in A\cap(S+(\Delta^\vee|V^*))\}
\leq \max\{\langle b_{\Gamma/N^*},a\rangle|a\in A\cap(S+(\Delta_0^\vee|V^*))\}$, 
\hfill\break
$\min\{\langle b_{\Gamma/N^*},a\rangle|a\in A\cap(S+(\Delta^\vee|V^*))\}
\geq \min\{\langle b_{\Gamma/N^*},a\rangle|a\in A\cap(S+(\Delta_0^\vee|V^*))\}$,
\hfill\break
$\max\{\langle b_{\Gamma/N^*},a\rangle|a\in A\cap(S+(\Delta^\vee|V^*))\}
-\min\{\langle b_{\Gamma/N^*},a\rangle|a\in A\cap(S+(\Delta^\vee|V^*))\}
\leq\max\{\langle b_{\Gamma/N^*},a\rangle|a\in A\cap(S+(\Delta_0^\vee|V^*))\}
-\min\{\langle b_{\Gamma/N^*},a\rangle|a\in A\cap(S+(\Delta_0^\vee|V^*))\}
\leq\Ht(H_{\hat{E}}(i_\Gamma), S+(\Delta_0^\vee|V^*))<\Ht(H,S+(\Delta^\vee|V^*))$, and
\begin{equation*}\begin{split}
&\max\{\langle b_{\Gamma/N^*},a\rangle|a\in A\cap(S+(\Delta^\vee|V^*))\}\\
&\qquad\qquad -\min\{\langle b_{\Gamma/N^*},a\rangle|a\in A\cap(S+(\Delta^\vee|V^*))\}\\
<\:&\Ht(H,S+(\Delta^\vee|V^*))
\end{split}\end{equation*}

We consider the case $\Theta^\circ\subset \hat{\bar{\Delta}}(i_\Gamma)^\circ$.

$\Theta\in\widetilde{\mathcal{C}}\backslash\hat{\bar{\Delta}}(i_\Gamma)
=(\widetilde{\mathcal{C}}\backslash\hat{\Delta}(i_\Gamma)) \backslash\hat{\bar{\Delta}}(i_\Gamma)
=(\bar{\mathcal{C}}(i_\Gamma)*\bar{F}(i_\Gamma)(2)* \bar{F}(i_\Gamma)(1)*\cdots*\bar{F}(i_\Gamma)(\hat{M}(i_\Gamma)-\hat{M}(i_\Gamma-1))) \backslash\hat{\bar{\Delta}}(i_\Gamma)
=\bar{\mathcal{C}}(i_\Gamma)-( \bar{\mathcal{C}}(i_\Gamma)/ H_{\hat{E}}(i_\Gamma))$ and we know $\Theta+ H_{\hat{E}}(i_\Gamma)\in\bar{\mathcal{C}}(i_\Gamma)$.

Since $\Theta= H_{\hat{E}}(i_\Gamma)\Op|(\Theta+ H_{\hat{E}}(i_\Gamma))$, we know $\Delta_0=\Theta+ H_{\hat{E}}(i_\Gamma)$ and $\dim\Theta=\dim\Delta_0-1$.

Since $\Theta\subset\Lambda\cap\Delta_0$, $\dim\Delta_0-1=\dim\Theta\leq\dim\Lambda\cap\Delta_0\leq\dim\Delta_0-1$,
$\dim\Theta=\dim\Lambda\cap\Delta_0=\dim\Delta_0-1$,
$\Vect(\Theta)=\Vect(\Lambda\cap\Delta_0)$ and
$\Lambda\cap\Delta_0\subset\Vect(\Lambda\cap\Delta_0)\cap\Delta_0=
\Vect(\Theta)\cap\Delta_0=\Vect(\Theta)\cap(\Theta+ H_{\hat{E}}(i_\Gamma))=\Theta$.
We know $\Theta=\Lambda\cap\Delta_0$.

Since $\Gamma\not\subset\hat{\bar{\Delta}}(i_\Gamma)$ and $\Theta\subset\hat{\bar{\Delta}}(i_\Gamma)$, we know $\Gamma\not\subset\Theta$.
Since $\Gamma\not\subset\Theta=\Lambda\cap\Delta_0$ and $\Gamma\subset\Delta_0$, we know $\Gamma\not\subset\Lambda$.
Since $\Gamma\not\subset\Lambda=\Vect(\Lambda)\cap\Delta$ and $\Gamma\subset\Delta$, we know $\Gamma\not\subset\Vect(\Lambda)$.

Since $\Gamma^\circ\cup H_{\hat{E}}(i_\Gamma)^\circ\subset\Delta_0-(H_{\hat{E}}(i_\Gamma)\Op|\Delta_0)$,
we know $\Vect(\Lambda)+\Gamma=\Vect(\Lambda)+ H_{\hat{E}}(i_\Gamma) =\Vect(\Lambda)+ H$.

We know $\Vect(\Lambda)+\Gamma=\Vect(\Lambda)+ H$,
$\langle b_{\Gamma/N^*}, \hat{a}(i_\Lambda)\rangle <\langle b_{\Gamma/N^*}, \hat{a}(i_\Lambda-1)\rangle$,
$\{\langle b_{\Gamma/N^*}, a\rangle|a\in A\cap(S+(\Delta^\vee|V^*))\}
=\{t\in\R|\langle b_{\Gamma/N^*}, \hat{a}(i_\Lambda)\rangle \leq t\leq\langle b_{\Gamma/N^*}, \hat{a}(i_\Lambda-1)\rangle\}$, and
$\max\{\langle b_{\Gamma/N^*}, a\rangle|a\in A\cap(S+(\Delta^\vee|V^*))\}
-\min\{\langle b_{\Gamma/N^*}, a\rangle|a\in A\cap(S+(\Delta^\vee|V^*))\}
=\langle b_{\Gamma/N^*}, \hat{a}(i_\Lambda-1)\rangle-\langle b_{\Gamma/N^*}, \hat{a}(i_\Lambda)\rangle$.

By Lemma~\ref{lower parts}.12 we know $b_{ H_{\hat{E}}(i_\Gamma)/N^*}-b_{\Gamma/N^*}\in N^*\cap\Vect(\Theta)
\subset\Vect(\Lambda)=\Vect(\hat{\bar{\Lambda}}(i_\Lambda))$.
Since $\{\hat{a}(i_\Lambda-1), \hat{a}(i_\Lambda)\}\subset\Conv(\hat{A}(i_\Lambda-1)\cup\hat{A}(i_\Lambda))\in\mathcal{F}(S+(\Delta^\vee|V^*))_{\ell+1}$ and $\Delta(\Conv(\hat{A}(i_\Lambda-1)\cup\hat{A}(i_\Lambda)), S+(\Delta^\vee|V^*)|V)= \hat{\bar{\Lambda}}(i_\Lambda)$, we know
$\langle b_{ H_{\hat{E}}(i_\Gamma)/N^*}-b_{\Gamma/N^*}, \hat{a}(i_\Lambda-1)\rangle=
\langle b_{H_{\hat{E}}(i_\Gamma)/N^*}-b_{\Gamma/N^*}, \hat{a}(i_\Lambda)\rangle$.
We know
$\langle b_{\Gamma/N^*}, \hat{a}(i_\Lambda-1)\rangle-\langle b_{\Gamma/N^*}, \hat{a}(i_\Lambda)\rangle=
\langle b_{H_{\hat{E}}(i_\Gamma)/N^*}, \hat{a}(i_\Lambda-1)\rangle-\langle b_{H_{\hat{E}}(i_\Gamma)/N^*}, \hat{a}(i_\Lambda)\rangle$.

We have two cases.
\begin{enumerate}
\item $\Lambda\neq\hat{\bar{\Delta}}(i_\Gamma)$.
\item $\Lambda=\hat{\bar{\Delta}}(i_\Gamma)$.
\end{enumerate}

Consider the case $\Lambda\neq\hat{\bar{\Delta}}(i_\Gamma)$.

Since $\hat{\bar{\Lambda}}(i_\Lambda)=\Lambda\neq \hat{\bar{\Delta}}(i_\Gamma)
= H_{\hat{E}}(i_\Gamma)\Op|\hat{\Delta}(i_\Gamma)$, we know  $\hat{\Lambda}(i_\Lambda)\cap \hat{\Delta}(i_\Gamma)^\circ\neq\emptyset$,  $\hat{\Lambda}(i_\Lambda-1)\cap \hat{\Delta}(i_\Gamma)^\circ\neq\emptyset$,
$\{\hat{A}(i_\Gamma), \hat{A}(i_\Gamma-1)\}\subset\mathcal{F}(S+(\hat{\Delta}(i_\Gamma)^\vee|V^*))_\ell$, and 
$\langle b_{H_{\hat{E}}(i_\Gamma)/N^*}, \hat{a}(i_\Lambda-1)\rangle-\langle b_{H_{\hat{E}}(i_\Gamma)/N^*}, \hat{a}(i_\Lambda)\rangle\leq
\Ht(H_{\hat{E}}(i_\Gamma), S+(\hat{\Delta}(i_\Gamma)^\vee|V^*))
<\Ht(H,S+(\Delta^\vee|$\hfill\break$V^*))$.

We know that 
$\max\{\langle b_{\Gamma/N^*}, a\rangle|a\in A\cap(S+(\Delta^\vee|V^*))\}
-\min\{\langle b_{\Gamma/N^*}, a\rangle|a\in A\cap(S+(\Delta^\vee|V^*))\}
<\Ht(H,S+(\Delta^\vee|V^*))$.

Consider the case  $\Lambda=\hat{\bar{\Delta}}(i_\Gamma)$.

If $i_\Gamma=\hat{m}+1$, then $ H_{\hat{E}}(i_\Gamma)=H$ and
$\langle b_{H_{\hat{E}}(i_\Gamma)/N^*}, \hat{a}(i_\Lambda-1)\rangle-\langle b_{H_{\hat{E}}(i_\Gamma)/N^*}, \hat{a}(i_\Lambda)\rangle=
\langle b_{H/N^*}, \hat{a}(i_\Lambda-1)\rangle-\langle b_{H/N^*}, \hat{a}(i_\Lambda)\rangle$.

We consider the case $i_\Gamma\neq\hat{m}+1$. $b_{H_{\hat{E}}(i_\Gamma)/N^*}-b_{H/N^*}=b_{G_{\hat{E}}(i_\Gamma)/N^*}\in 
\hat{\bar{\Delta}}(i_\Gamma)=\Lambda=\hat{\bar{\Lambda}}(i_\Lambda)$.
Therefore, 
$\langle b_{H_{\hat{E}}(i_\Gamma)/N^*}-b_{H/N^*}, \hat{a}(i_\Lambda-1)\rangle=
\langle b_{H_{\hat{E}}(i_\Gamma)/N^*}-b_{H/N^*}, \hat{a}(i_\Lambda)\rangle$, and 
$\langle b_{H_{\hat{E}}(i_\Gamma)/N^*}, \hat{a}(i_\Lambda-1)\rangle-\langle b_{H_{\hat{E}}(i_\Gamma)/N^*}, \hat{a}(i_\Lambda)\rangle=
\langle b_{H/N^*}, \hat{a}(i_\Lambda-1)\rangle-\langle b_{H/N^*}, \hat{a}(i_\Lambda)\rangle$.

We know that 
$\langle b_{H_{\hat{E}}(i_\Gamma)/N^*}-b_{H/N^*}, \hat{a}(i_\Lambda-1)\rangle=
\langle b_{H_{\hat{E}}(i_\Gamma)/N^*}-b_{H/N^*}, \hat{a}(i_\Lambda)\rangle$, and 
$\langle b_{H_{\hat{E}}(i_\Gamma)/N^*}, \hat{a}(i_\Lambda-1)\rangle-\langle b_{H_{\hat{E}}(i_\Gamma)/N^*}, \hat{a}(i_\Lambda)\rangle=
\langle b_{H/N^*}, \hat{a}(i_\Lambda-1)\rangle-\langle b_{H/N^*}, \hat{a}(i_\Lambda)\rangle$.

Note that 
$\langle b_{H/N^*}, \hat{a}(i_\Lambda-1)\rangle-\langle b_{H/N^*}, \hat{a}(i_\Lambda)\rangle
\leq
\langle b_{H/N^*}, \hat{a}(1)\rangle-\langle b_{H/N^*}, \hat{a}(c(S+(\Delta^\vee|V^*)))\rangle
=\Ht(H, S+(\Delta^\vee|V^*))$ and 
$\langle b_{H/N^*}, \hat{a}(i_\Lambda-1)\rangle-\langle b_{H/N^*}, \hat{a}(i_\Lambda)\rangle
=
\langle b_{H/N^*}, \hat{a}(1)\rangle-\langle b_{H/N^*}, \hat{a}(c(S+(\Delta^\vee|V^*)))\rangle$,
if and only if, $c(S+(\Delta^\vee|V^*))=i_\Lambda =2$.

We know that 
$\max\{\langle b_{\Gamma/N^*}, a\rangle|a\in A\cap(S+(\Delta^\vee|V^*))\}
-\min\{\langle b_{\Gamma/N^*}, a\rangle|a\in A\cap(S+(\Delta^\vee|V^*))\}
\leq\Ht(H,S+(\Delta^\vee|V^*))$, and that 
$\max\{\langle b_{\Gamma/N^*}, a\rangle|a\in A\cap(S+(\Delta^\vee|V^*))\}
-\min\{\langle b_{\Gamma/N^*}, a\rangle|a\in A\cap(S+(\Delta^\vee|V^*))\}
=\Ht(H,S+(\Delta^\vee|V^*))$,
if and only if, $c(S+(\Delta^\vee|V^*))=i_\Lambda =2$.

Now, by the arguments until here we know that the inequality 
$\max\{\langle b_{\Gamma/N^*}, a\rangle|a\in A\cap(S+(\Delta^\vee|V^*))\}
-\min\{\langle b_{\Gamma/N^*}, a\rangle|a\in A\cap(S+(\Delta^\vee|V^*))\}
\leq\Ht(H,S+(\Delta^\vee|V^*))$
always holds, and the equality 
$\max\{\langle b_{\Gamma/N^*}, a\rangle|a\in A\cap(S+(\Delta^\vee|V^*))\}
-\min\{\langle b_{\Gamma/N^*}, a\rangle|a\in A\cap(S+(\Delta^\vee|V^*))\}
=\Ht(H,S+(\Delta^\vee|V^*))$
holds, if and only if,  $\Theta^\circ\subset \hat{\bar{\Delta}}(i_\Gamma)^\circ$, 
$\Lambda=\hat{\bar{\Delta}}(i_\Gamma)$ and $c(S+(\Delta^\vee|V^*))=i_\Lambda =2$.

Assume  $\Theta^\circ\subset \hat{\bar{\Delta}}(i_\Gamma)^\circ$, 
$\Lambda=\hat{\bar{\Delta}}(i_\Gamma)$ and $c(S+(\Delta^\vee|V^*))=i_\Lambda =2$.
We have
$\hat{\bar{\Lambda}}(2)= \hat{\bar{\Lambda}}( i_\Lambda)= \Lambda=\hat{\bar{\Delta}}(i_\Gamma)$, and for any $\bar{E}\in\mathcal{F}(H\Op|\Delta)_1=\mathcal{F}(\Delta)_1-\{H\}$, $ c(\mathcal{D}(S|V)\hat{\cap}\mathcal{F}(\Delta), 2,\bar{E})=c(\mathcal{D}(S+(\Delta^\vee|V^*)|V), 2,\bar{E})=\hat{s}( i_\Gamma-1,\bar{E})\in\Z$.

We know that if 
$\max\{\langle b_{\Gamma/N^*}, a\rangle|a\in A\cap(S+(\Delta^\vee|V^*))\}
-\min\{\langle b_{\Gamma/N^*}, a\rangle|a\in A\cap(S+(\Delta^\vee|V^*))\}
=\Ht(H,S+(\Delta^\vee|V^*))$,
then  $c(S+(\Delta^\vee|V^*))= 2$ and the structure constant of  $\mathcal{D}(S|V)\hat{\cap}\mathcal{F}(\Delta)$ corresponding to the pair $(2,\bar{E})$ is an integer for any $\bar{E}\in\mathcal{F}(\Delta)-\{H\}$.

Convesely, assume that $c(S+(\Delta^\vee|V^*))= 2$ and the structure constant of  $\mathcal{D}(S|V)$\break$\hat{\cap}\mathcal{F}(\Delta)$ corresponding to the pair $(2,\bar{E})$ is an integer for any $\bar{E}\in\mathcal{F}(\Delta)-\{H\}$.
By Theorem~\ref{heightinequality}.20-26 we know that
$\hat{\bar{m}}=\hat{m}=\hat{M}$, $i_\Lambda=2$, $i_\Gamma=\hat{m}+1$,
$\Theta=\Lambda=\hat{\bar{\Lambda}}(2)=\hat{\bar{\Delta}}(\hat{m}+1)$, $\Gamma=H$ and the equality
$\max\{\langle b_{\Gamma/N^*}, a\rangle|a\in A\cap(S+(\Delta^\vee|V^*))\}
-\min\{\langle b_{\Gamma/N^*}, a\rangle|a\in A\cap(S+(\Delta^\vee|V^*))\}
=\Ht(H,S+(\Delta^\vee|V^*))$
holds.

Claim 12 follows from similar arguments as in claim 8 and 9 and Theorem~\ref{heightinequality}.27-32.
\end{proof}

\section{Schemes associated with simplicial cone decompositions}
\label{toric theory}
We develop the theory of torus embeddings (Kempf et al.~\cite{KKMS}, Fulton~\cite{F93}).
We define schemes associated with simplicial cone decompositions and examine their properties.

Let us begin with the ring theory.

Let $R$ be any ring. 
The set of all prime ideals of $R$ is denoted by $\Spec(R)$.
If $R$ is a field, then $\Spec(R)=\{\{0\}\}$.
We define a topology on $\Spec(R)$ called the \emph{Zariski topology}.
We define that a subset $\mathcal{X}$ of $\Spec(R)$ is closed, if there exists a subset $X$ of $R$ satisfying $\mathcal{X}=\{\p\in\Spec(R)|X\subset\p\}$.
A subset $U$ of $\Spec(R)$ is open, if the complement $\Spec(R)-U$ of $U$ in $\Spec(R)$ is closed.
We say that a point $\p\in\Spec(R)$ is a \emph{closed point}, if the set $\{\p\}$ is a closed subset of $\Spec(R)$.

For any $\phi\in R$ we denote $\Spec(R)_\phi=\{\p\in\Spec(R)|\phi\not\in\p\}$ and we call $\Spec(R)_\phi$ the \emph{principal open subset} or the \emph{principal open set} of $\Spec(R)$ associated with $\phi$.

Let $Q$ be another ring and let $\lambda:R\rightarrow Q$ be any ring homomorphism. We can check easily that $\lambda^{-1}(\p)\in\Spec(R)$ for any $\p\in\Spec(Q)$.
Putting $\lambda^*(\p)= \lambda^{-1}(\p)\in\Spec(R)$ for any $\p\in\Spec(Q)$, we define a mapping $\lambda^*:\Spec(Q)\rightarrow\Spec(R)$.

\begin{lemma}
\label{spec1}
\begin{enumerate}
\item
Let $R$ be any ring.

The empty set is an open subset of $\Spec(R)$.
$\Spec(R)$ is an open subset of $\Spec(R)$.
For any open subset $U$ of $\Spec(R)$ and any open subset $V$ of $\Spec(R)$, the intersection $U\cap V$ is an open subset of $\Spec(R)$.
For any non-empty set $\mathcal{U}$ whose elements are open subsets of $\Spec(R)$, the union $\cup_{U\in\mathcal{U}}U$ is an open subset of $\Spec(R)$.

The Zariski topology on  $\Spec(R)$ is a topology on  $\Spec(R)$.
\item
Let $R$ be any ring and let $\p\in\Spec(R)$ be any prime ideal of $R$.

The closure of the subset $\{\p\}$ of $\Spec(R)$ is equal to $\{\q\in\Spec(R)|\p\subset\q\}$.

The point  $\p$ of $\Spec(R)$ is closed, if and only if, $\p$ is a maximal ideal of $R$.

\item
For any ring $R$, any ring $Q$ such that there exists a ring homomorphism from $R$ to $Q$ and any ring homomorphism $\lambda:R\rightarrow Q$, the mapping $\lambda^*:\Spec(Q)\rightarrow\Spec(R)$ is continuous and preserves the inclusion relation.
\item
For any ring $R$, $\Id_R^*$ is equal to the identity mapping of $\Spec(R)$.
\item
For any ring $R$, any ring $Q$, any ring $P$ such that there exists a ring homomorphism from $R$ to $Q$  and there exists a ring homomorphism from $Q$ to $P$ and any ring homomorphism $\lambda:R\rightarrow Q$ and any ring homomorphism $\mu:Q\rightarrow P$,
$(\mu\lambda)^*=\lambda^*\mu^*$.
\end{enumerate}
\end{lemma}

Let $K$ be any field and let $R$ be any subring of $K$.

A subset $S$ of $R$ is called a \emph{multiplicatively closed subset}, if $1\in S$, $0\not\in S$ and $\phi\psi\in S$ for any $\phi\in S$ and any $\psi\in S$.

Let $S$ be any multiplicatively closed subset of $R$. We denote
$$R_S=\{\frac{\psi}{\phi}|\psi\in R, \phi\in S\}\subset K,$$
and we call $R_S$ the \emph{localization} of $R$ with respect to $S$.
We can check easily that $R_S$ is a subring of $K$ containing $R$.

Let $\p\in\Spec(R)$ be any prime ideal of $R$.
We can check easily that the complement $R-\p$ of $\p$ in $R$ is a multiplicatively closed subset of $R$. We denote
$$R_\p= R_{R-\p}=\{\frac{\psi}{\phi}|\psi\in R, \phi\in R-\p\}\subset K,$$
and we call $R_\p$ the \emph{local ring} of $R$ at $\p$.
$R_\p$ is a subring of $K$ containing $R$.

Note that $\{0\}$ is the unique ring contained in $K$ and satisfying $1=0$.
The set $\{0\}$ is not a subring of $K$.

For any open subset $U$ of $\Spec(R)$, we denote
\begin{equation*}
\mathcal{O}_R(U)=
\begin{cases}
\bigcap_{\p\in U}R_\p&\text{ if $U\neq\emptyset$},\\
\{0\}&\text{ if $U=\emptyset$}.
\end{cases}\end{equation*}

For any open subset $U$ of $\Spec(R)$, $\mathcal{O}_R(U)$ is a ring containd in $K$, and if $U\neq\emptyset$, then $\mathcal{O}_R(U)$ is a subring of $K$, it contains $R$ and it is contained in $R_\p$ for any $\p\in U$.

Consider any open subset $U$ of $\Spec(R)$ and any open subset $V$ of $\Spec(R)$ satisfying $U\subset V$. If $U\neq\emptyset$, then $V\neq\emptyset$ and $\mathcal{O}_R(U)\supset \mathcal{O}_R(V)$.
If $U=\emptyset$, then $\mathcal{O}_R(U)=\{0\}$ and there exists uniquely a surjective homomorphism $\mathcal{O}_R(V)\rightarrow \mathcal{O}_R(U)$.
We define a mapping $\Res^V_U: \mathcal{O}_R(V)\rightarrow \mathcal{O}_R(U)$ called \emph{restriction homomorphism} by putting
\begin{equation*}
\Res^V_U=
\begin{cases}
\text{ the inclusion homomorphism}&\text{ if $U\neq\emptyset$},\\
\text{ the unique surjective homomorphism}&\text{ if $U=\emptyset$}.
\end{cases}\end{equation*}
The restriction homomorphism $\Res^V_U$ is a ring homomorphism.

We denote the pair of sets $(\{\mathcal{O}_R(U)|U$ is an open subset of $\Spec(R)\},\{ \Res^V_U|U$ and $V$ are open subsets of $\Spec(R)$ with $U\subset V\})$ by a simple symbol $\mathcal{O}_R$.

Consider any $\p\in\Spec(R)$ and any open subset $U$ of $\Spec(R)$ with $\p\in U$.
We define a mapping $\Res^U_\p: \mathcal{O}_R(U)\rightarrow R_\p$ called \emph{restriction homomorphism} by putting $\Res^U_\p=$ the inclusion mapping.
The restriction homomorphism $\Res^U_\p$ is a ring homomorphism.

\begin{lemma}
\label{localization}
Let $K$ be any field; let $R$ be any subring of $K$ and let $S$ be any multiplicatively closed subset of $R$.
\begin{enumerate}
\item
$\{0\}\in\Spec(R)$.
$U\neq\emptyset$, if and only if, $\{0\}\in U$ for any open subset $U$  of  $\Spec(R)$.
For any non-empty open subset $U$ of $\Spec(R)$ and any non-empty open subset $V$ of $\Spec(R)$, the intersection $U\cap V$ is non-empty.
\item
$R_S$ is a subring of $K$ containing $R$. $S\subset R_S^\times$.

By $\nu:R\rightarrow R_S$ we denote the inclusion ring homomorphism.

Consider any ring $T$ such that there exists a ring homomorphism $\mu:R\rightarrow T$ from $R$ to $T$ satisfying  $\mu(S)\subset T^\times$. For any ring homomorphism $\mu:R\rightarrow T$ satisfying  $\mu(S)\subset T^\times$,  there exists uniquely a ring homomorphism $\bar{\mu}:R_S\rightarrow T$ satisfying $\bar{\mu}\nu=\mu$.
\item
For any $\bar{\p}\in\Spec(R_S)$, $\bar{\p}\cap R\in\Spec(R)$ and it satisfies $(\bar{\p}\cap R)\cap S=\emptyset$ and $(\bar{\p}\cap R) R_S=\bar{\p}$.
\item
For any $\p\in\Spec(R)$ satisfying $\p\cap S=\emptyset$, $\p R_S\in\Spec(R_S)$ and $\p R_S\cap R=\p$.
\item
The mapping from $\Spec(R_S)$ to $\{\p\in\Spec(R)|\p\cap S=\emptyset\}$ sending $\bar{\p}\in\Spec(R_S)$ to  $\bar{\p}\cap R$ and the mapping from $\{\p\in\Spec(R)|\p\cap S=\emptyset\}$ to $\Spec(R_S)$ sending any $\p\in\Spec(R)$ satisfying $\p\cap S=\emptyset$ to $\p R_S\in\Spec(R_S)$ are bijective mappings preserving the inclusion relation, and they are the inverse mappings of each other.

If $\bar{\p}\in\Spec(R_S)$ and $\p\in\Spec(R)$ satisfying $\p\cap S=\emptyset$ correspond to each other by these bijective mappings, then $(R_S)_{\bar{\p}}=R_\p$.
\item
Furthermore, if we define the topology on $\Spec(R_S)$ and we define the topology on $\{\p\in\Spec(R)|\p\cap S=\emptyset\}$ induced by the topology on $\Spec(R)$, then the mapping from $\Spec(R_S)$ to $\{\p\in\Spec(R)|\p\cap S=\emptyset\}$ sending $\bar{\p}\in\Spec(R_S)$ to  $\bar{\p}\cap R$ is a continuous bijective mapping whose inverse mapping is also continuous.
\item
If $R$ is noetherian, then $R_S$ is also noetherian.
\item
Consider any prime ideal $\p\in\Spec(R)$ of $R$.

$\{\bar{\q}\cap R|\bar{\q}\in\Spec(R_\p)\}=\{\q\in\Spec(R)|\q\subset\p\}$.

The ring $R_\p$ is a local ring whose maximal ideal is equal to $\p R_\p$.
If $R$ is noetherian, then $R_\p$ is also noetherian.
\item
Consider any $r\in\Z_+$ and any mapping $\phi:\{1,2,\ldots,r\}\rightarrow R-\{0\}$.
We denote $S(\phi)=\{\phi(1)^{m(1)}\phi(2)^{m(2)}\cdots\phi(r)^{m(r)}|m\in\Map(\{1,2,\ldots,r\},\Z_0)\}\subset R$.

The subset $S(\phi)$ is a multiplicatively closed subset of $R$.
$\phi(1)\phi(2)\cdots\phi(r)\in R-\{0\}$.

\begin{equation*}\begin{split}
&R_{S(\phi)}\\
=\:&\{\frac{\psi}{\phi(1)^{m(1)} \phi(2)^{m(2)}\cdots\phi(r)^{m(r)}}|\psi\in R, m\in\Map(\{1,2,\ldots,r\},\Z_0)\}\\
=\:&R[\{\frac{1}{\phi(i)}|i\in\{1,2,\ldots,r\}\}]
=R[\frac{1}{\phi(1)\phi(2)\cdots \phi(r)}]
\subset K.
\end{split}\end{equation*}
If $R$ is noetherian, then $ R_{S(\phi)}$ is also noetherian.
\item
Consider any $\phi\in R$.
The principal open subset $\Spec(R)_\phi$ of $\Spec(R)$ associated  with $\phi\in R$ is an open subset of $\Spec(R)$.

$\Spec(R)_\phi=\emptyset$, if and only if, $\phi=0$.

$\{\bar{\p}\cap R|\bar{\p}\in\Spec(R[1/\phi])\}= \Spec(R)_\phi$ if $\phi\neq 0$.
\item
Consider any $\phi\in R-\{0\}$ and any $\psi\in R-\{0\}$.

$\Spec(R)_\phi=\Spec(R)_\psi$, if and only if, $\{\A\in R|\A^m=\B\phi\text{ for some }m\in\Z_0\text{ and for some }\B\in R\}=\{\A\in R|\A^m=\B\psi\text{ for some }m\in\Z_0\text{ and for some }$\break$\B\in R\}$.

If $\Spec(R)_\phi=\Spec(R)_\psi$, then $R[\frac{1}{\phi}]= R[\frac{1}{\psi}]$.

$\phi\psi\in R-\{0\}$ and $\Spec(R)_{\phi\psi}=\Spec(R)_\phi\cap\Spec(R)_\psi$.
\item
$$\bigcap_{\phi\in X}R[\frac{1}{\phi}]=R$$
for any subset $X$ of $R-\{0\}$ satisfying $XR=R$.
\item
$$\bigcap_{\p\in\Spec(R)}R_\p=R.$$
\item
$$\bigcap_{\p\in\Spec(R)_\phi}R_\p=R[\frac{1}{\phi}]$$
for any non-zero element $\phi\in R-\{0\}$.
\item
$\mathcal{O}_R$ is a sheaf of rings on $\Spec(R)$, in other words, the following seven conditions are satisfied:
\begin{enumerate}
\item
For any open subset $U$ of $\Spec(R)$, $\mathcal{O}_R(U)$ is a ring.
\item 
For any open subsets $U$ and $V$ of $\Spec(R)$ with $U\subset V$, $\Res^V_U$ is a ring homomorphism.
\item
For any open subset $U$ of $\Spec(R)$, $\Res^U_U$ is the identity mapping of $\mathcal{O}_R(U)$.
\item
For any open subsets $U$, $V$ and $W$ of $\Spec(R)$ with $U\subset V\subset W$, $\Res^W_U=\Res^V_U \Res^W_V$.
\item
Consider any non-empty set $\mathcal{U}$ whose elements are open subsets of $\Spec(R)$. We denote $\hat{U}= \cup_{U\in\mathcal{U}}U$. Consider any 
$\phi\in\mathcal{O}_R(\hat{U})$.

If $\Res^{\hat{U}}_U(\phi)=0$ for any $U\in\mathcal{U}$, then $\phi=0$.
\item
Consider any non-empty set $\mathcal{U}$ whose elements are open subsets of $\Spec(R)$. We denote $\hat{U}= \cup_{U\in\mathcal{U}}U$. Consider any 
$\phi(U)\in\mathcal{O}_R(U)$ for any $U\in\mathcal{U}$.

If $\Res^{U}_{U\cap V}(\phi(U))= \Res^{V}_{U\cap V}(\phi(V))$ for any $U\in\mathcal{U}$ and any $V\in\mathcal{U}$, then there exists $\phi\in\mathcal{O}_R(\hat{U})$ satisfying $\Res^{\hat{U}}_U(\phi)=\phi(U)$ for any $U\in\mathcal{U}$.
\item
$\mathcal{O}_R(\emptyset)=\{0\}$.
\end{enumerate}
\item
$\mathcal{O}_R(\Spec(R))=R$.
$\mathcal{O}_R(\Spec(R)_\phi)=R[1/\phi]$ for any non-zero element $\phi\in R-\{0\}$.
\item
Consider any non-zero element $\phi\in R-\{0\}$ and any open subset $U$ of $\Spec(R)$ contained in $\Spec(R)_\phi$.
Let $\hat{U}=\{\bar{\p}\in\Spec(R[1/\phi])|\bar{\p}\cap R\in U\}$.
$\hat{U}$ is an open subset of $\Spec(R[1/\phi])$ and $\mathcal{O}_{R[1/\phi]}(\hat{U})=\mathcal{O}_R(U)$. 
\item
Note that the sheaf $\mathcal{O}_R$ on $\Spec(R)$ is defined using the field $K$ containing $R$.
The sheaf  $\mathcal{O}_R$ is isomorphic to Grothendieck's structure sheaf on $\Spec(R)$, which depends only on $R$ and independent of the choice of the field $K$ containing $R$, and the topological space with a sheaf $(\Spec(R),\mathcal{O}_R)$ is identified with Grothendieck's affine scheme $\Spec(R)$.
\item
For any $\p\in\Spec(R)$ and any $\phi\in R_\p$, there exist an open subset $U$ of $\Spec(R)$ with $\p\in U$ and an element $\psi\in \mathcal{O}_R(U)$ satisfying $\Res^U_\p(\psi)=\phi$.
\item
Consider any $\p\in\Spec(R)$.

Note that for any open subset $U$ of $\Spec(R)$ with $\p\in U$ and any open subset $V$ of $\Spec(R)$ with $\p\in V$, there exists an open subset $W$ of $\Spec(R)$ with $\p\in W\subset U\cap V$.

The pair $(R_\p, \{\Res^U_\p|U$ is an open subset of $\Spec(R)$ with $\p\in U\})$ is the inductive limit of the inductive system $(\{\mathcal{O}_R(U)|U$ is an open subset of $\Spec(R)$ with $\p\in U\},\{ \Res^V_U|U$ and $V$ are open subsets of $\Spec(R)$ with $\p\in U\subset V\})$, in other words, 
$\Res^V_\p=\Res^U_\p\Res^V_U$ for any open subsets $U$ and $V$ of $\Spec(R)$ with $\p\in U\subset V$ and the following condition is satisfied:

Assume that a ring $T$ is given and a ring homomorphism $\lambda(U): \mathcal{O}_R(U)\rightarrow T$ is given for any open subset $U$ of $\Spec(R)$ with $\p\in U$.
If $\lambda (V)=\lambda(U)\Res^V_U$ for any open subsets $U$ and $V$ of $\Spec(R)$ with $\p\in U\subset V$, then there exists uniquely a ring homomorphism $\lambda:R_\p\rightarrow T$ satisfying $\lambda(U)=\lambda\Res^U_\p$ for any open subset $U$ of $\Spec(R)$ with $\p\in U$.
\end{enumerate}

Let $J$ be any field; let $Q$ be any subring of $J$ such that there exists a ring homomorphism from $R$ to $Q$ and let $\lambda:R\rightarrow Q$ be any ring homomorphism. The continuous mapping $\lambda^*:\Spec(Q)\rightarrow\Spec(R)$ is defined associated with $\lambda$.
\begin{enumerate}
\setcounter{enumi}{20}
\item
For any $\p\in\Spec(Q)$, there exists uniquely a ring homomorphism $\lambda_*(\p):R_{\lambda^*(\p)}\rightarrow Q_\p$ satisfying $\Res^{\Spec(Q)}_\p\lambda=\lambda_*(\p)\Res^{\Spec(R)}_{\lambda^*(\p)}$.

For any open subset $U$ of $\Spec(R)$, there exists uniquely a ring homomorphism
$\lambda_*(U):\mathcal{O}_R(U)\rightarrow \mathcal{O}_Q(\lambda^{*-1}(U))$ satisfying
$\Res^{\Spec(Q)}_{\lambda^{*-1}(U)}\lambda=\lambda_*(U)$\hfill\break$\Res^{\Spec(R)}_U$.
\end{enumerate}

We take the unique ring homomorphism $\lambda_*(U):\mathcal{O}_R(U)\rightarrow \mathcal{O}_Q(\lambda^{*-1}(U))$ satisfying
$\Res^{\Spec(Q)}_{\lambda^{*-1}(U)}\lambda=\lambda_*(U)\Res^{\Spec(R)}_U$ for any open subset $U$ of $\Spec(R)$.
We denote the set $\{\lambda_*(U)|U$ is an open subset of $\Spec(R)\}$ by a single symbol $\lambda_*$.
We denote the pair $(\lambda^*,\lambda_*)$ by a single symbol $\lambda^*$.
\begin{enumerate}
\setcounter{enumi}{21}
\item
The pair $(\lambda^*,\lambda_*)$ is a morphism from the topological space with a sheaf $(\Spec($\hfill\break$Q),\mathcal{O}_Q)$ to the topological space with a sheaf $(\Spec(R),\mathcal{O}_R)$, in other words, \hfill\break
$\Res^{\lambda^{*-1}(V)}_{\lambda^{*-1}(U)}\lambda=\lambda_*(U)\Res^V_U$
for any open subsets $U$ and $V$ of $\Spec(R)$ with $U\subset V$.
\item
The pair $\lambda^*=(\lambda^*,\lambda_*)$ is identified with Grothendieck's morphism
$\lambda^*:$\hfill\break$\Spec(Q)\rightarrow\Spec(R)$ of affine schemes associated with the ring homomorphism $\lambda$.
\item
Assume that $Q$ is a subring of $K$ containing $R$ and $\lambda:R\rightarrow Q$ is the inclusion ring homomorphism.

For any $\p\in\Spec(Q)$, $R_{\lambda^*(\p)}\subset Q_\p$ and the homomorphism $\lambda_*(\p): R_{\lambda^*(\p)}\rightarrow Q_\p $ is equal to the inclusion homomorphism.

For any open subset $U$ of $\Spec(R)$, 
$\mathcal{O}_R(U)\subset\mathcal{O}_{Q}(\lambda^{*-1}(U))$ and the homomorphism $\lambda_*(U): \mathcal{O}_R(U)\rightarrow\mathcal{O}_Q(\nu^{*-1}(U))$ is equal to the inclusion homomorphism.
\end{enumerate}
\end{lemma}

Below, we consider any integral domain $S$, any finite dimensional vector space $V$ over $\R$ and any lattice $N$ of $V$. 

Note that $\Map'(N,S)$ is an $S$-module.

For any $e\in N$ we denote
$$\Delta(e)=\{(f,g)|f\in N, g\in N, f+g=e\}\subset N\times N.$$

If $e\in N$, $f\in N$, $g\in N$ and $(f,g)\in \Delta(e)$, then $(g,f)\in\Delta(e)$.
$\Delta(e)=\{(f,e-f)|f\in N\}=\{(e-g, g)|g\in N\}$ for any $e\in N$. If $\dim V\geq 1$, then $\Delta(e)$ is an infinite set for any $e\in N$.

Consider any $\phi\in \Map'(N,S)$ and any $\psi\in \Map'(N,S)$.

$\Supp(\phi)$, $\Supp(\psi)$ and $\Supp(\phi)+\Supp(\psi)$ are finite subsets of $N$.
$\Supp(\phi)\times\Supp(\psi)$ is a finite subset of $N\times N$.

Consider any $e\in N$.
$\{(f,g) |f\in N, g\in N, (f,g)\in\Delta(e), \phi(f)\psi(g)\neq 0\}=
\{(f,g)|f\in N, g\in N, (f,g)\in\Delta(e), \phi(f)\neq 0, \psi(g)\neq 0\}=
\Delta(e)\cap(\Supp(\phi)\times\Supp(\psi)$.
We know that the set $\{(f,g) |f\in N, g\in N, (f,g)\in\Delta(e), \phi(f)\psi(g)\neq 0\}$ is a finite set and an element $\sum_{ f\in N, g\in N, (f,g)\in\Delta(e)} \phi(f)\psi(g)\in S$ is defined. 

Putting
$$(\phi\psi)(e)= \sum_{ f\in N, g\in N, (f,g)\in\Delta(e)} \phi(f)\psi(g)\in S$$
for any $e\in N$, we define an element $\phi\psi\in\Map(N,S)$.

If $e\in N$ and $\Delta(e)\cap(\Supp(\phi)\times\Supp(\psi))=\emptyset$, then
$\sum_{ f\in N, g\in N, (f,g)\in\Delta(e)} \phi(f)\psi(g)=0$.
For any $e\in N$, $\Delta(e)\cap(\Supp(\phi)\times\Supp(\psi))\neq\emptyset
\Leftrightarrow e\in\Supp(\phi)+\Supp(\psi)$.
Therefore, $\Supp(\phi\psi)\subset\Supp(\phi)+\Supp(\psi)$, $\Supp(\phi\psi)$ is a finite set, and we know that $\phi\psi\in\Map'(N,S)$.

We call the element $\phi\psi\in\Map'(N,S)$ the \emph{product} of $\phi$ and $\psi$.

Putting
\begin{equation*}
x(e)(f)=
\begin{cases}
1& \text{ if $e=f$},\\
0& \text{ if $e\neq f$},
\end{cases}\end{equation*}
for any $e\in N$ and any $f\in N$, we define a mapping
$$x:N\rightarrow \Map'(N,S).$$

Putting
$$\nu(a)=a x(0)\in\Map'(N,S)$$
for any $a\in S$, we define a mapping 
$$\nu:S\rightarrow \Map'(N,S).$$

For any subset $\Theta$ of $V$, we denote
$$\Map'(N,S)\backslash \Theta=\{\phi\in\Map'(N,S)|\Supp(\phi)\subset\Theta\}.$$

\begin{lemma}
\label{interesting ring}
Consider any integral domain $S$, any finite dimensional vector space $V$ over $\R$ and any lattice $N$ of $V$. We consider the $S$-module $\Map'(N,S)$ and the product on $\Map'(N,S)$ defined above.
\begin{enumerate}
\item
$\Map'(N,S)$ is an integral domain. The identity element of $\Map'(N,S)$ is equal to $x(0)$. The mapping $\nu:S\rightarrow \Map'(N,S)$ is an injective ring homomophism.
\end{enumerate}

Below, using $\nu:S\rightarrow \Map'(N,S)$, we regard $S$ as a subring of $\Map'(N,S)$. Let $K$ be any field such that there exists an injective ring homomorphism from $\Map'(N,S)$ to $K$ and let $\iota:\Map'(N,S)\rightarrow K$ be any injective ring homomorphism. We fix such a pair $(K,\iota)$ and using $\iota$, we regard $\Map'(N,S)$ as a subring of the field $K$.
\begin{enumerate}
\setcounter{enumi}{1}
\item
For any $\phi\in\Map'(N,S)$ and any $\psi\in\Map'(N,S)$,
$$\Supp(\phi+\psi)\subset\Supp(\phi)\cup\Supp(\psi)\text{, and}$$
$$\Supp(\phi\psi)\subset\Supp(\phi)+\Supp(\psi).$$
\item
$\Supp(1)=\{0\}$. For any $\phi\in\Map'(N,S)$, $\Supp(\phi)=\emptyset\Leftrightarrow \phi=0$ and $\Supp(\phi)\subset\{0\}\Leftrightarrow \phi\in S$.
\item
$x(0)=1$. For any $e\in N$ and any $f\in N$, $x(e+f)=x(e)x(f)$.

For any $e\in N$, $x(e)x(-e)=1$, $x(e)\in \Map'(N,S)^\times$ and $\Supp(x(e))=\{e\}$. 

The mapping $x:N\rightarrow \Map'(N,S)$ is injective.

For any $\phi\in\Map'(N,S)$, $\phi=\sum_{e\in N}\phi(e)x(e)$.

For any $\phi\in\Map'(N,S)$ and any $e\in N$, $\Supp(\phi)\subset\{e\}\Leftrightarrow$ there exists $a\in S$ with $\phi=a x(e)
\Leftrightarrow$ there exists uniquely $a\in S$ with $\phi=a x(e)$.
\item
For any ring $T$ such that there exists a ring homomorphism from $S$ to $T$, any ring homomorphism $\mu:S\rightarrow T$ and any homomorphism $\kappa:N\rightarrow T^\times$ of abelian groups, there exists uniquely a ring homomorphism $\lambda: \Map'(N,S)\rightarrow T$ satisfying $\lambda(a)=\mu(a)$ for any $a\in S$ and $\lambda x(e)=\kappa(e)$ for any $e\in N$.
\item 
If a subset $\Theta$ of $V$ satisfies $0\in\Theta$ and $\Theta+\Theta\subset\Theta$, then $\Map'(N,S)\backslash\Theta$ is a subring of $\Map'(N,S)$ containing $S$.
\item
Consider any simplicial cone $\Theta$ over $N$ in $V$.
We denote $P_\Theta=\{x(b_{E/N})|E\in\mathcal{F}(\Theta)_1\}\subset\Map'(N,S)$.
Putting
$$\Phi_\Theta(\Gamma)=\sum_{E\in\mathcal{F}(\Theta)_1}\Gamma(x(b_{E/N}))b_{E/N}\in\Vect(\Theta),$$
for any $\Gamma\in\Map(P_\Theta,\R)$, we define a mapping 
$$\Phi_\Theta: \Map(P_\Theta,\R)\rightarrow\Vect(\Theta).$$
\begin{enumerate}
\item
$\Map'(N,S)\backslash\Theta$ is a subring of $\Map'(N,S)$ containing $S$.
$P_\Theta\subset\Map'(N,S)$\hfill\break$\backslash\Theta$.
$P_\Theta$ is a variable system of $\Map'(N,S)\backslash\Theta$ over $S$, and  $\Map'(N,S)\backslash\Theta$ is a polynomial ring over $S$ with $\dim \Theta$ variables.
\item
$\Phi_\Theta$ is an isomorphism of vector spaces over $\R$.
$\Phi_\Theta(\Map(P_\Theta,\R_0))=\Theta$.
$\Phi_\Theta(\Map(P_\Theta,\Z))=N\cap\Vect(\Theta)$.
$\Phi_\Theta(\Supp(P_\Theta,\phi))=\Supp(\phi)$ 
for any $\phi\in \Map'(N,S)\backslash\Theta$.

If $S$ is a field, then the Newton polyhedron $\Gamma_+(P_\Theta,\phi)$ over $P_\Theta$ is defined and 
$\Phi_\Theta(\Gamma_+(P_\Theta,\phi))=\Conv(\Supp(\phi))+\Theta$
for any $\phi\in \Map'(N,S)\backslash\Theta$.
\item
$\Map'(N,S)\backslash\Vect(\Theta)=( \Map'(N,S)\backslash\Theta)[\{1/x(b_{E/N})|E\in \mathcal{F}(\Theta)_1\}]$.
\end{enumerate}
\item
Consider any simplicial cone $\Lambda$ over $N^*$ in $V^*$ and 
any simplicial cone $\Delta$ over $N^*$ in $V^*$ satisfying $\dim \Delta=\dim V$ and $\Lambda\in\mathcal{F}(\Delta)$.
The set $\{b_{E/N^*}|E\in\mathcal{F}(\Delta)_1\}$ is a basis of the vector space $V^*$ over $R$. We denote the dual basis of $\{b_{E/N^*}|E\in\mathcal{F}(\Delta)_1\}$ by $\{{b_{E/N^*}}^\vee_\Delta|E\in\mathcal{F}(\Delta)_1\}$.
\begin{equation*}
\langle b_{E/N^*}, {b_{D/N^*}}^\vee_\Delta\rangle=
\begin{cases}
1&\text{ if $E=D$},\\
0&\text{ if $E\neq D$},
\end{cases}\end{equation*}
for any $E\in\mathcal{F}(\Delta)_1$ and any $D\in \mathcal{F}(\Delta)_1$.
The set $\{b_{E/N^*}|E\in\mathcal{F}(\Delta)_1\}$ is a $\Z$-basis of the lattice $N^*$ and 
$\{{b_{E/N^*}}^\vee_\Delta|E\in\mathcal{F}(\Delta)_1\}$ is a $\Z$-basis of the lattice $N$.
We denote
$P_\Delta=\{x({b_{E/N^*}}^\vee_\Delta)|E\in\mathcal{F}(\Delta)_1\}\subset\Map'(N,S)$.
Putting
$$\Phi_\Delta(\Gamma)=\sum_{E\in\mathcal{F}(\Delta)_1}\Gamma(x({b_{E/N^*}}^\vee_\Delta))b_{E/N}^\vee\in V,$$
for any $\Gamma\in\Map(P_\Delta,\R)$, we define a mapping 
$$\Phi_\Delta: \Map(P_\Delta,\R)\rightarrow V.$$

\begin{enumerate}
\item
$\Delta^\vee|V^*$ is a simplicial cone over $N$ in $V$.
$\dim \Delta^\vee|V^*=\dim V$.

$\Map'(N,S)\backslash(\Delta^\vee|V^*)$ is a subring of $\Map'(N,S)$ containing $S$.

$P_\Delta\subset\Map'(N,S)\backslash(\Delta^\vee|V^*)$.
$P_\Delta$ is a variable system of $\Map'(N,S)\backslash(\Delta^\vee$\hfill\break$|V^*)$ over $S$, and  $\Map'(N,S)\backslash(\Delta^\vee|V^*)$ is a polynomial ring over $S$ with $\dim V$ variables.
\item
$\Phi_\Delta$ is an isomorphism of vector spaces over $\R$.
$\Phi_\Delta(\Map(P_\Delta,\R_0))= \Delta^\vee|V^*$.
$\Phi_\Theta(\Map(P_\Delta,\Z))=N$.
$\Phi_\Theta(\Supp(P_\Delta,\phi))=\Supp(\phi)$ 
for any $\phi\in \Map'(N,S)\backslash(\Delta^\vee|V^*)$.

If $S$ is a field, then the Newton polyhedron $\Gamma_+(P_\Delta,\phi)$ over $P_\Delta$ is defined and 
$\Phi_\Delta(\Gamma_+(P_\Delta,\phi))=\Conv(\Supp(\phi))+( \Delta^\vee|V^*)$
for any $\phi\in \Map'(N,S)$\hfill\break$\backslash(\Delta^\vee|V^*)$.
\item
$\Map'(N,S)=( \Map'(N,S)\backslash(\Delta^\vee|V^*))[\{1/x({b_{E/N^*}}^\vee_\Delta)|E\in \mathcal{F}(\Delta)_1\}]$.
\item
$\Lambda^\vee|V^*\supset\Delta^\vee|V$.
$\Map'(N,S)\backslash(\Lambda^\vee|V^*)$ is a subring of $\Map'(N,S)$ containing $S$.
$\Map'(N,S) \backslash(\Lambda^\vee|V^*)=(\Map'(N,S)\backslash(\Delta^\vee|V^*))[\{1/ x({b_{E/N^*}}^\vee_\Delta)|E\in \mathcal{F}(\Delta)_1-\mathcal{F}(\Lambda)_1\}]
\supset \Map'(N,S)\backslash(\Delta^\vee|V^*)$.
\item
For any $\bar{\Lambda}\in\mathcal{F}(\Lambda)$,
$\Map'(N,S) \backslash(\bar{\Lambda}^\vee|V^*)=
\Map'(N,S) \backslash(\Lambda^\vee|V^*)$\hfill\break$[\{1/ x({b_{E/N^*}}^\vee_\Delta)|E\in \mathcal{F}(\Lambda)_1-\mathcal{F}(\bar{\Lambda})_1\}]
\supset \Map'(N,S) \backslash(\Lambda^\vee|V^*)$.
\item
Furthermore, consider any simplicial cone $\Delta'$ over $N^*$ in $V^*$ satisfying $\dim \Delta'=\dim V$ and $\Lambda\in\mathcal{F}(\Delta')$.
We denote the dual basis of $\{b_{E/N^*}|E\in\mathcal{F}(\Delta')_1\}$ by $\{{b_{E/N^*}}^\vee_{\Delta'}|E\in\mathcal{F}(\Delta')_1\}$.

For any $E\in \mathcal{F}(\Lambda)_1$, there exists uniquely a mapping $r(E): \mathcal{F}(\Delta)_1-\mathcal{F}(\Lambda)_1\rightarrow\Z$ satisfying
$$x({b_{E/N^*}}^\vee_{\Delta'})=x({b_{E/N^*}}^\vee_\Delta)\prod_{D\in\mathcal{F}(\Delta)_1-\mathcal{F}(\Lambda)_1} x({b_{D/N^*}}^\vee_\Delta)^{r(E)(D)}.$$

For any $E\in \mathcal{F}(\Lambda)_1$,
$x({b_{E/N^*}}^\vee_{\Delta'})(\Map'(N,S) \backslash(\Lambda^\vee|V^*))
= x({b_{E/N^*}}^\vee_\Delta)$\hfill\break$(\Map'(N,S) \backslash(\Lambda^\vee|V^*))$.

For any $E\in \mathcal{F}(\Lambda)_1$,
the ideal $x({b_{E/N^*}}^\vee_\Delta)(\Map'(N,S) \backslash(\Lambda^\vee|V^*))$ of $\Map'($\hfill\break$N, S) \backslash(\Lambda^\vee|V^*)$ depends only on $\Lambda$ and $E$, and it is independent of the choice of $\Delta$.

The ideal $\{x({b_{E/N^*}}^\vee_\Delta)|E\in\mathcal{F}(\Lambda)_1\}(\Map'(N,S) \backslash(\Lambda^\vee|V^*))$ of $\Map'(N,S)$\hfill\break$ \backslash(\Lambda^\vee|V^*)$ depends only on $\Lambda$, and it is independent of the choice of $\Delta$.
\end{enumerate}
\item
Consider any simplicial cones $\Delta$ and $\Delta'$ over $N^*$ in $V^*$ such that 
$\Delta\cap\Delta'$ is a face of $\Delta$ and $\Delta\cap\Delta'$ is a face of $\Delta'$.
We denote $\Lambda=\Delta\cap\Delta'$.
\begin{enumerate}
\item
$(\Lambda^\vee|V^*)\cap N=((\Delta^\vee|V^*)\cap N)+( \Delta^{\prime\vee}|V^*)\cap N)$.
\item
$\Map'(N,S) \backslash(\Lambda^\vee|V^*)\supset
(\Map'(N,S) \backslash(\Delta^\vee|V^*))\cup(\Map'(N,S) \backslash(\Delta^{\prime\vee}|V^*))$.

The minimum subring of $K$ containing $(\Map'(N,S) \backslash(\Delta^\vee|V^*))\cup(\Map'(N,$\hfill\break$S) \backslash(\Delta^{\prime\vee}|V^*))$ coincides with $\Map'(N,S) \backslash(\Lambda^\vee|V^*)$.
\item
Consider any $\p\in\Spec(\Map'(N,S) \backslash(\Delta^\vee|V^*))$ and any
$\p'\in\Spec(\Map'($\hfill\break$N, S) \backslash(\Delta^{\prime\vee}|V^*))$.

There exists $\q\in\Spec(\Map'(N,S) \backslash(\Lambda^\vee|V^*))$ satisfying
$\p=\q\cap(\Map'(N, $\hfill\break$S) \backslash(\Delta^\vee|V^*))$ and
$\p'=\q\cap(\Map'(N,S) \backslash(\Delta^{\prime\vee}|V^*))$,
if and only if,\hfill\break
$(\Map'(N,S) \backslash(\Delta^\vee|V^*))_\p=
(\Map'(N,S) \backslash(\Delta^{\prime\vee}|V^*))_{\p'}$.

If the equivalent conditions above are satisfied and $\q\in\Spec(\Map'(N,S) $\hfill\break$\backslash(\Lambda^\vee|V^*))$ satisfies
$\p=\q\cap(\Map'(N,S) \backslash(\Delta^\vee|V^*))$ and
$\p'=\q\cap(\Map'(N,$\hfill\break$S) \backslash(\Delta^{\prime\vee}|V^*))$, then
$(\Map'(N,S) \backslash(\Lambda^\vee|V^*))_\q=
(\Map'(N,S) \backslash(\Delta^\vee|V^*))_\p$\hfill\break$=
(\Map'(N,S) \backslash(\Delta^{\prime\vee}|V^*))_{\p'}$.
\item
\begin{equation*}\begin{split}
&\{(\Map'(N,S) \backslash(\Delta^\vee|V^*))_\p|
\p\in\Spec(\Map'(N,S) \backslash(\Delta^\vee|V^*))\}\\
&\qquad\cap
\{(\Map'(N,S) \backslash(\Delta^{\prime\vee}|V^*))_{\p'}|
\p'\in\Spec(\Map'(N,S) \backslash(\Delta^{\prime\vee}|V^*))\}\\
=\:&
\{(\Map'(N,S) \backslash(\Lambda^\vee|V^*))_\q|
\q\in\Spec(\Map'(N,S) \backslash(\Lambda^\vee|V^*))\}.
\end{split}\end{equation*}
\end{enumerate}
\end{enumerate}
\end{lemma}

Below, we consider any integral domain $S$, any finite dimensional vector space $V$ over $\R$, any lattice $N$ of $V$ and any simplicial cone decomposition $\mathcal{D}$ over $N^*$ in $V^*$.

We would like to define a scheme $\Sigma(S,V,N,\mathcal{D})$ associated with the quadruplet $(S,V,N,\mathcal{D})$ and would like to examine its properties.

Note that $\Map'(N,S)$ is an integral domain containing $S$.
Let $K$ be any field such that there exists an injective ring homomorphism from $\Map'(N,S)$ to $K$ and let $\iota:\Map'(N,S)\rightarrow K$ be any injective ring homomorphism. We fix such a pair $(K,\iota)$ and using $\iota$, we regard $\Map'(N,S)$ as a subring of the field $K$.

For simplicity we denote
$$R(\Delta)= \Map'(N,S)\backslash(\Delta^\vee|V^*)$$
for any convex polyhedral cone $\Delta$ in $V^*$.

We define the set $\Sigma(S,V,N,\mathcal{D})$ by putting
$$\Sigma(S,V,N,\mathcal{D})=\{ R(\Delta)_\p|\Delta\in\mathcal{D},\p\in\Spec(R(\Delta))\}.$$

Consider any $\Delta\in\mathcal{D}$. 
By putting
$$\pi(S,V,N,\mathcal{D},\Delta)(\p)= R(\Delta)_\p\in \Sigma(S,V,N,\mathcal{D})$$
for any $\p\in\Spec(R(\Delta))$, we define a mapping
$$\pi(S,V,N,\mathcal{D},\Delta): \Spec(R(\Delta))\rightarrow  \Sigma(S,V,N,\mathcal{D}).$$

We define the topology on $\Sigma(S,V,N,\mathcal{D})$.
Let $U$ be any subset of $\Sigma(S,V,N,\mathcal{D})$.
We define that $U$ is an open subset of $\Sigma(S,V,N,\mathcal{D})$, if 
$\pi(S,V,N,\mathcal{D},\Delta)^{-1}(U)$ is an open subset of $\Spec(R(\Delta))$ for any $\Delta\in\mathcal{D}$.

We define the sheaf $\mathcal{O}(S,V,N,\mathcal{D})$ on $\Sigma(S,V,N,\mathcal{D})$.
Note that any point $\A\in \Sigma(S,V,N,\mathcal{D})$ is a local subring of $K$ containing $S$.
When we regard a point $\A$ as a subring of $K$, we denote the corresponding  subring of $K$ by the symbol $\mathcal{O}(S,V,N,\mathcal{D})_\A$.
For any point $\A\in \Sigma(S,V,N,\mathcal{D})$,  $\mathcal{O}(S,V,N,\mathcal{D})_\A$ is a local subring of $K$ containing $S$.
Let $U$ be any open subset of $\Sigma(S,V,N,\mathcal{D})$.
We define $\mathcal{O}(S,V,N,\mathcal{D})(U)$ by putting
\begin{equation*}
\mathcal{O}(S,V,N,\mathcal{D})(U)=
\begin{cases}
\bigcap_{\A\in U}\mathcal{O}(S,V,N,\mathcal{D})_\A \subset K&\text{ if $U\neq\emptyset$},\\
\{0\}\subset K&\text{ if $U=\emptyset$}.
\end{cases}\end{equation*}
$\mathcal{O}(S,V,N,\mathcal{D})(U)$ is a ring contained in $K$.
If $U\neq\emptyset$, then $\mathcal{O}(S,V,N,\mathcal{D})(U)$ is a subring of $K$ containing $S$.

Let $U$ and $V$ be any open subsets of $\Sigma(S,V,N,\mathcal{D})$ with $U\subset V$.
If $U\neq\emptyset$, then $V\neq\emptyset$ and $\mathcal{O}(S,V,N,\mathcal{D})(U)\supset\mathcal{O}(S,V,N,\mathcal{D})(V)$ by definition.
If $U=\emptyset$, then $\mathcal{O}(S,V,N,\mathcal{D})(U)=\{0\}$ and there exists uniquely a surjective homomorphism $\mathcal{O}(S,V,N,\mathcal{D})(V)\rightarrow
\mathcal{O}(S,V,N,\mathcal{D})(U)$.
We define a ring homomorphism
$$\Res(S,V,N,\mathcal{D})^V_U: \mathcal{O}(S,V,N,\mathcal{D})(V)\rightarrow
\mathcal{O}(S,V,N,\mathcal{D})(U),$$
by putting
\begin{equation*}
\Res(S,V,N,\mathcal{D})^V_U=
\begin{cases}
\text{the inclusion homomorphism}&\text{ if $U\neq\emptyset$},\\
\text{the unique surjective homomorphism}&\text{ if $U=\emptyset$}.
\end{cases}\end{equation*}

We denote the pair of sets $(\{\mathcal{O}(S,V,N,\mathcal{D})(U)|U$ is an open subset of  $\Sigma(S,V,N,$\hfill\break$\mathcal{D})\}, \{\Res(S,V,N,\mathcal{D})^V_U|U$ and $V$ are open subsets of $\Sigma(S,V,N,\mathcal{D})$ with $U\subset V\})$ by a single symbol $\mathcal{O}(S,V,N,\mathcal{D})$.

Let $\A\in\Sigma(S,V,N,\mathcal{D})$ be any point and let $U$ be any open subset of $\Sigma(S,V,N,\mathcal{D})$ with $\A\in U$.
$\mathcal{O}(S,V,N,\mathcal{D})(U) \subset\mathcal{O}(S,V,N,\mathcal{D})_\A$ by definition.
We define the ring homomorphism $\Res(S,V,N,\mathcal{D})^U_\A:
\mathcal{O}(S,V,N,\mathcal{D})(U) \rightarrow\mathcal{O}(S,V,N,\mathcal{D})_\A$, by putting $\Res(S,V,N,\mathcal{D})^U_\A=$ the inclusion homomorphism.

For any $\Delta\in\mathcal{D}$, we denote
\begin{equation*}\begin{split}
U(S,V,N,\mathcal{D},\Delta)&= \pi(S,V,N,\mathcal{D},\Delta)( \Spec(R(\Delta)))
\subset \Sigma(S,V,N,\mathcal{D}),\\
V^\circ(S,V,N,\mathcal{D},\Delta)&=U(S,V,N,\mathcal{D},\Delta)-(\bigcup_{\Lambda\in\mathcal{F}(\Delta),\Lambda\neq\Delta} U(S,V,N,\mathcal{D},\Lambda))\\
&\qquad\qquad\quad\subset U(S,V,N,\mathcal{D},\Delta),\\
V(S,V,N,\mathcal{D},\Delta)&=\text{the closure of }V^\circ(S,V,N,\mathcal{D},\Delta)\text{ in }\Sigma(S,V,N,\mathcal{D})\\
&\qquad\qquad\quad\subset \Sigma(S,V,N,\mathcal{D}).
\end{split}\end{equation*}

When we need not refer to the quadruplet $(S,V,N,\mathcal{D})$, we also write simply
$\Sigma$, $\pi(\Delta)$, $\mathcal{O}$, $\Res$, $U(\Delta)$, $V^\circ(\Delta)$ and $V(\Delta)$, instead of $\Sigma(S,V,N,\mathcal{D})$, $\pi(S,V,N,\mathcal{D},\Delta)$, $\mathcal{O}(S,V,N,\mathcal{D})$, 
$\Res(S,V,N,\mathcal{D})$, $U(S,V,N,\mathcal{D},\Delta)$,
$V^\circ(S,V,N,\mathcal{D},\Delta)$ and $V(S,V,N,$\hfill\break$\mathcal{D},\Delta)$ respectively.

When we need not refer to the triplet $(V,N,\mathcal{D})$, we also write simply
$\Sigma(S)$, $\pi(S,\Delta)$, $\mathcal{O}(S)$, $\Res(S)$, $U(S,\Delta)$, $V^\circ(S,\Delta)$ and $V(S,\Delta)$, instead of $\Sigma(S,V,N,\mathcal{D})$, $\pi(S,V,N,\mathcal{D},\Delta)$, $\mathcal{O}(S,V,N,\mathcal{D})$, 
$\Res(S,V,N,\mathcal{D})$, $U(S,V,N,\mathcal{D},\Delta)$,
$V^\circ(S,V,N,\mathcal{D},\Delta)$ and $V(S,V,N,\mathcal{D},\Delta)$ respectively.

When we need not refer to the triplet $(S,V,N)$, we also write simply
$\Sigma(\mathcal{D})$, $\pi(\mathcal{D},\Delta)$, $\mathcal{O}(\mathcal{D})$, $\Res(\mathcal{D})$, $U(\mathcal{D},\Delta)$, $V^\circ(\mathcal{D},\Delta)$ and $V(\mathcal{D},\Delta)$, instead of $\Sigma(S,V,N,\mathcal{D})$, $\pi(S,V,N,\mathcal{D},\Delta)$, $\mathcal{O}(S,V,N,\mathcal{D})$, 
$\Res(S,V,N,\mathcal{D})$, $U(S,V,N,\mathcal{D},\Delta)$,
$V^\circ(S,V,N,\mathcal{D},\Delta)$ and $V(S,V,N,\mathcal{D},\Delta)$ respectively.

\begin{lemma}
\label{scheme1}
Let $S$ be any integral domain; let $V$ be any finite dimensional vector space over $\R$; let $N$ be any lattice of $V$ and let $\mathcal{D}$ be any simplicial cone decomposition over $N^*$ in $V^*$.

\begin{enumerate}
\item
Consider any $\Delta\in\mathcal{D}$.
$U(\Delta)$ is a non-empty open subset of $\Sigma$.
$\pi(\Delta)(\Spec(R(\Delta$\break$)))= U(\Delta)$.
The mapping $\pi(\Delta):\Spec(R(\Delta))\rightarrow U(\Delta)$ induced by $\pi(\Delta)$ is a continuous bijective mapping whose inverse mapping is also continuous.

For any open subset $U$ of $U(\Delta)$, $\mathcal{O}(U)=\mathcal{O}_{R(\Delta)}(\pi(\Delta)^{-1}(U))$.

$\mathcal{O}(U(\Delta))= R(\Delta)$.

The mapping $\pi(\Delta)$ induces an isomorphism from the affine scheme\hfill\break
$(\Spec(R(\Delta)),\mathcal{O}_{R(\Delta)})$ to the topological space with a sheaf
$(U(\Delta),\mathcal{O}|U(\Delta))$.
\item
Consider any $\Delta\in\mathcal{D}$.
Let $\hat{\Delta}$ be any simplicial cone over $N^*$ in $V^*$ satisfying $\dim \hat{\Delta}=\dim V$ and $\Delta\in\mathcal{F}(\hat{\Delta})$.

The set $\{b_{E/N^*}|E\in\mathcal{F}(\hat{\Delta})_1\}$ is a basis of the vector space $V^*$ over $\R$ and it is a $\Z$-basis of the lattice $N^*$.
By  $\{{b_{E/N^*}}^\vee_{\hat{\Delta}}|E\in\mathcal{F}(\hat{\Delta})_1\}$ we denote the dual basis of  $\{b_{E/N^*}|E\in\mathcal{F}(\hat{\Delta})_1\}$.
The set $\{{b_{E/N^*}}^\vee_{\hat{\Delta}}|E\in\mathcal{F}(\hat{\Delta})_1\}$ is a basis of the vector space $V$ over $\R$ and it is a $\Z$-basis of the lattice $N$.
For any $E\in\mathcal{F}(\hat{\Delta})_1$ and any $D\in E\in\mathcal{F}(\hat{\Delta})_1$,
\begin{equation*}
\langle b_{E/N^*}, {b_{D/N^*}}^\vee_{\hat{\Delta}}\rangle=
\begin{cases}
1&\text{ if $E=D$},\\
0&\text{ if $E\neq D$}.
\end{cases}\end{equation*}
\begin{enumerate}
\item
$R(\hat{\Delta})$ is a subring of $\Map'(N,S)$ containing $S$, and it is a polynomial ring over $S$ with $\dim V$ variables. The set $\{x({b_{E/N^*}}^\vee_{\hat{\Delta}})|E\in\mathcal{F}(\hat{\Delta})_1\}$ is a variable system of $R(\hat{\Delta})$ over $S$.
$\Spec(R(\hat{\Delta}))$ is smooth, if $\Spec(S)$ is smooth.
\item
$R(\Delta)=R(\hat{\Delta})[\{1/ x({b_{E/N^*}}^\vee_{\hat{\Delta}})|E\in\mathcal{F}(\hat{\Delta})_1-\mathcal{F}(\Delta)_1\}]
\supset R(\hat{\Delta})$.
$\Spec(R(\Delta))$ is smooth, if $\Spec(S)$ is smooth.
\item
$\{\bar{\p}\cap R(\hat{\Delta})|\bar{\p}\in\Spec(R(\Delta))\}=
\{\p\in\Spec(R(\hat{\Delta}))| x({b_{E/N^*}}^\vee_{\hat{\Delta}})\not\in\p\text{ for any }
$\hfill\break$E\in\mathcal{F}(\hat{\Delta})_1-\mathcal{F}(\Delta)_1\}$.
\item
$\{\bar{\p}\cap R(\hat{\Delta})|\bar{\p}\in\pi(\Delta)^{-1}(V^\circ(\Delta))\}=
\{\p\in\Spec(R(\hat{\Delta}))| x({b_{E/N^*}}^\vee_{\hat{\Delta}})\not\in\p\text{ for}$\break $\text{any }
E\in\mathcal{F}(\hat{\Delta})_1-\mathcal{F}(\Delta)_1,
x({b_{E/N^*}}^\vee_{\hat{\Delta}})\in\p\text{ for any } E\in\mathcal{F}(\Delta)_1
\}$.
\item
Consider any $\bar{\Delta}\in\mathcal{F}(\Delta)$.
$\{\bar{\p}\cap R(\hat{\Delta})|\bar{\p}\in\pi(\Delta)^{-1}(U(\Delta)\cap V(\bar{\Delta}))\}=
\{\p\in\Spec(R(\hat{\Delta}))| x({b_{E/N^*}}^\vee_{\hat{\Delta}})\not\in\p\text{ for any }
E\in\mathcal{F}(\hat{\Delta})_1-\mathcal{F}(\Delta)_1,
x({b_{E/N^*}}^\vee_{\hat{\Delta}})\in\p\text{ for any } E\in\mathcal{F}(\bar{\Delta})_1
\}$.

$$V(\bar{\Delta})\cap U(\Delta)=\bigcup_{\Lambda\in\mathcal{D},\bar{\Delta}\subset\Lambda\subset\Delta}V^\circ(\Lambda).$$
\end{enumerate}
\item
$$\Sigma=\bigcup_{\Delta\in\mathcal{D}}U(\Delta)=\bigcup_{\Delta\in\mathcal{D}}V^\circ(\Delta).$$
\item
Consider any $\Delta\in\mathcal{D}$ and any $\bar{\Delta}\in\mathcal{D}$.

$U(\Delta\cap\bar{\Delta})=U(\Delta)\cap U(\bar{\Delta})$.

$\Delta=\bar{\Delta}\Leftrightarrow U(\Delta)=U(\bar{\Delta})\Leftrightarrow V^\circ(\Delta)=V^\circ(\bar{\Delta})
\Leftrightarrow V^\circ(\Delta)\cap V^\circ(\bar{\Delta})\neq\emptyset\Leftrightarrow V(\Delta)=V(\bar{\Delta})$.

$\Delta\subset\bar{\Delta}\Leftrightarrow U(\Delta)\subset U(\bar{\Delta})\Leftrightarrow V^\circ(\Delta)\subset U(\bar{\Delta})
\Leftrightarrow V^\circ(\Delta)\cap U(\bar{\Delta})\neq\emptyset\Leftrightarrow
V(\Delta)\cap U(\bar{\Delta})\neq\emptyset\Leftrightarrow V(\Delta)\supset V(\bar{\Delta})$.
\item
Consider any $\Delta\in\mathcal{D}$.
\begin{enumerate}
\item
$V^\circ(\Delta)$ is a non-empty closed irreducible subset of $U(\Delta)$.
\item
$\Codim(V^\circ(\Delta), U(\Delta))=\dim \Delta$.
\item
If we give the reduced scheme structure to $V^\circ(\Delta)$, then $V^\circ (\Delta)$ is a complete intersection subscheme of $ U(\Delta)$, and 
$V^\circ(\Delta)$ is smooth if $\Spec(S)$ is smooth.
\item
$$U(\Delta)=\bigcup_{\bar{\Delta}\in\mathcal{D}, \bar{\Delta} \subset \Delta}V^\circ(\bar{\Delta}).$$
\item
$V(\Delta)\cap U(\Delta)=V^\circ(\Delta)$.
\item
$V(\Delta)$ is a non-empty closed irreducible subset of $\Sigma$.
\item
$\Codim(V(\Delta), \Sigma)=\dim \Delta$.
\item
If we give the reduced scheme structure to $V(\Delta)$, then $V(\Delta)$ is a local complete intersection subscheme of $\Sigma$, and 
$V(\Delta)$ is smooth if $\Spec(S)$ is smooth.
\item
$$V(\Delta)=\bigcup_{\bar{\Delta}\in\mathcal{D}, \bar{\Delta}\supset\Delta}V^\circ(\bar{\Delta}).$$
\end{enumerate}
\item
$\mathcal{O}$ is a sheaf of rings on $\Sigma$, in other words, the following seven conditions are satisfied:
\begin{enumerate}
\item
For any open subset $U$ of $\Sigma$, $\mathcal{O}(U)$ is a ring.
\item 
For any open subsets $U$ and $V$ of $\Sigma$ with $U\subset V$, $\Res^V_U$ is a ring homomorphism.
\item
For any open subset $U$ of $\Sigma$, $\Res^U_U$ is the identity mapping of $\mathcal{O}(U)$.
\item
For any open subsets $U$, $V$ and $W$ of $\Sigma$ with $U\subset V\subset W$, $\Res^W_U=\Res^V_U \Res^W_V$.
\item
Consider any non-empty set $\mathcal{U}$ whose elements are open subsets of $\Sigma$. We denote $\hat{U}= \cup_{U\in\mathcal{U}}U$. Consider any 
$\phi\in\mathcal{O}(\hat{U})$.

If $\Res^{\hat{U}}_U(\phi)=0$ for any $U\in\mathcal{U}$, then $\phi=0$.
\item
Consider any non-empty set $\mathcal{U}$ whose elements are open subsets of $\Sigma$. We denote $\hat{U}= \cup_{U\in\mathcal{U}}U$. Consider any 
$\phi(U)\in\mathcal{O}(U)$ for any $U\in\mathcal{U}$.

If $\Res^{U}_{U\cap V}(\phi(U))= \Res^{V}_{U\cap V}(\phi(V))$ for any $U\in\mathcal{U}$ and any $V\in\mathcal{U}$, then there exists $\phi\in\mathcal{O}(\hat{U})$ satisfying $\Res^{\hat{U}}_U(\phi)=\phi(U)$ for any $U\in\mathcal{U}$.
\item
$\mathcal{O}(\emptyset)=\{0\}$.
\end{enumerate}
\item
The pair $\Sigma=(\Sigma, \mathcal{O})$ of the topological space $\Sigma$ and the sheaf $\mathcal{O}$ on $\Sigma$ is a separated reduced irreducible scheme.
$\dim \Sigma=\dim S+\dim V$, where $\dim S\in\Z_0\cup\{\infty\}$ denotes the Krull dimension of the ring $S$.
$\Sigma$ is smooth, if $\Spec(S)$ is smooth.
\item
If $\Delta$ is a simplicial cone over $N^*$ in $V^*$ and $\mathcal{D}=\mathcal{F}(\Delta)$, then $\Sigma$ is isomorphic to $\Spec(\Map'(N,S)\backslash(\Delta^\vee|V^*))$.
\item
For any $\A\in\Sigma$ and any $\phi\in \mathcal{O}_\A$ there exist an open subset $U$ of $\Sigma$ with $\A\in U$ and an element $\psi\in \mathcal{O}(U)$ satisfying $\Res^U_\A(\psi)=\phi$.
\item
Consider any $\A\in\Sigma$.

Note that for any open subset $U$ of $\Sigma$ with $\A\in U$ and any open subset $V$ of $\Sigma$ with $\A\in V$, there exists an open subset $W$ of $\Sigma$ with $\A\in W\subset U\cap V$.

The pair $(\mathcal{O}_\A, \{\Res^U_\A|U$ is an open subset of $\Sigma$ with $\A\in U\})$ is the inductive limit of the inductive system $(\{\mathcal{O}(U)|U$ is an open subset of $\Sigma$ with $\A\in U\},\{ \Res^V_U|U$ and $V$ are open subsets of $\Sigma$ with $\A\in U\subset V\})$, in other words, 
$\Res^V_\A=\Res^U_\A\Res^V_U$ for any open subsets $U$ and $V$ of $\Sigma$ with $\A\in U\subset V$ and the following condition is satisfied:

Assume that a ring $T$ is given and a ring homomorphism $\mu(U): \mathcal{O}(U)\rightarrow T$ is given for any open subset $U$ of $\Sigma$ with $\A\in U$.
If $\mu (V)=\mu(U)\Res^V_U$ for any open subsets $U$ and $V$ of $\Sigma$ with $\A\in U\subset V$, then there exists uniquely a ring homomorphism $\mu: \mathcal{O}_\A \rightarrow T$ satisfying $\mu(U)=\mu\Res^U_\A$ for any open subset $U$ of $\Sigma$ with $\A\in U$.
\end{enumerate}
\end{lemma}

For any $\A\in\Sigma$, the ring $\mathcal{O}_\A$ is called the \emph{local ring} of $\Sigma$ at $\A$ and any element of $\mathcal{O}_\A$ is called a \emph{germ} of functions at $\A$.

Consider any $\Delta\in\mathcal{D}$.

Consider any open subset $U$ of $\Sigma$.
By claim 1 of the above lemma, we have $\mathcal{O}_{R(\Delta)}(\pi(\Delta)^{-1}(U))=\mathcal{O}(U\cap U(\Delta))$.
We define the ring homomorphism 
$$\pi(\Delta)^*(U):\mathcal{O}(U)\rightarrow
\mathcal{O}_{R(\Delta)}(\pi(\Delta)^{-1}(U))$$
by putting $\pi(\Delta)^*(U)=\Res^U_{ U\cap U(\Delta)}$.

We denote the set $\{\pi(\Delta)^*(U)|U$ is an open subset of $\Sigma\}$ by a single symbol $\pi(\Delta)^*$, and we denote the pair $(\pi(\Delta),\pi(\Delta)^*)$ by a single symbol $\pi(\Delta)$.

Since $S$ is a subring of $R(\Delta)$, we have a morphism of schemes $\iota(\Delta)^*:\Spec(R(\Delta))$\hfill\break$\rightarrow\Spec(S)$ associated with the inclusion ring homomorphism $\iota(\Delta):S\rightarrow R(\Delta)$.

\begin{lemma}
\label{structure morphism}
Let $S$ be any integral domain; let $V$ be any finite dimensional vector space over $\R$; let $N$ be any lattice of $V$ and let $\mathcal{D}$ be any simplicial cone decomposition over $N^*$ in $V^*$.

\begin{enumerate}
\item
Consider any $\Delta\in\mathcal{D}$.

$\pi(\Delta)^*$ is a morphism of sheaves over $\pi(\Delta)$, in other words, for any open subsets $U$ and $V$ of $\Sigma$ with $U\subset V$, $\Res^{\pi(\Delta)^{-1}(V)}_ {\pi(\Delta)^{-1}(U)}\pi(\Delta)^*(V)=\pi(\Delta)^*(U)\Res^V_U$.

The pair $\pi(\Delta)=(\pi(\Delta),\pi(\Delta)^*)$ is a morphism of schemes $\Spec(R(\Delta))\rightarrow\Sigma$, and it is an open embedding.
\item
There exists uniquely a morphism of schemes $\rho:\Sigma\rightarrow\Spec(S)$ satisfying  $\rho\pi(\Delta)=\iota(\Delta)^*$ for any $\Delta\in\mathcal{D}$, where $\iota(\Delta)^*:\Spec(R(\Delta))\rightarrow\Spec(S)$ denotes the morphism of schemes associated with the inclusion ring homomorphism $\iota(\Delta):S\rightarrow R(\Delta)$.
\end{enumerate}

We take the morphism $\rho:\Sigma\rightarrow\Spec(S)$ satisfyng $\rho\pi(\Delta)=\iota(\Delta)^*$ for any $\Delta\in\mathcal{D}$, where $\iota(\Delta)^*:\Spec(R(\Delta))\rightarrow\Spec(S)$ denotes the same as above.
\begin{enumerate}
\setcounter{enumi}{2}
\item
The morphism $\rho$ is separated, of finite type, smooth and surjective.

Consider any $\p\in\Spec(S)$. The inverse image of the point $\p$ by $\rho$ is the fiber product scheme $\Sigma\times_{\Spec(S)}\Spec(S_\p/\p S_p)$. The scheme $\Sigma\times_{\Spec(S)}\Spec(S_\p/\p S_p)$ is non-empty, irreducible and smooth, and
$\dim \Sigma\times_{\Spec(S)}\Spec(S_\p/\p S_p)=\dim V$.

\item
For any $\A\in\Sigma$, $\rho(\A)=M(\mathcal{O}_\A)\cap S$.
\item
Consider any open subset $U$ of $\Spec(S)$. Note that the morphism of schemes $\rho$ gives a ring homomorphism $\rho^*(U):\mathcal{O}_S(U)\rightarrow\mathcal{O}(\rho^{-1}(U))$.

If $U\neq\emptyset$, then $\rho^{-1}(U)\neq\emptyset$, $\mathcal{O}_S(U)$ and $\mathcal{O}(\rho^{-1}(U))$ are subrings of $K$, $\mathcal{O}_S(U)\subset\mathcal{O}(\rho^{-1}(U))$ and $\rho^*(U)$ is equal to the inclusion ring homomorphism.

If $U=\emptyset$, then $\rho^{-1}(U)=\emptyset$, $\mathcal{O}_S(U)=\mathcal{O}(\rho^{-1}(U))=\{0\}$ and $\rho^*(U)$ is equal to the inclusion ring homomorphism.
\item
Consider any $\A\in\Sigma$.

$\rho(\A)= M(\mathcal{O}_\A)\cap S\in\Spec(S)$.
$S_{\rho(\A)}\subset\mathcal{O}_\A$.
$M(S_{\rho(\A)})=M(\mathcal{O}_\A)\cap S_{\rho(\A)}$.

There exists uniquely a ring homomorphism $\rho^*(\A):S_{\rho(\A)}\rightarrow\mathcal{O}_\A$ satisfying $\Res^{\rho^{-1}(U)}_\A\rho^*(U)= \rho^*(\A)\Res^U_{\rho(\A)}$ for any open subset $U$ of $\Spec(S)$ with $\rho(\A)\in U$.

If a ring homomorphism $\rho^*(\A):S_{\rho(\A)}\rightarrow\mathcal{O}_\A$ satisfies $\Res^{\rho^{-1}(U)}_\A\rho^*(U)= \rho^*(\A)\Res^U_{\rho(\A)}$ for any open subset $U$ of $\Spec(S)$ with $\rho(\A)\in U$, then $\rho^*(\A)$ coincides with the inclusion ring homomorphism and 
$\rho^*(\A)(M(S_{\rho(\A)}))\subset M(\mathcal{O}_\A)$.
\item
The morphism $\rho$ is proper, if and only if, $|\mathcal{D}|=V^*$.
\end{enumerate}
\end{lemma}

We denote $\pi(\Delta)^*=\pi(S,V,N,\mathcal{D},\Delta)^*$ for any $\Delta\in\mathcal{D}$.

Let $\rho:\Sigma\rightarrow\Spec(S)$ be the morphism in the above lemma.
We denote $\rho=\rho(S,V,N,\mathcal{D},\Delta)$, and we call $\rho(S,V,N,\mathcal{D},\Delta)$ the \emph{structure morphism} of $\Sigma(S,V,N,$\hfill\break$\mathcal{D},\Delta)$.

When we need not refer to the quadruplet $(S,V,N,\mathcal{D})$, we also write simply
$\pi(\Delta)^*$ and $\rho$, instead of $\pi(S,V,N,\mathcal{D},\Delta)^*$ and $\rho(S,V,N,\mathcal{D})$ respectively.

When we need not refer to the triplet $(V,N,\mathcal{D})$, we also write simply
$\pi(S,\Delta)^*$ and $\rho(S)$, instead of $\pi(S,V,N,\mathcal{D},\Delta)^*$ and $\rho(S,V,N,\mathcal{D})$ respectively.

When we need not refer to the triplet $(S,V,N)$, we also write simply
$\pi(\mathcal{D},\Delta)^*$ and $\rho(\mathcal{D})$, instead of $\pi(S,V,N,\mathcal{D},\Delta)^*$ and $\rho(S,V,N,\mathcal{D})$ respectively.

\begin{lemma}
\label{base change}
Let $S$ and $T$ be any integral domains such that there exists a ring homomorphism from $S$ to $T$; let $\lambda:S\rightarrow T$ be any ring homomorphism: let $V$ be any finite dimensional vector space over $\R$; let $N$ be any lattice of $V$ and let $\mathcal{D}$ be any simplicial cone decomposition over $N^*$ in $V^*$.

\begin{enumerate}
\item
For any $\phi\in\Map(N,S)$, $\lambda\phi\in\Map(N,T)$ and $\Supp(\lambda\phi)\subset\Supp(\phi)$.

For any $\phi\in\Map'(N,S)$, $\lambda\phi\in\Map'(N,T)$.

Let $\Delta$ be any convex polyhedral cone in $V^*$.
For any $\phi\in\Map'(N,S)\backslash(\Delta^\vee|$\hfill\break$V^*)$, $\lambda\phi\in\Map'(N,T) \backslash(\Delta^\vee|V^*)$.
\end{enumerate}

Putting
$\lambda_\bullet(V,N)(\phi)=\lambda\phi\in\Map'(N,T)$
for any $\phi\in \Map'(N,S)$, we define a mapping $\lambda_\bullet(V,N): \Map'(N,S)\rightarrow\Map'(N,T)$.

Let $\Delta$ be any convex polyhedral cone in $V^*$.
Putting
$\lambda_\bullet(V,N, \Delta)(\phi)=\lambda\phi\in\Map'(N,T)\backslash(\Delta^\vee|V^*)$
for any $\phi\in \Map'(N,S)\backslash(\Delta^\vee|V^*)$, we define a mapping $\lambda_\bullet(V,$\hfill\break$N, \Delta): \Map'(N,S)\backslash(\Delta^\vee|V^*)\rightarrow\Map'(N,T)\backslash(\Delta^\vee|V^*)$.

\begin{enumerate}
\setcounter{enumi}{1}
\item
The mapping $\lambda_\bullet(V,N)$ is a ring homomorphism.
For any convex polyhedral cone $\Delta$ in $V^*$, $\lambda_\bullet(V,N,\Delta)$ is a ring homomorphism.
\item
Note that $\pi(S,\Delta):\Spec(\Map'(N,S)\backslash(\Delta^\vee|V^*))\rightarrow\Sigma(S)$,
$\pi(T,\Delta):\Spec(\Map'$\hfill\break$(N,T) \backslash(\Delta^\vee|V^*))\rightarrow\Sigma(T)$, and
$\lambda_\bullet(V,N, \Delta)^*: \Spec(\Map'(N,T)\backslash(\Delta^\vee|V^*))\rightarrow\Spec(\Map'(N,S) \backslash(\Delta^\vee|V^*))$
for any $\Delta\in\mathcal{D}$.
They are morphisms of schemes.

There exists uniquely a morphism $\lambda^\bullet(V,N,\mathcal{D}):\Sigma(T)\rightarrow\Sigma(S)$ of schemes satisfying $\pi(S,\Delta) \lambda_\bullet(V,N, \Delta)^*= \lambda^\bullet(V,N,\mathcal{D})\pi(T,\Delta)$ for any $\Delta\in\mathcal{D}$.
\end{enumerate}

We take the unique morphism $\lambda^\bullet(V,N,\mathcal{D}):\Sigma(T)\rightarrow\Sigma(S)$ of schemes satisfying $\pi(S,\Delta) \lambda_\bullet(V,N, \Delta)^*= \lambda^\bullet(V,N,\mathcal{D})\pi(T,\Delta)$ for any $\Delta\in\mathcal{D}$.
The morphism $\lambda^\bullet(V,N,\mathcal{D})$ is uniquely defined depending on the quadruplet $(\lambda,V,N,\mathcal{D})$.
\begin{enumerate}
\setcounter{enumi}{3}
\item
$\lambda^\bullet(V,N,\mathcal{D})^{-1}(U(S,\Delta))=U(T,\Delta)$ for any $\Delta\in\mathcal{D}$.
\item
$\rho(S) \lambda^\bullet(V,N,\mathcal{D})=\lambda^*\rho(T)$, where
$\rho(S):\Sigma(S)\rightarrow\Spec(S)$ and $\rho(T):\Sigma(T)\rightarrow\Spec(T)$ are structure morphisms and $\lambda^*:\Spec(T)\rightarrow\Spec(S)$ is the morphism of schemes induced by the ring homomorphism $\lambda:S\rightarrow T$.
\item
The morphism
$$(\lambda^\bullet(V,N,\mathcal{D}), \rho(T)):\Sigma(T)\rightarrow
\Sigma(S)\times_{\Spec(S)}\Spec(T)$$
induced by $\lambda^\bullet(V,N,\mathcal{D})$ and $\rho(T)$ is an isomorphism.
\item
$\Id_{S\bullet}(V,N)=\Id_{\Map'(N,S)}$.
$\Id_{S\bullet}(V,N,\Delta)=\Id_{\Map'(N,S)\backslash(\Delta^\vee|V^*)}$ for any convex polyhedral cone $\Delta$ in $V^*$.
$\Id_S^\bullet(V,N,\mathcal{D})=\Id_{\Sigma(S, V,N,\mathcal{D})}$.

Let $U$ be any integral domain such that there exists a ring homomorphism from $T$ to $U$ and let $\mu:T\rightarrow U$ be any ring homomorphism.
$$(\mu\lambda)_\bullet(V,N)= \mu_\bullet(V,N)\lambda_\bullet(V,N).$$
$$(\mu\lambda)_\bullet(V,N,\Delta)= \mu_\bullet(V,N,\Delta)\lambda_\bullet(V,N,\Delta) $$ for any convex polyhedral cone $\Delta$ in $V^*$.
$$(\mu\lambda)^\bullet(V,N,\mathcal{D})=\lambda^\bullet(V,N,\mathcal{D})\mu^\bullet(V,N,\mathcal{D}).$$
\end{enumerate}
\end{lemma}

\begin{theorem}
\label{subdivision morphism}
Let $S$ be any integral domain; let $V$ be any finite dimensional vector space over $\R$; let $N$ be any lattice of $V$ and let $\mathcal{D}$ and $\bar{\mathcal{D}}$ be any simplicial cone decompositions over $N^*$ in $V^*$ such that $\mathcal{D}$ is a subdivision of $\bar{\mathcal{D}}$.
\begin{enumerate}
\item
For any $\Delta\in\mathcal{D}$ and any $\bar{\Delta}\in\bar{\mathcal{D}}$ with $\Delta\subset\bar{\Delta}$, $R(\bar{\Delta})\subset R(\Delta)$.
\item
There exists uniquely a morphism $\sigma(S,V,N)^{\mathcal{D}}_{\bar{\mathcal{D}}}:
\Sigma(S,V,N,\mathcal{D})\rightarrow \Sigma(S,V,N,\bar{\mathcal{D}})$ of schemes satisfying $\pi(S,V,N,\bar{\mathcal{D}})\iota(\Delta,\bar{\Delta})^*=
\sigma(S,V,N)^{\mathcal{D}}_{\bar{\mathcal{D}}}\pi(S,V,N,\mathcal{D})$
for any $\Delta\in\mathcal{D}$ and any $\bar{\Delta}\in\bar{\mathcal{D}}$ with $\Delta\subset\bar{\Delta}$, where 
$\iota(\Delta,\bar{\Delta}): R(\bar{\Delta})\rightarrow R(\Delta)$ denotes the inclusion ring homomorphism and 
$\iota(\Delta,\bar{\Delta})^*:\Spec(R(\Delta))\rightarrow\Spec(R(\bar{\Delta}))$ denotes the morphism of schemes induced by $\iota(\Delta,\bar{\Delta})$.
\end{enumerate}

We take the unique morphism $\sigma(S,V,N)^{\mathcal{D}}_{\bar{\mathcal{D}}}:
\Sigma(S,V,N,\mathcal{D})\rightarrow \Sigma(S,V,N,\bar{\mathcal{D}})$ of schemes satisfying $\pi(S,V,N,\bar{\mathcal{D}})\iota(\Delta,\bar{\Delta})^*=
\sigma(S,V,N)^{\mathcal{D}}_{\bar{\mathcal{D}}}\pi(S,V,N,\mathcal{D})$
for any $\Delta\in\mathcal{D}$ and any $\bar{\Delta}\in\bar{\mathcal{D}}$ with $\Delta\subset\bar{\Delta}$, where 
$\iota(\Delta,\bar{\Delta})^*$ denotes the same as above.

\begin{enumerate}
\setcounter{enumi}{2}
\item
$$\sigma(S,V,N)^{\mathcal{D}}_{\mathcal{D}}=\Id_{\Sigma(S,V,N,\mathcal{D})}.$$

Let $\bar{\bar{\mathcal{D}}}$ be any simplicial cone decomposition over $N^*$ in $V^*$ such that $\bar{\mathcal{D}}$ is a subdivision of $\bar{\bar{\mathcal{D}}}$.
$$\sigma(S,V,N)^{\bar{\mathcal{D}}}_{\bar{\bar{\mathcal{D}}}}\sigma(S,V,N)^{\mathcal{D}}_{\bar{\mathcal{D}}}=
\sigma(S,V,N)^{\mathcal{D}}_{\bar{\bar{\mathcal{D}}}}.$$
\item
Let $T$ be any integral domain such that there exists a ring homomorphism from $S$ to $T$ and $\lambda:S\rightarrow T$ be any ring homomorphism.
$$\sigma(S,V,N)^{\mathcal{D}}_{\bar{\mathcal{D}}}\lambda^\bullet(V,N,\mathcal{D})
=\lambda^\bullet(V,N,\bar{\mathcal{D}})\sigma(T,V,N)^{\mathcal{D}}_{\bar{\mathcal{D}}}.$$
\end{enumerate}

Below, we denote
$$\sigma=\sigma(S,V,N)^{\mathcal{D}}_{\bar{\mathcal{D}}}:
\Sigma(S,V,N,\mathcal{D})=\Sigma(\mathcal{D})\rightarrow \Sigma(S,V,N,\bar{\mathcal{D}})=\Sigma(\bar{\mathcal{D}}),$$
for simplicity.
\begin{enumerate}
\setcounter{enumi}{4}
\item
The morphism $\sigma$ is separated, of finite type, dominating and birational.

$\rho(\bar{\mathcal{D}})\sigma=\rho(\mathcal{D})$.
\item
Consider any $\A\in\Sigma$.

$\mathcal{O}(\bar{\mathcal{D}})_{\sigma(\A)}\subset\mathcal{O}(\mathcal{D})_\A$.
$M(\mathcal{O}(\bar{\mathcal{D}})_{\sigma(\A)})=M(\mathcal{O}(\mathcal{D})_\A)\cap\mathcal{O}(\bar{\mathcal{D}})_{\sigma(\A)}$.

If $\bar{\A}\in\Sigma(\bar{\mathcal{D}})$,
$\mathcal{O}(\bar{\mathcal{D}})_{\bar{\A}}\subset\mathcal{O}(\mathcal{D})_\A$ and 
$M(\mathcal{O}(\bar{\mathcal{D}})_{\bar{\A}})=M(\mathcal{O}(\mathcal{D})_\A)\cap\mathcal{O}(\bar{\mathcal{D}})_{\bar{\A}}$, then $\bar{\A}=\sigma(\A)$.
\item
Consider any open subset $U$ of $\Sigma(\bar{\mathcal{D}})$. Note that the morphism of schemes $\sigma$ gives a ring homomorphism $\sigma^*(U):\mathcal{O}(\bar{\mathcal{D}}) (U)\rightarrow\mathcal{O}(\mathcal{D})(\rho^{-1}(U))$.

If $U\neq\emptyset$, then $\sigma^{-1}(U)\neq\emptyset$, $\mathcal{O}(\bar{\mathcal{D}})(U)$ and $\mathcal{O}(\mathcal{D})(\rho^{-1}(U))$ are subrings of $K$, $\mathcal{O}(\bar{\mathcal{D}}) (U)\subset\mathcal{O}(\mathcal{D})(\rho^{-1}(U))$ and $\sigma^*(U)$ is equal to the inclusion ring homomorphism.

If $U=\emptyset$, then $\sigma^{-1}(U)=\emptyset$, $\mathcal{O}(\bar{\mathcal{D}})(U)= \mathcal{O}(\mathcal{D})(\rho^{-1}(U))=\{0\}$ and $\sigma^*(U)$ is equal to the inclusion ring homomorphism.
\item
Consider any $\A\in\Sigma(\mathcal{D})$.

There exists uniquely a ring homomorphism $\sigma^*(\A): \mathcal{O}(\bar{\mathcal{D}})_{\sigma(\A)}\rightarrow\mathcal{O}(\mathcal{D})_\A$ satisfying $\Res^{\sigma^{-1}(U)}_\A\sigma^*(U)= \sigma^*(\A)\Res^U_{\rho(\A)}$ for any open subset $U$ of $\Sigma(\bar{\mathcal{D}})$ with $\sigma(\A)\in U$.

If a ring homomorphism $\sigma^*(\A): \mathcal{O}(\bar{\mathcal{D}})_{\sigma(\A)}\rightarrow\mathcal{O}(\mathcal{D})_\A$ satisfies $\Res^{\sigma^{-1}(U)}_\A\sigma^*($\hfill\break$U)= \sigma^*(\A)\Res^U_{\rho(\A)}$ for any open subset $U$ of $\Sigma(\bar{\mathcal{D}})$ with $\sigma(\A)\in U$, then $\sigma^*(\A)$ coincides with the inclusion ring homomorphism and 
$\sigma^*(\A)(M(\mathcal{O}(\bar{\mathcal{D}})_{\sigma(\A)}))\subset M(\mathcal{O}(\mathcal{D})_\A)$.
\item
The morphism $\sigma$ is proper, if and only if, $|\mathcal{D}|=|\bar{\mathcal{D}}|$.
\item
For any $\Delta\in\mathcal{D}$ and any $\bar{\Delta}\in\bar{\mathcal{D}}$ with $\Delta\subset\bar{\Delta}$, $\sigma(U(\mathcal{D},\Delta))\subset U(\bar{\mathcal{D}},\bar{\Delta})$.
\item
For any $\Delta\in\mathcal{D}\cap \bar{\mathcal{D}}$, $\sigma(U(\mathcal{D},\Delta))= U(\bar{\mathcal{D}},\Delta)$ and the morphism $\sigma: U(\mathcal{D},\Delta)\rightarrow U(\bar{\mathcal{D}},\Delta)$ induced by $\sigma$ is an isomorphism of schemes.
\item
If $\bar{\Delta}\in\bar{\mathcal{D}}$, $\dim \bar{\Delta}\geq 1$, $\mathcal{D}=\bar{\mathcal{D}}*\bar{\Delta}$ and we give the reduced scheme structure to the closed subset $V(\bar{\mathcal{D}},\bar{\Delta})$ of $\Sigma(\bar{\mathcal{D}})$, then the morphism $\sigma:\Sigma(\mathcal{D})\rightarrow \Sigma(\bar{\mathcal{D}})$ coincides with the blowing-up of $\Sigma(\bar{\mathcal{D}})$ with center in the subscheme $V(\bar{\mathcal{D}},\bar{\Delta})$.
\item
Consider any $\Delta\in\mathcal{D}$.
Let $\bar{\Delta}\in\bar{\mathcal{D}}$ be the unique element with $\Delta^\circ\subset\bar{\Delta}^\circ$.

Then, $\sigma(V^\circ(\mathcal{D},\Delta))= V^\circ(\bar{\mathcal{D}},\bar{\Delta}).$
\item
For any $\bar{\Delta}\in\bar{\mathcal{D}}$,
\begin{equation*}\begin{split}
\sigma^{-1}(U(\bar{\mathcal{D}},\bar{\Delta}))=&
\bigcup_{\Delta\in\mathcal{D},\Delta\subset\bar{\Delta}}U(\mathcal{D},\Delta),
\text{ and}\\
\sigma^{-1}(V^\circ(\bar{\mathcal{D}},\bar{\Delta}))=&
\bigcup_{\Delta\in\mathcal{D},\Delta^\circ\subset\bar{\Delta}^\circ}V^\circ (\mathcal{D},\Delta).
\end{split}\end{equation*}
\end{enumerate}
\end{theorem}

We call the above morphism $\sigma(S,V,N)^{\mathcal{D}}_{\bar{\mathcal{D}}}:
\Sigma(S,V,N,\mathcal{D})\rightarrow \Sigma(S,V,N,\bar{\mathcal{D}})$ the \emph{subdivision morphism} associated with $\bar{\mathcal{D}}$ and a subdivision $\mathcal{D}$ of $\bar{\mathcal{D}}$.

When we need not refer to the quintuplet $(S,V,N, \bar{\mathcal{D}}, \mathcal{D})$, we also write simply
$\sigma$, instead of $\sigma(S,V,N)^{\mathcal{D}}_{\bar{\mathcal{D}}}$.

When we need not refer to the triplet $(S,V,N)$, we also write simply
$\sigma^{\mathcal{D}}_{\bar{\mathcal{D}}}$, instead of $\sigma(S,V,N)^{\mathcal{D}}_{\bar{\mathcal{D}}}$,

\begin{lemma}
\label{relation to ncs}
Let $S$ be any algebraically closed field; let $V$ be any finite dimensional vector space over $\R$ with $\dim V\geq 1$; let $N$ be any lattice of $V$; let $\mathcal{D}$ and $\bar{\mathcal{D}}$ be any simplicial cone decompositions over $N^*$ in $V^*$ such that $\mathcal{D}$ is an iterated barycentric subdivision of $\bar{\mathcal{D}}$, 
$|\bar{\mathcal{D}}|$ is a convex polyhedral cone in $V^*$ and $\dim |\bar{\mathcal{D}}|=\dim V$; let $m\in\Z_0$ and let $F$ be the center sequence of $\bar{\mathcal{D}}$ of length $m$ such that $\mathcal{D}=\bar{\mathcal{D}}*F(1)*F(2)*\cdots *F(m)$.

We denote $D=\sum_{\Gamma\in\mathcal{D}_1}V(\mathcal{D},\Gamma)\in\Div(\Sigma(\mathcal{D}))$ and $\bar{D}=\sum_{\Gamma\in\bar{\mathcal{D}}_1}V(\bar{\mathcal{D}},\Gamma)\in\Div(\Sigma(\bar{\mathcal{D}}))$.
\begin{enumerate}
\item
The pair $(\Sigma(\mathcal{D}),D)$ is a normal crossing scheme over $S$.
$\Comp(D)=\{ V(\mathcal{D},\Gamma)| $\break$\Gamma\in\mathcal{D}_1\}$.
$(D)_0=\cup_{\Delta\in\mathcal{D}^0}V^\circ(\mathcal{D},\Delta)$.
If $\A\in (D)_0$, $\Delta\in\mathcal{D}^0$ and $\{\A\}= V^\circ(\mathcal{D},\Delta)$, then
$U(\Sigma(\mathcal{D}), D,\A)=U(\mathcal{D},\Delta)$ and
$\Comp(D)(\A)= \{ V(\mathcal{D},\Gamma)| \Gamma\in\mathcal{F}(\Delta)_1\}$.
\item
The pair $(\Sigma(\bar{\mathcal{D}}),\bar{D})$ is a normal crossing scheme over $S$.
$\Comp(\bar{D})=\{ V(\bar{\mathcal{D}},\Gamma)| $\break$\Gamma\in\bar{\mathcal{D}}_1\}$.
$(\bar{D})_0=\cup_{\Delta\in\bar{\mathcal{D}}^0}V^\circ(\bar{\mathcal{D}},\Delta)$.
If $\A\in (\bar{D})_0$, $\Delta\in\bar{\mathcal{D}}^0$ and $\{\A\}= V^\circ(\bar{\mathcal{D}},\Delta)$, then
$U(\Sigma(\bar{\mathcal{D}}), \bar{D},\A)=U(\bar{\mathcal{D}},\Delta)$ and
$\Comp(\bar{D})(\A)= \{ V(\bar{\mathcal{D}},\Gamma)| \Gamma\in\mathcal{F}(\Delta)_1\}$.
\item
The morphism $\sigma:\Sigma(\mathcal{D})\rightarrow\Sigma(\bar{\mathcal{D}})$
is an admissible composition of blowing-ups over $\bar{D}$.
If $\dim F(i)=2$ for any $i\in\{1,2,\ldots,m\}$, then $\sigma$ is an admissible composition of blowing-ups with centers in codimension two over $\bar{D}$.
\item
$\Comp(D)=\Comp(\sigma^*\bar{D})$.

$\sigma((D)_0)=(\bar{D})_0$.

If $\A\in (D)_0$, $\Delta\in\mathcal{D}^0$, $\{\A\}= V^\circ(\mathcal{D},\Delta)$,
$\bar{\Delta}\in\bar{\mathcal{D}}^0$ and $\Delta\subset\bar{\Delta}$, then
$\{\sigma(\A)\}= V^\circ(\bar{\mathcal{D}},\bar{\Delta})$.
\end{enumerate}

Consider any simplicial cone $\Delta$ over $N^*$ in $V^*$ with $\dim\Delta=\dim V$.
The set $\{b_{\Gamma/N^*}|\Gamma\in\mathcal{F}(\Delta)_1\}$ is an $\R$-basis of $V^*$ and it is a $\Z$-basis of $N^*$.
We denote the dual basis of $\{b_{\Gamma/N^*}|\Gamma\in\mathcal{F}(\Delta)_1\}$ by
$\{{b_{\Gamma/N^*}}^\vee_\Delta|\Gamma\in\mathcal{F}(\Delta)_1\}$.

Consider any $\A\in (D)_0$.
We take $\Delta\in\mathcal{D}^0$ with $\{\A\}= V^\circ(\mathcal{D},\Delta)$.
$\dim\Delta=\dim V$.
We define a mapping $\xi_\A:\Comp(D)(\A)\rightarrow
\mathcal{O}(\mathcal{D})(U(\Sigma(\mathcal{D}), D,\A))$ by putting
$\xi_\A(V(\mathcal{D},\Gamma))=\pi(\mathcal{D},\Delta)^*( U(\Sigma(\mathcal{D}), D,\A))^{-1}(x({b_{\Gamma/N^*}}^\vee_\Delta))\in\mathcal{O}(\mathcal{D})(U(\Sigma(\mathcal{D}), D,\A))$ for any $\Gamma\in\mathcal{F}(\Delta)_1$.

We put $\xi=\{\xi_\A|\A\in (D)_0\}$.

Consider any $\A\in (\bar{D})_0$.
We take $\Delta\in\bar{\mathcal{D}}^0$ with $\{\A\}= V^\circ(\bar{\mathcal{D}},\Delta)$.
$\dim\Delta=\dim V$.
We define a mapping $\bar{\xi}_\A:\Comp(\bar{D})(\A)\rightarrow
\mathcal{O}(\bar{\mathcal{D}})(U(\Sigma(\bar{\mathcal{D}}), \bar{D},\A))$ by putting
$\bar{\xi}_\A(V(\bar{\mathcal{D}},\Gamma))=\pi(\bar{\mathcal{D}},\Delta)^*( U(\Sigma(\bar{\mathcal{D}}), \bar{D},\A))^{-1}(x({b_{\Gamma/N^*}}^\vee_\Delta))\in\mathcal{O}(\bar{\mathcal{D}})(U(\Sigma(\bar{\mathcal{D}}), \bar{D},\A))$ for any $\Gamma\in\mathcal{F}(\Delta)_1$.

We put $\bar{\xi}=\{\bar{\xi}_\A|\A\in (\bar{D})_0\}$.

\begin{enumerate}
\setcounter{enumi}{4}
\item
The triplet $(\Sigma(\mathcal{D}),D,\xi)$ is a coordinated normal crossing scheme over $S$.
\item
The triplet $(\Sigma(\bar{\mathcal{D}}),\bar{D},\bar{\xi})$ is a coordinated normal crossing scheme over $S$.
\item
$\xi=\sigma^*\bar{\xi}$.
\end{enumerate}
\end{lemma}

\section{Proof of the main theorem}
\label{main proof}
We give the proof of our main theorem Theorem~\ref{main}.

Let $k$ be any algebraically closed field; let $A$ be any complete regular local ring such that $A$ contains $k$ as a subring, the residue field $A/M(A)$ is isomorphic to $k$ as $k$-algebras, and $\dim A\geq 2$; let $P$ be any parameter system of $A$, and let $z\in P$ be any element.

Let $A'$ denote the completion of $k[P - \{z\}]$ with respect to the maximal ideal $k[P - \{z\}]\cap M(A)$. The ring $A'$ is a local subring of $A$ and $M(A')=M(A)\cap A' =(P -\{z\})A'$. The completion of $A'[z]$ with respect to the prime ideal $zA'[z]$ is isomorphic to $A$ as $A'[z]$-algebras. The set $P - \{z\}$ is a parameter system of $A'$.

Recall the following notations:
\begin{equation*}\begin{split}
&PW(1)= \{\phi\in A|\phi= u\prod_{\chi\in\mathcal{X}}(z+\chi)^{a(\chi)}\prod_{x\in P-\{z\}} x^{b(x)} \\
&\quad\text{for some }u\in A^\times, \text{ some finite subset } \mathcal{X} \text{ of } M(A'),\\
&\quad \text{some mapping }a: \mathcal{X}\rightarrow \Z_+, \text{ and some mapping }b:P-\{z\}\rightarrow\Z_0\}.\\
\end{split}\end{equation*}
\vfill

For any $h\in\Z_+$ with $h\geq 2$,
\begin{equation*}\begin{split}
W(h)&=\{\phi\in A|\phi= z^h+\sum_{i=0}^{h-1} \phi'(i)z^i \\
&\qquad\quad \text{for some mapping }\phi':\{0,1,\ldots,h-1\}\rightarrow M(A') \text{ satisfying }\\
&\qquad\quad 
\chi^h+\sum_{i=0}^{h-1} \phi'(i)\chi^i\neq 0 \text{ for any }
\chi\in M(A').\},\\
PW(h)&=\{\phi\in A|\phi=\psi\psi'\text{ for some } \psi\in W(h) \text{ and some }\psi'\in PW(1).\},\\
SW(h)&=\{\phi\in A|\phi=\psi\psi'\text{ for some } \psi\in W(h) \text{ and some }\psi'\in PW(1)\text{ such that}\\
&\qquad\quad \Gamma_+(P,\psi) \text{ has no }z \text{-removable faces, and }\Gamma_+(P,\phi) \text{ is }z \text{-simple.}\}.\\
\end{split}\end{equation*}

Consider any $h\in\Z_+$ with $h\geq 2$ and any $\phi\in SW(h)$.

We take an element $\psi\in W(h)$, an element $u\in A^\times$, a finite subset $\mathcal{X}$, a mapping $a: \mathcal{X}\rightarrow \Z_+$ and a mapping $ b:P-\{z\}\rightarrow\Z_0$ satisfying
$$\phi=\psi u\prod_{\chi\in\mathcal{X}}(z+\chi)^{a(\chi)}\prod_{x\in P-\{z\}} x^{b(x)}.$$
The quintuplet $(\psi,u, \mathcal{X},a,b)$ is uniquely determined depending on $\phi$, since $A$ is a unique factorization domain.

We take a mapping $\psi':\{0,1,\ldots,h-1\}\rightarrow M(A')$ satisfying 
$\psi= z^h+\sum_{i=0}^{h-1} \psi'(i)z^i$
and $\omega^h+\sum_{i=0}^{h-1} \psi'(i)\omega^i\neq 0$ for any $\omega\in M(A')$.
$\psi'$ is uniquely determined depending on $\psi$, since Weierstrass' preparation theorem holds.

The Newton polyhedron $\Gamma_+(P,\psi)$ has no $z$-removable faces and $\Gamma_+(P,\phi)$ is $z$-simple.

We denote $V=\Map(P,\R)$, $N=\Map(P,\Z)$, $\bar{\Delta}=\Map(P,\R_0)^\vee|V$ and $S=\Gamma_+(P,\phi)$.
$V$ is a finite dimendional vector space over $\R$, $\dim V=\dim A$, $N$ is a lattice of $V$, $\bar{\Delta}$ is a simplicial cone over $N^*$ in $V^*$, $\bar{\Delta}^\vee|V^*=\Map(P,\R_0)$ is a simplicial cone over $N$ in $V$, $\dim \bar{\Delta}=\dim \bar{\Delta}^\vee|V^*=\dim V=\dim A$ and $S$ is a Newton polyhedron over $N$ in $V$.
$\Stab(S)= \bar{\Delta}^\vee|V^*$, $S\subset\bar{\Delta}^\vee|V^*$,
$|\mathcal{D}(S|V)|= \bar{\Delta}$, $\mathcal{V}(S)\subset(\bar{\Delta}^\vee|V^*)\cap N=\Map(P,\Z_0)$ and $\Den(S/N)=1$.
For any $(H,\mathcal{C})\in\mathcal{HC}(V,N,S)$, $\Ht(H,\mathcal{C}, S)\in\Z_0$.

The set $\{f^P_x|x\in P\}$ is a $\R$-basis of $V$, it is a $\Z$-basis of $N$, 
$\bar{\Delta}^\vee|V^*=\sum_{x\in P}\R_0f^P_x$, the dual basis $\{f^{P\vee}_x|x\in P\}$ of $\{f^P_x|x\in P\}$ is a $\R$-basis of $V^*$, it is a $\Z$-basis of $N^*$,
and $\bar{\Delta}=\sum_{x\in P}\R_0 f^{P\vee}_x$.

We denote $\bar{H}=\R_0 f^{P\vee}_z\in\mathcal{F}(\bar{\Delta})_1$. 
$S$ is $\bar{H}$-simple, $\mathcal{D}(S|V)$ is $\bar{H}$-simple, $(\bar{H},\mathcal{F}(\bar{\Delta})$\hfill\break$)\in \mathcal{HC}(V,N,S)\neq\emptyset$
and $\mathcal{USD}(\bar{H},\mathcal{F}(\bar{\Delta}),S)\neq\emptyset$.

For any $\zeta\in A-\{0\}$, $\Gamma_+(P,\zeta)$ is a Newton polyhedron over $N$ in $V$,
$\Stab(\Gamma_+(P,\zeta$\hfill\break$))= \bar{\Delta}^\vee|V^*$, $\Gamma_+(P,\zeta)\subset\bar{\Delta}^\vee|V^*$,
$|\mathcal{D}(\Gamma_+(P,\zeta)|V)|= \bar{\Delta}$, $\mathcal{V}(\Gamma_+(P,\zeta))\subset(\bar{\Delta}^\vee|V^*)\cap N$ and $\Den(\Gamma_+(P,\zeta)/N)=1$.
For any $\zeta\in A-\{0\}$ such that $\mathcal{D}(\Gamma_+(P,\zeta)|V)$ is $\bar{H}$-simple, we denote the $\bar{H}$-skeleton of $\mathcal{D}(\Gamma_+(P,\zeta)|V)$ by
$\bar{\mathcal{D}}(\Gamma_+(P,\zeta)|V)^1$.

Note that $S= \Gamma_+(P,\psi)+\sum_{\chi\in\mathcal{X}}a(\chi)\Gamma_*(P, z+\chi)
+\sum_{x\in P-\{z\}}b(x)\Gamma_+(P, x)$ and
$\mathcal{D}(S|V)=\mathcal{D}(\Gamma_+(P,\psi)|V)\hat{\cap}(\hat{\cap}_{\chi\in\mathcal{X}-\{0\}}\mathcal{D}(\Gamma_*(P, z+\chi)|V))$.
We know that $\Ht(\bar{H},$\break$\Gamma_+(P,\psi))=h\geq 2$, $\mathcal{D}(\Gamma_+(P,\psi)|V)$ is $\bar{H}$-simple and $c(\Gamma_+(P,\psi))\geq 2$, and we know that $\Ht(\bar{H}, \Gamma_+(P, z+\chi))=1$, $\mathcal{D}(\Gamma_+(P, z+\chi)|V)$ is $\bar{H}$-simple, and $\chi$ has normal crossings over $P-\{z\}$ and $c(\Gamma_+(P, z+\chi))=2$ for any $\chi\in\mathcal{X}-\{0\}$.

Take any $(M, F)\in \mathcal{USD}(\bar{H},\mathcal{F}(\bar{\Delta}),S)$.
We denote $\widetilde{\mathcal{C}}=\mathcal{F}(\bar{\Delta})*F(1)*F(2)*\cdots*F(M)$.
The set $\widetilde{\mathcal{C}}$ is a simplicial cone decomposition over $N^*$ in $V^*$, it is an upward subdivision of $(\bar{H},\mathcal{F}(\bar{\Delta}), S)$, it is a subdivision of $\mathcal{D}(S|V)$ and $|\widetilde{\mathcal{C}}|=\bar{\Delta}=|\mathcal{D}(S|V)|$.

We denote $\bar{\Sigma}=\Sigma(k,V,N, \widetilde{\mathcal{C}})$,
$R=\Map'(N,k)\backslash(\bar{\Delta}^\vee|V^*)$, and
$\bar{\sigma}=\sigma(k,V,N)^{ \widetilde{\mathcal{C}}}_{\mathcal{F}(\bar{\Delta})}:
\Sigma(k,V,N, \widetilde{\mathcal{C}})\rightarrow \Sigma(k,V,N, \mathcal{F}(\bar{\Delta}))$.
The structure sheaf of the scheme $\bar{\Sigma}$ is denoted by $\mathcal{O}(\widetilde{\mathcal{C}})$.
We identify $\Sigma(k,V,N, \mathcal{F}(\bar{\Delta}))$ and $\Spec(R)$ by the isomorphism $\pi(k,V,N, \mathcal{F}(\bar{\Delta}), \Delta)$.
$\bar{\sigma}:\bar{\Sigma}\rightarrow \Spec(R)$.
$R$ is a polynomial ring over $k$ and the set $\{x(f^P_y)|y\in P\}$ is a variable system of $R$ over $k$, where $x:N\rightarrow \Map'(N,k)$ denotes the mapping we defined just before Lemma~\ref{interesting ring}.
We denote $\bar{M}=\{x(f^P_y)|y\in P\}R$ and $\bar{D}=\Spec(R/
\prod_{y\in P}x(f^P_y) R)$. $\bar{M}$ is a maximal ideal of $R$ and $\bar{D}$ is a non-zero effective normal crossing divisor of $\Spec(R)$.
We define a coordinate system $\bar{\xi}_{\bar{M}}:\Comp(\bar{D})\rightarrow 
R$ of the normal crossing scheme $(\Spec(R),\bar{D})$ at $\bar{M}$ by putting
$\bar{\xi}_{\bar{M}}(\Spec(R/x(f^P_y)R))=x(f^P_y)$ for any $y\in P$.
Let $\bar{\xi}=\{\bar{\xi}_{\bar{M}}\}$.
The triplet $(\Spec(R),\bar{D},\bar{\xi})$ is a coordinated normal crossing scheme over $k$.
$\bar{\Sigma}$ is a smooth scheme over $\Spec(R)$, and $\bar{\sigma}$ is an admissible composition of blowing-ups with centers of codimension two over $\bar{D}$.

Let $\iota: R\rightarrow A$ denote the injective homomorphism of $k$-algebras satisfying $\iota(x(f^P_y))=y$ for any $y\in P$.
Let $\Sigma=\bar{\Sigma}\times_{\Spec(R)}\Spec(A)$ denote the fiber product scheme; let $\sigma:\Sigma\rightarrow\Spec(A)$ denote the projection induced by $\sigma$ and let $\tau:\Sigma\rightarrow\bar{\Sigma}$ denote the projection induced by $\iota^*:\Spec(A)\rightarrow \Spec(R)$.
The structure sheaf of the scheme $\Sigma$ is denoted by $\mathcal{O}_\Sigma$.
We denote $D=\Spec(A/ \prod_{y\in P}y A)$. $D$ is a non-zero effective normal crossing divisor of $\Spec(A)$.
We define a coordinate system $\xi_{M(A)}:\Comp(\Lambda)\rightarrow 
A$ of the normal crossing scheme $(\Spec(A),\Lambda)$ at $M(A)$ by putting
$\xi_{M(A)}(\Spec(A/yA))=y$ for any $y\in P$.
Let $\xi=\{\xi_{M(A)}\}$.
The triplet $(\Spec(A),D,\xi)$ is a coordinated normal crossing scheme over $k$.
$\Sigma$ is a smooth scheme over $\Spec(A)$, and $\sigma$ is an admissible composition of blowing-ups with centers of codimension two over $D$.

Note that $\iota^*(M(A))=\bar{M}$, $\iota^{*-1}(\bar{M})=\{M(A)\}$, and $\iota$ induces an isomorphism $R/\bar{M}\rightarrow A/M(A)$ of $k$-algebras.
Therefore, we know that $\tau$ induces an isomorphism 
$\tau:\Sigma\times_{\Spec(A)}\Spec(A/M(A))\rightarrow 
\bar{\Sigma}\times_{\Spec(R)}\Spec(R/\bar{M})$ of schemes.

$\sigma^{-1}(M(A))=\Sigma\times_{\Spec(A)}\Spec(A/M(A))$.

Since $\Spec(R/\bar{M})=\{\bar{M}\}=V^\circ(\mathcal{F}(\bar{\Delta}),\bar{\Delta})$,
\begin{equation*}\begin{split}
\bar{\sigma}^{-1}(\bar{M})&=\bar{\Sigma}\times_{\Spec(R)}\Spec(R/\bar{M})\\
&=\bar{\sigma}^{-1}( V^\circ(\mathcal{F}(\bar{\Delta}),\bar{\Delta}))
=\bigcup_{\Theta\in\widetilde{\mathcal{C}},\Theta^\circ\subset\bar{\Delta}^\circ}V^\circ(\widetilde{\mathcal{C}},\Theta)
\end{split}\end{equation*}
by Theorem~\ref{subdivision morphism}.14.

Consider any $\Theta\in \widetilde{\mathcal{C}}^0$.
$\dim \Theta=\dim V$, the set $\{b_{\Gamma/N^*}|\Gamma\in\mathcal{F}(\Theta)_1\}$ 
is a $\R$-basis of $V^*$ and it is a $\Z$-basis of $N^*$.
Let $\{{b_{\Gamma/N^*}}^\vee_\Theta|\Gamma\in\mathcal{F}(\Theta)_1\}$ denote the dual basis of
$\{b_{\Gamma/N^*}|\Gamma\in\mathcal{F}(\Theta)_1\}$.
The set  $\{{b_{\Gamma/N^*}}^\vee_\Theta|\Gamma\in\mathcal{F}(\Theta)_1\}$ is 
a $\R$-basis of $V$ and it is a $\Z$-basis of $N$.
We denote $R(\Theta)= \Map'(N,k)\backslash(\Theta^\vee|V^*)$.
$R(\Theta)$ is a polynomial ring over $k$ and 
the set $\{x({b_{\Gamma/N^*}}^\vee_\Theta) |\Gamma\in\mathcal{F}(\Theta)_1\}$
is a variable system of $R(\Theta)$ over $k$.
The morphism $\pi(\widetilde{\mathcal{C}},\Theta): \Spec(R(\Theta))\rightarrow U(\widetilde{\mathcal{C}},\Theta)$ is an isomorphism of schemes.

We consider the coordinated normal crossing scheme $(\bar{\Sigma},\bar{\sigma}^*\bar{D},\bar{\sigma}^*\bar{\xi})$.
By Lemma \ref{relation to ncs} we know that $\Comp(\bar{\sigma}^*\bar{D})=
\{V(\widetilde{\mathcal{C}},\Gamma)| \Gamma\in\widetilde{\mathcal{C}}_1\}$,
$(\bar{\sigma}^*\bar{D})_0=\cup_{\Theta\in\widetilde{\mathcal{C}}^0}
V^\circ(\widetilde{\mathcal{C}},\Theta)$, 
and if $\bar{\B}\in (\bar{\sigma}^*\bar{D})_0$, $\Theta\in \widetilde{\mathcal{C}}^0$ and $\{\bar{\B}\}= V^\circ(\widetilde{\mathcal{C}},\Theta)$, then
$U(\bar{\Sigma},\bar{\sigma}^*\bar{D},\bar{\B})=U(\widetilde{\mathcal{C}},\Theta)$
$\Comp(\bar{\sigma}^*\bar{D})(\bar{\B})=
\{V(\widetilde{\mathcal{C}},\Gamma)|\Gamma\in\mathcal{F}(\Theta)_1\}$ and
$(\bar{\sigma}^*\bar{\xi})_{\bar{\B}}(V(\widetilde{\mathcal{C}},\Gamma))=
\pi(\widetilde{\mathcal{C}},\Theta)^*( U(\bar{\Sigma},\bar{\sigma}^*\bar{D},$\break$\bar{\B}))^{-1}(x({b_{\Gamma/N^*}}^\vee_\Theta))$
for any $\Gamma\in\mathcal{F}(\Theta)_1$.

We consider the coordinated normal crossing scheme $(\Sigma,\sigma^*D,\sigma^*\xi)$.
We know that 
$\tau^*V(\widetilde{\mathcal{C}},\Gamma)$ is a prime divisor of $\Sigma$ for any $\Gamma\in\widetilde{\mathcal{C}}_1$,
$\Comp(\sigma^*D)=
\{\tau^*V(\widetilde{\mathcal{C}},\Gamma)| \Gamma\in\widetilde{\mathcal{C}}_1\}$,
$(\sigma^*D)_0=\tau^{-1}((\bar{\sigma}^*\bar{D})_0) =\cup_{\Theta\in\widetilde{\mathcal{C}}^0}
\tau^{-1}(V^\circ(\widetilde{\mathcal{C}},\Theta))$, and
$\tau$ induces a bijective morphism $\tau: (\sigma^*D)_0\rightarrow (\bar{\sigma}^*\bar{D})_0$. We know that 
$\tau^{-1}(U(\bar{\Sigma},\bar{\sigma}^*\bar{D},\tau(\B))$ is an affine open subset of $\Sigma$,
$U(\Sigma,\sigma^*D,\B)=\tau^{-1}(U(\bar{\Sigma},\bar{\sigma}^*\bar{D},\tau(\B)))$, and $\Comp(\sigma^*D)(\B)=
\{\tau^*E|E\in\Comp(\bar{\sigma}^*\bar{D})(\tau(\B))\}$
for any  $\B\in (\sigma^*D)_0$.

Consider any $\B\in (\sigma^*D)_0$.
$(\sigma^*\xi)_\B: \Comp(\sigma^*D)(\B)\rightarrow \mathcal{O}_\Sigma(U(\Sigma,\sigma^*D,\B))$, 
$(\bar{\sigma}^*\bar{\xi})_{\tau(\B)}: \Comp(\bar{\sigma}^*\bar{D})(\tau(\B))\rightarrow \mathcal{O}(\widetilde{\mathcal{C}})( U(\bar{\Sigma},\bar{\sigma}^*\bar{D},\tau(\B)))$,
$\tau^*(U(\bar{\Sigma},\bar{\sigma}^*\bar{D},\tau(\B))): \mathcal{O}(\widetilde{\mathcal{C}})( U(\bar{\Sigma},\bar{\sigma}^*\bar{D},\tau(\B)))\rightarrow \mathcal{O}_\Sigma(U(\Sigma, \sigma^*D,\B))$,
the homomorphism $\tau^*:\Div(\bar{\Sigma})\rightarrow \Div(\Sigma)$ induces a mapping
$\tau^*:\Comp(\bar{\sigma}^*\bar{D})(\tau(\B))\rightarrow \Comp(\sigma^*D)(\B)$
and we know that $(\sigma^*\xi)_\B\tau^*=
\tau^*(U(\bar{\Sigma},\bar{\sigma}^*\bar{D},\tau(\B))) (\bar{\sigma}^*\bar{\xi})_{\tau(\B)}$.

Consider any closed point $\A\in \Sigma$ with $\sigma(\A)=M(A)$.

$\tau(\A)\in\bar{\Sigma}$,
$\tau(\A)$ is a closed point of $\bar{\Sigma}$,
$\bar{\sigma}\tau(\A)=\bar{M}$ and
$\tau(\A)\in \bar{\sigma}^{-1}(\bar{M})=\cup_{\Theta\in\widetilde{\mathcal{C}},\Theta^\circ\subset\bar{\Delta}^\circ}V^\circ(\widetilde{\mathcal{C}},\Theta)$.
We take $\Theta\in\widetilde{\mathcal{C}}$ satisfying $\Theta^\circ\subset\bar{\Delta}^\circ$ and $\tau(\A)\in V^\circ(\widetilde{\mathcal{C}},\Theta)$.

$\Theta\not\subset \bar{H}\Op|\bar{\Delta}$. 
We take the unique element $\Lambda\in\mathcal{D}(S|V)$ with $\Theta^\circ\subset\Lambda^\circ$ and we take the unique element $\bar{A}\in\mathcal{F}(S+(\Theta^\vee|V))$ with $\Delta(\bar{A}, S+(\Theta^\vee|V))=\Theta$.

We denote 
\begin{equation*}\begin{split}
\widetilde{\mathcal{C}}^\circ_1&=\{\Gamma\in\widetilde{\mathcal{C}}_1|\Gamma\not\subset\bar{H}\Op|\bar{\Delta}\},\\
\mathcal{A}&=\mathcal{A}(V,N,\bar{H},\mathcal{F}(\bar{\Delta}),S,M,F): \widetilde{\mathcal{C}}^\circ_1\rightarrow 2^{2^{V^*}},\text{ and}\\
\mathcal{A}^\circ&=\mathcal{A}^\circ(V,N,\bar{H},\mathcal{F}(\bar{\Delta}),S,M,F): \widetilde{\mathcal{C}}^\circ_1\rightarrow 2^{2^{V^*}}.
\end{split}\end{equation*}

We take the unique element $\hat{\Gamma}\in\widetilde{\mathcal{C}}^\circ_1$ with
$\Theta\in\mathcal{A}^\circ(\hat{\Gamma})$.
Note that $\mathcal{A}^\circ(\hat{\Gamma})\subset\mathcal{A}(\hat{\Gamma})\subset
\widetilde{\mathcal{C}}$,
$\{\Theta,\hat{\Gamma}\}\subset\mathcal{A}(\hat{\Gamma})$ and $\Theta+\hat{\Gamma}\in\mathcal{A}(\hat{\Gamma})/\hat{\Gamma}$.
We take any element $\hat{\Theta}\in\mathcal{A}(\hat{\Gamma})^0/\hat{\Gamma}$ with 
$\Theta+\hat{\Gamma}\subset\hat{\Theta}$.
We have $\Theta\subset\Theta+\hat{\Gamma}\subset\hat{\Theta}$, $\Theta\in\mathcal{F}(\hat{\Theta})$,
$\hat{\Gamma}\in\mathcal{F}(\hat{\Theta})_1$ and 
$\tau(\A)\in V^\circ(\widetilde{\mathcal{C}},\Theta)\subset
U(\widetilde{\mathcal{C}},\Theta)\subset
U(\widetilde{\mathcal{C}},\hat{\Theta})$.
We take the unique element $\hat{\Lambda}\in\mathcal{D}(S|V)$ with $\hat{\Theta}^\circ\subset\hat{\Lambda}^\circ$.
$\dim\hat{\Theta}=\dim\hat{\Lambda}=\dim V$.
$\hat{\Lambda}\in \mathcal{D}(S|V)^0$.
Since $\emptyset\neq\Theta^\circ\subset\hat{\Lambda}\cap\Lambda^\circ$, we know that $\Lambda\in\mathcal{F}(\hat{\Lambda})$.
We take the unique element $\hat{a}\in\mathcal{V}(S)$ with $\hat{\Lambda}=\Delta(\{\hat{a}\}, S|V)$.
We know that $S+(\hat{\Theta}^\vee|V^*)=\{\hat{a}\}+(\hat{\Theta}^\vee|V^*)$.
Let $\bar{\B}$ be the unique point in $V^\circ(\widetilde{\mathcal{C}},\hat{\Theta})$.
$\bar{\B}\in V^\circ(\widetilde{\mathcal{C}},\hat{\Theta})\subset U(\widetilde{\mathcal{C}},\hat{\Theta})$.
$\bar{\B}\in(\bar{\sigma}^*\bar{D})_0$.
$\{\tau(\A),\bar{\B}\}\subset U(\widetilde{\mathcal{C}},\hat{\Theta})=
U(\bar{\Sigma},\bar{\sigma}^*\bar{D},\bar{\B})$.

Using the isomorphism $\pi(\widetilde{\mathcal{C}},\hat{\Theta}): \Spec(R(\hat{\Theta}))\rightarrow U(\widetilde{\mathcal{C}},\hat{\Theta})$, we identify schemes $\Spec(R(\hat{\Theta}))$ and
$U(\widetilde{\mathcal{C}},\hat{\Theta})$.
We know that
$\{x({b_{\Gamma/N^*}}^\vee_{\hat{\Theta}})|\Gamma\in\mathcal{F}(\hat{\Theta})_1\}\subset R(\hat{\Theta})=\mathcal{O}_{R(\hat{\Theta})}(\Spec(R(\hat{\Theta})))=\mathcal{O}(\widetilde{\mathcal{C}})(U(\widetilde{\mathcal{C}},\hat{\Theta}))$,
$x({b_{\Gamma/N^*}}^\vee_{\hat{\Theta}})\tau(\A)=0$ for any $\Gamma\in\mathcal{F}(\Theta)_1$ and
$x({b_{\Gamma/N^*}}^\vee_{\hat{\Theta}})\tau(\A)\neq0$ for any $\Gamma\in\mathcal{F}(\hat{\Theta})_1-\mathcal{F}(\Theta)_1$.
Since $\tau(\A)$ is a closed point of $U(\widetilde{\mathcal{C}},\hat{\Theta})$
and $k$ is algebraically closed, $x({b_{\Gamma/N^*}}^\vee_{\hat{\Theta}})\tau(\A)\in k$ for any 
$\Gamma\in\mathcal{F}(\hat{\Theta})_1$.
$x({b_{\Gamma/N^*}}^\vee_{\hat{\Theta}})(\bar{\B})=0$ for any $\Gamma\in\mathcal{F}(\hat{\Theta})_1$.

Since $\bar{\sigma}\tau(\A)=\bar{M}$, we have a homomorphism of $k$-algebras
$\bar{\sigma}^*(\tau(\A)):R_{\bar{M}}\rightarrow\mathcal{O}(\widetilde{\mathcal{C}})_{\tau(\A)}$ satisfying $\bar{\sigma}^*(\tau(\A))(M(R_{\bar{M}}))\subset M(\mathcal{O}(\widetilde{\mathcal{C}})_{\tau(\A)})$. 
This homomorphism has the unique extension 
$\bar{\sigma}^*(\tau(\A)):R_{\bar{M}}^c\rightarrow\mathcal{O}(\widetilde{\mathcal{C}})_{\tau(\A)}^c$,
where the superscript ${}^c$ denotes the completion of a noetherian local ring.
The set $\{x(f^P_y)|y\in P\}$ is a parameter system of $R_{\bar{M}}^c$.
Let $\hat{P}=\{x({b_{\Gamma/N^*}}^\vee_{\hat{\Theta}})- x({b_{\Gamma/N^*}}^\vee_{\hat{\Theta}})\tau(\A)|\Gamma\in\mathcal{F}(\hat{\Theta})_1\}$.
This set $\hat{P}$ is a parameter system of $\mathcal{O}(\widetilde{\mathcal{C}})_{\tau(\A)}^c$.
$\hat{P}=\{(\bar{\sigma}^*\bar{\xi})_{\bar{\B}}(E)- (\bar{\sigma}^*\bar{\xi})_{\bar{\B}}(E)\tau(\A)|E\in\Comp(\bar{\sigma}^*\bar{D})(\bar{\B})\}$.
Recall that $V=\Map(P,\R)$.
Note that there exists uniquely an isomorphism of vector spaces over $\R$ from $\Map(\hat{P},\R)$ to $V$ sending\hfill\break $f^{\hat{P}}_{ x({b_{\Gamma/N^*}}^\vee_{\hat{\Theta}})- x({b_{\Gamma/N^*}}^\vee_{\hat{\Theta}})\tau(\A)}\in\Map(\hat{P},\R)$ to ${b_{\Gamma/N^*}}^\vee_{\hat{\Theta}}\in V$ for any $\Gamma\in\mathcal{F}(\hat{\Theta})_1$.
Using this isomorphism, we identify $\Map(\hat{P},\R)$ and $V$.
Pairs $\Map(\hat{P},\Z)$ and $N$, $\Map(\hat{P},\R_0)$ and $\hat{\Theta}^\vee|V^*$ are identified.

Since $\hat{\Theta}^\circ\subset\bar{\Delta}^\circ$, $\bar{\sigma}(\bar{\B})=\bar{M}$, we have a homomorphism of $k$-algebras
$\bar{\sigma}^*(\bar{\B}):R_{\bar{M}}\rightarrow\mathcal{O}(\widetilde{\mathcal{C}})_{\bar{\B}}$ satisfying $\bar{\sigma}^*(\bar{\B})(M(R_{\bar{M}}))\subset M(\mathcal{O}(\widetilde{\mathcal{C}})_{\bar{\B}})$. 
This homomorphism has the unique extension 
$\bar{\sigma}^*(\bar{\B}):R_{\bar{M}}^c\rightarrow\mathcal{O}(\widetilde{\mathcal{C}})_{\bar{\B}}^c$.
Let $\hat{P}_{\bar{\B}}=\{x({b_{\Gamma/N^*}}^\vee_{\hat{\Theta}})|\Gamma\in\mathcal{F}(\hat{\Theta})_1\}$.
This set $\hat{P}_{\bar{\B}}$ is a parameter system of $\mathcal{O}(\widetilde{\mathcal{C}})_{\bar{\B}}^c$.
$\hat{P}_{\bar{\B}}=\{(\bar{\sigma}^*\bar{\xi})_{\bar{\B}}(E)|E\in\Comp(\bar{\sigma}^*\bar{D})(\bar{\B})\}$.
Note that there exists uniquely an isomorphism of vector spaces over $\R$ from $\Map(\hat{P}_{\bar{\B}},\R)$ to $V$ sending $f^{\hat{P}_{\bar{\B}}}_{x({b_{\Gamma/N^*}}^\vee_{\hat{\Theta}})}\in\Map(\hat{P}_{\bar{\B}},\R)$ to ${b_{\Gamma/N^*}}^\vee_{\hat{\Theta}}\in V$ for any $\Gamma\in\mathcal{F}(\hat{\Theta})_1$.
Using this isomorphism, we identify $\Map(\hat{P}_{\bar{\B}},\R)$ and $V$.

Note that the homomorphism of $k$-algebras $\iota:R\rightarrow A$ induces an isomorphism $R_{\bar{M}}^c\rightarrow A$ of complete $k$-algebras. 
By this isomorphism we identify  $R_{\bar{M}}^c$ and $A$.
$\phi\in A= R_{\bar{M}}^c$. $P=\{x(f^P_y)|y\in P\}$.

Consider any $\omega\in\Theta$. It is easy to see that
\begin{equation*}\begin{split}
&\Ord(\hat{P},\omega, x({b_{\Gamma/N^*}}^\vee_{\hat{\Theta}}))=\Ord(\hat{P}_{\bar{\B}},\omega, x({b_{\Gamma/N^*}}^\vee_{\hat{\Theta}}))=0
\text{ for any }\Gamma\in\mathcal{F}(\hat{\Theta})_1-\mathcal{F}(\Theta)_1, \\
&\Ord(\hat{P},\omega, x({b_{\Gamma/N^*}}^\vee_{\hat{\Theta}}))=\Ord(\hat{P}_{\bar{\B}},\omega, x({b_{\Gamma/N^*}}^\vee_{\hat{\Theta}}))
\text{ for any }\Gamma\in\mathcal{F}(\Theta)_1\text{, and}\\
&\Ord(\hat{P},\omega,\bar{\sigma}^*(\tau(\A))(\zeta))= \Ord(\hat{P}_{\bar{\B}},\omega,\bar{\sigma}^*(\bar{\B})(\zeta))=\Ord(P,\omega,\zeta),\\
&\In(\hat{P},\omega,\bar{\sigma}^*(\tau(\A))(\zeta))= \bar{\sigma}^*(\tau(\A))(\In(P,\omega,\zeta))\text{, and}\\
&\In(\hat{P}_{\bar{\B}},\omega,\bar{\sigma}^*(\bar{\B})(\zeta))=\bar{\sigma}^*(\bar{\B})(\In(P,\omega,\zeta)),
\end{split}\end{equation*}
for any $\zeta\in R_{\bar{M}}^c=A$.

By the above
$\Ord(\hat{P},b_{\Theta/N^*},\bar{\sigma}^*(\tau(\A))(\phi))=
\Ord(\hat{P}_{\bar{\B}},b_{\Theta/N^*},\bar{\sigma}^*(\bar{\B})(\phi))=
\Ord(P, $\hfill\break$b_{\Theta/N^*},\phi)$, $\In(\hat{P},b_{\Theta/N^*},\bar{\sigma}^*(\tau(\A))(\phi))= \bar{\sigma}^*(\tau(\A))(\In(P, b_{\Theta/N^*},\phi))$, and
$\In(\hat{P}_{\bar{\B}},b_{\Theta/N^*},$\hfill\break$\bar{\sigma}^*(\bar{\B})(\phi))= \bar{\sigma}^*(\bar{\B})(\In(P, b_{\Theta/N^*},\phi))$.

Since $\Gamma_+(\hat{P}_{\bar{\B}}, \bar{\sigma}^*\phi)=S+(\hat{\Theta}^\vee|V^*)=\{\hat{a}\}+(\hat{\Theta}^\vee|V^*)$,
$\Ord(\hat{P}_{\bar{\B}}, b_{\Gamma/N^*}, \In(\hat{P}_{\bar{\B}},b_{\Theta/N^*},$\hfill\break$\bar{\sigma}^*(\bar{\B})(\phi)))
=\langle b_{\Gamma/N^*}, \hat{a}\rangle=\Ord(\hat{P}_{\bar{\B}}, b_{\Gamma/N^*},\bar{\sigma}^*(\bar{\B})(\phi))
=\Ord(\hat{P}, b_{\Gamma/N^*},\bar{\sigma}^*(\tau(\A))(\phi))$ for any $\Gamma\in\mathcal{F}(\Theta)_1$.

Since $b_{\Theta/N^*}\in\Theta^\circ\subset\bar{\Delta}^\circ$,
$\In(P, b_{\Theta/N^*},\phi)\in R$ and
$\bar{\sigma}^*(\tau(\A))(\In(P, b_{\Theta/N^*},\phi))=
\Res^{\bar{\Sigma}}_{U(\widetilde{\mathcal{C}}, \hat{\Theta})}\bar{\sigma}^*(\Spec(R))(\In(P, b_{\Theta/N^*},\phi))=
\bar{\sigma}^*(\bar{\B})(\In(P, b_{\Theta/N^*},\phi))
\in\mathcal{O}(\widetilde{\mathcal{C}})( U(\widetilde{\mathcal{C}}, \hat{\Theta}))$.
We know $\In(\hat{P},b_{\Theta/N^*},\bar{\sigma}^*(\tau(\A))(\phi))=
\In(\hat{P}_{\bar{\B}}, b_{\Theta/N^*},\bar{\sigma}^*(\bar{\B})(\phi))$.
Therefore, $\Ord(\hat{P}, $\hfill\break$b_{\Gamma/N^*}, \In(\hat{P},b_{\Theta/N^*},\bar{\sigma}^*(\tau(\A))(\phi)))
=\Ord(\hat{P}_{\bar{\B}}, b_{\Gamma/N^*}, \In(\hat{P}_{\bar{\B}},b_{\Theta/N^*},\bar{\sigma}^*(\bar{\B})(\phi)))$ for any $\Gamma\in\mathcal{F}(\Theta)_1$.

We conclude that 
\begin{equation*}\begin{split}
&\Ord(\hat{P},b_{\Theta/N^*},\bar{\sigma}^*(\tau(\A))(\phi))=
\Ord(P, b_{\Theta/N^*},\phi),\\
&\In(\hat{P},b_{\Theta/N^*},\bar{\sigma}^*(\tau(\A))(\phi))= \bar{\sigma}^*(\tau(\A))(\In(P, b_{\Theta/N^*},\phi))\text{, and}\\
&\Ord(\hat{P}, b_{\Gamma/N^*}, \In(\hat{P},b_{\Theta/N^*},\bar{\sigma}^*(\tau(\A))(\phi)))
=\Ord(\hat{P}, b_{\Gamma/N^*},\bar{\sigma}^*(\tau(\A))(\phi))
\end{split}\end{equation*}
for any $\Gamma\in\mathcal{F}(\Theta)_1$.

Note that $b_{\Theta/N^*}\in\Theta^\circ\subset\Lambda^\circ\subset\bar{\Delta}^\circ$,
$\bar{A}\in\mathcal{F}(S+(\Theta^\vee|V^*))$, $\Theta=\Delta(\bar{A}, S+(\Theta^\vee|V^*)|V) \in\mathcal{D}(S+(\Theta^\vee|V^*)|V)$, 
$S+(\Theta^\vee|V^*)\supset S$, $\Lambda=\Delta(\bar{A}\cap S, S|V) \in\mathcal{D}(S|V)$ and $\In(P, b_{\Theta/N^*},\phi)=\Ps(P, \bar{A}\cap S,\phi)$.

Since $\mathcal{D}(S|V)$ is $\bar{H}$-simple and $\dim|\mathcal{D}(S|V)|=\dim\bar{\Delta}=\dim V$, $\dim\Lambda=\dim V$ or $\dim\Lambda=\dim V-1$.

We consider the case $\dim\Lambda=\dim V$.

Since $\Lambda\in\mathcal{F}(\hat{\Lambda})$ and $\dim\Lambda=\dim V=\dim\hat{\Lambda}$, we know that $\Delta(\bar{A}\cap S,S|V)=\Lambda=\hat{\Lambda}=\Delta(\{\hat{a}\}, S|V)$, $\bar{A}\cap S=\{\hat{a}\}$,
and $\Ps(P, \bar{A}\cap S,\phi)=cx(\hat{a})$ for some $c\in k-\{0\}$.
We take $c\in k-\{0\}$ with $\Ps(P, \bar{A}\cap S,\phi)=cx(\hat{a})$.
We know that $\In(P, b_{\Theta/N^*},\phi)=\Ps(P, \bar{A}\cap S,\phi) =cx(\hat{a})$ and 
\begin{equation*}\begin{split}
&\In(\hat{P},b_{\Theta/N^*},\bar{\sigma}^*(\tau(\A))(\phi))= \bar{\sigma}^*(\tau(\A))(\In(P, b_{\Theta/N^*},\phi))=
\bar{\sigma}^*(\tau(\A))( cx(\hat{a}))\\
=\:&cx(\hat{a})=
c\prod_{\Gamma\in\mathcal{F}(\hat{\Theta})}x({b_{\Gamma/N^*}}^\vee_{\hat{\Theta}})^{\langle b_{\Gamma/N^*},\hat{a}\rangle}.
\end{split}\end{equation*}

Let $\hat{b}=\sum_{\Gamma\in\mathcal{F}(\Theta)_1}\langle b_{\Gamma/N^*},\hat{a}\rangle {b_{\Gamma/N^*}}^\vee_{\hat{\Theta}}\in(\hat{\Theta}^\vee|V^*)\cap N$.
Since $\hat{P}=\{x({b_{\Gamma/N^*}}^\vee_{\hat{\Theta}})|\Gamma\in\mathcal{F}(\Theta)_1\}\cup
\{x({b_{\Gamma/N^*}}^\vee_{\hat{\Theta}})- x({b_{\Gamma/N^*}}^\vee_{\hat{\Theta}})(\tau(\A))|\Gamma\in\mathcal{F}(\hat{\Theta})_1- \mathcal{F}(\Theta)_1\}$ and 
$ x({b_{\Gamma/N^*}}^\vee_{\hat{\Theta}})(\tau(\A$\hfill\break$))\neq 0$ for any 
$\Gamma\in\mathcal{F}(\hat{\Theta})_1- \mathcal{F}(\Theta)_1$, we know that
$\Delta(b_{\Theta/N^*},\Gamma_+(\hat{P}, \bar{\sigma}^*(\tau(\A))(\phi))|V) +(\hat{\Theta}^\vee|V^*)=
\Gamma_+(\hat{P}, \In(\hat{P},b_{\Theta/N^*},\bar{\sigma}^*(\tau(\A))(\phi)))=
\Supp(\In(\hat{P},b_{\Theta/N^*},\bar{\sigma}^*(\tau(\A))(\phi)))+(\hat{\Theta}^\vee|V^*)=\{\hat{b}\}+(\hat{\Theta}^\vee|V^*)$ and
$\hat{b}\in\mathcal{V}(\Gamma_+(\hat{P}, \In(\hat{P},b_{\Theta/N^*},\bar{\sigma}^*(\tau(\A))(\phi))))\subset
\mathcal{V}(\Gamma_+(\hat{P},$\hfill\break$\bar{\sigma}^*(\tau(\A))(\phi)))$.

For any $\Gamma\in\mathcal{F}(\Theta)_1$, we have
$\Ord(\hat{P}, b_{\Gamma/N^*},\bar{\sigma}^*(\tau(\A))(\phi))=
\Ord(\hat{P}, b_{\Gamma/N^*}, \In(\hat{P},$\hfill\break$b_{\Theta/N^*},\bar{\sigma}^*(\tau(\A))(\phi)))=
\Ord(\hat{P}, b_{\Gamma/N^*}, cx(\hat{a}))=
\langle b_{\Gamma/N^*},\hat{b}\rangle$.

For any $\Gamma\in\mathcal{F}(\hat{\Theta})_1- \mathcal{F}(\Theta)_1$,
we have
$0\leq\Ord(\hat{P}, b_{\Gamma/N^*},\bar{\sigma}^*(\tau(\A))(\phi))\leq
\Ord(\hat{P}, $\hfill\break$b_{\Gamma/N^*}, \In(\hat{P},b_{\Theta/N^*},\bar{\sigma}^*(\tau(\A))(\phi)))=
\Ord(\hat{P}, b_{\Gamma/N^*}, cx(\hat{a}))=
0$ and 
$\Ord(\hat{P}, b_{\Gamma/N^*},$\hfill\break$\bar{\sigma}^*(\tau(\A))(\phi))=0=
\langle b_{\Gamma/N^*},\hat{b}\rangle $.

We know that $\Gamma_+(\hat{P},\bar{\sigma}^*(\tau(\A))(\phi))=\{\hat{b}\}+(\hat{\Theta}^\vee|V^*)$ and $\bar{\sigma}^*(\tau(\A))(\phi)$ has normal crossings over $\hat{P}$.

We know that there exists uniquely an element $\B\in(\sigma^*D)_0$ with $\tau(\B)=\bar{\B}$. We take the unique element $\B\in(\sigma^*D)_0$ with $\tau(\B)=\bar{\B}$.
Since $\tau(\A)\in U(\bar{\Sigma},\bar{\sigma}^*\bar{D},\bar{\B})$, $\A\in\tau^{-1}( U(\bar{\Sigma},\bar{\sigma}^*\bar{D},\bar{\B}))=
U(\Sigma, \sigma^*D,\B)$.
It is easy to see that the homomorphism $\tau^*(\A): \mathcal{O}(\widetilde{\mathcal{C}})_{\tau(\A)}\rightarrow\mathcal{O}_{\Sigma, \A}$ induces an isomorphism $\tau^*(\A):\mathcal{O}(\widetilde{\mathcal{C}})_{\tau(\A)}^c\rightarrow\mathcal{O}_{\Sigma, \A}^c$ of complete $k$-algebras.
Since $\sigma(\A)=M(A)$, we have a homomorphism of $k$-algebras $\sigma^*(\A):A\rightarrow\mathcal{O}_{\Sigma, \A}$ satisfying
$\sigma^*(\A)(M(A))\subset M(\mathcal{O}_{\Sigma, \A})$. This homomorphism has the unique extension $\sigma^*(\A):A \rightarrow\mathcal{O}_{\Sigma, \A}^c$.
Since $\iota^*\sigma=\bar{\sigma}\tau$, we have $\sigma^*(\A)(\phi)=\tau^*(\A) \bar{\sigma}^*(\tau(\A))(\phi)$.

Let $\bar{P}=\{(\sigma^*\xi)_\B(E)- (\sigma^*\xi)_\B(E)(\A)|E\in\Comp(\sigma^*D)(\B)\}$.
$\bar{P}$ is a parameter system of $\mathcal{O}_{\Sigma, \A}^c$.
Since $\hat{P}=\{(\bar{\sigma}^*\bar{\xi})_{\bar{\B}}(E)- (\bar{\sigma}^*\bar{\xi})_{\bar{\B}}(E)(\tau(\A))|E\in\Comp(\bar{\sigma}^*\bar{D})(\bar{\B})\}$, $\bar{P}=\tau^*(\A)(\hat{P})$.
We know that $\sigma^*(\A)(\phi)$ has normal crossings over $\bar{P}$.

We take any isomorphism $\rho: \mathcal{O}_{\Sigma, \A}^c\rightarrow A$ of complete $k$-algebras satisfying $\rho(\bar{P})=P$ and we put $g=1\in\Z_+$.
Since $h\geq 2$, $g<h$.
$\rho\sigma^*(\A)(\phi)\in A$ and $\rho\sigma^*(\A)(\phi)$ has normal crossings over $P$.
Therefore, $\rho\sigma^*(\A)(\phi)\in PW(1)=PW(g)$.

We conclude that Theorem~\ref{main} holds, if $\dim\Lambda=\dim V$.

We consider the case $\dim\Lambda=\dim V-1$.

We take any isomorphism $\bar{\rho}:\mathcal{O}(\widetilde{\mathcal{C}})_{\tau(\A)}^c\rightarrow A$ 
of $k$-algebras satisfying $\bar{\rho}(\hat{P})=P$ and $\bar{\rho}
(x({b_{\hat{\Gamma}/N^*}}^\vee_{\hat{\Theta}})- x({b_{\hat{\Gamma}/N^*}}^\vee_{\hat{\Theta}})(\tau(\A)))=z$.

Recall that 
$\phi=\psi u\prod_{\chi\in\mathcal{X}}(z+\chi)^{a(\chi)}\prod_{x\in P-\{z\}} x^{b(x)}$.

We consider the element $u\prod_{x\in P-\{z\}} x^{b(x)}\in A$.

$u\in A^\times=(R_{\bar{M}}^c)^\times$ and $\bar{\sigma}^*(\tau(\A))(u)\in(\mathcal{O}(\widetilde{\mathcal{C}})_{\tau(\A)}^c)^\times$.
$y=x(f^P_y)$ for any $y\in P$.
$$
\prod_{x\in P-\{z\}} x^{b(x)}=
\prod_{y\in P-\{z\}} y^{b(y)}=
\prod_{y\in P-\{z\}}x(f_y^P)^{b(y)}=
x(\sum_{ y\in P-\{z\}} b(y) f_y^P).
$$

We denote $\hat{a}_{00}=\sum_{ y\in P-\{z\}} b(y) f_y^P \in(\bar{\Delta}^\vee|V^*)\cap N\subset
(\hat{\Theta}^\vee|V^*)\cap N$.
We have $\prod_{x\in P-\{z\}} x^{b(x)}=x(\hat{a}_{00})$.

\begin{equation*}\begin{split}
&\bar{\sigma}^*(\tau(\A))(u\prod_{x\in P-\{z\}} x^{b(x)})= \bar{\sigma}^*(\tau(\A))(u) \bar{\sigma}^*(\tau(\A))( \prod_{x\in P-\{z\}} x^{b(x)})\\
=\:&
\bar{\sigma}^*(\tau(\A))(u) \bar{\sigma}^*(\tau(\A))( x(\hat{a}_{00}))=
\bar{\sigma}^*(\tau(\A))(u)x(\hat{a}_{00})\\
=\:&
\bar{\sigma}^*(\tau(\A))(u)\prod_{\Gamma\in\mathcal{F}(\hat{\Theta})_1}
x({b_{\Gamma/N^*}}^\vee_{\hat{\Theta}})^{\langle b_{\Gamma/N^*},\hat{a}_{00}\rangle}.
\end{split}\end{equation*}

We know that $\bar{\sigma}^*(\tau(\A))(u)\prod_{\Gamma\in\mathcal{F}(\hat{\Theta})_1-\mathcal{F}(\Theta)_1}
x({b_{\Gamma/N^*}}^\vee_{\hat{\Theta}})^{\langle b_{\Gamma/N^*},\hat{a}_{00}\rangle}\in(\mathcal{O}(\widetilde{\mathcal{C}})_{\tau(\A)}^c)^\times$ and $\bar{\sigma}^*(\tau(\A))(u\prod_{x\in P-\{z\}} x^{b(x)})$ has normal crossings over $\hat{P}$.

We conclude that $\bar{\rho}\bar{\sigma}^*(\tau(\A))(u\prod_{x\in P-\{z\}} x^{b(x)})$ has normal crossings over $P$ and 
$\bar{\rho}\bar{\sigma}^*(\tau(\A))(u\prod_{x\in P-\{z\}} x^{b(x)})\in PW(1)$.

We consider any $\chi\in\mathcal{X}$ and the element $z+\chi\in A$.

We consider the case $\chi=0$.

$z+\chi=z=x(f^P_z)$.
$f^P_z\in(\bar{\Delta}^\vee|V^*)\cap N
\subset (\hat{\Theta}^\vee|V^*)\cap N$.
$$\bar{\sigma}^*(\tau(\A))(z+\chi)= \bar{\sigma}^*(\tau(\A))( x(f^P_z))=
x(f^P_z)
=\prod_{\Gamma\in\mathcal{F}(\hat{\Theta})_1}x({b_{\Gamma/N^*}}^\vee_{\hat{\Theta}})^{\langle b_{\Gamma/N^*}, f^P_z \rangle}$$

We know that $\prod_{\Gamma\in\mathcal{F}(\hat{\Theta})_1-\mathcal{F}(\Theta)_1}x({b_{\Gamma/N^*}}^\vee_{\hat{\Theta}})^{\langle b_{\Gamma/N^*}, f^P_z \rangle}
\in(\mathcal{O}(\widetilde{\mathcal{C}})_{\tau(\A)}^c)^\times$ and\hfill\break 
$\bar{\sigma}^*(\tau(\A))(z+\chi)$ has normal crossings over $\hat{P}$.

We conclude that $\bar{\rho}\bar{\sigma}^*(\tau(\A))(z+\chi)$ has normal crossings over $P$ and  $\bar{\rho}\bar{\sigma}^*(\tau(\A))$\hfill\break$(z+\chi)\in PW(1)$.

We consider the case $\chi\neq 0$

$\Gamma_+(P,z+\chi)$ is $\bar{H}$-simple and $c(\Gamma_+(P,z+\chi))=2$.
$\Ht(\bar{H}, \Gamma_+(P,z+\chi))=1$.
$\chi\in M(A')-\{0\}$ and $\chi$ has normal crossings over $P-\{z\}$.
We take an element $u(\chi)\in A^{\prime\times}$ and a mapping $e(\chi):P-\{z\}\rightarrow\Z_0$ satisfying $\chi=u(\chi)\prod_{x\in P-\{z\}}x^{e(\chi)(x)}$. 
Let $\hat{e}(\chi)=\sum_{x\in P-\{z\}}e(\chi)(x)f^P_x\in(\bar{\Delta}^\vee|V^*)\cap N$.
\begin{equation*}\begin{split}
&\prod_{x\in P-\{z\}}x^{e(\chi)(x)}=\prod_{y\in P-\{z\}}y^{e(\chi)(y)}=
\prod_{y\in P-\{z\}}x(f^P_y)^{e(\chi)(y)}=x(\sum_{y\in P-\{z\}}e(\chi)(y)f^P_y)\\
=\:&x(\hat{e}(\chi)).
\end{split}\end{equation*}
$\chi=u(\chi)x(\hat{e}(\chi))$.
$\mathcal{V}(\Gamma_+(P,z+\chi))=\{f^P_z,\hat{e}(\chi)\}$.

Note that $\widetilde{\mathcal{C}}$ is a subdivision of $\mathcal{D}(S|V)$, $\mathcal{D}(S|V)$ is a subdivision of $\mathcal{D}(\Gamma_+(P,z+\chi)|V)$
and thus $\widetilde{\mathcal{C}}$ is a subdivision of $\mathcal{D}(\Gamma_+(P,z+\chi)|V)$. 
We take the unique element $\Lambda(\chi)\in \mathcal{D}(\Gamma_+(P,z+\chi)|V)$
with $\Theta^\circ\subset \Lambda(\chi)^\circ$ and we take the unique element $\hat{\Lambda}(\chi)\in \mathcal{D}(\Gamma_+(P,z+\chi)|V)$
with $\smash{\hat{\Theta}}^\circ\subset \hat{\Lambda}(\chi)^\circ$.
$\Theta^\circ\subset \Lambda^\circ\subset \Lambda(\chi)^\circ\subset\smash{\bar{\Delta}}^\circ$.
$\smash{\hat{\Theta}}^\circ\subset\smash{\hat{\Lambda}}^\circ$.
Since $\Theta\in\mathcal{F}(\hat{\Theta})$, 
$\Lambda(\chi)\in\mathcal{F}(\hat{\Lambda}(\chi))$.
We take the unique element $\bar{A}(\chi)\in\mathcal{F}(\Gamma_+(P,z+\chi)+(\Theta^\vee|V^*))$ satisfying
$\Theta=\Delta(\bar{A}(\chi), \Gamma_+(P,z+\chi(\chi))+(\Theta^\vee|V^*)|V)$.
$\Lambda(\chi)=\Delta(\bar{A}(\chi)\cap \Gamma_+(P,z+\chi), \Gamma_+(P,z+\chi)|V)$.
$\dim\hat{\Theta}=\dim \hat{\Lambda}(\chi)=\dim V$.
$\hat{\Lambda}(\chi)\in\mathcal{D}(\Gamma_+(P,z+\chi)|V)^0$.
We take the unique element $\hat{a}(\chi)\in\mathcal{V}(\Gamma_+(P,z+\chi))$
with $\hat{\Lambda}(\chi)=\Delta(\{\hat{a}(\chi)\}, \Gamma_+(P,z+\chi(\chi))|V)$.
$\hat{a}(\chi)=f^P_z$ or $\hat{a}(\chi)=\hat{e}(\chi)$.
$\Gamma_+(P,z+\chi)+(\hat{\Theta}^\vee|V^*)=\{\hat{a}(\chi)\} +(\hat{\Theta}^\vee|V^*)$.

Since $\mathcal{D}(\Gamma_+(P,z+\chi(i))|V)$ is $\bar{H}$-simple, 
$\dim \Lambda(\chi)=\dim V$ or $\dim \Lambda(\chi)=\dim V-1$

We consider the case $\dim \Lambda(\chi)=\dim V$.

By the same argument as in the case $\dim \Lambda=\dim V$ above, we know that
$\bar{\sigma}^*(\tau(\A))(z+\chi)$ has normal crossings over $\hat{P}$,
$\bar{\rho}\bar{\sigma}^*(\tau(\A))(z+\chi)$ has normal crossings over $P$ and
$\bar{\rho}\bar{\sigma}^*(\tau(\A))(z+\chi)\in PW(1)$.

We consider the case $\dim \Lambda(\chi)=\dim V-1$.

$\dim\Lambda=\dim \Lambda(\chi)=\dim V-1$.
Since $\Lambda\subset\Lambda(\chi)$, we have $\Vect(\Lambda)=\Vect(\Lambda(\chi))$.
Since $\mathcal{D}(S|V)$ is $\bar{H}$-simple, we have
$\Lambda=\Vect(\Lambda)\cap\bar{\Delta}=\Vect(\Lambda(\chi))\cap\bar{\Delta}\supset
\Lambda(\chi)$. Therefore, we know $\Lambda=\Lambda(\chi)$.

Now, by Theorem~\ref{important}.8.(a) and Theorem~\ref{important}.8.(b),
we know that $\hat{\Gamma}\not\subset\Vect(\Lambda)$ and $\Vect(\Lambda)+\hat{\Gamma}=\Vect(\Lambda)+\bar{H}$.

Since $\Vect(\Lambda)+\hat{\Gamma}=\Vect(\Lambda)+\hat{H}$ and
$\hat{\Gamma}\subset\hat{\Theta}\subset\hat{\Lambda}$, we know that
$\hat{\Lambda}(\chi)=\Delta(\{\hat{e}(\chi)\}, \Gamma_+(P,z+\chi)|V)$.
$\Delta(\bar{A}(\chi)\cap  \Gamma_+(P,z+\chi), \Gamma_+(P,z+\chi)|V)=\Lambda(\chi)=
\Delta(\{\hat{e}(\chi)\}, \Gamma_+(P,z+\chi(\chi))|V)\cap
\Delta(\{f^P_z\}, \Gamma_+(P,z+\chi)|V)$,
$\Delta(\Stab(\bar{A}(\chi)\cap  \Gamma_+(P,z+\chi)),\Stab(\Gamma_+(P,z+\chi))|V)=\bar{\Delta}$, 
$\Stab(\bar{A}(\chi)\cap  \Gamma_+(P,z+\chi))=\{0\}$, and $\bar{A}(\chi)\cap  \Gamma_+(P,z+\chi)=\Conv(\{f^P_z, \hat{e}(\chi)\})$.

Since $\hat{\Gamma}\not\subset\Vect(\Lambda)\supset\Lambda\supset\Theta$, we know that  $\hat{\Gamma}\in\mathcal{F}(\hat{\Theta})_1-\mathcal{F}(\Theta)_1$ and
$x({b_{\hat{\Gamma}/N^*}}^\vee_{\hat{\Theta}})\tau$\hfill\break$(\A)\neq 0$.

Since $1=\langle b_{\bar{H}/N^*},f^P_z\rangle>\langle b_{\bar{H}/N^*},\hat{e}(\chi)\rangle=0$ and $\Vect(\Lambda)+\hat{\Gamma}=\Vect(\Lambda)+\hat{H}$, we know
$\langle b_{\hat{\Gamma}/N^*},f^P_z\rangle>\langle b_{\hat{\Gamma}/N^*},\hat{e}(\chi)\rangle$,
$\{\langle b_{\hat{\Gamma}/N^*},a\rangle|a\in \bar{A}(\chi)\cap  \Gamma_+(P,z+\chi)\}=
\{t\in\R|\langle b_{\hat{\Gamma}/N^*},f^P_z\rangle\geq t\geq\langle b_{\hat{\Gamma}/N^*},\hat{e}(\chi)\rangle\}$ and
$\Z\ni\max\{\langle b_{\hat{\Gamma}/N^*},a\rangle|a\in \bar{A}(\chi)\cap  \Gamma_+(P,z+\chi)\}
-\min\{\langle b_{\hat{\Gamma}/N^*},a\rangle|a\in \bar{A}(\chi)\cap  \Gamma_+(P,z+\chi)\}= \langle b_{\hat{\Gamma}/N^*}, f^P_z -\hat{e}(\chi)\rangle>0$.

By Theorem~\ref{important}.12.(b) we know  $\langle b_{\hat{\Gamma}/N^*}, f^P_z -\hat{e}(\chi)\rangle\leq\Ht(\bar{H}, \Gamma_+(P,z+\chi))=1$.
We conclude that $\langle b_{\hat{\Gamma}/N^*}, f^P_z -\hat{e}(\chi)\rangle=1$.

We denote $m=\langle b_{\hat{\Gamma}/N^*}, \hat{e}(\chi)\rangle\in\Z_0$.

For any $\Gamma\in\mathcal{F}(\Theta)_1$, we have
$b_{\Gamma/N^*}\in\Gamma\subset\Theta\subset\Lambda(\chi)=
\Delta(\{\hat{e}(\chi)\}, \Gamma_+(P,z+\chi)|V)\cap
\Delta(\{f^P_z\}, \Gamma_+(P,z+\chi)|V)$ and $\langle b_{\Gamma/N^*}, f^P_z \rangle=\langle b_{\Gamma/N^*}, \hat{e}(\chi)\rangle$.
We denote 
\begin{equation*}\begin{split}
\hat{b}&=\sum_{\Gamma\in\mathcal{F}(\Theta)_1}\langle b_{\Gamma/N^*}, f^P_z \rangle{ b_{\Gamma/N^*}}^\vee_{\hat{\Theta}}\in\Delta(\Theta\Op|\hat{\Theta},\hat{\Theta}|V^*)\cap N,\\
\hat{c}_0&=\sum_{\Gamma\in\mathcal{F}((\Theta+\hat{\Gamma})\Op|\hat{\Theta})_1}\langle b_{\Gamma/N^*}, f^P_z\rangle{ b_{\Gamma/N^*}}^\vee_{\hat{\Theta}}\in \Delta(\Theta+\hat{\Gamma},\hat{\Theta}|V^*)\cap N,\text{ and}\\
\hat{c}&=\sum_{\Gamma\in\mathcal{F}((\Theta+\hat{\Gamma})\Op|\hat{\Theta})_1}\langle b_{\Gamma/N^*}, \hat{e}(\chi)\rangle{b_{\Gamma/N^*}}^\vee_{\hat{\Theta}}\in \Delta(\Theta+\hat{\Gamma},\hat{\Theta}|V^*)\cap N.
\end{split}\end{equation*}
We have $f^P_z=\hat{b}+(1+m) {b_{\hat{\Gamma}/N^*}}^\vee_{\hat{\Theta}}+\hat{c}_0$ and $\hat{e}(\chi)=\hat{b}+m{b_{\hat{\Gamma}/N^*}}^\vee_{\hat{\Theta}}+\hat{c}$.

Let $u(\chi)(0)\in k-\{0\}$ denote the unique element satisfying $u(\chi)-u(\chi)(0)\in M(A')$.
We have 
\begin{equation*}\begin{split}
&\Ps(P, \bar{A}(\chi)\cap \Gamma_+(P,z+\chi),z+\chi)\\
=\:&x(\hat{b})x({b_{\hat{\Gamma}/N^*}}^\vee_{\hat{\Theta}})^{1+m}x(\hat{c}_0)
+ u(\chi)(0) x(\hat{b}) x({b_{\hat{\Gamma}/N^*}}^\vee_{\hat{\Theta}})^mx(\hat{c})\\
=\:& x(\hat{b})x({b_{\hat{\Gamma}/N^*}}^\vee_{\hat{\Theta}})^m
(x({b_{\hat{\Gamma}/N^*}}^\vee_{\hat{\Theta}}) x(\hat{c}_0)+ x(\hat{c}))\text{, and}\\
&\In(\hat{P},b_{\Theta/N^*},\bar{\sigma}^*(\tau(\A))(z+\chi))\\
=\:&\bar{\sigma}^*(\tau(\A))( \Ps(P,\bar{A}(\chi)\cap \Gamma_+(P,z+\chi),z+\chi))\\
=\:&x(\hat{b})x({b_{\hat{\Gamma}/N^*}}^\vee_{\hat{\Theta}})^m
(x({b_{\hat{\Gamma}/N^*}}^\vee_{\hat{\Theta}}) x(\hat{c}_0)+ x(\hat{c})).
\end{split}\end{equation*}

Recall that $x({b_{\Gamma/N^*}}^\vee_{\hat{\Theta}})(\tau(\A))=0$ for any $\Gamma\in\mathcal{F}(\Theta)_1$,
$x({b_{\hat{\Gamma}/N^*}}^\vee_{\hat{\Theta}})(\tau(\A))\neq 0$,
$x({b_{\Gamma/N^*}}^\vee_{\hat{\Theta}})(\tau(\A))\neq 0$ for any $\Gamma\in\mathcal{F}((\Theta+\hat{\Gamma})\Op|\hat{\Theta})_1$ and
$\hat{P}=\{x({b_{\Gamma/N^*}}^\vee_{\hat{\Theta}})|\Gamma\in\mathcal{F}(\Theta)_1\}\cup\{x({b_{\hat{\Gamma}/N^*}}^\vee_{\hat{\Theta}})- x({b_{\hat{\Gamma}/N^*}}^\vee_{\hat{\Theta}})\tau(\A)\}\cup
\{ x({b_{\Gamma/N^*}}^\vee_{\hat{\Theta}})-
x({b_{\Gamma/N^*}}^\vee_{\hat{\Theta}})\tau(\A)|
\Gamma\in\mathcal{F}((\Theta+\hat{\Gamma})\Op|\hat{\Theta})_1\}$.

Note that $k[\hat{P}-\{x({b_{\hat{\Gamma}/N^*}}^\vee_{\hat{\Theta}})- x({b_{\hat{\Gamma}/N^*}}^\vee_{\hat{\Theta}})\tau(\A)\}]\subset 
\mathcal{O}(\widetilde{\mathcal{C}})_{\tau(\A)}^c$.
By $\mathcal{O}(\widetilde{\mathcal{C}})_{\tau(\A)}^{c\:\prime}$ we denote the
completion of $k[\hat{P}-\{x({b_{\hat{\Gamma}/N^*}}^\vee_{\hat{\Theta}})- x({b_{\hat{\Gamma}/N^*}}^\vee_{\hat{\Theta}})\tau(\A)\}]$ with respect to the maximal ideal $k[\hat{P}-\{x({b_{\hat{\Gamma}/N^*}}^\vee_{\hat{\Theta}})- x({b_{\hat{\Gamma}/N^*}}^\vee_{\hat{\Theta}})\tau(\A)\}]\cap M(\mathcal{O}(\widetilde{\mathcal{C}})_{\tau(\A)}^c)$.
$\mathcal{O}(\widetilde{\mathcal{C}})_{\tau(\A)}^{c\:\prime}$ is a complete local subring of 
$\mathcal{O}(\widetilde{\mathcal{C}})_{\tau(\A)}^c$.
$\mathcal{O}(\widetilde{\mathcal{C}})_{\tau(\A)}^{c\:\prime}$ is a $k$-subalgebra of $\mathcal{O}(\widetilde{\mathcal{C}})_{\tau(\A)}^c$.

$x({b_{\hat{\Gamma}/N^*}}^\vee_{\hat{\Theta}})^m\in (\mathcal{O}(\widetilde{\mathcal{C}})_{\tau(\A)}^c)^\times$, 
$x(\hat{c}_0)\in (\mathcal{O}(\widetilde{\mathcal{C}})_{\tau(\A)}^{c\:\prime})^\times$ and
$x(\hat{c})\in (\mathcal{O}(\widetilde{\mathcal{C}})_{\tau(\A)}^{c\:\prime})^\times$.
We know that there exist $v'\in (\mathcal{O}(\widetilde{\mathcal{C}})_{\tau(\A)}^c)^\times$ and 
$w'\in (\mathcal{O}(\widetilde{\mathcal{C}})_{\tau(\A)}^{c\:\prime})^\times$
satisfying 
$\In(\hat{P},b_{\Theta/N^*},$\break$\bar{\sigma}^*(\tau(\A))(z+\chi))=
x(\hat{b}) v'(x({b_{\hat{\Gamma}/N^*}}^\vee_{\hat{\Theta}})+ w')$.

Note that $\Gamma_+(P,z+\chi)+(\hat{\Theta}^\vee|V^*)=\{\hat{a}(\chi)\} +(\hat{\Theta}^\vee|V^*)$ and $ b_{\Theta/N^*}\in\Theta^\circ\subset\hat{\Theta}$.
We know that there exist $v\in (\mathcal{O}(\widetilde{\mathcal{C}})_{\tau(\A)}^c)^\times$ and 
$w\in (\mathcal{O}(\widetilde{\mathcal{C}})_{\tau(\A)}^{c\:\prime})^\times$
satisfying 
$\bar{\sigma}^*(\tau(\A))(z+\chi)=
x(\hat{b}) v(x({b_{\hat{\Gamma}/N^*}}^\vee_{\hat{\Theta}})+ w)$.
We take elements $v\in (\mathcal{O}(\widetilde{\mathcal{C}})_{\tau(\A)}^c)^\times$ and 
$w\in (\mathcal{O}(\widetilde{\mathcal{C}})_{\tau(\A)}^{c\:\prime})^\times$
satisfying 
$\bar{\sigma}^*(\tau(\A))(z+\chi)=
x(\hat{b}) v(x({b_{\hat{\Gamma}/N^*}}^\vee_{\hat{\Theta}})+ w)$.
We take the unique element $w(0)\in k-\{0\}$ satisfying
$w-w(0)\in M(\mathcal{O}(\widetilde{\mathcal{C}})_{\tau(\A)}^{c\:\prime})$.

We consider the case $ x({b_{\hat{\Gamma}/N^*}}^\vee_{\hat{\Theta}})(\tau(\A))+w(0)\neq 0$.
We know that $ x({b_{\hat{\Gamma}/N^*}}^\vee_{\hat{\Theta}})+ w\in (\mathcal{O}(\widetilde{\mathcal{C}})_{\tau(\A)}^c)^\times$,
$ v(x({b_{\hat{\Gamma}/N^*}}^\vee_{\hat{\Theta}})+ w) \in (\mathcal{O}(\widetilde{\mathcal{C}})_{\tau(\A)}^c)^\times$ and
$\bar{\sigma}^*(\tau(\A))(z+\chi)$ has normal crossings over $\hat{P}$.
We know that $\bar{\rho}\bar{\sigma}^*(\tau(\A))(z+\chi)$ has normal crossings over $P$ and $\bar{\rho}\bar{\sigma}^*(\tau(\A))(z+\chi)\in PW(1)$.

We consider the case $ x({b_{\hat{\Gamma}/N^*}}^\vee_{\hat{\Theta}})(\tau(\A))+w(0)= 0$.
$ x(\hat{b})$ has normal crossings over $\hat{P}$ and $\bar{\rho}( x(\hat{b}))$ has normal crossings over $P$. $\bar{\rho}(v)\in A^\times$.
We know that $x({b_{\hat{\Gamma}/N^*}}^\vee_{\hat{\Theta}})(\tau(\A))+ w\in M(\mathcal{O}(\widetilde{\mathcal{C}})_{\tau(\A)}^{c\:\prime})$,
$\bar{\rho}( x({b_{\hat{\Gamma}/N^*}}^\vee_{\hat{\Theta}})(\tau(\A))+ w)
\in M(A')$ and
$\bar{\rho}\bar{\sigma}^*(\tau(\A))(z+\chi)= \bar{\rho}( x(\hat{b}))\bar{\rho}(v)(z+\bar{\rho}( x({b_{\hat{\Gamma}/N^*}}^\vee_{\hat{\Theta}})(\tau(\A))+ w))\in PW(1)$.

We conclude that $\bar{\rho}\bar{\sigma}^*(\tau(\A))(z+\chi)\in PW(1)$ in all cases.

Note that $\zeta\eta\in PW(1)$ for any $\zeta\in PW(1)$ and any $\eta\in PW(1)$.
Therefore, we conclude that
$\bar{\rho}\bar{\sigma}^*(\tau(\A))(u\prod_{\chi\in\mathcal{X}}(z+\chi)^{a(\chi)}\prod_{x\in P-\{z\}} x^{b(x)})=$\hfill\break$
\bar{\rho}\bar{\sigma}^*(\tau(\A))(u\prod_{x\in P-\{z\}} x^{b(x)})
\prod_{\chi\in\mathcal{X}}\bar{\rho}\bar{\sigma}^*(\tau(\A))(z+\chi)^{a(\chi)}\in PW(1)$.

We consider the element $\psi\in W(h)$.

Recall that $\mathcal{D}(\Gamma_+(P,\psi)|V)$ is $\bar{H}$-simple,
$\Theta\in\widetilde{\mathcal{C}}$,
$\tau(\A)\in V^\circ(\widetilde{\mathcal{C}},\Theta)\subset\bar{\Sigma}$,
$\hat{\Gamma}\in\widetilde{\mathcal{C}}_1$,
$\hat{\Gamma}\not\subset \bar{H}\Op|\bar{\Delta}$,
$\Theta\in\mathcal{A}^\circ(\hat{\Gamma})$,
$\hat{\Theta}\in\smash{\widetilde{\mathcal{C}}}^0$,
$\Theta\in\mathcal{F}(\hat{\Theta})$ and
$\hat{\Gamma}\in\mathcal{F}(\hat{\Theta})_1$.

Note that $\widetilde{\mathcal{C}}$ is a subdivision of $\mathcal{D}(S|V)$,
$\mathcal{D}(S|V)$ is a subdivision of $\mathcal{D}(\Gamma_+(P,\psi)|V)$ and thus 
$\widetilde{\mathcal{C}}$ is a subdivision of $\mathcal{D}(\Gamma_+(P,\psi)|V)$.
We take the unique element $\Lambda(\psi)\in\mathcal{D}(\Gamma_+(P,\psi)|V)$ with
$\Theta^\circ\subset\Lambda(\psi)^\circ$.
We take the unique element $\bar{A}(\psi)\in\mathcal{F}(\Gamma_+(P,\psi)+(\Theta^\vee|V^*))$ with
$\Theta=\Delta(\bar{A}(\psi), \Gamma_+(P,\psi)+(\Theta^\vee|V^*)|V)$.
We take the unique element $\hat{\Lambda}(\psi)\in\mathcal{D}(\Gamma_+(P,\psi)|V)$ with
$\smash{\hat{\Theta}}^\circ\subset\hat{\Lambda}(\psi)^\circ$.
$\dim \hat{\Theta}=\dim\hat{\Lambda}(\psi)=\dim V$.
$\hat{\Lambda}(\psi)\in\mathcal{D}(\Gamma_+(P,\psi)|V)^0$.
We take the unique element $\hat{a}(\psi)\in\mathcal{V}(\Gamma_+(P,\psi))$ with
$\hat{\Lambda}(\psi)=\Delta(\{\hat{a}(\psi)\}, \Gamma_+(P,\psi)|V)$.

$\Theta^\circ\subset\Lambda^\circ\subset\Lambda(\psi)^\circ\subset\bar{\Delta}^\circ$.
$\Lambda(\psi)=\Delta(\bar{A}(\psi)\cap \Gamma_+(P,\psi), \Gamma_+(P,\psi)|V)$.
$\smash{\hat{\Theta}}^\circ\subset\smash{\hat{\Lambda}}^\circ$.
$\Lambda(\psi)\in\mathcal{F}(\hat{\Lambda}(\psi))$.
$\Gamma_+(P,\psi)+(\hat{\Theta}^\vee|V^*)=\{\hat{a}(\psi)\} +(\hat{\Theta}^\vee|V^*)$.

Since $\Gamma_+(P,\psi)$ is $\bar{H}$-simple, we know that $\dim\Lambda(\psi)=\dim V$ or $\dim\Lambda(\psi)=\dim V-1$.

We consider case $\dim\Lambda(\psi)=\dim V$.

By the same argument as in the case $\dim \Lambda=\dim V$ above, we know that 
$\bar{\sigma}^*(\tau(\A))(\psi)$ has normal crossings over $\hat{P}$, 
$\bar{\rho}\bar{\sigma}^*(\tau(\A))(\psi)$ has normal crossings over $P$ and
$\bar{\rho}\bar{\sigma}^*(\tau(\A))(\psi)\in PW(1)$.
Therefore, $\bar{\rho}\bar{\sigma}^*(\tau(\A))(\phi)=
\bar{\rho}\bar{\sigma}^*(\tau(\A))(\psi) \bar{\rho}\bar{\sigma}^*(\tau(\A))$\hfill\break$(u\prod_{\chi\in\mathcal{X}}(z+\chi)^{a(\chi)}\prod_{x\in P-\{z\}} x^{b(x)})\in PW(1)$.

We know that there exists uniquely an element $\B\in(\sigma^*D)_0$ with $\tau(\B)=\bar{\B}$. We take the unique element $\B\in(\sigma^*D)_0$ with $\tau(\B)=\bar{\B}$.
Since $\tau(\A)\in U(\bar{\Sigma},\bar{\sigma}^*\bar{D},\bar{\B})$, $\A\in\tau^{-1}( U(\bar{\Sigma},\bar{\sigma}^*\bar{D},\bar{\B}))=
U(\Sigma, \sigma^*D,\B)$.
It is easy to see that the homomorphism $\tau^*(\A): \mathcal{O}(\widetilde{\mathcal{C}})_{\tau(\A)}\rightarrow\mathcal{O}_{\Sigma, \A}$ induces an isomorphism $\tau^*(\A):\mathcal{O}(\widetilde{\mathcal{C}})_{\tau(\A)}^c\rightarrow\mathcal{O}_{\Sigma, \A}^c$ of complete $k$-algebras.
Since $\sigma(\A)=M(A)$, we have a homomorphism of $k$-algebras $\sigma^*(\A):A\rightarrow\mathcal{O}_{\Sigma, \A}$ satisfying
$\sigma^*(\A)(M(A))\subset M(\mathcal{O}_{\Sigma, \A})$. This homomorphism has the unique extension $\sigma^*(\A):A \rightarrow\mathcal{O}_{\Sigma, \A}^c$.
Since $\iota^*\sigma=\bar{\sigma}\tau$, we have $\sigma^*(\A)(\phi)=\tau^*(\A) \bar{\sigma}^*(\tau(\A))(\phi)$.

Let $\bar{P}=\{(\sigma^*\xi)_\B(E)- (\sigma^*\xi)_\B(E)(\A)|E\in\Comp(\sigma^*D)(\B)\}$.
$\bar{P}$ is a parameter system of $\mathcal{O}_{\Sigma, \A}^c$.
Since $\hat{P}=\{(\bar{\sigma}^*\bar{\xi})_{\bar{\B}}(E)- (\bar{\sigma}^*\bar{\xi})_{\bar{\B}}(E)(\tau(\A))|E\in\Comp(\bar{\sigma}^*\bar{D})(\bar{\B})\}$, $\bar{P}=\tau^*(\A)(\hat{P})$.

Let $\rho=\bar{\rho}\tau^*(\A)^{-1}:\mathcal{O}_{\Sigma, \A}^c\rightarrow A$.
$\rho$ is an isomorphism of $k$-algebras.
We put $g=1\in\Z_+$.
$\rho(\bar{P})=\bar{\rho}\tau^*(\A)^{-1}(\tau^*(\A)(\hat{P}))=\bar{\rho}(\hat{P})=P$. We know $\rho(\bar{P})=P$.
$\rho\sigma^*(\A)(\phi)= \bar{\rho}\tau^*(\A)^{-1}\tau^*(\A) \bar{\sigma}^*(\tau(\A))(\phi)
=\bar{\rho}\bar{\sigma}^*(\tau(\A))(\phi)\in PW(1)=PW(g)$.
We know $\rho\sigma^*(\A)(\phi)\in PW(g)$.
Since $h\geq 2$, $g<h$.

We conclude that Theorem~\ref{main} holds, if $\dim\Lambda=\dim V-1$ and $\dim\Lambda(\psi)=\dim V$.

We consider the case $\dim\Lambda(\psi)=\dim V-1$.

We know that
$\Lambda(\psi)\in\bar{\mathcal{D}}(\Gamma_+(P,\psi)|V)^1$. 
$\Lambda(\psi)\neq \bar{H}\Op|\bar{\Delta}$ and $\Ht(\bar{H}, \Gamma_+(P,\psi$\break$))>0$.

Since $\Lambda\subset\Lambda(\psi)\subset\bar{\Delta}$ and $\dim\Lambda=\dim V-1=\dim\Lambda(\psi)$,
we have $\Vect(\Lambda)=\Vect(\Lambda(\psi))$.
Since $\mathcal{D}(S|V)$ is $\bar{H}$-simple, we have $\Lambda=\Vect(\Lambda)\cap\bar{\Delta}=\Vect(\Lambda(\psi))\cap\bar{\Delta}\supset\Lambda(\psi)$.
Therefore, we know $\Lambda=\Lambda(\psi)$.

We know that there exists uniquely an element $\hat{\Lambda}_0(\psi)\in\mathcal{D}(\Gamma_+(P,\psi)|V)^0$ satisfying $\Lambda(\psi)=\hat{\Lambda}(\psi)\cap\hat{\Lambda}_0(\psi)$. 
We take the unique element $\hat{\Lambda}_0(\psi)\in\mathcal{D}(\Gamma_+(P,\psi)|V)^0$ satisfying $\Lambda(\psi)=\hat{\Lambda}(\psi)\cap\hat{\Lambda}_0(\psi)$.
$\hat{\Lambda}(\psi)\neq\hat{\Lambda}_0(\psi)$.
We take the unique element $\hat{a}_0(\psi)\in\mathcal{V}(\Gamma_+(P,\psi))$ with $\hat{\Lambda}_0(\psi)
=\Delta(\{\hat{a}_0(\psi)\}, \Gamma_+(P,\psi)|V)$.

$\hat{a}_0(\psi)\neq\hat{a}(\psi)$,
$\Delta(\bar{A}(\psi)\cap \Gamma_+(P,\psi), \Gamma_+(P,\psi)|V)=\Lambda(\psi)=\hat{\Lambda}(\psi)\cap\hat{\Lambda}_0(\psi)=
\Delta(\{\hat{a}(\psi)\}, \Gamma_+(P,\psi)|V)\cap \Delta(\{\hat{a}_0(\psi)\}, \Gamma_+(P,\psi)|V)$,
$\Delta(\Stab(\bar{A}(\psi)\cap \Gamma_+(P,\psi)),\Stab($\hfill\break$\Gamma_+(P,\psi))|V)=\bar{\Delta}$,
$\Stab(\bar{A}(\psi)\cap \Gamma_+(P,\psi))=\{0\}$
and $\bar{A}(\psi)\cap \Gamma_+(P,\psi)=\Conv(\{$\hfill\break$\hat{a}_0(\psi), \hat{a}(\psi)\})$.

Now, by Theorem~\ref{important}.8.(a) and Theorem~\ref{important}.8.(b),
we know that $\hat{\Gamma}\not\subset\Vect(\Lambda)$ and $\Vect(\Lambda)+\hat{\Gamma}=\Vect(\Lambda)+\bar{H}$.

Since $\hat{\Gamma}\not\subset\Vect(\Lambda)\supset\Lambda\supset\Theta$, we know that  $\hat{\Gamma}\in\mathcal{F}(\hat{\Theta})_1-\mathcal{F}(\Theta)_1$ and
$x({b_{\hat{\Gamma}/N^*}}^\vee_{\hat{\Theta}})\tau$\hfill\break$(\A)\neq 0$.

Since $\hat{\Gamma}\in\mathcal{F}(\hat{\Theta})_1$, we know that
$\hat{\Gamma}\subset\hat{\Theta}\subset\hat{\Lambda}$,
$\hat{\Lambda}\subset \Vect(\Lambda)+\hat{\Gamma}=\Vect(\Lambda)+\bar{H}$,
$\hat{\Lambda}+\bar{H}=\Lambda+\bar{H}$, $\hat{\Lambda}+\bar{H}\subset
\hat{\Lambda}_0+\bar{H}$,
$\langle b_{\bar{H}/N^*}, \hat{a}_0(\psi)\rangle>\langle b_{\bar{H}/N^*}, \hat{a}(\psi)\rangle$,
$\langle b_{\hat{\Gamma}/N^*}, \hat{a}_0(\psi)\rangle>\langle b_{\hat{\Gamma}/N^*}, \hat{a}(\psi)\rangle$,
$\{\langle b_{\hat{\Gamma}/N^*}, a\rangle|a\in \bar{A}(\psi)\cap \Gamma_+(P,\psi)\}=
\{t\in\R|\langle b_{\hat{\Gamma}/N^*}, \hat{a}_0(\psi)\rangle\geq t\geq\langle b_{\hat{\Gamma}/N^*}, \hat{a}(\psi)\rangle\}$ and
$\max\{\langle b_{\hat{\Gamma}/N^*}, a\rangle|a\in \bar{A}(\psi)\cap \Gamma_+(P,\psi)\}
-\min\{\langle b_{\hat{\Gamma}/N^*}, a\rangle|a\in \bar{A}(\psi)\cap \Gamma_+(P,\psi)\}=
\langle b_{\hat{\Gamma}/N^*}, \hat{a}_0(\psi)-\hat{a}(\psi)\rangle>0$.

We denote $\ell=\langle b_{\hat{\Gamma}/N^*}, \hat{a}_0(\psi)-\hat{a}(\psi)\rangle\in\Z_+$ and
$m=\langle b_{\hat{\Gamma}/N^*}, \hat{a}(\psi)\rangle\in\Z_0$ for simplicity.
By Theorem~\ref{important}.12.(b) we know 
$\ell\leq\Ht(\bar{H},\Gamma_+(P,\psi))=h$.

For any $\Gamma\in\mathcal{F}(\Theta)_1$, we have
$b_{\Gamma/N^*}\in\Gamma\subset\Theta\subset\Lambda(\psi)=\hat{\Lambda}(\psi)\cap\hat{\Lambda}_0(\psi)$ and $\langle b_{\Gamma/N^*}, \hat{a}_0(\psi)\rangle=\langle b_{\Gamma/N^*}, \hat{a}(\psi)\rangle$.
We denote 
\begin{equation*}\begin{split}
\hat{b}&=\sum_{\Gamma\in\mathcal{F}(\Theta)_1}\langle b_{\Gamma/N^*}, \hat{a}_0(\psi)\rangle{ b_{\Gamma/N^*}}^\vee_{\hat{\Theta}}\in\Delta(\Theta\Op|\hat{\Theta},\hat{\Theta}|V^*)\cap N,\\
\hat{c}_0&=\sum_{\Gamma\in\mathcal{F}((\Theta+\hat{\Gamma})\Op|\hat{\Theta})_1}\langle b_{\Gamma/N^*}, \hat{a}_0(\psi)\rangle{ b_{\Gamma/N^*}}^\vee_{\hat{\Theta}}\in \Delta(\Theta+\hat{\Gamma},\hat{\Theta}|V^*)\cap N,\text{ and}\\
\hat{c}&=\sum_{\Gamma\in\mathcal{F}((\Theta+\hat{\Gamma})\Op|\hat{\Theta})_1}\langle b_{\Gamma/N^*}, \hat{a}(\psi)\rangle{b_{\Gamma/N^*}}^\vee_{\hat{\Theta}}\in \Delta(\Theta+\hat{\Gamma},\hat{\Theta}|V^*)\cap N.
\end{split}\end{equation*}
We have $\hat{a}_0(\psi)=\hat{b}+(\ell+m) {b_{\hat{\Gamma}/N^*}}^\vee_{\hat{\Theta}}
+\hat{c}_0$ and $\hat{a}(\psi)=\hat{b}+m {b_{\hat{\Gamma}/N^*}}^\vee_{\hat{\Theta}}
+\hat{c}$.

Let $L=\{i\in\{0,1,\ldots,\ell\}|((\ell-i)/\ell) \hat{c}_0+(i/\ell) \hat{c}\in N\}$.
$\{0,\ell\}\subset L\subset\{0,1,\ldots,\ell\}$ and
$\Conv(\{\hat{a}_0(\psi),\hat{a}(\psi)\})\cap N=
\{\hat{b}+(\ell-i+m) {b_{\hat{\Gamma}/N^*}}^\vee_{\hat{\Theta}}
+((\ell-i)/\ell) \hat{c}_0+(i/\ell) \hat{c}|i\in L\}$.
$((\ell-i)/\ell) \hat{c}_0+(i/\ell) \hat{c}\in \Delta(\Theta+\hat{\Gamma},\hat{\Theta}|V^*)\cap N$ for any $i\in L$.

We know that there exists uniquely a mapping $e:L\rightarrow k$ satisfying
$$\Ps(P,\bar{A}(\psi)\cap \Gamma_+(P,\psi),\psi)=\sum_{i\in L}e(i)x(\hat{b})x({b_{\hat{\Gamma}/N^*}}^\vee_{\hat{\Theta}})^{\ell-i+m}
x((\frac{\ell-i}{\ell}) \hat{c}_0+(\frac{i}{\ell}) \hat{c}),$$
$e(0)\neq 0$ and $e(\ell)\neq 0$.
We take the unique mapping $e:L\rightarrow k$ satisfying the above three conditions.
We have
\begin{equation*}\begin{split}
&\In(\hat{P},b_{\Theta/N^*},\bar{\sigma}^*(\tau(\A))(\psi))=
\bar{\sigma}^*(\tau(\A))( \Ps(P,\bar{A}(\psi)\cap \Gamma_+(P,\psi),\psi))\\
=\:&
\sum_{i\in L}e(i)x(\hat{b})x({b_{\hat{\Gamma}/N^*}}^\vee_{\hat{\Theta}})^{\ell-i+m}
x((\frac{\ell-i}{\ell}) \hat{c}_0+(\frac{i}{\ell}) \hat{c}).
\end{split}\end{equation*}

Recall that $x({b_{\Gamma/N^*}}^\vee_{\hat{\Theta}})\tau(\A)=0$ for any $\Gamma\in\mathcal{F}(\Theta)_1$,
$x({b_{\hat{\Gamma}/N^*}}^\vee_{\hat{\Theta}})\tau(\A)\neq 0$,\break
$x({b_{\Gamma/N^*}}^\vee_{\hat{\Theta}})\tau(\A)\neq 0$ for any $\Gamma\in\mathcal{F}((\Theta+\hat{\Gamma})\Op|\hat{\Theta})_1$ and
$\hat{P}=\{x({b_{\Gamma/N^*}}^\vee_{\hat{\Theta}})|\Gamma\in$\break$\mathcal{F}(\Theta)_1\}\cup\{x({b_{\hat{\Gamma}/N^*}}^\vee_{\hat{\Theta}})- x({b_{\hat{\Gamma}/N^*}}^\vee_{\hat{\Theta}})\tau(\A)\}\cup
\{ x({b_{\Gamma/N^*}}^\vee_{\hat{\Theta}})-
x({b_{\Gamma/N^*}}^\vee_{\hat{\Theta}})\tau(\A)|
\Gamma\in$\break$\mathcal{F}((\Theta+\hat{\Gamma})\Op|\hat{\Theta})_1\}$.

We know that 
$x(\hat{c}_0)\tau(\A)\neq 0$, $\deg(\hat{P},b_{\hat{\Gamma}/N^*},
\In(\hat{P},b_{\Theta/N^*},\bar{\sigma}^*(\tau(\A))(\psi))=\ell+m$,
the sum of terms $T$ in $\In(\hat{P},b_{\Theta/N^*},\bar{\sigma}^*(\tau(\A))(\psi))$ with 
$\deg(\hat{P},b_{\hat{\Gamma}/N^*},T)=\ell+$\break$m$ is equal to $e(0) x(\hat{c}_0)\tau(\A)
x(\hat{b}+(\ell+m) {b_{\hat{\Gamma}/N^*}}^\vee_{\hat{\Theta}})$ and
$\hat{b}+(\ell+m) {b_{\hat{\Gamma}/N^*}}^\vee_{\hat{\Theta}}\in$\break$
\Supp(\hat{P}, \In(\hat{P},b_{\Theta/N^*},\bar{\sigma}^*(\tau(\A))(\psi)))\subset
\Supp(\hat{P},\bar{\sigma}^*(\tau(\A))(\psi))\subset
\Gamma_+(\hat{P},\bar{\sigma}^*(\tau(\A))(\psi))$.

Consider any $\Gamma\in\mathcal{F}(\Theta)_1$.
$\Ord(\hat{P},b_{\Gamma/N^*},\bar{\sigma}^*(\tau(\A))(\psi))=
\Ord(\hat{P},b_{\Gamma/N^*},\In(\hat{P},$\hfill\break$b_{\Theta/N^*},\bar{\sigma}^*(\tau(\A))(\psi)))=
\langle b_{\Gamma/N^*},\hat{b}\rangle=
\langle b_{\Gamma/N^*},\hat{b}+(\ell+m) {b_{\hat{\Gamma}/N^*}}^\vee_{\hat{\Theta}}\rangle$.

Consider any $\Gamma\in\mathcal{F}((\Theta+\hat{\Gamma})\Op|\hat{\Theta})_1$.
$0\leq
\Ord(\hat{P},b_{\Gamma/N^*},\bar{\sigma}^*(\tau(\A))(\psi))\leq
\Ord(\hat{P},$\hfill\break$b_{\Gamma/N^*},\In(\hat{P},b_{\Theta/N^*},\bar{\sigma}^*(\tau(\A))(\psi)))=0$.
We know that
$\Ord(\hat{P},b_{\Gamma/N^*},\bar{\sigma}^*(\tau(\A))(\psi))=0=
\langle b_{\Gamma/N^*},\hat{b}+(\ell+m) {b_{\hat{\Gamma}/N^*}}^\vee_{\hat{\Theta}}\rangle$.

We know that 
$\hat{b}+(\ell+m) {b_{\hat{\Gamma}/N^*}}^\vee_{\hat{\Theta}}\in 
\Gamma_+(\hat{P},\bar{\sigma}^*(\tau(\A))(\psi))$ and 
$\Ord(\hat{P},b_{\Gamma/N^*},\bar{\sigma}^*(\tau(\A)$\hfill\break$)(\psi))=
\langle b_{\Gamma/N^*},\hat{b}+(\ell+m) {b_{\hat{\Gamma}/N^*}}^\vee_{\hat{\Theta}}\rangle$
for any $\Gamma\in\mathcal{F}(\hat{\Theta})_1-\{\hat{\Gamma}\}$ and we conclude that $\Gamma_+(\hat{P}, \bar{\sigma}^*(\tau(\A))(\psi))$ is of $\hat{\Gamma}$-Weierstrass type.

We denote $\bar{\Theta}=\hat{\Gamma}\Op|\hat{\Theta}\in\mathcal{F}(\hat{\Theta})^1$.
Since $\Gamma_+(\hat{P}, \bar{\sigma}^*(\tau(\A))(\psi))$ is of $\hat{\Gamma}$-Weierstrass type, we have
$\Ord(\hat{P}, b_{\bar{\Theta}/N^*}, \bar{\sigma}^*(\tau(\A))(\psi))=
\langle b_{\bar{\Theta}/N^*},\hat{b}+(\ell+m) {b_{\hat{\Gamma}/N^*}}^\vee_{\hat{\Theta}}\rangle=
\Ord(\hat{P},$\hfill\break$ b_{\bar{\Theta}/N^*},\In(\hat{P},b_{\Theta/N^*},\bar{\sigma}^*(\tau(\A))(\psi)))$
and
\begin{equation*}\begin{split}
&\In(\hat{P}, b_{\bar{\Theta}/N^*}, \bar{\sigma}^*(\tau(\A))(\psi))=
\In(\hat{P}, b_{\bar{\Theta}/N^*},\In(\hat{P},b_{\Theta/N^*},\bar{\sigma}^*(\tau(\A))(\psi)))
\\
=\:& x(\hat{b})x({b_{\hat{\Gamma}/N^*}}^\vee_{\hat{\Theta}})^m\sum_{i\in L}\bar{e}(i) x({b_{\hat{\Gamma}/N^*}}^\vee_{\hat{\Theta}})^{\ell-i},
\end{split}\end{equation*}
where $\bar{e}:L\rightarrow k$ is the mapping satisfying
$\bar{e}(i)= e(i) x(((\ell-i)/\ell) \hat{c}_0+(i/\ell) \hat{c})\tau(\A)$ for any $i\in L$.
$\bar{e}(0)=e(0)x(\hat{c}_0)(\tau(\A))\neq 0$.
There exists uniquely a mapping $\bar{\bar{e}}:\{0,1,\ldots,\ell\}\rightarrow k$
satisfying
$\sum_{i\in L}\bar{e}(i) x({b_{\hat{\Gamma}/N^*}}^\vee_{\hat{\Theta}})^{\ell-i}=
\sum_{i=0}^{\ell}\bar{\bar{e}}(i)
(x({b_{\hat{\Gamma}/N^*}}^\vee_{\hat{\Theta}})- x({b_{\hat{\Gamma}/N^*}}^\vee_{\hat{\Theta}})\tau(\A))^{\ell-i}$.
We take the unique mapping $\bar{\bar{e}}:\{0,1,\ldots,\ell\}\rightarrow k$
satisfying the above equality.
$\bar{\bar{e}}(0)=\bar{e}(0)\neq 0$.
$ x({b_{\hat{\Gamma}/N^*}}^\vee_{\hat{\Theta}})(\tau(\A))\neq 0$.
$\Ord(\hat{P},b_{\hat{\Gamma}/N^*},$\break$ x(\hat{b})x({b_{\hat{\Gamma}/N^*}}^\vee_{\hat{\Theta}})^m)=0$.

Since $\Gamma_+(\hat{P}, \bar{\sigma}^*(\tau(\A))(\psi))$ is of $\hat{\Gamma}$-Weierstrass type, we have
\begin{equation*}\begin{split}
&\Ht(\hat{\Gamma},\Gamma_+(\hat{P}, \bar{\sigma}^*(\tau(\A))(\psi)))\\
=\:&
\Ord(\hat{P},b_{\hat{\Gamma}/N^*}, \In(\hat{P}, b_{\bar{\Theta}/N^*}, \bar{\sigma}^*(\tau(\A))(\psi)))-
\Ord(\hat{P},b_{\hat{\Gamma}/N^*},\bar{\sigma}^*(\tau(\A))(\psi))\\
\leq\:&
\Ord(\hat{P},b_{\hat{\Gamma}/N^*},\In(\hat{P}, b_{\bar{\Theta}/N^*}, \bar{\sigma}^*(\tau(\A))(\psi)))\\
=\:&
\Ord(\hat{P},b_{\hat{\Gamma}/N^*}, x(\hat{b})x({b_{\hat{\Gamma}/N^*}}^\vee_{\hat{\Theta}})^m\sum_{i\in L}\bar{e}(i) x({b_{\hat{\Gamma}/N^*}}^\vee_{\hat{\Theta}})^{\ell-i})\\
=\:&
\Ord(\hat{P},b_{\hat{\Gamma}/N^*},\sum_{i=0}^{\ell}\bar{\bar{e}}(i)
(x({b_{\hat{\Gamma}/N^*}}^\vee_{\hat{\Theta}})- x({b_{\hat{\Gamma}/N^*}}^\vee_{\hat{\Theta}})\tau(\A))^{\ell-i})\\
=\:&\ell-\max\{i\in\{0,1,\ldots,\ell\}|\bar{\bar{e}}(i)\neq 0\}
\leq \ell.
\end{split}\end{equation*}

We know that 
$\Ht(\hat{\Gamma},\Gamma_+(\hat{P}, \bar{\sigma}^*(\tau(\A))(\psi)))=\ell$,
if and only if,
$\Ord(\hat{P},b_{\hat{\Gamma}/N^*},$\hfill\break$\bar{\sigma}^*(\tau(\A))(\psi))=0$ and
$\bar{\bar{e}}(i)=0$ for any $ i\in\{1,2,\ldots,\ell\}$,
if and only if,
$\Ord(\hat{P},b_{\hat{\Gamma}/N^*},$\hfill\break$\bar{\sigma}^*(\tau(\A))(\psi))=0$ and
$\sum_{i\in L}\bar{e}(i) x({b_{\hat{\Gamma}/N^*}}^\vee_{\hat{\Theta}})^{\ell-i}=
\bar{e}(0) (x({b_{\hat{\Gamma}/N^*}}^\vee_{\hat{\Theta}})- x({b_{\hat{\Gamma}/N^*}}^\vee_{\hat{\Theta}})\tau(\A))^\ell$.

We know that $\Ht(\hat{\Gamma},\Gamma_+(\hat{P}, \bar{\sigma}^*(\tau(\A))(\psi)))\leq\ell\leq h$.

Assume  $\Ht(\hat{\Gamma},\Gamma_+(\hat{P}, \bar{\sigma}^*(\tau(\A))(\psi)))=h$.
We will deduce a contradiction from this assumption.

We have $\ell=h$. 
Note that $\mathcal{D}(\Gamma_+(P,u\prod_{\chi\in\mathcal{X}}(z+\chi)^{a(\chi)}\prod_{x\in P-\{z\}}x^{b(x)})|V)=\hat{\cap}_{\chi\in\mathcal{X}-\{0\}}\mathcal{D}(\Gamma_+(P,z+\chi)|V)$ is $\bar{H}$-simple, and any structure constant of it is an integer.
By Theorem~\ref{important}.12.(c),
we know that $c(\Gamma_+(P,\psi))=2$, the structure constant of $\mathcal{D}(\Gamma_+(P,\psi)|V)$ corresponding to the pair $(2,\bar{E})$ is an integer for any $\bar{E}\in\mathcal{F}(\bar{\Delta})-\{\bar{H}\}$ and $\Theta=\Lambda$.

We have $\dim\Theta=\dim\Lambda=\dim V-1=\dim\hat{\Theta}-1$, $\Theta=\hat{\Gamma}\Op|\hat{\Theta}$,
$(\Theta+\hat{\Gamma})\Op|\hat{\Theta}=\{0\}$, $\hat{c}_0=0$, $\hat{c}=0$,
$L=\{0,1,\ldots, \ell\}$ and $\bar{e}(i)=e(i)$ for any $i\in\{0,1,\ldots, \ell\}$.

We have $\Ht(\hat{\Gamma},\Gamma_+(\hat{P}, \bar{\sigma}^*(\tau(\A))(\psi)))=\ell$ and 
\begin{equation*}\begin{split}
&\sum_{i=0}^{\ell}e(i) x({b_{\hat{\Gamma}/N^*}}^\vee_{\hat{\Theta}})^{\ell-i}\\
=\:&\sum_{i\in L}\bar{e}(i) x({b_{\hat{\Gamma}/N^*}}^\vee_{\hat{\Theta}})^{\ell-i}\\
=\:&\bar{e}(0) (x({b_{\hat{\Gamma}/N^*}}^\vee_{\hat{\Theta}})- x({b_{\hat{\Gamma}/N^*}}^\vee_{\hat{\Theta}})(\tau(\A)))^\ell\\
=\:&e(0) (x({b_{\hat{\Gamma}/N^*}}^\vee_{\hat{\Theta}})- x({b_{\hat{\Gamma}/N^*}}^\vee_{\hat{\Theta}})(\tau(\A)))^\ell \\
\end{split}\end{equation*}\begin{equation*}\begin{split}
=\:&\sum_{i=0}^{\ell}e(0) \binom{\ell}{i} (-x({b_{\hat{\Gamma}/N^*}}^\vee_{\hat{\Theta}})(\tau(\A)))^i
x({b_{\hat{\Gamma}/N^*}}^\vee_{\hat{\Theta}})^{\ell-i}.
\end{split}\end{equation*}
We know $$e(i)=e(0) \binom{\ell}{i}(-x({b_{\hat{\Gamma}/N^*}}^\vee_{\hat{\Theta}})(\tau(\A)))^i$$
for any $i\in \{0,1,\ldots, \ell\}$.

We have
\begin{equation*}\begin{split}
&\Ps(P,\bar{A}(\psi)\cap \Gamma_+(P,\psi),\psi)=
\bar{\sigma}^*(\tau(\A))( \Ps(P,\bar{A}(\psi)\cap \Gamma_+(P,\psi),\psi))\\
=\:&
\sum_{i\in L}e(i)x(\hat{b})x({b_{\hat{\Gamma}/N^*}}^\vee_{\hat{\Theta}})^{\ell-i+m}
x((\frac{\ell-i}{\ell}) \hat{c}_0+(\frac{i}{\ell}) \hat{c})\\
=\:&\sum_{i=0}^{\ell}e(i)x(\hat{b}) x({b_{\hat{\Gamma}/N^*}}^\vee_{\hat{\Theta}})^{\ell-i+m}\\
=\:& x(\hat{b}) x({b_{\bar{H}/N^*}}^\vee_{\hat{\Theta}})^m
\sum_{i=0}^{\ell} e(0) \binom{\ell}{i}(-x({b_{\hat{\Gamma}/N^*}}^\vee_{\hat{\Theta}})(\tau(\A)))^i
x({b_{\bar{H}/N^*}}^\vee_{\hat{\Theta}})^{\ell-i}\\
=\:&e(0)x(\hat{b}+m{b_{\bar{H}/N^*}}^\vee_{\hat{\Theta}})
(x({b_{\bar{H}/N^*}}^\vee_{\hat{\Theta}})- x({b_{\hat{\Gamma}/N^*}}^\vee_{\hat{\Theta}})(\tau(\A)))^\ell.
\end{split}\end{equation*}

Recall that $\{{b_{\bar{E}/N^*}}^\vee_{\bar{\Delta}}|\bar{E}\in \mathcal{F}(\bar{\Delta})_1\}=\{f^P_y|y\in P\}$, ${b_{\bar{H}/N^*}}^\vee_{\bar{\Delta}}=f^P_z$ and $x(f^P_y)=y$ for any $y\in P$.
It is easy to see that there exists uniquely a mapping $\bar{r}:\mathcal{F}(\bar{\Delta})_1-\{\bar{H}\}\rightarrow \Z$ satisfying
${b_{\hat{\Gamma}/N^*}}^\vee_{\hat{\Theta}}=
{b_{\bar{H}/N^*}}^\vee_{\bar{\Delta}}-\sum_{\bar{E}\in \mathcal{F}(\bar{\Delta})_1-\{\bar{H}\}}\bar{r}(\bar{E}){b_{\bar{E}/N^*}}^\vee_{\bar{\Delta}}$,
since $\hat{\Gamma}=\bar{H}\subset\hat{\Theta}\subset\bar{\Delta}$.
We know that there exists uniquely a mapping $r:P-\{z\}\rightarrow\Z$ satisfying
${b_{\hat{\Gamma}/N^*}}^\vee_{\hat{\Theta}}=f^P_z-\sum_{y\in P-\{z\}}r(y)f^P_y$.
We take the unique mapping $r:P-\{z\}\rightarrow\Z$ satisfying
${b_{\hat{\Gamma}/N^*}}^\vee_{\hat{\Theta}}=f^P_z-\sum_{y\in P-\{z\}}r(y)f^P_y$.
$x({b_{\hat{\Gamma}/N^*}}^\vee_{\hat{\Theta}})=z/\prod_{y\in P-\{z\}}y^{r(y)}$.
We denote $\hat{b}_0=\hat{b}+m{b_{\hat{H}/N^*}}^\vee_{\hat{\Theta}}-\sum_{y\in P-\{z\}}\ell r(y)f^P_y\in N$ and $\gamma=x({b_{\hat{H}/N^*}}^\vee_{\hat{\Theta}})\tau(\A)\in k-\{0\}$ for simplicity.
We have
\begin{equation*}\begin{split}
\Ps(P,\bar{A}(\psi)\cap \Gamma_+(P,\psi),\psi)
&=e(0)x(\hat{b}_0)(z-\gamma\prod_{y\in P-\{z\}}y^{r(y)})^\ell\\
&=e(0)\prod_{y\in P}y^{\langle f^{P\vee}_y,\hat{b}_0\rangle}(z-\gamma\prod_{y\in P-\{z\}}y^{r(y)})^\ell.
\end{split}\end{equation*}

Since $\psi\in W(h)$, $\Ps(P,\bar{A}(\psi)\cap \Gamma_+(P,\psi),\psi)\in A$ and $A$ is a unique factorization domain, we know that $r(y)\in\Z_0$ for any $y\in P-\{z\}$ and
$\hat{b}_0=0$.
Since $c(\Gamma_+(P,\psi))=2$, we know that the $z$-top vertex of $\Gamma_+(P,\psi)$ is equal to $\ell f^P_z$, $\ell f^P_z\in
\bar{A}(\psi)\cap \Gamma_+(P,\psi)$ and $ \bar{A}(\psi)\cap \Gamma_+(P,\psi)$ is a $z$-removable face of $\Gamma_+(P,\psi)$.

Since $\Gamma_+(P,\psi)$ has no $z$-removable faces, we obtain a contradiction.
We know that $\Ht(\hat{\Gamma},\Gamma_+(\hat{P}, \bar{\sigma}^*(\tau(\A))(\psi)))<h$.
We denote $\bar{g}=\Ht(\hat{\Gamma},\Gamma_+(\hat{P}, \bar{\sigma}^*(\tau(\A))($\break$\psi)))\in\Z_0$. $\bar{g}<h$.

We denote $\bar{z}=x({b_{\hat{\Gamma}/N^*}}^\vee_{\hat{\Theta}})-
x({b_{\hat{\Gamma}/N^*}}^\vee_{\hat{\Theta}})\tau(\A)\in\hat{P}$ for simplicity.

Note that $k[\hat{P}-\{\bar{z}\}]\subset 
\mathcal{O}(\widetilde{\mathcal{C}})_{\tau(\A)}^c$.
By $\mathcal{O}(\widetilde{\mathcal{C}})_{\tau(\A)}^{c\:\prime}$ we denote the
completion of $k[\hat{P}-\{\bar{z}\}]$ with respect to the maximal ideal $k[\hat{P}-\{\bar{z}\}]\cap M(\mathcal{O}(\widetilde{\mathcal{C}})_{\tau(\A)}^c)$.
$\mathcal{O}(\widetilde{\mathcal{C}})_{\tau(\A)}^{c\:\prime}$ is a complete local subring of 
$\mathcal{O}(\widetilde{\mathcal{C}})_{\tau(\A)}^c$.
$\mathcal{O}(\widetilde{\mathcal{C}})_{\tau(\A)}^{c\:\prime}$ is a $k$-subalgebra of $\mathcal{O}(\widetilde{\mathcal{C}})_{\tau(\A)}^c$.

Since $\Gamma_+(\hat{P}, \bar{\sigma}^*(\tau(\A))(\psi))$ is of $\hat{\Gamma}$-Weierstrass type, $\bar{g}=\Ht(\hat{\Gamma},\Gamma_+(\hat{P}, \bar{\sigma}^*(\tau(\A$\hfill\break$))(\psi)))$ and Weierstrass' preparation theorem holds, there exist uniquely an element $v\in (\mathcal{O}(\widetilde{\mathcal{C}})_{\tau(\A)}^c)^\times$ and a mapping
$\bar{\psi}:\{0,1,\ldots,\bar{g}-1\}\rightarrow M(\mathcal{O}(\widetilde{\mathcal{C}})_{\tau(\A)}^{c\:\prime})$ satisfying
$\bar{\sigma}^*(\tau(\A))(\psi)=vx(\hat{b})(\bar{z}^{\bar{g}}+\sum_{i=0}^{\bar{g}-1}\bar{\psi}(i)\bar{z}^i)$.
We take the unique pair $v\in (\mathcal{O}(\widetilde{\mathcal{C}})_{\tau(\A)}^c)^\times$ and a mapping
$\bar{\psi}:\{0,1,\ldots,\bar{g}-1\}\rightarrow M(\mathcal{O}(\widetilde{\mathcal{C}})_{\tau(\A)}^{c\:\prime})$ satisfying
this equality.

Let $\mathcal{R}=\{\chi\in M(\mathcal{O}(\widetilde{\mathcal{C}})_{\tau(\A)}^{c\:\prime})|\chi^{\bar{g}}+\sum_{i=0}^{\bar{g}-1}\bar{\psi}(i)\chi^i=0\}$,
$\mathcal{R}$ is a finite set and $\sharp \mathcal{R}\leq\bar{g}$.
Consider any $\chi\in\mathcal{R}$. 
Since $\mathcal{O}(\widetilde{\mathcal{C}})_{\tau(\A)}^c$ is a unique factorization domain, there exists uniquely a positive integer $\mu(\chi)\in\Z_+$ satisfying 
$\bar{z}^{\bar{g}}+\sum_{i=0}^{\bar{g}-1}\bar{\psi}(i)\bar{z}^i \in (\bar{z}-\chi)^{\mu(\chi)} \mathcal{O}(\widetilde{\mathcal{C}})_{\tau(\A)}^c$
and
$\bar{z}^{\bar{g}}+\sum_{i=0}^{\bar{g}-1}\bar{\psi}(i)\bar{z}^i \not\in (\bar{z}-\chi)^{\mu(\chi)+1} \mathcal{O}(\widetilde{\mathcal{C}})_{\tau(\A)}^c$.
We take the unique $\mu(\chi)\in\Z_+$ satisfying these conditions.
$\sum_{\chi\in\mathcal{R}}\mu(\chi)\leq\bar{g}$.
Let $\hat{g}=\bar{g}-\sum_{\chi\in\mathcal{R}}\mu(\chi)\in\Z_0$.
$\hat{g}\neq 1$.
$\hat{g}\leq\bar{g}<h$.
Let $g=\max\{\hat{g},1\}\in\Z_+$.
Since $h\geq 2$, $g<h$.

There exists uniquely a mapping $\hat{\psi}:\{0,1,\ldots,\hat{g}-1\}\rightarrow M(\mathcal{O}(\widetilde{\mathcal{C}})_{\tau(\A)}^{c\:\prime})$ satisfying
$\bar{\sigma}^*(\tau(\A))(\psi)=vx(\hat{b})\prod_{\chi\in\mathcal{R}}
(\bar{z}-\chi)^{\mu(\chi)}(\bar{z}^{\hat{g}}+\sum_{i=0}^{\hat{g}-1}\hat{\psi} (i)\bar{z}^i)$.
We take the unique mapping $\hat{\psi}:\{0,1,\ldots,\hat{g}-1\}\rightarrow M(\mathcal{O}(\widetilde{\mathcal{C}})_{\tau(\A)}^{c\:\prime})$ satisfying this equality.
$\chi^{\hat{g}}+\sum_{i=0}^{\hat{g}-1}\hat{\psi} (i)\chi^i\neq 0$ for any
$\chi\in M(\mathcal{O}(\widetilde{\mathcal{C}})_{\tau(\A)}^{c\:\prime})$.

$\bar{\rho}( vx(\hat{b})\prod_{\chi\in\mathcal{R}}
(\bar{z}-\chi)^{\mu(\chi)})\in PW(1)$.

If $\hat{g}\geq 2$, then $\hat{g}=g$ and
$\bar{\rho}(\bar{z}^{\hat{g}}+\sum_{i=0}^{\hat{g}-1}\hat{\psi} (i)\bar{z}^i)\in W(\hat{g})=W(g)$.
Therefore, $\bar{\rho}\bar{\sigma}^*(\tau(\A))(\phi)=
\bar{\rho}\bar{\sigma}^*(\tau(\A))(u\prod_{\chi\in\mathcal{X}}(z+\chi)^{a(\chi)}\prod_{x\in P-\{z\}} x^{b(x)})
\bar{\rho}(vx(\hat{b})\prod_{\chi\in\mathcal{R}}
(\bar{z}-\chi)^{\mu(\chi)})
\bar{\rho}(\bar{z}^{\hat{g}}+\sum_{i=0}^{\hat{g}-1}\hat{\psi} (i)\bar{z}^i)
\in PW(g)$.

If $\hat{g}\leq 1$, then $\hat{g}=0$, $g=1$ and 
$\bar{\rho}(\bar{z}^{\hat{g}}+\sum_{i=0}^{\hat{g}-1}\hat{\psi} (i)\bar{z}^i)=1\in PW(1)$.
Therefore, $\bar{\rho}\bar{\sigma}^*(\tau(\A))(\phi)=
\bar{\rho}\bar{\sigma}^*(\tau(\A))(u\prod_{\chi\in\mathcal{X}}(z+\chi)^{a(\chi)}\prod_{x\in P-\{z\}} x^{b(x)})
\bar{\rho}(vx(\hat{b})\prod_{\chi\in\mathcal{R}}
(\bar{z}-\chi)^{\mu(\chi)})
\bar{\rho}(\bar{z}^{\hat{g}}+\sum_{i=0}^{\hat{g}-1}\hat{\psi} (i)\bar{z}^i)
\in PW(1)=PW(g)$.

We conclude that $\bar{\rho}\bar{\sigma}^*(\tau(\A))(\phi)\in PW(g)$ in all cases.

We know that there exists uniquely an element $\B\in(\sigma^*D)_0$ with $\tau(\B)=\bar{\B}$. We take the unique element $\B\in(\sigma^*D)_0$ with $\tau(\B)=\bar{\B}$.
Since $\tau(\A)\in U(\bar{\Sigma},\bar{\sigma}^*\bar{D},\bar{\B})$, $\A\in\tau^{-1}( U(\bar{\Sigma},\bar{\sigma}^*\bar{D},\bar{\B}))=
U(\Sigma, \sigma^*D,\B)$.
It is easy to see that the homomorphism $\tau^*(\A): \mathcal{O}(\widetilde{\mathcal{C}})_{\tau(\A)}\rightarrow\mathcal{O}_{\Sigma, \A}$ induces an isomorphism $\tau^*(\A):\mathcal{O}(\widetilde{\mathcal{C}})_{\tau(\A)}^c\rightarrow\mathcal{O}_{\Sigma, \A}^c$ of complete $k$-algebras.
Since $\sigma(\A)=M(A)$, we have a homomorphism of $k$-algebras $\sigma^*(\A):A\rightarrow\mathcal{O}_{\Sigma, \A}$ satisfying
$\sigma^*(\A)(M(A))\subset M(\mathcal{O}_{\Sigma, \A})$. This homomorphism has the unique extension $\sigma^*(\A):A \rightarrow\mathcal{O}_{\Sigma, \A}^c$.
Since $\iota^*\sigma=\bar{\sigma}\tau$, we have $\sigma^*(\A)(\phi)=\tau^*(\A) \bar{\sigma}^*(\tau(\A))(\phi)$.

Let $\bar{P}=\{(\sigma^*\xi)_\B(E)- (\sigma^*\xi)_\B(E)(\A)|E\in\Comp(\sigma^*D)(\B)\}$.
$\bar{P}$ is a parameter system of $\mathcal{O}_{\Sigma, \A}^c$.
Since $\hat{P}=\{(\bar{\sigma}^*\bar{\xi})_{\bar{\B}}(E)- (\bar{\sigma}^*\bar{\xi})_{\bar{\B}}(E)(\tau(\A))|E\in\Comp(\bar{\sigma}^*\bar{D})(\bar{\B})\}$, $\bar{P}=\tau^*(\A)(\hat{P})$.

Let $\rho=\bar{\rho}\tau^*(\A)^{-1}:\mathcal{O}_{\Sigma, \A}^c\rightarrow A$.
$\rho$ is an isomorphism of $k$-algebras.
$\rho(\bar{P})=\bar{\rho}\tau^*(\A)^{-1}(\tau^*(\A)(\hat{P}))=\bar{\rho}(\hat{P})=P$. We know $\rho(\bar{P})=P$.
$\rho\sigma^*(\A)(\phi)= \bar{\rho}\tau^*(\A)^{-1}\tau^*(\A)$\break$ \bar{\sigma}^*(\tau(\A))(\phi)
=\bar{\rho}\bar{\sigma}^*(\tau(\A))(\phi)\in PW(g)$.
We know $\rho\sigma^*(\A)(\phi)\in PW(g)$.

We conclude that Theorem~\ref{main} holds, if $\dim\Lambda=\dim V-1$ and $\dim\Lambda(\psi)=\dim V-1$.

We conclude that Theorem~\ref{main} holds in all cases.

\section{Proof of the submain theorems}
\label{submain proofs}
We give the proof of our submain theorems Theorem~\ref{erase faces}, Theorem~\ref{make simple}, Theorem~\ref{make Weierstrass type} and Theorem~\ref{make normal crossings}.

Let $k$ be any algebraically closed field; let $A$ be any complete regular local ring such that $A$ contains $k$ as a subring, the residue field $A/M(A)$ is isomorphic to $k$ as $k$-algebras, and $\dim A\geq 2$; let $P$ be any parameter system of $A$, and let $z\in P$ be any element.

Let $A'$ denote the completion of $k[P - \{z\}]$ with respect to the maximal ideal $k[P - \{z\}]\cap M(A)$. The ring $A'$ is a local subring of $A$ and $M(A')=M(A)\cap A' =(P -\{z\})A'$. The set $P - \{z\}$ is a parameter system of $A'$.

Let $\Delta=\Spec(A/\prod_{x\in P}xA)$, and $\Delta'=\Spec(A'/\prod_{x\in P -\{z\}}xA')$. We define a coordinate system $\xi_{M(A)}:\Comp(\Delta)\rightarrow A$ of the normal crossing scheme $(\Spec(A),\Delta)$ at $M(A)$ by putting $\xi_{M(A)}(\Spec(A/xA))=x$ for any $x\in P$. Let $\xi=\{ \xi_{M(A)}\}$. We define a coordinate system $\xi'_{M(A')}:\Comp(\Delta')\rightarrow A'$ of the normal crossing scheme $(\Spec(A'),\Delta')$ at $M(A')$ by putting $\xi_{M(A')}'(\Spec(A'/xA'))=x$ for any $x\in P-\{z\}$. Let $\xi'=\{ \xi_{M(A')} '\}$. The triplets $(\Spec(A),\Delta, \xi)$ and $(\Spec(A'),\Delta', \xi')$ are coordinated normal crossing schemes over $k$.

Recall the following notations:
\begin{equation*}\begin{split}
&PW(1)= \{\phi\in A|\phi= u\prod_{\chi\in\mathcal{X}}(z+\chi)^{a(\chi)}\prod_{x\in P-\{z\}} x^{b(x)} \\
&\quad\text{for some }u\in A^\times, \text{ some finite subset } \mathcal{X} \text{ of } M(A'),\\
&\quad \text{some mapping }a: \mathcal{X}\rightarrow \Z_+, \text{ and some mapping }b:P-\{z\}\rightarrow\Z_0\}.\\
\end{split}\end{equation*}
\vfill

For any $h\in\Z_+$ with $h\geq 2$ we denote
\begin{equation*}\begin{split}
W(h)&=\{\phi\in A|\phi= z^h+\sum_{i=0}^{h-1} \phi'(i)z^i \\
&\qquad\quad \text{for some mapping }\phi':\{0,1,\ldots,h-1\}\rightarrow M(A') \text{ satisfying }\\
&\qquad\quad 
\chi^h+\sum_{i=0}^{h-1} \phi'(i)\chi^i\neq 0 \text{ for any }
\chi\in M(A').\},\\
PW(h)&=\{\phi\in A|\phi=\psi\psi'\text{ for some } \psi\in W(h) \text{ and some }\psi'\in PW(1).\},\\
RW(h)&=\{\phi\in A|\text{ with }\phi=\psi\psi'\text{ for some } \psi\in W(h) \text{ and some }\psi'\in PW(1)\\
&\qquad\quad \text{such that }\Gamma_+(P,\psi) \text{ has no }z \text{-removable faces.}\},\\
SW(h)&=\{\phi\in A|\phi=\psi\psi'\text{ for some } \psi\in W(h) \text{ and some }\psi'\in PW(1)\text{ such that}\\
&\qquad\quad \Gamma_+(P,\psi) \text{ has no }z \text{-removable faces, and }\Gamma_+(P,\phi) \text{ is }z \text{-simple.}\}.
\end{split}\end{equation*}

We give the proof of Theorem~\ref{erase faces}.

Consider any $h\in\Z_+$ with $h\geq 2$ and any $\phi\in PW(h)$.

We take an element $u\in A^\times$, a finite subset $\mathcal{X}$ of $M(A')$, a mapping $a:\mathcal{X}\rightarrow\Z_+$, a mapping $b:P-\{z\}\rightarrow\Z_0$ and an element $\psi\in W(h)$ satisfying
$$\phi=\psi u\prod_{\chi\in\mathcal{X}}(z+\chi)^{a(\chi)}\prod_{x\in P-\{z\}} x^{b(x)}.$$

We take a mapping $\psi':\{0,1,\ldots,h-1\}\rightarrow M(A')$ satisfying
$\psi= z^h+\sum_{i=0}^{h-1} \psi'(i)z^i$
and $\bar{\chi}^h+\sum_{i=0}^{h-1} \psi'(i)\bar{\chi}^i\neq 0$ for any $\bar{\chi}\in M(A')$.

The quintuplet $(u, \mathcal{X}, a, b, \psi')$ is uniquely determined depending on $\phi$, since $A$ is a unique factorization domain and Weierstrass' preparation theorem holds.

In particular, $\psi'(0)= (-0)^h+\sum_{i=0}^{h-1} \psi'(i)(-0)^i\neq 0$ and $\Ord(P, f^{P\vee}_z, \psi)=0$

We denote $V=\Map(P,\R)$, $N=\Map(P,\Z)$, $\Delta=\Map(P,\R_0)$, $W=\{a\in V|a(z)=0\}$ and $U_-=\{a\in V|a(z)<h\}$.

$V$ is a finite dimensional vector space over $\R$. $\dim V=\sharp P=\dim A$.
$N$ is a lattice of $V$.
The set $\{f^P_x|x\in P\}$ is an $\R$-basis of $V$, and it is a $\Z$-basis of $N$.
$\Delta$ is a simplicial cone over $N$ in $V$. $\dim\Delta=\dim V$.
$W$ is a vector subspace of $V$ over $\R$. $\dim W=\dim V-1$.
$N\cap W$ is a lattice of $W$.
$\Delta\cap W$ is a simplicial cone oner $N\cap W$ in $W$.
$\dim\Delta\cap W=\dim W$.
The set $\{f^P_x|x\in P-\{z\}\}$ is an $\R$-basis of $W$, and it is a $\Z$-basis of $N\cap W$. $U_-$ is a non-empty open subset of $V$.

The dual vector space $V^*$ of $V$ is a vector space over $\R$ and the dual basis
$\{f^{P\vee}_x|x\in P\}$ of  $\{f^P_x|x\in P\}$ is an $\R$-basis of $V^*$. 
The dual lattice $N^*$ of $N$ is a lattice of $V^*$.
The set $\{f^{P\vee}_x|x\in P\}$ is a $\Z$-basis of $N^*$.
$\langle f^{P\vee}_x, a\rangle=a(x)$ for any $x\in P$ and any $a\in\Map(P,\R)$.
The dual cone $\Delta^\vee|V$ of $\Delta$ is a simplicial cone over $N^*$ in $V^*$.
$\dim \Delta^\vee|V=\dim V$.

We identify the dual vector space $W^*$ of $W$ and the vector subspace $\Vect(\{f^{P\vee}_x|x\in P-\{z\}\})$ of $V^*$.
$N^*\cap W^*$ is a lattice of $W^*$.
The intersection $(\Delta^\vee|V)\cap W^*$ is a simplicial cone over $N^*\cap W^*$ in $W^*$.
$\dim (\Delta^\vee|V)\cap W^*=\dim W$.
$(\Delta^\vee|V)\cap W^*=(\Delta\cap W)^\vee|W$.

We put $\sigma(a)=(a-a(z)f^P_z)/(h-a(z))\in W$ for any $a\in U_-$, and 
we define a mapping $\sigma:U_-\rightarrow W$.

We consider any $\chi\in M(A')$.

We denote $P(\chi)=(P-\{z\})\cup\{z+\chi\}\subset M(A)$.
The set $P(\chi)$ is a parameter system of $A$.

Let $\iota:P(\chi)\rightarrow P$ denote the bijective mapping satisfying $\iota(x)=x$ for any $x\in P-\{z\}$ and $\iota(z+\chi)=z$.
The mapping $\iota$ induces an isomorphism $\iota^*:V=\Map(P,\R)\rightarrow\Map(P(\chi),\R)$ of vector spaces over $\R$.
Using $\iota^*$, we identify $V$ and $\Map(P(\chi),\R)$.
We have $f^{P(\chi)}_{z+\chi}=f^P_z$ and $f^{P(\chi)}_x=f^P_x$ for any $x\in P-\{z\}$.

$\emptyset\neq\Gamma_+(P(\chi),\psi)\subset\Delta\subset V$.
$\Gamma_+(P(\chi),\psi)$ is a Newton polyhedron over $N$ in $V$.

The following claims hold. See Lemma~\ref{first correspondence},  Proposition~\ref{miracle} and Lemma~\ref{basic ord}.16:
\begin{enumerate}
\item
There exists uniquely a mapping $\bar{\psi}':\{0,1,\ldots,h-1\}\rightarrow M(A')$ satisfying
$\psi= (z+\chi)^h+\sum_{i=0}^{h-1} \bar{\psi}'(i)(z+\chi)^i$
and $\bar{\chi}^h+\sum_{i=0}^{h-1} \bar{\psi}'(i)\bar{\chi}^i\neq 0$ for any $\bar{\chi}\in M(A')$.
\item
$\Stab(\Gamma_+(P(\chi),\psi))=\Delta$.
\item
$\Gamma_+(P(\chi),\psi)$ is of $z$-Weierstrass type. The unique $z$-top vertex of $\Gamma_+(P(\chi),\psi)$ is equal to $\{hf^P_z\}$.
\item
$\Ord(P(\chi),f^{P\vee}_z,\psi)=0$. $\Ht(z, \Gamma_+(P(\chi),\psi))=h$.
\item
$\Gamma_+(P(\chi),\psi) \cap U_-\neq\emptyset$. $\sigma(\Gamma_+(P(\chi),\psi) \cap U_-)\neq\emptyset$.
\end{enumerate}

Below, we denote $$\bar{\Gamma}_+(P(\chi),\psi)= \sigma(\Gamma_+(P(\chi),\psi) \cap U_-).$$
\begin{enumerate}
\setcounter{enumi}{5}
\item
$\bar{\Gamma}_+(P(\chi),\psi)=
\sigma((\Convcone(\Gamma_+(P(\chi),\psi)+\{-hf^P_z\})+\{ hf^P_z \})\cap U_-)$ and
$\bar{\Gamma}_+(P(\chi),\psi)$ is a rational convex pseudo polyhedron over $N\cap W$ in $W$.

$0\not\in\bar{\Gamma}_+(P(\chi),\psi)\subset\Delta\cap W$.
$\Stab(\bar{\Gamma}_+(P(\chi),\psi))=\Delta\cap W$.

\item
Consider any face $F$ of $\Gamma_+(P(\chi),\psi)$ satisfying $hf^P_z\in F$ and
$F\cap U_-\neq \emptyset$.

$\sigma(F\cap U_-)$ is a face of $\bar{\Gamma}_+(P(\chi),\psi)$.

Below, we denote $\bar{F}=\sigma(F\cap U_-)$.

$\dim F\geq 1$ and $\dim \bar{F}=\dim F-1$.

If $\dim F=1$, then $\Stab(F)=\{0\}$.

Consider any $\omega\in\Delta^\circ(F, \Gamma_+(P(\chi),\psi)|V)$.
If we take unique pair of elements $\bar{\omega}\in (\Delta^\vee|V)\cap W^*$ and $t\in\R_0$ satisfying $\omega=\bar{\omega}+tf^{P\vee}_z$, then
$\bar{\omega}\in\Delta^\circ(\bar{F}, \bar{\Gamma}_+(P(\chi),\psi)|W)$ and $t=\Ord(\bar{\omega},\bar{\Gamma}_+(P(\chi),\psi)|W)$.

Consider any $\bar{\omega}\in\Delta^\circ(\bar{F}, \bar{\Gamma}_+(P(\chi),\psi)|W)$.
If $t=\Ord(\bar{\omega},\bar{\Gamma}_+(P(\chi),\psi)|W)$, then $t\in\R_0$ and $\bar{\omega}+tf^{P\vee}_z\in \Delta^\circ(F, \Gamma_+(P(\chi),\psi)|V)$.
\item
The mapping from the set of faces of $\Gamma_+(P(\chi),\psi)$ satisfying $hf^P_z\in F$ and
$F\cap U_-\neq \emptyset$ to the set of faces of $\bar{\Gamma}_+(P(\chi),\psi)$ sending $F$ to $\sigma(F\cap U_-)$ is bijective and it preserves the inclusion relation.
\item
Consider any face $F$ of $\Gamma_+(P(\chi),\psi)$.

$F$ is a $z$-removable face, if and only if, $hf^P_z\in F$ and there exists $\bar{\chi}\in M(A')$ satisfying $\Ps(P(\chi), F,\psi)=(z+\chi+\bar{\chi})^h$ and $\bar{\chi}\neq 0$.
\item
Any $z$-removable face $F$ of $\Gamma_+(P(\chi),\psi)$ satisfies $hf^P_z\in F$, 
and $F\cap U_-\supset F\cap W\neq \emptyset$.

\item
If $\Gamma_+(P(\chi),\psi)$ has a $z$-removable face, then it has a $z$-removable face of dimension one.
\item
Assume that $\Gamma_+(P(\chi),\psi)$ has a $z$-removable face.
We consider any $z$-removable face $F$ of dimension one of $\Gamma_+(P(\chi),\psi)$.
\begin{enumerate}
\item
$\sigma(F\cap U_-)$ is a vertex of $\bar{\Gamma}_+(P(\chi),\psi)$.
\end{enumerate}

We take the unique point $c(F)\in\bar{\Gamma}_+(P(\chi),\psi)$ with $\{c(F)\}=\sigma(F\cap U_-)$.
\begin{enumerate}
\setcounter{enumii}{1}
\item
The set $\{c(F)\}$ is a vertex of $\bar{\Gamma}_+(P(\chi),\psi)$. $c(F)\in \bar{\Gamma}_+(P(\chi),\psi)\cap N$. $c(F)\neq 0$. $c(F)\in\Delta\cap W\cap N-\{0\}$.
\item
There exists uniquely an element $\gamma(F)\in k-\{0\}$ satisfying
$\Ps(P(\chi), F,$\break$\psi)=(z+\chi+\gamma(F)\prod_{x\in P-\{z\}}x^{\langle f^{P\vee}_x,c(F)\rangle})^h$.
\end{enumerate}

We take the unique element $\gamma(F)\in k-\{0\}$ satisfying
$\Ps(P(\chi), F,\psi)=(z+\chi+\gamma(F)\prod_{x\in P-\{z\}}x^{\langle f^{P\vee}_x,c(F)\rangle})^h$ and we denote
\begin{equation*}\begin{split}
\chi(F)&= \gamma(F)\prod_{x\in P-\{z\}}x^{\langle f^{P\vee}_x,c(F)\rangle}\in M(A')-\{0\}\text{, and}\\
\bar{\delta}_0&=\sum_{x\in P-\{z\}}f^{P\vee}_x=b_{(\Delta^\vee|V)\cap W^* /N^*}\in((\Delta^\vee|V)\cap W^*)^\circ\cap N^*.
\end{split}\end{equation*}

\begin{enumerate}
\setcounter{enumii}{3}
\item
$\Ps(P(\chi), F,\psi)=(z+\chi+\chi(F))^h$.
\item

$c(F)\not\in\bar{\Gamma}_+(P(\chi+\chi(F)),\psi)\subset\bar{\Gamma}_+(P(\chi),\psi)$.
Any face $\bar{G}$ of $\bar{\Gamma}_+(P(\chi),\psi)$ satisfying $c(F)\not\in\bar{G}$ is a face of $\bar{\Gamma}_+(P(\chi+\chi(F)),\psi)$.

$\Ord(P-\{z\}, \bar{\delta}_0, \chi(F))=\langle  \bar{\delta}_0, c(F)\rangle$.
$\In(P-\{z\}, \bar{\delta}_0, \chi(F))=\chi(F)$. $\Supp(P-\{z\}, \chi(F))=\{c(F)\}$.
\end{enumerate}
\end{enumerate}

Below, we use the above notations $\bar{\Gamma}_+(P(\chi),\psi)= \sigma(\Gamma_+(P(\chi),\psi)\cap U_-)$, $c(F)\in(\Delta\cap W\cap N)-\{0\}$, $\gamma(F)\in k-\{0\}$, $\chi(F)\in M(A')-\{0\}$ and $\bar{\delta}_0\in((\Delta^\vee|V)\cap W^*)^\circ\cap N^*$.

We consider the following algorithm starting from Step 0.

In Step 0 we put $\bar{\chi}(0)=0\in M(A')$ and proceed to Step 1. 

Consider any positive integer $i$. In Step $i$, if $\Gamma_+(P(\bar{\chi} (i-1)),\psi)$ has no $z$-removable faces, then we finish the algorithm. In Step $i$, if $\Gamma_+(P(\bar{\chi}(i-1)),\psi)$ has $z$-removable faces, then we choose any $z$-removable face $F(i)$ of dimension one of $\Gamma_+(P(\bar{\chi}(i-1)),\psi)$ satisfying
$\langle \bar{\delta}_0,c(F(i))\rangle=\min\{\langle \bar{\delta}_0,c(F)\rangle|
F$ is a $z$-removable face of dimension one of $\Gamma_+(P(\bar{\chi}(i-1)),\psi)\}$, we put $\bar{\chi}(i)= \bar{\chi}(i-1)+\chi(F(i))\in M(A')$ and we proceed to Step $i+1$.

Consider the case where we finish this algorithm in finite steps.
Assume that the algorithm has finished in Step $i$ for some positive integer $i$.
$\bar{\chi}(i-1)\in M(A')$ and $\Gamma_+(P(\bar{\chi}(i-1)),\psi)$ has no $z$-removable faces.
We conclude that there exists $\bar{\chi}\in M(A')$ such that $\Gamma_+(P(\bar{\chi}),\psi)$ has no $z$-removable faces.

Consider the case where this algorithm has infinite steps.
By 12.(e) and 12.(b) we know that
$\chi(F(i))\in M(A')$, $\Ord(P-\{z\},\bar{\delta}_0, \chi(F(i)))
=\langle \bar{\delta}_0,c(F(i))\rangle$, 
$c(F(i))\in \Delta\cap W\cap N$, and
$\bar{\chi}(i)=\sum_{j=1}^i\chi(F(j))\in M(A')$ for any $i\in \Z_+$.
By 12.(e) we know that
$c(F(i))\neq c(F(j))$ for any $i\in \Z_+$ and any $j\in \Z_+$ with $i\neq j$.

Since $\{e\in \Delta\cap W\cap N|\langle \bar{\delta}_0,e\rangle\leq m\}$ is a finite set for any $m\in\Z_0$, we know that $\lim_{i\rightarrow\infty}\langle \bar{\delta}_0,c(F(i))\rangle=\infty$ and the sequence $\bar{\chi}(i)$, $i\in\Z_+$ converges. We put $\bar{\chi}=\lim_{i\rightarrow\infty}\bar{\chi}(i)=
\sum_{i=1}^\infty\chi(F(i))\in M(A')$.

Assume that $\Gamma_+(P(\bar{\chi}),\psi)$ has $z$-removable faces.
We will deduce a contradiction. Take any $z$-removable face $F$ of dimension one of $\Gamma_+(P(\bar{\chi}),\psi)$. 
Take any $\bar{\omega}\in\Delta^\circ(\{c(F)\},\bar{\Gamma}_+(P(\bar{\chi}),\psi)|W)\subset((\Delta^\vee|V)\cap W^*)^\circ$.
Put $t=\Ord(\bar{\omega}, \bar{\Gamma}_+(P(\bar{\chi}),\psi)|W)\in\R_0$ and $\omega=\bar{\omega}+tf^{P\vee}_z\in \Delta^\circ(F, \Gamma_+(P(\bar{\chi}),\psi)|V)$.
We know that $t=\langle \bar{\omega},c(F)\rangle$ and $\In(P(\bar{\chi}), \omega,\psi)=\Ps(P(\bar{\chi}), F,\psi)=(z+\bar{\chi}+\chi(F))^h$.

Since $\bar{\omega}\in((\Delta^\vee|V)\cap W^*)^\circ$,
$\{e\in \Delta\cap W\cap N|\langle \bar{\omega},e\rangle\leq t \}$ is a finite set.
Take any $i\in\Z_+$ such that $\langle \bar{\delta}_0,c(F(j))\rangle>\langle \bar{\delta}_0,c(F)\rangle$ and $\langle \bar{\omega},c(F(j))\rangle>t$ for any $j\in \Z_+$ with $j\geq i$.

$\Ord(P(\bar{\chi}),\omega,z+\bar{\chi})=\langle \omega, f^P_z\rangle =t$.
$(z+\bar{\chi})-(z+\bar{\chi}(i))=\sum_{j=i+1}^\infty \chi(F(j))$.
For any $j\in\Z_+$ with $j\geq i+1$, $\Ord(P(\bar{\chi}),\omega, \chi(F(j)))=\Ord(P-\{z\},\bar{\omega},\chi(F(j)))=\langle \bar{\omega},c(F(j))\rangle>t$.
Therefore, $\Ord(P(\bar{\chi}),\omega, (z+\bar{\chi})-(z+\bar{\chi}(i)))= \Ord(P(\bar{\chi}),\omega, \sum_{j=i+1}^\infty $\break$\chi(F(j)))>t$ and
$\Ord(P(\bar{\chi}),\omega,z+\bar{\chi})= \Ord(P(\bar{\bar{\chi}}),\omega,z+\chi(i))$.
By Lemma~\ref{basic ord}.16, we know that $\Ord(P(\bar{\chi}),\omega,\zeta)= \Ord(P(\bar{\chi}(i)),\omega,\zeta)$ for any $\zeta\in A$, $\In(P(\bar{\chi}(i)), \omega,\psi)=(z+\bar{\chi}(i)+\chi(F))^h$,
$F=h\Conv(\{f^P_z,c(F)\})=\Conv(\Supp(P(\bar{\chi}(i)), \In(P(\bar{\chi}(i)), \omega,\psi$\break$)))=\Delta(\omega, \Gamma_+(P(\bar{\chi}(i)),\psi)|V)\in \mathcal{F}(\Gamma_+(P(\bar{\chi}(i)),\psi))$ and $F$ is a $z$-removable face of dimension one of $\Gamma_+(P(\bar{\chi}(i)),\psi)$.
By how to choose $F(i)$ we have
$\langle \bar{\delta}_0,c(F(i))\rangle\leq\langle \bar{\delta}_0,c(F)\rangle$.

Since $\langle \bar{\delta}_0,c(F(i))\rangle>\langle \bar{\delta}_0,c(F)\rangle$, we obtain a contradiction. 

We conclude that $\Gamma_+(P(\bar{\chi}),\psi)$ has no $z$-removable faces.

We know that there exists $\bar{\chi}\in M(A')$ such that $\Gamma_+(P(\bar{\chi}),\psi)$ has no $z$-removable faces in all cases.

We take any $\bar{\chi}\in M(A')$ such that $\Gamma_+(P(\bar{\chi}),\psi)$ has no $z$-removable faces.
Let $\rho:A\rightarrow A$ denote the unique isomorphism of $k$-algebras satisfying $\rho(z+\bar{\chi})=z$ and $\rho(x)=x$ for any $x\in P-\{z\}$.
We know that $\rho(\psi)\in W(h)$, $\Gamma_+(P,\rho(\psi))$ has no $z$-removable faces and $\rho(u)\in A^\times$.
For any $\chi\in\mathcal{X}$, $\rho(z+\chi(i))=z+\chi(i)-\bar{\chi}$ and $\chi(i)-\bar{\chi}\in M(A')$.

Since 
$$\rho(\phi)=\rho(\psi) \rho(u)\prod_{\chi\in\mathcal{X}}(z+\chi-\bar{\chi})^{a(\chi)}\prod_{x\in P-\{z\}} x^{b(x)},$$
We know $\rho(\phi)\in RW(h)$.

We know that Theorem~\ref{erase faces} holds.

Note here that $\dim A'=\dim A-1<\dim A$, and any $\phi'\in A'$ with $\phi'\neq 0$ has normal crossings over $P'$ if $\dim A=2$. Therefore, we decide that we use induction on $\dim A$, and we can assume the following claim $(*)$:

\begin{description}
\item[$(*)$]
For any  $\phi'\in A'$ with $\phi'\neq 0$, there exists a weakly admissible composition of blowing-ups $\sigma':\Sigma'\rightarrow\Spec(A')$ over $(\Delta',\xi')$ and an extended pull-back $(\Sigma',\bar{\Delta}',\bar{\xi}')$ of the coordinated normal crossing scheme $(\Spec(A'),\Delta',\xi')$ by $\sigma'$ satisfying $\Supp(\sigma^{\prime*}(\Spec(A'/\phi' A')+ \Delta'))\subset \Supp(\bar{\Delta}')$.
\end{description}

In the case of $\dim A=2$, putting $(\Sigma',\bar{\Delta}',\bar{\xi}')=(\Spec(A'),\Delta',\xi')$ and considering the identity morphism $\sigma':\Sigma'\rightarrow\Spec(A')=\Sigma'$, we know that $\sigma'$ is a weakly admissible composition of blowing-ups over $(\Delta',\xi')$ and $\Supp(\sigma^{\prime*}(\Spec(A'/\phi' A')+ \Delta'))\subset \Supp(\bar{\Delta}')$ for any $\phi'\in A'$ with $\phi'\neq 0$.

Let $\sigma':\Sigma'\rightarrow\Spec(A')$ be any weakly admissible composition of blowing-ups $\sigma':\Sigma'\rightarrow\Spec(A')$ over $(\Delta',\xi')$, and let $(\Sigma',\bar{\Delta}',\bar{\xi}')$ be an extended pull-back of the coordinated normal crossing scheme $(\Spec(A'),\Delta',\xi')$ by $\sigma'$. We consider a morphism $\Spec(A)\rightarrow\Spec(A')$ induced by the inclusion ring homomorphism $A'\rightarrow A$, the product scheme $\Sigma=\Sigma'\times_{\Spec(A')}\Spec(A)$, the projection $\sigma:\Sigma\rightarrow\Spec(A)$, and the projection $\pi:\Sigma\rightarrow\Sigma'$. We know the following (See Lemma~\ref{pull back blowing-ups}.):
\begin{enumerate}
\item The morphism $\sigma$ is a weakly admissible composition of blowing-ups over $(\Delta,\xi)$.
\item The pull-back $\sigma^*\Spec(A/zA)$ of the prime divisor $\Spec(A/zA)$ of $\Spec(A)$ by $\sigma$ is a smooth prime divisor of $\Sigma$, and $\sigma^*\Spec(A/zA)\supset \sigma^{-1}(M(A))$.
\item The projection $\pi:\Sigma\rightarrow\Sigma'$ induces an isomorphism $\sigma^*\Spec(A/zA)\rightarrow\Sigma'$.
\end{enumerate}

Let $\bar{\Delta}=\pi^*\bar{\Delta}'+\sigma^*\Spec(A/zA)$. 

\begin{enumerate}
\setcounter{enumi}{3}
\item
The pair $(\Sigma,\bar{\Delta})$ is a normal crossing scheme over $k$. 
$(\bar{\Delta})_0=\pi^{-1}((\bar{\Delta}')_0)\cap\sigma^*\Spec(A/zA)$. For any $\A\in(\bar{\Delta})_0$ we have $\pi(\A)\in (\bar{\Delta}')_0$. The mapping $\pi: (\bar{\Delta})_0\rightarrow(\bar{\Delta}')_0$ induced by $\pi$ is bijective. 
\item
For any $\A\in(\bar{\Delta})_0$, we have $\Comp(\bar{\Delta})(\A)=\{\pi^*\Lambda|\Lambda\in \Comp(\bar{\Delta}')(\pi(\A))\}\cup\{\sigma^*\Spec(A/zA)\}$, and $U(\Sigma, \bar{\Delta},\A)=\pi^{-1}(U(\Sigma', \bar{\Delta}',\pi(\A)))$.
\end{enumerate}

Consider any $\A\in(\bar{\Delta})_0$. We put
$\bar{\xi}_\A(\pi^*\Lambda)=\pi^*(U(\Sigma',\bar{\Delta}',\pi(\A)))(\bar{\xi}'_{\pi(\A)}(\Lambda))$
for any $\Lambda\in \Comp(\bar{\Delta}')(\pi(\A))$, and we put
$\bar{\xi}_\A(\sigma^*\Spec(A/zA))=\Res^\Sigma_{U(\Sigma,\bar{\Delta},\A)}\sigma^*(\Spec(A))$\break$(z)$.
We have a mapping $\bar{\xi}_\A:\Comp(\bar{\Delta})(\A)\rightarrow\mathcal{O}_\Sigma(U(\Sigma, \bar{\Delta},\A))$.
Let $\bar{\xi}=\{\bar{\xi}_\A|\A\in(\bar{\Delta})_0\}$.

\begin{enumerate}
\setcounter{enumi}{5}
\item
For any $\A\in(\bar{\Delta})_0$, $\bar{\xi}_\A$ is a coordinate system of $(\Sigma,\bar{\Delta})$ at $\A$, and $\bar{\xi}$ is a coordinate system of $(\Sigma,\bar{\Delta})$.
\item
The triplet $(\Sigma,\bar{\Delta}, \bar{\xi})$ is an extended pull-back of $(\Spec(A),\Delta,\xi)$ by $\sigma$.
\end{enumerate}

The lemma below plays the role of a key in our proofs below.

Let $X$ be any finite set.
We define a partial order on $\Map(X,\R)$.
Let $e\in\Map(X,\R)$ and $f\in\Map(X,\R)$ be any elements. We denote $e\leq f$ or $f\geq e$, if $e(x)\leq f(x)$ for any $x\in X$. Obviously the relation $\leq$ is a partial order on $\Map(X,\R)$.
We denote $e< f$ or $f>e$, if $e\leq f$ and $e\neq f$.

\begin{lemma}
\label{BM}
\emph{(Bierstone and Milman~\cite{B88}, p. 25, Lemma 4.7)}

Let $\A\in\Map(P, \Z_0)$, $\B\in\Map(P, \Z_0)$ and $\gamma\in\Map(P, \Z_0)$ be any mappings from $P$ to $\Z_0$ and let $u\in A^\times$, $v\in A^\times$ and $w\in A^\times$ be any elements.

If
$$u\prod_{x\in P}x^{\A(x)}- v\prod_{x\in P}x^{\B(x)}= w\prod_{x\in P}x^{\gamma(x)},$$
then, either $\A\leq\B$, or, $\B\leq\A$ with respect to the partial order $\leq$ on $\Map(P,\R)$.
\end{lemma}

We give the proof of Theorem~\ref{make simple}.

Assume the above $(*)$ and consider any $h\in\Z_+$ with $h\geq 2$ and any $\phi\in RW(h)$.

We take an element $\psi\in W(h)$, an element $u\in A^\times$, a finite subset $\mathcal{X}$ of $M(A')$, a mapping $a: \mathcal{X}\rightarrow \Z_+$ and a mapping $b:P-\{z\}\rightarrow\Z_0$ satisfying
$$\phi=\psi u\prod_{\chi\in\mathcal{X}}(z+\chi)^{a(\chi)}\prod_{x\in P-\{z\}} x^{b(x)}.$$
The quintuplet $(\psi,u,\mathcal{X},a,b)$ is uniquely determined depending on $\phi$, since $A$ is a unique factorization domain.

We take a mapping $\psi':\{0,1,\ldots,h-1\}\rightarrow M(A')$ satisfying 
$\psi= z^h+\sum_{i=0}^{h-1} \psi'(i)z^i$
and $\chi^h+\sum_{i=0}^{h-1} \psi'(i) \chi^i\neq 0$ for any $\chi\in M(A')$.
$\psi'$ is uniquely determined depending on $\psi$, since Weierstrass' preparation theorem holds.

In particular $\psi'(0)=0^h+\sum_{i=0}^{h-1} \psi'(i)0^i\neq 0$.

The Newton polyhedron $\Gamma_+(P,\psi)$ has no $z$-removable faces.

We denote
\begin{equation*}\begin{split}
\omega&=\psi\prod_{\chi\in\mathcal{X}-\{0\}}(z+\chi)^{a(\chi)}\in A\text{, and}\\
\hat{h}&=h+\sum_{\chi\in\mathcal{X}-\{0\}}a(\chi)\in\Z_+
\end{split}\end{equation*}

$2\leq h\leq\hat{h}$. It is easy to see that there exists uniquely a mapping
$\omega':\{0,1,\ldots,\hat{h}-1\}\rightarrow M(A')$ satisfying
$\omega= z^{\hat{h}}+\sum_{i=0}^{\hat{h}-1} \omega'(i)z^i$
and $\omega'(0)\neq 0$.
We take the unique mapping
$\omega':\{0,1,\ldots,\hat{h}-1\}\rightarrow M(A')$ satisfying the above equality and $\omega'(0)\neq 0$.

We denote $I=\{i\in\{0,1,\ldots,\hat{h}-1\}|\omega'(i)\neq 0\}$.
$0\in I\subset\{0,1,\ldots,\hat{h}-1\}$.
We put $\omega'(\hat{h})=1\in A'$.

Note that for any integers $j$, $k$, $\ell$ with $0\leq k<\ell<j$, $j!/(j-k)\in\Z_+$,
$j!/(j-\ell)\in\Z_+$ and $ j!/(j-k)< j!/(j-\ell)$.

For any $j\in I\cup\{\hat{h}\}$ we denote
\begin{equation*}\begin{split}
K(j)&=\{(k,\ell)|k\in I, \ell\in I, k<\ell<j,\\
&\qquad\quad\omega'(k)^{j!/(j-k)} \omega'(j)^{(j!/(j-\ell))-(j!/(j-k))}\neq\omega'(\ell)^{j!/(j-\ell)}\}\subset I\times I.
\end{split}\end{equation*}
We denote $J=\{j\in I\cup\{\hat{h}\}|K(j)\neq\emptyset\}$.
$J\subset (I-\{0,1\})\cup \{\hat{h}\}$.
We put
$$\phi'=\prod_{i\in I}\omega'(i)\prod_{j\in J}\prod_{(k,\ell)\in K(j)}
(\omega'(k)^{j!/(j-k)} \omega'(j)^{(j!/(j-\ell))-(j!/(j-k))}-\omega'(\ell)^{j!/(j-\ell)}).$$
$\phi'\in M(A')$. $\phi'\neq 0$.

By $(*)$ we know that 
there exists a weakly admissible composition of blowing-ups $\sigma':\Sigma'\rightarrow\Spec(A')$@over $(\Delta',\xi')$ and an extended pull-back $(\Sigma',\bar{\Delta}',\bar{\xi}')$ of the coordinated normal crossing scheme $(\Spec(A'),\Delta',\xi')$ by $\sigma'$ satisfying $\Supp(\sigma^{\prime*}(\Spec(A'/$\break$\phi' A')+ \Delta'))\subset \Supp(\bar{\Delta}')$.
We take a weakly admissible composition of blowing-ups $\sigma':\Sigma'\rightarrow\Spec(A')$ over $(\Delta',\xi')$ and an extended pull-back $(\Sigma',\bar{\Delta}',\bar{\xi}')$ of the coordinated normal crossing scheme $(\Spec(A'),\Delta',\xi')$ by $\sigma'$ satisfying $\Supp(\sigma^{\prime*}(\Spec(A'/$\break$\phi' A')+ \Delta'))\subset \Supp(\bar{\Delta}')$.

We consider a morphism $\Spec(A)\rightarrow\Spec(A')$ induced by the inclusion ring homomorphism $A'\rightarrow A$, the product scheme $\Sigma=\Sigma'\times_{\Spec(A')}\Spec(A)$, the projection $\sigma:\Sigma\rightarrow\Spec(A)$, and the projection $\pi:\Sigma\rightarrow\Sigma'$.
The structure sheaf of the scheme $\Sigma$ is denoted by $\mathcal{O}_\Sigma$.
Let $\bar{\Delta}=\pi^*\bar{\Delta}'+\sigma^*\Spec(A/zA)$.
We have a normal crossing scheme $(\Sigma,\bar{\Delta})$ with $(\bar{\Delta})_0\subset \sigma^*\Spec(A/zA)$.
Consider any $\A\in(\bar{\Delta})_0$. We put
$\bar{\xi}_\A(\pi^*\Lambda)=\pi^*(U(\Sigma',\bar{\Delta}',\pi(\A)))(\bar{\xi}'_{\pi(\A)}(\Lambda))$
for any $\Lambda\in \Comp(\bar{\Delta}')(\pi(\A))$, and we put
$\bar{\xi}_\A(\sigma^*\Spec(A/zA))= \Res^\Sigma_{U(\Sigma, \bar{\Delta},\A)}\sigma^*(\Spec(A))(z)$.
We have a coordinate system $\bar{\xi}_\A:\Comp(\bar{\Delta})(\A)\rightarrow\mathcal{O}_\Sigma(U(\Sigma, \bar{\Delta},\A))$ of $(\Sigma,\bar{\Delta})$ at $\A$, and a coordinate system $\bar{\xi}=\{\bar{\xi}_\A|\A\in(\bar{\Delta})_0\}$ of $(\Sigma,\bar{\Delta})$.

Consider any closed point $\A\in\Sigma$ with $\sigma(\A)=M(A)$.

We take any $\B\in (\bar{\Delta})_0$ satisfying $\A\in U(\Sigma,\bar{\Delta},\B)$.

We consider the homomorphism $\sigma^*(\A):A\rightarrow \mathcal{O}_{\Sigma, \A}$ of $k$-algebras induced by $\sigma$.
It satisfies $\sigma^*(\A)(M(A))\subset M(\mathcal{O}_{\Sigma, \A})$ and it has a unique extension $\sigma^*(\A):A\rightarrow \mathcal{O}_{\Sigma, \A}^c$.
Since  the morphism $\sigma$ is a weakly admissible composition of blowing-ups over $(\Delta,\xi)$, $\sigma^*(\A)$ is injective.

Note that $\A\in\sigma^{-1}(M(A))\subset\sigma^*\Spec(A/zA)$, $\B\in (\bar{\Delta})_0\subset\sigma^*\Spec(A/zA)$, $\sigma^*\Spec(A/zA)\in\Comp(\bar{\Delta})(\B)$ and $\bar{\xi}_\B(\sigma^*\Spec(A/zA))(\A)=(\sigma^*(\A)(z))(\A)=$\break$z(M(A))=0$.
Note that $\bar{\xi}_\B(\Lambda)(\A)\in k$ for any $\Lambda\in \Comp(\bar{\Delta})(\B)$, since $k$ is algebraically closed.

We denote $\bar{P}=\{\bar{\xi}_\B(\Lambda) - \bar{\xi}_\B(\Lambda)(\A)|
\Lambda\in \Comp(\bar{\Delta})(\B)\}$. $\bar{P}$ is a parameter system of $\mathcal{O}_{\Sigma, \A}^c$.
We denote $\bar{z}=\sigma^*(\A)(z)$.
$\bar{z}\in \bar{P}$.

We take any isomorphism $\rho: \mathcal{O}_{\Sigma, \A}^c\rightarrow A$ of $k$-algebras satisfying $\rho(\bar{P})=P$ and $\rho(\bar{z})=z$. $\rho \sigma^*(\A)(z)=z$. $\rho(\bar{P}-\{\bar{z}\})=P-\{z\}$.

Note that $k[\bar{P}-\{\bar{z}\}]\subset \mathcal{O}_{\Sigma, \A}^c$.
By $\mathcal{O}_{\Sigma, \A}^{c\:\prime}$ we denote the completion of $k[\bar{P}-\{\bar{z}\}]$ with respect to the maximal ideal 
$k[\bar{P}-\{\bar{z}\}]\cap M(\mathcal{O}_{\Sigma, \A}^c)$.
$\mathcal{O}_{\Sigma, \A}^{c\:\prime}$ is a complete local subring of $\mathcal{O}_{\Sigma, \A}^c$.$\mathcal{O}_{\Sigma, \A}^{c\:\prime}$ is a $k$-subalgebra of $\mathcal{O}_{\Sigma, \A}^c$.$\bar{P}-\{\bar{z}\}$ is a parameter system of $\mathcal{O}_{\Sigma, \A}^{c\:\prime}$.
$\sigma^*(\A)(A')\subset \mathcal{O}_{\Sigma, \A}^{c\:\prime}$.
$\sigma^*(\A)(M(A'))\subset M(\mathcal{O}_{\Sigma, \A}^{c\:\prime})$.
$\rho(\mathcal{O}_{\Sigma, \A}^{c\:\prime})=A'$.

We consider the element $\sigma^*(\A)(\omega)\in\mathcal{O}_{\Sigma, \A}^c$.

$$\sigma^*(\A)(\omega)=\bar{z}^{\hat{h}}+\sum_{i\in I}\sigma^*(\A)(\omega'(i))\bar{z}^i,$$
and $\sigma^*(\A)(\omega'(0))\neq 0$.
We know that the Newton polyhedron $\Gamma_+(\bar{P}, \sigma^*(\A)(\omega))$ is of $\bar{z}$-Weierstrass type, the unique $\bar{z}$-vertex of $\Gamma_+(\bar{P}, \sigma^*(\A)(\omega))$ is equal to $\{\hat{h}f^{\bar{P}}_{\bar{z}}\}$ and $\Ord(\bar{P}, f^{\bar{P}\vee}_{\bar{z}}, \sigma^*(\A)(\omega))=0$.

Consider any $i\in I$. 

$0\neq \sigma^*(\A)(\omega'(i))\in M(\mathcal{O}_{\Sigma, \A}^{c\:\prime})$.
Since $\phi'\in\omega'(i)A'$ and 
$\Supp(\sigma^{\prime*}\Spec(A'/\phi' A'))\subset\Supp(\sigma^{\prime*}(\Spec(A'/\phi' A')+ \Delta'))\subset \Supp(\bar{\Delta}')$, $\sigma^*(\A)(\omega'(i))$ has normal crossings over $\bar{P}-\{\bar{z}\}$. 
We take the unique pair of an element $\bar{v}(i)\in(\mathcal{O}_{\Sigma, \A}^{c\:\prime})^\times$ and a mapping $\bar{c}(i):\bar{P}-\{\bar{z}\}\rightarrow\Z_0$ satisfying 
$\sigma^*(\A)(\omega'(i))= \bar{v}(i)\prod_{\bar{x}\in \bar{P}-\{\bar{z}\}}\bar{x}^{\bar{c}(i)(\bar{x})}$.
We know $\bar{c}(i)\neq 0$, since $\sigma^*(\A)(\omega'(i))\in M(\mathcal{O}_{\Sigma, \A}^{c\:\prime})$.

We put $\bar{c}(\hat{h})=0\in\Map(\bar{P}-\{\bar{z}\},\Z_0)$.
$\sigma^*(\A)(\omega'(\hat{h}))=1=\prod_{\bar{x}\in \bar{P}-\{\bar{z}\}}\bar{x}^{\bar{c}(\hat{h})(\bar{x})}$.

Assume that a non-negative integer $m$ with $m<\sharp I$ and a mapping $\nu:\{0,1,\ldots,$\break$m\}\rightarrow I\cup\{\hat{h}\}$ satisfying the following five conditions are given:
\begin{enumerate}
\item
$\nu$ is injective and it reverses the order.
\item
$\nu(0)=\hat{h}$.
\item
If $m>0$, then for any $i\in \{1,2,\ldots,m\}$ and for any $j\in I$ with $\nu(i)< j< \nu(i-1)$, 
$((\nu(i-1)-j) /(\nu(i-1)-\nu(i)))\bar{c}\nu(i)
 +((j-\nu(i))/(\nu(i-1)-\nu(i)))\bar{c}\nu(i-1)\leq \bar{c}(j)$.
\item
If $m>0$, then for any $i\in \{1,2,\ldots,m\}$ and for any $j\in I$ with $j<\nu(i)$, 
$((\nu(i-1)-j) /(\nu(i-1)-\nu(i)))\bar{c}\nu(i)
 +((j-\nu(i))/(\nu(i-1)-\nu(i)))\bar{c}\nu(i-1)< \bar{c}(j)$.
\item
$\nu(m)\neq 0$.
\end{enumerate}

Note that if $m=0$, then there exists  uniquely a mapping $\nu:\{0,1,\ldots,m\}\rightarrow I\cup\{\hat{h}\}$ satisfying the above conditions.
If $m=0$, a mapping $\nu:\{0\}=\{0,1,\ldots,m\}\rightarrow I\cup\{\hat{h}\}$ satisfying $\nu(0)=\hat{h}$ satisfies the conditions.

$\Supp(\sigma^{\prime*}\Spec(A'/\phi' A'))\subset\Supp(\sigma^{\prime*}(\Spec(A'/\phi' A')+ \Delta'))\subset \Supp(\bar{\Delta}')$.

If $K\nu(m)\neq \emptyset$, then for any $(k,\ell)\in K\nu(m)$,
$0\leq k<\ell<\nu(m)$, 
$\phi'\in(\omega'(k)^{\nu(m)!/(\nu(m)-k)} \omega'\nu(m)^{(\nu(m)!/(\nu(m)-\ell))-(\nu(m)!/(\nu(m)-k))}-\omega'(\ell)^{\nu(m)!/(\nu(m)-\ell)}) A'$ and the element $\sigma^*(\A)(\omega'(k)^{\nu(m)!/(\nu(m)-k)} \omega'\nu(m)^{(\nu(m)!/(\nu(m)-\ell))-(\nu(m)!/(\nu(m)-k))}-\omega'(\ell)^{\nu(m)!/(\nu(m)-\ell)})\in \mathcal{O}_{\Sigma, \A}^{c\:\prime}$ has normal crossings over $\bar{P}-\{\bar{z}\}$.
By Lemma~\ref{BM}, we know that either $(\bar{c}(k)-\bar{c}\nu(m))/(\nu(m)-k)\leq (\bar{c}(\ell) -\bar{c}\nu(m))/(\nu(m)-\ell)$ or $(\bar{c}(\ell) -\bar{c}\nu(m))/(\nu(m)-\ell)\leq (\bar{c}(k)-\bar{c}\nu(m))/(\nu(m)-k)$ holds for any $k\in I$ and any $\ell\in I$ with $k<\nu(m)$ and $\ell<\nu(m)$.
Note that $(\bar{c}(0) -\bar{c}\nu(m))/\nu(m)\in\{(\bar{c}(k) -\bar{c}\nu(m))/(\nu(m)-k)|k\in I, k<\nu(m)\}\neq\emptyset$.
We know that the set $\{(\bar{c}(k) -\bar{c}\nu(m))/(\nu(m)-k)|k\in I, k<\nu(m)\}$ has the minimum element $\min \{(\bar{c}(k) -\bar{c}\nu(m))/(\nu(m)-k)|k\in I, k<\nu(m)\}$ with respect to the partial order $\leq$.

Putting $\nu(m+1)=\min\{j\in I|j<\nu(m), (\bar{c}(j) -\bar{c}\nu(m))/(\nu(m)-j)=
\min \{(\bar{c}(k) -\bar{c}\nu(m))/(\nu(m)-k)|k\in I, k<\nu(m)\}\}$, we define an extension $\nu:\{0,1,\ldots,m+1\}\rightarrow I\cup\{\hat{h}\}$ of $\nu:\{0,1,\ldots,m\}\rightarrow I\cup\{\hat{h}\}$.
$\nu(m+1)<\nu(m)$. 

Consider any $j\in I$ with $\nu(m+1)<j<\nu(m)$. By how to choose $\nu(m+1)$ we know
$(\bar{c}\nu(m+1) -\bar{c}\nu(m))/(\nu(m)-\nu(m+1))\leq(\bar{c}(j) -\bar{c}\nu(m))/(\nu(m)-j)$.
It follows
$((\nu(m)-j)/(\nu(m)-\nu(m+1))\bar{c}\nu(m+1)+((j-\nu(m+1))/ (\nu(m)-\nu(m+1))\bar{c}\nu(m)\leq \bar{c}(j)$.

Consider any $j\in I$ with $j<\nu(m+1) $. By how to choose $\nu(m+1)$ we know
$(\bar{c}\nu(m+1) -\bar{c}\nu(m))/(\nu(m)-\nu(m+1))<(\bar{c}(j) -\bar{c}\nu(m))/(\nu(m)-j)$.
It follows
$((\nu(m)-j)/(\nu(m)-\nu(m+1))\bar{c}\nu(m+1)+((j-\nu(m+1))/ (\nu(m)-\nu(m+1))\bar{c}\nu(m)< \bar{c}(j)$.

We know that if $\nu(m+1)\neq 0$, then the extension $\nu:\{0,1,\ldots,m+1\}\rightarrow I\cup\{\hat{h}\}$ satisfies the above five conditions.
If $\nu(m+1)=0$, then the extension satisfies the four conditions except the last one of the above five.

By induction we know that there exists uniquely a pair of a positive integer $m$ and a mapping $\nu:\{0,1,\ldots,m\}\rightarrow I\cup\{\hat{h}\}$ satisfying the following five conditions:
\begin{enumerate}
\item
$\nu$ is injective and it reverses the order.
\item
$\nu(0)=\hat{h}$.
\item
For any $i\in \{1,2,\ldots,m\}$ and for any $j\in I$ with $\nu(i)< j< \nu(i-1)$, 
$((\nu(i-1)-j) /(\nu(i-1)-\nu(i)))\bar{c}\nu(i)
 +((j-\nu(i))/(\nu(i-1)-\nu(i)))\bar{c}\nu(i-1)\leq \bar{c}(j)$.
\item
For any $i\in \{1,2,\ldots,m\}$ and for any $j\in I$ with $j<\nu(i)$, 
$((\nu(i-1)-j) /(\nu(i-1)-\nu(i)))\bar{c}\nu(i)
 +((j-\nu(i))/(\nu(i-1)-\nu(i)))\bar{c}\nu(i-1)< \bar{c}(j)$.
\item
$\nu(m)=0$.
\end{enumerate}

We take the unique pair of a positive integer $m$ and a mapping $\nu:\{0,1,\ldots,m\}\rightarrow I\cup\{\hat{h}\}$ satisfying the above five conditions.

We know that $(\bar{c}\nu(i)-\bar{c}\nu(i-1))/(\nu(i-1)-\nu(i))<
(\bar{c}\nu(i+1)-\bar{c}\nu(i))/(\nu(i)-\nu(i+1))$ for any $i\in\{1,2,\ldots,m-1\}$, if $m\geq 2$.

By Lemma~\ref{Newton2}.$13$ we know that the Newton polyhedron $\Gamma_+(\bar{P},\sigma^*(\A)(\omega))$ is $\bar{z}$-simple.

We consider the element $ \sigma^*(\A)(u\prod_{x\in P-\{z\}}x^{b(x)})\in \mathcal{O}_{\Sigma, \A}^c$.

$\sigma^*(\A)(u)\in (\mathcal{O}_{\Sigma, \A}^c)^\times$.
$\sigma^*(\A)( \prod_{x\in P-\{z\}}x^{b(x)})\in \mathcal{O}_{\Sigma, \A}^{c\:\prime}$.

Since $\Supp(\sigma^{\prime*}\Delta')\subset\Supp(\sigma^{\prime*}(\Spec(A'/\phi' A')+ \Delta'))\subset \Supp(\bar{\Delta}')$, we know that $\sigma^*(\A)(\prod_{x\in P-\{z\}}x^{b(x)})$ has normal crossings over $\bar{P}-\{\bar{z}\}$.
We take the unique pair of an element $\bar{u}'\in(\mathcal{O}_{\Sigma, \A}^{c\:\prime})^\times$ and a mapping $\bar{b}:\bar{P}-\{\bar{z}\}\rightarrow \Z_0$ satisfying $\sigma^*(\A)(\prod_{x\in P-\{z\}}x^{b(x)})= \bar{u}'\prod_{\bar{x}\in \bar{P}-\{\bar{z}\}}\bar{x}^{\bar{b}(\bar{x})}$.
Let $\bar{u}=\sigma^*(\A)(u) \bar{u}'\in (\mathcal{O}_{\Sigma, \A}^c)^\times$.
$\rho(\bar{u})\in A^\times$.
We have
\begin{equation*}\begin{split}
\sigma^*(\A)(u\prod_{x\in P-\{z\}}x^{b(x)})&=
\bar{u}\prod_{\bar{x}\in \bar{P}-\{\bar{z}\}}\bar{x}^{\bar{b}(\bar{x})},\\
\rho\sigma^*(\A)(u\prod_{x\in P-\{z\}}x^{b(x)})&=
\rho(\bar{u})\prod_{x\in P-\{z\}}x^{\bar{b}\rho^{-1}(x)}.
\end{split}\end{equation*}

Consider the case where $\chi\neq 0$ for any $\chi\in\mathcal{X}$.

$\sigma^*(\A)(\phi)= \sigma^*(\A)(\omega) \sigma^*(\A)(u\prod_{x\in P-\{z\}}x^{b(x)})= \sigma^*(\A)(\omega) \bar{u}\prod_{\bar{x}\in \bar{P}-\{\bar{z}\}}\bar{x}^{\bar{b}(\bar{x})}$.
We know that $\Gamma_+(\bar{P},\sigma^*(\A)(\phi))= \Gamma_+(\bar{P},\sigma^*(\A)(\omega))+\{\bar{b}\}$ and $\Gamma_+(\bar{P},\sigma^*(\A)(\phi))$ is $\bar{z}$-simple.

Consider the case where $\chi= 0$ for some $\chi\in\mathcal{X}$.

We take the unique $\chi_0\in\mathcal{X}$ with $\chi_0=0$.
$\sigma^*(\A)(\phi)= \sigma^*(\A)(\omega) \sigma^*(\A)$\break$(u\prod_{x\in P-\{z\}}x^{b(x)})\bar{z}^{a(\chi_0)}= \sigma^*(\A)(\omega) \bar{u}\bar{z}^{a(\chi_0)}\prod_{\bar{x}\in \bar{P}-\{\bar{z}\}}\bar{x}^{\bar{b}(\bar{x})}$.
We know that $\Gamma_+(\bar{P},$\break$\sigma^*(\A)(\phi))= \Gamma_+(\bar{P},\sigma^*(\A)(\omega))+\{a(\chi_0)f^{\bar{P}}_{\bar{z}}+\bar{b}\}$ and $\Gamma_+(\bar{P},\sigma^*(\A)(\phi))$ is $\bar{z}$-simple.

We conclude that the Newton polyhedron $\Gamma_+(\bar{P},\sigma^*(\A)(\phi))$ is $\bar{z}$-simple in all cases and the Newton polyhedron $\Gamma_+(P,\rho\sigma^*(\A)(\phi))$ is $z$-simple.

We consider any $\chi\in\mathcal{X}$ and the element $\sigma^*(\A)(z+\chi)\in \mathcal{O}_{\Sigma, \A}^c$.

$\rho\sigma^*(\A)(z+\chi)=z+\rho\sigma^*(\A)(\chi)$.
$\rho\sigma^*(\A)(\chi)\in M(A')$.

Since $\rho\sigma^*(\A)$ is injective, $\rho\sigma^*(\A)(\chi)\neq \rho\sigma^*(\A)(\bar{\chi})$ for any $\chi\in\mathcal{X}$ and any $\bar{\chi}\in\mathcal{X}$ with $\chi\neq \bar{\chi}$.

$$\rho\sigma^*(\A)(\prod_{\chi\in\mathcal{X}}(z+\chi)^{a(\chi)})=
\prod_{\chi\in\mathcal{X}}(z+\rho\sigma^*(\A)(\chi))^{a(\chi)}.$$

Therefore, we know

\begin{equation*}\begin{split}
&\rho\sigma^*(\A)(u\prod_{\chi\in\mathcal{X}}(z+\chi)^{a(\chi)} \prod_{x\in P-\{z\}}x^{b(x)})\\
=\:&\rho(\bar{u})\prod_{\chi\in\mathcal{X}}(z+\rho\sigma^*(\A)(\chi))^{a(\chi)}
\prod_{x\in P-\{z\}}x^{\bar{b}\rho^{-1}(x)}\in PW(1).
\end{split}\end{equation*}

We consider the element $\sigma^*(\A)(\psi)\in \mathcal{O}_{\Sigma, \A}^c$.
$$\sigma^*(\A)(\psi)= \bar{z}^h+\sum_{i=0}^{h-1} \sigma^*(\A)(\psi'(i))\bar{z}^i.$$
$\sigma^*(\A)(\psi'(i))\in M(\mathcal{O}_{\Sigma, \A}^{c\:\prime})$ for any $i\in\{0,1,\ldots, h-1\}$.
Since $\sigma^*(\A)$ is injective, $\sigma^*(\A)(\psi'(0))\neq 0$.
$\Ord(\bar{P},f^{\bar{P}\vee}_z, \sigma^*(\A)(\psi))=0$.
Since $\Gamma_+(\bar{P}, \sigma^*(\A)(\phi))$ is $\bar{z}$-simple and $\sigma^*(\A)(\phi)\in\sigma^*(\A)(\psi) \mathcal{O}_{\Sigma, \A}^c$, we know that $\Gamma_+(\bar{P}, \sigma^*(\A)(\psi))$ is also $\bar{z}$-simple. The unique  $\bar{z}$-top vertex of $\Gamma_+(\bar{P}, \sigma^*(\A)(\psi))$ is equal to $\{hf^{\bar{P}}_{\bar{z}}\}$.

We would like to show that $\Gamma_+(\bar{P}, \sigma^*(\A)(\psi))$ has no $\bar{z}$-removable faces.

Assume that $\Gamma_+(\bar{P}, \sigma^*(\A)(\psi))$ has $\bar{z}$-removable faces.
We will deduce a contradiction.

Take any $\bar{z}$-removable face $\bar{F}$ of dimension one of $\Gamma_+(\bar{P}, \sigma^*(\A)(\psi))$.
$\Stab(\bar{F})=\{0\}$.
$\Delta^\circ(\bar{F}, \Gamma_+(\bar{P}, \sigma^*(\A)(\psi))| \Map(\bar{P},\R))\subset (\Map(\bar{P},\R_0)^\vee|\Map(\bar{P},\R))^\circ$.
Take any $\bar{\lambda}\in (\Map(\bar{P},\R_0)^\vee|\Map(\bar{P},\R))\cap\sum_{\bar{x}\in\bar{P}
-\{\bar{z}\}}\R f^{\bar{P}\vee}_{\bar{x}}$ such that $\bar{\lambda}+ f^{\bar{P}\vee}_{\bar{z}}\in\Delta^\circ(\bar{F}, \Gamma_+(\bar{P},$\break$ \sigma^*(\A)(\psi))| \Map(\bar{P},\R))$.
$ hf^{\bar{P}}_{\bar{z}}\in \bar{F}$ and $\Ord(\bar{P},\bar{\lambda}+ f^{\bar{P}\vee}_{\bar{z}},\sigma^*(\A)(\psi))=\langle \bar{\lambda}+ f^{\bar{P}\vee}_{\bar{z}}, hf^{\bar{P}}_{\bar{z}}\rangle=h$.
Take $\bar{\gamma}\in k-\{0\}$ and $\bar{e}\in\Map(\bar{P}-\{\bar{z}\},\Z_0)$ satisfying
$\In(\bar{P},\bar{\lambda}+ f^{\bar{P}\vee}_{\bar{z}},\sigma^*(\A)(\psi))=\Ps(\bar{P},F, \sigma^*(\A)(\psi))=(\bar{z}+\bar{\gamma}\prod_{\bar{x}\in\bar{P}-\{\bar{z}\}}\bar{x}^{\bar{e}(\bar{x})})^h=\bar{z}^h+\sum_{i=1}^{h}\binom{h}{i}\bar{\gamma}^i
\prod_{\bar{x}\in\bar{P}-\{\bar{z}\}}\bar{x}^{i\bar{e}(\bar{x})}\bar{z}^{h-i}$.

We define an element $\lambda\in\Map(P,\R_0)^\vee|\Map(P,\R)\cap\sum_{x\in P-\{z\}}\R f^{P\vee}_x$ by putting $\langle \lambda, f^P_x\rangle=\Ord(\bar{P},\bar{\lambda},\sigma^*(\A)(x))\in\R_0$ for any $x\in P$.
$\lambda+f^{P\vee}_z\in \Map(P,\R_0)^\vee|\Map(P,\R)$.
We denote $F=\Delta(\lambda+f^{P\vee}_z,\Gamma_+(P,\psi)|\Map(P,\R))\in\mathcal{F}(\Gamma_+(P,\psi))$.
We can show that $\Ord(P, \lambda+f^{P\vee}_z,\psi)= \Ord(\bar{P},\bar{\lambda}+ f^{\bar{P}\vee}_{\bar{z}},\sigma^*(\A)(\psi))=h$ and
$\sigma^*(\A)(\In(P, \lambda+f^{P\vee}_z,\psi))= \In(\bar{P},\bar{\lambda}+ f^{\bar{P}\vee}_{\bar{z}},\sigma^*(\A)(\psi))= \bar{z}^h+\sum_{i=1}^{h}\binom{h}{i}\bar{\gamma}^i
\prod_{\bar{x}\in\bar{P}-\{\bar{z}\}}\bar{x}^{i\bar{e}(\bar{x})}\bar{z}^{h-i}$.

We know that $hf^P_z\in F$ and $\Ord(P-\{z\},\lambda, \psi'(h-i))\geq i\langle \bar{\lambda}, \bar{e}\rangle$ for any $i\in\{1,2,\ldots, h\}$. We put
\begin{equation*}
\hat{\psi}(h-i)=
\begin{cases}
\In(P-\{z\},\lambda, \psi'(h-i))& \text{if $\Ord(P-\{z\},\lambda, \psi'(h-i))= i\langle \bar{\lambda}, \bar{e}\rangle$},\\
0&\text{if $\Ord(P-\{z\},\lambda, \psi'(h-i))> i\langle \bar{\lambda}, \bar{e}\rangle$}.
\end{cases}\end{equation*}
We have $\In(P, \lambda+f^{P\vee}_z,\psi)=z^h+\sum_{i=1}^h\hat{\psi}(h-i)z^{h-i}$ and
$\sigma^*(\A)(\In(P, \lambda+f^{P\vee}_z,\psi))= \bar{z}^h+\sum_{i=1}^h\sigma^*(\A)(\hat{\psi}(h-i))\bar{z}^{h-i}$.
We conclude $\binom{h}{i}\bar{\gamma}^i
\prod_{\bar{x}\in\bar{P}-\{\bar{z}\}}\bar{x}^{i\bar{e}(\bar{x})}=
\sigma^*(\A)(\hat{\psi}(h-i))$ for any $i\in\{1,2,\ldots,h\}$.

By $\bar{1}\in k$ we denote the identity element of the field $k$.

We consider the case where the characteristic number of the field $k$ is equal to $0$.

$\binom{h}{1}\bar{1}=h\bar{1}\neq 0$. We put $\hat{\chi}=\hat{\psi}(h-1)/(h\bar{1})\in M(A')$.
We have $\bar{\gamma}\prod_{\bar{x}\in\bar{P}-\{\bar{z}\}}\bar{x}^{\bar{e}(\bar{x})}=
\sigma^*(\A)(\hat{\chi})$.

We consider the case where the characteristic number of the field $k$ is positive.
By $p$ we denote the characteristic number of the field $k$. The integer $p$ is a prime number.
We take the unique pair of $\delta\in\Z_0$ and $\bar{h}\in\Z_+$ such that $h=p^\delta\bar{h}$ and $\bar{h}$ is not a multiple of $p$.
$\binom{h}{p^\delta}\bar{1}=\bar{h}\bar{1}\neq 0$.
We have $(\bar{\gamma}\prod_{\bar{x}\in\bar{P}-\{\bar{z}\}}\bar{x}^{\bar{e}(\bar{x})})^{p^\delta}=\sigma^*(\A)(\hat{\psi}(h- p^\delta)/(\bar{h}\bar{1}))$ and 
$\hat{\psi}(h- p^\delta)/(\bar{h}\bar{1})\in M(A')$.

For any complete regular local ring $R$ such that $R$ contains $k$ and the residue field $R/M(R)$ is isomorphic to $k$ as $k$-algebras, we denote $R^{p^\delta}=\{s^{p^\delta}|s\in R\}$. $R^{p^\delta}$ is a local subring of $R$.

$(\bar{\gamma}\prod_{\bar{x}\in\bar{P}-\{\bar{z}\}}\bar{x}^{\bar{e}(\bar{x})})^{p^\delta}\in \mathcal{O}_{\Sigma, \A}^{c\: p^\delta}$.
$\hat{\psi}(h- p^\delta)/(\bar{h}\bar{1})\in\sigma^*(\A)^{-1}(\mathcal{O}_{\Sigma, \A}^{c\: p^\delta})$.
Now, since $\sigma$ is a weakly admissible composition of blowing-ups over $(\Delta,\xi)$, we can show that $\sigma^*(\A)^{-1}(\mathcal{O}_{\Sigma, \A}^{c\: p^\delta})=A^{p^\delta}$.
We know that there exists uniquely an element $\hat{\chi}\in A$ with $\hat{\chi}^{p^\delta}=\hat{\psi}(h- p^\delta)/(\bar{h}\bar{1})$. 
We take the unique element $\hat{\chi}\in A$ with $\hat{\chi}^{p^\delta}=\hat{\psi}(h- p^\delta)/(\bar{h}\bar{1})$.
Since $\hat{\psi}(h- p^\delta)/(\bar{h}\bar{1})\in M(A')$, we know $\hat{\chi}\in M(A')$.
We know $(\bar{\gamma}\prod_{\bar{x}\in\bar{P}-\{\bar{z}\}}\bar{x}^{\bar{e}(\bar{x})})^{p^\delta}=\sigma^*(\A)( \hat{\chi}^{p^\delta})$ and 
$\bar{\gamma}\prod_{\bar{x}\in\bar{P}-\{\bar{z}\}}\bar{x}^{\bar{e}(\bar{x})}=\sigma^*(\A)( \hat{\chi})$.

We conclude that there exists $\hat{\chi}\in M(A')$ satisfying $\bar{\gamma}\prod_{\bar{x}\in\bar{P}-\{\bar{z}\}}\bar{x}^{\bar{e}(\bar{x})}=\sigma^*(\A)( \hat{\chi})$ in all cases. We take any  $\hat{\chi}\in M(A')$ satisfying $\bar{\gamma}\prod_{\bar{x}\in\bar{P}-\{\bar{z}\}}\bar{x}^{\bar{e}(\bar{x})}=\sigma^*(\A)( \hat{\chi})$.

We have 
$\sigma^*(\A)(\In(P, \lambda+f^{P\vee}_z,\psi))= \In(\bar{P},\bar{\lambda}+ f^{\bar{P}\vee}_{\bar{z}},\sigma^*(\A)(\psi))= (\bar{z}+\bar{\gamma}\prod_{\bar{x}\in\bar{P}-\{\bar{z}\}}$\break$\bar{x}^{\bar{e}(\bar{x})})^h=(\bar{z}+\sigma^*(\A)( \hat{\chi}))^h=\sigma^*(\A)((z+\hat{\chi})^h)$.
Since $\sigma^*(\A)$ is injective we have
$\Ps(P, F,\psi)= \In(P, \lambda+f^{P\vee}_z,\psi)= (z+\hat{\chi})^h$ and we know that the face $F$ of $\Gamma_+(P,\psi)$ is $z$-removable.

Since the Newton polyhedron $\Gamma_+(P,\psi)$ has no $z$-removable faces, we obtain a contradiction.

We conclude that the Newton polyhedron $\Gamma_+(\bar{P},\sigma^*(\A)(\psi))$ has no $\bar{z}$-removable faces and the Newton polyhedron $\Gamma_+(P,\rho\sigma^*(\A)(\psi))$ has no $z$-removable faces.

We consider the case where $\bar{\chi}^h+\sum_{i=0}^{h-1}\sigma^*(\A)(\psi'(i))\bar{\chi}^i\neq 0$ for any $\bar{\chi}\in M(\mathcal{O}_{\Sigma, \A}^{c\:\prime})$.

We have $\rho\sigma^*(\A)(\psi)\in W(h)$.
Since $\rho\sigma^*(\A)(\phi)= \rho\sigma^*(\A)(\psi) \rho\sigma^*(\A)(u\prod_{\chi\in\mathcal{X}}$\break$(z+\chi)^{a(\chi)} \prod_{x\in P-\{z\}}x^{b(x)})$, $\rho\sigma^*(\A)(u\prod_{\chi\in\mathcal{X}}(z+\chi)^{a(\chi)} \prod_{x\in P-\{z\}}x^{b(x)})\in PW(1)$, $\Gamma_+(P, \rho\sigma^*(\A)(\phi))$ is $z$-simple and $\Gamma_+(P, \rho\sigma^*(\A)(\psi))$ has no $z$-removable faces, we conclude that$\rho\sigma^*(\A)(\phi)\in SW(h)$.

We consider the case where $\bar{\chi}^h+\sum_{i=0}^{h-1}\sigma^*(\A)(\psi'(i))\bar{\chi}^i=0$ for some $\bar{\chi}\in M(\mathcal{O}_{\Sigma, \A}^{c\:\prime})$.

Let $\mathcal{R}=\{\bar{\chi}\in M(\mathcal{O}_{\Sigma, \A}^{c\:\prime})|
\bar{\chi}^h+\sum_{i=0}^{h-1}\sigma^*(\A)(\psi'(i))\bar{\chi}^i=0\}$.
$\mathcal{R}$ is a non-empty finite set. $1\leq\sharp\mathcal{R}\leq h$.
Consider any $\bar{\chi}\in\mathcal{R}$. Since $\mathcal{O}_{\Sigma, \A}^c$ is a unique factorization domain, there exists uniquely a positive integer $\mu(\bar{\chi})\in\Z_+$ satisfying $\sigma^*(\A)(\psi) \in(\bar{z}-\bar{\chi})^{\mu(\bar{\chi})}\mathcal{O}_{\Sigma, \A}^c$ and 
$\sigma^*(\A)(\psi)\not\in(\bar{z}-\bar{\chi})^{\mu(\bar{\chi})+1}\mathcal{O}_{\Sigma, \A}^c$. We take the unique $\mu(\bar{\chi})\in\Z_+$ satisfying these conditions. $1\leq \sum_{\bar{\chi}\in\mathcal{R}}\mu(\bar{\chi})\leq h$. Let $\hat{g}=h-\sum_{\bar{\chi}\in\mathcal{R}}\mu(\bar{\chi})\in\Z_0$. $\hat{g}\neq 1$. $\hat{g}<h$.
Let $g=\max\{\hat{g}, 1\}\in\Z_+$. Since $h\geq 2$, $g<h$.

There exists uniquely a mapping $\zeta':\{0,1,\ldots, \hat{g}-1\}\rightarrow M(\mathcal{O}_{\Sigma, \A}^{c\:\prime})$ satisfying $\sigma^*(\A)(\psi)=
\prod_{\bar{\chi}\in\mathcal{R}}(\bar{z}-\bar{\chi})^{\mu(\bar{\chi})}(\bar{z}^{\hat{g}}+\sum_{i=0}^{\hat{g}-1}\zeta'(i) \bar{z}^i)$.
We take the unique mapping $\zeta':\{0,1,\ldots, \hat{g}-1\}\rightarrow M(\mathcal{O}_{\Sigma, \A}^{c\:\prime})$ satisfying this equality.
$\bar{\chi}^{\hat{g}}+\sum_{i=0}^{\hat{g}-1}\zeta'(i) \bar{\chi}^i\neq 0$ for any $\bar{\chi}\in M(\mathcal{O}_{\Sigma, \A}^{c\:\prime})$.

$\rho(\prod_{\bar{\chi}\in\mathcal{R}}(\bar{z}-\bar{\chi})^{\mu(\bar{\chi})})\in PW(1)$.

If $\hat{g}\geq 2$, then $\hat{g}=g$ and $\rho(\bar{z}^{\hat{g}}+\sum_{i=0}^{\hat{g}-1}\zeta'(i) \bar{z}^i)\in W(\hat{g})=W(g)$. Therefore, 
$\rho\sigma^*(\A)(\phi)= \rho(\bar{z}^{\hat{g}}+\sum_{i=0}^{\hat{g}-1}\zeta'(i) \bar{z}^i) (\prod_{\bar{\chi}\in\mathcal{R}}(z-\rho(\bar{\chi}))^{\mu(\bar{\chi})}) \rho\sigma^*(\A)(u\prod_{\chi\in\mathcal{X}}(z+\chi)^{a(\chi)} \prod_{x\in P-\{z\}}x^{b(x)})\in PW(g)$.

If $\hat{g}\leq 1$, then $\hat{g}=0$, $g=1$ and $\rho(\bar{z}^{\hat{g}}+\sum_{i=0}^{\hat{g}-1}\zeta'(i) \bar{z}^i)=1\in PW(1)$. Therefore, 
$\rho\sigma^*(\A)(\phi)= \rho(\bar{z}^{\hat{g}}+\sum_{i=0}^{\hat{g}-1}\zeta'(i) \bar{z}^i) (\prod_{\bar{\chi}\in\mathcal{R}}(z-\rho(\bar{\chi}))^{\mu(\bar{\chi})}) \rho\sigma^*(\A)(u\prod_{\chi\in\mathcal{X}}(z+\chi)^{a(\chi)} \prod_{x\in P-\{z\}}x^{b(x)})\in PW(1)=PW(g)$.

We conclude that there exists $g\in\Z_+$ satisfying $g<h$ and $\rho\sigma^*(\A)(\phi)\in PW(g)$, if $\bar{\chi}^h+\sum_{i=0}^{h-1}\sigma^*(\A)(\psi'(i))\bar{\chi}^i=0$ for some $\bar{\chi}\in M(\mathcal{O}_{\Sigma, \A}^{c\:\prime})$.

We know that Theorem~\ref{make simple} holds.

We give the proof of Theorem~\ref{make Weierstrass type}.

Assume the above $(*)$ and consider any $\phi\in A$ with $\phi\neq 0$.

We take the unique mapping $\phi':\Z_0\rightarrow A'$ satisfying $\phi=\sum_{i\in\Z_0}\phi'(i)z^i$. 
Since $\phi\neq 0$, $\phi'(i)\neq 0$ for some $i\in\Z_0$.

Since $A'$ is noetherian, there exists $m\in\Z_0$ satisfying $\{\phi'(i)| i\in\{0,1,\ldots,m\}\} A'=\{\phi'(i)| i\in\Z_0\} A'$.
We take $m\in\Z_0$ satisfying $\{\phi'(i)| i\in\{0,1,\ldots,m\}\} A'=\{\phi'(i)| i\in\Z_0\} A'$.
$\{0\}\neq \{\phi'(i)| i\in\{0,1,\ldots,m\}\} A'\subset A'$ and $\phi'(j)\in\{\phi'(i)| i\in\{0,1,\ldots,m\}\} A'$ for any $j\in\Z_0$.

We denote $I=\{i\in\{0,1,\ldots,m\}|\phi'(i)\neq 0\}$.
$\emptyset\neq I\subset\{0,1,\ldots,m\}$.

We denote
$$
K=\{(k,\ell)|k\in I, \ell\in I, k<\ell, 
\phi'(k) \neq\phi'(\ell) \}\subset I\times I.
$$
We put
$$\psi'=\prod_{i\in I}\phi'(i) \prod_{(k,\ell)\in K}(\phi'(k) -\phi'(\ell)).$$
$\psi'\in A'$. $\psi'\neq 0$.

By $(*)$ we know that 
there exists a weakly admissible composition of blowing-ups $\sigma':\Sigma'\rightarrow\Spec(A')$ over $(\Delta',\xi')$ and an extended pull-back $(\Sigma',\bar{\Delta}',\bar{\xi}')$ of the coordinated normal crossing scheme $(\Spec(A'),\Delta',\xi')$ by $\sigma'$ satisfying $\Supp(\sigma^{\prime*}(\Spec(A'/$\break$\psi' A')+ \Delta'))\subset \Supp(\bar{\Delta}')$.
We take a weakly admissible composition of blowing-ups $\sigma':\Sigma'\rightarrow\Spec(A')$ over $(\Delta',\xi')$ and an extended pull-back $(\Sigma',\bar{\Delta}',\bar{\xi}')$ of the coordinated normal crossing scheme $(\Spec(A'),\Delta',\xi')$ by $\sigma'$ satisfying $\Supp(\sigma^{\prime*}(\Spec(A'/$\break$\psi' A')+ \Delta'))\subset \Supp(\bar{\Delta}')$.

We consider a morphism $\Spec(A)\rightarrow\Spec(A')$ induced by the inclusion ring homomorphism $A'\rightarrow A$, the product scheme $\Sigma=\Sigma'\times_{\Spec(A')}\Spec(A)$, the projection $\sigma:\Sigma\rightarrow\Spec(A)$, and the projection $\pi:\Sigma\rightarrow\Sigma'$.
The structure sheaf of the scheme $\Sigma$ is denoted by $\mathcal{O}_\Sigma$.
Let $\bar{\Delta}=\pi^*\bar{\Delta}'+\sigma^*\Spec(A/zA)$.
We have a normal crossing scheme $(\Sigma,\bar{\Delta})$ with $(\bar{\Delta})_0\subset \sigma^*\Spec(A/zA)$.
Consider any $\A\in(\bar{\Delta})_0$. We put
$\bar{\xi}_\A(\pi^*\Lambda)=\pi^*(U(\Sigma',\bar{\Delta}',\pi(\A)))(\bar{\xi}'_{\pi(\A)}(\Lambda))$
for any $\Lambda\in \Comp(\bar{\Delta}')(\pi(\A))$, and we put
$\bar{\xi}_\A(\sigma^*\Spec(A/zA))= \Res^\Sigma_{U(\Sigma, \bar{\Delta},\A)}\sigma^*(\Spec(A))(z)$.
We have a coordinate system $\bar{\xi}_\A:\Comp(\bar{\Delta})(\A)\rightarrow\mathcal{O}_\Sigma(U(\Sigma, \bar{\Delta},\A))$ of $(\Sigma,\bar{\Delta})$ at $\A$, and a coordinate system $\bar{\xi}=\{\bar{\xi}_\A|\A\in(\bar{\Delta})_0\}$ of $(\Sigma,\bar{\Delta})$.

Note that if $\dim A=2$, then $\Sigma=\Spec(A)$, $\sigma=\Id_{\Spec(A)}$, $\bar{\Delta}=\Delta$ and $\bar{\xi}=\xi$.

Consider any closed point $\A\in\Sigma$ with $\sigma(\A)=M(A)$.

We take any $\B\in (\bar{\Delta})_0$ satisfying $\A\in U(\Sigma,\bar{\Delta},\B)$.

We consider the homomorphism $\sigma^*(\A):A\rightarrow \mathcal{O}_{\Sigma, \A}$ of $k$-algebras induced by $\sigma$.
It satisfies $\sigma^*(\A)(M(A))\subset M(\mathcal{O}_{\Sigma, \A})$ and it has a unique extension $\sigma^*(\A):A\rightarrow \mathcal{O}_{\Sigma, \A}^c$.
Since  the morphism $\sigma$ is a weakly admissible composition of blowing-ups over $(\Delta,\xi)$, $\sigma^*(\A)$ is injective.

Note that $\A\in\sigma^{-1}(M(A))\subset\sigma^*\Spec(A/zA)$, $\B\in (\bar{\Delta})_0\subset\sigma^*\Spec(A/zA)$, $\sigma^*\Spec(A/zA)\in\Comp(\bar{\Delta})(\B)$ and $\bar{\xi}_\B(\sigma^*\Spec(A/zA))(\A)=(\sigma^*(\A)(z))(\A)=$\break$z(M(A))=0$.
Note that $\bar{\xi}_\B(\Lambda)(\A)\in k$ for any $\Lambda\in \Comp(\bar{\Delta})(\B)$, since $k$ is algebraically closed.

We denote $\bar{P}=\{\bar{\xi}_\B(\Lambda) - \bar{\xi}_\B(\Lambda)(\A)|
\Lambda\in \Comp(\bar{\Delta})(\B)\}$. $\bar{P}$ is a parameter system of $\mathcal{O}_{\Sigma, \A}^c$.
We denote $\bar{z}=\sigma^*(\A)(z)$.
$\bar{z}\in \bar{P}$.

We take any isomorphism $\rho: \mathcal{O}_{\Sigma, \A}^c\rightarrow A$ of $k$-algebras satisfying $\rho(\bar{P})=P$ and $\rho(\bar{z})=z$. $\rho \sigma^*(\A)(z)=z$. $\rho(\bar{P}-\{\bar{z}\})=P-\{z\}$.

Note that $k[\bar{P}-\{\bar{z}\}]\subset \mathcal{O}_{\Sigma, \A}^c$.
By $\mathcal{O}_{\Sigma, \A}^{c\:\prime}$ we denote the completion of $k[\bar{P}-\{\bar{z}\}]$ with respect to the maximal ideal 
$k[\bar{P}-\{\bar{z}\}]\cap M(\mathcal{O}_{\Sigma, \A}^c)$.
$\mathcal{O}_{\Sigma, \A}^{c\:\prime}$ is a complete local subring of $\mathcal{O}_{\Sigma, \A}^c$.$\mathcal{O}_{\Sigma, \A}^{c\:\prime}$ is a $k$-subalgebra of $\mathcal{O}_{\Sigma, \A}^c$.$\bar{P}-\{\bar{z}\}$ is a parameter system of $\mathcal{O}_{\Sigma, \A}^{c\:\prime}$.
$\sigma^*(\A)(A')\subset \mathcal{O}_{\Sigma, \A}^{c\:\prime}$.
$\sigma^*(\A)(M(A'))\subset M(\mathcal{O}_{\Sigma, \A}^{c\:\prime})$.
$\rho(\mathcal{O}_{\Sigma, \A}^{c\:\prime})=A'$.

We consider the element $\sigma^*(\A)(\phi)\in\mathcal{O}_{\Sigma, \A}^c$.
$$\sigma^*(\A)(\phi)=\sum_{i\in\Z_0}\sigma^*(\A)(\phi'(i))\bar{z}^i.$$
For any $i\in\Z_0$, $\sigma^*(\A)(\phi'(i))\in \mathcal{O}_{\Sigma, \A}^{c\:\prime}$.

$\Supp(\sigma^{\prime*}\Spec(A'/\psi' A'))\subset\Supp(\sigma^{\prime*}(\Spec(A'/\psi' A')+ \Delta'))\subset \Supp(\bar{\Delta}')$.

Consider any $i\in I$. 

$0\neq \sigma^*(\A)(\phi'(i))\in \mathcal{O}_{\Sigma, \A}^{c\:\prime}$.
Since $\psi'\in\phi'(i)A'$ and 
$\Supp(\sigma^{\prime*}\Spec(A'/\psi' A'))\subset \Supp(\bar{\Delta}')$, $\sigma^*(\A)(\phi'(i))$ has normal crossings over $\bar{P}-\{\bar{z}\}$. 
We take the unique pair of elements $\bar{v}(i)\in(\mathcal{O}_{\Sigma, \A}^{c\:\prime})^\times$ and $\bar{c}(i)\in\Map(\bar{P}-\{\bar{z}\}, \Z_0)$ satisfying 
$\sigma^*(\A)(\phi'(i))= \bar{v}(i)\prod_{\bar{x}\in \bar{P}-\{\bar{z}\}}\bar{x}^{\bar{c}(i)(\bar{x})}$.

If $K\neq \emptyset$, then for any $(k,\ell)\in K$,
$0\leq k<\ell$, 
$\psi'\in(\phi'(k)-\phi'(\ell)) A'$ and the element $\sigma^*(\A)(\phi'(k)-\phi'(\ell))\in\mathcal{O}_{\Sigma, \A}^{c\:\prime}$ has normal crossings over $\bar{P}-\{\bar{z}\}$.
By Lemma~\ref{BM}, we know that either $\bar{c}(k) \leq \bar{c}(\ell)$ or $\bar{c}(\ell)\leq \bar{c}(k)$ holds for any $k\in I$ and any $\ell\in I$ with $k\neq\ell$.
Since $\{\bar{c}(k)|k\in I\}\neq\emptyset$,
we know that the set $\{\bar{c}(k) |k\in I\}$ has the minimum element $\min \{\bar{c}(k) |k\in I\}$ with respect to the partial order $\leq$.

We put $a=\min\{j\in I|\bar{c}(j)=
\min \{\bar{c}(k)|k\in I\}\}$ and $b=\min I$. $a\in\Z_0$. $b\in\Z_0$. $b\leq a$.
$\{\prod_{\bar{x}\in\bar{P}-\{\bar{z}\}}\bar{x}^{\bar{c}(i)(\bar{x})}\bar{z}^i|i\in I, b\leq i\leq a\}\mathcal{O}_{\Sigma, \A}^c=\{\bar{v}(i)\prod_{\bar{x}\in\bar{P}-\{\bar{z}\}}\bar{x}^{\bar{c}(i)(\bar{x})}\bar{z}^i|i\in I\}\mathcal{O}_{\Sigma, \A}^c
=\{\sigma^*(\A)(\phi(i)')\bar{z}^i|i\in I\}\mathcal{O}_{\Sigma, \A}^c
=\{\sigma^*(\A)(\phi(i)')\bar{z}^i|i\in \{0,1,\ldots,m\}\}\mathcal{O}_{\Sigma, \A}^c$.
Consider any $j\in\Z_0$ with $j>m$.
Since $\phi(j)'\in\{\phi(i)'| i\in \{0,1,\ldots,m\}\}A$, 
$\sigma^*(\A)(\phi(j)')\in \{\sigma^*(\A)(\phi(i)')|i\in \{0,1,\ldots,m\}\}\mathcal{O}_{\Sigma, \A}^{c\:\prime}$, and
$\sigma^*(\A)(\phi(j)')\bar{z}^j\in $\break$\{\sigma^*(\A)(\phi(i)') \bar{z}^i|i\in \{0,1,\ldots,m\}\}\mathcal{O}_{\Sigma, \A}^c$.
We know that $\{\prod_{\bar{x}\in\bar{P}-\{\bar{z}\}}\bar{x}^{\bar{c}(i)(\bar{x})}\bar{z}^i|i\in I, b\leq i\leq a\}\mathcal{O}_{\Sigma, \A}^c=\{\sigma^*(\A)(\phi(i)')\bar{z}^i|i\in \Z_0\}\mathcal{O}_{\Sigma, \A}^c$.
We know that there exists uniquely a pair of an element $w\in(\mathcal{O}_{\Sigma, \A}^c)^\times$ and a mapping $\zeta':\{0,1,\ldots, a-b-1\}\rightarrow M(\mathcal{O}_{\Sigma, \A}^{c\:\prime})$ satisfying 
$\sigma^*(\A)(\phi)=w\prod_{\bar{x}\in\bar{P}-\{\bar{z}\}}\bar{x}^{\bar{c}(a)(\bar{x})}\bar{z}^{b}
(\bar{z}^{a-b}+\sum_{i=0}^{a-b-1}\zeta'(i) \bar{z}^i)$ and
$\zeta'(0)\neq 0$ if $b<a$. We take $w\in(\mathcal{O}_{\Sigma, \A}^c)^\times$ and a mapping $\zeta':\{0,1,\ldots, a-b-1\}\rightarrow M(\mathcal{O}_{\Sigma, \A}^{c\:\prime})$ satisfying these conditions.

Let $\mathcal{R}=\{\bar{\chi}\in M(\mathcal{O}_{\Sigma, \A}^{c\:\prime})|
\bar{\chi}^h+\sum_{i=0}^{a-b-1}\zeta'(i) \bar{\chi}^i=0\}$.
$\mathcal{R}$ is a finite set. $\sharp\mathcal{R}\leq a-b$.
Consider any $\bar{\chi}\in\mathcal{R}$. Since $\mathcal{O}_{\Sigma, \A}^c$ is a unique factorization domain, there exists uniquely a positive integer $\mu(\bar{\chi})\in\Z_+$ satisfying $\bar{z}^{a-b}+\sum_{i=0}^{a-b-1}\zeta'(i) \bar{z}^i \in(\bar{z}-\bar{\chi})^{\mu(\bar{\chi})} \mathcal{O}_{\Sigma, \A}^c $ and 
$\bar{z}^{a-b}+\sum_{i=0}^{a-b-1}\zeta'(i) \bar{z}^i \not\in(\bar{z}-\bar{\chi})^{\mu(\bar{\chi})+1}\mathcal{O}_{\Sigma, \A}^c $. We take the unique $\mu(\bar{\chi})\in\Z_+$ satisfying these conditions. $\sum_{\bar{\chi}\in\mathcal{R}}\mu(\bar{\chi})\leq a-b$. Let $\hat{g}= a-b -\sum_{\bar{\chi}\in\mathcal{R}}\mu(\bar{\chi})\in\Z_0$. $\hat{g}\neq 1$. $\hat{g}\leq a-b$.
Let $g=\max\{\hat{g}, 1\}\in\Z_+$. 
There exists uniquely a mapping $\bar{\zeta}':\{0,1,\ldots, \hat{g}-1\}\rightarrow M(A')$ satisfying $\bar{z}^{a-b}+\sum_{i=0}^{a-b-1}\zeta'(i) \bar{z}^i =
\prod_{\bar{\chi}\in\mathcal{R}}(\bar{z}-\bar{\chi})^{\mu(\bar{\chi})}(\bar{z}^{\hat{g}}+\sum_{i=0}^{\hat{g}-1}\bar{\zeta}'(i) \bar{z}^i)$.
We take the unique mapping $\bar{\zeta}':\{0,1,\ldots, \hat{g}-1\}\rightarrow M(A')$ satisfying this equality.
$\bar{\chi}^{\hat{g}}+\sum_{i=0}^{\hat{g}-1}\bar{\zeta}'(i) \bar{\chi}^i\neq 0$ for any $\bar{\chi}\in M(\mathcal{O}_{\Sigma, \A}^{c\:\prime})$.

$\rho(w)\in A^\times$. $\rho(\bar{\chi})\in M(A')$ for any $\bar{\chi}\in \mathcal{R}$. $\rho(\bar{\zeta}'(i)) \in M(A')$ for any $i\in\{0,1,\ldots, \hat{g}-1\}$.
$\rho(w\prod_{\bar{x}\in\bar{P}-\{\bar{z}\}}\bar{x}^{\bar{c}(a)(\bar{x})}\bar{z}^b)=\rho(w)\prod_{x\in P-\{z\}}x^{\bar{c}(a)\rho^{-1}(x)}z^b\in PW(1)$.
$\rho(\prod_{\bar{\chi}\in\mathcal{R}}(\bar{z}-\bar{\chi})^{\mu(\bar{\chi})})=\prod_{\bar{\chi}\in\mathcal{R}}(z-\rho(\bar{\chi}))^{\mu(\bar{\chi})}\in PW(1)$.

If $\hat{g}\geq 2$, then $\hat{g}=g$ and $\rho(\bar{z}^{\hat{g}}+\sum_{i=0}^{\hat{g}-1}\bar{\zeta}'(i) \bar{z}^i)\in W(\hat{g})=W(g)$. Therefore, 
$\rho\sigma^*(\A)(\phi)=\rho(w)\prod_{x\in P-\{z\}}x^{\bar{c}(a)\rho^{-1}(x)}z^b \prod_{\bar{\chi}\in\mathcal{R}}(z-\rho(\bar{\chi}))^{\mu(\bar{\chi})} (z^{\hat{g}}+\sum_{i=0}^{\hat{g}-1}\rho(\bar{\zeta}'(i)) z^i$\break$) \in PW(g)$.

If $\hat{g}\leq 1$, then $\hat{g}=0$, $g=1$ and $z^{\hat{g}}+\sum_{i=0}^{\hat{g}-1}\rho(\bar{\zeta}'(i) )z^i=1\in PW(1)$. Therefore, 
$\rho\sigma^*(\A)(\phi)=\rho(w)\prod_{x\in P-\{z\}}x^{\bar{c}(g)\rho^{-1}(x)}z^b \prod_{\bar{\chi}\in\mathcal{R}}(z-\rho(\bar{\chi}))^{\mu(\bar{\chi})} (z^{\hat{g}}+\sum_{i=0}^{\hat{g}-1}\rho(\bar{\zeta}'(i)) z^i$\break$) \in PW(1)=PW(g)$.

We conclude that there exists $g\in\Z_+$ satisfying $\rho\sigma^*(\A)(\phi)\in PW(g)$ and 
$\rho\sigma^*(\A)(\phi)\in \cup_{h\in\Z_+}PW(h)$.

We conclude that Theorem~\ref{make Weierstrass type} holds.

We give the proof of Theorem~\ref{make normal crossings}. See Section~\ref{scheme} for notations and concepts related to normal crossing schemes.

Assume the above $(*)$ and consider any $\phi\in PW(1)$.

We take an element $u\in A^\times$, a finite subset $\mathcal{X}$ of  $M(A')$, a mapping $a: \mathcal{X}\rightarrow \Z_+$ and a mapping $b:P-\{z\}\rightarrow\Z_0$ satisfying
$$\phi= u\prod_{\chi\in\mathcal{X}}(z+\chi)^{a(\chi)}\prod_{x\in P-\{z\}} x^{b(x)}.$$
The quadruplet $(u,\mathcal{X},a,b)$ is uniquely determined depending on $\phi$, since $A$ is a unique factorization domain.

We consider the case $\mathcal{X}\subset \{0\}$.

$\mathcal{X}=\emptyset$ or $\mathcal{X}=\{0\}$.
We know that $\phi$ has normal crossings over $P$.

Let $\sigma: \Spec(A)\rightarrow\Spec(A)$ be the identity morphism $\Id_{\Spec(A)}$ of $\Spec(A)$. 
The scheme $\Spec(A)$ is a smooth scheme over $\Spec(A)$. The morphism $\sigma:\Spec(A)\rightarrow\Spec(A)$ is a weakly admissible composition of blowing-ups over $(\Delta,\xi)$ and the triplet $(\Spec(A), \Delta,\xi)$ is an extended pull-back of the coordinated normal crossing scheme $(\Spec(A),\Delta,\xi)$ by $\sigma$.

$\Supp(\Spec(A/\phi A))\subset\Supp(\Delta)$. $\Supp(\sigma^*(\Spec(A/\phi A)+\Delta))=\Supp(\Delta)$.

We know that Theorem~\ref{make normal crossings} holds, if  $\mathcal{X}\subset \{0\}$.

We consider the case $\mathcal{X}\not\subset \{0\}$. $\mathcal{X}-\{0\}\neq\emptyset$.

By $(\mathcal{X}-\{0\})_2$ we denote the set of all subsets $\mathcal{Y}$ of $\mathcal{X}-\{0\}$ with $\sharp\mathcal{Y}=2$.

Consider the case $\sharp(\mathcal{X}-\{0\})\geq 2$. $(\mathcal{X}-\{0\})_2\neq\emptyset$.
Consider any $\mathcal{Y}\in (\mathcal{X}-\{0\})_2$.
We choose any element $\chi_0(\mathcal{Y})\in\mathcal{Y}$.
The unique element in $\mathcal{Y}$ different from $\chi_0(\mathcal{Y})$ is denoted by $\chi_1(\mathcal{Y})$. $\chi_0(\mathcal{Y})\neq\chi_1(\mathcal{Y})$.
$\{\chi_0(\mathcal{Y}), \chi_1(\mathcal{Y})\}=\mathcal{Y}$.

In case $\sharp(\mathcal{X}-\{0\})=1$, we denote the unique element in $\mathcal{X}-\{0\}$ by $\chi_0$.

We put
\begin{equation*}
\phi'=
\begin{cases}
\prod_{\chi\in\mathcal{X}-\{0\}}\chi\prod_{\mathcal{Y}\in (\mathcal{X}-\{0\})_2}(\chi_0(\mathcal{Y})-\chi_1(\mathcal{Y}))&
\text{ if $\sharp(\mathcal{X}-\{0\})\geq 2$},\\
\chi_0&\text{ if $\sharp(\mathcal{X}-\{0\})=1$}.
\end{cases}\end{equation*}

$\phi'\in M(A')$ and $\phi'\neq 0$.

By $(*)$ we know that 
there exists a weakly admissible composition of blowing-ups $\hat{\sigma}':\hat{\Sigma}'\rightarrow\Spec(A')$ over $(\Delta',\xi')$ and an extended pull-back $(\hat{\Sigma}',\hat{\Delta}',\hat{\xi'})$ of the coordinated normal crossing scheme $(\Spec(A'),\Delta',\xi')$ by $\hat{\sigma}'$ satisfying $\Supp(\hat{\sigma}^{\prime*}(\Spec(A'/$\break$\phi' A')+ \Delta'))\subset \Supp(\hat{\Delta}')$.
We take a weakly admissible composition of blowing-ups $\hat{\sigma}':\hat{\Sigma}'\rightarrow\Spec(A')$ over $(\Delta',\xi')$ and an extended pull-back $(\hat{\Sigma}',\hat{\Delta}',\hat{\xi'})$ of the coordinated normal crossing scheme $(\Spec(A'),\Delta',\xi')$ by $\hat{\sigma}'$ satisfying $\Supp(\hat{\sigma}^{\prime*}(\Spec(A'/$\break$\phi' A')+ \Delta'))\subset \Supp(\hat{\Delta}')$. The structure sheaf of the scheme $\hat{\Sigma}'$ is denoted by $\mathcal{O}_{\hat{\Sigma}'}$.

We consider a morphism $\Spec(A)\rightarrow\Spec(A')$ induced by the inclusion ring homomorphism $A'\rightarrow A$, the product scheme $\hat{\Sigma}=\hat{\Sigma}'\times_{\Spec(A')}\Spec(A)$, the projection $\hat{\sigma}:\hat{\Sigma}\rightarrow\Spec(A)$, and the projection $\pi:\hat{\Sigma}\rightarrow\hat{\Sigma}'$.
The structure sheaf of the scheme $\hat{\Sigma}$ is denoted by $\mathcal{O}_{\hat{\Sigma}}$.
Let $\hat{\Delta}=\pi^*\hat{\Delta}'+\hat{\sigma}^*\Spec(A/zA)$.
We have a normal crossing scheme $(\hat{\Sigma},\hat{\Delta})$ with $(\hat{\Delta})_0\subset \hat{\sigma}^*\Spec(A/zA)$.
Consider any $\A\in(\hat{\Delta})_0$. We put
$\hat{\xi}_\A(\pi^*\Lambda)=\pi^*(U(\hat{\Sigma}',\hat{\Delta}',\pi(\A)))(\hat{\xi'}_{\pi(\A)}(\Lambda))$
for any $\Lambda\in \Comp(\hat{\Delta}')(\pi(\A))$, and we put
$\hat{\xi}_\A(\hat{\sigma}^*\Spec(A/zA))= \Res^{\hat{\Sigma}}_{U(\hat{\Sigma}, \hat{\Delta},\A)}\hat{\sigma}^*(\Spec(A))(z)$.
We have a coordinate system $\hat{\xi}_\A:\Comp(\hat{\Delta})(\A)\rightarrow\mathcal{O}_{\hat{\Sigma}}(U(\hat{\Sigma}, \hat{\Delta},\A))$ of $(\hat{\Sigma},\hat{\Delta})$ at $\A$, and a coordinate system $\hat{\xi}=\{\hat{\xi}_\A|\A\in(\hat{\Delta})_0\}$ of $(\hat{\Sigma},\hat{\Delta})$.

We develop some general theory of schemes.
Let $\Omega$ be any separated irreducible noetherian smooth scheme.

For any closed subset $E$ of $\Omega$, we denote the set of irreducible components of $E$ by $\Comp(E)$. The set $\Comp(E)$ is a finite set whose elements are non-empty irreducible closed subsets of $\Omega$. $E=\cup_{F\in\Comp(E)}F$.

Let $U$ be any non-empty open subset of $\Omega$ and let $D$ be any divisor of $\Omega$. Let $\iota_U: U\rightarrow\Omega$ denote the inclusion morphism. We denote the pull-back $\iota_U^*D$ of $D$ by $\iota_U$ by the symbol $D|U$ and we call it the \emph{restriction} of $D$ to $U$.

Let $r\in \Z_+$ be any positive integer; let $\mathcal{E}$ be any subset with $\sharp\mathcal{E}=r$ of the set $\Prm(\Omega)$ of all prime divisors of $\Omega$ and let $G$ be any irreducible component of $\cap_{E\in\mathcal{E}}E$ with $r=\Codim(G, \Omega)$.

Let $[G]\in\ G$ denote the generic point of $G$ and let $\iota:\Spec(\mathcal{O}_{\Omega, [G]})\rightarrow \Omega$ denote the canonical morphism. We take any $\epsilon_E\in M(\mathcal{O}_{\Omega, [G]})$ satisfying $\iota^*E=\Spec(\mathcal{O}_{\Omega, [G]}/\epsilon_E\mathcal{O}_{\Omega, [G]})$ for any $E\in\mathcal{E}$. The residue ring $\mathcal{O}_{\Omega, [G]}/\{\epsilon_E| E\in\mathcal{E}\}\mathcal{O}_{\Omega, [G]}$ is a local noetherian ring with dimension zero. 
We denote $(\mathcal{E};G)=\mathrm{length}(\mathcal{O}_{\Omega, [G]}/\{\epsilon_E|$\break$ E\in\mathcal{E}\}\mathcal{O}_{\Omega, [G]})\in\Z_+$ and we call $(\mathcal{E};G)$ the \emph{intersection number} of the set $\mathcal{E}$ of prime divisors at $G$.
The intersection number $(\mathcal{E};G)$ of $\mathcal{E}$ at $G$ depends only on the pair $(\mathcal{E}, G)$, and it does not depend on the choice of elements $\epsilon_E \in M(\mathcal{O}_{\Omega, [G]})$, $E\in\mathcal{E}$ we used for the definition. We also write $(E(1), E(2),\ldots, E(r);G)$ instead of $(\mathcal{E};G)$, if $\mathcal{E}=\{E(1), E(2),\ldots, E(r)\}$. For any bijective mapping $\tau: \{1,2,\ldots,r\}\rightarrow \{1,2,\ldots,r\}$, $(E(1), E(2),\ldots, E(r);G)= (E(\tau(1)), E(\tau(2)),\ldots, E(\tau(r));G)$.

We return to our situation under consideration. We denote $\hat{D}=\hat{\sigma}^*\Spec(A/zA)+\sum_{\chi\in\mathcal{X}-\{0\}}\hat{\sigma}^*\Spec(A/(z+\chi)A)\in\Div(\hat{\Sigma})$. The following claims holds:
\begin{enumerate}
\item
The morphism $\hat{\sigma}:\hat{\Sigma}\rightarrow\Spec(A)$ is a weakly admissible composition of blowing-ups over $(\Delta,\xi)$ and the triplet $(\hat{\Sigma}, \hat{\Delta},\hat{\xi})$ is an extended pull-back of $(\Spec(A),\Delta,\xi)$.
\item
The morphism $\pi:\hat{\Sigma}\rightarrow\hat{\Sigma}'$ is surjective. For any $\A'\in\hat{\Sigma}'$,
$\dim \hat{\Sigma}\times_{\hat{\Sigma}'}\Spec(\mathcal{O}_{\hat{\Sigma}', \A'}/M(\mathcal{O}_{\hat{\Sigma}', \A'}))=1$.
\item
$\hat{\sigma}^*\Spec(A/zA)$ is a smooth prime divisor of $\hat{\Sigma}$. For any $\chi\in\mathcal{X}-\{0\}$, $\hat{\sigma}^*\Spec(A/(z+\chi)A)$ is a smooth prime divisor of $\hat{\Sigma}$. $\Comp(\hat{D})=\{\hat{\sigma}^*\Spec($\break$A/zA)\}\cup\{\hat{\sigma}^*\Spec(A/(z+\chi)A)| \chi\in\mathcal{X}-\{0\}\}$.
\item 
Consider any $\Lambda\in\Comp(\hat{D})$.

The component $\Lambda$ is smooth, 
the induced morphism $\pi: \Lambda\rightarrow\hat{\Sigma}'$ by $\pi$ is an isomorphism and the pair $(\hat{\Sigma}, \Lambda+\pi^*\hat{\Delta}')$ is a normal crossing scheme over $k$.
\item
For any $\Lambda\in\Comp(\hat{D})$ and any $\Gamma\in\Comp(\hat{D})$ with $\Lambda\neq\Gamma$, $\emptyset\neq\Lambda\cap\Gamma\subset \Supp(\pi^*\hat{\Delta}')$.
\item
$\Comp(\hat{D})\cap\Comp(\hat{\Delta})=\{\hat{\sigma}^*\Spec(A/zA)\}$.
\item
$\Supp(\hat{\sigma}^*(\Spec(A/\phi A)+\Delta))\subset\Supp(\hat{\Delta}+\hat{D})$.
\end{enumerate}

We consider any pair $(\tau, (\Sigma,\bar{\Delta},\bar{\xi}))$ of a weakly admissible composition of blowing-ups $\tau:\Sigma\rightarrow\hat{\Sigma}$ over $(\hat{\Sigma},\hat{\Delta})$ and an extended pull-back $(\Sigma,\bar{\Delta},\bar{\xi})$ of $(\hat{\Sigma},\hat{\Delta},\hat{\xi})$ by $\tau$ satisfying the following four conditions (Z).
Let $\bar{D}$ denote the sum of strict transforms of elements in $\Comp(\hat{D})$ by $\tau$. $\bar{D}\in\Div(\Sigma)$. Let $\bar{\Delta}_0=\sum_{\Gamma\in\Comp(\bar{\Delta})-\Comp(\bar{D})}\Gamma\in\Div(\Sigma)$:
\begin{enumerate}
\item
Consider any $\Lambda\in\Comp(\bar{D})$.

The component $\Lambda$ is smooth.

Let $U_\Lambda=\Sigma-(\cup_{\Gamma\in\Comp(\bar{\Delta}_0), \Gamma\cap\Lambda=\emptyset}\Gamma)\subset\Sigma$. Note that the subset $U_\Lambda$ is open in $\Sigma$ and it contains $\Lambda$.

The pair $(U_\Lambda, (\Lambda+\bar{\Delta}_0)|U_\Lambda)$ is a normal crossing scheme over $k$.

We take the unique element $\chi\in\mathcal{X}\cup\{0\}$ such that $\Lambda$ is the strict transform of $\hat{\sigma}^*\Spec(A/(z+\chi)A)$ by $\tau$. 
$\Supp(\Lambda+\bar{\Delta}_0)\cap U_\Lambda=\tau^{-1}(\Supp(\hat{\sigma}^*\Spec(A/(z+\chi)A)+\pi^*\hat{\Delta}'))\cap U_\Lambda$.
\item
For any $\Lambda\in\Comp(\bar{D})$ and any $\Gamma\in\Comp(\bar{D})$ with $\Lambda\neq\Gamma$, $\Lambda\cap\Gamma\subset\Supp(\bar{\Delta}_0)$.
\item
For any $\Lambda\in\Comp(\bar{\Delta})\cap\Comp(\bar{D})$ and any $\Gamma\in\Comp(\bar{\Delta})\cap\Comp(\bar{D})$ with $\Lambda\neq\Gamma$, $\Lambda\cap\Gamma=\emptyset$.
\item
For any $\Lambda\in\Comp(\bar{D})$, there exists an element $\Gamma\in\Comp(\bar{\Delta})\cap\Comp(\bar{D})$ with $\Lambda\cap\Gamma\neq\emptyset$. 
\end{enumerate}

$\Supp((\hat{\sigma}\tau)^*(\Spec(A/\phi A)+\Delta))\subset
\Supp(\tau^*(\hat{\Delta}+\hat{D}))\subset
\Supp(\bar{\Delta}+\bar{D})$.

$\sharp\Comp(\bar{D})=\sharp\Comp(\hat{D})\geq 2$. 
$1\leq\sharp(\Comp(\bar{\Delta})\cap\Comp(\bar{D}))\leq \sharp\Comp(\bar{D})$.

If $\sharp(\Comp(\bar{\Delta})\cap\Comp(\bar{D}))=\sharp\Comp(\bar{D})$, then $\Supp(\bar{\Delta}+\bar{D})=\Supp\bar{\Delta}$ and 
$\Supp($\break$ (\hat{\sigma}\tau)^*(\Spec(A/\phi A)+\Delta))\subset\Supp\bar{\Delta}$.
Since $\Sigma$ is a separated irreducible noetherian smooth scheme over $\Spec(A)$, the morphism $\hat{\sigma}\tau:\Sigma\rightarrow\Spec(A)$ is a weakly admissible composition of blowing-ups over $(\Delta,\xi)$ and the triplet $(\Sigma,\bar{\Delta},\bar{\xi})$ is an extended pull-back of $(\Spec(A),\Delta, \xi)$ by $\hat{\sigma}\tau$, we know that Theorem~\ref{make normal crossings} holds, if $\sharp(\Comp(\bar{\Delta})\cap\Comp(\bar{D}))=\sharp\Comp(\bar{D})$.

We would like to show that there exists a pair $(\tau,(\Sigma,\bar{\Delta},\bar{\xi}))$ satisfying the above four conditions (Z) and $\sharp(\Comp(\bar{\Delta})\cap\Comp(\bar{D}))=\sharp\Comp(\bar{D})$.

Note that the pair $(\Id_{\hat{\Sigma}},(\hat{\Sigma},\hat{\Delta},\hat{\xi}))$ satisfies the above four conditions (Z) and $\sharp(\Comp(\hat{\Delta})\cap\Comp(\hat{D}))=1$.

We use induction on $\sharp\Comp(\bar{D})-\sharp(\Comp(\bar{\Delta})\cap\Comp(\bar{D}))$.
Assume that the pair  $(\tau,(\Sigma,\bar{\Delta},\bar{\xi}))$ satisfies the above four conditions (Z) and $\sharp(\Comp(\bar{\Delta})\cap\Comp(\bar{D}))<\sharp\Comp(\bar{D})$.

We consider any admissible composition of blowing-ups $\nu:\tilde{\Sigma}\rightarrow\Sigma$ over $\bar{\Delta}$ such that the pair $(\tau\nu, (\tilde{\Sigma},\nu^*\bar{\Delta}, \nu^*\bar{\xi}))$ satisfies three conditions in the above four conditions (Z) except the last one. 

Let $\tilde{D}$ denote the sum of strict transforms of elements in $\Comp(\bar{D})$ by $\nu$. $\tilde{D}\in\Div(\tilde{\Sigma})$. The divisor  $\tilde{D}$ is equal to the sum of strict transforms of elements in $\Comp(\hat{D})$ by $\tau\nu$.
Since $\Comp(\bar{D})-\Comp(\bar{\Delta})\neq\emptyset$, we know $\Comp(\tilde{D})-\Comp(\nu^*\bar{\Delta})\neq\emptyset$.

Let $\tilde{\Delta}_0=\sum_{\Gamma\in \Comp(\nu^*\bar{\Delta})-\Comp(\tilde{D})}\Gamma\in\Div(\tilde{\Sigma})$.
$\Supp(\tilde{\Delta}_0)\subset\Supp(\nu^*\bar{\Delta})$.

We consider the case where the pair $(\tau\nu, (\tilde{\Sigma},\nu^*\bar{\Delta}, \nu^*\bar{\xi}))$ satisfies also the last condition in (Z). For any $\Lambda\in \Comp(\tilde{D})-\Comp(\nu^*\bar{\Delta})$, there exists an element $\Gamma\in \Comp(\tilde{D})\cap\Comp(\nu^*\bar{\Delta})$ satisfying $\Lambda\cap\Gamma\neq\emptyset$.

We take any $\Lambda_0\in \Comp(\tilde{D})-\Comp(\nu^*\bar{\Delta})$.

There exists an element $\Gamma\in \Comp(\tilde{D})\cap\Comp(\nu^*\bar{\Delta})$ satisfying $\Lambda_0\cap\Gamma\neq\emptyset$.

We denote
$$\mathcal{Z}=\{\mathcal{L}\subset\Comp(\tilde{D})-\{\Lambda_0\}|
\mathcal{L}\cap\Comp(\nu^*\bar{\Delta})\neq\emptyset,
\Lambda_0\cap(\bigcap_{\Gamma\in\mathcal{L}}\Gamma)\neq\emptyset\}.$$
$\mathcal{Z}\neq\emptyset$.

Let $\mathcal{L}$ be any maximal element in $\mathcal{Z}$ with respect to the inclusion relation. $\mathcal{L}\subset\Comp(\tilde{D})$.  $\Lambda_0\not\in \mathcal{L}$. $\mathcal{L}\cap\Comp(\nu^*\bar{\Delta})\neq\emptyset$. $\Lambda_0\cap(\bigcap_{\Gamma\in\mathcal{L}}\Gamma)\neq\emptyset$. 
For any $\Theta\in \Comp(\tilde{D})$ satisfying $\Theta\neq\Lambda_0$ and $\Theta\not\in \mathcal{L}$, $\Lambda_0\cap(\bigcap_{\Gamma\in\mathcal{L}}\Gamma)\cap\Theta=\emptyset$.

Consider any $\Gamma\in\mathcal{L}$. 
$\Lambda_0\neq\Gamma$ and every irreducible component of $\Lambda_0\cap\Gamma$ has codimension two in $\tilde{\Sigma}$.

Consider any irreducible component $\Phi$ of $\Lambda_0\cap\Gamma$. 
Since $\Lambda_0\cap\Gamma\subset \Supp(\tilde{\Delta}_0)$, we know that there exists $\Theta\in\Comp(\tilde{\Delta}_0)$ with $\Phi\subset\Theta$. We take $\Theta\in\Comp(\tilde{\Delta}_0)$ with $\Phi\subset\Theta$. $\Phi\subset \Lambda_0\cap\Gamma\subset\Lambda_0$ and $\Phi\subset \Theta\cap\Lambda_0$.
Since $\Lambda_0\in \Comp(\tilde{D})$ and $\Theta\not\in\Comp(\tilde{D})$, $\Theta\neq\Lambda_0$. Any irreducible component of $\Theta\cap\Lambda_0$ has codimension two in $\tilde{\Sigma}$. We know $\dim \Phi=\dim \Theta\cap\Lambda_0$. Since there exists an open subset $U_{\Lambda_0}$ of $\tilde{\Sigma}$ satisfying $\Lambda_0\subset U_{\Lambda_0}$ and $(U_{\Lambda_0}, (\Lambda_0+\tilde{\Delta}_0)| U_{\Lambda_0})$ is a normal crossing scheme over $k$, we know that $\Theta\cap\Lambda_0$ is irreducible, $\Phi=\Theta\cap\Lambda_0$ and if $\Theta'\in \Comp(\tilde{\Delta}_0)$ and $\Phi=\Theta'\cap\Lambda_0$, then $\Theta'=\Theta$.
Similarly we know $\Phi=\Theta\cap\Gamma$, since there exists an open subset $U_\Gamma$ of $\tilde{\Sigma}$ satisfying $\Gamma\subset U_\Gamma$ and $( U_\Gamma, (\Gamma+\tilde{\Delta}_0)| U_\Gamma)$ is a normal crossing scheme over $k$.
If $\Gamma\in \Comp(\nu^*\bar{\Delta})$, then it follows that $\Phi$ is a stratum of the normal crossing divisor $\nu^*\bar{\Delta}$ of codimension two contained in $\Lambda_0\cap\Gamma$. 

For any $\Gamma\in\mathcal{L}$ and any irreducible component $\Phi_\Gamma$ of $\Lambda_0\cap\Gamma$, we take the unique element $\Theta_{\Phi_\Gamma}\in \Comp(\tilde{\Delta}_0)$ with $\Phi_\Gamma=\Theta_{\Phi_\Gamma}\cap\Lambda_0$.
$\Theta_{\Phi_\Gamma}\cap\Lambda_0=\Theta_{\Phi_\Gamma}\cap\Gamma$. We take any element $\Gamma_0\in \Comp(\nu^*\bar{\Delta})\cap\mathcal{L}$. For any irreducible component $\Phi_{\Gamma_0}$ of $\Lambda_0\cap\Gamma_0$,
$\Theta_{\Phi_{\Gamma_0}}\cap\Lambda_0=\Theta_{\Phi_{\Gamma_0}}\cap\Gamma_0$.

We denote $\mathcal{N}=\{\Phi\in\Map(\mathcal{L},\cup_{\Gamma\in\mathcal{L}}\Comp(\Lambda_0\cap\Gamma))|$For any $\Gamma\in\mathcal{L}$, $\Phi(\Gamma)\in \Comp(\Lambda_0\cap\Gamma)\}$.
\begin{equation*}\begin{split}
\emptyset&\neq
\Lambda_0\cap(\bigcap_{\Gamma\in\mathcal{L}}\Gamma)
=\bigcap_{\Gamma\in\mathcal{L}}(\Lambda_0\cap\Gamma)
=\bigcap_{\Gamma\in\mathcal{L}}(\bigcup_{\Phi_\Gamma\in\Comp(\Lambda_0\cap\Gamma)} \Phi_\Gamma)
=\bigcup_{\Phi\in\mathcal{N}}
(\bigcap_{\Gamma\in\mathcal{L}}\Phi(\Gamma))\\
&=\bigcup_{\Phi\in\mathcal{N}}
(\bigcap_{\Gamma\in\mathcal{L}}\Theta_{\Phi(\Gamma)}\cap\Lambda_0)
=\bigcup_{\Phi\in\mathcal{N}}((\bigcap_{\Gamma\in\mathcal{L}-\{\Gamma_0\}}\Theta_{\Phi(\Gamma)})\cap(\Theta_{\Phi(\Gamma_0)}\cap\Lambda_0))\\
&=\bigcup_{\Phi\in\mathcal{N}}((\bigcap_{\Gamma\in\mathcal{L}-\{\Gamma_0\}}\Theta_{\Phi(\Gamma)})\cap(\Theta_{\Phi(\Gamma_0)}\cap\Gamma_0))
=\bigcup_{\Phi\in\mathcal{N}}
((\bigcap_{\Gamma\in\mathcal{L}}\Theta_{\Phi(\Gamma)})\cap\Gamma_0).
\end{split}\end{equation*}

We know that any irreducible component $\Psi$ of $\Lambda_0\cap(\bigcap_{\Gamma\in\mathcal{L}}\Gamma)$ is a stratum of the normal crossing divisor $\nu^*\bar{\Delta}$ with $\Codim(\Psi,\tilde{\Sigma})\geq 2$ and it is a stratum of $\Gamma+\tilde{\Delta}_0$ contained in $\Gamma$ for any $\Gamma\in\mathcal{L}\cup\{\Lambda_0\}$.

Consider any irreducible component $\Psi$ of  $\Lambda_0\cap(\bigcap_{\Gamma\in\mathcal{L}}\Gamma)$. $\Psi$ is a stratum of $\nu^*\bar{\Delta}$. $\Codim(\Psi, \tilde{\Sigma})\geq 2$. For any $\Gamma\in\mathcal{L}\cup\{\Lambda_0\}$, $\Psi\subset\Gamma$. For any $\Gamma\in\Comp(\tilde{D})-(\mathcal{L}\cup\{\Lambda_0\})$, $\Psi\cap\Gamma=\emptyset$.

Let $\nu_1:\Sigma_1\rightarrow\tilde{\Sigma}$ denote the blowing-up with center in $\Psi$. The morphism $\nu_1$ is an admissible blowing-up over $\nu^*\bar{\Delta}$.
The composition $\nu\nu_1: \Sigma_1\rightarrow\Sigma$ is an admissible composition of blowing-ups over $\bar{\Delta}$.

The exceptional divisor of $\nu_1$ is denoted by $\Psi_1$. $\Psi_1=\nu_1^{-1}(\Psi)\in\Div(\Sigma_1)$. For any closed irreducible subset $\Phi$ of  $\tilde{\Sigma}$ with $\Phi\not\subset\Psi$, by $\Phi_1$ we denote the closure of $\nu_1^{-1}(\Phi-\Psi)$ in $\Sigma_1$. $\Phi_1$ is a closed irreducible subset $\Phi$ of  $\Sigma_1$, $\nu_1(\Phi_1)=\Phi$, the induced morphism $\nu_1:\Phi_1\rightarrow \Phi$ by $\nu_1$ is birational and $\nu_1^*:\mathcal{O}_{\tilde{\Sigma},[\Phi]}\rightarrow\mathcal{O}_{\Sigma_1,[\Phi_1]}$ is an isomorphism. For any prime divisor $\Gamma$ of $\tilde{\Sigma}$, the strict transform of $\Gamma$ by $\nu_1$ is equal to $\Gamma_1$.

For any $\Gamma\in\mathcal{L}\cup\{\Lambda_0\}$, $\nu_1^*\Gamma=\Gamma_1+\Psi_1$. For any $\Gamma\in\Comp(\tilde{D})-(\mathcal{L}\cup\{\Lambda_0\})$, $\Psi\cap\Gamma=\emptyset$ and $\nu_1^*\Gamma=\Gamma_1$. We know that for any $\Gamma\in\Comp(\tilde{D})$, $\Gamma_1$ is smooth.

Let $$D_1=\sum_{\Gamma\in\Comp(\tilde{D})}\Gamma_1\in\Div(\Sigma_1)\text{ and }
\Delta_{10}=\sum_{\Gamma\in\Comp((\nu\nu_1)^*\bar{\Delta})-\Comp(D_1)}\Gamma\in\Div(\Sigma_1).$$ The divisor $D_1$ is equal to the sum of strict transforms of elements in $\Comp(\hat{D})$ by $\tau\nu\nu_1$.
$\Comp((\nu\nu_1)^*\bar{\Delta})=\{\Gamma_1|\Gamma\in\Comp(\nu^*\bar{\Delta})\}\cup\{\Psi_1\}$. $\Comp(\Delta_{10})=\{\Gamma_1|\Gamma\in\Comp(\tilde{\Delta}_0)\}\cup\{\Psi_1\}$. $\Psi_1\subset\Supp(\Delta_{10})=\nu_1^{-1}(\Supp(\tilde{\Delta}_0))$.

We would like to show that the pair $(\tau\nu\nu_1, (\Sigma_1, (\nu\nu_1)^*\bar{\Delta}, (\nu\nu_1)^*\bar{\xi}))$ satisfies three conditions in the above four conditions (Z) except the last one.
Recall that pair $(\tau\nu, (\tilde{\Sigma}, \nu^*\bar{\Delta}, \nu^*\bar{\xi}))$ satisfies the four conditions in (Z).

By the just above we know that any component of $D_1$ is smooth.

Consider any $\Lambda\in\Comp(\tilde{D})$.
$\Lambda_1\in\Comp(D_1)$. We denote
$$\tilde{U}_\Lambda=\tilde{\Sigma}-(\cup_{\Gamma\in\tilde{\Delta}_0, \Gamma\cap\Lambda=\emptyset}\Gamma)\text{ and }U_{1 \lambda}=\Sigma_1-(\cup_{\Gamma\in\Delta_{10}, \Gamma\cap\Lambda_1=\emptyset}\Gamma).$$
The pair $(\tilde{U}_\Lambda, (\Lambda+\tilde{\Delta}_0)| \tilde{U}_\Lambda)$ is a normal crossing scheme over $k$. We take the unique element $\chi\in\mathcal{X}\cup\{0\}$ such that $\Lambda$ is the strict transform of $\hat{\sigma}^*\Spec(A/(z+\chi)A)$ by $\tau\nu$. 
$\Lambda_1$ is the strict transform of $\hat{\sigma}^*\Spec(A/(z+\chi)A)$ by $\tau\nu\nu_1$. 
$\Supp(\Lambda+\tilde{\Delta}_0)\cap \tilde{U}_\Lambda=(\tau\nu)^{-1}(\Supp(\hat{\sigma}^*\Spec(A/(z+\chi)A)+\pi^*\hat{\Delta}'))\cap \tilde{U}_\Lambda$.

For any $\Gamma\in\Comp(\tilde{\Delta}_0)$ with $\Gamma\cap\Lambda=\emptyset$, $\Gamma_1\in\Comp(\Delta_{10})$, $\Gamma_1\cap\Lambda_1=\emptyset$ and $\Gamma_1=\nu_1^{-1}(\Gamma)$. We know $\Lambda_1\subset U_{1 \lambda}\subset\nu_1^{-1}(\tilde{U}_\Lambda)$ and $U_{1 \lambda}\neq\nu_1^{-1}(\tilde{U}_\Lambda)$, if and only if, there exists $\Gamma\in\tilde{\Delta}_0$ with $\Gamma\cap\Lambda\neq\emptyset$ and $\Gamma_1\cap\Lambda_1=\emptyset$.

$\Lambda_1\cup\Psi_1=\nu_1^{-1}(\Lambda)\cup\Psi_1$. 
$\Supp(\Lambda_1+\Delta_{10})=\Lambda_1\cup\Supp(\Delta_{10})=\Lambda_1\cup\Psi_1\cup\Supp(\Delta_{10})
=\nu_1^{-1}(\Lambda)\cup\Psi_1\cup\Supp(\Delta_{10})
=\nu_1^{-1}(\Lambda)\cup\Supp(\Delta_{10})
=\nu_1^{-1}(\Lambda)\cup\nu_1^{-1}(\Supp(\tilde{\Delta}_0))=\nu_1^{-1}(\Lambda\cup\Supp(\tilde{\Delta}_0))=\nu_1^{-1}(\Supp(\Lambda+\tilde{\Delta}_0))$.

$\Supp(\Lambda_1+\Delta_{10})\cap \nu_1^{-1}(\tilde{U}_\Lambda)
=\nu_1^{-1}(\Supp(\Lambda+\tilde{\Delta}_0)) \cap \nu_1^{-1}(\tilde{U}_\Lambda)
=\nu_1^{-1}(\Supp(\Lambda+\tilde{\Delta}_0)\cap \tilde{U}_\Lambda)
=\nu_1^{-1}((\tau\nu)^{-1}(\Supp(\hat{\sigma}^*\Spec(A/(z+\chi)A)+\pi^*\hat{\Delta}'))\cap \tilde{U}_\Lambda)
=(\tau\nu\nu_1)^{-1}(\Supp($\break$\hat{\sigma}^*\Spec(A/(z+\chi)A)+\pi^*\hat{\Delta}'))\cap\nu_1^{-1}(\tilde{U}_\Lambda)$.
We know that
$\Supp(\Lambda_1+\Delta_{10})\cap U_{1 \lambda}
=(\tau\nu\nu_1)^{-1}(\Supp(\hat{\sigma}^*\Spec(A/(z+\chi)A)+\pi^*\hat{\Delta}'))\cap U_{1 \lambda}$.

We know that $(\nu_1^{-1}(\tilde{U}_\Lambda), (\Lambda_1+\Delta_{10})|\nu_1^{-1}(\tilde{U}_\Lambda))$ is a normal crossing scheme over $k$. We know that if there does not exist $\Gamma\in\tilde{\Delta}_0$ with $\Gamma\cap\Lambda\neq\emptyset$ and $\Gamma_1\cap\Lambda_1=\emptyset$, then $(U_{1 \lambda}, (\Lambda_1+\Delta_{10})|U_{1 \lambda})$ is a normal crossing scheme over $k$.

We consider the case where there exists $\Gamma\in\tilde{\Delta}_0$ with $\Gamma\cap\Lambda\neq\emptyset$ and $\Gamma_1\cap\Lambda_1=\emptyset$.
Take any $\Gamma\in\tilde{\Delta}_0$ with $\Gamma\cap\Lambda\neq\emptyset$ and $\Gamma_1\cap\Lambda_1=\emptyset$. 
If $\Gamma\cap\Lambda\not\subset\Psi$, then $\Gamma_1\cap\Lambda_1\neq\emptyset$.
Therefore, $\Gamma\cap\Lambda \subset\Psi$. 
Since $\emptyset\neq \Gamma\cap\Lambda \subset\Psi\cap \Lambda$, $\Lambda\in\mathcal{L}\cup\{\Lambda_0\}$.
Since $\Codim(\Psi,\tilde{\Sigma})\geq 2$, we know $\Gamma\cap\Lambda=\Psi$ and $\Codim(\Psi,\tilde{\Sigma})=2$.
We know that if $\Gamma'\in\tilde{\Delta}_0$,  $\Gamma'\cap\Lambda\neq\emptyset$ and $\Gamma'_1\cap\Lambda_1=\emptyset$, then $\Gamma'=\Gamma$, since the pair $(\tilde{U}_\Lambda, (\Lambda+\tilde{\Delta}_0)| \tilde{U}_\Lambda)$ is a normal crossing scheme over $k$.
We know that $U_{1 \lambda}=\nu_1^{-1}(\tilde{U}_\Lambda)-\Gamma$ and we conclude
$(U_{1 \lambda}, (\Lambda_1+\Delta_{10})|U_{1 \lambda})$ is a normal crossing scheme over $k$, if there exists $\Gamma\in\tilde{\Delta}_0$ with $\Gamma\cap\Lambda\neq\emptyset$ and $\Gamma_1\cap\Lambda_1=\emptyset$.

We know that the pair $(U_{1 \lambda}, (\Lambda_1+\Delta_{10})|U_{1 \lambda})$ is a normal crossing scheme over $k$.

We know that the pair $(\tau\nu\nu_1, (\Sigma_1, (\nu\nu_1)^*\bar{\Delta}, (\nu\nu_1)^*\bar{\xi}))$ satisfies the first condition in the above four conditions (Z).

Consider any $\Lambda\in\Comp(\tilde{D})$ and any $\Gamma\in\Comp(\tilde{D})$ with $\Lambda\neq\Gamma$. $\Lambda\cap\Gamma\subset\Supp(\tilde{\Delta}_0)$, since the pair $(\tau\nu, (\tilde{\Sigma}, \nu^*\bar{\Delta}, \nu^*\bar{\xi}))$ satisfies the second condition of (Z). 
$\Lambda_1\cap\Gamma_1\subset\nu_1^{-1}(\Lambda\cap\Gamma)\subset\nu_1^{-1}(\Supp(\tilde{\Delta}_0))=\Supp(\Delta_{10})$. 

We know that the pair $(\tau\nu\nu_1, (\Sigma_1, (\nu\nu_1)^*\bar{\Delta}, (\nu\nu_1)^*\bar{\xi}))$ satisfies the second condition in the above four conditions (Z).

Consider any $\hat{\Lambda}\in\Comp((\nu\nu_1)^*\bar{\Delta})\cap\Comp(D_1)$ and any $\hat{\Gamma}\in\Comp((\nu\nu_1)^*\bar{\Delta})\cap\Comp(D_1)$ with $\hat{\Lambda}\neq\hat{\Gamma}$.
Since $\hat{\Lambda}\in\Comp(D_1)$, $\hat{\Lambda}\neq\Psi_1$ and there exists uniquely $\Lambda\in\Comp(\tilde{D})$ with $\hat{\Lambda}=\Lambda_1$.
We take the unique $\Lambda\in\Comp(\tilde{D})$ with $\hat{\Lambda}=\Lambda_1$.
Since $\hat{\Lambda}\in\Comp((\nu\nu_1)^*\bar{\Delta})$, $\Lambda\in \Comp(\nu^*\bar{\Delta})$. $\Lambda\in \Comp(\nu^*\bar{\Delta})\cap \Comp(\tilde{D})$. 
Similarly, we know there exists uniquely $\Gamma\in\Comp(\tilde{D})$ with $\hat{\Gamma}=\Gamma_1$. We take the unique $\Gamma\in\Comp(\tilde{D})$ with $\hat{\Gamma}=\Gamma_1$. $\Gamma\in \Comp(\nu^*\bar{\Delta})\cap \Comp(\tilde{D})$. Since $\Lambda_1=\hat{\Lambda}\neq\hat{\Gamma}=\Gamma_1$, $\Lambda\neq \Gamma$. We know $\Lambda\cap\Gamma=\emptyset$, since the pair $(\tau\nu, (\tilde{\Sigma}, \nu^*\bar{\Delta}, \nu^*\bar{\xi}))$ satisfies the third condition of (Z).
$\hat{\Lambda}\cap\hat{\Gamma}=\Lambda_1\cap\Gamma_1\subset\nu_1^{-1}(\Lambda\cap\Gamma)=\emptyset$. We know $\hat{\Lambda}\cap\hat{\Gamma}=\emptyset$.

We know that the pair $(\tau\nu\nu_1, (\Sigma_1, (\nu\nu_1)^*\bar{\Delta}, (\nu\nu_1)^*\bar{\xi}))$ satisfies the third condition in the above four conditions (Z).

We conclude that the pair $(\tau\nu\nu_1, (\Sigma_1, (\nu\nu_1)^*\bar{\Delta}, (\nu\nu_1)^*\bar{\xi}))$ satisfies three conditions in the above four conditions (Z) except the last one.

Recall that $2\leq\Codim(\Psi,\tilde{\Sigma})\in\Z_0$. We know that there exists a subset $\mathcal{M}$ of $\mathcal{L}$ such that $\sharp\mathcal{M}=\Codim(\Psi,\tilde{\Sigma})-1$ and $\Psi$ is an irreducible component of $\Lambda_0\cap(\cap_{\Gamma\in\mathcal{M}}\Gamma)$. We take any subset $\mathcal{M}$ of $\mathcal{L}$ satisfying these conditions.

Let $\delta:\Spec(\mathcal{O}_{\tilde{\Sigma}, [\Psi]}^c)\rightarrow \tilde{\Sigma}$ denote the canonical morphism, where $\mathcal{O}_{\tilde{\Sigma}, [\Psi]}^c$ denotes the completion of the local ring of $\tilde{\Sigma}$ at the generic point $[\Psi]$ of $\Psi$. For simplicity we denote $R=\mathcal{O}_{\tilde{\Sigma}, [\Psi]}^c$. There exists an element $\tilde{z}\in M(R)-M(R)^2$ with $\delta^*\Lambda_0=\Spec(R/\tilde{z}R)$. We take an element $\tilde{z}\in M(R)-M(R)^2$ with $\delta^*\Lambda_0=\Spec(R/\tilde{z}R)$. We know that there exist a parameter system $\tilde{P}$ of $R$ with $\tilde{z}\in\tilde{P}$, a bijective mapping $\tilde{y}:\mathcal{M}\rightarrow \tilde{P}-\{\tilde{z}\}$, a mapping $\tilde{e}:\mathcal{M}\rightarrow \Z_+$ and a mapping $\tilde{u}:\mathcal{M}\rightarrow R^{\times}$ satisfying $\delta^*\Gamma=
\Spec(R/(\tilde{z}+\tilde{y}(\Gamma)^{\tilde{e}(\Gamma)}\tilde{u}(\Gamma))R)$ for any $\Gamma\in\mathcal{M}$. We take a parameter system $\tilde{P}$ of $R$ with $\tilde{z}\in\tilde{P}$, a bijective mapping $\tilde{y}:\mathcal{M}\rightarrow \tilde{P}-\{\tilde{z}\}$, a mapping $\tilde{e}:\mathcal{M}\rightarrow \Z_+$ and a mapping $\tilde{u}:\mathcal{M}\rightarrow R^{\times}$ satisfying $\delta^*\Gamma=
\Spec(R/(\tilde{z}+\tilde{y}(\Gamma)^{\tilde{e}(\Gamma)}\tilde{u}(\Gamma))R)$ for any $\Gamma\in\mathcal{M}$. By calculation we know the following:
\begin{enumerate}
\item
The following three conditions are equivalent:
\begin{enumerate}
\item
There exists an irreducible component $\Psi_{(1)}$ of $\Lambda_{01}\cap(\cap_{\Gamma\in\mathcal{M}}\Gamma_1)$ with $\nu_1(\Psi_{(1)})=\Psi$.
\item
There exists uniquely an irreducible component $\Psi_{(1)}$ of $\Lambda_{01}\cap(\cap_{\Gamma\in\mathcal{M}}\Gamma_1)$ with $\nu_1(\Psi_{(1)})=\Psi$.
\item $\tilde{e}(\Gamma)\geq 2$ for some $\Gamma\in\mathcal{M}$.
\end{enumerate}
\end{enumerate}

Below, we assume that $\tilde{e}(\Gamma)\geq 2$ for some $\Gamma\in\mathcal{M}$ and $\Psi_{(1)}$ is the unique irreducible component of $\Lambda_{01}\cap(\cap_{\Gamma\in\mathcal{M}}\Gamma_1)$ with $\nu_1(\Psi_{(1)})=\Psi$.

\begin{enumerate}
\setcounter{enumi}{1}
\item
$\dim \Psi_{(1)}=\dim \Psi+\sharp\{\Gamma\in\mathcal{M}|\tilde{e}(\Gamma)\geq 2\}-1\geq\dim \Psi$. 
\item
Assume that $\sharp\{\Gamma\in\mathcal{M}|\tilde{e}(\Gamma)\geq 2\}=1$. 
We take the unique element $\Gamma_0\in\mathcal{M}$ with $\tilde{e}(\Gamma_0)\geq 2$.

The intersection number of divisors $\Lambda_0$ and $\Gamma$, $\Gamma\in\mathcal{M}$ at $\Psi$ is equal to $\tilde{e}(\Gamma_0)$.

The intersection number of divisors $\Lambda_{01}$ and $\Gamma_1$, $\Gamma\in\mathcal{M}$ at $\Psi_{(1)}$ is equal to $\tilde{e}(\Gamma_0)-1$.
\end{enumerate}

We denote 
$$\mathcal{Z}_1=\{\tilde{\mathcal{L}}\subset\Comp(D_1)-\{\Lambda_{01}\}|
\tilde{\mathcal{L}}\cap\Comp((\nu\nu_1)^*\bar{\Delta})\neq\emptyset,
\Lambda_{01}\cap(\bigcap_{\Gamma\in\tilde{\mathcal{L}}}\Gamma)\neq\emptyset\}.$$
We denote the set of maximal elements of $\mathcal{Z}$ by $\mathcal{Z}\Mx$ and 
the set of maximal elements of $\mathcal{Z}_1$ by $\mathcal{Z}\Mx_1$.

We know the following:
\begin{enumerate}
\item
$\mathcal{Z}_1=\emptyset\Leftrightarrow $ $\Lambda_{01}\cap\Gamma=\emptyset$ for any $\Gamma\in \Comp((\nu\nu_1)^*\bar{\Delta})\cap\Comp(D_1)\Leftrightarrow \mathcal{Z}\Mx=\{\mathcal{L}\}$, $\Comp(\Lambda_0\cap(\cap_{\Gamma\in\mathcal{L}}\Gamma))=\{\Psi\}$ and $\Lambda_{01}\cap\Gamma_1=\emptyset$ for any $\Gamma\in\mathcal{L} $.
\item
If $\hat{\mathcal{L}}\in\mathcal{Z}\Mx_1$ and $\Gamma_1\not\in \hat{\mathcal{L}}$ for any $\Gamma\in\mathcal{L}$, then there exists uniquely an element $\tilde{\mathcal{L}}\in\mathcal{Z}\Mx$ with $\tilde{\mathcal{L}}\neq\mathcal{L}$ and $\hat{\mathcal{L}}=\{\Gamma_1|\Gamma\in\tilde{\mathcal{L}}\}$.
\item
If $\hat{\mathcal{L}}\in\mathcal{Z}\Mx_1$ and $\Gamma_1 \in \hat{\mathcal{L}}$ for some $\Gamma\in\mathcal{L}$, then $\hat{\mathcal{L}}\subset \{\Gamma_1|\Gamma\in\mathcal{L}\}$.
\item
Assume that there exists $\Phi\in\Comp(\Lambda_0\cap(\cap_{\Gamma\in\mathcal{L}}\Gamma))$ with $\Phi\neq\Psi$.

$\{\Gamma_1|\Gamma\in\mathcal{L}\}\in \mathcal{Z}\Mx_1$.

For any $\Phi\in\Comp(\Lambda_0\cap(\cap_{\Gamma\in\mathcal{L}}\Gamma))$ with $\Phi\neq\Psi$, $\Phi\not\subset\Psi$ and $\Phi_1$ is the unique irreducible component $\Theta$ of $\Lambda_{01}\cap(\cap_{\Gamma\in\mathcal{L}}\Gamma_1)$ with $\nu_1(\Theta)=\Phi$. $\Phi_1\not\subset\Psi_1$.

For any $\Theta\in\Comp(\Lambda_{01}\cap(\cap_{\Gamma\in\mathcal{L}}\Gamma_1))$ with $\Theta\not\subset\Psi_1$, $\nu_1(\Theta) \in\Comp(\Lambda_0\cap(\cap_{\Gamma\in\mathcal{L}}\Gamma))$, $\nu_1(\Theta) \neq\Psi$ and $\nu_1(\Theta)_1=\Theta$.
\item
If there exist $\hat{\mathcal{L}}\in \mathcal{Z}\Mx_1$ and an irreducible component $\Psi_{(1)}$ of $\Lambda_{01}\cap(\cap_{\Gamma\in\hat{\mathcal{L}}}\Gamma)$ with $\Psi_{(1)}\subset\Psi_1$, then $\hat{\mathcal{L}}\subset\{\Gamma_1|\Gamma\in\mathcal{L}\}$.
\item
Consider any $\tilde{\mathcal{L}}\in \mathcal{Z}\Mx$ with $\tilde{\mathcal{L}}\neq\mathcal{L}$.

$\{\Gamma_1|\Gamma\in\tilde{\mathcal{L}}\}\in \mathcal{Z}\Mx_1$.

For any  $\Phi\in\Comp(\Lambda_0\cap(\cap_{\Gamma\in\tilde{\mathcal{L}}}\Gamma))$, $\Phi\not\subset\Psi$ and $\Phi_1$ is the unique irreducible component $\Theta$ of $\Lambda_{01}\cap(\cap_{\Gamma\in\mathcal{L}}\Gamma_1)$ with $\nu_1(\Theta)=\Phi$. $\Phi_1\not\subset\Psi_1$.

For any $\Theta\in\Comp(\Lambda_{01}\cap(\cap_{\Gamma\in\tilde{\mathcal{L}}}\Gamma_1))$,
$\nu_1(\Theta) \in\Comp(\Lambda_0\cap(\cap_{\Gamma\in\tilde{\mathcal{L}}}\Gamma))$, $\Theta\not\subset\Psi_1$ $\nu_1(\Theta)\not\subset\Psi$ and $\nu_1(\Theta)_1=\Theta$
\end{enumerate}

We denote
\begin{equation*}\begin{split}
r&=\max\{\Codim (\Phi, \tilde{\Sigma})| \tilde{\mathcal{L}}\in\mathcal{Z}\Mx, \Phi\in\Comp(\Lambda_0\cap(\cap_{\Gamma\in\tilde{\mathcal{L}}}\Gamma))\}\in\Z_+\text{, and}\\ 
s&=\sum_{\tilde{\mathcal{L}}\in\mathcal{Z}\Mx}\sum_{\Phi\in\Comp(\Lambda_0\cap(\cap_{\Gamma\in\tilde{\mathcal{L}}}\Gamma)), \Codim (\Phi, \tilde{\Sigma})=r}\\
&\qquad\quad\sum_{\tilde{\mathcal{M}}\subset\tilde{\mathcal{L}}, \sharp \tilde{\mathcal{M}}=r-1, \Phi\in\Comp(\Lambda_0\cap(\cap_{\Gamma\in\tilde{\mathcal{M}}}\Gamma))}( \{\Lambda_0\}\cup\tilde{\mathcal{M}};\Phi)\in\Z_+.
\end{split}\end{equation*}

Consider the case $\mathcal{Z}_1\neq\emptyset$. We denote
\begin{equation*}\begin{split}
r_1&=\max\{\Codim (\Phi, \Sigma_1)| \tilde{\mathcal{L}}\in\mathcal{Z}\Mx_1, \Phi\in\Comp(\Lambda_{01}\cap(\cap_{\Gamma\in\tilde{\mathcal{L}}}\Gamma))\}\in\Z_+\text{, and}\\ 
s_1&=\sum_{\tilde{\mathcal{L}}\in\mathcal{Z}\Mx_1}\sum_{\Phi\in\Comp(\Lambda_{01}\cap(\cap_{\Gamma\in\tilde{\mathcal{L}}}\Gamma)), \Codim (\Phi, \Sigma_1)=r_1}\\
&\qquad\quad\sum_{\tilde{\mathcal{M}}\subset\tilde{\mathcal{L}}, \sharp \tilde{\mathcal{M}}=r_1-1, \Phi\in\Comp(\Lambda_{01}\cap(\cap_{\Gamma\in\tilde{\mathcal{M}}}\Gamma))}(\{\Lambda_{01}\}\cup\tilde{\mathcal{M}};\Phi)\in\Z_+.
\end{split}\end{equation*}

We know that if $\Codim(\Psi,\tilde{\Sigma})=r$ and $\mathcal{Z}_1\neq\emptyset$, then $r_1\leq r$, and $s_1<s$ if $r_1=r$.

By induction we conclude that there exists an admissible composition of blowing-ups $\nu_2:\Sigma_2\rightarrow\tilde{\Sigma}$ such that $\Lambda_{02}\cap\Gamma=\emptyset$ for any $\Gamma \in\Comp((\nu\nu_2)^*\bar{\Delta})\cap\Comp(D_2)$, where $\Lambda_{02}$ denotes the strict transform of $\Lambda_0$ by $\nu_2$ and $D_2$ denotes the sum of strict transforms of elements in $\Comp(\tilde{D})$ by $\nu_2$. Note that $\Lambda_{02}\in\Comp(D_2)$.

We conclude that there exists an admissible composition of blowing-ups $\nu:\tilde{\Sigma}\rightarrow\Sigma$ over $\bar{\Delta}$ such that the pair $(\tau\nu, (\tilde{\Sigma},\nu^*\bar{\Delta}, \nu^*\bar{\xi}))$ satisfies three conditions in the above four conditions (Z) except the last one and it does not satisfy the last condition in (Z). 

We consider the case where the pair $(\tau\nu, (\tilde{\Sigma},\nu^*\bar{\Delta}, \nu^*\bar{\xi}))$ does not satisfy the last condition in (Z). There exists $\Lambda\in\Comp(\tilde{D})-\Comp(\nu^*\bar{\Delta})$ such that $\Lambda\cap\Gamma=\emptyset$ for any $\Gamma\in \Comp(\nu^*\bar{\Delta})\cap \Comp(\tilde{D})$.

We take a non-empty maximal subset $\mathcal{K}$ of $\Comp(\tilde{D})-\Comp(\nu^*\bar{\Delta})$ satisfying two conditions $\Lambda\cap\Gamma=\emptyset$ for any $\Lambda\in\mathcal{K}$ and any $\Gamma\in\Comp(\nu^*\bar{\Delta})\cap\Comp(\tilde{D})$, and $\Lambda\cap \Gamma=\emptyset$ for any $\Lambda\in\mathcal{K}$ and any $\Gamma\in\mathcal{K}$ with $\Lambda\neq\Gamma$.
By maximality of $\mathcal{K}$, for any $\Lambda\in\Comp(\tilde{D})$, there exists an element $\Gamma\in (\Comp(\nu^*\bar{\Delta})\cap\Comp(\tilde{D}))\cup\mathcal{K}$ satisfying $\Lambda\cap\Gamma\neq\emptyset$.

We put $\tilde{\Delta}=\nu^*\bar{\Delta}+\sum_{\Lambda\in\mathcal{K}}\Lambda$.
$\tilde{\Delta}$ is an effective divisor of $\tilde{\Sigma}$. $\Supp(\nu^*\bar{\Delta})\subset\Supp(\tilde{\Delta})$. 
$\Comp(\tilde{\Delta})\cap\Comp(\tilde{D})=(\Comp(\nu^*\bar{\Delta})\cap\Comp(\tilde{D}))\cup\mathcal{K}$.
$\sharp (\Comp(\bar{\Delta})\cap\Comp(\bar{D}))=\sharp (\Comp(\nu^*\bar{\Delta})\cap\Comp(\tilde{D}))<\sharp (\Comp(\tilde{\Delta})\cap\Comp(\tilde{D}))$.
For any $\Lambda\in\Comp(\tilde{\Delta})\cap\Comp(\tilde{D})$ and any $\Gamma\in\Comp(\tilde{\Delta})\cap\Comp(\tilde{D})$ satisfying $\Lambda\neq\Gamma$, $\Lambda\cap\Gamma=\emptyset$.
For any $\Lambda\in\Comp(\tilde{D})$, there exists an element $\Gamma\in (\Comp(\tilde{\Delta})\cap\Comp(\tilde{D}))$ satisfying $\Lambda\cap\Gamma\neq\emptyset$.

Consider any $\Lambda\in\mathcal{K}$. By the first condition in (Z) we know the following.
\begin{enumerate}
\item
The component $\Lambda$ is smooth.
\end{enumerate}

Let $\tilde{U}_\Lambda=\tilde{\Sigma}-(\cup_{\Gamma\in\Comp(\tilde{\Delta}_0), \Gamma\cap\Lambda=\emptyset}\Gamma)\subset\tilde{\Sigma}$. Note that the subset $\tilde{U}_\Lambda$ is open in $\tilde{\Sigma}$ and it contains $\Lambda$.

\begin{enumerate}
\setcounter{enumi}{1}
\item
The pair $(\tilde{U}_\Lambda, (\Lambda+\tilde{\Delta}_0)|\tilde{U}_\Lambda)$ is a normal crossing scheme over $k$.
\end{enumerate}

We take the unique element $\chi\in\mathcal{X}\cup\{0\}$ such that $\Lambda$ is the strict transform of $\hat{\sigma}^*\Spec(A/(z+\chi)A)$ by $\tau\nu$. 

\begin{enumerate}
\setcounter{enumi}{2}
\item
$\Supp(\Lambda+\tilde{\Delta}_0)\cap \tilde{U}_\Lambda=(\tau\nu)^{-1}(\Supp(\hat{\sigma}^*\Spec(A/(z+\chi)A)+\pi^*\hat{\Delta}'))\cap \tilde{U}_\Lambda$.
\end{enumerate}

We know that $\tilde{\Delta}$ has normal crossings and the pair $(\tilde{\Sigma}, \tilde{\Delta})$ is a normal crossing scheme over $k$. 
$\Supp(\nu^*\bar{\Delta})\subset \Supp(\tilde{\Delta})$.
$(\nu^*\bar{\Delta})_0\subset (\tilde{\Delta})_0$.

Consider any $\B\in (\nu^*\bar{\Delta})_0$.

$\Comp(\tilde{\Delta})(\B)=\Comp(\nu^*\bar{\Delta})(\B)$, $U(\tilde{\Sigma},\tilde{\Delta},\B)\subset U(\tilde{\Sigma},\nu^*\bar{\Delta},\B)$.
We put
$\tilde{\xi}_\B(\Lambda)=\Res^{ U(\tilde{\Sigma},\nu^*\bar{\Delta},\B)}_{ U(\tilde{\Sigma},\tilde{\Delta},\B)}(\nu^*\bar{\xi})_\B(\Lambda)$ for any $\Lambda\in \Comp(\tilde{\Delta})(\B)$ and we define a mapping $\tilde{\xi}_\B: \Comp(\tilde{\Delta})(\B)\rightarrow\mathcal{O}_{\tilde{\Sigma}}( U(\tilde{\Sigma},\tilde{\Delta},\B))$.
Obviously $\tilde{\xi}_\B$ is a coordinate system of $(\tilde{\Sigma}, \tilde{\Delta})$ at $\B$.

Consider any $\B\in(\tilde{\Delta})_0-(\nu^*\bar{\Delta})_0$.

We take any $\A\in (\nu^*\bar{\Delta})_0$ with $\B\in U(\tilde{\Sigma},\nu^*\bar{\Delta},\A)$. 
$\sharp(\Comp(\tilde{\Delta})(\B)\cap\Comp(\nu^*\bar{\Delta})($\break$\A))+1=\sharp (\Comp(\tilde{\Delta})(\B))=\sharp(\Comp(\nu^*\bar{\Delta})(\A))$.
$\Comp(\tilde{\Delta})(\B)- \Comp(\nu^*\bar{\Delta})(\A)\in\mathcal{K}$.
$\Comp(\nu^*\bar{\Delta})(\A)- \Comp(\tilde{\Delta})(\B)\in \Comp(\nu^*\bar{\Delta})\cap\Comp(\tilde{D})$. 
$U(\tilde{\Sigma},\tilde{\Delta},\B)\subset U(\tilde{\Sigma},\nu^*\bar{\Delta},\A)$.
We put
$\tilde{\xi}_\B(\Lambda)=\Res^{ U(\tilde{\Sigma},\nu^*\bar{\Delta},\A)}_{ U(\tilde{\Sigma},\tilde{\Delta},\B)}(\nu^*\bar{\xi})_\A(\Lambda)$ for any $\Lambda\in \Comp(\tilde{\Delta})(\B)\cap\Comp(\nu^*\bar{\Delta})(\A)$.

Consider the unique element $\Lambda\in \Comp(\tilde{\Delta})(\B)- \Comp(\nu^*\bar{\Delta})(\A)$. $\Lambda\in\mathcal{K}\subset\Comp(\tilde{D})-\Comp(\nu^*\bar{\Delta})$. 
$\B\in\Lambda\cap U(\tilde{\Sigma},\nu^*\bar{\Delta},\A)\neq\emptyset$.
We take the unique element $\chi\in\mathcal{X}\cup\{0\}$ such that $\Lambda$ is the strict transform of $\hat{\sigma}^*\Spec(A/(z+\chi)A)$ by $\tau\nu$. 
We know that there exists uniquely a pair of an element $\lambda\in\mathcal{O}_{\tilde{\Sigma}}( U(\tilde{\Sigma},\nu^*\bar{\Delta},\A))$ and a mapping $c: \Comp(\nu^*\bar{\Delta})(\A)\rightarrow\Z_0$ satisfying 
$$\Res^{\tilde{\Sigma}}_{ U(\tilde{\Sigma},\nu^*\bar{\Delta},\A)}( \hat{\sigma}\tau\nu)(\Spec(A))( z+\chi)=\lambda\prod_{\Gamma\in \Comp(\nu^*\bar{\Delta})(\A)} (\nu^*\bar{\xi})_\A(\Gamma)^{c(\Gamma)}.$$
We take the unique pair of an element $\lambda\in\mathcal{O}_{\tilde{\Sigma}}( U(\tilde{\Sigma},\nu^*\bar{\Delta},\A))$ and a mapping $c: \Comp(\nu^*\bar{\Delta})(\A)\rightarrow\Z_0$ satisfying this equality. We know that $\Lambda\cap U(\tilde{\Sigma},\nu^*\bar{\Delta},\A)=
\Spec(\mathcal{O}_{\tilde{\Sigma}}( U(\tilde{\Sigma},\nu^*\bar{\Delta},\A))/\lambda\mathcal{O}_{\tilde{\Sigma}}( U(\tilde{\Sigma},\nu^*\bar{\Delta},\A)))$.
We put $\tilde{\xi}_\B(\Lambda)=\Res^{ U(\tilde{\Sigma},\nu^*\bar{\Delta},\A)}_{ U(\tilde{\Sigma},\tilde{\Delta},\B)}(\lambda)$.

We have a mapping $\tilde{\xi}_\B: \Comp(\tilde{\Delta})(\B)\rightarrow\mathcal{O}_{\tilde{\Sigma}}(U(\tilde{\Sigma},\tilde{\Delta},\B))$. We know that $\tilde{\xi}_\B$ is a coordinate system of $(\tilde{\Sigma},\tilde{\Delta})$ at $\B$.

Let $\tilde{\xi}=\{\tilde{\xi}_\B|\B\in (\tilde{\Delta})_0\}$. The set $\tilde{\xi}$ is a coordinate system of $(\tilde{\Sigma},\tilde{\Delta})$ and it is an extension of $\nu^*\bar{\xi}$.

The morphism $\tau\nu:\tilde{\Sigma}\rightarrow \hat{\Sigma}$ is a weakly admissible composition of blowing-ups over $(\hat{\Delta},\hat{\xi})$. The coordinated normal crossing scheme $(\tilde{\Sigma},\tilde{\Delta},\tilde{\xi})$ is an extended pull-back of $(\hat{\Sigma},\hat{\Delta},\hat{\xi})$ by $\tau\nu$. The triplet $(\tau\nu, (\tilde{\Sigma},\tilde{\Delta},\tilde{\xi}),\tilde{D})$ satisfies the five conditions (Z). $\sharp(\Comp(\tilde{\Delta})\cap\Comp(\tilde{D}))>\sharp(\Comp(\bar{\Delta})\cap\Comp(\bar{D}))$.

By induction on $\sharp\Comp(\bar{D})-\sharp(\Comp(\bar{\Delta})\cap\Comp(\bar{D}))$ we know that there exists a triplet $(\tau,(\Sigma,\bar{\Delta},\bar{\xi}),\bar{D})$ satisfying the above five conditions (Z) and $\sharp(\Comp(\bar{\Delta})\cap\Comp(\bar{D}))=\sharp\Comp(\bar{D})$.

We conclude that Theorem~\ref{make normal crossings} holds, if $\mathcal{X}\not\subset\{0\}$.

We know that Theorem~\ref{make normal crossings} holds in all cases.

\bibliographystyle{amsplain}

\end{document}